\documentclass[numbers=enddot,12pt,final,onecolumn,notitlepage]{scrartcl}%
\usepackage[headsepline,footsepline,manualmark]{scrlayer-scrpage}
\usepackage{amsfonts}
\usepackage{amssymb}
\usepackage{amsmath}
\usepackage{amsthm}
\usepackage{framed}
\usepackage{comment}
\usepackage[breaklinks=true]{hyperref}
\usepackage[sc]{mathpazo}
\usepackage[T1]{fontenc}
\usepackage{needspace}
\providecommand{\U}[1]{\protect\rule{.1in}{.1in}}
\theoremstyle{definition}
\newtheorem{theo}{Theorem}[section]
\newenvironment{theorem}[1][]
{\begin{theo}[#1]\begin{leftbar}}
{\end{leftbar}\end{theo}}
\newtheorem{lem}[theo]{Lemma}
\newenvironment{lemma}[1][]
{\begin{lem}[#1]\begin{leftbar}}
{\end{leftbar}\end{lem}}
\newtheorem{prop}[theo]{Proposition}
\newenvironment{proposition}[1][]
{\begin{prop}[#1]\begin{leftbar}}
{\end{leftbar}\end{prop}}
\newtheorem{defi}[theo]{Definition}
\newenvironment{definition}[1][]
{\begin{defi}[#1]\begin{leftbar}}
{\end{leftbar}\end{defi}}
\newtheorem{remk}[theo]{Remark}
\newenvironment{remark}[1][]
{\begin{remk}[#1]\begin{leftbar}}
{\end{leftbar}\end{remk}}
\newtheorem{coro}[theo]{Corollary}
\newenvironment{corollary}[1][]
{\begin{coro}[#1]\begin{leftbar}}
{\end{leftbar}\end{coro}}
\newtheorem{conv}[theo]{Convention}
\newenvironment{convention}[1][]
{\begin{conv}[#1]\begin{leftbar}}
{\end{leftbar}\end{conv}}
\newtheorem{exam}[theo]{Example}
\newenvironment{example}[1][]
{\begin{exam}[#1]\begin{leftbar}}
{\end{leftbar}\end{exam}}
\newenvironment{statement}{\begin{quote}}{\end{quote}}
\newenvironment{fineprint}{\begin{small}}{\end{small}}

\let\sumnonlimits\sum
\let\prodnonlimits\prod
\let\oplusnonlimits\bigoplus
\renewcommand{\sum}{\sumnonlimits\limits}
\renewcommand{\prod}{\prodnonlimits\limits}
\renewcommand{\bigoplus}{\oplusnonlimits\limits}
\newenvironment{verlong}{}{}
\newenvironment{vershort}{}{}
\newenvironment{noncompile}{}{}
\excludecomment{verlong}
\includecomment{vershort}
\excludecomment{noncompile}
\ihead{The eta-basis of QSym}
\ohead{page \thepage}
\begin{document}

\title{The enriched $q$-monomial basis of the quasisymmetric functions}
\author{Darij Grinberg and Ekaterina A. Vassilieva}
\date{1 June 2026}
\maketitle

\begin{abstract}
\textbf{Abstract.} We construct a new family $\left(  \eta_{\alpha}^{\left(
q\right)  }\right)  _{\alpha\in\operatorname*{Comp}}$ of quasisymmetric
functions for each element $q$ of the base ring. We call them the
\textquotedblleft enriched $q$-monomial quasisymmetric
functions\textquotedblright. When $r:=q+1$ is invertible, this family is a
basis of $\operatorname*{QSym}$. It generalizes Hoffman's \textquotedblleft
essential quasi-symmetric functions\textquotedblright\ (obtained for $q=0$)
and Hsiao's \textquotedblleft monomial peak functions\textquotedblright%
\ (obtained for $q=1$), but also includes the monomial quasisymmetric
functions as a limiting case.

We describe these functions $\eta_{\alpha}^{\left(  q\right)  }$ by several
formulas, and compute their products, coproducts and antipodes. The product
expansion is given by an exotic variant of the shuffle product which we call
the \textquotedblleft stufufuffle product\textquotedblright\ due to its
ability to pick several consecutive entries from each composition. This
\textquotedblleft stufufuffle product\textquotedblright\ has previously
appeared in recent work by Bouillot, Novelli and Thibon, generalizing the
\textquotedblleft block shuffle product\textquotedblright\ from the theory of
multizeta values. \medskip



\end{abstract}
\tableofcontents

\section{\label{sec.intro}Introduction}

Among the combinatorial Hopf algebras that consist of power series in
commuting indeterminates, one of the largest and most all-embracing is that of
the \textit{quasisymmetric functions}, called $\operatorname*{QSym}$.
Originally introduced by Gessel in 1984 \cite{Gessel84}, it has since found
applications (e.g.) to enumerative combinatorics (\cite[Chapter 8]{Sagan20},
\cite[\S 7.19--7.23]{Stanley-EC2}, \cite{GesZhu18}), multizeta values (e.g.,
\cite{Hoffma15}), algebraic geometry (\cite{Oesing18}) and the representation
theory of $0$-Hecke algebras (\cite[\S 6.2]{Meliot17}).

It was observed by Ehrenborg (\cite[Lemma 4.2]{Ehrenb96}; see \cite[\S 3.3]%
{Biller10} for a survey) that quasisymmetric functions can also be used to
encode the \textquotedblleft flag $f$-vector\textquotedblright\ of a finite
graded poset -- i.e., essentially, the number of chains over a given sequence
of ranks, for each possible sequence of ranks. Soon after, work of Bergeron,
Mykytiuk, Sottile and van Willigenburg (\cite[Example 5.3]{BMSW00}, but see
\cite[\S 3.4]{Biller10} for an explicit statement) showed that if the graded
poset is Eulerian (a property shared by face posets of polytopes and
simplicial spheres), then the resulting quasisymmetric function is not
arbitrary but rather belongs to a certain subalgebra of $\operatorname*{QSym}$
called \textit{Stembridge's Hopf algebra} or the \textit{peak algebra} or the
\textit{odd subalgebra }$\Pi_{-}$ \textit{of }$\operatorname*{QSym}$. It was
initially defined by Stembridge \cite[\S 3]{Stembr97} in order to find a
fundamental expansion of the Schur $P$- and $Q$-functions, and has since been
studied by others for related and unrelated reasons (\cite[\S 6, particularly
Proposition 6.5]{AgBeSo14}, \cite{BMSW99}, \cite[\S 5]{BMSW00}, \cite{Hsiao07}
etc.); among other properties, it is a Hopf subalgebra of
$\operatorname*{QSym}$.

Almost all bases of $\operatorname*{QSym}$ constructed so far are indexed by
\textit{compositions} (i.e., tuples of positive integers), and their structure
constants are often governed by versions of shuffle products and
deconcatenation coproducts. The peak algebra is smaller, and its bases are
often indexed by \textit{odd compositions}, i.e., compositions whose entries
are all odd. One of its simplest bases is defined as follows (for the sake of
simplicity, we use $\mathbb{Q}$ as a base ring here): If $n\in\mathbb{N}$ and
if $\alpha=\left(  \alpha_{1},\alpha_{2},\ldots,\alpha_{\ell}\right)  $ is a
composition of $n$ (that is, a tuple of positive integers with $\alpha
_{1}+\alpha_{2}+\cdots+\alpha_{\ell}=n$), then we define the formal power
series%
\begin{align}
\eta_{\alpha}  &  =\sum_{\substack{1\leq g_{1}\leq g_{2}\leq\cdots\leq
g_{n};\\g_{i}=g_{i+1}\text{ for each }i\in E\left(  \alpha\right)
}}2^{\left\vert \left\{  g_{1},g_{2},\ldots,g_{n}\right\}  \right\vert
}x_{g_{1}}x_{g_{2}}\cdots x_{g_{n}}\label{eq.intro.eta=}\\
&  \in\mathbb{Q}\left[  \left[  x_{1},x_{2},x_{3},\ldots\right]  \right]
,\nonumber
\end{align}
where $E\left(  \alpha\right)  $ denotes the set $\left\{  1,2,\ldots
,n-1\right\}  \setminus\left\{  \alpha_{1}+\alpha_{2}+\cdots+\alpha_{i}%
\ \mid\ 0<i<\ell\right\}  $. This $\eta_{\alpha}$ belongs to the $\mathbb{Q}%
$-algebra $\operatorname*{QSym}$ of quasisymmetric functions over $\mathbb{Q}%
$. If we let $\alpha$ range over all \textit{odd compositions} (i.e.,
compositions $\left(  \alpha_{1},\alpha_{2},\ldots,\alpha_{\ell}\right)  $
whose entries $\alpha_{i}$ are all odd), then the $\eta_{\alpha}$ form a basis
of the peak algebra over $\mathbb{Q}$.

In this form, the power series $\eta_{\alpha}$ have been introduced by Hsiao
(\cite[Proposition 2.1]{Hsiao07}, although his $\eta_{\alpha}$ differ from
ours by a sign), who called them the \textit{monomial peak functions}. Hsiao
computed their products, coproducts (in the sense of Hopf algebra) and
antipodes, and obtained some structural results for the peak algebra.

In this paper, we generalize the $\eta_{\alpha}$ by replacing the power of $2$
in (\ref{eq.intro.eta=}) by a power of an arbitrary element $r$ of the base
ring. We furthermore study the resulting quasisymmetric functions for all
compositions $\alpha$ (not only for the odd ones). Thus we obtain a new family
$\left(  \eta_{\alpha}^{\left(  q\right)  }\right)  _{\alpha\in
\operatorname*{Comp}}$ of quasisymmetric functions for each element $q$ of the
base ring. When $r:=q+1$ is invertible, this family is a basis of
$\operatorname*{QSym}$. It generalizes Hoffman's \textquotedblleft essential
quasi-symmetric functions\textquotedblright\ (obtained for $q=0$) and Hsiao's
monomial peak functions\ (obtained for $q=1$), but also includes the monomial
quasisymmetric functions as a limiting case.

We call our functions $\eta_{\alpha}^{\left(  q\right)  }$ the
\textit{enriched }$q$\textit{-monomial quasisymmetric functions}. We describe
them by several formulas, and compute their products, coproducts and antipodes
(generalizing Hsiao's results). The product expansion is the most interesting
one, as it is given by an exotic variant of the shuffle product which we call
the \textquotedblleft stufufuffle product\textquotedblright\ due to its
ability to pick several consecutive entries from each composition. This
\textquotedblleft stufufuffle product\textquotedblright\ has previously
appeared in recent work by Bouillot, Novelli and Thibon \cite[(1)]{BoNoTh22},
where it was proposed as a generalization of the \textquotedblleft block
shuffle product\textquotedblright\ from the theory of multizeta values
(\cite{HirSat22}). While the authors of \cite{BoNoTh22} have already found a
basis of $\operatorname*{QSym}$ that multiplies according to this product,
ours is simpler and more natural. The coproduct and antipode formulas for
$\eta_{\alpha}^{\left(  q\right)  }$ are fairly simple (the coproduct is given
by deconcatenation, whereas the antipode involves the parameter $q$ being
replaced by its reciprocal $1/q$ and the composition $\alpha$ being reversed).
We also express the functions $\eta_{\alpha}^{\left(  q\right)  }$ in terms of
the monomial and fundamental bases of $\operatorname*{QSym}$ and vice versa.
Finally, we discuss how Hopf subalgebras of $\operatorname*{QSym}$ can be
constructed by picking a subset of the set of all compositions. (This
generalizes the peak subalgebra.)

This paper is the first of (at least) two. The next shall extend the theory of
extended $P$-partitions to incorporate the parameter $q$, which will shed a
new light on the enriched $q$-monomial quasisymmetric functions $\eta_{\alpha
}^{\left(  q\right)  }$ while also leading to a new basis of
$\operatorname*{QSym}$.

Several results in this paper have appeared (mostly without proof) in the
extended abstracts \cite{GriVas21} and \cite{GriVas22}.

\subsection{Structure of the paper}

This paper is organized as follows:

We begin by recalling the definition of quasisymmetric functions (and some
concomitant notions) in Section \ref{sec.qsym}.

Then, in Section \ref{sec.eta}, we define the quasisymmetric functions
$\eta_{\alpha}^{\left(  q\right)  }$ and prove their simplest properties
(conversion formulas to the $M$- and $L$-bases, formulas for antipode and
coproduct). In particular, we show that the family of these functions
$\eta_{\alpha}^{\left(  q\right)  }$ (where $\alpha$ ranges over all
compositions) forms a basis of $\operatorname*{QSym}$ if and only if $r:=q+1$
is invertible in the base ring.

Consequently, in Section \ref{sec.dual}, we introduce and study the basis of
$\operatorname*{NSym}$ dual to this basis of $\operatorname*{QSym}$.

In Section \ref{sec.prod}, we use this to express the product $\eta_{\delta
}^{\left(  q\right)  }\eta_{\varepsilon}^{\left(  q\right)  }$ in three
equivalent ways.

Finally, we discuss some applications in Section \ref{sec.app}, and establish
one last formula for $\eta_{\alpha}^{\left(  q\right)  }$ in Section
\ref{sec.Rq}.

\begin{verlong}
You are reading the \textbf{detailed version} of the present paper. In this
version, some proofs in Section \ref{sec.eta} have been expanded to full
detail. (Other sections are mostly unaffected.)
\end{verlong}

\subsubsection*{Acknowledgments}

We thank Marcelo Aguiar, G\'{e}rard H. E. Duchamp, Daniel Harrer, Angela
Hicks, Vasu Tewari, Alexander Zhang, and Yan Zhuang for interesting and
helpful conversations (in particular, Angela Hicks suggested the possibility
of an explicit expression that became Proposition \ref{prop.eta.F-through}).
Two referees have contributed corrections and helpful comments. The SageMath
computer algebra system \cite{SageMath} helped discover some of the results.

DG is grateful to Sara Billey, Petter Br\"{a}nd\'{e}n, Sylvie Corteel, and
Svante Linusson for organizing the Spring Semester 2020 in Algebraic and
Enumerative Combinatorics at the Institut Mittag-Leffler, at which parts of
this paper have been written.

\begin{fineprint}
This material is based upon work supported by the Swedish Research Council
under grant no. 2016-06596 while the first author was in residence at Institut
Mittag-Leffler in Djursholm, Sweden during Spring 2020.
\end{fineprint}

\section{\label{sec.qsym}Quasisymmetric functions}

\subsection{Formal power series and quasisymmetry}

We will use some of the standard notations from \cite[Chapter 5]{GriRei}. Namely:

\begin{itemize}
\item We let $\mathbb{N}=\left\{  0,1,2,\ldots\right\}  $.

\item We fix a commutative ring $\mathbf{k}$.

\item We consider the ring $\mathbf{k}\left[  \left[  x_{1},x_{2},x_{3}%
,\ldots\right]  \right]  $ of formal power series in countably many commuting
variables $x_{1},x_{2},x_{3},\ldots$. A \emph{monomial} shall mean a formal
expression of the form $x_{1}^{\alpha_{1}}x_{2}^{\alpha_{2}}x_{3}^{\alpha_{3}%
}\cdots$, where $\alpha=\left(  \alpha_{1},\alpha_{2},\alpha_{3}%
,\ldots\right)  \in\mathbb{N}^{\infty}$ is a sequence of nonnegative integers
such that only finitely many $\alpha_{i}$ are positive. Formal power series
are formal infinite $\mathbf{k}$-linear combinations of such monomials.

\item Each monomial $x_{1}^{\alpha_{1}}x_{2}^{\alpha_{2}}x_{3}^{\alpha_{3}%
}\cdots$ has degree $\alpha_{1}+\alpha_{2}+\alpha_{3}+\cdots$.

\item A formal power series $f\in\mathbf{k}\left[  \left[  x_{1},x_{2}%
,x_{3},\ldots\right]  \right]  $ is said to be \emph{of bounded degree} if
there exists some $d\in\mathbb{N}$ such that each monomial in $f$ has degree
$\leq d$ (that is, each monomial of degree $>d$ has coefficient $0$ in $f$).
\end{itemize}

For example, the formal power series $\left(  x_{1}+x_{2}+x_{3}+\cdots\right)
^{3}$ is of bounded degree, but the formal power series $\dfrac{1}{1-x_{1}%
}=1+x_{1}+x_{1}^{2}+x_{1}^{3}+\cdots$ is not.

We shall now introduce the notion of \textit{pack-equivalent monomials}. Let
us first illustrate it by an example:

\begin{example}
\textbf{Question:} What do the monomials $x_{1}^{4}x_{3}^{7}x_{4}x_{9}^{2}$
and $x_{3}^{4}x_{4}^{7}x_{10}x_{16}^{2}$ and $x_{5}^{4}x_{6}^{7}x_{8}x_{9}%
^{2}$ have in common (but not in common with $x_{1}^{7}x_{3}^{4}x_{4}x_{9}%
^{2}$) ?

\textbf{Answer:} They have the same sequence of nonzero exponents (when the
variables are ordered in increasing order -- i.e., if $i<j$, then $x_{i}$ goes
before $x_{j}$). Or, to put it differently, they all have the form $x_{a}%
^{4}x_{b}^{7}x_{c}x_{d}^{2}$ for some $a<b<c<d$. We shall call such pairs of
monomials \textit{pack-equivalent}.
\end{example}

Let us define this concept more rigorously:

\begin{definition}
Two monomials $\mathfrak{m}$ and $\mathfrak{n}$ are said to be
\emph{pack-equivalent} if they can be written in the forms%
\[
\mathfrak{m}=x_{i_{1}}^{a_{1}}x_{i_{2}}^{a_{2}}\cdots x_{i_{\ell}}^{a_{\ell}%
}\ \ \ \ \ \ \ \ \ \ \text{and}\ \ \ \ \ \ \ \ \ \ \mathfrak{n}=x_{j_{1}%
}^{a_{1}}x_{j_{2}}^{a_{2}}\cdots x_{j_{\ell}}^{a_{\ell}}%
\]
for some $\ell\in\mathbb{N}$, some positive integers $a_{1},a_{2}%
,\ldots,a_{\ell}$ and two strictly increasing $\ell$-tuples $\left(
i_{1}<i_{2}<\cdots<i_{\ell}\right)  $ and $\left(  j_{1}<j_{2}<\cdots<j_{\ell
}\right)  $ of positive integers.
\end{definition}

For example, the monomials $x_{1}^{4}x_{3}^{7}x_{4}x_{9}^{2}$ and $x_{3}%
^{4}x_{4}^{7}x_{10}x_{16}^{2}$ are pack-equivalent, since they can be written
as $x_{1}^{4}x_{3}^{7}x_{4}x_{9}^{2}=x_{i_{1}}^{a_{1}}x_{i_{2}}^{a_{2}}\cdots
x_{i_{\ell}}^{a_{\ell}}$ and $x_{3}^{4}x_{4}^{7}x_{10}x_{16}^{2}=x_{j_{1}%
}^{a_{1}}x_{j_{2}}^{a_{2}}\cdots x_{j_{\ell}}^{a_{\ell}}$ for $\ell=4$ and
$\left(  a_{1},a_{2},\ldots,a_{\ell}\right)  =\left(  4,7,1,2\right)  $ and
$\left(  i_{1}<i_{2}<\cdots<i_{\ell}\right)  =\left(  1,3,4,9\right)  $ and
$\left(  j_{1}<j_{2}<\cdots<j_{\ell}\right)  =\left(  3,4,10,16\right)  $.

We are now ready to define the quasisymmetric functions:

\begin{definition}
\begin{enumerate}
\item[\textbf{(a)}] A formal power series $f\in\mathbf{k}\left[  \left[
x_{1},x_{2},x_{3},\ldots\right]  \right]  $ is said to be
\emph{quasisymmetric} if it has the property that any two pack-equivalent
monomials have the same coefficient in $f$ (that is: if $\mathfrak{m}$ and
$\mathfrak{n}$ are two pack-equivalent monomials, then the coefficient of
$\mathfrak{m}$ in $f$ equals the coefficient of $\mathfrak{n}$ in $f$).

\item[\textbf{(b)}] A \emph{quasisymmetric function} means a formal power
series $f\in\mathbf{k}\left[  \left[  x_{1},x_{2},x_{3},\ldots\right]
\right]  $ that is quasisymmetric and of bounded degree.
\end{enumerate}
\end{definition}

Quasisymmetric functions have been introduced by Gessel in \cite{Gessel84}
(for $\mathbf{k}=\mathbb{Z}$ at least, but the general case is not much
different). Introductions to their theory can be found in \cite[Chapters
5--6]{GriRei}, \cite[\S 7.19]{Stanley-EC2}, \cite[Chapter 8]{Sagan20},
\cite[\S 4]{Malven93} and various other texts.

It is known (see \cite[Corollaire 4.7]{Malven93} or \cite[Proposition
5.1.3]{GriRei}) that the set of all quasisymmetric functions is a $\mathbf{k}%
$-subalgebra of $\mathbf{k}\left[  \left[  x_{1},x_{2},x_{3},\ldots\right]
\right]  $. It is denoted by $\operatorname*{QSym}$ and called the \emph{ring
of quasisymmetric functions}. It has several bases (as a $\mathbf{k}$-module),
most of which are indexed by compositions.

\subsection{Compositions}

A \emph{composition} means a finite list $\left(  \alpha_{1},\alpha_{2}%
,\ldots,\alpha_{k}\right)  $ of positive integers. The set of all compositions
is denoted by $\operatorname*{Comp}$. The \emph{empty composition
}$\varnothing$ is the composition $\left(  {}\right)  $, which is a $0$-tuple.

The \emph{length }$\ell\left(  \alpha\right)  $ of a composition
$\alpha=\left(  \alpha_{1},\alpha_{2},\ldots,\alpha_{k}\right)  $ is defined
to be the number $k$.

If $\alpha=\left(  \alpha_{1},\alpha_{2},\ldots,\alpha_{k}\right)  $ is a
composition, then the nonnegative integer $\alpha_{1}+\alpha_{2}+\cdots
+\alpha_{k}$ is called the \emph{size} of $\alpha$ and is denoted by
$\left\vert \alpha\right\vert $. For any $n\in\mathbb{N}$, we define a
\emph{composition of }$n$ to be a composition that has size $n$. We let
$\operatorname*{Comp}\nolimits_{n}$ be the set of all compositions of $n$ (for
given $n\in\mathbb{N}$). For example, $\left(  1,5,2,1\right)  $ is a
composition with size $9$ (since $\left\vert \left(  1,5,2,1\right)
\right\vert =1+5+2+1=9$), so that $\left(  1,5,2,1\right)  \in
\operatorname*{Comp}\nolimits_{9}$.

\begin{verlong}
Some authors use the notation \textquotedblleft$\alpha\models n$%
\textquotedblright\ for \textquotedblleft$\alpha\in\operatorname*{Comp}%
\nolimits_{n}$\textquotedblright. Thus, for example, $\left(  1,5,2,1\right)
\in\operatorname*{Comp}\nolimits_{9}$ can be rewritten as $\left(
1,5,2,1\right)  \models9$.
\end{verlong}

For any $n\in\mathbb{Z}$, we let $\left[  n\right]  $ denote the set $\left\{
1,2,\ldots,n\right\}  $. This set is empty whenever $n\leq0$, and otherwise
has size $n$.

It is well-known that any positive integer $n$ has exactly $2^{n-1}$
compositions. This has a standard bijective proof (\textquotedblleft stars and
bars\textquotedblright) which is worth recalling, as the bijection itself will
be used a lot:

\begin{definition}
\label{def.comps.D-comp}Let $n\in\mathbb{N}$. Let $\mathcal{P}\left(  \left[
n-1\right]  \right)  $ be the powerset of $\left[  n-1\right]  $ (that is, the
set of all subsets of $\left[  n-1\right]  $).

\begin{enumerate}
\item[\textbf{(a)}] We define a map $D:\operatorname*{Comp}\nolimits_{n}%
\rightarrow\mathcal{P}\left(  \left[  n-1\right]  \right)  $ by
\begin{align*}
D\left(  \alpha_{1},\alpha_{2},\ldots,\alpha_{k}\right)   &  =\left\{
\alpha_{1}+\alpha_{2}+\cdots+\alpha_{i}\ \mid\ i\in\left[  k-1\right]
\right\} \\
&  =\left\{  \alpha_{1},\ \ \alpha_{1}+\alpha_{2},\ \ \alpha_{1}+\alpha
_{2}+\alpha_{3},\ \ \ldots,\ \ \alpha_{1}+\alpha_{2}+\cdots+\alpha
_{k-1}\right\}  .
\end{align*}

\item[\textbf{(b)}] We define a map $\operatorname*{comp}:\mathcal{P}\left(
\left[  n-1\right]  \right)  \rightarrow\operatorname*{Comp}\nolimits_{n}$ as
follows: For any $I\in\mathcal{P}\left(  \left[  n-1\right]  \right)  $, we
set%
\[
\operatorname*{comp}\left(  I\right)  =\left(  i_{1}-i_{0},i_{2}-i_{1}%
,\ldots,i_{m}-i_{m-1}\right)  ,
\]
where $i_{0},i_{1},\ldots,i_{m}$ are the elements of the set $I\cup\left\{
0,n\right\}  $ listed in increasing order (so that $i_{0}<i_{1}<\cdots<i_{m}$,
therefore $i_{0}=0$ and $i_{m}=n$ and $\left\{  i_{1}<i_{2}<\cdots
<i_{m-1}\right\}  =I$).
\end{enumerate}

The maps $D$ and $\operatorname*{comp}$ are mutually inverse bijections. (See
\cite[detailed version, Proposition 10.17]{Grinbe15} for a detailed proof of this.)
\end{definition}

For example, for $n=8$, we have $D\left(  2,1,3,2\right)  =\left\{
2,\ 2+1,\ 2+1+3\right\}  =\left\{  2,3,6\right\}  $ and $\operatorname*{comp}%
\left\{  2,3,6\right\}  =\left(  2-0,\ 3-2,\ 6-3,\ 8-6\right)  =\left(
2,1,3,2\right)  $. Note that the meaning of $\operatorname*{comp}\left(
I\right)  $ for a given set $I$ depends on $n$, and thus the notation is
ambiguous unless $n$ is specified. In contrast, the notation $D\left(
\alpha\right)  $ is unambiguous, since $\alpha\in\operatorname*{Comp}%
\nolimits_{n}$ uniquely determines $n$ to be $\left\vert \alpha\right\vert $.

The notation $D$ in Definition \ref{def.comps.D-comp} presumably originates in
the word \textquotedblleft\textbf{d}escent\textquotedblright, but the
connection between $D$ and actual descents is indirect and rather misleading.
We prefer to call $D$ the \textquotedblleft partial sum map\textquotedblright%
\ (as $D\left(  \alpha\right)  $ consists of the partial sums of the
composition $\alpha$) and its inverse $\operatorname*{comp}$ the
\textquotedblleft interstitial map\textquotedblright\ (as
$\operatorname*{comp}\left(  I\right)  $ consists of the lengths of the
intervals into which the elements of $I$ split the interval $\left[  n\right]
$). Note that Stanley, in \cite[\S 7.19]{Stanley-EC2}, writes $S_{\alpha}$ for
$D\left(  \alpha\right)  $ and writes $\operatorname*{co}\left(  I\right)  $
for $\operatorname*{comp}\left(  I\right)  $.

Note that every composition $\alpha$ of size $\left\vert \alpha\right\vert >0$
satisfies $\left\vert D\left(  \alpha\right)  \right\vert =\ell\left(
\alpha\right)  -1$, so that $\left\vert D\left(  \alpha\right)  \right\vert
+1=\ell\left(  \alpha\right)  $. But this fails if $\alpha$ is the empty
composition $\varnothing=\left(  {}\right)  $ (since $D\left(  {}\right)
=\varnothing$ and $\ell\left(  {}\right)  =0$).

\subsection{The monomial and fundamental bases of $\operatorname*{QSym}$}

We will only need two bases of $\operatorname*{QSym}$: the monomial basis and
the fundamental basis.

If $\alpha=\left(  \alpha_{1},\alpha_{2},\ldots,\alpha_{\ell}\right)  $ is a
composition, then we define the \emph{monomial quasisymmetric function}
$M_{\alpha}\in\operatorname*{QSym}$ by%
\begin{equation}
M_{\alpha}=\sum_{i_{1}<i_{2}<\cdots<i_{\ell}}x_{i_{1}}^{\alpha_{1}}x_{i_{2}%
}^{\alpha_{2}}\cdots x_{i_{\ell}}^{\alpha_{\ell}}=\sum_{\substack{\mathfrak{m}%
\text{ is a monomial}\\\text{pack-equivalent}\\\text{to }x_{1}^{\alpha_{1}%
}x_{2}^{\alpha_{2}}\cdots x_{\ell}^{\alpha_{\ell}}}}\mathfrak{m}.
\label{eq.Malpha.def}%
\end{equation}

For example,%
\[
M_{\left(  2,1\right)  }=\sum_{i<j}x_{i}^{2}x_{j}=x_{1}^{2}x_{2}+x_{1}%
^{2}x_{3}+x_{2}^{2}x_{3}+x_{1}^{2}x_{4}+x_{2}^{2}x_{4}+x_{3}^{2}x_{4}+\cdots.
\]

The family $\left(  M_{\alpha}\right)  _{\alpha\in\operatorname*{Comp}}$ is a
basis of the $\mathbf{k}$-module $\operatorname*{QSym}$, and is known as the
\emph{monomial basis} of $\operatorname*{QSym}$.

For any composition $\alpha$, we define the \emph{fundamental quasisymmetric
function }$L_{\alpha}\in\operatorname*{QSym}$ by
\begin{equation}
L_{\alpha}=\sum_{\substack{\beta\in\operatorname*{Comp}\nolimits_{n}%
;\\D\left(  \beta\right)  \supseteq D\left(  \alpha\right)  }}M_{\beta},
\label{eq.Lalpha.def}%
\end{equation}
where $n=\left\vert \alpha\right\vert $ (so that $\alpha\in
\operatorname*{Comp}\nolimits_{n}$). It is not hard to rewrite this
as\footnote{See \cite[detailed version, Corollary 10.18]{Grinbe15} for a proof
that the right hand side of (\ref{eq.Lalpha.xprod}) equals the right hand side
of (\ref{eq.Lalpha.def}).}%
\begin{equation}
L_{\alpha}=\sum_{\substack{i_{1}\leq i_{2}\leq\cdots\leq i_{n};\\i_{j}%
<i_{j+1}\text{ whenever }j\in D\left(  \alpha\right)  }}x_{i_{1}}x_{i_{2}%
}\cdots x_{i_{n}} \label{eq.Lalpha.xprod}%
\end{equation}
(again with $n=\left\vert \alpha\right\vert $). This quasisymmetric function
$L_{\alpha}$ was originally called $F_{\alpha}$ in Gessel's paper
\cite{Gessel84} (and in some later work such as \cite{Malven93}), but the
notation $L_{\alpha}$ has since spread more widely.

The family $\left(  L_{\alpha}\right)  _{\alpha\in\operatorname*{Comp}}$ is a
basis of the $\mathbf{k}$-module $\operatorname*{QSym}$, and is known as the
\emph{fundamental basis} of $\operatorname*{QSym}$.

\begin{verlong}

\begin{remark}
Using M\"{o}bius inversion on the Boolean lattice $\mathcal{P}\left(  \left[
n-1\right]  \right)  $, the definition (\ref{eq.Lalpha.def}) of the
fundamental basis can be turned around to obtain an expression of the
$M_{\alpha}$ in the fundamental basis. Namely, if $\alpha$ is a composition,
and if $n=\left\vert \alpha\right\vert $, then%
\[
M_{\alpha}=\sum_{\substack{\beta\in\operatorname*{Comp}\nolimits_{n}%
;\\D\left(  \beta\right)  \supseteq D\left(  \alpha\right)  }}\left(
-1\right)  ^{\ell\left(  \beta\right)  -\ell\left(  \alpha\right)  }L_{\beta
}.
\]
(See \cite[Proposition 5.2.8]{GriRei} for more details of the proof. In a
nutshell, the equality follows from M\"{o}bius inversion using the fact that
$\left\vert D\left(  \beta\right)  \setminus D\left(  \alpha\right)
\right\vert =\ell\left(  \beta\right)  -\ell\left(  \alpha\right)  $ whenever
$\alpha,\beta\in\operatorname*{Comp}\nolimits_{n}$ satisfy $D\left(
\beta\right)  \supseteq D\left(  \alpha\right)  $.)
\end{remark}
\end{verlong}

\section{\label{sec.eta}The enriched $q$-monomial functions}

\subsection{Definition and restatements}

\begin{convention}
\label{conv.q-r}From now on, we fix an element $q$ of the base ring
$\mathbf{k}$. We set%
\[
r:=q+1.
\]

\end{convention}

We shall now introduce a new family of quasisymmetric functions depending on
$q$:

\begin{definition}
\label{def.etaalpha}For any $n\in\mathbb{N}$ and any composition $\alpha
\in\operatorname*{Comp}\nolimits_{n}$, we define a quasisymmetric function
$\eta_{\alpha}^{\left(  q\right)  }\in\operatorname*{QSym}$ by%
\begin{equation}
\eta_{\alpha}^{\left(  q\right)  }=\sum_{\substack{\beta\in
\operatorname*{Comp}\nolimits_{n};\\D\left(  \beta\right)  \subseteq D\left(
\alpha\right)  }}r^{\ell\left(  \beta\right)  }M_{\beta}.
\label{eq.def.etaalpha.def}%
\end{equation}
We shall refer to $\eta_{\alpha}^{\left(  q\right)  }$ as the \emph{enriched
$q$-monomial function} corresponding to $\alpha$.
\end{definition}

\begin{example}
\label{exa.etaalpha.1}\ \ 

\begin{enumerate}
\item[\textbf{(a)}] Setting $n=5$ and $\alpha=\left(  1,3,1\right)  $ in this
definition, we obtain%
\begin{align*}
\eta_{\left(  1,3,1\right)  }^{\left(  q\right)  }  &  =\sum_{\substack{\beta
\in\operatorname*{Comp}\nolimits_{5};\\D\left(  \beta\right)  \subseteq
D\left(  1,3,1\right)  }}r^{\ell\left(  \beta\right)  }M_{\beta}\\
&  =\sum_{\substack{\beta\in\operatorname*{Comp}\nolimits_{5};\\D\left(
\beta\right)  \subseteq\left\{  1,4\right\}  }}r^{\ell\left(  \beta\right)
}M_{\beta}\ \ \ \ \ \ \ \ \ \ \left(  \text{since }D\left(  1,3,1\right)
=\left\{  1,4\right\}  \right) \\
&  =r^{\ell\left(  5\right)  }M_{\left(  5\right)  }+r^{\ell\left(
1,4\right)  }M_{\left(  1,4\right)  }+r^{\ell\left(  4,1\right)  }M_{\left(
4,1\right)  }+r^{\ell\left(  1,3,1\right)  }M_{\left(  1,3,1\right)  }%
\end{align*}
(since the compositions $\beta\in\operatorname*{Comp}\nolimits_{5}$ satisfying
$D\left(  \beta\right)  \subseteq\left\{  1,4\right\}  $ are $\left(
5\right)  $, $\left(  1,4\right)  $, $\left(  4,1\right)  $ and $\left(
1,3,1\right)  $). This simplifies to%
\[
\eta_{\left(  1,3,1\right)  }^{\left(  q\right)  }=rM_{\left(  5\right)
}+r^{2}M_{\left(  1,4\right)  }+r^{2}M_{\left(  4,1\right)  }+r^{3}M_{\left(
1,3,1\right)  }.
\]

\item[\textbf{(b)}] For any positive integer $n$, we have
\[
\eta_{\left(  n\right)  }^{\left(  q\right)  }=rM_{\left(  n\right)  },
\]
because the only composition $\beta\in\operatorname*{Comp}\nolimits_{n}$
satisfying $D\left(  \beta\right)  \subseteq D\left(  n\right)  $ is the
composition $\left(  n\right)  $ itself (since $D\left(  n\right)  $ is the
empty set $\varnothing$) and has length $\ell\left(  n\right)  =1$. Likewise,
the empty composition $\varnothing=\left(  {}\right)  $ satisfies
\[
\eta_{\varnothing}^{\left(  q\right)  }=M_{\varnothing}=1.
\]

\end{enumerate}
\end{example}

The quasisymmetric function $\eta_{\alpha}^{\left(  q\right)  }$ generalizes
several known power series. For $q=0$, the series $\eta_{\alpha}^{\left(
q\right)  }=\eta_{\alpha}^{\left(  0\right)  }$ is the \textquotedblleft
essential quasi-symmetric function\textquotedblright\ $E_{I}$ (for $I=D\left(
\alpha\right)  $) defined in \cite[(8)]{Hoffma15}. When $\alpha$ is an odd
composition (i.e., all entries of $\alpha$ are odd) and $q=1$, the series
$\eta_{\alpha}^{\left(  q\right)  }=\eta_{\alpha}^{\left(  1\right)  }$ is
precisely the $\eta_{\alpha}$ defined in \cite[(6.1)]{AgBeSo14}, and differs
only in sign from the $\eta_{\alpha}$ given in \cite[(2.1)]{Hsiao07} (because
of \cite[Proposition 2.1]{Hsiao07}). (This is the reason for the notation
$\eta_{\alpha}^{\left(  q\right)  }$.) Finally, in an appropriate sense, we
can view $M_{\alpha}$ as the \textquotedblleft$q\rightarrow\infty$
limit\textquotedblright\ of $\eta_{\alpha}^{\left(  q\right)  }$; to be
precise, this is saying that when $\eta_{\alpha}^{\left(  q\right)  }$ is
considered as a polynomial in $q$ (over $\operatorname*{QSym}$), its leading
term is $q^{\ell\left(  \alpha\right)  }M_{\alpha}$ (which is obvious from
(\ref{eq.def.etaalpha.def}) and $r=q+1$).

The following two propositions are essentially restatements of
(\ref{eq.def.etaalpha.def}):

\begin{proposition}
\label{prop.eta.through-x}Let $n\in\mathbb{N}$ and $\alpha\in
\operatorname*{Comp}\nolimits_{n}$. Then,%
\begin{equation}
\eta_{\alpha}^{\left(  q\right)  }=\sum_{\substack{g_{1}\leq g_{2}\leq
\cdots\leq g_{n};\\g_{i}=g_{i+1}\text{ for each }i\in\left[  n-1\right]
\setminus D\left(  \alpha\right)  }}r^{\left\vert \left\{  g_{1},g_{2}%
,\ldots,g_{n}\right\}  \right\vert }x_{g_{1}}x_{g_{2}}\cdots x_{g_{n}},
\label{eq.prop.eta.through-x.eq}%
\end{equation}
where the sum is over all weakly increasing $n$-tuples $\left(  g_{1}\leq
g_{2}\leq\cdots\leq g_{n}\right)  $ of positive integers that satisfy $\left(
g_{i}=g_{i+1}\text{ for each }i\in\left[  n-1\right]  \setminus D\left(
\alpha\right)  \right)  $.
\end{proposition}

\begin{vershort}

\begin{proof}
If $\mathbf{g}=\left(  g_{1}\leq g_{2}\leq\cdots\leq g_{n}\right)  $ is a
weakly increasing $n$-tuple of positive integers, then we let
$\operatorname*{Asc}\mathbf{g}$ denote the set of all $i\in\left[  n-1\right]
$ satisfying $g_{i}<g_{i+1}$ (or, equivalently, $g_{i}\neq g_{i+1}$). Thus,
$\operatorname*{Asc}\mathbf{g}$ is the set of all positions at which the
entries of $\mathbf{g}$ get larger (or, equivalently, change values). For
instance, $\operatorname*{Asc}\left(  3,3,4,6,8,8,8\right)  =\left\{
2,3,4\right\}  $.

Let $\beta\in\operatorname*{Comp}\nolimits_{n}$. Write $\beta$ as
$\beta=\left(  \beta_{1},\beta_{2},\ldots,\beta_{\ell}\right)  $; thus,
$\ell\left(  \beta\right)  =\ell$ and $\beta_{1}+\beta_{2}+\cdots+\beta_{\ell
}=n$ and%
\begin{equation}
M_{\beta}=\sum_{i_{1}<i_{2}<\cdots<i_{\ell}}x_{i_{1}}^{\beta_{1}}x_{i_{2}%
}^{\beta_{2}}\cdots x_{i_{\ell}}^{\beta_{\ell}}
\label{pf.prop.eta.through-x.Mbeta=}%
\end{equation}
(by the definition of $M_{\beta}$). Moreover, the definition of $D\left(
\beta\right)  $ yields that
\begin{align*}
D\left(  \beta\right)   &  =\left\{  \beta_{1},\ \ \beta_{1}+\beta
_{2},\ \ \beta_{1}+\beta_{2}+\beta_{3},\ \ \ldots,\ \ \beta_{1}+\beta
_{2}+\cdots+\beta_{\ell-1}\right\} \\
&  =\left\{  \beta_{1}<\beta_{1}+\beta_{2}<\beta_{1}+\beta_{2}+\beta
_{3}<\cdots<\beta_{1}+\beta_{2}+\cdots+\beta_{\ell-1}\right\}  ;
\end{align*}
thus, the elements of $D\left(  \beta\right)  $ subdivide the interval
$\left\{  1,2,\ldots,n\right\}  $ into $\ell$ subintervals of sizes $\beta
_{1},\beta_{2},\ldots,\beta_{\ell}$ (from left to right).

A weakly increasing $n$-tuple $\mathbf{g}=\left(  g_{1}\leq g_{2}\leq
\cdots\leq g_{n}\right)  $ satisfies $\operatorname*{Asc}\mathbf{g}=D\left(
\beta\right)  $ if and only if its entries are constant on each of these
$\ell$ subintervals (i.e., we must have $g_{i}=g_{j}$ whenever $i$ and $j$
belong to the same subinterval) but strictly increase as we pass from one
subinterval to the next (since $\operatorname*{Asc}\mathbf{g}$ is the set of
all positions at which the entries of $\mathbf{g}$ change values). Hence, a
weakly increasing $n$-tuple $\mathbf{g}=\left(  g_{1}\leq g_{2}\leq\cdots\leq
g_{n}\right)  $ satisfies $\operatorname*{Asc}\mathbf{g}=D\left(
\beta\right)  $ if and only if it begins with $\beta_{1}$ copies of a positive
integer $i_{1}$, then continues with $\beta_{2}$ copies of a larger positive
integer $i_{2}$, then continues further with $\beta_{3}$ copies of a
yet-larger positive integer $i_{3}$, and so on. Consequently, for any weakly
increasing $n$-tuple $\mathbf{g}=\left(  g_{1}\leq g_{2}\leq\cdots\leq
g_{n}\right)  $ that satisfies $\operatorname*{Asc}\mathbf{g}=D\left(
\beta\right)  $, we can rewrite the monomial $x_{g_{1}}x_{g_{2}}\cdots
x_{g_{n}}$ as $x_{i_{1}}^{\beta_{1}}x_{i_{2}}^{\beta_{2}}\cdots x_{i_{\ell}%
}^{\beta_{\ell}}$ using these positive integers $i_{1}<i_{2}<\cdots<i_{\ell}$.

Conversely, any monomial of the form $x_{i_{1}}^{\beta_{1}}x_{i_{2}}%
^{\beta_{2}}\cdots x_{i_{\ell}}^{\beta_{\ell}}$ with $i_{1}<i_{2}%
<\cdots<i_{\ell}$ can be rewritten in the form $x_{g_{1}}x_{g_{2}}\cdots
x_{g_{n}}$ for a unique weakly increasing $n$-tuple $\mathbf{g}=\left(
g_{1}\leq g_{2}\leq\cdots\leq g_{n}\right)  $ that satisfies
$\operatorname*{Asc}\mathbf{g}=D\left(  \beta\right)  $ (indeed, it has degree
$\beta_{1}+\beta_{2}+\cdots+\beta_{\ell}=n$, and so can be written in the form
$x_{g_{1}}x_{g_{2}}\cdots x_{g_{n}}$ for a unique weakly increasing $n$-tuple
$\mathbf{g}=\left(  g_{1}\leq g_{2}\leq\cdots\leq g_{n}\right)  $; but then,
the property $\operatorname*{Asc}\mathbf{g}=D\left(  \beta\right)  $ follows
from the arguments in the previous paragraph).

Combining the results of the previous two paragraphs, we conclude that the
monomials $x_{i_{1}}^{\beta_{1}}x_{i_{2}}^{\beta_{2}}\cdots x_{i_{\ell}%
}^{\beta_{\ell}}$ that appear on the right hand side of
(\ref{pf.prop.eta.through-x.Mbeta=}) are precisely the monomials $x_{g_{1}%
}x_{g_{2}}\cdots x_{g_{n}}$ for all weakly increasing $n$-tuples
$\mathbf{g}=\left(  g_{1}\leq g_{2}\leq\cdots\leq g_{n}\right)  $ satisfying
$\operatorname*{Asc}\mathbf{g}=D\left(  \beta\right)  $. Therefore, we can
rewrite (\ref{pf.prop.eta.through-x.Mbeta=}) as%
\begin{equation}
M_{\beta}=\sum_{\substack{\mathbf{g}=\left(  g_{1}\leq g_{2}\leq\cdots\leq
g_{n}\right)  ;\\\operatorname*{Asc}\mathbf{g}=D\left(  \beta\right)
}}x_{g_{1}}x_{g_{2}}\cdots x_{g_{n}}. \label{pf.prop.eta.through-x.Mbeta=2}%
\end{equation}

Let us furthermore observe the following: If a weakly increasing $n$-tuple
$\mathbf{g}=\left(  g_{1}\leq g_{2}\leq\cdots\leq g_{n}\right)  $ satisfies
$\operatorname*{Asc}\mathbf{g}=D\left(  \beta\right)  $, then%
\[
\left\vert \left\{  g_{1},g_{2},\ldots,g_{n}\right\}  \right\vert =\ell
\]
(since we have proved above that the monomial $x_{g_{1}}x_{g_{2}}\cdots
x_{g_{n}}$ can be written as $x_{i_{1}}^{\beta_{1}}x_{i_{2}}^{\beta_{2}}\cdots
x_{i_{\ell}}^{\beta_{\ell}}$ for some integers $i_{1}<i_{2}<\cdots<i_{\ell}$,
and thus contains exactly $\ell$ distinct indeterminates, but this is saying
precisely that the set $\left\{  g_{1},g_{2},\ldots,g_{n}\right\}  $ has
exactly $\ell$ distinct elements) and thus
\begin{equation}
\ell\left(  \beta\right)  =\ell=\left\vert \left\{  g_{1},g_{2},\ldots
,g_{n}\right\}  \right\vert . \label{pf.prop.eta.through-x.len}%
\end{equation}

Now, multiplying the equality (\ref{pf.prop.eta.through-x.Mbeta=2}) by
$r^{\ell\left(  \beta\right)  }$, we find%
\begin{align}
r^{\ell\left(  \beta\right)  }M_{\beta}  &  =\sum_{\substack{\mathbf{g}%
=\left(  g_{1}\leq g_{2}\leq\cdots\leq g_{n}\right)  ;\\\operatorname*{Asc}%
\mathbf{g}=D\left(  \beta\right)  }}\underbrace{r^{\ell\left(  \beta\right)
}}_{\substack{=r^{\left\vert \left\{  g_{1},g_{2},\ldots,g_{n}\right\}
\right\vert }\\\text{(by (\ref{pf.prop.eta.through-x.len}))}}}x_{g_{1}%
}x_{g_{2}}\cdots x_{g_{n}}\nonumber\\
&  =\sum_{\substack{\mathbf{g}=\left(  g_{1}\leq g_{2}\leq\cdots\leq
g_{n}\right)  ;\\\operatorname*{Asc}\mathbf{g}=D\left(  \beta\right)
}}r^{\left\vert \left\{  g_{1},g_{2},\ldots,g_{n}\right\}  \right\vert
}x_{g_{1}}x_{g_{2}}\cdots x_{g_{n}}. \label{pf.prop.eta.through-x.qlM=}%
\end{align}

Forget that we fixed $\beta$. We thus have proved
(\ref{pf.prop.eta.through-x.qlM=}) for each $\beta\in\operatorname*{Comp}%
\nolimits_{n}$. Now, (\ref{eq.def.etaalpha.def}) becomes%
\begin{align*}
\eta_{\alpha}^{\left(  q\right)  }  &  =\sum_{\substack{\beta\in
\operatorname*{Comp}\nolimits_{n};\\D\left(  \beta\right)  \subseteq D\left(
\alpha\right)  }}r^{\ell\left(  \beta\right)  }M_{\beta}\\
&  =\sum_{\substack{\beta\in\operatorname*{Comp}\nolimits_{n};\\D\left(
\beta\right)  \subseteq D\left(  \alpha\right)  }}\ \ \sum
_{\substack{\mathbf{g}=\left(  g_{1}\leq g_{2}\leq\cdots\leq g_{n}\right)
;\\\operatorname*{Asc}\mathbf{g}=D\left(  \beta\right)  }}r^{\left\vert
\left\{  g_{1},g_{2},\ldots,g_{n}\right\}  \right\vert }x_{g_{1}}x_{g_{2}%
}\cdots x_{g_{n}}\ \ \ \ \ \ \ \ \ \ \left(  \text{by
(\ref{pf.prop.eta.through-x.qlM=})}\right) \\
&  =\sum_{\substack{I\subseteq\left[  n-1\right]  ;\\I\subseteq D\left(
\alpha\right)  }}\ \ \sum_{\substack{\mathbf{g}=\left(  g_{1}\leq g_{2}%
\leq\cdots\leq g_{n}\right)  ;\\\operatorname*{Asc}\mathbf{g}=I}}r^{\left\vert
\left\{  g_{1},g_{2},\ldots,g_{n}\right\}  \right\vert }x_{g_{1}}x_{g_{2}%
}\cdots x_{g_{n}}\\
&  \ \ \ \ \ \ \ \ \ \ \ \ \ \ \ \ \ \ \ \ \left(
\begin{array}
[c]{c}%
\text{here, we have substituted }I\text{ for }D\left(  \beta\right)  \text{ in
the first sum,}\\
\text{since the map }D:\operatorname*{Comp}{}_{n}\rightarrow\mathcal{P}\left(
\left[  n-1\right]  \right)  \text{ is a bijection}%
\end{array}
\right) \\
&  =\sum_{\substack{\mathbf{g}=\left(  g_{1}\leq g_{2}\leq\cdots\leq
g_{n}\right)  ;\\\operatorname*{Asc}\mathbf{g}\subseteq D\left(
\alpha\right)  }}r^{\left\vert \left\{  g_{1},g_{2},\ldots,g_{n}\right\}
\right\vert }x_{g_{1}}x_{g_{2}}\cdots x_{g_{n}}\\
&  =\sum_{\substack{\mathbf{g}=\left(  g_{1}\leq g_{2}\leq\cdots\leq
g_{n}\right)  ;\\g_{i}=g_{i+1}\text{ for each }i\in\left[  n-1\right]
\setminus D\left(  \alpha\right)  }}r^{\left\vert \left\{  g_{1},g_{2}%
,\ldots,g_{n}\right\}  \right\vert }x_{g_{1}}x_{g_{2}}\cdots x_{g_{n}}%
\end{align*}
(since the condition \textquotedblleft$\operatorname*{Asc}\mathbf{g}\subseteq
D\left(  \alpha\right)  $\textquotedblright\ imposed on a weakly increasing
$n$-tuple $\mathbf{g}=\left(  g_{1}\leq g_{2}\leq\cdots\leq g_{n}\right)  $ is
equivalent to the condition \textquotedblleft$g_{i}=g_{i+1}$ for each
$i\in\left[  n-1\right]  \setminus D\left(  \alpha\right)  $\textquotedblright%
). This proves Proposition \ref{prop.eta.through-x}.
\end{proof}
\end{vershort}

\begin{verlong}
Our proof of Proposition \ref{prop.eta.through-x} will rely on the following lemma:

\begin{lemma}
\label{lem.eta.through-x.Mlem}Let $n\in\mathbb{N}$ and $\alpha\in
\operatorname*{Comp}\nolimits_{n}$. If $\mathbf{g}=\left(  g_{1}\leq g_{2}%
\leq\cdots\leq g_{n}\right)  $ is a weakly increasing tuple of positive
integers, then we let $\operatorname*{Asc}\mathbf{g}$ denote the set of all
$j\in\left[  n-1\right]  $ satisfying $g_{j}<g_{j+1}$. Then:

\begin{enumerate}
\item[\textbf{(a)}] If $\mathbf{g}=\left(  g_{1}\leq g_{2}\leq\cdots\leq
g_{n}\right)  $ is any weakly increasing $n$-tuple of positive integers that
satisfies $\operatorname*{Asc}\mathbf{g}=D\left(  \alpha\right)  $, then
\[
\ell\left(  \alpha\right)  =\left\vert \left\{  g_{1},g_{2},\ldots
,g_{n}\right\}  \right\vert .
\]

\item[\textbf{(b)}] We have%
\[
M_{\alpha}=\sum_{\substack{\mathbf{g}=\left(  g_{1}\leq g_{2}\leq\cdots\leq
g_{n}\right)  ;\\\operatorname*{Asc}\mathbf{g}=D\left(  \alpha\right)
}}x_{g_{1}}x_{g_{2}}\cdots x_{g_{n}},
\]
where the sum is over all weakly increasing $n$-tuples $\mathbf{g}=\left(
g_{1}\leq g_{2}\leq\cdots\leq g_{n}\right)  $ of positive integers that
satisfy $\operatorname*{Asc}\mathbf{g}=D\left(  \alpha\right)  $.

\item[\textbf{(c)}] We have%
\[
r^{\ell\left(  \alpha\right)  }M_{\alpha}=\sum_{\substack{\mathbf{g}=\left(
g_{1}\leq g_{2}\leq\cdots\leq g_{n}\right)  ;\\\operatorname*{Asc}%
\mathbf{g}=D\left(  \alpha\right)  }}r^{\left\vert \left\{  g_{1},g_{2}%
,\ldots,g_{n}\right\}  \right\vert }x_{g_{1}}x_{g_{2}}\cdots x_{g_{n}}.
\]

\end{enumerate}
\end{lemma}

\begin{proof}
Recall that $\alpha\in\operatorname*{Comp}\nolimits_{n}$; in other words,
$\alpha$ is a composition of $n$. In other words, $\alpha$ is a composition
with $\left\vert \alpha\right\vert =n$. Also, $D\left(  \alpha\right)
\in\mathcal{P}\left(  \left[  n-1\right]  \right)  $ (since $D$ is a map from
$\operatorname*{Comp}\nolimits_{n}$ to $\mathcal{P}\left(  \left[  n-1\right]
\right)  $), so that $D\left(  \alpha\right)  \subseteq\left[  n-1\right]
$.\medskip

\textbf{(a)} Let $\mathbf{g}=\left(  g_{1}\leq g_{2}\leq\cdots\leq
g_{n}\right)  $ be any weakly increasing $n$-tuple of positive integers that
satisfies $\operatorname*{Asc}\mathbf{g}=D\left(  \alpha\right)  $. We claim
the following:

\begin{statement}
\textit{Claim 1:} Let $i\in D\left(  \alpha\right)  $. Then, $\left\vert
\left\{  g_{1},g_{2},\ldots,g_{i+1}\right\}  \right\vert -\left\vert \left\{
g_{1},g_{2},\ldots,g_{i}\right\}  \right\vert =1$.
\end{statement}

[\textit{Proof of Claim 1:} We have $i\in D\left(  \alpha\right)
=\operatorname*{Asc}\mathbf{g}$ (since $\operatorname*{Asc}\mathbf{g}=D\left(
\alpha\right)  $). But $\operatorname*{Asc}\mathbf{g}$ was defined as the set
of all $j\in\left[  n-1\right]  $ satisfying $g_{j}<g_{j+1}$. Hence, $i$ is
such a $j$ (since $i\in\operatorname*{Asc}\mathbf{g}$). In other words,
$i\in\left[  n-1\right]  $ and $g_{i}<g_{i+1}$.

Every $j\in\left[  i\right]  $ satisfies $j\leq i$ and therefore%
\begin{align*}
g_{j}  &  \leq g_{i}\ \ \ \ \ \ \ \ \ \ \left(  \text{since }g_{1}\leq
g_{2}\leq\cdots\leq g_{n}\right) \\
&  <g_{i+1},
\end{align*}
so that $g_{i+1}>g_{j}$ and thus $g_{i+1}\neq g_{j}$. In other words,
$g_{i+1}$ is distinct from all the numbers $g_{1},g_{2},\ldots,g_{i}$. In
other words, $g_{i+1}\notin\left\{  g_{1},g_{2},\ldots,g_{i}\right\}  $.

However, if $A$ is a finite set, and if $b$ is an object such that $b\notin
A$, then $\left\vert A\cup\left\{  b\right\}  \right\vert =\left\vert
A\right\vert +1$. Applying this to $A=\left\{  g_{1},g_{2},\ldots
,g_{i}\right\}  $ and $b=g_{i+1}$, we obtain
\[
\left\vert \left\{  g_{1},g_{2},\ldots,g_{i}\right\}  \cup\left\{
g_{i+1}\right\}  \right\vert =\left\vert \left\{  g_{1},g_{2},\ldots
,g_{i}\right\}  \right\vert +1
\]
(since $g_{i+1}\notin\left\{  g_{1},g_{2},\ldots,g_{i}\right\}  $). In view
of
\[
\left\{  g_{1},g_{2},\ldots,g_{i}\right\}  \cup\left\{  g_{i+1}\right\}
=\left\{  g_{1},g_{2},\ldots,g_{i},g_{i+1}\right\}  =\left\{  g_{1}%
,g_{2},\ldots,g_{i+1}\right\}  ,
\]
we can rewrite this as%
\[
\left\vert \left\{  g_{1},g_{2},\ldots,g_{i+1}\right\}  \right\vert
=\left\vert \left\{  g_{1},g_{2},\ldots,g_{i}\right\}  \right\vert +1.
\]
In other words, $\left\vert \left\{  g_{1},g_{2},\ldots,g_{i+1}\right\}
\right\vert -\left\vert \left\{  g_{1},g_{2},\ldots,g_{i}\right\}  \right\vert
=1$. This proves Claim 1.]

\begin{statement}
\textit{Claim 2:} Let $i\in\left[  n-1\right]  $ be such that $i\notin
D\left(  \alpha\right)  $. Then, $\left\vert \left\{  g_{1},g_{2}%
,\ldots,g_{i+1}\right\}  \right\vert -\left\vert \left\{  g_{1},g_{2}%
,\ldots,g_{i}\right\}  \right\vert =0$.
\end{statement}

[\textit{Proof of Claim 2:} We have $i\notin D\left(  \alpha\right)
=\operatorname*{Asc}\mathbf{g}$ (since $\operatorname*{Asc}\mathbf{g}=D\left(
\alpha\right)  $). But $\operatorname*{Asc}\mathbf{g}$ was defined as the set
of all $j\in\left[  n-1\right]  $ satisfying $g_{j}<g_{j+1}$. Hence, $i$ is
not such a $j$ (since $i\notin\operatorname*{Asc}\mathbf{g}$). Thus, $i$ does
not satisfy $g_{i}<g_{i+1}$ (because if $i$ satisfied $g_{i}<g_{i+1}$, then
$i$ would be a $j\in\left[  n-1\right]  $ satisfying $g_{j}<g_{j+1}$, but this
would contradict the previous sentence). In other words, we have $g_{i}\geq
g_{i+1}$.

However, from $g_{1}\leq g_{2}\leq\cdots\leq g_{n}$, we obtain $g_{i}\leq
g_{i+1}$. Combining this with $g_{i}\geq g_{i+1}$, we find $g_{i}=g_{i+1}$.
Thus, $g_{i+1}=g_{i}\in\left\{  g_{1},g_{2},\ldots,g_{i}\right\}  $ (since
$i\geq1$).

However, if $A$ is a finite set, and if $b\in A$, then $A\cup\left\{
b\right\}  =A$. Applying this to $A=\left\{  g_{1},g_{2},\ldots,g_{i}\right\}
$ and $b=g_{i+1}$, we obtain
\[
\left\{  g_{1},g_{2},\ldots,g_{i}\right\}  \cup\left\{  g_{i+1}\right\}
=\left\{  g_{1},g_{2},\ldots,g_{i}\right\}
\]
(since $g_{i+1}\in\left\{  g_{1},g_{2},\ldots,g_{i}\right\}  $). In view of
\[
\left\{  g_{1},g_{2},\ldots,g_{i}\right\}  \cup\left\{  g_{i+1}\right\}
=\left\{  g_{1},g_{2},\ldots,g_{i},g_{i+1}\right\}  =\left\{  g_{1}%
,g_{2},\ldots,g_{i+1}\right\}  ,
\]
we can rewrite this as%
\[
\left\{  g_{1},g_{2},\ldots,g_{i+1}\right\}  =\left\{  g_{1},g_{2}%
,\ldots,g_{i}\right\}  .
\]
Hence, $\left\vert \left\{  g_{1},g_{2},\ldots,g_{i+1}\right\}  \right\vert
=\left\vert \left\{  g_{1},g_{2},\ldots,g_{i}\right\}  \right\vert $. In other
words, $\left\vert \left\{  g_{1},g_{2},\ldots,g_{i+1}\right\}  \right\vert
-\left\vert \left\{  g_{1},g_{2},\ldots,g_{i}\right\}  \right\vert =0$. This
proves Claim 2.] \medskip

Now, recall that we must prove that $\ell\left(  \alpha\right)  =\left\vert
\left\{  g_{1},g_{2},\ldots,g_{n}\right\}  \right\vert $. If $n=0$, then this
is easy to check\footnote{\textit{Proof.} Assume that $n=0$. Thus, $\alpha$ is
a composition with $\left\vert \alpha\right\vert =0$ (since $\left\vert
\alpha\right\vert =n=0$). Hence, $\ell\left(  \alpha\right)  =0$ (by
\cite[Proposition 2.4]{comps}). Comparing this with
\[
\left\vert \underbrace{\left\{  g_{1},g_{2},\ldots,g_{n}\right\}
}_{\substack{=\varnothing\\\text{(since }n=0\text{)}}}\right\vert =\left\vert
\varnothing\right\vert =0,
\]
we obtain $\ell\left(  \alpha\right)  =\left\vert \left\{  g_{1},g_{2}%
,\ldots,g_{n}\right\}  \right\vert $, qed.}. Thus, for the rest of this proof,
we WLOG assume that $n\neq0$. Hence, $n\geq1$ (since $n\in\mathbb{N}$). Thus,
$\left\vert \alpha\right\vert =n\geq1>0$. Therefore, \cite[Proposition
2.3]{comps} yields $\left\vert D\left(  \alpha\right)  \right\vert
=\ell\left(  \alpha\right)  -1$.

The well-known \emph{telescope principle} says that any $n$-tuple $\left(
a_{1},a_{2},\ldots,a_{n}\right)  $ of real numbers satisfies%
\[
\sum_{i=1}^{n-1}\left(  a_{i+1}-a_{i}\right)  =a_{n}-a_{1}.
\]
Applying this to $a_{i}=\left\vert \left\{  g_{1},g_{2},\ldots,g_{i}\right\}
\right\vert $, we obtain
\begin{align*}
&  \sum_{i=1}^{n-1}\left(  \left\vert \left\{  g_{1},g_{2},\ldots
,g_{i+1}\right\}  \right\vert -\left\vert \left\{  g_{1},g_{2},\ldots
,g_{i}\right\}  \right\vert \right) \\
&  =\left\vert \left\{  g_{1},g_{2},\ldots,g_{n}\right\}  \right\vert
-\left\vert \underbrace{\left\{  g_{1},g_{2},\ldots,g_{1}\right\}
}_{=\left\{  g_{1}\right\}  }\right\vert \\
&  =\left\vert \left\{  g_{1},g_{2},\ldots,g_{n}\right\}  \right\vert
-\underbrace{\left\vert \left\{  g_{1}\right\}  \right\vert }_{=1}\\
&  =\left\vert \left\{  g_{1},g_{2},\ldots,g_{n}\right\}  \right\vert -1.
\end{align*}
Hence,%
\begin{align*}
\left\vert \left\{  g_{1},g_{2},\ldots,g_{n}\right\}  \right\vert -1  &
=\underbrace{\sum_{i=1}^{n-1}}_{=\sum_{i\in\left[  n-1\right]  }}\left(
\left\vert \left\{  g_{1},g_{2},\ldots,g_{i+1}\right\}  \right\vert
-\left\vert \left\{  g_{1},g_{2},\ldots,g_{i}\right\}  \right\vert \right) \\
&  =\sum_{i\in\left[  n-1\right]  }\left(  \left\vert \left\{  g_{1}%
,g_{2},\ldots,g_{i+1}\right\}  \right\vert -\left\vert \left\{  g_{1}%
,g_{2},\ldots,g_{i}\right\}  \right\vert \right) \\
&  =\sum_{\substack{i\in\left[  n-1\right]  ;\\i\in D\left(  \alpha\right)
}}\underbrace{\left(  \left\vert \left\{  g_{1},g_{2},\ldots,g_{i+1}\right\}
\right\vert -\left\vert \left\{  g_{1},g_{2},\ldots,g_{i}\right\}  \right\vert
\right)  }_{\substack{=1\\\text{(by Claim 1)}}}\\
&  \ \ \ \ \ \ \ \ \ \ +\sum_{\substack{i\in\left[  n-1\right]  ;\\i\notin
D\left(  \alpha\right)  }}\underbrace{\left(  \left\vert \left\{  g_{1}%
,g_{2},\ldots,g_{i+1}\right\}  \right\vert -\left\vert \left\{  g_{1}%
,g_{2},\ldots,g_{i}\right\}  \right\vert \right)  }_{\substack{=0\\\text{(by
Claim 2)}}}\\
&  \ \ \ \ \ \ \ \ \ \ \ \ \ \ \ \ \ \ \ \ \left(
\begin{array}
[c]{c}%
\text{since each }i\in\left[  n-1\right]  \text{ satisfies}\\
\text{either }i\in D\left(  \alpha\right)  \text{ or }i\notin D\left(
\alpha\right) \\
\text{(but not both at the same time)}%
\end{array}
\right) \\
&  =\underbrace{\sum_{\substack{i\in\left[  n-1\right]  ;\\i\in D\left(
\alpha\right)  }}}_{\substack{=\sum_{i\in D\left(  \alpha\right)
}\\\text{(since }D\left(  \alpha\right)  \subseteq\left[  n-1\right]
\text{)}}}1+\underbrace{\sum_{\substack{i\in\left[  n-1\right]  ;\\i\notin
D\left(  \alpha\right)  }}0}_{=0}=\sum_{i\in D\left(  \alpha\right)  }1\\
&  =\left\vert D\left(  \alpha\right)  \right\vert \cdot1=\left\vert D\left(
\alpha\right)  \right\vert =\ell\left(  \alpha\right)  -1.
\end{align*}
Solving this for $\ell\left(  \alpha\right)  $, we find%
\[
\ell\left(  \alpha\right)  =\left\vert \left\{  g_{1},g_{2},\ldots
,g_{n}\right\}  \right\vert -1+1=\left\vert \left\{  g_{1},g_{2},\ldots
,g_{n}\right\}  \right\vert .
\]
This proves Lemma \ref{lem.eta.through-x.Mlem} \textbf{(a)}. \medskip

\textbf{(b)} From \cite[detailed version, Proposition 10.10]{Grinbe15}, we
know that
\begin{align*}
M_{\alpha}  &  =\sum_{\substack{i_{1}\leq i_{2}\leq\cdots\leq i_{n};\\\left\{
j\in\left[  n-1\right]  \ \mid\ i_{j}<i_{j+1}\right\}  =D\left(
\alpha\right)  }}x_{i_{1}}x_{i_{2}}\cdots x_{i_{n}}\\
&  =\sum_{\substack{g_{1}\leq g_{2}\leq\cdots\leq g_{n};\\\left\{  j\in\left[
n-1\right]  \ \mid\ g_{j}<g_{j+1}\right\}  =D\left(  \alpha\right)  }%
}x_{g_{1}}x_{g_{2}}\cdots x_{g_{n}}%
\end{align*}
(here, we have renamed the summation index $\left(  i_{1},i_{2},\ldots
,i_{n}\right)  $ as $\left(  g_{1},g_{2},\ldots,g_{n}\right)  $). We can
rewrite this further as%
\begin{equation}
M_{\alpha}=\sum_{\substack{\mathbf{g}=\left(  g_{1}\leq g_{2}\leq\cdots\leq
g_{n}\right)  ;\\\left\{  j\in\left[  n-1\right]  \ \mid\ g_{j}<g_{j+1}%
\right\}  =D\left(  \alpha\right)  }}x_{g_{1}}x_{g_{2}}\cdots x_{g_{n}}
\label{pf.lem.eta.through-x.Mlem.1}%
\end{equation}
(here, we have denoted the $n$-tuple $\left(  g_{1},g_{2},\ldots,g_{n}\right)
$ by $\mathbf{g}$).

However, if $\mathbf{g}=\left(  g_{1}\leq g_{2}\leq\cdots\leq g_{n}\right)  $
is any weakly increasing $n$-tuple of positive integers, then%
\begin{equation}
\left\{  j\in\left[  n-1\right]  \ \mid\ g_{j}<g_{j+1}\right\}
=\operatorname*{Asc}\mathbf{g} \label{pf.lem.eta.through-x.Mlem.2}%
\end{equation}
\footnote{\textit{Proof of (\ref{pf.lem.eta.through-x.Mlem.2}):} Let
$\mathbf{g}=\left(  g_{1}\leq g_{2}\leq\cdots\leq g_{n}\right)  $ be any
weakly increasing $n$-tuple of positive integers. The definition of
$\operatorname*{Asc}\mathbf{g}$ says that $\operatorname*{Asc}\mathbf{g}$ is
the set of all $j\in\left[  n-1\right]  $ satisfying $g_{j}<g_{j+1}$. In other
words, $\operatorname*{Asc}\mathbf{g}=\left\{  j\in\left[  n-1\right]
\ \mid\ g_{j}<g_{j+1}\right\}  $. This proves
(\ref{pf.lem.eta.through-x.Mlem.2}).}. Thus, we can rewrite
(\ref{pf.lem.eta.through-x.Mlem.1}) as
\[
M_{\alpha}=\sum_{\substack{\mathbf{g}=\left(  g_{1}\leq g_{2}\leq\cdots\leq
g_{n}\right)  ;\\\operatorname*{Asc}\mathbf{g}=D\left(  \alpha\right)
}}x_{g_{1}}x_{g_{2}}\cdots x_{g_{n}}.
\]
This proves Lemma \ref{lem.eta.through-x.Mlem} \textbf{(b)}. \medskip

\textbf{(c)} From Lemma \ref{lem.eta.through-x.Mlem} \textbf{(b)}, we have%
\[
M_{\alpha}=\sum_{\substack{\mathbf{g}=\left(  g_{1}\leq g_{2}\leq\cdots\leq
g_{n}\right)  ;\\\operatorname*{Asc}\mathbf{g}=D\left(  \alpha\right)
}}x_{g_{1}}x_{g_{2}}\cdots x_{g_{n}}.
\]
Multiplying both sides of this equality by $r^{\ell\left(  \alpha\right)  }$,
we obtain%
\begin{align*}
r^{\ell\left(  \alpha\right)  }M_{\alpha}  &  =r^{\ell\left(  \alpha\right)
}\sum_{\substack{\mathbf{g}=\left(  g_{1}\leq g_{2}\leq\cdots\leq
g_{n}\right)  ;\\\operatorname*{Asc}\mathbf{g}=D\left(  \alpha\right)
}}x_{g_{1}}x_{g_{2}}\cdots x_{g_{n}}\\
&  =\sum_{\substack{\mathbf{g}=\left(  g_{1}\leq g_{2}\leq\cdots\leq
g_{n}\right)  ;\\\operatorname*{Asc}\mathbf{g}=D\left(  \alpha\right)
}}\underbrace{r^{\ell\left(  \alpha\right)  }}_{\substack{=r^{\left\vert
\left\{  g_{1},g_{2},\ldots,g_{n}\right\}  \right\vert }\\\text{(since Lemma
\ref{lem.eta.through-x.Mlem} \textbf{(a)}}\\\text{yields }\ell\left(
\alpha\right)  =\left\vert \left\{  g_{1},g_{2},\ldots,g_{n}\right\}
\right\vert \text{)}}}x_{g_{1}}x_{g_{2}}\cdots x_{g_{n}}\\
&  =\sum_{\substack{\mathbf{g}=\left(  g_{1}\leq g_{2}\leq\cdots\leq
g_{n}\right)  ;\\\operatorname*{Asc}\mathbf{g}=D\left(  \alpha\right)
}}r^{\left\vert \left\{  g_{1},g_{2},\ldots,g_{n}\right\}  \right\vert
}x_{g_{1}}x_{g_{2}}\cdots x_{g_{n}}.
\end{align*}
This proves Lemma \ref{lem.eta.through-x.Mlem} \textbf{(c)}.
\end{proof}

We can now prove Proposition \ref{prop.eta.through-x}:

\begin{proof}
[Proof of Proposition \ref{prop.eta.through-x}.]If $\mathbf{g}=\left(
g_{1}\leq g_{2}\leq\cdots\leq g_{n}\right)  $ is a weakly increasing tuple of
positive integers, then we let $\operatorname*{Asc}\mathbf{g}$ denote the set
of all $j\in\left[  n-1\right]  $ satisfying $g_{j}<g_{j+1}$.

Now, the equality (\ref{eq.def.etaalpha.def}) becomes%
\begin{align}
\eta_{\alpha}^{\left(  q\right)  }  &  =\sum_{\substack{\beta\in
\operatorname*{Comp}\nolimits_{n};\\D\left(  \beta\right)  \subseteq D\left(
\alpha\right)  }}\underbrace{r^{\ell\left(  \beta\right)  }M_{\beta}%
}_{\substack{=\sum_{\substack{\mathbf{g}=\left(  g_{1}\leq g_{2}\leq\cdots\leq
g_{n}\right)  ;\\\operatorname*{Asc}\mathbf{g}=D\left(  \beta\right)
}}r^{\left\vert \left\{  g_{1},g_{2},\ldots,g_{n}\right\}  \right\vert
}x_{g_{1}}x_{g_{2}}\cdots x_{g_{n}}\\\text{(by Lemma
\ref{lem.eta.through-x.Mlem} \textbf{(c)}, applied to }\beta\text{ instead of
}\alpha\text{)}}}\nonumber\\
&  =\sum_{\substack{\beta\in\operatorname*{Comp}\nolimits_{n};\\D\left(
\beta\right)  \subseteq D\left(  \alpha\right)  }}\ \ \sum
_{\substack{\mathbf{g}=\left(  g_{1}\leq g_{2}\leq\cdots\leq g_{n}\right)
;\\\operatorname*{Asc}\mathbf{g}=D\left(  \beta\right)  }}r^{\left\vert
\left\{  g_{1},g_{2},\ldots,g_{n}\right\}  \right\vert }x_{g_{1}}x_{g_{2}%
}\cdots x_{g_{n}}\nonumber\\
&  =\underbrace{\sum_{\substack{I\in\mathcal{P}\left(  \left[  n-1\right]
\right)  ;\\I\subseteq D\left(  \alpha\right)  }}\ \ \sum
_{\substack{\mathbf{g}=\left(  g_{1}\leq g_{2}\leq\cdots\leq g_{n}\right)
;\\\operatorname*{Asc}\mathbf{g}=I}}}_{\substack{=\sum_{\substack{\mathbf{g}%
=\left(  g_{1}\leq g_{2}\leq\cdots\leq g_{n}\right)  ;\\\operatorname*{Asc}%
\mathbf{g}\subseteq D\left(  \alpha\right)  }}\\\text{(since }%
\operatorname*{Asc}\mathbf{g}\in\mathcal{P}\left(  \left[  n-1\right]
\right)  \\\text{for every }\mathbf{g}=\left(  g_{1}\leq g_{2}\leq\cdots\leq
g_{n}\right)  \text{)}}}r^{\left\vert \left\{  g_{1},g_{2},\ldots
,g_{n}\right\}  \right\vert }x_{g_{1}}x_{g_{2}}\cdots x_{g_{n}}\nonumber\\
&  \ \ \ \ \ \ \ \ \ \ \ \ \ \ \ \ \ \ \ \ \left(
\begin{array}
[c]{c}%
\text{here, we have substituted }I\text{ for }D\left(  \beta\right)  \text{ in
the first sum,}\\
\text{since the map }D:\operatorname*{Comp}\nolimits_{n}\rightarrow
\mathcal{P}\left(  \left[  n-1\right]  \right)  \text{ is a bijection}%
\end{array}
\right) \nonumber\\
&  =\sum_{\substack{\mathbf{g}=\left(  g_{1}\leq g_{2}\leq\cdots\leq
g_{n}\right)  ;\\\operatorname*{Asc}\mathbf{g}\subseteq D\left(
\alpha\right)  }}r^{\left\vert \left\{  g_{1},g_{2},\ldots,g_{n}\right\}
\right\vert }x_{g_{1}}x_{g_{2}}\cdots x_{g_{n}}.
\label{pf.prop.eta.through-x.long.1}%
\end{align}

Now, let $\mathbf{g}=\left(  g_{1}\leq g_{2}\leq\cdots\leq g_{n}\right)  $ be
a weakly increasing $n$-tuple of positive integers. We shall prove that the
statement \textquotedblleft$\operatorname*{Asc}\mathbf{g}\subseteq D\left(
\alpha\right)  $\textquotedblright\ is equivalent to the statement
\textquotedblleft$g_{i}=g_{i+1}$ for each $i\in\left[  n-1\right]  \setminus
D\left(  \alpha\right)  $\textquotedblright. Indeed, it is easy to show that
the former statement implies the latter statement\footnote{\textit{Proof.}
Assume that the former statement (i.e., the statement \textquotedblleft%
$\operatorname*{Asc}\mathbf{g}\subseteq D\left(  \alpha\right)  $%
\textquotedblright) holds. We must prove that the latter statement (i.e., the
statement \textquotedblleft$g_{i}=g_{i+1}$ for each $i\in\left[  n-1\right]
\setminus D\left(  \alpha\right)  $\textquotedblright) also holds.
\par
Let $i\in\left[  n-1\right]  \setminus D\left(  \alpha\right)  $ be arbitrary.
Thus, $i\in\left[  n-1\right]  $ and $i\notin D\left(  \alpha\right)  $. We
have $\operatorname*{Asc}\mathbf{g}\subseteq D\left(  \alpha\right)  $ (since
we assumed that the statement \textquotedblleft$\operatorname*{Asc}%
\mathbf{g}\subseteq D\left(  \alpha\right)  $\textquotedblright\ holds). Thus,
we cannot have $i\in\operatorname*{Asc}\mathbf{g}$ (since $i\in
\operatorname*{Asc}\mathbf{g}$ would imply $i\in\operatorname*{Asc}%
\mathbf{g}\subseteq D\left(  \alpha\right)  $, which would contradict $i\notin
D\left(  \alpha\right)  $).
\par
From $i\in\left[  n-1\right]  $, we obtain $g_{i}\leq g_{i+1}$ (since
$g_{1}\leq g_{2}\leq\cdots\leq g_{n}$). However, if we had $g_{i}<g_{i+1}$,
then $i$ would be a $j\in\left[  n-1\right]  $ satisfying $g_{j}<g_{j+1}$, and
thus would belong to $\operatorname*{Asc}\mathbf{g}$ (since
$\operatorname*{Asc}\mathbf{g}$ is defined as the set of all $j\in\left[
n-1\right]  $ satisfying $g_{j}<g_{j+1}$); but this would contradict the fact
that we cannot have $i\in\operatorname*{Asc}\mathbf{g}$. Thus, we cannot have
$g_{i}<g_{i+1}$. Hence, we must have $g_{i}\geq g_{i+1}$. Combining this with
$g_{i}\leq g_{i+1}$, we find $g_{i}=g_{i+1}$.
\par
Forget that we fixed $i$. We thus have shown that $g_{i}=g_{i+1}$ for each
$i\in\left[  n-1\right]  \setminus D\left(  \alpha\right)  $. In other words,
we have proved that the statement \textquotedblleft$g_{i}=g_{i+1}$ for each
$i\in\left[  n-1\right]  \setminus D\left(  \alpha\right)  $\textquotedblright%
\ holds. Qed.}, and it is also easy to show that the latter statement implies
the former statement\footnote{\textit{Proof.} Assume that the latter statement
(i.e., the statement \textquotedblleft$g_{i}=g_{i+1}$ for each $i\in\left[
n-1\right]  \setminus D\left(  \alpha\right)  $\textquotedblright) holds. We
must prove that the former statement (i.e., the statement \textquotedblleft%
$\operatorname*{Asc}\mathbf{g}\subseteq D\left(  \alpha\right)  $%
\textquotedblright) also holds.
\par
Let $k\in\operatorname*{Asc}\mathbf{g}$. We shall show that $k\in D\left(
\alpha\right)  $. Indeed, assume the contrary. Thus, $k\notin D\left(
\alpha\right)  $. Combining $k\in\operatorname*{Asc}\mathbf{g}\subseteq\left[
n-1\right]  $ with $k\notin D\left(  \alpha\right)  $, we find $k\in\left[
n-1\right]  \setminus D\left(  \alpha\right)  $.
\par
However, we assumed that the statement \textquotedblleft$g_{i}=g_{i+1}$ for
each $i\in\left[  n-1\right]  \setminus D\left(  \alpha\right)  $%
\textquotedblright\ holds. Applying this statement to $i=k$, we obtain
$g_{k}=g_{k+1}$ (since $k\in\left[  n-1\right]  \setminus D\left(
\alpha\right)  $). However, $\operatorname*{Asc}\mathbf{g}$ is defined as the
set of all $j\in\left[  n-1\right]  $ satisfying $g_{j}<g_{j+1}$. Hence, $k$
is such a $j$ (since $k\in\operatorname*{Asc}\mathbf{g}$). In other words,
$k\in\left[  n-1\right]  $ and $g_{k}<g_{k+1}$. But $g_{k}<g_{k+1}$ clearly
contradicts $g_{k}=g_{k+1}$. This contradiction shows that our assumption was
false. Hence, $k\in D\left(  \alpha\right)  $ is proved.
\par
Forget that we fixed $k$. We thus have shown that $k\in D\left(
\alpha\right)  $ for each $k\in\operatorname*{Asc}\mathbf{g}$. In other words,
$\operatorname*{Asc}\mathbf{g}\subseteq D\left(  \alpha\right)  $. Hence, we
have proved that the statement \textquotedblleft$\operatorname*{Asc}%
\mathbf{g}\subseteq D\left(  \alpha\right)  $\textquotedblright\ holds. Qed.}.
Hence, these two statements are equivalent.

Forget that we fixed $\mathbf{g}$. We thus have shown that for every weakly
increasing $n$-tuple $\mathbf{g}=\left(  g_{1}\leq g_{2}\leq\cdots\leq
g_{n}\right)  $ of positive integers, the statement \textquotedblleft%
$\operatorname*{Asc}\mathbf{g}\subseteq D\left(  \alpha\right)  $%
\textquotedblright\ is equivalent to the statement \textquotedblleft%
$g_{i}=g_{i+1}$ for each $i\in\left[  n-1\right]  \setminus D\left(
\alpha\right)  $\textquotedblright. Therefore, the summation sign%
\[
\sum_{\substack{\mathbf{g}=\left(  g_{1}\leq g_{2}\leq\cdots\leq g_{n}\right)
;\\\operatorname*{Asc}\mathbf{g}\subseteq D\left(  \alpha\right)  }}\text{ can
be rewritten as }\sum_{\substack{\mathbf{g}=\left(  g_{1}\leq g_{2}\leq
\cdots\leq g_{n}\right)  ;\\g_{i}=g_{i+1}\text{ for each }i\in\left[
n-1\right]  \setminus D\left(  \alpha\right)  }}\text{.}%
\]
Hence, we can rewrite (\ref{pf.prop.eta.through-x.long.1}) as%
\begin{align*}
\eta_{\alpha}^{\left(  q\right)  }  &  =\sum_{\substack{\mathbf{g}=\left(
g_{1}\leq g_{2}\leq\cdots\leq g_{n}\right)  ;\\g_{i}=g_{i+1}\text{ for each
}i\in\left[  n-1\right]  \setminus D\left(  \alpha\right)  }}r^{\left\vert
\left\{  g_{1},g_{2},\ldots,g_{n}\right\}  \right\vert }x_{g_{1}}x_{g_{2}%
}\cdots x_{g_{n}}\\
&  =\sum_{\substack{g_{1}\leq g_{2}\leq\cdots\leq g_{n};\\g_{i}=g_{i+1}\text{
for each }i\in\left[  n-1\right]  \setminus D\left(  \alpha\right)
}}r^{\left\vert \left\{  g_{1},g_{2},\ldots,g_{n}\right\}  \right\vert
}x_{g_{1}}x_{g_{2}}\cdots x_{g_{n}}%
\end{align*}
(here, we have removed the unused label $\mathbf{g}$ for the summation index).
This proves Proposition \ref{prop.eta.through-x}.
\end{proof}
\end{verlong}

\begin{proposition}
\label{prop.eta.through-x2}Let $\alpha=\left(  \alpha_{1},\alpha_{2}%
,\ldots,\alpha_{\ell}\right)  \in\operatorname*{Comp}$. Then,%
\begin{equation}
\eta_{\alpha}^{\left(  q\right)  }=\sum_{i_{1}\leq i_{2}\leq\cdots\leq
i_{\ell}}r^{\left\vert \left\{  i_{1},i_{2},\ldots,i_{\ell}\right\}
\right\vert }x_{i_{1}}^{\alpha_{1}}x_{i_{2}}^{\alpha_{2}}\cdots x_{i_{\ell}%
}^{\alpha_{\ell}}, \label{eq.prop.eta.through-x2.eq}%
\end{equation}
where the sum is over all weakly increasing $\ell$-tuples $\left(  i_{1}\leq
i_{2}\leq\cdots\leq i_{\ell}\right)  $ of positive integers.
\end{proposition}

\begin{vershort}

\begin{proof}
Let $n=\left\vert \alpha\right\vert $, so that $\alpha\in\operatorname*{Comp}%
\nolimits_{n}$. Clearly, it suffices to show that the right hand sides of
(\ref{eq.prop.eta.through-x.eq}) and (\ref{eq.prop.eta.through-x2.eq}) are identical.

We claim that the monomials $x_{g_{1}}x_{g_{2}}\cdots x_{g_{n}}$ for all
$n$-tuples\footnote{We are no longer saying ``weakly increasing'', since this
is automatically implied by the notation $\left(  g_{1}\leq g_{2}\leq
\cdots\leq g_{n}\right)  $.} $\left(  g_{1}\leq g_{2}\leq\cdots\leq
g_{n}\right)  $ satisfying $\left(  g_{i}=g_{i+1}\text{ for each }i\in\left[
n-1\right]  \setminus D\left(  \alpha\right)  \right)  $ are precisely the
monomials of the form $x_{i_{1}}^{\alpha_{1}}x_{i_{2}}^{\alpha_{2}}\cdots
x_{i_{\ell}}^{\alpha_{\ell}}$ for all $\ell$-tuples $\left(  i_{1}\leq
i_{2}\leq\cdots\leq i_{\ell}\right)  $.

Indeed, let us set%
\[
s_{k}:=\alpha_{1}+\alpha_{2}+\cdots+\alpha_{k}\ \ \ \ \ \ \ \ \ \ \text{for
each }k\in\left[  \ell\right]  .
\]
Thus, $D\left(  \alpha\right)  =\left\{  s_{1},s_{2},\ldots,s_{\ell
-1}\right\}  $ and $s_{\ell}=\left\vert \alpha\right\vert =n$. Hence, an
$n$-tuple $\left(  g_{1}\leq g_{2}\leq\cdots\leq g_{n}\right)  $ satisfies%
\[
\left(  g_{i}=g_{i+1}\text{ for each }i\in\left[  n-1\right]  \setminus
D\left(  \alpha\right)  \right)
\]
if and only if it satisfies%
\begin{align*}
&  \ g_{1}=g_{2}=\cdots=g_{s_{1}}\\
\leq &  \ g_{s_{1}+1}=g_{s_{1}+2}=\cdots=g_{s_{2}}\\
\leq &  \ g_{s_{2}+1}=g_{s_{2}+2}=\cdots=g_{s_{3}}\\
\leq &  \ \cdots\\
\leq &  \ g_{s_{\ell-1}+1}=g_{s_{\ell-1}+2}=\cdots=g_{s_{\ell}}%
\end{align*}
(that is, if its entries can grow only at the positions $s_{1},s_{2}%
,\ldots,s_{\ell-1}$). Therefore, such an $n$-tuple is uniquely determined by
its \textquotedblleft essential\textquotedblright\ entries $g_{s_{1}}%
,g_{s_{2}},\ldots,g_{s_{\ell}}$. More precisely, there is a bijection from the
set%
\[
\left\{  n\text{-tuples }\left(  g_{1}\leq g_{2}\leq\cdots\leq g_{n}\right)
\text{ satisfying }\left(  g_{i}=g_{i+1}\text{ for each }i\in\left[
n-1\right]  \setminus D\left(  \alpha\right)  \right)  \right\}
\]
to the set%
\[
\left\{  \ell\text{-tuples }\left(  i_{1}\leq i_{2}\leq\cdots\leq i_{\ell
}\right)  \right\}  ,
\]
which sends each $n$-tuple $\left(  g_{1}\leq g_{2}\leq\cdots\leq
g_{n}\right)  $ from the former set to its \textquotedblleft
subword\textquotedblright\ $\left(  g_{s_{1}}\leq g_{s_{2}}\leq\cdots\leq
g_{s_{\ell}}\right)  $\ \ \ \ \footnote{This bijection is precisely the
bijection $\Phi$ from \cite[detailed version, proof of Proposition
10.69]{Grinbe15}. Indeed, our set%
\[
\left\{  n\text{-tuples }\left(  g_{1}\leq g_{2}\leq\cdots\leq g_{n}\right)
\text{ satisfying }\left(  g_{i}=g_{i+1}\text{ for each }i\in\left[
n-1\right]  \setminus D\left(  \alpha\right)  \right)  \right\}
\]
is the set $\mathcal{I}$ from \cite[detailed version, proof of Proposition
10.69]{Grinbe15} (since the condition \textquotedblleft$g_{i}=g_{i+1}$ for
each $i\in\left[  n-1\right]  \setminus D\left(  \alpha\right)  $%
\textquotedblright\ can be rewritten as \textquotedblleft$\left\{  j\in\left[
n-1\right]  \ \mid\ g_{j}<g_{j+1}\right\}  \subseteq D\left(  \alpha\right)
$\textquotedblright), whereas our set%
\[
\left\{  \ell\text{-tuples }\left(  i_{1}\leq i_{2}\leq\cdots\leq i_{\ell
}\right)  \right\}
\]
is the set $\mathcal{J}$ from \cite[detailed version, proof of Proposition
10.69]{Grinbe15}.}. This bijection has the property that
\[
x_{g_{1}}x_{g_{2}}\cdots x_{g_{n}}=x_{i_{1}}^{\alpha_{1}}x_{i_{2}}^{\alpha
_{2}}\cdots x_{i_{\ell}}^{\alpha_{\ell}}\ \ \ \ \ \ \ \ \ \ \text{and}%
\ \ \ \ \ \ \ \ \ \ \left\{  g_{1},g_{2},\ldots,g_{n}\right\}  =\left\{
i_{1},i_{2},\ldots,i_{\ell}\right\}
\]
whenever it sends an $n$-tuple $\left(  g_{1}\leq g_{2}\leq\cdots\leq
g_{n}\right)  $ to the $\ell$-tuple $\left(  i_{1}\leq i_{2}\leq\cdots\leq
i_{\ell}\right)  =\left(  g_{s_{1}}\leq g_{s_{2}}\leq\cdots\leq g_{s_{\ell}%
}\right)  $ (because the $n$-tuple $\left(  g_{1}\leq g_{2}\leq\cdots\leq
g_{n}\right)  $ consists of $\alpha_{1}$ copies of $i_{1}$, followed by
$\alpha_{2}$ copies of $i_{2}$, followed by $\alpha_{3}$ copies of $i_{3}$,
and so on). This shows that the right hand sides of the equalities
(\ref{eq.prop.eta.through-x.eq}) and (\ref{eq.prop.eta.through-x2.eq}) are
identical. Hence, (\ref{eq.prop.eta.through-x2.eq}) follows from
(\ref{eq.prop.eta.through-x.eq}). This proves Proposition
\ref{prop.eta.through-x2}.
\end{proof}
\end{vershort}

\begin{verlong}

\begin{proof}
Let $n=\left\vert \alpha\right\vert $. Thus, $\alpha$ is a composition of $n$.
In other words, $\alpha\in\operatorname*{Comp}\nolimits_{n}$.

If $\mathbf{g}=\left(  g_{1}\leq g_{2}\leq\cdots\leq g_{n}\right)  $ is a
weakly increasing tuple of positive integers, then we let $\operatorname*{Asc}%
\mathbf{g}$ denote the set of all $j\in\left[  n-1\right]  $ satisfying
$g_{j}<g_{j+1}$. Thus, if $\mathbf{g}=\left(  g_{1}\leq g_{2}\leq\cdots\leq
g_{n}\right)  $ is a weakly increasing tuple of positive integers, then
\begin{equation}
\operatorname*{Asc}\mathbf{g}=\left\{  j\in\left[  n-1\right]  \ \mid
\ g_{j}<g_{j+1}\right\}  . \label{pf.prop.eta.through-x2.long.Asc=}%
\end{equation}

Now, the equality (\ref{pf.prop.eta.through-x.long.1}) (which we proved above)
yields%
\begin{equation}
\eta_{\alpha}^{\left(  q\right)  }=\sum_{\substack{\mathbf{g}=\left(
g_{1}\leq g_{2}\leq\cdots\leq g_{n}\right)  ;\\\operatorname*{Asc}%
\mathbf{g}\subseteq D\left(  \alpha\right)  }}r^{\left\vert \left\{
g_{1},g_{2},\ldots,g_{n}\right\}  \right\vert }x_{g_{1}}x_{g_{2}}\cdots
x_{g_{n}}. \label{pf.prop.eta.through-x2.long.1}%
\end{equation}

Let $\mathcal{J}$ denote the set of all length-$\ell$ weakly increasing
sequences of positive integers. In other words,%
\[
\mathcal{J}=\left\{  \left(  i_{1},i_{2},\ldots,i_{\ell}\right)  \in\left\{
1,2,3,\ldots\right\}  ^{\ell}\ \mid\ i_{1}\leq i_{2}\leq\cdots\leq i_{\ell
}\right\}  .
\]

Define a set $\mathcal{I}$ by%
\begin{align*}
\mathcal{I}  &  =\left\{  \left(  i_{1},i_{2},\ldots,i_{n}\right)  \in\left\{
1,2,3,\ldots\right\}  ^{n}\ \mid\ i_{1}\leq i_{2}\leq\cdots\leq i_{n}\right.
\\
&  \ \ \ \ \ \ \ \ \ \ \left.  \text{and }\left\{  j\in\left[  n-1\right]
\ \mid\ i_{j}<i_{j+1}\right\}  \subseteq D\left(  \alpha\right)  \right\}  .
\end{align*}
Thus, we have the following equality between summation signs:%
\begin{equation}
\sum_{\substack{i_{1}\leq i_{2}\leq\cdots\leq i_{n};\\\left\{  j\in\left[
n-1\right]  \ \mid\ i_{j}<i_{j+1}\right\}  \subseteq D\left(  \alpha\right)
}}=\sum_{\left(  i_{1},i_{2},\ldots,i_{n}\right)  \in\mathcal{I}}.
\label{pf.prop.eta.through-x2.long.fact1}%
\end{equation}

For every $i\in\left\{  0,1,\ldots,\ell\right\}  $, define a nonnegative
integer $s_{i}$ by
\[
s_{i}=\alpha_{1}+\alpha_{2}+\cdots+\alpha_{i}.
\]

The following facts have been proved in \cite[detailed version, proof of
Proposition 10.69]{Grinbe15}:

\begin{enumerate}
\item The map
\[
\mathcal{I}\rightarrow\mathcal{J},\ \ \ \ \ \ \ \ \ \ \left(  i_{1}%
,i_{2},\ldots,i_{n}\right)  \mapsto\left(  i_{s_{1}},i_{s_{2}},\ldots
,i_{s_{\ell}}\right)
\]
is well-defined and is a bijection.

\item For every $\left(  i_{1},i_{2},\ldots,i_{n}\right)  \in\mathcal{I}$, we
have
\begin{equation}
x_{i_{1}}x_{i_{2}}\cdots x_{i_{n}}=x_{i_{s_{1}}}^{\alpha_{1}}x_{i_{s_{2}}%
}^{\alpha_{2}}\cdots x_{i_{s_{\ell}}}^{\alpha_{\ell}}.
\label{pf.prop.eta.through-x2.long.fact3}%
\end{equation}

\end{enumerate}

Furthermore, it is easy to see that every $\left(  i_{1},i_{2},\ldots
,i_{n}\right)  \in\mathcal{I}$ satisfies%
\begin{equation}
\left\{  i_{s_{1}},i_{s_{2}},\ldots,i_{s_{\ell}}\right\}  =\left\{
i_{1},i_{2},\ldots,i_{n}\right\}  \label{pf.prop.eta.through-x2.long.set}%
\end{equation}
\footnote{\textit{Proof of (\ref{pf.prop.eta.through-x2.long.set}):} Let
$\left(  i_{1},i_{2},\ldots,i_{n}\right)  \in\mathcal{I}$. Let $j\in\left[
n\right]  $.
\par
Define a map $f:\left[  n\right]  \rightarrow\left[  \ell\right]  $ as in
\cite[detailed version, Lemma 10.7]{Grinbe15}. Then, the equality
\cite[detailed version, (101)]{Grinbe15} (which is proved in \cite{Grinbe15})
says that%
\[
i_{s_{f\left(  k\right)  }}=i_{k}\ \ \ \ \ \ \ \ \ \ \text{for every }%
k\in\left[  n\right]  .
\]
Applying this to $k=j$, we find $i_{s_{f\left(  j\right)  }}=i_{j}$. Hence,%
\[
i_{j}=i_{s_{f\left(  j\right)  }}\in\left\{  i_{s_{1}},i_{s_{2}}%
,\ldots,i_{s_{\ell}}\right\}  \ \ \ \ \ \ \ \ \ \ \left(  \text{since
}f\left(  j\right)  \in\left[  \ell\right]  =\left\{  1,2,\ldots,\ell\right\}
\right)  .
\]
\par
Forget that we fixed $j$. We thus have shown that $i_{j}\in\left\{  i_{s_{1}%
},i_{s_{2}},\ldots,i_{s_{\ell}}\right\}  $ for each $j\in\left[  n\right]  $.
In other words, all $n$ elements $i_{1},i_{2},\ldots,i_{n}$ belong to
$\left\{  i_{s_{1}},i_{s_{2}},\ldots,i_{s_{\ell}}\right\}  $. In other words,
$\left\{  i_{1},i_{2},\ldots,i_{n}\right\}  \subseteq\left\{  i_{s_{1}%
},i_{s_{2}},\ldots,i_{s_{\ell}}\right\}  $. Combining this relation with
$\left\{  i_{s_{1}},i_{s_{2}},\ldots,i_{s_{\ell}}\right\}  \subseteq\left\{
i_{1},i_{2},\ldots,i_{n}\right\}  $ (which is obvious, since each of the
$\ell$ elements $i_{s_{1}},i_{s_{2}},\ldots,i_{s_{\ell}}$ belongs to $\left\{
i_{1},i_{2},\ldots,i_{n}\right\}  $), we obtain $\left\{  i_{s_{1}},i_{s_{2}%
},\ldots,i_{s_{\ell}}\right\}  =\left\{  i_{1},i_{2},\ldots,i_{n}\right\}  $.
Thus, (\ref{pf.prop.eta.through-x2.long.set}) is proved.}.

Now,
\begin{align*}
&  \underbrace{\sum_{i_{1}\leq i_{2}\leq\cdots\leq i_{\ell}}}_{\substack{=\sum
_{\substack{\left(  i_{1},i_{2},\ldots,i_{\ell}\right)  \in\left\{
1,2,3,\ldots\right\}  ^{\ell};\\i_{1}\leq i_{2}\leq\cdots\leq i_{\ell}}%
}=\sum_{\left(  i_{1},i_{2},\ldots,i_{\ell}\right)  \in\mathcal{J}%
}\\\text{(since }\left\{  \left(  i_{1},i_{2},\ldots,i_{\ell}\right)
\in\left\{  1,2,3,\ldots\right\}  ^{\ell}\ \mid\ i_{1}\leq i_{2}\leq\cdots\leq
i_{\ell}\right\}  =\mathcal{J}\text{)}}}r^{\left\vert \left\{  i_{1}%
,i_{2},\ldots,i_{\ell}\right\}  \right\vert }x_{i_{1}}^{\alpha_{1}}x_{i_{2}%
}^{\alpha_{2}}\cdots x_{i_{\ell}}^{\alpha_{\ell}}\\
&  =\sum_{\left(  i_{1},i_{2},\ldots,i_{\ell}\right)  \in\mathcal{J}%
}r^{\left\vert \left\{  i_{1},i_{2},\ldots,i_{\ell}\right\}  \right\vert
}x_{i_{1}}^{\alpha_{1}}x_{i_{2}}^{\alpha_{2}}\cdots x_{i_{\ell}}^{\alpha
_{\ell}}\\
&  =\sum_{\left(  j_{1},j_{2},\ldots,j_{\ell}\right)  \in\mathcal{J}%
}r^{\left\vert \left\{  j_{1},j_{2},\ldots,j_{\ell}\right\}  \right\vert
}x_{j_{1}}^{\alpha_{1}}x_{j_{2}}^{\alpha_{2}}\cdots x_{j_{\ell}}^{\alpha
_{\ell}}\\
&  \ \ \ \ \ \ \ \ \ \ \ \ \ \ \ \ \ \ \ \ \left(
\begin{array}
[c]{c}%
\text{here, we have renamed the summation}\\
\text{index }\left(  i_{1},i_{2},\ldots,i_{\ell}\right)  \text{ as }\left(
j_{1},j_{2},\ldots,j_{\ell}\right)
\end{array}
\right) \\
&  =\sum_{\left(  i_{1},i_{2},\ldots,i_{n}\right)  \in\mathcal{I}%
}r^{\left\vert \left\{  i_{s_{1}},i_{s_{2}},\ldots,i_{s_{\ell}}\right\}
\right\vert }x_{i_{s_{1}}}^{\alpha_{1}}x_{i_{s_{2}}}^{\alpha_{2}}\cdots
x_{i_{s_{\ell}}}^{\alpha_{\ell}}%
\end{align*}
(here, we have substituted $\left(  i_{s_{1}},i_{s_{2}},\ldots,i_{s_{\ell}%
}\right)  $ for $\left(  j_{1},j_{2},\ldots,j_{\ell}\right)  $ in the sum,
since the map $\mathcal{I}\rightarrow\mathcal{J},\ \left(  i_{1},i_{2}%
,\ldots,i_{n}\right)  \mapsto\left(  i_{s_{1}},i_{s_{2}},\ldots,i_{s_{\ell}%
}\right)  $ is a bijection). Thus,%
\begin{align*}
&  \sum_{i_{1}\leq i_{2}\leq\cdots\leq i_{\ell}}r^{\left\vert \left\{
i_{1},i_{2},\ldots,i_{\ell}\right\}  \right\vert }x_{i_{1}}^{\alpha_{1}%
}x_{i_{2}}^{\alpha_{2}}\cdots x_{i_{\ell}}^{\alpha_{\ell}}\\
&  =\sum_{\left(  i_{1},i_{2},\ldots,i_{n}\right)  \in\mathcal{I}%
}\underbrace{r^{\left\vert \left\{  i_{s_{1}},i_{s_{2}},\ldots,i_{s_{\ell}%
}\right\}  \right\vert }}_{\substack{=r^{\left\vert \left\{  i_{1}%
,i_{2},\ldots,i_{n}\right\}  \right\vert }\\\text{(since
(\ref{pf.prop.eta.through-x2.long.set}) shows}\\\text{that }\left\{  i_{s_{1}%
},i_{s_{2}},\ldots,i_{s_{\ell}}\right\}  =\left\{  i_{1},i_{2},\ldots
,i_{n}\right\}  \text{)}}}\underbrace{x_{i_{s_{1}}}^{\alpha_{1}}x_{i_{s_{2}}%
}^{\alpha_{2}}\cdots x_{i_{s_{\ell}}}^{\alpha_{\ell}}}_{\substack{=x_{i_{1}%
}x_{i_{2}}\cdots x_{i_{n}}\\\text{(by (\ref{pf.prop.eta.through-x2.long.fact3}%
))}}}\\
&  =\underbrace{\sum_{\left(  i_{1},i_{2},\ldots,i_{n}\right)  \in\mathcal{I}%
}}_{\substack{=\sum_{\substack{i_{1}\leq i_{2}\leq\cdots\leq i_{n};\\\left\{
j\in\left[  n-1\right]  \ \mid\ i_{j}<i_{j+1}\right\}  \subseteq D\left(
\alpha\right)  }}\\\text{(by (\ref{pf.prop.eta.through-x2.long.fact1}))}%
}}r^{\left\vert \left\{  i_{1},i_{2},\ldots,i_{n}\right\}  \right\vert
}x_{i_{1}}x_{i_{2}}\cdots x_{i_{n}}\\
&  =\sum_{\substack{i_{1}\leq i_{2}\leq\cdots\leq i_{n};\\\left\{  j\in\left[
n-1\right]  \ \mid\ i_{j}<i_{j+1}\right\}  \subseteq D\left(  \alpha\right)
}}r^{\left\vert \left\{  i_{1},i_{2},\ldots,i_{n}\right\}  \right\vert
}x_{i_{1}}x_{i_{2}}\cdots x_{i_{n}}\\
&  =\sum_{\substack{g_{1}\leq g_{2}\leq\cdots\leq g_{n};\\\left\{  j\in\left[
n-1\right]  \ \mid\ g_{j}<g_{j+1}\right\}  \subseteq D\left(  \alpha\right)
}}r^{\left\vert \left\{  g_{1},g_{2},\ldots,g_{n}\right\}  \right\vert
}x_{g_{1}}x_{g_{2}}\cdots x_{g_{n}}\\
&  \ \ \ \ \ \ \ \ \ \ \ \ \ \ \ \ \ \ \ \ \left(
\begin{array}
[c]{c}%
\text{here, we have renamed the summation}\\
\text{index }\left(  i_{1},i_{2},\ldots,i_{n}\right)  \text{ as }\left(
g_{1},g_{2},\ldots,g_{n}\right)
\end{array}
\right) \\
&  =\underbrace{\sum_{\substack{\mathbf{g}=\left(  g_{1}\leq g_{2}\leq
\cdots\leq g_{n}\right)  ;\\\left\{  j\in\left[  n-1\right]  \ \mid
\ g_{j}<g_{j+1}\right\}  \subseteq D\left(  \alpha\right)  }}}%
_{\substack{=\sum_{\substack{\mathbf{g}=\left(  g_{1}\leq g_{2}\leq\cdots\leq
g_{n}\right)  ;\\\operatorname*{Asc}\mathbf{g}\subseteq D\left(
\alpha\right)  }}\\\text{(since (\ref{pf.prop.eta.through-x2.long.Asc=})
yields}\\\text{that }\left\{  j\in\left[  n-1\right]  \ \mid\ g_{j}%
<g_{j+1}\right\}  =\operatorname*{Asc}\mathbf{g}\\\text{for any }%
\mathbf{g}=\left(  g_{1}\leq g_{2}\leq\cdots\leq g_{n}\right)  \text{)}%
}}r^{\left\vert \left\{  g_{1},g_{2},\ldots,g_{n}\right\}  \right\vert
}x_{g_{1}}x_{g_{2}}\cdots x_{g_{n}}\\
&  \ \ \ \ \ \ \ \ \ \ \ \ \ \ \ \ \ \ \ \ \left(
\begin{array}
[c]{c}%
\text{here, we have introduced the additional}\\
\text{notation }\mathbf{g}\text{ for the summation index }\left(  g_{1}%
,g_{2},\ldots,g_{n}\right)
\end{array}
\right) \\
&  =\sum_{\substack{\mathbf{g}=\left(  g_{1}\leq g_{2}\leq\cdots\leq
g_{n}\right)  ;\\\operatorname*{Asc}\mathbf{g}\subseteq D\left(
\alpha\right)  }}r^{\left\vert \left\{  g_{1},g_{2},\ldots,g_{n}\right\}
\right\vert }x_{g_{1}}x_{g_{2}}\cdots x_{g_{n}}=\eta_{\alpha}^{\left(
q\right)  }%
\end{align*}
(by (\ref{pf.prop.eta.through-x2.long.1})). This proves Proposition
\ref{prop.eta.through-x2}.
\end{proof}
\end{verlong}

\subsection{The $\eta_{\alpha}^{\left(  q\right)  }$ as a basis}

The equality (\ref{eq.def.etaalpha.def}) writes each enriched $q$-monomial
function $\eta_{\alpha}^{\left(  q\right)  }$ as a $\mathbf{k}$-linear
combination of $M_{\beta}$'s. Conversely, we can expand each monomial
quasisymmetric function $M_{\beta}$ as a $\mathbf{k}$-linear combination of
$\eta_{\alpha}^{\left(  q\right)  }$'s, at least after multiplying it by
$r^{\ell\left(  \beta\right)  }$:

\begin{proposition}
\label{prop.eta.M-through-eta}Let $n\in\mathbb{N}$. Let $\beta\in
\operatorname*{Comp}\nolimits_{n}$ be a composition. Then,%
\[
r^{\ell\left(  \beta\right)  }M_{\beta}=\sum_{\substack{\alpha\in
\operatorname*{Comp}\nolimits_{n};\\D\left(  \alpha\right)  \subseteq D\left(
\beta\right)  }}\left(  -1\right)  ^{\ell\left(  \beta\right)  -\ell\left(
\alpha\right)  }\eta_{\alpha}^{\left(  q\right)  }.
\]

\end{proposition}

For the proof of this proposition (and some later ones as well), we will need
the \emph{Iverson bracket notation}:

\begin{convention}
\label{conv.iverson}If $\mathcal{A}$ is a logical statement, then $\left[
\mathcal{A}\right]  $ shall denote the truth value of $\mathcal{A}$ (that is,
the number $1$ if $\mathcal{A}$ is true, and the number $0$ if $\mathcal{A}$
is false).
\end{convention}

For example, $\left[  2+2=4\right]  =1$ and $\left[  2+2=5\right]  =0$.

The following lemma is a classical elementary property of finite sets, but we
recall its proof for the sake of completeness:

\begin{lemma}
\label{lem.altsum.1}Let $S$ and $T$ be two finite sets. Then,%
\[
\sum_{\substack{I\subseteq S;\\T\subseteq I}}\left(  -1\right)  ^{\left\vert
S\right\vert -\left\vert I\right\vert }=\left[  S=T\right]  .
\]

\end{lemma}

\begin{vershort}
\begin{proof}
[Proof.]If $T\not \subseteq S$, then the left hand side is an empty sum and
thus equals $0$, as does the right hand side. Hence, for the rest of this
proof, we WLOG assume that $T\subseteq S$. Hence, $S\setminus T=\varnothing$
holds if and only if $S=T$. Therefore, $\left[  S\setminus T=\varnothing
\right]  =\left[  S=T\right]  $.

Each subset $I$ of $S$ can be uniquely written as $S\setminus J$ for a unique
subset $J$ of $S$ (namely, this $J$ is the complement of $I$ in $S$). Thus, we
can reindex the sum $\sum_{\substack{I\subseteq S;\\T\subseteq I}}\left(
-1\right)  ^{\left\vert S\right\vert -\left\vert I\right\vert }$ by
substituting $S\setminus J$ for $I$. Thus, we obtain%
\begin{align}
\sum_{\substack{I\subseteq S;\\T\subseteq I}}\left(  -1\right)  ^{\left\vert
S\right\vert -\left\vert I\right\vert }  &  =\sum_{\substack{J\subseteq
S;\\T\subseteq S\setminus J}}\underbrace{\left(  -1\right)  ^{\left\vert
S\right\vert -\left\vert S\setminus J\right\vert }}_{\substack{=\left(
-1\right)  ^{\left\vert J\right\vert }\\\text{(since }\left\vert S\right\vert
-\left\vert S\setminus J\right\vert =\left\vert J\right\vert \text{)}}%
}=\sum_{\substack{J\subseteq S;\\T\subseteq S\setminus J}}\left(  -1\right)
^{\left\vert J\right\vert }=\sum_{\substack{J\subseteq S;\\J\subseteq
S\setminus T}}\left(  -1\right)  ^{\left\vert J\right\vert }\nonumber\\
&  \ \ \ \ \ \ \ \ \ \ \ \ \ \ \ \ \ \ \ \ \left(
\begin{array}
[c]{c}%
\text{since the condition \textquotedblleft}T\subseteq S\setminus
J\text{\textquotedblright}\\
\text{is equivalent to \textquotedblleft}J\subseteq S\setminus
T\text{\textquotedblright\ when}\\
\text{both }J\text{ and }T\text{ are subsets of }S
\end{array}
\right) \nonumber\\
&  =\sum_{J\subseteq S\setminus T}\left(  -1\right)  ^{\left\vert J\right\vert
} \label{pf.lem.altsum.1.1}%
\end{align}
(since each subset $J\subseteq S\setminus T$ satisfies $J\subseteq S$).

However, a known fact about finite sets (see, e.g., \cite[Proposition
2.9.10]{Grinbe20}) says that $\sum_{I\subseteq S}\left(  -1\right)
^{\left\vert I\right\vert }=\left[  S=\varnothing\right]  $. Applying this
fact to $S\setminus T$ instead of $S$, we obtain $\sum_{I\subseteq S\setminus
T}\left(  -1\right)  ^{\left\vert I\right\vert }=\left[  S\setminus
T=\varnothing\right]  =\left[  S=T\right]  $. Thus, (\ref{pf.lem.altsum.1.1})
becomes%
\[
\sum_{\substack{I\subseteq S;\\T\subseteq I}}\left(  -1\right)  ^{\left\vert
S\right\vert -\left\vert I\right\vert }=\sum_{J\subseteq S\setminus T}\left(
-1\right)  ^{\left\vert J\right\vert }=\sum_{I\subseteq S\setminus T}\left(
-1\right)  ^{\left\vert I\right\vert }=\left[  S=T\right]  .
\]
This proves Lemma \ref{lem.altsum.1}.
\end{proof}
\end{vershort}

\begin{verlong}
\begin{proof}
[Proof of Lemma \ref{lem.altsum.1}.]If we don't have $T\subseteq S$, then the
claim of Lemma \ref{lem.altsum.1} is easy to check
directly\footnote{\textit{Proof.} Assume that we don't have $T\subseteq S$. We
must prove the claim of Lemma \ref{lem.altsum.1}. In other words, we must
prove that $\sum_{\substack{I\subseteq S;\\T\subseteq I}}\left(  -1\right)
^{\left\vert S\right\vert -\left\vert I\right\vert }=\left[  S=T\right]  $.
\par
If we had $S=T$, then we would have $T\subseteq S$, which would contradict the
fact that we don't have $T\subseteq S$. Hence, we do not have $S=T$. Thus,
$\left[  S=T\right]  =0$.
\par
On the other hand, there exists no subset $I$ of $S$ satisfying $T\subseteq I$
(because if $I$ was such a subset, then we would have $T\subseteq I\subseteq
S$ and thus $T\subseteq S$, but this would contradict the fact that we don't
have $T\subseteq S$). Hence, the sum $\sum_{\substack{I\subseteq
S;\\T\subseteq I}}\left(  -1\right)  ^{\left\vert S\right\vert -\left\vert
I\right\vert }$ is an empty sum. Thus,%
\[
\sum_{\substack{I\subseteq S;\\T\subseteq I}}\left(  -1\right)  ^{\left\vert
S\right\vert -\left\vert I\right\vert }=0=\left[  S=T\right]
\ \ \ \ \ \ \ \ \ \ \left(  \text{since }\left[  S=T\right]  =0\right)  .
\]
This proves the claim of Lemma \ref{lem.altsum.1} (under the assumption that
we don't have $T\subseteq S$).}. Hence, for the rest of this proof, we WLOG
assume that we do have $T\subseteq S$.

It is thus easy to see that $\left[  S\setminus T=\varnothing\right]  =\left[
S=T\right]  $\ \ \ \ \footnote{\textit{Proof.} If $S\subseteq T$, then $S=T$
(because we already know that $T\subseteq S$, so that we can obtain $S=T$ by
combining $S\subseteq T$ with $T\subseteq S$). Conversely, if $S=T$, then
$S\subseteq T$ (obviously). Combining these two results, we conclude that
$S\subseteq T$ holds if and only if $S=T$ holds. In other words, we have the
logical equivalence $\left(  S\subseteq T\right)  \ \Longleftrightarrow
\ \left(  S=T\right)  $.
\par
On the other hand, the logical equivalence $\left(  S\setminus T=\varnothing
\right)  \ \Longleftrightarrow\ \left(  S\subseteq T\right)  $ holds (by
elementary set theory).
\par
We thus obtain the chain of logical equivalences%
\[
\left(  S\setminus T=\varnothing\right)  \ \Longleftrightarrow\ \left(
S\subseteq T\right)  \ \Longleftrightarrow\ \left(  S=T\right)  .
\]
Hence, the statements \textquotedblleft$S\setminus T=\varnothing
$\textquotedblright\ and \textquotedblleft$S=T$\textquotedblright\ are
equivalent. Since equivalent statements have the same truth value, we thus
conclude that $\left[  S\setminus T=\varnothing\right]  =\left[  S=T\right]
$.}.

Next, we recall the classical fact that the map%
\begin{align*}
\left\{  \text{subsets of }S\right\}   &  \rightarrow\left\{  \text{subsets of
}S\right\}  ,\\
J  &  \mapsto S\setminus J
\end{align*}
(which sends each subset $J$ of $S$ to its complement $S\setminus J$ in $S$)
is a bijection (in fact, this map is its own inverse). Hence, we can
substitute $S\setminus J$ for $I$ in the sum $\sum_{\substack{I\subseteq
S;\\T\subseteq I}}\left(  -1\right)  ^{\left\vert S\right\vert -\left\vert
I\right\vert }$. As a result, we obtain%
\begin{align}
\sum_{\substack{I\subseteq S;\\T\subseteq I}}\left(  -1\right)  ^{\left\vert
S\right\vert -\left\vert I\right\vert }  &  =\sum_{\substack{J\subseteq
S;\\T\subseteq S\setminus J}}\underbrace{\left(  -1\right)  ^{\left\vert
S\right\vert -\left\vert S\setminus J\right\vert }}_{\substack{=\left(
-1\right)  ^{\left\vert S\right\vert -\left(  \left\vert S\right\vert
-\left\vert J\right\vert \right)  }\\\text{(since }J\subseteq S\text{ entails
}\left\vert S\setminus J\right\vert =\left\vert S\right\vert -\left\vert
J\right\vert \text{)}}}=\sum_{\substack{J\subseteq S;\\T\subseteq S\setminus
J}}\underbrace{\left(  -1\right)  ^{\left\vert S\right\vert -\left(
\left\vert S\right\vert -\left\vert J\right\vert \right)  }}%
_{\substack{=\left(  -1\right)  ^{\left\vert J\right\vert }\\\text{(since
}\left\vert S\right\vert -\left(  \left\vert S\right\vert -\left\vert
J\right\vert \right)  =\left\vert J\right\vert \text{)}}}\nonumber\\
&  =\sum_{\substack{J\subseteq S;\\T\subseteq S\setminus J}}\left(  -1\right)
^{\left\vert J\right\vert }. \label{pf.lem.altsum.1.long.2}%
\end{align}

However, it is easy to see that the subsets $J$ of $S$ satisfying $T\subseteq
S\setminus J$ are precisely the subsets of $S\setminus T$%
\ \ \ \ \footnote{\textit{Proof.} Let $A$ be a subset $J$ of $S$ satisfying
$T\subseteq S\setminus J$. Thus, $A$ is a subset of $S$ and satisfies
$T\subseteq S\setminus A$. From $T\subseteq S\setminus A$, we conclude that
each element of $T$ belongs to $S\setminus A$. Therefore, no element of $T$
belongs to $A$ (since an element cannot belong to $S\setminus A$ and to $A$ at
the same time). In other words, $T$ is disjoint from $A$. In other words, $A$
is disjoint from $T$. Hence, no element of $A$ belongs to $T$. Thus, if $a\in
A$, then $a\in S$ (since $A$ is a subset of $S$) but $a\notin T$ (since no
element of $A$ belongs to $T$), so that $a\in S\setminus T$ (this follows by
combining $a\in S$ with $a\notin T$). In other words, $A$ is a subset of
$S\setminus T$.
\par
Forget that we fixed $A$. We thus have shown that if $A$ is a subset $J$ of
$S$ satisfying $T\subseteq S\setminus J$, then $A$ is a subset of $S\setminus
T$. In other words, we have shown that%
\begin{align}
&  \text{every subset }J\text{ of }S\text{ satisfying }T\subseteq S\setminus
J\nonumber\\
&  \text{is a subset of }S\setminus T. \label{pf.lem.altsum.1.long.3.pf.1}%
\end{align}
\par
On the other hand, let $B$ be a subset of $S\setminus T$. Thus, $B\subseteq
S\setminus T$. Hence, each element of $B$ belongs to $S\setminus T$.
Therefore, no element of $B$ belongs to $T$ (since an element cannot belong to
$S\setminus T$ and to $T$ at the same time). In other words, $B$ is disjoint
from $T$. In other words, $T$ is disjoint from $B$. In other words, no element
of $T$ belongs to $B$. Hence, if $t\in T$, then $t\notin B$, and therefore
$t\in S\setminus B$ (indeed, we can obtain this by combining $t\in T\subseteq
S$ with $t\notin B$). In other words, $T\subseteq S\setminus B$. Hence, $B$ is
a subset $J$ of $S$ satisfying $T\subseteq S\setminus J$ (since $B\subseteq
S\setminus T\subseteq S$ shows that $B$ is a subset of $S$).
\par
Forget that we fixed $B$. We thus have shown that if $B$ is a subset of
$S\setminus T$, then $B$ is a subset $J$ of $S$ satisfying $T\subseteq
S\setminus J$. In other words, we have shown that%
\begin{align*}
&  \text{every subset of }S\setminus T\\
&  \text{is a subset }J\text{ of }S\text{ satisfying }T\subseteq S\setminus J.
\end{align*}
Combining this with (\ref{pf.lem.altsum.1.long.3.pf.1}), we conclude that the
subsets $J$ of $S$ satisfying $T\subseteq S\setminus J$ are precisely the
subsets of $S\setminus T$. Qed.}. In other words,%
\[
\left\{  \text{subsets }J\text{ of }S\text{ satisfying }T\subseteq S\setminus
J\right\}  =\left\{  \text{subsets of }S\setminus T\right\}  .
\]
Hence, the summation sign $\sum_{\substack{J\subseteq S;\\T\subseteq
S\setminus J}}$ can be rewritten as $\sum_{J\subseteq S\setminus T}$.
Therefore, we can rewrite (\ref{pf.lem.altsum.1.long.2}) as%
\begin{align}
\sum_{\substack{I\subseteq S;\\T\subseteq I}}\left(  -1\right)  ^{\left\vert
S\right\vert -\left\vert I\right\vert }  &  =\sum_{J\subseteq S\setminus
T}\left(  -1\right)  ^{\left\vert J\right\vert }\nonumber\\
&  =\sum_{I\subseteq S\setminus T}\left(  -1\right)  ^{\left\vert I\right\vert
} \label{pf.lem.altsum.1.long.4}%
\end{align}
(here, we have renamed the summation index $J$ as $I$).

However, a known fact about finite sets (see, e.g., \cite[Proposition
2.9.10]{Grinbe20}) says that $\sum_{I\subseteq S}\left(  -1\right)
^{\left\vert I\right\vert }=\left[  S=\varnothing\right]  $. Applying this
fact to $S\setminus T$ instead of $S$, we obtain $\sum_{I\subseteq S\setminus
T}\left(  -1\right)  ^{\left\vert I\right\vert }=\left[  S\setminus
T=\varnothing\right]  $. Thus, (\ref{pf.lem.altsum.1.long.4}) becomes%
\[
\sum_{\substack{I\subseteq S;\\T\subseteq I}}\left(  -1\right)  ^{\left\vert
S\right\vert -\left\vert I\right\vert }=\sum_{I\subseteq S\setminus T}\left(
-1\right)  ^{\left\vert I\right\vert }=\left[  S\setminus T=\varnothing
\right]  =\left[  S=T\right]  .
\]
This proves Lemma \ref{lem.altsum.1}.
\end{proof}
\end{verlong}

We will also use the following property of compositions:

\begin{lemma}
\label{lem.comps.l-vs-size}Let $n\in\mathbb{N}$. Then:

\begin{enumerate}
\item[\textbf{(a)}] We have $\ell\left(  \delta\right)  =\left\vert D\left(
\delta\right)  \right\vert +\left[  n\neq0\right]  $ for each $\delta
\in\operatorname*{Comp}\nolimits_{n}$.

\item[\textbf{(b)}] We have $\ell\left(  \beta\right)  -\ell\left(
\alpha\right)  =\left\vert D\left(  \beta\right)  \right\vert -\left\vert
D\left(  \alpha\right)  \right\vert $ for any $\alpha\in\operatorname*{Comp}%
\nolimits_{n}$ and $\beta\in\operatorname*{Comp}\nolimits_{n}$.
\end{enumerate}
\end{lemma}

\begin{proof}
\textbf{(a)} Part \textbf{(a)} follows easily from the definition of $D\left(
\delta\right)  $. For a detailed proof, see \cite[Corollary 2.6]{comps}.

\begin{vershort}
\textbf{(b)} Use part \textbf{(a)}. For a detailed proof, see \cite[Corollary
2.7]{comps}. \qedhere

\end{vershort}

\begin{verlong}
\textbf{(b)} Part \textbf{(b)} follows from part \textbf{(a)}. For a detailed
proof, see \cite[Corollary 2.7]{comps}. \qedhere

\end{verlong}
\end{proof}

\begin{vershort}

\begin{proof}
[Proof of Proposition \ref{prop.eta.M-through-eta}.]For each $\alpha
\in\operatorname*{Comp}\nolimits_{n}$, we have%
\[
\eta_{\alpha}^{\left(  q\right)  }=\sum_{\substack{\gamma\in
\operatorname*{Comp}\nolimits_{n};\\D\left(  \gamma\right)  \subseteq D\left(
\alpha\right)  }}r^{\ell\left(  \gamma\right)  }M_{\gamma}%
\]
(by (\ref{eq.def.etaalpha.def}), with $\beta$ renamed as $\gamma$). Hence,%
\begin{align}
&  \sum_{\substack{\alpha\in\operatorname*{Comp}\nolimits_{n};\\D\left(
\alpha\right)  \subseteq D\left(  \beta\right)  }}\left(  -1\right)
^{\ell\left(  \beta\right)  -\ell\left(  \alpha\right)  }\eta_{\alpha
}^{\left(  q\right)  }\nonumber\\
&  =\sum_{\substack{\alpha\in\operatorname*{Comp}\nolimits_{n};\\D\left(
\alpha\right)  \subseteq D\left(  \beta\right)  }}\left(  -1\right)
^{\ell\left(  \beta\right)  -\ell\left(  \alpha\right)  }\sum
_{\substack{\gamma\in\operatorname*{Comp}\nolimits_{n};\\D\left(
\gamma\right)  \subseteq D\left(  \alpha\right)  }}r^{\ell\left(
\gamma\right)  }M_{\gamma}\nonumber\\
&  =\underbrace{\sum_{\substack{\alpha\in\operatorname*{Comp}\nolimits_{n}%
;\\D\left(  \alpha\right)  \subseteq D\left(  \beta\right)  }}\ \ \sum
_{\substack{\gamma\in\operatorname*{Comp}\nolimits_{n};\\D\left(
\gamma\right)  \subseteq D\left(  \alpha\right)  }}}_{=\sum_{\gamma
\in\operatorname*{Comp}\nolimits_{n}}\ \ \sum_{\substack{\alpha\in
\operatorname*{Comp}\nolimits_{n};\\D\left(  \gamma\right)  \subseteq D\left(
\alpha\right)  \subseteq D\left(  \beta\right)  }}}\left(  -1\right)
^{\ell\left(  \beta\right)  -\ell\left(  \alpha\right)  }r^{\ell\left(
\gamma\right)  }M_{\gamma}\nonumber\\
&  =\sum_{\gamma\in\operatorname*{Comp}\nolimits_{n}}r^{\ell\left(
\gamma\right)  }M_{\gamma}\sum_{\substack{\alpha\in\operatorname*{Comp}%
\nolimits_{n};\\D\left(  \gamma\right)  \subseteq D\left(  \alpha\right)
\subseteq D\left(  \beta\right)  }}\left(  -1\right)  ^{\ell\left(
\beta\right)  -\ell\left(  \alpha\right)  }.
\label{pf.prop.eta.M-through-eta.2}%
\end{align}

However, for each $\gamma\in\operatorname*{Comp}\nolimits_{n}$, we have
\begin{align*}
&  \sum_{\substack{\alpha\in\operatorname*{Comp}\nolimits_{n};\\D\left(
\gamma\right)  \subseteq D\left(  \alpha\right)  \subseteq D\left(
\beta\right)  }}\underbrace{\left(  -1\right)  ^{\ell\left(  \beta\right)
-\ell\left(  \alpha\right)  }}_{\substack{=\left(  -1\right)  ^{\left\vert
D\left(  \beta\right)  \right\vert -\left\vert D\left(  \alpha\right)
\right\vert }\\\text{(by Lemma \ref{lem.comps.l-vs-size} \textbf{(b)})}}}\\
&  =\sum_{\substack{\alpha\in\operatorname*{Comp}\nolimits_{n};\\D\left(
\gamma\right)  \subseteq D\left(  \alpha\right)  \subseteq D\left(
\beta\right)  }}\left(  -1\right)  ^{\left\vert D\left(  \beta\right)
\right\vert -\left\vert D\left(  \alpha\right)  \right\vert }=\sum
_{\substack{I\subseteq\left[  n-1\right]  ;\\D\left(  \gamma\right)  \subseteq
I\subseteq D\left(  \beta\right)  }}\left(  -1\right)  ^{\left\vert D\left(
\beta\right)  \right\vert -\left\vert I\right\vert }\\
&  \ \ \ \ \ \ \ \ \ \ \ \ \ \ \ \ \ \ \ \ \left(
\begin{array}
[c]{c}%
\text{here, we have substituted }I\text{ for }D\left(  \alpha\right)  \text{
in the sum,}\\
\text{since the map }D:\operatorname*{Comp}\nolimits_{n}\rightarrow
\mathcal{P}\left(  \left[  n-1\right]  \right)  \text{ is a bijection}%
\end{array}
\right) \\
&  =\sum_{\substack{I\subseteq D\left(  \beta\right)  ;\\D\left(
\gamma\right)  \subseteq I}}\left(  -1\right)  ^{\left\vert D\left(
\beta\right)  \right\vert -\left\vert I\right\vert }%
\ \ \ \ \ \ \ \ \ \ \left(
\begin{array}
[c]{c}%
\text{since }D\left(  \beta\right)  \subseteq\left[  n-1\right]  \text{, so
that each subset }I\\
\text{of }D\left(  \beta\right)  \text{ is also a subset of }\left[
n-1\right]
\end{array}
\right) \\
&  =\left[  D\left(  \beta\right)  =D\left(  \gamma\right)  \right]
\ \ \ \ \ \ \ \ \ \ \left(  \text{by Lemma \ref{lem.altsum.1}, applied to
}S=D\left(  \beta\right)  \text{ and }T=D\left(  \gamma\right)  \right) \\
&  =\left[  \beta=\gamma\right]  \ \ \ \ \ \ \ \ \ \ \left(  \text{since the
map }D\text{ is a bijection}\right)  .
\end{align*}
Hence, (\ref{pf.prop.eta.M-through-eta.2}) becomes%
\begin{align*}
\sum_{\substack{\alpha\in\operatorname*{Comp}\nolimits_{n};\\D\left(
\alpha\right)  \subseteq D\left(  \beta\right)  }}\left(  -1\right)
^{\ell\left(  \beta\right)  -\ell\left(  \alpha\right)  }\eta_{\alpha
}^{\left(  q\right)  }  &  =\sum_{\gamma\in\operatorname*{Comp}\nolimits_{n}%
}r^{\ell\left(  \gamma\right)  }M_{\gamma}\underbrace{\sum_{\substack{\alpha
\in\operatorname*{Comp}\nolimits_{n};\\D\left(  \gamma\right)  \subseteq
D\left(  \alpha\right)  \subseteq D\left(  \beta\right)  }}\left(  -1\right)
^{\ell\left(  \beta\right)  -\ell\left(  \alpha\right)  }}_{=\left[
\beta=\gamma\right]  }\\
&  =\sum_{\gamma\in\operatorname*{Comp}\nolimits_{n}}r^{\ell\left(
\gamma\right)  }M_{\gamma}\left[  \beta=\gamma\right]  =r^{\ell\left(
\beta\right)  }M_{\beta}%
\end{align*}
(since the factor $\left[  \beta=\gamma\right]  $ in the sum ensures that the
only nonzero addend in the sum is the addend for $\gamma=\beta$). This proves
Proposition \ref{prop.eta.M-through-eta}.
\end{proof}
\end{vershort}

\begin{verlong}

\begin{proof}
[Proof of Proposition \ref{prop.eta.M-through-eta}.]Recall that
$D:\operatorname*{Comp}\nolimits_{n}\rightarrow\mathcal{P}\left(  \left[
n-1\right]  \right)  $ is a bijection. Hence, from $\beta\in
\operatorname*{Comp}\nolimits_{n}$, we obtain $D\left(  \beta\right)
\in\mathcal{P}\left(  \left[  n-1\right]  \right)  $. In other words,
$D\left(  \beta\right)  \subseteq\left[  n-1\right]  $.

The map $D:\operatorname*{Comp}\nolimits_{n}\rightarrow\mathcal{P}\left(
\left[  n-1\right]  \right)  $ is a bijection, thus injective. In other words,
for any $\varphi,\psi\in\operatorname*{Comp}\nolimits_{n}$, we have $D\left(
\varphi\right)  =D\left(  \psi\right)  $ if and only if $\varphi=\psi$. In
other words, for any $\varphi,\psi\in\operatorname*{Comp}\nolimits_{n}$, the
statement \textquotedblleft$D\left(  \varphi\right)  =D\left(  \psi\right)
$\textquotedblright\ is equivalent to \textquotedblleft$\varphi=\psi
$\textquotedblright. Thus, for any $\varphi,\psi\in\operatorname*{Comp}%
\nolimits_{n}$, we have%
\begin{equation}
\left[  D\left(  \varphi\right)  =D\left(  \psi\right)  \right]  =\left[
\varphi=\psi\right]  \label{pf.prop.eta.M-through-eta.long.0}%
\end{equation}
(since equivalent statements have the same truth value).

For each $\alpha\in\operatorname*{Comp}\nolimits_{n}$, we have%
\begin{equation}
\eta_{\alpha}^{\left(  q\right)  }=\sum_{\substack{\gamma\in
\operatorname*{Comp}\nolimits_{n};\\D\left(  \gamma\right)  \subseteq D\left(
\alpha\right)  }}r^{\ell\left(  \gamma\right)  }M_{\gamma}
\label{pf.prop.eta.M-through-eta.long.1}%
\end{equation}
(this is just the equality (\ref{eq.def.etaalpha.def}), with $\beta$ renamed
as $\gamma$). Hence,%
\begin{align}
&  \sum_{\substack{\alpha\in\operatorname*{Comp}\nolimits_{n};\\D\left(
\alpha\right)  \subseteq D\left(  \beta\right)  }}\left(  -1\right)
^{\ell\left(  \beta\right)  -\ell\left(  \alpha\right)  }\underbrace{\eta
_{\alpha}^{\left(  q\right)  }}_{\substack{=\sum_{\substack{\gamma
\in\operatorname*{Comp}\nolimits_{n};\\D\left(  \gamma\right)  \subseteq
D\left(  \alpha\right)  }}r^{\ell\left(  \gamma\right)  }M_{\gamma}\\\text{(by
(\ref{pf.prop.eta.M-through-eta.long.1}))}}}\nonumber\\
&  =\sum_{\substack{\alpha\in\operatorname*{Comp}\nolimits_{n};\\D\left(
\alpha\right)  \subseteq D\left(  \beta\right)  }}\left(  -1\right)
^{\ell\left(  \beta\right)  -\ell\left(  \alpha\right)  }\sum
_{\substack{\gamma\in\operatorname*{Comp}\nolimits_{n};\\D\left(
\gamma\right)  \subseteq D\left(  \alpha\right)  }}r^{\ell\left(
\gamma\right)  }M_{\gamma}\nonumber\\
&  =\underbrace{\sum_{\substack{\alpha\in\operatorname*{Comp}\nolimits_{n}%
;\\D\left(  \alpha\right)  \subseteq D\left(  \beta\right)  }}\ \ \sum
_{\substack{\gamma\in\operatorname*{Comp}\nolimits_{n};\\D\left(
\gamma\right)  \subseteq D\left(  \alpha\right)  }}}_{=\sum_{\gamma
\in\operatorname*{Comp}\nolimits_{n}}\ \ \sum_{\substack{\alpha\in
\operatorname*{Comp}\nolimits_{n};\\D\left(  \gamma\right)  \subseteq D\left(
\alpha\right)  ;\\D\left(  \alpha\right)  \subseteq D\left(  \beta\right)  }%
}}\left(  -1\right)  ^{\ell\left(  \beta\right)  -\ell\left(  \alpha\right)
}r^{\ell\left(  \gamma\right)  }M_{\gamma}\nonumber\\
&  =\sum_{\gamma\in\operatorname*{Comp}\nolimits_{n}}\ \ \underbrace{\sum
_{\substack{\alpha\in\operatorname*{Comp}\nolimits_{n};\\D\left(
\gamma\right)  \subseteq D\left(  \alpha\right)  ;\\D\left(  \alpha\right)
\subseteq D\left(  \beta\right)  }}\left(  -1\right)  ^{\ell\left(
\beta\right)  -\ell\left(  \alpha\right)  }r^{\ell\left(  \gamma\right)
}M_{\gamma}}_{=r^{\ell\left(  \gamma\right)  }M_{\gamma}\sum_{\substack{\alpha
\in\operatorname*{Comp}\nolimits_{n};\\D\left(  \gamma\right)  \subseteq
D\left(  \alpha\right)  ;\\D\left(  \alpha\right)  \subseteq D\left(
\beta\right)  }}\left(  -1\right)  ^{\ell\left(  \beta\right)  -\ell\left(
\alpha\right)  }}\nonumber\\
&  =\sum_{\gamma\in\operatorname*{Comp}\nolimits_{n}}r^{\ell\left(
\gamma\right)  }M_{\gamma}\sum_{\substack{\alpha\in\operatorname*{Comp}%
\nolimits_{n};\\D\left(  \gamma\right)  \subseteq D\left(  \alpha\right)
;\\D\left(  \alpha\right)  \subseteq D\left(  \beta\right)  }}\left(
-1\right)  ^{\ell\left(  \beta\right)  -\ell\left(  \alpha\right)  }.
\label{pf.prop.eta.M-through-eta.long.2}%
\end{align}

However, for each $\gamma\in\operatorname*{Comp}\nolimits_{n}$, we have
\begin{align}
&  \sum_{\substack{\alpha\in\operatorname*{Comp}\nolimits_{n};\\D\left(
\gamma\right)  \subseteq D\left(  \alpha\right)  ;\\D\left(  \alpha\right)
\subseteq D\left(  \beta\right)  }}\underbrace{\left(  -1\right)
^{\ell\left(  \beta\right)  -\ell\left(  \alpha\right)  }}_{\substack{=\left(
-1\right)  ^{\left\vert D\left(  \beta\right)  \right\vert -\left\vert
D\left(  \alpha\right)  \right\vert }\\\text{(since Lemma
\ref{lem.comps.l-vs-size} \textbf{(b)}}\\\text{yields }\ell\left(
\beta\right)  -\ell\left(  \alpha\right)  =\left\vert D\left(  \beta\right)
\right\vert -\left\vert D\left(  \alpha\right)  \right\vert \text{)}%
}}\nonumber\\
&  =\sum_{\substack{\alpha\in\operatorname*{Comp}\nolimits_{n};\\D\left(
\gamma\right)  \subseteq D\left(  \alpha\right)  ;\\D\left(  \alpha\right)
\subseteq D\left(  \beta\right)  }}\left(  -1\right)  ^{\left\vert D\left(
\beta\right)  \right\vert -\left\vert D\left(  \alpha\right)  \right\vert
}=\underbrace{\sum_{\substack{I\in\mathcal{P}\left(  \left[  n-1\right]
\right)  ;\\D\left(  \gamma\right)  \subseteq I;\\I\subseteq D\left(
\beta\right)  }}}_{\substack{=\sum_{\substack{I\subseteq\left[  n-1\right]
;\\D\left(  \gamma\right)  \subseteq I;\\I\subseteq D\left(  \beta\right)
}}\\\text{(here, we have rewritten}\\\text{the condition \textquotedblleft%
}I\in\mathcal{P}\left(  \left[  n-1\right]  \right)  \text{\textquotedblright%
}\\\text{as \textquotedblleft}I\subseteq\left[  n-1\right]
\text{\textquotedblright)}}}\left(  -1\right)  ^{\left\vert D\left(
\beta\right)  \right\vert -\left\vert I\right\vert }\nonumber\\
&  \ \ \ \ \ \ \ \ \ \ \ \ \ \ \ \ \ \ \ \ \left(
\begin{array}
[c]{c}%
\text{here, we have substituted }I\text{ for }D\left(  \alpha\right)  \text{
in the sum,}\\
\text{since the map }D:\operatorname*{Comp}\nolimits_{n}\rightarrow
\mathcal{P}\left(  \left[  n-1\right]  \right)  \text{ is a bijection}%
\end{array}
\right) \nonumber\\
&  =\underbrace{\sum_{\substack{I\subseteq\left[  n-1\right]  ;\\D\left(
\gamma\right)  \subseteq I;\\I\subseteq D\left(  \beta\right)  }%
}}_{\substack{=\sum_{\substack{I\subseteq D\left(  \beta\right)  ;\\D\left(
\gamma\right)  \subseteq I;\\I\subseteq\left[  n-1\right]  }}=\sum
_{\substack{I\subseteq D\left(  \beta\right)  ;\\D\left(  \gamma\right)
\subseteq I}}\\\text{(here, we have removed the}\\\text{condition
\textquotedblleft}I\subseteq\left[  n-1\right]  \text{\textquotedblright%
\ under}\\\text{the summation sign, since it}\\\text{automatically holds
for}\\\text{every }I\subseteq D\left(  \beta\right)  \text{ (because }D\left(
\beta\right)  \subseteq\left[  n-1\right]  \text{))}}}\left(  -1\right)
^{\left\vert D\left(  \beta\right)  \right\vert -\left\vert I\right\vert
}=\sum_{\substack{I\subseteq D\left(  \beta\right)  ;\\D\left(  \gamma\right)
\subseteq I}}\left(  -1\right)  ^{\left\vert D\left(  \beta\right)
\right\vert -\left\vert I\right\vert }\nonumber\\
&  =\left[  D\left(  \beta\right)  =D\left(  \gamma\right)  \right]
\ \ \ \ \ \ \ \ \ \ \left(  \text{by Lemma \ref{lem.altsum.1}, applied to
}S=D\left(  \beta\right)  \text{ and }T=D\left(  \gamma\right)  \right)
\nonumber\\
&  =\left[  \beta=\gamma\right]  \label{pf.prop.eta.M-through-eta.long.3}%
\end{align}
(by (\ref{pf.prop.eta.M-through-eta.long.0}), applied to $\varphi=\beta$ and
$\psi=\gamma$).

Hence, (\ref{pf.prop.eta.M-through-eta.long.2}) becomes%
\begin{align*}
&  \sum_{\substack{\alpha\in\operatorname*{Comp}\nolimits_{n};\\D\left(
\alpha\right)  \subseteq D\left(  \beta\right)  }}\left(  -1\right)
^{\ell\left(  \beta\right)  -\ell\left(  \alpha\right)  }\eta_{\alpha
}^{\left(  q\right)  }\\
&  =\sum_{\gamma\in\operatorname*{Comp}\nolimits_{n}}r^{\ell\left(
\gamma\right)  }M_{\gamma}\underbrace{\sum_{\substack{\alpha\in
\operatorname*{Comp}\nolimits_{n};\\D\left(  \gamma\right)  \subseteq D\left(
\alpha\right)  ;\\D\left(  \alpha\right)  \subseteq D\left(  \beta\right)
}}\left(  -1\right)  ^{\ell\left(  \beta\right)  -\ell\left(  \alpha\right)
}}_{\substack{=\left[  \beta=\gamma\right]  \\\text{(by
(\ref{pf.prop.eta.M-through-eta.long.3}))}}}\\
&  =\sum_{\gamma\in\operatorname*{Comp}\nolimits_{n}}r^{\ell\left(
\gamma\right)  }M_{\gamma}\left[  \beta=\gamma\right] \\
&  =r^{\ell\left(  \beta\right)  }M_{\beta}\underbrace{\left[  \beta
=\beta\right]  }_{\substack{=1\\\text{(since }\beta=\beta\text{)}}%
}+\sum_{\substack{\gamma\in\operatorname*{Comp}\nolimits_{n};\\\gamma\neq
\beta}}r^{\ell\left(  \gamma\right)  }M_{\gamma}\underbrace{\left[
\beta=\gamma\right]  }_{\substack{=0\\\text{(since }\beta\neq\gamma
\\\text{(because }\gamma\neq\beta\text{))}}}\\
&  \ \ \ \ \ \ \ \ \ \ \ \ \ \ \ \ \ \ \ \ \left(  \text{here, we have split
off the addend for }\gamma=\beta\text{ from the sum}\right) \\
&  =r^{\ell\left(  \beta\right)  }M_{\beta}+\underbrace{\sum_{\substack{\gamma
\in\operatorname*{Comp}\nolimits_{n};\\\gamma\neq\beta}}r^{\ell\left(
\gamma\right)  }M_{\gamma}0}_{=0}=r^{\ell\left(  \beta\right)  }M_{\beta}.
\end{align*}
This proves Proposition \ref{prop.eta.M-through-eta}.
\end{proof}
\end{verlong}

Proposition \ref{prop.eta.M-through-eta} shows that the quasisymmetric
functions $r^{\ell\left(  \beta\right)  }M_{\beta}$ for all $\beta
\in\operatorname*{Comp}$ are $\mathbf{k}$-linear combinations of the enriched
$q$-monomial quasisymmetric functions $\eta_{\alpha}^{\left(  q\right)  }$. If
$r$ is invertible in $\mathbf{k}$, then it follows that the monomial
quasisymmetric functions $M_{\beta}$ are such combinations as well, and thus
the family $\left(  \eta_{\alpha}^{\left(  q\right)  }\right)  _{\alpha
\in\operatorname*{Comp}}$ spans the $\mathbf{k}$-module $\operatorname*{QSym}$
in this case. But we can actually say more:

\begin{theorem}
\label{thm.eta.basis}Assume that $r$ is invertible in $\mathbf{k}$. Then:

\begin{enumerate}
\item[\textbf{(a)}] The family $\left(  \eta_{\alpha}^{\left(  q\right)
}\right)  _{\alpha\in\operatorname*{Comp}}$ is a basis of the $\mathbf{k}%
$-module $\operatorname*{QSym}$.

\item[\textbf{(b)}] Let $n\in\mathbb{N}$. Consider the $n$-th graded component
$\operatorname*{QSym}\nolimits_{n}$ of the graded $\mathbf{k}$-module
$\operatorname*{QSym}$. Then, the family $\left(  \eta_{\alpha}^{\left(
q\right)  }\right)  _{\alpha\in\operatorname*{Comp}\nolimits_{n}}$ is a basis
of the $\mathbf{k}$-module $\operatorname*{QSym}\nolimits_{n}$.
\end{enumerate}
\end{theorem}

\begin{vershort}

\begin{proof}
\textbf{(b)} Define a strict partial order $\prec$ on the finite set
$\operatorname*{Comp}\nolimits_{n}$ by setting
\[
\beta\prec\alpha\ \ \ \ \ \ \ \ \ \ \text{if and only if}%
\ \ \ \ \ \ \ \ \ \ \ell\left(  \beta\right)  <\ell\left(  \alpha\right)  .
\]

If two compositions $\alpha,\beta\in\operatorname*{Comp}\nolimits_{n}$ satisfy
$D\left(  \beta\right)  \subseteq D\left(  \alpha\right)  $ and $\beta
\neq\alpha$, then $D\left(  \beta\right)  $ is a proper subset of $D\left(
\alpha\right)  $ (since $D$ is a bijection, so that $\beta\neq\alpha$ entails
$D\left(  \beta\right)  \neq D\left(  \alpha\right)  $), and thus we have
$\left\vert D\left(  \beta\right)  \right\vert <\left\vert D\left(
\alpha\right)  \right\vert $ and therefore $\ell\left(  \beta\right)
<\ell\left(  \alpha\right)  $ (by Lemma \ref{lem.comps.l-vs-size}
\textbf{(b)}) and thus%
\begin{equation}
\beta\prec\alpha\label{pf.thm.eta.basis.short.less}%
\end{equation}
(by the definition of the partial order $\prec$).

For each $\alpha\in\operatorname*{Comp}\nolimits_{n}$, we have%
\begin{align*}
\eta_{\alpha}^{\left(  q\right)  }  &  =\sum_{\substack{\beta\in
\operatorname*{Comp}\nolimits_{n};\\D\left(  \beta\right)  \subseteq D\left(
\alpha\right)  }}r^{\ell\left(  \beta\right)  }M_{\beta}%
\ \ \ \ \ \ \ \ \ \ \left(  \text{by the definition of }\eta_{\alpha}^{\left(
q\right)  }\right) \\
&  =r^{\ell\left(  \alpha\right)  }M_{\alpha}+\sum_{\substack{\beta
\in\operatorname*{Comp}\nolimits_{n};\\D\left(  \beta\right)  \subseteq
D\left(  \alpha\right)  ;\\\beta\neq\alpha}}r^{\ell\left(  \beta\right)
}M_{\beta}\\
&  =r^{\ell\left(  \alpha\right)  }M_{\alpha}+\left(  \text{a linear
combination of }M_{\beta}\text{ with }\beta\in\operatorname*{Comp}%
\nolimits_{n}\text{ satisfying }\beta\prec\alpha\right)
\end{align*}
(by (\ref{pf.thm.eta.basis.short.less})). Since $r$ is invertible, this shows
that the family $\left(  \eta_{\alpha}^{\left(  q\right)  }\right)
_{\alpha\in\operatorname*{Comp}\nolimits_{n}}$ expands invertibly triangularly
in the family $\left(  M_{\alpha}\right)  _{\alpha\in\operatorname*{Comp}%
\nolimits_{n}}$ with respect to the partial order $\prec$ (where we are using
the terminology from \cite[\S 11.1]{GriRei}). Hence, a classical fact about
triangular expansions (\cite[Corollary 11.1.19(e)]{GriRei}) shows that the
family $\left(  \eta_{\alpha}^{\left(  q\right)  }\right)  _{\alpha
\in\operatorname*{Comp}\nolimits_{n}}$ is a basis of the $\mathbf{k}$-module
$\operatorname*{QSym}\nolimits_{n}$ (since the family $\left(  M_{\alpha
}\right)  _{\alpha\in\operatorname*{Comp}\nolimits_{n}}$ is a basis of
$\operatorname*{QSym}\nolimits_{n}$). This proves Theorem \ref{thm.eta.basis}
\textbf{(b)}. \medskip

\textbf{(a)} Theorem \ref{thm.eta.basis} \textbf{(b)} says that the family
$\left(  \eta_{\alpha}^{\left(  q\right)  }\right)  _{\alpha\in
\operatorname*{Comp}\nolimits_{n}}$ is a basis of the $\mathbf{k}$-module
$\operatorname*{QSym}\nolimits_{n}$ for each $n\in\mathbb{N}$. Hence, the
family $\left(  \eta_{\alpha}^{\left(  q\right)  }\right)  _{\alpha
\in\operatorname*{Comp}}$ is a basis of the $\mathbf{k}$-module $\bigoplus
_{n\in\mathbb{N}}\operatorname*{QSym}\nolimits_{n}=\operatorname*{QSym}$. This
proves Theorem \ref{thm.eta.basis} \textbf{(a)}.
\end{proof}
\end{vershort}

\begin{verlong}

\begin{proof}
\textbf{(b)} It is well-known that the family $\left(  M_{\alpha}\right)
_{\alpha\in\operatorname*{Comp}\nolimits_{n}}$ is a basis of the $\mathbf{k}%
$-module $\operatorname*{QSym}\nolimits_{n}$. Renaming the letter $\alpha$ as
$s$ in this statement, we obtain the following: The family $\left(
M_{s}\right)  _{s\in\operatorname*{Comp}\nolimits_{n}}$ is a basis of the
$\mathbf{k}$-module $\operatorname*{QSym}\nolimits_{n}$. Hence, of course,
this family is a family of elements of $\operatorname*{QSym}\nolimits_{n}$.

Furthermore, for each composition $\alpha\in\operatorname*{Comp}\nolimits_{n}%
$, we have%
\begin{align*}
\eta_{\alpha}^{\left(  q\right)  }  &  =\sum_{\substack{\beta\in
\operatorname*{Comp}\nolimits_{n};\\D\left(  \beta\right)  \subseteq D\left(
\alpha\right)  }}r^{\ell\left(  \beta\right)  }\underbrace{M_{\beta}%
}_{\substack{\in\operatorname*{QSym}\nolimits_{n}\\\text{(since }\beta
\in\operatorname*{Comp}\nolimits_{n}\text{)}}}\ \ \ \ \ \ \ \ \ \ \left(
\text{by (\ref{eq.def.etaalpha.def})}\right) \\
&  \in\sum_{\substack{\beta\in\operatorname*{Comp}\nolimits_{n};\\D\left(
\beta\right)  \subseteq D\left(  \alpha\right)  }}r^{\ell\left(  \beta\right)
}\operatorname*{QSym}\nolimits_{n}\subseteq\operatorname*{QSym}\nolimits_{n}%
\end{align*}
(since $\operatorname*{QSym}\nolimits_{n}$ is a $\mathbf{k}$-module). Renaming
the letter $\alpha$ as $s$ in this statement, we obtain the following: For
each composition $s\in\operatorname*{Comp}\nolimits_{n}$, we have $\eta
_{s}^{\left(  q\right)  }\in\operatorname*{QSym}\nolimits_{n}$. Thus, the
family $\left(  \eta_{s}^{\left(  q\right)  }\right)  _{s\in
\operatorname*{Comp}\nolimits_{n}}$ is a family of elements of
$\operatorname*{QSym}\nolimits_{n}$ as well.

We shall use the terminology of \cite[\S 11.1]{GriRei}, specifically the
notion of \textquotedblleft expanding invertibly
triangularly\textquotedblright\ (see \cite[Definition 11.1.16 (b)]{GriRei} for
its meaning).

Define a binary relation $\prec$ on the finite set $\operatorname*{Comp}%
\nolimits_{n}$ as follows: Two compositions $\beta,\alpha\in
\operatorname*{Comp}\nolimits_{n}$ shall satisfy $\beta\prec\alpha$ if and
only if $\ell\left(  \beta\right)  <\ell\left(  \alpha\right)  $.

This relation $\prec$ is transitive\footnote{\textit{Proof.} Let $\gamma
,\beta,\alpha\in\operatorname*{Comp}\nolimits_{n}$ be three compositions that
satisfy $\gamma\prec\beta$ and $\beta\prec\alpha$. We must show that
$\gamma\prec\alpha$.
\par
By the definition of the relation $\prec$, we have $\gamma\prec\beta$ if and
only if $\ell\left(  \gamma\right)  <\ell\left(  \beta\right)  $. Thus, we
have $\ell\left(  \gamma\right)  <\ell\left(  \beta\right)  $ (since
$\gamma\prec\beta$).
\par
By the definition of the relation $\prec$, we have $\beta\prec\alpha$ if and
only if $\ell\left(  \beta\right)  <\ell\left(  \alpha\right)  $. Thus, we
have $\ell\left(  \beta\right)  <\ell\left(  \alpha\right)  $ (since
$\beta\prec\alpha$).
\par
By the definition of the relation $\prec$, we have $\gamma\prec\alpha$ if and
only if $\ell\left(  \gamma\right)  <\ell\left(  \alpha\right)  $. Thus, we
have $\gamma\prec\alpha$ (since $\ell\left(  \gamma\right)  <\ell\left(
\beta\right)  <\ell\left(  \alpha\right)  $).
\par
Forget that we fixed $\gamma,\beta,\alpha$. We thus have shown that if
$\gamma,\beta,\alpha\in\operatorname*{Comp}\nolimits_{n}$ are three
compositions that satisfy $\gamma\prec\beta$ and $\beta\prec\alpha$, then
$\gamma\prec\alpha$. In other words, the relation $\prec$ is transitive.},
irreflexive\footnote{\textit{Proof.} Let $\alpha\in\operatorname*{Comp}%
\nolimits_{n}$ be a composition that satisfies $\alpha\prec\alpha$. We must
find a contradiction.
\par
By the definition of the relation $\prec$, we have $\alpha\prec\alpha$ if and
only if $\ell\left(  \alpha\right)  <\ell\left(  \alpha\right)  $. Thus, we
have $\ell\left(  \alpha\right)  <\ell\left(  \alpha\right)  $ (since
$\alpha\prec\alpha$). But this is absurd. Hence, we have found a
contradiction.
\par
Forget that we fixed $\alpha$. We thus have found a contradiction for each
$\alpha\in\operatorname*{Comp}\nolimits_{n}$ that satisfies $\alpha\prec
\alpha$. In other words, there exists no $\alpha\in\operatorname*{Comp}%
\nolimits_{n}$ that satisfies $\alpha\prec\alpha$. In other words, the
relation $\prec$ is irreflexive.} and asymmetric\footnote{\textit{Proof.} It
is well-known that any binary relation that is transitive and irreflexive must
necessarily be asymmetric. Hence, the relation $\prec$ is asymmetric (since it
is transitive and irreflexive).}. In other words, it is a strict partial
order. Consider the set $\operatorname*{Comp}\nolimits_{n}$ as a poset
equipped with this strict partial order (so that the smaller relation of this
poset shall be the relation $\prec$).

Now, it is easy to see that every $\alpha\in\operatorname*{Comp}\nolimits_{n}$
satisfies%
\begin{align*}
\eta_{\alpha}^{\left(  q\right)  }  &  =r^{\ell\left(  \alpha\right)
}M_{\alpha}+\left(  \text{a }\mathbf{k}\text{-linear combination of
the}\right. \\
&  \ \ \ \ \ \ \ \ \ \ \ \ \ \ \ \ \ \ \ \ \left.  \text{elements }M_{\beta
}\text{ for }\beta\in\operatorname*{Comp}\nolimits_{n}\text{ satisfying }%
\beta\prec\alpha\right)
\end{align*}
\footnote{\textit{Proof.} Let $\alpha\in\operatorname*{Comp}\nolimits_{n}$.
From (\ref{eq.def.etaalpha.def}), we obtain%
\[
\eta_{\alpha}^{\left(  q\right)  }=\sum_{\substack{\beta\in
\operatorname*{Comp}\nolimits_{n};\\D\left(  \beta\right)  \subseteq D\left(
\alpha\right)  }}r^{\ell\left(  \beta\right)  }M_{\beta}.
\]
The sum on the right hand side here has an addend for $\beta=\alpha$ (since
$\alpha\in\operatorname*{Comp}\nolimits_{n}$ and $D\left(  \alpha\right)
\subseteq D\left(  \alpha\right)  $). Splitting this addend off from this sum,
we obtain%
\[
\sum_{\substack{\beta\in\operatorname*{Comp}\nolimits_{n};\\D\left(
\beta\right)  \subseteq D\left(  \alpha\right)  }}r^{\ell\left(  \beta\right)
}M_{\beta}=r^{\ell\left(  \alpha\right)  }M_{\alpha}+\sum_{\substack{\beta
\in\operatorname*{Comp}\nolimits_{n};\\D\left(  \beta\right)  \subseteq
D\left(  \alpha\right)  ;\\\beta\neq\alpha}}r^{\ell\left(  \beta\right)
}M_{\beta}.
\]
We shall now analyze the sum on the right hand side here.
\par
Let $\beta\in\operatorname*{Comp}\nolimits_{n}$ satisfy $D\left(
\beta\right)  \subseteq D\left(  \alpha\right)  $ and $\beta\neq\alpha$. We
shall show that $\beta\prec\alpha$.
\par
Indeed, we recall that $D:\operatorname*{Comp}\nolimits_{n}\rightarrow
\mathcal{P}\left(  \left[  n-1\right]  \right)  $ is a bijection. Hence, this
map $D$ is injective. Thus, from $\beta\neq\alpha$, we obtain $D\left(
\beta\right)  \neq D\left(  \alpha\right)  $. Combining this with $D\left(
\beta\right)  \subseteq D\left(  \alpha\right)  $, we see that $D\left(
\beta\right)  $ is a proper subset of $D\left(  \alpha\right)  $. Therefore,
$\left\vert D\left(  \beta\right)  \right\vert <\left\vert D\left(
\alpha\right)  \right\vert $ (because if $T$ is a proper subset of a finite
set $S$, then $\left\vert T\right\vert <\left\vert S\right\vert $). However,
Lemma \ref{lem.comps.l-vs-size} \textbf{(b)} yields $\ell\left(  \beta\right)
-\ell\left(  \alpha\right)  =\left\vert D\left(  \beta\right)  \right\vert
-\left\vert D\left(  \alpha\right)  \right\vert <0$ (since $\left\vert
D\left(  \beta\right)  \right\vert <\left\vert D\left(  \alpha\right)
\right\vert $). In other words, $\ell\left(  \beta\right)  <\ell\left(
\alpha\right)  $.
\par
By the definition of the relation $\prec$, we have $\beta\prec\alpha$ if and
only if $\ell\left(  \beta\right)  <\ell\left(  \alpha\right)  $. Thus, we
have $\beta\prec\alpha$ (since $\ell\left(  \beta\right)  <\ell\left(
\alpha\right)  $).
\par
Forget that we fixed $\beta$. We thus have shown that if $\beta\in
\operatorname*{Comp}\nolimits_{n}$ satisfies $D\left(  \beta\right)  \subseteq
D\left(  \alpha\right)  $ and $\beta\neq\alpha$, then $\beta\prec\alpha$.
Hence, in each addend of the sum $\sum_{\substack{\beta\in\operatorname*{Comp}%
\nolimits_{n};\\D\left(  \beta\right)  \subseteq D\left(  \alpha\right)
;\\\beta\neq\alpha}}r^{\ell\left(  \beta\right)  }M_{\beta}$, the summation
index $\beta$ satisfies $\beta\prec\alpha$. Thus, this sum is a $\mathbf{k}%
$-linear combination of the elements $M_{\beta}$ for $\beta\in
\operatorname*{Comp}\nolimits_{n}$ satisfying $\beta\prec\alpha$. In other
words,%
\begin{align*}
&  \sum_{\substack{\beta\in\operatorname*{Comp}\nolimits_{n};\\D\left(
\beta\right)  \subseteq D\left(  \alpha\right)  ;\\\beta\neq\alpha}%
}r^{\ell\left(  \beta\right)  }M_{\beta}\\
&  =\left(  \text{a }\mathbf{k}\text{-linear combination of the}\right. \\
&  \ \ \ \ \ \ \ \ \ \ \ \ \ \ \ \ \ \ \ \ \left.  \text{elements }M_{\beta
}\text{ for }\beta\in\operatorname*{Comp}\nolimits_{n}\text{ satisfying }%
\beta\prec\alpha\right)  .
\end{align*}
\par
Altogether, we now obtain%
\begin{align*}
\eta_{\alpha}^{\left(  q\right)  }  &  =\sum_{\substack{\beta\in
\operatorname*{Comp}\nolimits_{n};\\D\left(  \beta\right)  \subseteq D\left(
\alpha\right)  }}r^{\ell\left(  \beta\right)  }M_{\beta}=r^{\ell\left(
\alpha\right)  }M_{\alpha}+\underbrace{\sum_{\substack{\beta\in
\operatorname*{Comp}\nolimits_{n};\\D\left(  \beta\right)  \subseteq D\left(
\alpha\right)  ;\\\beta\neq\alpha}}r^{\ell\left(  \beta\right)  }M_{\beta}%
}_{\substack{=\left(  \text{a }\mathbf{k}\text{-linear combination of
the}\right.  \\\left.  \text{elements }M_{\beta}\text{ for }\beta
\in\operatorname*{Comp}\nolimits_{n}\text{ satisfying }\beta\prec
\alpha\right)  }}\\
&  =r^{\ell\left(  \alpha\right)  }M_{\alpha}+\left(  \text{a }\mathbf{k}%
\text{-linear combination of the}\right. \\
&  \ \ \ \ \ \ \ \ \ \ \ \ \ \ \ \ \ \ \ \ \left.  \text{elements }M_{\beta
}\text{ for }\beta\in\operatorname*{Comp}\nolimits_{n}\text{ satisfying }%
\beta\prec\alpha\right)  .
\end{align*}
Qed.}. Renaming the letters $\alpha$ and $\beta$ as $s$ and $t$ in this
sentence, we obtain the following: Every $s\in\operatorname*{Comp}%
\nolimits_{n}$ satisfies%
\begin{align*}
\eta_{s}^{\left(  q\right)  }  &  =r^{\ell\left(  s\right)  }M_{s}+\left(
\text{a }\mathbf{k}\text{-linear combination of the}\right. \\
&  \ \ \ \ \ \ \ \ \ \ \ \ \ \ \ \ \ \ \ \ \left.  \text{elements }M_{t}\text{
for }t\in\operatorname*{Comp}\nolimits_{n}\text{ satisfying }t\prec s\right)
.
\end{align*}
Furthermore, the coefficient $r^{\ell\left(  s\right)  }\in\mathbf{k}$ here is
invertible for each $s\in\operatorname*{Comp}\nolimits_{n}$ (because $r$ is
invertible). Moreover, we know that both families $\left(  M_{s}\right)
_{s\in\operatorname*{Comp}\nolimits_{n}}$ and $\left(  \eta_{s}^{\left(
q\right)  }\right)  _{s\in\operatorname*{Comp}\nolimits_{n}}$ are families of
elements of $\operatorname*{QSym}\nolimits_{n}$.

Hence, \cite[Remark 11.1.17 (b)]{GriRei} (applied to $M=\operatorname*{QSym}%
\nolimits_{n}$, $S=\operatorname*{Comp}\nolimits_{n}$, $\left(  e_{s}\right)
_{s\in S}=\left(  \eta_{s}^{\left(  q\right)  }\right)  _{s\in
\operatorname*{Comp}\nolimits_{n}}$ and $\left(  f_{s}\right)  _{s\in
S}=\left(  M_{s}\right)  _{s\in\operatorname*{Comp}\nolimits_{n}}$) shows that
the family $\left(  \eta_{s}^{\left(  q\right)  }\right)  _{s\in
\operatorname*{Comp}\nolimits_{n}}$ expands invertibly triangularly in the
family $\left(  M_{s}\right)  _{s\in\operatorname*{Comp}\nolimits_{n}}$ if and
only if every $s\in\operatorname*{Comp}\nolimits_{n}$ satisfies\footnote{Note
that the smaller relation of the poset $\operatorname*{Comp}\nolimits_{n}$
(which is denoted by $<$ in \cite[Remark 11.1.17 (b)]{GriRei}) is called
$\prec$ in our proof here.}%
\begin{align*}
\eta_{s}^{\left(  q\right)  }  &  =\alpha_{s}M_{s}+\left(  \text{a }%
\mathbf{k}\text{-linear combination of the}\right. \\
&  \ \ \ \ \ \ \ \ \ \ \ \ \ \ \ \ \ \ \ \ \left.  \text{elements }M_{t}\text{
for }t\in\operatorname*{Comp}\nolimits_{n}\text{ satisfying }t\prec s\right)
\end{align*}
for some invertible $\alpha_{s}\in\mathbf{k}$. Since every $s\in
\operatorname*{Comp}\nolimits_{n}$ does indeed satisfy%
\begin{align*}
\eta_{s}^{\left(  q\right)  }  &  =\alpha_{s}M_{s}+\left(  \text{a }%
\mathbf{k}\text{-linear combination of the}\right. \\
&  \ \ \ \ \ \ \ \ \ \ \ \ \ \ \ \ \ \ \ \ \left.  \text{elements }M_{t}\text{
for }t\in\operatorname*{Comp}\nolimits_{n}\text{ satisfying }t\prec s\right)
\end{align*}
for some invertible $\alpha_{s}\in\mathbf{k}$ (namely, for $\alpha_{s}%
=r^{\ell\left(  s\right)  }$) (because every $s\in\operatorname*{Comp}%
\nolimits_{n}$ satisfies%
\begin{align*}
\eta_{s}^{\left(  q\right)  }  &  =r^{\ell\left(  s\right)  }M_{s}+\left(
\text{a }\mathbf{k}\text{-linear combination of the}\right. \\
&  \ \ \ \ \ \ \ \ \ \ \ \ \ \ \ \ \ \ \ \ \left.  \text{elements }M_{t}\text{
for }t\in\operatorname*{Comp}\nolimits_{n}\text{ satisfying }t\prec s\right)
,
\end{align*}
and because $r^{\ell\left(  s\right)  }$ is invertible for each $s\in
\operatorname*{Comp}\nolimits_{n}$), we thus conclude that the family $\left(
\eta_{s}^{\left(  q\right)  }\right)  _{s\in\operatorname*{Comp}\nolimits_{n}%
}$ expands invertibly triangularly in the family $\left(  M_{s}\right)
_{s\in\operatorname*{Comp}\nolimits_{n}}$.

Hence, \cite[Corollary 11.1.19 (e)]{GriRei} (applied to
$M=\operatorname*{QSym}\nolimits_{n}$, $S=\operatorname*{Comp}\nolimits_{n}$,
$\left(  e_{s}\right)  _{s\in S}=\left(  \eta_{s}^{\left(  q\right)  }\right)
_{s\in\operatorname*{Comp}\nolimits_{n}}$ and $\left(  f_{s}\right)  _{s\in
S}=\left(  M_{s}\right)  _{s\in\operatorname*{Comp}\nolimits_{n}}$) shows that
the family $\left(  \eta_{s}^{\left(  q\right)  }\right)  _{s\in
\operatorname*{Comp}\nolimits_{n}}$ is a basis of the $\mathbf{k}$-module
$\operatorname*{QSym}\nolimits_{n}$ if and only if the family $\left(
M_{s}\right)  _{s\in\operatorname*{Comp}\nolimits_{n}}$ is a basis of the
$\mathbf{k}$-module $\operatorname*{QSym}\nolimits_{n}$. Hence, the family
$\left(  \eta_{s}^{\left(  q\right)  }\right)  _{s\in\operatorname*{Comp}%
\nolimits_{n}}$ is a basis of the $\mathbf{k}$-module $\operatorname*{QSym}%
\nolimits_{n}$ (since we know that the family $\left(  M_{s}\right)
_{s\in\operatorname*{Comp}\nolimits_{n}}$ is a basis of the $\mathbf{k}%
$-module $\operatorname*{QSym}\nolimits_{n}$). Renaming the letter $s$ as
$\alpha$ in this sentence, we obtain that the family $\left(  \eta_{\alpha
}^{\left(  q\right)  }\right)  _{\alpha\in\operatorname*{Comp}\nolimits_{n}}$
is a basis of the $\mathbf{k}$-module $\operatorname*{QSym}\nolimits_{n}$.
This proves Theorem \ref{thm.eta.basis} \textbf{(b)}. \medskip

\textbf{(a)} Since $\operatorname*{QSym}$ is a graded $\mathbf{k}$-module, we
have $\operatorname*{QSym}=\bigoplus\limits_{n\in\mathbb{N}}%
\operatorname*{QSym}\nolimits_{n}$.

Theorem \ref{thm.eta.basis} \textbf{(b)} shows that for each $n\in\mathbb{N}$,
the family $\left(  \eta_{\alpha}^{\left(  q\right)  }\right)  _{\alpha
\in\operatorname*{Comp}\nolimits_{n}}$ is a basis of the $\mathbf{k}$-module
$\operatorname*{QSym}\nolimits_{n}$. Hence, the union $\left(  \eta_{\alpha
}^{\left(  q\right)  }\right)  _{n\in\mathbb{N},\ \alpha\in
\operatorname*{Comp}\nolimits_{n}}$ of all these families is a basis of the
direct sum $\bigoplus\limits_{n\in\mathbb{N}}\operatorname*{QSym}%
\nolimits_{n}=\operatorname*{QSym}$. In other words, the family $\left(
\eta_{\alpha}^{\left(  q\right)  }\right)  _{\alpha\in\operatorname*{Comp}}$
is a basis of $\operatorname*{QSym}$ (since the family $\left(  \eta_{\alpha
}^{\left(  q\right)  }\right)  _{n\in\mathbb{N},\ \alpha\in
\operatorname*{Comp}\nolimits_{n}}$ is just a reindexing of the family
$\left(  \eta_{\alpha}^{\left(  q\right)  }\right)  _{\alpha\in
\operatorname*{Comp}}$ (because $\operatorname*{Comp}=\bigsqcup\limits_{n\in
\mathbb{N}}\operatorname*{Comp}\nolimits_{n}$)). This proves Theorem
\ref{thm.eta.basis} \textbf{(a)}.
\end{proof}
\end{verlong}

Theorem \ref{thm.eta.basis} \textbf{(a)} has a converse: If the family
$\left(  \eta_{\alpha}^{\left(  q\right)  }\right)  _{\alpha\in
\operatorname*{Comp}}$ is a basis of $\operatorname*{QSym}$, then $r$ is
invertible. (This is already clear from considering its unique degree-$1$
entry $\eta_{\left(  1\right)  }^{\left(  q\right)  }=rM_{\left(  1\right)  }$.)

\subsection{Relation to the fundamental basis}

We can also expand the $\eta_{\alpha}^{\left(  q\right)  }$ in the fundamental
basis and vice versa:

\begin{proposition}
\label{prop.eta.through-F}Let $n$ be a positive integer. Let $\alpha
\in\operatorname*{Comp}\nolimits_{n}$. Then,%
\[
\eta_{\alpha}^{\left(  q\right)  }=r\sum_{\gamma\in\operatorname*{Comp}%
\nolimits_{n}}\left(  -1\right)  ^{\left\vert D\left(  \gamma\right)
\setminus D\left(  \alpha\right)  \right\vert }q^{\left\vert D\left(
\gamma\right)  \cap D\left(  \alpha\right)  \right\vert }L_{\gamma}.
\]

\end{proposition}

\begin{proposition}
\label{prop.eta.F-through}Let $n$ be a positive integer. Let $\gamma
\in\operatorname*{Comp}\nolimits_{n}$. Then,%
\[
r^{n}L_{\gamma}=\sum_{\alpha\in\operatorname*{Comp}\nolimits_{n}}\left(
-1\right)  ^{\left\vert D\left(  \gamma\right)  \setminus D\left(
\alpha\right)  \right\vert }q^{\left\vert \left[  n-1\right]  \setminus\left(
D\left(  \gamma\right)  \cup D\left(  \alpha\right)  \right)  \right\vert
}\eta_{\alpha}^{\left(  q\right)  }.
\]

\end{proposition}

Note that Proposition \ref{prop.eta.through-F} generalizes \cite[Proposition
2.2]{Hsiao07}.

Both propositions can be proved with the help of a rather simple
identity:\footnote{We will use Convention \ref{conv.iverson}.}

\begin{lemma}
\label{lem.eta.through-F.lem1}Let $S$ and $T$ be two finite sets. Then,%
\[
\sum_{I\subseteq S}\left(  -1\right)  ^{\left\vert I\setminus T\right\vert
}q^{\left\vert I\cap T\right\vert }=\left[  S\subseteq T\right]  \cdot
r^{\left\vert S\right\vert }.
\]

\end{lemma}

\begin{vershort}

\begin{proof}
[Proof of Lemma \ref{lem.eta.through-F.lem1}.]We are in one of the following
two cases:

\textit{Case 1:} We have $S\subseteq T$.

\textit{Case 2:} We have $S\not \subseteq T$.

Let us first consider Case 1. In this case, we have $S\subseteq T$. Hence,
$\left[  S\subseteq T\right]  =1$, so that $\left[  S\subseteq T\right]  \cdot
r^{\left\vert S\right\vert }=1\cdot r^{\left\vert S\right\vert }=r^{\left\vert
S\right\vert }$.

Also, each subset $I$ of $S$ satisfies $I\subseteq S\subseteq T$ and therefore
satisfies $I\setminus T=\varnothing$ and $I\cap T=I$. Hence,%
\begin{align}
\sum_{I\subseteq S}\underbrace{\left(  -1\right)  ^{\left\vert I\setminus
T\right\vert }q^{\left\vert I\cap T\right\vert }}_{\substack{=\left(
-1\right)  ^{\left\vert \varnothing\right\vert }q^{\left\vert I\right\vert
}\\\text{(since }I\setminus T=\varnothing\text{ and }I\cap T=I\text{)}}}  &
=\sum_{I\subseteq S}\underbrace{\left(  -1\right)  ^{\left\vert \varnothing
\right\vert }}_{=\left(  -1\right)  ^{0}=1}q^{\left\vert I\right\vert
}\nonumber\\
&  =\sum_{I\subseteq S}q^{\left\vert I\right\vert }.
\label{pf.lem.eta.through-F.lem1.1}%
\end{align}

Now, consider the sum $\sum_{I\subseteq S}q^{\left\vert I\right\vert }$ on the
right hand side. This sum contains $\dbinom{\left\vert S\right\vert }{k}$
copies of each possible power $q^{k}$ (since the number of $k$-element subsets
of $S$ is $\dbinom{\left\vert S\right\vert }{k}$), and thus can be rewritten
as $\sum_{k=0}^{\left\vert S\right\vert }\dbinom{\left\vert S\right\vert }%
{k}q^{k}$. Hence, (\ref{pf.lem.eta.through-F.lem1.1}) rewrites as
\begin{align*}
\sum_{I\subseteq S}\left(  -1\right)  ^{\left\vert I\setminus T\right\vert
}q^{\left\vert I\cap T\right\vert }  &  =\sum_{k=0}^{\left\vert S\right\vert
}\dbinom{\left\vert S\right\vert }{k}q^{k}\\
&  =\left(  1+q\right)  ^{\left\vert S\right\vert }\ \ \ \ \ \ \ \ \ \ \left(
\text{by the binomial formula}\right) \\
&  =r^{\left\vert S\right\vert }\ \ \ \ \ \ \ \ \ \ \left(  \text{since
}1+q=q+1=r\right) \\
&  =\left[  S\subseteq T\right]  \cdot r^{\left\vert S\right\vert
}\ \ \ \ \ \ \ \ \ \ \left(  \text{since }\left[  S\subseteq T\right]  \cdot
r^{\left\vert S\right\vert }=r^{\left\vert S\right\vert }\right)  .
\end{align*}
Thus, Lemma \ref{lem.eta.through-F.lem1} is proved in Case 1.

Let us now consider Case 2. In this case, we have $S\not \subseteq T$. Hence,
$\left[  S\subseteq T\right]  =0$, so that $\left[  S\subseteq T\right]  \cdot
r^{\left\vert S\right\vert }=0\cdot r^{\left\vert S\right\vert }=0$.

There exists some $s\in S$ such that $s\notin T$ (since $S\not \subseteq T$).
Consider this $s$. Now, each subset $I$ of $S$ satisfies either $s\in I$ or
$s\notin I$ (but not both). Hence,
\begin{align*}
&  \sum_{I\subseteq S}\left(  -1\right)  ^{\left\vert I\setminus T\right\vert
}q^{\left\vert I\cap T\right\vert }\\
&  =\sum_{\substack{I\subseteq S;\\s\in I}}\left(  -1\right)  ^{\left\vert
I\setminus T\right\vert }q^{\left\vert I\cap T\right\vert }+\sum
_{\substack{I\subseteq S;\\s\notin I}}\left(  -1\right)  ^{\left\vert
I\setminus T\right\vert }q^{\left\vert I\cap T\right\vert }\\
&  =\sum_{\substack{I\subseteq S;\\s\notin I}}\underbrace{\left(  -1\right)
^{\left\vert \left(  I\cup\left\{  s\right\}  \right)  \setminus T\right\vert
}}_{\substack{=\left(  -1\right)  ^{\left\vert I\setminus T\right\vert
+1}\\\text{(since }s\notin T\text{ and }s\notin I\text{, so}\\\text{that
}\left\vert \left(  I\cup\left\{  s\right\}  \right)  \setminus T\right\vert
=\left\vert I\setminus T\right\vert +1\text{)}}}\ \ \underbrace{q^{\left\vert
\left(  I\cup\left\{  s\right\}  \right)  \cap T\right\vert }}%
_{\substack{=q^{\left\vert I\cap T\right\vert }\\\text{(since }s\notin T\text{
and}\\\text{thus }\left(  I\cup\left\{  s\right\}  \right)  \cap T=I\cap
T\text{)}}}+\sum_{\substack{I\subseteq S;\\s\notin I}}\left(  -1\right)
^{\left\vert I\setminus T\right\vert }q^{\left\vert I\cap T\right\vert }\\
&  \ \ \ \ \ \ \ \ \ \ \ \ \ \ \ \ \ \ \ \ \left(
\begin{array}
[c]{c}%
\text{here, we have substituted }I\cup\left\{  s\right\}  \text{ for }I\text{
in the first sum,}\\
\text{since the map }\left\{  I\subseteq S\ \mid\ s\notin I\right\}
\rightarrow\left\{  I\subseteq S\ \mid\ s\in I\right\} \\
\text{that sends each }I\text{ to }I\cup\left\{  s\right\}  \text{ is a
bijection}%
\end{array}
\right) \\
&  =\sum_{\substack{I\subseteq S;\\s\notin I}}\underbrace{\left(  -1\right)
^{\left\vert I\setminus T\right\vert +1}}_{=-\left(  -1\right)  ^{\left\vert
I\setminus T\right\vert }}q^{\left\vert I\cap T\right\vert }+\sum
_{\substack{I\subseteq S;\\s\notin I}}\left(  -1\right)  ^{\left\vert
I\setminus T\right\vert }q^{\left\vert I\cap T\right\vert }\\
&  =-\sum_{\substack{I\subseteq S;\\s\notin I}}\left(  -1\right)  ^{\left\vert
I\setminus T\right\vert }q^{\left\vert I\cap T\right\vert }+\sum
_{\substack{I\subseteq S;\\s\notin I}}\left(  -1\right)  ^{\left\vert
I\setminus T\right\vert }q^{\left\vert I\cap T\right\vert }=0=\left[
S\subseteq T\right]  \cdot r^{\left\vert S\right\vert }%
\end{align*}
(since $\left[  S\subseteq T\right]  \cdot r^{\left\vert S\right\vert }=0$).
Thus, Lemma \ref{lem.eta.through-F.lem1} is proved in Case 2. The proof of
Lemma \ref{lem.eta.through-F.lem1} is thus complete.
\end{proof}
\end{vershort}

\begin{verlong}

\begin{proof}
[Proof of Lemma \ref{lem.eta.through-F.lem1}.]We are in one of the following
two cases:

\textit{Case 1:} We have $S\subseteq T$.

\textit{Case 2:} We have $S\not \subseteq T$.

Let us first consider Case 1. In this case, we have $S\subseteq T$. Therefore,
for each subset $I$ of $S$, we have $I\subseteq S\subseteq T$ and therefore
\[
\underbrace{\left(  -1\right)  ^{\left\vert I\setminus T\right\vert }%
}_{\substack{=\left(  -1\right)  ^{\left\vert \varnothing\right\vert
}\\\text{(since }I\setminus T=\varnothing\\\text{(because }I\subseteq
T\text{))}}}\ \ \underbrace{q^{\left\vert I\cap T\right\vert }}%
_{\substack{=q^{\left\vert I\right\vert }\\\text{(since }I\cap
T=I\\\text{(because }I\subseteq T\text{))}}}=\underbrace{\left(  -1\right)
^{\left\vert \varnothing\right\vert }}_{\substack{=\left(  -1\right)
^{0}\\\text{(since }\left\vert \varnothing\right\vert =0\text{)}%
}}q^{\left\vert I\right\vert }=\underbrace{\left(  -1\right)  ^{0}}%
_{=1}q^{\left\vert I\right\vert }=q^{\left\vert I\right\vert }.
\]
Summing this equality over all subsets $I$ of $S$, we obtain%
\begin{align*}
\sum_{I\subseteq S}\left(  -1\right)  ^{\left\vert I\setminus T\right\vert
}q^{\left\vert I\cap T\right\vert }  &  =\sum_{I\subseteq S}q^{\left\vert
I\right\vert }\\
&  =\sum_{k=0}^{\left\vert S\right\vert }\ \ \sum_{\substack{I\subseteq
S;\\\left\vert I\right\vert =k}}\underbrace{q^{\left\vert I\right\vert }%
}_{\substack{=q^{k}\\\text{(since }\left\vert I\right\vert =k\text{)}%
}}\ \ \ \ \ \ \ \ \ \ \left(
\begin{array}
[c]{c}%
\text{here, we have split the sum}\\
\text{according to the value of }\left\vert I\right\vert \text{,}\\
\text{since }\left\vert I\right\vert \in\left\{  0,1,\ldots,\left\vert
S\right\vert \right\}  \text{ for}\\
\text{every subset }I\text{ of }S
\end{array}
\right) \\
&  =\sum_{k=0}^{\left\vert S\right\vert }\ \ \underbrace{\sum
_{\substack{I\subseteq S;\\\left\vert I\right\vert =k}}q^{k}}_{=\left(
\text{number of all subsets }I\text{ of }S\text{ satisfying }\left\vert
I\right\vert =k\right)  \cdot q^{k}}\\
&  =\sum_{k=0}^{\left\vert S\right\vert }\underbrace{\left(  \text{number of
all subsets }I\text{ of }S\text{ satisfying }\left\vert I\right\vert
=k\right)  }_{\substack{=\left(  \text{number of all }k\text{-element subsets
of }S\right)  \\=\dbinom{\left\vert S\right\vert }{k}\\\text{(by the
combinatorial interpretation}\\\text{of the binomial coefficients)}}%
}\cdot\,q^{k}\\
&  =\sum_{k=0}^{\left\vert S\right\vert }\dbinom{\left\vert S\right\vert }%
{k}q^{k}.
\end{align*}
Comparing this with%
\begin{align*}
\underbrace{\left[  S\subseteq T\right]  }_{\substack{=1\\\text{(since
}S\subseteq T\text{)}}}\cdot\,r^{\left\vert S\right\vert }  &  =r^{\left\vert
S\right\vert }=\left(  q+1\right)  ^{\left\vert S\right\vert }%
\ \ \ \ \ \ \ \ \ \ \left(  \text{since }r=q+1\right) \\
&  =\sum_{k=0}^{\left\vert S\right\vert }\dbinom{\left\vert S\right\vert }%
{k}q^{k}\underbrace{1^{\left\vert S\right\vert -k}}_{=1}%
\ \ \ \ \ \ \ \ \ \ \left(  \text{by the binomial formula}\right) \\
&  =\sum_{k=0}^{\left\vert S\right\vert }\dbinom{\left\vert S\right\vert }%
{k}q^{k},
\end{align*}
we obtain%
\[
\sum_{I\subseteq S}\left(  -1\right)  ^{\left\vert I\setminus T\right\vert
}q^{\left\vert I\cap T\right\vert }=\left[  S\subseteq T\right]  \cdot
r^{\left\vert S\right\vert }.
\]
Thus, Lemma \ref{lem.eta.through-F.lem1} is proved in Case 1.

Let us now consider Case 2. In this case, we have $S\not \subseteq T$. Hence,
there exists some $s\in S$ such that $s\notin T$. Consider this $s$. Now, each
subset $I$ of $S$ satisfies either $s\in I$ or $s\notin I$ (but not both).
Hence, we can break up the sum $\sum_{I\subseteq S}\left(  -1\right)
^{\left\vert I\setminus T\right\vert }q^{\left\vert I\cap T\right\vert }$ as
follows:%
\begin{align*}
&  \sum_{I\subseteq S}\left(  -1\right)  ^{\left\vert I\setminus T\right\vert
}q^{\left\vert I\cap T\right\vert }\\
&  =\underbrace{\sum_{\substack{I\subseteq S;\\s\in I}}}_{=\sum_{I\in\left\{
J\subseteq S\ \mid\ s\in J\right\}  }}\left(  -1\right)  ^{\left\vert
I\setminus T\right\vert }q^{\left\vert I\cap T\right\vert }+\underbrace{\sum
_{\substack{I\subseteq S;\\s\notin I}}}_{=\sum_{I\in\left\{  J\subseteq
S\ \mid\ s\notin J\right\}  }}\left(  -1\right)  ^{\left\vert I\setminus
T\right\vert }q^{\left\vert I\cap T\right\vert }\\
&  =\sum_{I\in\left\{  J\subseteq S\ \mid\ s\in J\right\}  }\left(  -1\right)
^{\left\vert I\setminus T\right\vert }q^{\left\vert I\cap T\right\vert }%
+\sum_{I\in\left\{  J\subseteq S\ \mid\ s\notin J\right\}  }\left(  -1\right)
^{\left\vert I\setminus T\right\vert }q^{\left\vert I\cap T\right\vert }.
\end{align*}

However, if $I\in\left\{  J\subseteq S\ \mid\ s\notin J\right\}  $, then
$I\cup\left\{  s\right\}  \in\left\{  J\subseteq S\ \mid\ s\in J\right\}
$\ \ \ \ \footnote{\textit{Proof.} Let $I\in\left\{  J\subseteq S\ \mid
\ s\notin J\right\}  $. Thus, $I$ is a subset $J$ of $S$ satisfying $s\notin
J$. In other words, $I$ is a subset of $S$, and we have $s\notin I$.
Furthermore, $\left\{  s\right\}  $ is a subset of $S$ (since $s\in S$).
\par
Therefore, both $I$ and $\left\{  s\right\}  $ are subsets of $S$. Thus, their
union $I\cup\left\{  s\right\}  $ is a subset of $S$ as well. Hence,
$I\cup\left\{  s\right\}  $ is a subset $J$ of $S$ satisfying $s\in J$ (since
$s\in\left\{  s\right\}  \subseteq I\cup\left\{  s\right\}  $). In other
words, $I\cup\left\{  s\right\}  \in\left\{  J\subseteq S\ \mid\ s\in
J\right\}  $. Qed.}. Hence, we can define a map%
\begin{align*}
\Phi:\left\{  J\subseteq S\ \mid\ s\notin J\right\}   &  \rightarrow\left\{
J\subseteq S\ \mid\ s\in J\right\}  ,\\
I  &  \mapsto I\cup\left\{  s\right\}  .
\end{align*}
Consider this map $\Phi$.

Furthermore, if $I\in\left\{  J\subseteq S\ \mid\ s\in J\right\}  $, then
$I\setminus\left\{  s\right\}  \in\left\{  J\subseteq S\ \mid\ s\notin
J\right\}  $\ \ \ \ \footnote{\textit{Proof.} Let $I\in\left\{  J\subseteq
S\ \mid\ s\in J\right\}  $. Thus, $I$ is a subset $J$ of $S$ satisfying $s\in
J$. In other words, $I$ is a subset of $S$, and we have $s\in I$.
\par
We have $I\subseteq S$ (since $I$ is a subset of $S$) and thus $I\setminus
\left\{  s\right\}  \subseteq I\subseteq S$. In other words, $I\setminus
\left\{  s\right\}  $ is a subset of $S$. Thus, $I\setminus\left\{  s\right\}
$ is a subset $J$ of $S$ satisfying $s\notin J$ (since $s\notin I\setminus
\left\{  s\right\}  $ (because $s\in\left\{  s\right\}  $)). In other words,
$I\setminus\left\{  s\right\}  \in\left\{  J\subseteq S\ \mid\ s\notin
J\right\}  $. Qed.}. Hence, we can define a map%
\begin{align*}
\Psi:\left\{  J\subseteq S\ \mid\ s\in J\right\}   &  \rightarrow\left\{
J\subseteq S\ \mid\ s\notin J\right\}  ,\\
I  &  \mapsto I\setminus\left\{  s\right\}  .
\end{align*}
Consider this map $\Psi$.

We have $\Phi\circ\Psi=\operatorname*{id}$\ \ \ \ \footnote{\textit{Proof.}
Let $I\in\left\{  J\subseteq S\ \mid\ s\in J\right\}  $. Thus, $I$ is a subset
$J$ of $S$ satisfying $s\in J$. In other words, $I$ is a subset of $S$, and we
have $s\in I$. Furthermore, the definition of $\Phi$ yields$\ $
\[
\Phi\left(  \Psi\left(  I\right)  \right)  =\underbrace{\left(  \Psi\left(
I\right)  \right)  }_{\substack{=I\setminus\left\{  s\right\}  \\\text{(by the
definition of }\Psi\text{)}}}\cup\left\{  s\right\}  =\left(  I\setminus
\left\{  s\right\}  \right)  \cup\left\{  s\right\}  =I
\]
(since $s\in I$). Hence, $\left(  \Phi\circ\Psi\right)  \left(  I\right)
=\Phi\left(  \Psi\left(  I\right)  \right)  =I=\operatorname*{id}\left(
I\right)  $.
\par
Forget that we fixed $I$. We thus have shown that $\left(  \Phi\circ
\Psi\right)  \left(  I\right)  =\operatorname*{id}\left(  I\right)  $ for each
$I\in\left\{  J\subseteq S\ \mid\ s\in J\right\}  $. In other words,
$\Phi\circ\Psi=\operatorname*{id}$.} and $\Psi\circ\Phi=\operatorname*{id}%
$\ \ \ \ \footnote{\textit{Proof.} Let $I\in\left\{  J\subseteq S\ \mid
\ s\notin J\right\}  $. Thus, $I$ is a subset $J$ of $S$ satisfying $s\notin
J$. In other words, $I$ is a subset of $S$, and we have $s\notin I$.
Furthermore, the definition of $\Psi$ yields$\ $
\[
\Psi\left(  \Phi\left(  I\right)  \right)  =\underbrace{\left(  \Phi\left(
I\right)  \right)  }_{\substack{=I\cup\left\{  s\right\}  \\\text{(by the
definition of }\Phi\text{)}}}\setminus\left\{  s\right\}  =\left(
I\cup\left\{  s\right\}  \right)  \setminus\left\{  s\right\}  =I
\]
(since $s\notin I$). Hence, $\left(  \Psi\circ\Phi\right)  \left(  I\right)
=\Psi\left(  \Phi\left(  I\right)  \right)  =I=\operatorname*{id}\left(
I\right)  $.
\par
Forget that we fixed $I$. We thus have shown that $\left(  \Psi\circ
\Phi\right)  \left(  I\right)  =\operatorname*{id}\left(  I\right)  $ for each
$I\in\left\{  J\subseteq S\ \mid\ s\notin J\right\}  $. In other words,
$\Psi\circ\Phi=\operatorname*{id}$.}. Hence, the two maps $\Phi$ and $\Psi$
are mutually inverse. Thus, the map $\Phi$ is invertible, i.e., is a bijection.

Moreover, each $I\in\left\{  J\subseteq S\ \mid\ s\notin J\right\}  $
satisfies%
\begin{equation}
\left(  -1\right)  ^{\left\vert \Phi\left(  I\right)  \setminus T\right\vert
}=-\left(  -1\right)  ^{\left\vert I\setminus T\right\vert }
\label{pf.lem.eta.through-F.lem1.long.8}%
\end{equation}
\footnote{\textit{Proof.} Let $I\in\left\{  J\subseteq S\ \mid\ s\notin
J\right\}  $. Thus, $I$ is a subset $J$ of $S$ satisfying $s\notin J$. In
other words, $I$ is a subset of $S$, and we have $s\notin I$.
\par
We have $s\notin I\setminus T$ (since $s\in I\setminus T$ would yield $s\in
I\setminus T\subseteq I$, which would contradict $s\notin I$).
\par
The definition of $\Phi$ yields $\Phi\left(  I\right)  =I\cup\left\{
s\right\}  $. Hence,%
\begin{align*}
\underbrace{\Phi\left(  I\right)  }_{=I\cup\left\{  s\right\}  }\setminus T
&  =\left(  I\cup\left\{  s\right\}  \right)  \setminus T\\
&  =\left(  I\setminus T\right)  \cup\underbrace{\left(  \left\{  s\right\}
\setminus T\right)  }_{\substack{=\left\{  s\right\}  \\\text{(since }s\notin
T\text{)}}}\ \ \ \ \ \ \ \ \ \ \left(
\begin{array}
[c]{c}%
\text{by the rule }\left(  A\cup B\right)  \setminus C=\left(  A\setminus
C\right)  \cup\left(  B\setminus C\right)  \text{,}\\
\text{which holds for any three sets }A\text{, }B\text{ and }C
\end{array}
\right) \\
&  =\left(  I\setminus T\right)  \cup\left\{  s\right\}  .
\end{align*}
Thus,%
\[
\left\vert \Phi\left(  I\right)  \setminus T\right\vert =\left\vert \left(
I\setminus T\right)  \cup\left\{  s\right\}  \right\vert =\left\vert
I\setminus T\right\vert +1
\]
(since $s\notin I\setminus T$). Hence, $\left(  -1\right)  ^{\left\vert
\Phi\left(  I\right)  \setminus T\right\vert }=\left(  -1\right)  ^{\left\vert
I\setminus T\right\vert +1}=-\left(  -1\right)  ^{\left\vert I\setminus
T\right\vert }$, qed.} and%
\begin{equation}
q^{\left\vert \Phi\left(  I\right)  \cap T\right\vert }=q^{\left\vert I\cap
T\right\vert } \label{pf.lem.eta.through-F.lem1.long.9}%
\end{equation}
\footnote{\textit{Proof.} Let $I\in\left\{  J\subseteq S\ \mid\ s\notin
J\right\}  $. Thus, $I$ is a subset $J$ of $S$ satisfying $s\notin J$. In
other words, $I$ is a subset of $S$, and we have $s\notin I$.
\par
The definition of $\Phi$ yields $\Phi\left(  I\right)  =I\cup\left\{
s\right\}  $. Hence,%
\begin{align*}
\underbrace{\Phi\left(  I\right)  }_{=I\cup\left\{  s\right\}  }\cap T  &
=\left(  I\cup\left\{  s\right\}  \right)  \cap T\\
&  =\left(  I\cap T\right)  \cup\underbrace{\left(  \left\{  s\right\}  \cap
T\right)  }_{\substack{=\varnothing\\\text{(since }s\notin T\text{)}%
}}\ \ \ \ \ \ \ \ \ \ \left(
\begin{array}
[c]{c}%
\text{by the rule }\left(  A\cup B\right)  \cap C=\left(  A\cap C\right)
\cup\left(  B\cap C\right)  \text{,}\\
\text{which holds for any three sets }A\text{, }B\text{ and }C
\end{array}
\right) \\
&  =I\cap T.
\end{align*}
Thus, $q^{\left\vert \Phi\left(  I\right)  \cap T\right\vert }=q^{\left\vert
I\cap T\right\vert }$, qed.}.

Now,%
\begin{align*}
&  \sum_{I\in\left\{  J\subseteq S\ \mid\ s\in J\right\}  }\left(  -1\right)
^{\left\vert I\setminus T\right\vert }q^{\left\vert I\cap T\right\vert }\\
&  =\sum_{I\in\left\{  J\subseteq S\ \mid\ s\notin J\right\}  }%
\underbrace{\left(  -1\right)  ^{\left\vert \Phi\left(  I\right)  \setminus
T\right\vert }}_{\substack{=-\left(  -1\right)  ^{\left\vert I\setminus
T\right\vert }\\\text{(by (\ref{pf.lem.eta.through-F.lem1.long.8}))}%
}}\underbrace{q^{\left\vert \Phi\left(  I\right)  \cap T\right\vert }%
}_{\substack{=q^{\left\vert I\cap T\right\vert }\\\text{(by
(\ref{pf.lem.eta.through-F.lem1.long.9}))}}}\\
&  \ \ \ \ \ \ \ \ \ \ \ \ \ \ \ \ \ \ \ \ \left(
\begin{array}
[c]{c}%
\text{here, we have substituted }\Phi\left(  I\right)  \text{ for }I\text{ in
the sum,}\\
\text{since the map }\Phi:\left\{  J\subseteq S\ \mid\ s\notin J\right\}
\rightarrow\left\{  J\subseteq S\ \mid\ s\in J\right\} \\
\text{is a bijection}%
\end{array}
\right) \\
&  =\sum_{I\in\left\{  J\subseteq S\ \mid\ s\notin J\right\}  }\left(
-\left(  -1\right)  ^{\left\vert I\setminus T\right\vert }\right)
q^{\left\vert I\cap T\right\vert }=-\sum_{I\in\left\{  J\subseteq
S\ \mid\ s\notin J\right\}  }\left(  -1\right)  ^{\left\vert I\setminus
T\right\vert }q^{\left\vert I\cap T\right\vert }.
\end{align*}

Now, recall that%
\begin{align*}
&  \sum_{I\subseteq S}\left(  -1\right)  ^{\left\vert I\setminus T\right\vert
}q^{\left\vert I\cap T\right\vert }\\
&  =\underbrace{\sum_{I\in\left\{  J\subseteq S\ \mid\ s\in J\right\}
}\left(  -1\right)  ^{\left\vert I\setminus T\right\vert }q^{\left\vert I\cap
T\right\vert }}_{=-\sum_{I\in\left\{  J\subseteq S\ \mid\ s\notin J\right\}
}\left(  -1\right)  ^{\left\vert I\setminus T\right\vert }q^{\left\vert I\cap
T\right\vert }}+\sum_{I\in\left\{  J\subseteq S\ \mid\ s\notin J\right\}
}\left(  -1\right)  ^{\left\vert I\setminus T\right\vert }q^{\left\vert I\cap
T\right\vert }\\
&  =-\sum_{I\in\left\{  J\subseteq S\ \mid\ s\notin J\right\}  }\left(
-1\right)  ^{\left\vert I\setminus T\right\vert }q^{\left\vert I\cap
T\right\vert }+\sum_{I\in\left\{  J\subseteq S\ \mid\ s\notin J\right\}
}\left(  -1\right)  ^{\left\vert I\setminus T\right\vert }q^{\left\vert I\cap
T\right\vert }=0.
\end{align*}
Comparing this with%
\[
\underbrace{\left[  S\subseteq T\right]  }_{\substack{=0\\\text{(since
}S\not \subseteq T\text{)}}}\cdot\,r^{\left\vert S\right\vert }=0,
\]
we obtain%
\[
\sum_{I\subseteq S}\left(  -1\right)  ^{\left\vert I\setminus T\right\vert
}q^{\left\vert I\cap T\right\vert }=\left[  S\subseteq T\right]  \cdot
r^{\left\vert S\right\vert }.
\]
Thus, Lemma \ref{lem.eta.through-F.lem1} is proved in Case 2.

We have now proved Lemma \ref{lem.eta.through-F.lem1} in both Cases 1 and 2.
The proof of Lemma \ref{lem.eta.through-F.lem1} is thus complete.
\end{proof}
\end{verlong}

\begin{vershort}

\begin{proof}
[Proof of Proposition \ref{prop.eta.through-F}.]We begin by observing that%
\begin{equation}
\left\vert D\left(  \beta\right)  \right\vert +1=\ell\left(  \beta\right)
\label{pf.prop.eta.through-F.short.Dl}%
\end{equation}
for every $\beta\in\operatorname*{Comp}\nolimits_{n}$%
\ \ \ \ \footnote{\textit{Proof.} Let $\beta\in\operatorname*{Comp}%
\nolimits_{n}$. From $n\neq0$ (since $n$ is positive), we obtain $\left[
n\neq0\right]  =1$. However, Lemma \ref{lem.comps.l-vs-size} \textbf{(a)}
(applied to $\delta=\beta$) yields $\ell\left(  \beta\right)  =\left\vert
D\left(  \beta\right)  \right\vert +\underbrace{\left[  n\neq0\right]  }%
_{=1}=\left\vert D\left(  \beta\right)  \right\vert +1$. This proves
(\ref{pf.prop.eta.through-F.short.Dl}).}.

Let $T:=D\left(  \alpha\right)  $. Thus, $D\left(  \alpha\right)  =T$, so that%
\begin{align*}
&  r\sum_{\gamma\in\operatorname*{Comp}\nolimits_{n}}\left(  -1\right)
^{\left\vert D\left(  \gamma\right)  \setminus D\left(  \alpha\right)
\right\vert }q^{\left\vert D\left(  \gamma\right)  \cap D\left(
\alpha\right)  \right\vert }L_{\gamma}\\
&  =r\sum_{\gamma\in\operatorname*{Comp}\nolimits_{n}}\left(  -1\right)
^{\left\vert D\left(  \gamma\right)  \setminus T\right\vert }q^{\left\vert
D\left(  \gamma\right)  \cap T\right\vert }\underbrace{L_{\gamma}%
}_{\substack{=\sum_{\substack{\beta\in\operatorname*{Comp}\nolimits_{n}%
;\\D\left(  \beta\right)  \supseteq D\left(  \gamma\right)  }}M_{\beta
}\\\text{(by the definition of }L_{\gamma}\text{)}}}\\
&  =r\sum_{\gamma\in\operatorname*{Comp}\nolimits_{n}}\left(  -1\right)
^{\left\vert D\left(  \gamma\right)  \setminus T\right\vert }q^{\left\vert
D\left(  \gamma\right)  \cap T\right\vert }\sum_{\substack{\beta
\in\operatorname*{Comp}\nolimits_{n};\\D\left(  \beta\right)  \supseteq
D\left(  \gamma\right)  }}M_{\beta}\\
&  =r\sum_{\beta\in\operatorname*{Comp}\nolimits_{n}}\ \ \sum
_{\substack{\gamma\in\operatorname*{Comp}\nolimits_{n};\\D\left(
\beta\right)  \supseteq D\left(  \gamma\right)  }}\left(  -1\right)
^{\left\vert D\left(  \gamma\right)  \setminus T\right\vert }q^{\left\vert
D\left(  \gamma\right)  \cap T\right\vert }M_{\beta}.
\end{align*}
However, every $\beta\in\operatorname*{Comp}\nolimits_{n}$ satisfies%
\begin{align*}
&  \sum_{\substack{\gamma\in\operatorname*{Comp}\nolimits_{n};\\D\left(
\beta\right)  \supseteq D\left(  \gamma\right)  }}\left(  -1\right)
^{\left\vert D\left(  \gamma\right)  \setminus T\right\vert }q^{\left\vert
D\left(  \gamma\right)  \cap T\right\vert }\\
&  =\sum_{\substack{I\subseteq\left[  n-1\right]  ;\\D\left(  \beta\right)
\supseteq I}}\left(  -1\right)  ^{\left\vert I\setminus T\right\vert
}q^{\left\vert I\cap T\right\vert }\\
&  \ \ \ \ \ \ \ \ \ \ \ \ \ \ \ \ \ \ \ \ \left(
\begin{array}
[c]{c}%
\text{here, we have substituted }I\text{ for }D\left(  \gamma\right)  \text{
in the sum,}\\
\text{since the map }D:\operatorname*{Comp}\nolimits_{n}\rightarrow
\mathcal{P}\left(  \left[  n-1\right]  \right)  \text{ is a bijection}%
\end{array}
\right) \\
&  =\sum_{I\subseteq D\left(  \beta\right)  }\left(  -1\right)  ^{\left\vert
I\setminus T\right\vert }q^{\left\vert I\cap T\right\vert }%
\ \ \ \ \ \ \ \ \ \ \left(  \text{since }D\left(  \beta\right)  \subseteq
\left[  n-1\right]  \right) \\
&  =\left[  D\left(  \beta\right)  \subseteq T\right]  \cdot r^{\left\vert
D\left(  \beta\right)  \right\vert }%
\end{align*}
(by Lemma \ref{lem.eta.through-F.lem1}, applied to $S=D\left(  \beta\right)
$). Hence, this becomes%
\begin{align*}
&  r\sum_{\gamma\in\operatorname*{Comp}\nolimits_{n}}\left(  -1\right)
^{\left\vert D\left(  \gamma\right)  \setminus D\left(  \alpha\right)
\right\vert }q^{\left\vert D\left(  \gamma\right)  \cap D\left(
\alpha\right)  \right\vert }L_{\gamma}\\
&  =r\sum_{\beta\in\operatorname*{Comp}\nolimits_{n}}\ \ \underbrace{\sum
_{\substack{\gamma\in\operatorname*{Comp}\nolimits_{n};\\D\left(
\beta\right)  \supseteq D\left(  \gamma\right)  }}\left(  -1\right)
^{\left\vert D\left(  \gamma\right)  \setminus T\right\vert }q^{\left\vert
D\left(  \gamma\right)  \cap T\right\vert }}_{=\left[  D\left(  \beta\right)
\subseteq T\right]  \cdot r^{\left\vert D\left(  \beta\right)  \right\vert }%
}M_{\beta}\\
&  =r\sum_{\beta\in\operatorname*{Comp}\nolimits_{n}}\left[  D\left(
\beta\right)  \subseteq T\right]  \cdot r^{\left\vert D\left(  \beta\right)
\right\vert }M_{\beta}=r\sum_{\substack{\beta\in\operatorname*{Comp}%
\nolimits_{n};\\D\left(  \beta\right)  \subseteq T}}r^{\left\vert D\left(
\beta\right)  \right\vert }M_{\beta}\\
&  =\sum_{\substack{\beta\in\operatorname*{Comp}\nolimits_{n};\\D\left(
\beta\right)  \subseteq T}}\underbrace{r^{\left\vert D\left(  \beta\right)
\right\vert +1}}_{\substack{=r^{\ell\left(  \beta\right)  }\\\text{(by
(\ref{pf.prop.eta.through-F.short.Dl}))}}}M_{\beta}=\sum_{\substack{\beta
\in\operatorname*{Comp}\nolimits_{n};\\D\left(  \beta\right)  \subseteq
T}}r^{\ell\left(  \beta\right)  }M_{\beta}\\
&  =\sum_{\substack{\beta\in\operatorname*{Comp}\nolimits_{n};\\D\left(
\beta\right)  \subseteq D\left(  \alpha\right)  }}r^{\ell\left(  \beta\right)
}M_{\beta}\ \ \ \ \ \ \ \ \ \ \left(  \text{since }T=D\left(  \alpha\right)
\right) \\
&  =\eta_{\alpha}^{\left(  q\right)  }\ \ \ \ \ \ \ \ \ \ \left(  \text{by the
definition of }\eta_{\alpha}^{\left(  q\right)  }\right)  .
\end{align*}
This proves Proposition \ref{prop.eta.through-F}.
\end{proof}
\end{vershort}

\begin{verlong}

\begin{proof}
[Proof of Proposition \ref{prop.eta.through-F}.]Let $T:=D\left(
\alpha\right)  $. Thus, $D\left(  \alpha\right)  =T$, so that%
\begin{align}
&  r\sum_{\gamma\in\operatorname*{Comp}\nolimits_{n}}\left(  -1\right)
^{\left\vert D\left(  \gamma\right)  \setminus D\left(  \alpha\right)
\right\vert }q^{\left\vert D\left(  \gamma\right)  \cap D\left(
\alpha\right)  \right\vert }L_{\gamma}\nonumber\\
&  =r\sum_{\gamma\in\operatorname*{Comp}\nolimits_{n}}\left(  -1\right)
^{\left\vert D\left(  \gamma\right)  \setminus T\right\vert }q^{\left\vert
D\left(  \gamma\right)  \cap T\right\vert }\underbrace{L_{\gamma}%
}_{\substack{=\sum_{\substack{\beta\in\operatorname*{Comp}\nolimits_{n}%
;\\D\left(  \beta\right)  \supseteq D\left(  \gamma\right)  }}M_{\beta
}\\\text{(by the definition of }L_{\gamma}\text{)}}}\nonumber\\
&  =r\sum_{\gamma\in\operatorname*{Comp}\nolimits_{n}}\left(  -1\right)
^{\left\vert D\left(  \gamma\right)  \setminus T\right\vert }q^{\left\vert
D\left(  \gamma\right)  \cap T\right\vert }\sum_{\substack{\beta
\in\operatorname*{Comp}\nolimits_{n};\\D\left(  \beta\right)  \supseteq
D\left(  \gamma\right)  }}M_{\beta}\nonumber\\
&  =r\underbrace{\sum_{\gamma\in\operatorname*{Comp}\nolimits_{n}}%
\ \ \sum_{\substack{\beta\in\operatorname*{Comp}\nolimits_{n};\\D\left(
\beta\right)  \supseteq D\left(  \gamma\right)  }}}_{=\sum_{\beta
\in\operatorname*{Comp}\nolimits_{n}}\ \ \sum_{\substack{\gamma\in
\operatorname*{Comp}\nolimits_{n};\\D\left(  \beta\right)  \supseteq D\left(
\gamma\right)  }}}\left(  -1\right)  ^{\left\vert D\left(  \gamma\right)
\setminus T\right\vert }q^{\left\vert D\left(  \gamma\right)  \cap
T\right\vert }M_{\beta}\nonumber\\
&  =r\sum_{\beta\in\operatorname*{Comp}\nolimits_{n}}\ \ \sum
_{\substack{\gamma\in\operatorname*{Comp}\nolimits_{n};\\D\left(
\beta\right)  \supseteq D\left(  \gamma\right)  }}\left(  -1\right)
^{\left\vert D\left(  \gamma\right)  \setminus T\right\vert }q^{\left\vert
D\left(  \gamma\right)  \cap T\right\vert }M_{\beta}\nonumber\\
&  =r\sum_{\beta\in\operatorname*{Comp}\nolimits_{n}}\left(  \sum
_{\substack{\gamma\in\operatorname*{Comp}\nolimits_{n};\\D\left(
\beta\right)  \supseteq D\left(  \gamma\right)  }}\left(  -1\right)
^{\left\vert D\left(  \gamma\right)  \setminus T\right\vert }q^{\left\vert
D\left(  \gamma\right)  \cap T\right\vert }\right)  M_{\beta}.
\label{prop.eta.through-F.long.1}%
\end{align}
Now, let $\beta\in\operatorname*{Comp}\nolimits_{n}$ be arbitrary. Recall that
$D:\operatorname*{Comp}\nolimits_{n}\rightarrow\mathcal{P}\left(  \left[
n-1\right]  \right)  $ is a bijection. Hence, from $\beta\in
\operatorname*{Comp}\nolimits_{n}$, we obtain $D\left(  \beta\right)
\in\mathcal{P}\left(  \left[  n-1\right]  \right)  $. In other words,
$D\left(  \beta\right)  \subseteq\left[  n-1\right]  $. Furthermore,%
\begin{align}
&  \sum_{\substack{\gamma\in\operatorname*{Comp}\nolimits_{n};\\D\left(
\beta\right)  \supseteq D\left(  \gamma\right)  }}\left(  -1\right)
^{\left\vert D\left(  \gamma\right)  \setminus T\right\vert }q^{\left\vert
D\left(  \gamma\right)  \cap T\right\vert }\nonumber\\
&  =\underbrace{\sum_{\substack{I\in\mathcal{P}\left(  \left[  n-1\right]
\right)  ;\\D\left(  \beta\right)  \supseteq I}}}_{\substack{=\sum
_{\substack{I\subseteq\left[  n-1\right]  ;\\D\left(  \beta\right)  \supseteq
I}}\\\text{(here, we have replaced}\\\text{the condition \textquotedblleft%
}I\in\mathcal{P}\left(  \left[  n-1\right]  \right)  \text{\textquotedblright%
}\\\text{under the summation sign}\\\text{by the equivalent}\\\text{condition
\textquotedblleft}I\subseteq\left[  n-1\right]  \text{\textquotedblright)}%
}}\left(  -1\right)  ^{\left\vert I\setminus T\right\vert }q^{\left\vert I\cap
T\right\vert }\nonumber\\
&  \ \ \ \ \ \ \ \ \ \ \ \ \ \ \ \ \ \ \ \ \left(
\begin{array}
[c]{c}%
\text{here, we have substituted }I\text{ for }D\left(  \gamma\right)  \text{
in the sum,}\\
\text{since the map }D:\operatorname*{Comp}\nolimits_{n}\rightarrow
\mathcal{P}\left(  \left[  n-1\right]  \right)  \text{ is a bijection}%
\end{array}
\right) \nonumber\\
&  =\underbrace{\sum_{\substack{I\subseteq\left[  n-1\right]  ;\\D\left(
\beta\right)  \supseteq I}}}_{\substack{=\sum_{\substack{I\subseteq\left[
n-1\right]  ;\\I\subseteq D\left(  \beta\right)  }}\\\text{(here, we have
replaced}\\\text{the condition \textquotedblleft}D\left(  \beta\right)
\supseteq I\text{\textquotedblright}\\\text{under the summation sign}%
\\\text{by the equivalent}\\\text{condition \textquotedblleft}I\subseteq
D\left(  \beta\right)  \text{\textquotedblright)}}}\left(  -1\right)
^{\left\vert I\setminus T\right\vert }q^{\left\vert I\cap T\right\vert
}\nonumber\\
&  =\underbrace{\sum_{\substack{I\subseteq\left[  n-1\right]  ;\\I\subseteq
D\left(  \beta\right)  }}}_{\substack{=\sum_{I\subseteq D\left(  \beta\right)
}\\\text{(since }D\left(  \beta\right)  \subseteq\left[  n-1\right]  \text{)}%
}}\left(  -1\right)  ^{\left\vert I\setminus T\right\vert }q^{\left\vert I\cap
T\right\vert }\nonumber\\
&  =\sum_{I\subseteq D\left(  \beta\right)  }\left(  -1\right)  ^{\left\vert
I\setminus T\right\vert }q^{\left\vert I\cap T\right\vert }\nonumber\\
&  =\left[  D\left(  \beta\right)  \subseteq T\right]  \cdot r^{\left\vert
D\left(  \beta\right)  \right\vert } \label{prop.eta.through-F.long.2}%
\end{align}
(by Lemma \ref{lem.eta.through-F.lem1}, applied to $S=D\left(  \beta\right)  $).

Also, $n\neq0$ (since $n$ is positive), so that $\left[  n\neq0\right]  =1$.
However, Lemma \ref{lem.comps.l-vs-size} \textbf{(a)} (applied to
$\delta=\beta$) yields $\ell\left(  \beta\right)  =\left\vert D\left(
\beta\right)  \right\vert +\underbrace{\left[  n\neq0\right]  }_{=1}%
=\left\vert D\left(  \beta\right)  \right\vert +1$, so that%
\begin{equation}
r^{\ell\left(  \beta\right)  }=r^{\left\vert D\left(  \beta\right)
\right\vert +1}=rr^{\left\vert D\left(  \beta\right)  \right\vert }.
\label{pf.lem.eta.through-F.lem1.3}%
\end{equation}

Forget that we fixed $\beta$. We thus have proved the equalities
(\ref{prop.eta.through-F.long.2}) and (\ref{pf.lem.eta.through-F.lem1.3}) for
each $\beta\in\operatorname*{Comp}\nolimits_{n}$.

Now, (\ref{prop.eta.through-F.long.1}) becomes%
\begin{align*}
&  r\sum_{\gamma\in\operatorname*{Comp}\nolimits_{n}}\left(  -1\right)
^{\left\vert D\left(  \gamma\right)  \setminus D\left(  \alpha\right)
\right\vert }q^{\left\vert D\left(  \gamma\right)  \cap D\left(
\alpha\right)  \right\vert }L_{\gamma}\\
&  =r\sum_{\beta\in\operatorname*{Comp}\nolimits_{n}}\underbrace{\left(
\sum_{\substack{\gamma\in\operatorname*{Comp}\nolimits_{n};\\D\left(
\beta\right)  \supseteq D\left(  \gamma\right)  }}\left(  -1\right)
^{\left\vert D\left(  \gamma\right)  \setminus T\right\vert }q^{\left\vert
D\left(  \gamma\right)  \cap T\right\vert }\right)  }_{\substack{=\left[
D\left(  \beta\right)  \subseteq T\right]  \cdot r^{\left\vert D\left(
\beta\right)  \right\vert }\\\text{(by (\ref{prop.eta.through-F.long.2}))}%
}}M_{\beta}\\
&  =r\sum_{\beta\in\operatorname*{Comp}\nolimits_{n}}\left[  D\left(
\beta\right)  \subseteq T\right]  \cdot r^{\left\vert D\left(  \beta\right)
\right\vert }M_{\beta}=\sum_{\beta\in\operatorname*{Comp}\nolimits_{n}}\left[
D\left(  \beta\right)  \subseteq T\right]  \cdot\underbrace{rr^{\left\vert
D\left(  \beta\right)  \right\vert }}_{\substack{=r^{\ell\left(  \beta\right)
}\\\text{(by (\ref{pf.lem.eta.through-F.lem1.3}))}}}M_{\beta}\\
&  =\sum_{\beta\in\operatorname*{Comp}\nolimits_{n}}\left[  D\left(
\beta\right)  \subseteq T\right]  \cdot r^{\ell\left(  \beta\right)  }%
M_{\beta}\\
&  =\sum_{\substack{\beta\in\operatorname*{Comp}\nolimits_{n};\\D\left(
\beta\right)  \subseteq T}}\underbrace{\left[  D\left(  \beta\right)
\subseteq T\right]  }_{\substack{=1\\\text{(since }D\left(  \beta\right)
\subseteq T\text{)}}}\cdot\,r^{\ell\left(  \beta\right)  }M_{\beta}%
+\sum_{\substack{\beta\in\operatorname*{Comp}\nolimits_{n};\\\text{not
}D\left(  \beta\right)  \subseteq T}}\ \ \underbrace{\left[  D\left(
\beta\right)  \subseteq T\right]  }_{\substack{=0\\\text{(since we
don't}\\\text{have }D\left(  \beta\right)  \subseteq T\text{)}}}\cdot
\,r^{\ell\left(  \beta\right)  }M_{\beta}\\
&  \ \ \ \ \ \ \ \ \ \ \ \ \ \ \ \ \ \ \ \ \left(
\begin{array}
[c]{c}%
\text{since each }\beta\in\operatorname*{Comp}\nolimits_{n}\text{ satisfies
either }D\left(  \beta\right)  \subseteq T\\
\text{or }\left(  \text{not }D\left(  \beta\right)  \subseteq T\right)
\text{, but not both simultaneously}%
\end{array}
\right) \\
&  =\sum_{\substack{\beta\in\operatorname*{Comp}\nolimits_{n};\\D\left(
\beta\right)  \subseteq T}}r^{\ell\left(  \beta\right)  }M_{\beta
}+\underbrace{\sum_{\substack{\beta\in\operatorname*{Comp}\nolimits_{n}%
;\\\text{not }D\left(  \beta\right)  \subseteq T}}0\cdot r^{\ell\left(
\beta\right)  }M_{\beta}}_{=0}\\
&  =\sum_{\substack{\beta\in\operatorname*{Comp}\nolimits_{n};\\D\left(
\beta\right)  \subseteq T}}r^{\ell\left(  \beta\right)  }M_{\beta}%
=\sum_{\substack{\beta\in\operatorname*{Comp}\nolimits_{n};\\D\left(
\beta\right)  \subseteq D\left(  \alpha\right)  }}r^{\ell\left(  \beta\right)
}M_{\beta}\ \ \ \ \ \ \ \ \ \ \left(  \text{since }T=D\left(  \alpha\right)
\right) \\
&  =\eta_{\alpha}^{\left(  q\right)  }\ \ \ \ \ \ \ \ \ \ \left(  \text{by
(\ref{eq.def.etaalpha.def})}\right)  .
\end{align*}
This proves Proposition \ref{prop.eta.through-F}.
\end{proof}
\end{verlong}

\begin{vershort}

\begin{proof}
[Proof of Proposition \ref{prop.eta.F-through}.]For each subset $J$ of
$\left[  n-1\right]  $, we let $\overline{J}$ denote its complement $\left[
n-1\right]  \setminus J$. It is easy to see that%
\begin{equation}
\left\vert \overline{D\left(  \beta\right)  }\right\vert =n-\ell\left(
\beta\right)  \label{pf.prop.eta.F-through.size-of-comp}%
\end{equation}
for every $\beta\in\operatorname*{Comp}\nolimits_{n}$%
\ \ \ \ \footnote{\textit{Proof.} Let $\beta\in\operatorname*{Comp}%
\nolimits_{n}$. From $n\neq0$ (since $n$ is positive), we obtain $\left[
n\neq0\right]  =1$. However, Lemma \ref{lem.comps.l-vs-size} \textbf{(a)}
(applied to $\delta=\beta$) yields $\ell\left(  \beta\right)  =\left\vert
D\left(  \beta\right)  \right\vert +\underbrace{\left[  n\neq0\right]  }%
_{=1}=\left\vert D\left(  \beta\right)  \right\vert +1$, so that $\left\vert
D\left(  \beta\right)  \right\vert =\ell\left(  \beta\right)  -1$. However,
the definition of $\overline{D\left(  \beta\right)  }$ yields $\overline
{D\left(  \beta\right)  }=\left[  n-1\right]  \setminus D\left(  \beta\right)
$. Hence,%
\begin{align*}
\left\vert \overline{D\left(  \beta\right)  }\right\vert  &  =\left\vert
\left[  n-1\right]  \setminus D\left(  \beta\right)  \right\vert \\
&  =\underbrace{\left\vert \left[  n-1\right]  \right\vert }%
_{\substack{=n-1\\\text{(since }n-1\in\mathbb{N}\text{)}}%
}-\underbrace{\left\vert D\left(  \beta\right)  \right\vert }_{=\ell\left(
\beta\right)  -1}\ \ \ \ \ \ \ \ \ \ \left(  \text{since }D\left(
\beta\right)  \subseteq\left[  n-1\right]  \right) \\
&  =\left(  n-1\right)  -\left(  \ell\left(  \beta\right)  -1\right)
=n-\ell\left(  \beta\right)  .
\end{align*}
This proves (\ref{pf.prop.eta.F-through.size-of-comp}).}.

Let $T:=\overline{D\left(  \gamma\right)  }$. Thus, $D\left(  \gamma\right)
=\overline{T}$, so that%
\begin{align*}
&  \sum_{\alpha\in\operatorname*{Comp}\nolimits_{n}}\left(  -1\right)
^{\left\vert D\left(  \gamma\right)  \setminus D\left(  \alpha\right)
\right\vert }q^{\left\vert \left[  n-1\right]  \setminus\left(  D\left(
\gamma\right)  \cup D\left(  \alpha\right)  \right)  \right\vert }\eta
_{\alpha}^{\left(  q\right)  }\\
&  =\sum_{\alpha\in\operatorname*{Comp}\nolimits_{n}}\left(  -1\right)
^{\left\vert \overline{T}\setminus D\left(  \alpha\right)  \right\vert
}q^{\left\vert \left[  n-1\right]  \setminus\left(  \overline{T}\cup D\left(
\alpha\right)  \right)  \right\vert }\underbrace{\eta_{\alpha}^{\left(
q\right)  }}_{\substack{=\sum_{\substack{\beta\in\operatorname*{Comp}%
\nolimits_{n};\\D\left(  \beta\right)  \subseteq D\left(  \alpha\right)
}}r^{\ell\left(  \beta\right)  }M_{\beta}\\\text{(by the definition of }%
\eta_{\alpha}^{\left(  q\right)  }\text{)}}}\\
&  =\sum_{\alpha\in\operatorname*{Comp}\nolimits_{n}}\left(  -1\right)
^{\left\vert \overline{T}\setminus D\left(  \alpha\right)  \right\vert
}q^{\left\vert \left[  n-1\right]  \setminus\left(  \overline{T}\cup D\left(
\alpha\right)  \right)  \right\vert }\sum_{\substack{\beta\in
\operatorname*{Comp}\nolimits_{n};\\D\left(  \beta\right)  \subseteq D\left(
\alpha\right)  }}r^{\ell\left(  \beta\right)  }M_{\beta}\\
&  =\sum_{\beta\in\operatorname*{Comp}\nolimits_{n}}\ \ \sum_{\substack{\alpha
\in\operatorname*{Comp}\nolimits_{n};\\D\left(  \beta\right)  \subseteq
D\left(  \alpha\right)  }}\left(  -1\right)  ^{\left\vert \overline
{T}\setminus D\left(  \alpha\right)  \right\vert }q^{\left\vert \left[
n-1\right]  \setminus\left(  \overline{T}\cup D\left(  \alpha\right)  \right)
\right\vert }r^{\ell\left(  \beta\right)  }M_{\beta}.
\end{align*}
However, every $\beta\in\operatorname*{Comp}\nolimits_{n}$ satisfies%
\begin{align*}
&  \sum_{\substack{\alpha\in\operatorname*{Comp}\nolimits_{n};\\D\left(
\beta\right)  \subseteq D\left(  \alpha\right)  }}\left(  -1\right)
^{\left\vert \overline{T}\setminus D\left(  \alpha\right)  \right\vert
}q^{\left\vert \left[  n-1\right]  \setminus\left(  \overline{T}\cup D\left(
\alpha\right)  \right)  \right\vert }\\
&  =\sum_{\substack{K\subseteq\left[  n-1\right]  ;\\D\left(  \beta\right)
\subseteq K}}\left(  -1\right)  ^{\left\vert \overline{T}\setminus
K\right\vert }q^{\left\vert \left[  n-1\right]  \setminus\left(  \overline
{T}\cup K\right)  \right\vert }\\
&  \ \ \ \ \ \ \ \ \ \ \ \ \ \ \ \ \ \ \ \ \left(
\begin{array}
[c]{c}%
\text{here, we have substituted }K\text{ for }D\left(  \alpha\right)  \text{
in the sum,}\\
\text{since the map }D:\operatorname*{Comp}\nolimits_{n}\rightarrow
\mathcal{P}\left(  \left[  n-1\right]  \right)  \text{ is a bijection}%
\end{array}
\right) \\
&  =\underbrace{\sum_{\substack{I\subseteq\left[  n-1\right]  ;\\D\left(
\beta\right)  \subseteq\overline{I}}}}_{\substack{=\sum_{I\subseteq
\overline{D\left(  \beta\right)  }}\\\text{(since the subsets }I\text{ of
}\left[  n-1\right]  \\\text{satisfying }D\left(  \beta\right)  \subseteq
\overline{I}\text{ are precisely}\\\text{the subsets of }\overline{D\left(
\beta\right)  }\text{)}}}\underbrace{\left(  -1\right)  ^{\left\vert
\overline{T}\setminus\overline{I}\right\vert }}_{\substack{=\left(  -1\right)
^{\left\vert I\setminus T\right\vert }\\\text{(since }\overline{T}%
\setminus\overline{I}=I\setminus T\text{)}}}\ \ \underbrace{q^{\left\vert
\left[  n-1\right]  \setminus\left(  \overline{T}\cup\overline{I}\right)
\right\vert }}_{\substack{=q^{\left\vert I\cap T\right\vert }\\\text{(since we
have}\\\left[  n-1\right]  \setminus\left(  \overline{T}\cup\overline
{I}\right)  =\overline{\overline{T}\cup\overline{I}}=\overline{\overline
{I}\cup\overline{T}}=I\cap T\\\text{by de Morgan's laws)}}}\\
&  \ \ \ \ \ \ \ \ \ \ \ \ \ \ \ \ \ \ \ \ \left(
\begin{array}
[c]{c}%
\text{here, we have substituted }\overline{I}\text{ for }K\text{ in the sum,
since}\\
\text{the map }\mathcal{P}\left(  \left[  n-1\right]  \right)  \rightarrow
\mathcal{P}\left(  \left[  n-1\right]  \right)  \text{ that sends each}\\
\text{subset }I\text{ to its complement }\overline{I}\text{ is a bijection}%
\end{array}
\right) \\
&  =\sum_{I\subseteq\overline{D\left(  \beta\right)  }}\left(  -1\right)
^{\left\vert I\setminus T\right\vert }q^{\left\vert I\cap T\right\vert
}=\left[  \overline{D\left(  \beta\right)  }\subseteq T\right]  \cdot
r^{\left\vert \overline{D\left(  \beta\right)  }\right\vert }%
\end{align*}
(by Lemma \ref{lem.eta.through-F.lem1}, applied to $S=\overline{D\left(
\beta\right)  }$). Hence, this becomes%
\begin{align*}
&  \sum_{\alpha\in\operatorname*{Comp}\nolimits_{n}}\left(  -1\right)
^{\left\vert D\left(  \gamma\right)  \setminus D\left(  \alpha\right)
\right\vert }q^{\left\vert \left[  n-1\right]  \setminus\left(  D\left(
\gamma\right)  \cup D\left(  \alpha\right)  \right)  \right\vert }\eta
_{\alpha}^{\left(  q\right)  }\\
&  =\sum_{\beta\in\operatorname*{Comp}\nolimits_{n}}\ \ \underbrace{\sum
_{\substack{\alpha\in\operatorname*{Comp}\nolimits_{n};\\D\left(
\beta\right)  \subseteq D\left(  \alpha\right)  }}\left(  -1\right)
^{\left\vert \overline{T}\setminus D\left(  \alpha\right)  \right\vert
}q^{\left\vert \left[  n-1\right]  \setminus\left(  \overline{T}\cup D\left(
\alpha\right)  \right)  \right\vert }}_{=\left[  \overline{D\left(
\beta\right)  }\subseteq T\right]  \cdot r^{\left\vert \overline{D\left(
\beta\right)  }\right\vert }}r^{\ell\left(  \beta\right)  }M_{\beta}\\
&  =\sum_{\beta\in\operatorname*{Comp}\nolimits_{n}}\left[  \overline{D\left(
\beta\right)  }\subseteq T\right]  \cdot r^{\left\vert \overline{D\left(
\beta\right)  }\right\vert }r^{\ell\left(  \beta\right)  }M_{\beta}%
=\sum_{\substack{\beta\in\operatorname*{Comp}\nolimits_{n};\\\overline
{D\left(  \beta\right)  }\subseteq T}}\ \ \underbrace{r^{\left\vert
\overline{D\left(  \beta\right)  }\right\vert }}_{\substack{=r^{n-\ell\left(
\beta\right)  }\\\text{(by (\ref{pf.prop.eta.F-through.size-of-comp}))}%
}}r^{\ell\left(  \beta\right)  }M_{\beta}\\
&  =\sum_{\substack{\beta\in\operatorname*{Comp}\nolimits_{n};\\\overline
{D\left(  \beta\right)  }\subseteq T}}\underbrace{r^{n-\ell\left(
\beta\right)  }r^{\ell\left(  \beta\right)  }}_{=r^{n}}M_{\beta}=r^{n}%
\sum_{\substack{\beta\in\operatorname*{Comp}\nolimits_{n};\\\overline{D\left(
\beta\right)  }\subseteq T}}M_{\beta}\\
&  =r^{n}\sum_{\substack{\beta\in\operatorname*{Comp}\nolimits_{n}%
;\\\overline{D\left(  \beta\right)  }\subseteq\overline{D\left(
\gamma\right)  }}}M_{\beta}\ \ \ \ \ \ \ \ \ \ \left(  \text{since
}T=\overline{D\left(  \gamma\right)  }\right) \\
&  =r^{n}\underbrace{\sum_{\substack{\beta\in\operatorname*{Comp}%
\nolimits_{n};\\D\left(  \beta\right)  \supseteq D\left(  \gamma\right)
}}M_{\beta}}_{\substack{=L_{\gamma}\\\text{(by (\ref{eq.Lalpha.def}), applied
to }\alpha=\gamma\text{)}}}\ \ \ \ \ \ \ \ \ \ \left(
\begin{array}
[c]{c}%
\text{since the condition \textquotedblleft}\overline{D\left(  \beta\right)
}\subseteq\overline{D\left(  \gamma\right)  }\text{\textquotedblright\ on a}\\
\text{composition }\beta\in\operatorname*{Comp}\nolimits_{n}\text{ is
equivalent}\\
\text{to the condition \textquotedblleft}D\left(  \beta\right)  \supseteq
D\left(  \gamma\right)  \text{\textquotedblright}%
\end{array}
\right) \\
&  =r^{n}L_{\gamma}.
\end{align*}
This proves Proposition \ref{prop.eta.F-through}.
\end{proof}
\end{vershort}

\begin{verlong}

\begin{proof}
[Proof of Proposition \ref{prop.eta.F-through}.]For each subset $J$ of
$\left[  n-1\right]  $, we let $\overline{J}$ denote its complement $\left[
n-1\right]  \setminus J$. The following properties of complements are
well-known (and easy to check):

\begin{itemize}
\item Every subset $J$ of $\left[  n-1\right]  $ satisfies
\begin{equation}
\overline{\overline{J}}=J. \label{pf.prop.eta.F-through.long.JJJ}%
\end{equation}

\item The map $\mathcal{P}\left(  \left[  n-1\right]  \right)  \rightarrow
\mathcal{P}\left(  \left[  n-1\right]  \right)  $ that sends each subset $J$
to its complement $\overline{J}$ is a bijection. (In fact, this map is its own
inverse, because of (\ref{pf.prop.eta.F-through.long.JJJ}).)

\item If $A$ and $B$ are two subsets of $\left[  n-1\right]  $, then%
\begin{equation}
\overline{A\cap B}=\overline{A}\cup\overline{B}.
\label{pf.prop.eta.F-through.long.deMor}%
\end{equation}

\item If $A$ and $B$ are two subsets of $\left[  n-1\right]  $ satisfying
$A\subseteq B$, then $\overline{A}\supseteq\overline{B}$.

\item Any two subsets $A$ and $B$ of $\left[  n-1\right]  $ satisfy
\begin{equation}
\overline{A}\setminus\overline{B}=B\setminus A
\label{pf.prop.eta.F-through.long.AminB}%
\end{equation}
\footnote{\textit{Proof of (\ref{pf.prop.eta.F-through.long.AminB}):} Let $A$
and $B$ be two subsets of $\left[  n-1\right]  $. Then, the definition of
$\overline{A}$ yields $\overline{A}=\left[  n-1\right]  \setminus
A\subseteq\left[  n-1\right]  $. The definition of $\overline{B}$ yields
$\overline{B}=\left[  n-1\right]  \setminus B\subseteq\left[  n-1\right]  $.
\par
Hence, the definition of $\overline{\overline{B}}$ yields $\overline
{\overline{B}}=\left[  n-1\right]  \setminus\overline{B}$. Thus, $\left[
n-1\right]  \setminus\overline{B}=\overline{\overline{B}}=B$ (since every
subset $J$ of $\left[  n-1\right]  $ satisfies $\overline{\overline{J}}=J$).
\par
Now,%
\begin{align*}
\underbrace{\overline{A}}_{\substack{=\overline{A}\cap\left[  n-1\right]
\\\text{(since }\overline{A}\subseteq\left[  n-1\right]  \text{)}}%
}\setminus\overline{B}  &  =\left(  \overline{A}\cap\left[  n-1\right]
\right)  \setminus\overline{B}\\
&  =\overline{A}\cap\underbrace{\left(  \left[  n-1\right]  \setminus
\overline{B}\right)  }_{=B}\ \ \ \ \ \ \ \ \ \ \left(
\begin{array}
[c]{c}%
\text{since }\left(  X\cap Y\right)  \setminus Z=X\cap\left(  Y\setminus
Z\right) \\
\text{for any three sets }X\text{, }Y\text{ and }Z
\end{array}
\right) \\
&  =\overline{A}\cap B=B\cap\underbrace{\overline{A}}_{=\left[  n-1\right]
\setminus A}=B\cap\left(  \left[  n-1\right]  \setminus A\right) \\
&  =\underbrace{\left(  B\cap\left[  n-1\right]  \right)  }%
_{\substack{=B\\\text{(since }B\subseteq\left[  n-1\right]  \text{)}%
}}\setminus A\ \ \ \ \ \ \ \ \ \ \left(
\begin{array}
[c]{c}%
\text{since }X\cap\left(  Y\setminus Z\right)  =\left(  X\cap Y\right)
\setminus Z\\
\text{for any three sets }X\text{, }Y\text{ and }Z
\end{array}
\right) \\
&  =B\setminus A.
\end{align*}
This proves (\ref{pf.prop.eta.F-through.long.AminB}).}.

\item Any two subsets $A$ and $B$ of $\left[  n-1\right]  $ satisfy
\begin{equation}
\left[  n-1\right]  \setminus\left(  \overline{A}\cup\overline{B}\right)
=B\cap A \label{pf.prop.eta.F-through.long.uni}%
\end{equation}
\footnote{\textit{Proof of (\ref{pf.prop.eta.F-through.long.uni}):} Let $A$
and $B$ be two subsets of $\left[  n-1\right]  $.
\par
Then, $A\cap B\subseteq A\subseteq\left[  n-1\right]  $ (since $A$ is a subset
of $\left[  n-1\right]  $). In other words, $A\cap B$ is a subset of $\left[
n-1\right]  $.
\par
Recall that every subset $J$ of $\left[  n-1\right]  $ satisfies
$\overline{\overline{J}}=J$. Applying this to $J=A\cap B$, we obtain
$\overline{\overline{A\cap B}}=A\cap B$ (since $A\cap B$ is a subset of
$\left[  n-1\right]  $). In view of (\ref{pf.prop.eta.F-through.long.deMor}),
we can rewrite this as $\overline{\overline{A}\cup\overline{B}}=A\cap B$.
However, the definition of $\overline{\overline{A}\cup\overline{B}}$ yields
$\overline{\overline{A}\cup\overline{B}}=\left[  n-1\right]  \setminus\left(
\overline{A}\cup\overline{B}\right)  $. Hence,%
\[
\left[  n-1\right]  \setminus\left(  \overline{A}\cup\overline{B}\right)
=\overline{\overline{A}\cup\overline{B}}=A\cap B=B\cap A.
\]
This proves (\ref{pf.prop.eta.F-through.long.uni}).}.
\end{itemize}

Recall that $D:\operatorname*{Comp}\nolimits_{n}\rightarrow\mathcal{P}\left(
\left[  n-1\right]  \right)  $ is a bijection. Hence, from $\gamma
\in\operatorname*{Comp}\nolimits_{n}$, we obtain $D\left(  \gamma\right)
\in\mathcal{P}\left(  \left[  n-1\right]  \right)  $. In other words,
$D\left(  \gamma\right)  \subseteq\left[  n-1\right]  $. In other words,
$D\left(  \gamma\right)  $ is a subset of $\left[  n-1\right]  $. Hence, its
complement $\overline{D\left(  \gamma\right)  }$ is well-defined.

Let $T:=\overline{D\left(  \gamma\right)  }$. Thus,
\begin{align*}
T  &  =\overline{D\left(  \gamma\right)  }=\left[  n-1\right]  \setminus
D\left(  \gamma\right)  \ \ \ \ \ \ \ \ \ \ \left(  \text{by the definition of
}\overline{D\left(  \gamma\right)  }\right) \\
&  \subseteq\left[  n-1\right]  .
\end{align*}
Furthermore, from $T=\overline{D\left(  \gamma\right)  }$, we obtain
$\overline{T}=\overline{\overline{D\left(  \gamma\right)  }}=D\left(
\gamma\right)  $ (since every subset $J$ of $\left[  n-1\right]  $ satisfies
$\overline{\overline{J}}=J$). In other words, $D\left(  \gamma\right)
=\overline{T}$. Hence,%
\begin{align}
&  \sum_{\alpha\in\operatorname*{Comp}\nolimits_{n}}\left(  -1\right)
^{\left\vert D\left(  \gamma\right)  \setminus D\left(  \alpha\right)
\right\vert }q^{\left\vert \left[  n-1\right]  \setminus\left(  D\left(
\gamma\right)  \cup D\left(  \alpha\right)  \right)  \right\vert }\eta
_{\alpha}^{\left(  q\right)  }\nonumber\\
&  =\sum_{\alpha\in\operatorname*{Comp}\nolimits_{n}}\left(  -1\right)
^{\left\vert \overline{T}\setminus D\left(  \alpha\right)  \right\vert
}q^{\left\vert \left[  n-1\right]  \setminus\left(  \overline{T}\cup D\left(
\alpha\right)  \right)  \right\vert }\underbrace{\eta_{\alpha}^{\left(
q\right)  }}_{\substack{=\sum_{\substack{\beta\in\operatorname*{Comp}%
\nolimits_{n};\\D\left(  \beta\right)  \subseteq D\left(  \alpha\right)
}}r^{\ell\left(  \beta\right)  }M_{\beta}\\\text{(by the definition of }%
\eta_{\alpha}^{\left(  q\right)  }\text{)}}}\nonumber\\
&  =\sum_{\alpha\in\operatorname*{Comp}\nolimits_{n}}\left(  -1\right)
^{\left\vert \overline{T}\setminus D\left(  \alpha\right)  \right\vert
}q^{\left\vert \left[  n-1\right]  \setminus\left(  \overline{T}\cup D\left(
\alpha\right)  \right)  \right\vert }\sum_{\substack{\beta\in
\operatorname*{Comp}\nolimits_{n};\\D\left(  \beta\right)  \subseteq D\left(
\alpha\right)  }}r^{\ell\left(  \beta\right)  }M_{\beta}\nonumber\\
&  =\underbrace{\sum_{\alpha\in\operatorname*{Comp}\nolimits_{n}}%
\ \ \sum_{\substack{\beta\in\operatorname*{Comp}\nolimits_{n};\\D\left(
\beta\right)  \subseteq D\left(  \alpha\right)  }}}_{=\sum_{\beta
\in\operatorname*{Comp}\nolimits_{n}}\ \ \sum_{\substack{\alpha\in
\operatorname*{Comp}\nolimits_{n};\\D\left(  \beta\right)  \subseteq D\left(
\alpha\right)  }}}\left(  -1\right)  ^{\left\vert \overline{T}\setminus
D\left(  \alpha\right)  \right\vert }q^{\left\vert \left[  n-1\right]
\setminus\left(  \overline{T}\cup D\left(  \alpha\right)  \right)  \right\vert
}r^{\ell\left(  \beta\right)  }M_{\beta}\nonumber\\
&  =\sum_{\beta\in\operatorname*{Comp}\nolimits_{n}}\ \ \sum_{\substack{\alpha
\in\operatorname*{Comp}\nolimits_{n};\\D\left(  \beta\right)  \subseteq
D\left(  \alpha\right)  }}\left(  -1\right)  ^{\left\vert \overline
{T}\setminus D\left(  \alpha\right)  \right\vert }q^{\left\vert \left[
n-1\right]  \setminus\left(  \overline{T}\cup D\left(  \alpha\right)  \right)
\right\vert }r^{\ell\left(  \beta\right)  }M_{\beta}\nonumber\\
&  =\sum_{\beta\in\operatorname*{Comp}\nolimits_{n}}\left(  \sum
_{\substack{\alpha\in\operatorname*{Comp}\nolimits_{n};\\D\left(
\beta\right)  \subseteq D\left(  \alpha\right)  }}\left(  -1\right)
^{\left\vert \overline{T}\setminus D\left(  \alpha\right)  \right\vert
}q^{\left\vert \left[  n-1\right]  \setminus\left(  \overline{T}\cup D\left(
\alpha\right)  \right)  \right\vert }\right)  r^{\ell\left(  \beta\right)
}M_{\beta}. \label{pf.prop.eta.F-through.long.1}%
\end{align}

Now, let $\beta\in\operatorname*{Comp}\nolimits_{n}$ be arbitrary. Recall that
$D:\operatorname*{Comp}\nolimits_{n}\rightarrow\mathcal{P}\left(  \left[
n-1\right]  \right)  $ is a bijection. Hence, from $\beta\in
\operatorname*{Comp}\nolimits_{n}$, we obtain $D\left(  \beta\right)
\in\mathcal{P}\left(  \left[  n-1\right]  \right)  $. In other words,
$D\left(  \beta\right)  \subseteq\left[  n-1\right]  $. In other words,
$D\left(  \beta\right)  $ is a subset of $\left[  n-1\right]  $. Hence, its
complement $\overline{D\left(  \beta\right)  }$ is well-defined. We shall now
prove that%
\begin{equation}
\left\vert \overline{D\left(  \beta\right)  }\right\vert =n-\ell\left(
\beta\right)  . \label{pf.prop.eta.F-through.long.size-of-comp}%
\end{equation}

[\textit{Proof of (\ref{pf.prop.eta.F-through.long.size-of-comp}):} We have
$n\neq0$ (since $n$ is positive), and thus $\left[  n\neq0\right]  =1$.
However, Lemma \ref{lem.comps.l-vs-size} \textbf{(a)} (applied to
$\delta=\beta$) yields $\ell\left(  \beta\right)  =\left\vert D\left(
\beta\right)  \right\vert +\underbrace{\left[  n\neq0\right]  }_{=1}%
=\left\vert D\left(  \beta\right)  \right\vert +1$, so that $\left\vert
D\left(  \beta\right)  \right\vert =\ell\left(  \beta\right)  -1$. However,
the definition of $\overline{D\left(  \beta\right)  }$ yields $\overline
{D\left(  \beta\right)  }=\left[  n-1\right]  \setminus D\left(  \beta\right)
$. Hence,%
\begin{align*}
\left\vert \overline{D\left(  \beta\right)  }\right\vert  &  =\left\vert
\left[  n-1\right]  \setminus D\left(  \beta\right)  \right\vert \\
&  =\underbrace{\left\vert \left[  n-1\right]  \right\vert }%
_{\substack{=n-1\\\text{(since }n-1\in\mathbb{N}\\\text{(because }n\text{ is
a}\\\text{positive integer))}}}-\underbrace{\left\vert D\left(  \beta\right)
\right\vert }_{=\ell\left(  \beta\right)  -1}\ \ \ \ \ \ \ \ \ \ \left(
\text{since }D\left(  \beta\right)  \subseteq\left[  n-1\right]  \right) \\
&  =\left(  n-1\right)  -\left(  \ell\left(  \beta\right)  -1\right)
=n-\ell\left(  \beta\right)  .
\end{align*}
This proves (\ref{pf.prop.eta.F-through.long.size-of-comp}).]

Next, we observe that%
\begin{align}
r^{\left\vert \overline{D\left(  \beta\right)  }\right\vert }r^{\ell\left(
\beta\right)  }  &  =r^{n-\ell\left(  \beta\right)  }r^{\ell\left(
\beta\right)  }\ \ \ \ \ \ \ \ \ \ \left(  \text{since
(\ref{pf.prop.eta.F-through.long.size-of-comp}) says that }\left\vert
\overline{D\left(  \beta\right)  }\right\vert =n-\ell\left(  \beta\right)
\right) \nonumber\\
&  =r^{\left(  n-\ell\left(  \beta\right)  \right)  +\ell\left(  \beta\right)
}=r^{n}. \label{pf.prop.eta.F-through.long.size-of-comp2}%
\end{align}

Furthermore,%
\begin{align}
&  \sum_{\substack{\alpha\in\operatorname*{Comp}\nolimits_{n};\\D\left(
\beta\right)  \subseteq D\left(  \alpha\right)  }}\left(  -1\right)
^{\left\vert \overline{T}\setminus D\left(  \alpha\right)  \right\vert
}q^{\left\vert \left[  n-1\right]  \setminus\left(  \overline{T}\cup D\left(
\alpha\right)  \right)  \right\vert }\nonumber\\
&  =\sum_{\substack{I\in\mathcal{P}\left(  \left[  n-1\right]  \right)
;\\D\left(  \beta\right)  \subseteq I}}\left(  -1\right)  ^{\left\vert
\overline{T}\setminus I\right\vert }q^{\left\vert \left[  n-1\right]
\setminus\left(  \overline{T}\cup I\right)  \right\vert }\nonumber\\
&  \ \ \ \ \ \ \ \ \ \ \ \ \ \ \ \ \ \ \ \ \left(
\begin{array}
[c]{c}%
\text{here, we have substituted }I\text{ for }D\left(  \alpha\right)  \text{
in the sum,}\\
\text{since the map }D:\operatorname*{Comp}\nolimits_{n}\rightarrow
\mathcal{P}\left(  \left[  n-1\right]  \right)  \text{ is a bijection}%
\end{array}
\right) \nonumber\\
&  =\sum_{\substack{J\in\mathcal{P}\left(  \left[  n-1\right]  \right)
;\\D\left(  \beta\right)  \subseteq\overline{J}}}\underbrace{\left(
-1\right)  ^{\left\vert \overline{T}\setminus\overline{J}\right\vert }%
}_{\substack{=\left(  -1\right)  ^{\left\vert J\setminus T\right\vert
}\\\text{(since }\overline{T}\setminus\overline{J}=J\setminus T\\\text{(by
(\ref{pf.prop.eta.F-through.long.AminB}), applied to }A=T\text{ and
}B=J\text{))}}}\underbrace{q^{\left\vert \left[  n-1\right]  \setminus\left(
\overline{T}\cup\overline{J}\right)  \right\vert }}_{\substack{=q^{\left\vert
J\cap T\right\vert }\\\text{(since }\left[  n-1\right]  \setminus\left(
\overline{T}\cup\overline{J}\right)  =J\cap T\\\text{(by
(\ref{pf.prop.eta.F-through.long.uni}), applied to }A=T\text{ and
}B=J\text{))}}}\nonumber\\
&  \ \ \ \ \ \ \ \ \ \ \ \ \ \ \ \ \ \ \ \ \left(
\begin{array}
[c]{c}%
\text{here, we have substituted }\overline{J}\text{ for }I\text{ in the sum,
since}\\
\text{the map }\mathcal{P}\left(  \left[  n-1\right]  \right)  \rightarrow
\mathcal{P}\left(  \left[  n-1\right]  \right)  \text{ that sends each}\\
\text{subset }J\text{ to its complement }\overline{J}\text{ is a bijection}%
\end{array}
\right) \nonumber\\
&  =\sum_{\substack{J\in\mathcal{P}\left(  \left[  n-1\right]  \right)
;\\D\left(  \beta\right)  \subseteq\overline{J}}}\left(  -1\right)
^{\left\vert J\setminus T\right\vert }q^{\left\vert J\cap T\right\vert }.
\label{pf.prop.eta.F-through.long.4}%
\end{align}

However, we have\footnote{Here, the notation $\mathcal{P}\left(  T\right)  $
denotes the powerset of a given set $T$ (that is, the set of all subsets of
$T$). This generalizes our above notation $\mathcal{P}\left(  \left[
n-1\right]  \right)  $.}%
\begin{equation}
\left\{  J\in\mathcal{P}\left(  \left[  n-1\right]  \right)  \ \mid\ D\left(
\beta\right)  \subseteq\overline{J}\right\}  =\mathcal{P}\left(
\overline{D\left(  \beta\right)  }\right)
\label{pf.prop.eta.F-through.long.sets}%
\end{equation}
\footnote{\textit{Proof.} Let $J\in\mathcal{P}\left(  \left[  n-1\right]
\right)  $ be such that $D\left(  \beta\right)  \subseteq\overline{J}$. We
shall prove that $J\in\mathcal{P}\left(  \overline{D\left(  \beta\right)
}\right)  $.
\par
Indeed, from $J\in\mathcal{P}\left(  \left[  n-1\right]  \right)  $, we obtain
$J\subseteq\left[  n-1\right]  $.
\par
We know that $D\left(  \beta\right)  $ is a subset of $\left[  n-1\right]  $.
Also, $\overline{J}$ is a subset of $\left[  n-1\right]  $ (since the
definition of $\overline{J}$ yields $\overline{J}=\left[  n-1\right]
\setminus J\subseteq\left[  n-1\right]  $).
\par
Recall that if $A$ and $B$ are two subsets of $\left[  n-1\right]  $
satisfying $A\subseteq B$, then $\overline{A}\supseteq\overline{B}$. Applying
this to $A=D\left(  \beta\right)  $ and $B=\overline{J}$, we obtain
$\overline{D\left(  \beta\right)  }\supseteq\overline{\overline{J}}$ (since
$D\left(  \beta\right)  \subseteq\overline{J}$). In view of
(\ref{pf.prop.eta.F-through.long.JJJ}), we can rewrite this as $\overline
{D\left(  \beta\right)  }\supseteq J$. In other words, $J\subseteq
\overline{D\left(  \beta\right)  }$. In other words, $J\in\mathcal{P}\left(
\overline{D\left(  \beta\right)  }\right)  $.
\par
Forget that we fixed $J$. Thus, we have shown that $J\in\mathcal{P}\left(
\overline{D\left(  \beta\right)  }\right)  $ for every $J\in\mathcal{P}\left(
\left[  n-1\right]  \right)  $ such that $D\left(  \beta\right)
\subseteq\overline{J}$. In other words,%
\begin{equation}
\left\{  J\in\mathcal{P}\left(  \left[  n-1\right]  \right)  \ \mid\ D\left(
\beta\right)  \subseteq\overline{J}\right\}  \subseteq\mathcal{P}\left(
\overline{D\left(  \beta\right)  }\right)  .
\label{pf.prop.eta.F-through.long.sets.pf.1}%
\end{equation}
\par
On the other hand, let $I\in\mathcal{P}\left(  \overline{D\left(
\beta\right)  }\right)  $ be arbitrary. We shall prove that $I\in\left\{
J\in\mathcal{P}\left(  \left[  n-1\right]  \right)  \ \mid\ D\left(
\beta\right)  \subseteq\overline{J}\right\}  $.
\par
Indeed, we have $I\in\mathcal{P}\left(  \overline{D\left(  \beta\right)
}\right)  $. In other words, $I\subseteq\overline{D\left(  \beta\right)  }$.
However, $D\left(  \beta\right)  $ is a subset of $\left[  n-1\right]  $;
thus, the definition of $\overline{D\left(  \beta\right)  }$ yields
$\overline{D\left(  \beta\right)  }=\left[  n-1\right]  \setminus D\left(
\beta\right)  \subseteq\left[  n-1\right]  $. In other words, $\overline
{D\left(  \beta\right)  }$ is a subset of $\left[  n-1\right]  $. Hence, $I$
is a subset of $\left[  n-1\right]  $ as well (since $I\subseteq
\overline{D\left(  \beta\right)  }$). In other words, $I\in\mathcal{P}\left(
\left[  n-1\right]  \right)  $.
\par
Recall that if $A$ and $B$ are two subsets of $\left[  n-1\right]  $
satisfying $A\subseteq B$, then $\overline{A}\supseteq\overline{B}$. Applying
this to $A=I$ and $B=\overline{D\left(  \beta\right)  }$, we obtain
$\overline{I}\supseteq\overline{\overline{D\left(  \beta\right)  }}$ (since
$I\subseteq\overline{D\left(  \beta\right)  }$). However,
(\ref{pf.prop.eta.F-through.long.JJJ}) (applied to $D\left(  \beta\right)  $
instead of $J$) yields $\overline{\overline{D\left(  \beta\right)  }}=D\left(
\beta\right)  $. Thus, $D\left(  \beta\right)  =\overline{\overline{D\left(
\beta\right)  }}\subseteq\overline{I}$ (since $\overline{I}\supseteq
\overline{\overline{D\left(  \beta\right)  }}$).
\par
Hence, we conclude that $I$ is a $J\in\mathcal{P}\left(  \left[  n-1\right]
\right)  $ satisfying $D\left(  \beta\right)  \subseteq\overline{J}$ (since
$I\in\mathcal{P}\left(  \left[  n-1\right]  \right)  $ and $D\left(
\beta\right)  \subseteq\overline{I}$). In other words, $I\in\left\{
J\in\mathcal{P}\left(  \left[  n-1\right]  \right)  \ \mid\ D\left(
\beta\right)  \subseteq\overline{J}\right\}  $.
\par
Forget that we fixed $I$. We thus have shown that every $I\in\mathcal{P}%
\left(  \overline{D\left(  \beta\right)  }\right)  $ satisfies $I\in\left\{
J\in\mathcal{P}\left(  \left[  n-1\right]  \right)  \ \mid\ D\left(
\beta\right)  \subseteq\overline{J}\right\}  $. In other words,%
\[
\mathcal{P}\left(  \overline{D\left(  \beta\right)  }\right)  \subseteq
\left\{  J\in\mathcal{P}\left(  \left[  n-1\right]  \right)  \ \mid\ D\left(
\beta\right)  \subseteq\overline{J}\right\}  .
\]
Combining this with (\ref{pf.prop.eta.F-through.long.sets.pf.1}), we obtain%
\[
\left\{  J\in\mathcal{P}\left(  \left[  n-1\right]  \right)  \ \mid\ D\left(
\beta\right)  \subseteq\overline{J}\right\}  =\mathcal{P}\left(
\overline{D\left(  \beta\right)  }\right)  .
\]
}. Thus, we have the following equality of summation signs:
\[
\sum_{\substack{J\in\mathcal{P}\left(  \left[  n-1\right]  \right)
;\\D\left(  \beta\right)  \subseteq\overline{J}}}=\sum_{J\in\mathcal{P}\left(
\overline{D\left(  \beta\right)  }\right)  }=\sum_{J\subseteq\overline
{D\left(  \beta\right)  }}%
\]
(since the condition \textquotedblleft$J\in\mathcal{P}\left(  \overline
{D\left(  \beta\right)  }\right)  $\textquotedblright\ is clearly equivalent
to \textquotedblleft$J\subseteq\overline{D\left(  \beta\right)  }%
$\textquotedblright). Hence, (\ref{pf.prop.eta.F-through.long.4}) becomes%
\begin{align}
&  \sum_{\substack{\alpha\in\operatorname*{Comp}\nolimits_{n};\\D\left(
\beta\right)  \subseteq D\left(  \alpha\right)  }}\left(  -1\right)
^{\left\vert \overline{T}\setminus D\left(  \alpha\right)  \right\vert
}q^{\left\vert \left[  n-1\right]  \setminus\left(  \overline{T}\cup D\left(
\alpha\right)  \right)  \right\vert }\nonumber\\
&  =\underbrace{\sum_{\substack{J\in\mathcal{P}\left(  \left[  n-1\right]
\right)  ;\\D\left(  \beta\right)  \subseteq\overline{J}}}}_{\substack{=\sum
_{J\subseteq\overline{D\left(  \beta\right)  }}}}\left(  -1\right)
^{\left\vert J\setminus T\right\vert }q^{\left\vert J\cap T\right\vert
}\nonumber\\
&  =\sum_{J\subseteq\overline{D\left(  \beta\right)  }}\left(  -1\right)
^{\left\vert J\setminus T\right\vert }q^{\left\vert J\cap T\right\vert
}\nonumber\\
&  =\sum_{I\subseteq\overline{D\left(  \beta\right)  }}\left(  -1\right)
^{\left\vert I\setminus T\right\vert }q^{\left\vert I\cap T\right\vert
}\ \ \ \ \ \ \ \ \ \ \left(
\begin{array}
[c]{c}%
\text{here, we have renamed}\\
\text{the summation index }J\text{ as }I
\end{array}
\right) \nonumber\\
&  =\left[  \overline{D\left(  \beta\right)  }\subseteq T\right]  \cdot
r^{\left\vert \overline{D\left(  \beta\right)  }\right\vert }
\label{pf.prop.eta.F-through.long.2}%
\end{align}
(by Lemma \ref{lem.eta.through-F.lem1}, applied to $S=\overline{D\left(
\beta\right)  }$).

Forget that we fixed $\beta$. We thus have proved the two equalities
(\ref{pf.prop.eta.F-through.long.size-of-comp2}) and
(\ref{pf.prop.eta.F-through.long.2}) for each $\beta\in\operatorname*{Comp}%
\nolimits_{n}$.

Hence, (\ref{pf.prop.eta.F-through.long.1}) becomes%
\begin{align}
&  \sum_{\alpha\in\operatorname*{Comp}\nolimits_{n}}\left(  -1\right)
^{\left\vert D\left(  \gamma\right)  \setminus D\left(  \alpha\right)
\right\vert }q^{\left\vert \left[  n-1\right]  \setminus\left(  D\left(
\gamma\right)  \cup D\left(  \alpha\right)  \right)  \right\vert }\eta
_{\alpha}^{\left(  q\right)  }\nonumber\\
&  =\sum_{\beta\in\operatorname*{Comp}\nolimits_{n}}\underbrace{\left(
\sum_{\substack{\alpha\in\operatorname*{Comp}\nolimits_{n};\\D\left(
\beta\right)  \subseteq D\left(  \alpha\right)  }}\left(  -1\right)
^{\left\vert \overline{T}\setminus D\left(  \alpha\right)  \right\vert
}q^{\left\vert \left[  n-1\right]  \setminus\left(  \overline{T}\cup D\left(
\alpha\right)  \right)  \right\vert }\right)  }_{\substack{=\left[
\overline{D\left(  \beta\right)  }\subseteq T\right]  \cdot r^{\left\vert
\overline{D\left(  \beta\right)  }\right\vert }\\\text{(by
(\ref{pf.prop.eta.F-through.long.2}))}}}r^{\ell\left(  \beta\right)  }%
M_{\beta}\nonumber\\
&  =\sum_{\beta\in\operatorname*{Comp}\nolimits_{n}}\left[  \overline{D\left(
\beta\right)  }\subseteq T\right]  \cdot\underbrace{r^{\left\vert
\overline{D\left(  \beta\right)  }\right\vert }r^{\ell\left(  \beta\right)  }%
}_{\substack{=r^{n}\\\text{(by (\ref{pf.prop.eta.F-through.long.size-of-comp2}%
))}}}M_{\beta}=\sum_{\beta\in\operatorname*{Comp}\nolimits_{n}}\left[
\overline{D\left(  \beta\right)  }\subseteq T\right]  \cdot r^{n}M_{\beta
}\nonumber\\
&  =\sum_{\substack{\beta\in\operatorname*{Comp}\nolimits_{n};\\\overline
{D\left(  \beta\right)  }\subseteq T}}\underbrace{\left[  \overline{D\left(
\beta\right)  }\subseteq T\right]  }_{\substack{=1\\\text{(since }%
\overline{D\left(  \beta\right)  }\subseteq T\text{)}}}\cdot\,r^{n}M_{\beta
}+\sum_{\substack{\beta\in\operatorname*{Comp}\nolimits_{n};\\\text{not
}\overline{D\left(  \beta\right)  }\subseteq T}}\ \ \underbrace{\left[
\overline{D\left(  \beta\right)  }\subseteq T\right]  }%
_{\substack{=0\\\text{(since we don't}\\\text{have }\overline{D\left(
\beta\right)  }\subseteq T\text{)}}}\cdot\,r^{n}M_{\beta}\nonumber\\
&  \ \ \ \ \ \ \ \ \ \ \ \ \ \ \ \ \ \ \ \ \left(
\begin{array}
[c]{c}%
\text{since each }\beta\in\operatorname*{Comp}\nolimits_{n}\text{ satisfies
either }\overline{D\left(  \beta\right)  }\subseteq T\\
\text{or }\left(  \text{not }\overline{D\left(  \beta\right)  }\subseteq
T\right)  \text{, but not both simultaneously}%
\end{array}
\right) \nonumber\\
&  =\sum_{\substack{\beta\in\operatorname*{Comp}\nolimits_{n};\\\overline
{D\left(  \beta\right)  }\subseteq T}}r^{n}M_{\beta}+\underbrace{\sum
_{\substack{\beta\in\operatorname*{Comp}\nolimits_{n};\\\text{not }%
\overline{D\left(  \beta\right)  }\subseteq T}}0\cdot r^{n}M_{\beta}}%
_{=0}\nonumber\\
&  =\sum_{\substack{\beta\in\operatorname*{Comp}\nolimits_{n};\\\overline
{D\left(  \beta\right)  }\subseteq T}}r^{n}M_{\beta}=r^{n}\sum
_{\substack{\beta\in\operatorname*{Comp}\nolimits_{n};\\\overline{D\left(
\beta\right)  }\subseteq T}}M_{\beta}\nonumber\\
&  =r^{n}\sum_{\substack{\beta\in\operatorname*{Comp}\nolimits_{n}%
;\\\overline{D\left(  \beta\right)  }\subseteq\overline{D\left(
\gamma\right)  }}}M_{\beta}\ \ \ \ \ \ \ \ \ \ \left(  \text{since
}T=\overline{D\left(  \gamma\right)  }\right)  .
\label{pf.prop.eta.F-through.long.at}%
\end{align}

Now, we observe that for any $\beta\in\operatorname*{Comp}\nolimits_{n}$, the
equivalence%
\[
\left(  D\left(  \beta\right)  \supseteq D\left(  \gamma\right)  \right)
\ \Longleftrightarrow\ \left(  \overline{D\left(  \beta\right)  }%
\subseteq\overline{D\left(  \gamma\right)  }\right)
\]
holds\footnote{\textit{Proof.} Let $\beta\in\operatorname*{Comp}\nolimits_{n}%
$. Recall that $D:\operatorname*{Comp}\nolimits_{n}\rightarrow\mathcal{P}%
\left(  \left[  n-1\right]  \right)  $ is a bijection. Hence, from $\beta
\in\operatorname*{Comp}\nolimits_{n}$, we obtain $D\left(  \beta\right)
\in\mathcal{P}\left(  \left[  n-1\right]  \right)  $. In other words,
$D\left(  \beta\right)  \subseteq\left[  n-1\right]  $. In other words,
$D\left(  \beta\right)  $ is a subset of $\left[  n-1\right]  $. Recall also
that $D\left(  \gamma\right)  $ is a subset of $\left[  n-1\right]  $.
\par
Hence, (\ref{pf.prop.eta.F-through.long.AminB}) (applied to $A=D\left(
\beta\right)  $ and $B=D\left(  \gamma\right)  $) yields $\overline{D\left(
\beta\right)  }\setminus\overline{D\left(  \gamma\right)  }=D\left(
\gamma\right)  \setminus D\left(  \beta\right)  $. In other words,%
\[
D\left(  \gamma\right)  \setminus D\left(  \beta\right)  =\overline{D\left(
\beta\right)  }\setminus\overline{D\left(  \gamma\right)  }.
\]
\par
However, we have the following chain of equivalences:%
\begin{align*}
\left(  D\left(  \beta\right)  \supseteq D\left(  \gamma\right)  \right)  \
&  \Longleftrightarrow\ \left(  D\left(  \gamma\right)  \subseteq D\left(
\beta\right)  \right) \\
&  \Longleftrightarrow\ \left(  \underbrace{D\left(  \gamma\right)  \setminus
D\left(  \beta\right)  }_{=\overline{D\left(  \beta\right)  }\setminus
\overline{D\left(  \gamma\right)  }}=\varnothing\right) \\
&  \Longleftrightarrow\ \left(  \overline{D\left(  \beta\right)  }%
\setminus\overline{D\left(  \gamma\right)  }=\varnothing\right) \\
&  \Longleftrightarrow\ \left(  \overline{D\left(  \beta\right)  }%
\subseteq\overline{D\left(  \gamma\right)  }\right)  .
\end{align*}
We have thus proved the equivalence $\left(  D\left(  \beta\right)  \supseteq
D\left(  \gamma\right)  \right)  \ \Longleftrightarrow\ \left(  \overline
{D\left(  \beta\right)  }\subseteq\overline{D\left(  \gamma\right)  }\right)
$, qed.}. Hence, we have the following equality of summation signs:%
\begin{equation}
\sum_{\substack{\beta\in\operatorname*{Comp}\nolimits_{n};\\D\left(
\beta\right)  \supseteq D\left(  \gamma\right)  }}=\sum_{\substack{\beta
\in\operatorname*{Comp}\nolimits_{n};\\\overline{D\left(  \beta\right)
}\subseteq\overline{D\left(  \gamma\right)  }}}.
\label{pf.prop.eta.F-through.long.eqsum}%
\end{equation}

Now, (\ref{eq.Lalpha.def}) (applied to $\alpha=\gamma$) yields%
\[
L_{\gamma}=\underbrace{\sum_{\substack{\beta\in\operatorname*{Comp}%
\nolimits_{n};\\D\left(  \beta\right)  \supseteq D\left(  \gamma\right)  }%
}}_{\substack{=\sum_{\substack{\beta\in\operatorname*{Comp}\nolimits_{n}%
;\\\overline{D\left(  \beta\right)  }\subseteq\overline{D\left(
\gamma\right)  }}}\\\text{(by (\ref{pf.prop.eta.F-through.long.eqsum}))}%
}}M_{\beta}=\sum_{\substack{\beta\in\operatorname*{Comp}\nolimits_{n}%
;\\\overline{D\left(  \beta\right)  }\subseteq\overline{D\left(
\gamma\right)  }}}M_{\beta}.
\]
Multiplying both sides of this equality by $r^{n}$, we find%
\[
r^{n}L_{\gamma}=r^{n}\sum_{\substack{\beta\in\operatorname*{Comp}%
\nolimits_{n};\\\overline{D\left(  \beta\right)  }\subseteq\overline{D\left(
\gamma\right)  }}}M_{\beta}=\sum_{\alpha\in\operatorname*{Comp}\nolimits_{n}%
}\left(  -1\right)  ^{\left\vert D\left(  \gamma\right)  \setminus D\left(
\alpha\right)  \right\vert }q^{\left\vert \left[  n-1\right]  \setminus\left(
D\left(  \gamma\right)  \cup D\left(  \alpha\right)  \right)  \right\vert
}\eta_{\alpha}^{\left(  q\right)  }%
\]
(by (\ref{pf.prop.eta.F-through.long.at})). This proves Proposition
\ref{prop.eta.F-through}.
\end{proof}
\end{verlong}

\subsection{The antipode of $\eta_{\alpha}^{\left(  q\right)  }$}

The \emph{antipode} of $\operatorname*{QSym}$ is a certain $\mathbf{k}$-linear
map $S:\operatorname*{QSym}\rightarrow\operatorname*{QSym}$ that can be
defined in terms of the Hopf algebra structure of $\operatorname*{QSym}$,
which we have not defined so far. But there are various formulas for its
values on certain quasisymmetric functions that can be used as alternative
definitions. For example, for any $n\in\mathbb{N}$ and any $\alpha=\left(
\alpha_{1},\alpha_{2},\ldots,\alpha_{\ell}\right)  \in\operatorname*{Comp}%
\nolimits_{n}$, we have%
\begin{equation}
S\left(  M_{\alpha}\right)  =\left(  -1\right)  ^{\ell}\sum_{\substack{\gamma
\in\operatorname*{Comp}\nolimits_{n};\\D\left(  \gamma\right)  \subseteq
D\left(  \alpha_{\ell},\alpha_{\ell-1},\ldots,\alpha_{1}\right)  }}M_{\gamma}.
\label{eq.SMalpha}%
\end{equation}
This formula (which appears, e.g., in \cite[(4.26)]{Malven93}\footnote{The
proof given in \cite{Malven93} requires $\mathbf{k}$ to be a $\mathbb{Q}%
$-algebra, but it is easy to see that the truth of \eqref{eq.SMalpha} for
$\mathbf{k} = \mathbb{Q}$ implies the truth of \eqref{eq.SMalpha} for every
commutative ring $\mathbf{k}$.} and in \cite[Theorem 5.1.11]{GriRei}%
\footnote{Note that \cite[Theorem 5.1.11]{GriRei} uses the notation
$\operatorname*{rev}\alpha$ for the composition $\left(  \alpha_{\ell}%
,\alpha_{\ell-1},\ldots,\alpha_{1}\right)  $, and writes \textquotedblleft%
$\gamma$ coarsens $\operatorname*{rev}\alpha$\textquotedblright\ for what we
call \textquotedblleft$\gamma\in\operatorname*{Comp}\nolimits_{n}$ and
$D\left(  \gamma\right)  \subseteq D\left(  \operatorname*{rev}\alpha\right)
$\textquotedblright.} or in \cite[detailed version, Proposition 10.70]%
{Grinbe15}) can be used to define $S$ (since $S$ is to be $\mathbf{k}%
$-linear). Also, for each composition $\alpha$, we have $S\left(  L_{\alpha
}\right)  =\left(  -1\right)  ^{\left\vert \alpha\right\vert }L_{\omega\left(
\alpha\right)  }$, where $\omega\left(  \alpha\right)  $ is a certain
composition known as the \emph{conjugate} of $\alpha$. See \cite[Corollaire
4.20]{Malven93} or \cite[Theorem 5.1.11 and Proposition 5.2.15]{GriRei} for
details and proofs. Note that $S$ is a $\mathbf{k}$-algebra homomorphism and
an involution (that is, $S^{2}=\operatorname*{id}$). (Again, this is derived
from abstract algebraic properties of antipodes in \cite{GriRei}, but can also
be shown more directly.)

We will prove two formulas for the antipode of $\eta_{\alpha}^{\left(
q\right)  }$. Both rely on the following notation:

\begin{definition}
\label{def.rev}If $\alpha=\left(  \alpha_{1},\alpha_{2},\ldots,\alpha_{\ell
}\right)  $ is a composition, then the \emph{reversal} of $\alpha$ is defined
to be the composition $\left(  \alpha_{\ell},\alpha_{\ell-1},\ldots,\alpha
_{1}\right)  $. It is denoted by $\operatorname*{rev}\alpha$.
\end{definition}

We are now ready to state our first formula for the antipode of $\eta_{\alpha
}^{\left(  q\right)  }$ in the case when $q$ is invertible:

\begin{theorem}
\label{thm.eta.S}Let $p\in\mathbf{k}$ be such that $pq=1$. Let $\alpha
\in\operatorname*{Comp}$, and let $n=\left\vert \alpha\right\vert $. Then, the
antipode $S$ of $\operatorname*{QSym}$ satisfies
\[
S\left(  \eta_{\alpha}^{\left(  q\right)  }\right)  =\left(  -q\right)
^{\ell\left(  \alpha\right)  }\eta_{\operatorname*{rev}\alpha}^{\left(
p\right)  }.
\]

\end{theorem}

\begin{proof}
From $pq=1$, we obtain $p=q^{-1}$. Furthermore, $\left(  p+1\right)
q=\underbrace{pq}_{=1}+q=1+q=q+1=r$. Solving this for $p+1$, we obtain%
\begin{equation}
p+1=rq^{-1}. \label{pf.prop.eta.S.r=}%
\end{equation}

We shall need a few more features of compositions. For any composition
$\gamma\in\operatorname*{Comp}\nolimits_{n}$, we let $\omega\left(
\gamma\right)  $ denote the unique composition of $n$ satisfying%
\begin{equation}
D\left(  \omega\left(  \gamma\right)  \right)  =\left[  n-1\right]  \setminus
D\left(  \operatorname*{rev}\gamma\right)  . \label{pf.prop.eta.S.Domgam}%
\end{equation}
(This $\omega\left(  \gamma\right)  $ is indeed unique, since the map $D$ is a
bijection.) Then, a classical formula (\cite[(4.27)]{Malven93} or
\cite[(5.2.7)]{GriRei}) says that each $\gamma\in\operatorname*{Comp}%
\nolimits_{n}$ satisfies%
\begin{equation}
S\left(  L_{\gamma}\right)  =\left(  -1\right)  ^{n}L_{\omega\left(
\gamma\right)  }. \label{pf.prop.eta.S.SLgam}%
\end{equation}
It is also easy to prove (see, e.g., \cite[Proposition 4.3 \textbf{(d)}%
]{comps}) that%
\begin{equation}
\omega\left(  \omega\left(  \gamma\right)  \right)  =\gamma
\ \ \ \ \ \ \ \ \ \ \text{for any }\gamma\in\operatorname*{Comp}\nolimits_{n}.
\label{pf.prop.eta.S.oo=id}%
\end{equation}
Thus, the map $\omega:\operatorname*{Comp}\nolimits_{n}\rightarrow
\operatorname*{Comp}\nolimits_{n}$ (which sends each $\gamma\in
\operatorname*{Comp}\nolimits_{n}$ to $\omega\left(  \gamma\right)  $) is a bijection.

We WLOG assume that $n\neq0$ (since the claim of Theorem \ref{thm.eta.S} is
easily checked by hand in the case when $n=0$).

From $n=\left\vert \alpha\right\vert $, we obtain $\alpha\in
\operatorname*{Comp}\nolimits_{n}$.

\begin{noncompile}
Thus, Lemma \ref{lem.comps.l-vs-size} \textbf{(a)} (applied to $\delta=\alpha
$) yields $\ell\left(  \alpha\right)  =\left\vert D\left(  \alpha\right)
\right\vert +\underbrace{\left[  n\neq0\right]  }_{\substack{=1\\\text{(since
}n\neq0\text{)}}}=\left\vert D\left(  \alpha\right)  \right\vert +1$. Hence,%
\begin{equation}
\left\vert D\left(  \alpha\right)  \right\vert =\ell\left(  \alpha\right)  -1.
\label{pf.prop.eta.S.Da}%
\end{equation}

\end{noncompile}

Now, we make the following combinatorial observation:

\begin{statement}
\textit{Observation 1:} Let $\gamma\in\operatorname*{Comp}\nolimits_{n}$.
Then,%
\begin{equation}
\left\vert D\left(  \omega\left(  \gamma\right)  \right)  \cap D\left(
\alpha\right)  \right\vert =\ell\left(  \alpha\right)  -1-\left\vert D\left(
\gamma\right)  \cap D\left(  \operatorname*{rev}\alpha\right)  \right\vert
\label{pf.prop.eta.S.O1.1}%
\end{equation}
and%
\begin{equation}
\left\vert D\left(  \omega\left(  \gamma\right)  \right)  \setminus D\left(
\alpha\right)  \right\vert =n-\ell\left(  \alpha\right)  -\left\vert D\left(
\gamma\right)  \setminus D\left(  \operatorname*{rev}\alpha\right)
\right\vert . \label{pf.prop.eta.S.O1.2}%
\end{equation}

\end{statement}

The proof of Observation 1 is laborious but fairly straightforward, and can be
found in \cite[Proposition 4.4]{comps}.

\begin{noncompile}
[\textit{Proof of Observation 1:} For any subset $X$ of $\left[  n-1\right]
$, we let $\operatorname*{rev}X$ denote the subset $\left\{  n-x\ \mid\ x\in
X\right\}  $ of $\left[  n-1\right]  $. Thus, we have defined a map
$\operatorname*{rev}:\mathcal{P}\left(  \left[  n-1\right]  \right)
\rightarrow\mathcal{P}\left(  \left[  n-1\right]  \right)  $ (where
$\mathcal{P}\left(  \left[  n-1\right]  \right)  $ is the powerset of $\left[
n-1\right]  $). It is easy to see that this map $\operatorname*{rev}$

\begin{itemize}
\item is an involution (i.e., it satisfies $\operatorname*{rev}\left(
\operatorname*{rev}X\right)  =X$ for any subset $X$ of $\left[  n-1\right]  $),

\item preserves sizes (i.e., we have $\left\vert \operatorname*{rev}%
X\right\vert =\left\vert X\right\vert $ for any subset $X$ of $\left[
n-1\right]  $),

\item respects set-theoretic operations like intersection and set difference
(i.e., any two subsets $X$ and $Y$ of $\left[  n-1\right]  $ satisfy
$\operatorname*{rev}\left(  X\cap Y\right)  =\left(  \operatorname*{rev}%
X\right)  \cap\left(  \operatorname*{rev}Y\right)  $ and $\operatorname*{rev}%
\left(  X\setminus Y\right)  =\left(  \operatorname*{rev}X\right)
\setminus\left(  \operatorname*{rev}Y\right)  $),

\item and satisfies $\operatorname*{rev}\left(  \left[  n-1\right]  \right)
=\left[  n-1\right]  $.
\end{itemize}

It is furthermore easy to see that $D\left(  \operatorname*{rev}\gamma\right)
=\operatorname*{rev}\left(  D\left(  \gamma\right)  \right)  $ (because if we
write the composition $\gamma$ as $\gamma=\left(  \gamma_{1},\gamma_{2}%
,\ldots,\gamma_{k}\right)  $, then $\gamma_{1}+\gamma_{2}+\cdots+\gamma
_{k}=\left\vert \gamma\right\vert =n$ (since $\gamma\in\operatorname*{Comp}%
\nolimits_{n}$), and therefore $\gamma_{i+1}+\gamma_{i+2}+\cdots+\gamma
_{k}=n-\left(  \gamma_{1}+\gamma_{2}+\cdots+\gamma_{i}\right)  $ for each
$i\in\left[  k-1\right]  $). Hence, (\ref{pf.prop.eta.S.Domgam}) can be
rewritten as%
\begin{equation}
D\left(  \omega\left(  \gamma\right)  \right)  =\left[  n-1\right]
\setminus\operatorname*{rev}\left(  D\left(  \gamma\right)  \right)  .
\label{pf.prop.eta.S.O1.Dwg}%
\end{equation}

Recall that we proved that $D\left(  \operatorname*{rev}\gamma\right)
=\operatorname*{rev}\left(  D\left(  \gamma\right)  \right)  $. The same
argument (applied to $\alpha$ instead of $\gamma$) yields $D\left(
\operatorname*{rev}\alpha\right)  =\operatorname*{rev}\left(  D\left(
\alpha\right)  \right)  $. Thus, $\operatorname*{rev}\left(  D\left(
\operatorname*{rev}\alpha\right)  \right)  =\operatorname*{rev}\left(
\operatorname*{rev}\left(  D\left(  \alpha\right)  \right)  \right)  =D\left(
\alpha\right)  $ (since $\operatorname*{rev}\left(  \operatorname*{rev}%
X\right)  =X$ for any subset $X$ of $\left[  n-1\right]  $). In other words,%
\begin{equation}
D\left(  \alpha\right)  =\operatorname*{rev}\left(  D\left(
\operatorname*{rev}\alpha\right)  \right)  . \label{pf.prop.eta.S.O1.rra}%
\end{equation}

Now,%
\begin{align*}
&  \underbrace{D\left(  \omega\left(  \gamma\right)  \right)  }%
_{\substack{=\left[  n-1\right]  \setminus\operatorname*{rev}\left(  D\left(
\gamma\right)  \right)  \\\text{(by (\ref{pf.prop.eta.S.O1.Dwg}))}}}\cap
D\left(  \alpha\right) \\
&  =\left(  \left[  n-1\right]  \setminus\operatorname*{rev}\left(  D\left(
\gamma\right)  \right)  \right)  \cap D\left(  \alpha\right) \\
&  =\underbrace{\left(  \left[  n-1\right]  \cap D\left(  \alpha\right)
\right)  }_{\substack{=D\left(  \alpha\right)  \\\text{(since }D\left(
\alpha\right)  \subseteq\left[  n-1\right]  \text{)}}}\setminus
\operatorname*{rev}\left(  D\left(  \gamma\right)  \right)
\ \ \ \ \ \ \ \ \ \ \left(
\begin{array}
[c]{c}%
\text{since }\left(  X\setminus Y\right)  \cap Z=\left(  X\cap Z\right)
\setminus Y\\
\text{for any sets }X\text{, }Y\text{ and }Z
\end{array}
\right) \\
&  =D\left(  \alpha\right)  \setminus\operatorname*{rev}\left(  D\left(
\gamma\right)  \right)  .
\end{align*}
Thus,%
\begin{align}
\left\vert D\left(  \omega\left(  \gamma\right)  \right)  \cap D\left(
\alpha\right)  \right\vert  &  =\left\vert D\left(  \alpha\right)
\setminus\operatorname*{rev}\left(  D\left(  \gamma\right)  \right)
\right\vert \nonumber\\
&  =\left\vert D\left(  \alpha\right)  \right\vert -\left\vert D\left(
\alpha\right)  \cap\operatorname*{rev}\left(  D\left(  \gamma\right)  \right)
\right\vert \label{pf.prop.eta.S.O1.3}%
\end{align}
(since any finite sets $X$ and $Y$ satisfy $\left\vert X\setminus Y\right\vert
=\left\vert X\right\vert -\left\vert X\cap Y\right\vert $).

However, we have%
\begin{align*}
\underbrace{D\left(  \alpha\right)  }_{\substack{=\operatorname*{rev}\left(
D\left(  \operatorname*{rev}\alpha\right)  \right)  \\\text{(by
(\ref{pf.prop.eta.S.O1.rra}))}}}\cap\operatorname*{rev}\left(  D\left(
\gamma\right)  \right)   &  =\operatorname*{rev}\left(  D\left(
\operatorname*{rev}\alpha\right)  \right)  \cap\operatorname*{rev}\left(
D\left(  \gamma\right)  \right) \\
&  =\operatorname*{rev}\left(  D\left(  \operatorname*{rev}\alpha\right)  \cap
D\left(  \gamma\right)  \right)
\end{align*}
(since the map $\operatorname*{rev}$ respects intersection of sets) and thus%
\[
\left\vert D\left(  \alpha\right)  \cap\operatorname*{rev}\left(  D\left(
\gamma\right)  \right)  \right\vert =\left\vert \operatorname*{rev}\left(
D\left(  \operatorname*{rev}\alpha\right)  \cap D\left(  \gamma\right)
\right)  \right\vert =\left\vert D\left(  \operatorname*{rev}\alpha\right)
\cap D\left(  \gamma\right)  \right\vert
\]
(since $\left\vert \operatorname*{rev}X\right\vert =\left\vert X\right\vert $
for any subset $X$ of $\left[  n-1\right]  $). Therefore,
(\ref{pf.prop.eta.S.O1.3}) becomes%
\begin{align*}
\left\vert D\left(  \omega\left(  \gamma\right)  \right)  \cap D\left(
\alpha\right)  \right\vert  &  =\underbrace{\left\vert D\left(  \alpha\right)
\right\vert }_{\substack{=\ell\left(  \alpha\right)  -1\\\text{(by
(\ref{pf.prop.eta.S.Da}))}}}-\underbrace{\left\vert D\left(  \alpha\right)
\cap\operatorname*{rev}\left(  D\left(  \gamma\right)  \right)  \right\vert
}_{\substack{=\left\vert D\left(  \operatorname*{rev}\alpha\right)  \cap
D\left(  \gamma\right)  \right\vert \\=\left\vert D\left(  \gamma\right)  \cap
D\left(  \operatorname*{rev}\alpha\right)  \right\vert }}\\
&  =\ell\left(  \alpha\right)  -1-\left\vert D\left(  \gamma\right)  \cap
D\left(  \operatorname*{rev}\alpha\right)  \right\vert .
\end{align*}

From (\ref{pf.prop.eta.S.O1.Dwg}), we obtain%
\begin{align}
\left\vert D\left(  \omega\left(  \gamma\right)  \right)  \right\vert  &
=\left\vert \left[  n-1\right]  \setminus\operatorname*{rev}\left(  D\left(
\gamma\right)  \right)  \right\vert \nonumber\\
&  =\underbrace{\left\vert \left[  n-1\right]  \right\vert }%
_{\substack{=n-1\\\text{(since }n\neq0\text{ and}\\\text{thus }n-1\in
\mathbb{N}\text{)}}}-\underbrace{\left\vert \operatorname*{rev}\left(
D\left(  \gamma\right)  \right)  \right\vert }_{\substack{=\left\vert D\left(
\gamma\right)  \right\vert \\\text{(since }\left\vert \operatorname*{rev}%
X\right\vert =\left\vert X\right\vert \\\text{for any subset }X\\\text{of
}\left[  n-1\right]  \text{)}}}\ \ \ \ \ \ \ \ \ \ \left(  \text{since
}\operatorname*{rev}\left(  D\left(  \gamma\right)  \right)  \subseteq\left[
n-1\right]  \right) \nonumber\\
&  =n-1-\left\vert D\left(  \gamma\right)  \right\vert .
\label{pf.prop.eta.S.O1.Dwgsize}%
\end{align}

Next, recall that $\left\vert X\setminus Y\right\vert =\left\vert X\right\vert
-\left\vert X\cap Y\right\vert $ for any two finite sets $X$ and $Y$. From
this equality, we obtain%
\begin{equation}
\left\vert D\left(  \omega\left(  \gamma\right)  \right)  \setminus D\left(
\alpha\right)  \right\vert =\left\vert D\left(  \omega\left(  \gamma\right)
\right)  \right\vert -\left\vert D\left(  \omega\left(  \gamma\right)
\right)  \cap D\left(  \alpha\right)  \right\vert \label{pf.prop.eta.S.O1.b1}%
\end{equation}
and%
\begin{equation}
\left\vert D\left(  \gamma\right)  \setminus D\left(  \operatorname*{rev}%
\alpha\right)  \right\vert =\left\vert D\left(  \gamma\right)  \right\vert
-\left\vert D\left(  \gamma\right)  \cap D\left(  \operatorname*{rev}%
\alpha\right)  \right\vert . \label{pf.prop.eta.S.O1.b2}%
\end{equation}
Adding these two equalities together, we find%
\begin{align*}
&  \left\vert D\left(  \omega\left(  \gamma\right)  \right)  \setminus
D\left(  \alpha\right)  \right\vert +\left\vert D\left(  \gamma\right)
\setminus D\left(  \operatorname*{rev}\alpha\right)  \right\vert \\
&  =\underbrace{\left\vert D\left(  \omega\left(  \gamma\right)  \right)
\right\vert }_{\substack{=n-1-\left\vert D\left(  \gamma\right)  \right\vert
\\\text{(by (\ref{pf.prop.eta.S.O1.Dwgsize}))}}}-\underbrace{\left\vert
D\left(  \omega\left(  \gamma\right)  \right)  \cap D\left(  \alpha\right)
\right\vert }_{\substack{=\ell\left(  \alpha\right)  -1-\left\vert D\left(
\gamma\right)  \cap D\left(  \operatorname*{rev}\alpha\right)  \right\vert
\\\text{(by (\ref{pf.prop.eta.S.O1.1}))}}}+\left\vert D\left(  \gamma\right)
\right\vert -\left\vert D\left(  \gamma\right)  \cap D\left(
\operatorname*{rev}\alpha\right)  \right\vert \\
&  =n-1-\left\vert D\left(  \gamma\right)  \right\vert -\left(  \ell\left(
\alpha\right)  -1-\left\vert D\left(  \gamma\right)  \cap D\left(
\operatorname*{rev}\alpha\right)  \right\vert \right)  +\left\vert D\left(
\gamma\right)  \right\vert -\left\vert D\left(  \gamma\right)  \cap D\left(
\operatorname*{rev}\alpha\right)  \right\vert \\
&  =n-\ell\left(  \alpha\right)  .
\end{align*}
In other words,%
\[
\left\vert D\left(  \omega\left(  \gamma\right)  \right)  \setminus D\left(
\alpha\right)  \right\vert =n-\ell\left(  \alpha\right)  -\left\vert D\left(
\gamma\right)  \setminus D\left(  \operatorname*{rev}\alpha\right)
\right\vert .
\]
This proves (\ref{pf.prop.eta.S.O1.2}). Thus, Observation 1 is proven.]
\medskip
\end{noncompile}

Now, Proposition \ref{prop.eta.through-F} (applied to $\operatorname*{rev}%
\alpha$, $p$ and $p+1$ instead of $\alpha$, $q$ and $r$) yields
\begin{equation}
\eta_{\operatorname*{rev}\alpha}^{\left(  p\right)  }=\left(  p+1\right)
\sum_{\gamma\in\operatorname*{Comp}\nolimits_{n}}\left(  -1\right)
^{\left\vert D\left(  \gamma\right)  \setminus D\left(  \operatorname*{rev}%
\alpha\right)  \right\vert }p^{\left\vert D\left(  \gamma\right)  \cap
D\left(  \operatorname*{rev}\alpha\right)  \right\vert }L_{\gamma}.
\label{pf.prop.eta.S.O1.etarev}%
\end{equation}

On the other hand, Proposition \ref{prop.eta.through-F} yields%
\[
\eta_{\alpha}^{\left(  q\right)  }=r\sum_{\gamma\in\operatorname*{Comp}%
\nolimits_{n}}\left(  -1\right)  ^{\left\vert D\left(  \gamma\right)
\setminus D\left(  \alpha\right)  \right\vert }q^{\left\vert D\left(
\gamma\right)  \cap D\left(  \alpha\right)  \right\vert }L_{\gamma}.
\]
Applying the map $S$ to both sides of this equality, we obtain%
\begin{align*}
S\left(  \eta_{\alpha}^{\left(  q\right)  }\right)   &  =S\left(
r\sum_{\gamma\in\operatorname*{Comp}\nolimits_{n}}\left(  -1\right)
^{\left\vert D\left(  \gamma\right)  \setminus D\left(  \alpha\right)
\right\vert }q^{\left\vert D\left(  \gamma\right)  \cap D\left(
\alpha\right)  \right\vert }L_{\gamma}\right) \\
&  =r\sum_{\gamma\in\operatorname*{Comp}\nolimits_{n}}\left(  -1\right)
^{\left\vert D\left(  \gamma\right)  \setminus D\left(  \alpha\right)
\right\vert }q^{\left\vert D\left(  \gamma\right)  \cap D\left(
\alpha\right)  \right\vert }\underbrace{S\left(  L_{\gamma}\right)
}_{\substack{=\left(  -1\right)  ^{n}L_{\omega\left(  \gamma\right)
}\\\text{(by (\ref{pf.prop.eta.S.SLgam}))}}}\ \ \ \ \ \ \ \ \ \ \left(
\begin{array}
[c]{c}%
\text{since the map }S\\
\text{is }\mathbf{k}\text{-linear}%
\end{array}
\right) \\
&  =r\sum_{\gamma\in\operatorname*{Comp}\nolimits_{n}}\left(  -1\right)
^{\left\vert D\left(  \gamma\right)  \setminus D\left(  \alpha\right)
\right\vert }q^{\left\vert D\left(  \gamma\right)  \cap D\left(
\alpha\right)  \right\vert }\left(  -1\right)  ^{n}L_{\omega\left(
\gamma\right)  }\\
&  =r\sum_{\gamma\in\operatorname*{Comp}\nolimits_{n}}\underbrace{\left(
-1\right)  ^{\left\vert D\left(  \omega\left(  \gamma\right)  \right)
\setminus D\left(  \alpha\right)  \right\vert }q^{\left\vert D\left(
\omega\left(  \gamma\right)  \right)  \cap D\left(  \alpha\right)  \right\vert
}\left(  -1\right)  ^{n}}_{=\left(  -1\right)  ^{n}\left(  -1\right)
^{\left\vert D\left(  \omega\left(  \gamma\right)  \right)  \setminus D\left(
\alpha\right)  \right\vert }q^{\left\vert D\left(  \omega\left(
\gamma\right)  \right)  \cap D\left(  \alpha\right)  \right\vert }%
}\underbrace{L_{\omega\left(  \omega\left(  \gamma\right)  \right)  }%
}_{\substack{=L_{\gamma}\\\text{(by (\ref{pf.prop.eta.S.oo=id}))}}}\\
&  \ \ \ \ \ \ \ \ \ \ \ \ \ \ \ \ \ \ \ \ \left(
\begin{array}
[c]{c}%
\text{here, we have substituted }\omega\left(  \gamma\right)  \text{ for
}\gamma\text{ in the sum,}\\
\text{since the map }\omega:\operatorname*{Comp}\nolimits_{n}\rightarrow
\operatorname*{Comp}\nolimits_{n}\text{ is a bijection}%
\end{array}
\right) \\
&  =r\sum_{\gamma\in\operatorname*{Comp}\nolimits_{n}}\left(  -1\right)
^{n}\underbrace{\left(  -1\right)  ^{\left\vert D\left(  \omega\left(
\gamma\right)  \right)  \setminus D\left(  \alpha\right)  \right\vert }%
}_{\substack{=\left(  -1\right)  ^{n-\ell\left(  \alpha\right)  -\left\vert
D\left(  \gamma\right)  \setminus D\left(  \operatorname*{rev}\alpha\right)
\right\vert }\\\text{(by (\ref{pf.prop.eta.S.O1.2}))}}%
}\ \ \underbrace{q^{\left\vert D\left(  \omega\left(  \gamma\right)  \right)
\cap D\left(  \alpha\right)  \right\vert }}_{\substack{=q^{\ell\left(
\alpha\right)  -1-\left\vert D\left(  \gamma\right)  \cap D\left(
\operatorname*{rev}\alpha\right)  \right\vert }\\\text{(by
(\ref{pf.prop.eta.S.O1.1}))}}}L_{\gamma}\\
&  =r\sum_{\gamma\in\operatorname*{Comp}\nolimits_{n}}\underbrace{\left(
-1\right)  ^{n}\left(  -1\right)  ^{n-\ell\left(  \alpha\right)  -\left\vert
D\left(  \gamma\right)  \setminus D\left(  \operatorname*{rev}\alpha\right)
\right\vert }}_{=\left(  -1\right)  ^{\ell\left(  \alpha\right)  }\left(
-1\right)  ^{\left\vert D\left(  \gamma\right)  \setminus D\left(
\operatorname*{rev}\alpha\right)  \right\vert }}\underbrace{q^{\ell\left(
\alpha\right)  -1-\left\vert D\left(  \gamma\right)  \cap D\left(
\operatorname*{rev}\alpha\right)  \right\vert }}_{\substack{=q^{\ell\left(
\alpha\right)  }q^{-1}q^{-\left\vert D\left(  \gamma\right)  \cap D\left(
\operatorname*{rev}\alpha\right)  \right\vert }}}L_{\gamma}\\
&  =r\sum_{\gamma\in\operatorname*{Comp}\nolimits_{n}}\left(  -1\right)
^{\ell\left(  \alpha\right)  }\left(  -1\right)  ^{\left\vert D\left(
\gamma\right)  \setminus D\left(  \operatorname*{rev}\alpha\right)
\right\vert }q^{\ell\left(  \alpha\right)  }q^{-1}q^{-\left\vert D\left(
\gamma\right)  \cap D\left(  \operatorname*{rev}\alpha\right)  \right\vert
}L_{\gamma}\\
&  =\underbrace{rq^{-1}}_{\substack{=p+1\\\text{(by (\ref{pf.prop.eta.S.r=}%
))}}}\underbrace{\left(  -1\right)  ^{\ell\left(  \alpha\right)  }%
q^{\ell\left(  \alpha\right)  }}_{=\left(  -q\right)  ^{\ell\left(
\alpha\right)  }}\sum_{\gamma\in\operatorname*{Comp}\nolimits_{n}}\left(
-1\right)  ^{\left\vert D\left(  \gamma\right)  \setminus D\left(
\operatorname*{rev}\alpha\right)  \right\vert }\underbrace{q^{-\left\vert
D\left(  \gamma\right)  \cap D\left(  \operatorname*{rev}\alpha\right)
\right\vert }}_{\substack{=\left(  q^{-1}\right)  ^{\left\vert D\left(
\gamma\right)  \cap D\left(  \operatorname*{rev}\alpha\right)  \right\vert
}\\=p^{\left\vert D\left(  \gamma\right)  \cap D\left(  \operatorname*{rev}%
\alpha\right)  \right\vert }\\\text{(since }q^{-1}=p\text{)}}}L_{\gamma}\\
&  =\left(  p+1\right)  \left(  -q\right)  ^{\ell\left(  \alpha\right)  }%
\sum_{\gamma\in\operatorname*{Comp}\nolimits_{n}}\left(  -1\right)
^{\left\vert D\left(  \gamma\right)  \setminus D\left(  \operatorname*{rev}%
\alpha\right)  \right\vert }p^{\left\vert D\left(  \gamma\right)  \cap
D\left(  \operatorname*{rev}\alpha\right)  \right\vert }L_{\gamma}\\
&  =\left(  -q\right)  ^{\ell\left(  \alpha\right)  }\underbrace{\left(
p+1\right)  \sum_{\gamma\in\operatorname*{Comp}\nolimits_{n}}\left(
-1\right)  ^{\left\vert D\left(  \gamma\right)  \setminus D\left(
\operatorname*{rev}\alpha\right)  \right\vert }p^{\left\vert D\left(
\gamma\right)  \cap D\left(  \operatorname*{rev}\alpha\right)  \right\vert
}L_{\gamma}}_{\substack{=\eta_{\operatorname*{rev}\alpha}^{\left(  p\right)
}\\\text{(by (\ref{pf.prop.eta.S.O1.etarev}))}}}\\
&  =\left(  -q\right)  ^{\ell\left(  \alpha\right)  }\eta_{\operatorname*{rev}%
\alpha}^{\left(  p\right)  }.
\end{align*}
This proves Theorem \ref{thm.eta.S}.
\end{proof}

Theorem \ref{thm.eta.S} generalizes \cite[Proposition 2.9]{Hsiao07}.

Our second formula for the antipode of $\eta_{\alpha}^{\left(  q\right)  }$
comes with no requirement on $q$, but is somewhat more complicated:

\begin{theorem}
\label{thm.eta.S2}Let $n\in\mathbb{N}$. Let $\alpha\in\operatorname*{Comp}%
\nolimits_{n}$. Then, the antipode $S$ of $\operatorname*{QSym}$ satisfies
\[
S\left(  \eta_{\alpha}^{\left(  q\right)  }\right)  =\left(  -1\right)
^{\ell\left(  \alpha\right)  }\sum_{\substack{\beta\in\operatorname*{Comp}%
\nolimits_{n};\\D\left(  \beta\right)  \subseteq D\left(  \operatorname*{rev}%
\alpha\right)  }}\left(  q-1\right)  ^{\ell\left(  \alpha\right)  -\ell\left(
\beta\right)  }\eta_{\beta}^{\left(  q\right)  }.
\]

\end{theorem}

\begin{vershort}
We will prove this theorem using the following summation lemma:
\end{vershort}

\begin{verlong}
We will prove this theorem via a series of three lemmas. The first one is a
simple combinatorial identity for finite sets:

\begin{lemma}
\label{lem.set-sandwich-sum}Let $A$ and $C$ be two finite sets such that
$C\subseteq A$. Let $u,v\in\mathbf{k}$. Then,%
\[
\sum_{\substack{I\subseteq A;\\C\subseteq I}}u^{\left\vert I\right\vert
-\left\vert C\right\vert }v^{\left\vert A\right\vert -\left\vert I\right\vert
}=\left(  u+v\right)  ^{\left\vert A\right\vert -\left\vert C\right\vert }.
\]

\end{lemma}

\begin{proof}
[Proof of Lemma \ref{lem.set-sandwich-sum}.]Let $F:=A\setminus C$. Thus, $F$
is a finite set (since $A$ and $C$ are finite sets). Moreover, from
$F=A\setminus C$, we obtain
\begin{equation}
\left\vert F\right\vert =\left\vert A\setminus C\right\vert =\left\vert
A\right\vert -\left\vert C\right\vert \label{pf.lem.set-sandwich-sum.absF=}%
\end{equation}
(since $C\subseteq A$).

If $I$ is a subset of $A$ satisfying $C\subseteq I$, then $I\setminus C$ is a
subset of $F$\ \ \ \ \footnote{\textit{Proof.} Let $I$ be a subset of $A$
satisfying $C\subseteq I$. Then, $I\subseteq A$ (since $I$ is a subset of
$A$), so that $\underbrace{I}_{\subseteq A}\setminus C\subseteq A\setminus
C=F$. In other words, $I\setminus C$ is a subset of $F$. Qed.}. Hence, the map%
\begin{align*}
\Phi:\left\{  \text{subsets }I\text{ of }A\text{ satisfying }C\subseteq
I\right\}   &  \rightarrow\left\{  \text{subsets of }F\right\}  ,\\
I  &  \mapsto I\setminus C
\end{align*}
is well-defined. Consider this map $\Phi$.

If $J$ is a subset of $F$, then $J\cup C$ is a subset $I$ of $A$ satisfying
$C\subseteq I$\ \ \ \ \footnote{\textit{Proof.} Let $J$ be a subset of $F$.
Thus, $J\subseteq F=A\setminus C\subseteq A$. Hence, $\underbrace{J}%
_{\subseteq A}\cup\underbrace{C}_{\subseteq A}\subseteq A\cup A=A$. Thus,
$J\cup C$ is a subset of $A$. Since $C\subseteq J\cup C$, we thus conclude
that $J\cup C$ is a subset $I$ of $A$ satisfying $C\subseteq I$. Qed.}. Hence,
the map%
\begin{align*}
\Psi:\left\{  \text{subsets of }F\right\}   &  \rightarrow\left\{
\text{subsets }I\text{ of }A\text{ satisfying }C\subseteq I\right\}  ,\\
J  &  \mapsto J\cup C
\end{align*}
is well-defined. Consider this map $\Psi$.

We have $\Phi\circ\Psi=\operatorname*{id}$\ \ \ \ \footnote{\textit{Proof.}
Let $J\in\left\{  \text{subsets of }F\right\}  $. Then, $J$ is a subset of
$F$. In other words, $J\subseteq F=A\setminus C$. Hence, each $j\in J$
satisfies $j\in J\subseteq A\setminus C$ and therefore $j\notin C$. In other
words, the sets $J$ and $C$ are disjoint.
\par
But the definition of $\Phi$ yields
\[
\Phi\left(  \Psi\left(  J\right)  \right)  =\underbrace{\Psi\left(  J\right)
}_{\substack{=J\cup C\\\text{(by the definition of }\Psi\text{)}}}\setminus
C=\left(  J\cup C\right)  \setminus C=J
\]
(since the sets $J$ and $C$ are disjoint). Thus, $\left(  \Phi\circ
\Psi\right)  \left(  J\right)  =\Phi\left(  \Psi\left(  J\right)  \right)
=J=\operatorname*{id}\left(  J\right)  $.
\par
Forget that we fixed $J$. We thus have shown that $\left(  \Phi\circ
\Psi\right)  \left(  J\right)  =\operatorname*{id}\left(  J\right)  $ for each
$J\in\left\{  \text{subsets of }F\right\}  $. In other words, $\Phi\circ
\Psi=\operatorname*{id}$.} and $\Psi\circ\Phi=\operatorname*{id}%
$\ \ \ \ \footnote{\textit{Proof.} Let $K\in\left\{  \text{subsets }I\text{ of
}A\text{ satisfying }C\subseteq I\right\}  $. Then, $K$ is a subset $I$ of $A$
satisfying $C\subseteq I$. In other words, $K$ is a subset of $A$ and
satisfies $C\subseteq K$.
\par
But the definition of $\Psi$ yields%
\[
\Psi\left(  \Phi\left(  K\right)  \right)  =\underbrace{\Phi\left(  K\right)
}_{\substack{=K\setminus C\\\text{(by the definition of }\Phi\text{)}}}\cup
C=\left(  K\setminus C\right)  \cup C=K\ \ \ \ \ \ \ \ \ \ \left(  \text{since
}C\subseteq K\right)  .
\]
Thus, $\left(  \Psi\circ\Phi\right)  \left(  K\right)  =\Psi\left(
\Phi\left(  K\right)  \right)  =K=\operatorname*{id}\left(  K\right)  $.
\par
Forget that we fixed $K$. We thus have shown that $\left(  \Psi\circ
\Phi\right)  \left(  K\right)  =\operatorname*{id}\left(  K\right)  $ for each
$K\in\left\{  \text{subsets }I\text{ of }A\text{ satisfying }C\subseteq
I\right\}  $. In other words, $\Psi\circ\Phi=\operatorname*{id}$.}. Hence, the
maps $\Phi$ and $\Psi$ are mutually inverse. Thus, the map $\Psi$ is
invertible, i.e., is a bijection. Therefore, we can substitute $\Psi\left(
J\right)  $ for $I$ in the sum $\sum_{\substack{I\subseteq A;\\C\subseteq
I}}u^{\left\vert I\right\vert -\left\vert C\right\vert }v^{\left\vert
A\right\vert -\left\vert I\right\vert }$. As a result, we obtain%
\begin{equation}
\sum_{\substack{I\subseteq A;\\C\subseteq I}}u^{\left\vert I\right\vert
-\left\vert C\right\vert }v^{\left\vert A\right\vert -\left\vert I\right\vert
}=\sum_{J\subseteq F}u^{\left\vert \Psi\left(  J\right)  \right\vert
-\left\vert C\right\vert }v^{\left\vert A\right\vert -\left\vert \Psi\left(
J\right)  \right\vert }. \label{pf.lem.set-sandwich-sum.1}%
\end{equation}

However, if $J$ is any subset of $F$, then%
\begin{equation}
\left\vert \Psi\left(  J\right)  \right\vert =\left\vert J\right\vert
+\left\vert C\right\vert \label{pf.lem.set-sandwich-sum.2a}%
\end{equation}
\footnote{\textit{Proof.} Let $J$ be a subset of $F$. Thus, $J\subseteq F$.
Hence, each $j\in J$ satisfies $j\in J\subseteq F=A\setminus C$ and therefore
$j\notin C$. In other words, the sets $J$ and $C$ are disjoint. Hence,
$\left\vert J\cup C\right\vert =\left\vert J\right\vert +\left\vert
C\right\vert $. But the definition of $\Psi$ yields $\Psi\left(  J\right)
=J\cup C$. Hence, $\left\vert \Psi\left(  J\right)  \right\vert =\left\vert
J\cup C\right\vert =\left\vert J\right\vert +\left\vert C\right\vert $. This
proves (\ref{pf.lem.set-sandwich-sum.2a}).} and therefore%
\begin{equation}
\left\vert \Psi\left(  J\right)  \right\vert -\left\vert C\right\vert
=\left\vert J\right\vert \label{pf.lem.set-sandwich-sum.2b}%
\end{equation}
and%
\begin{align}
\left\vert A\right\vert -\underbrace{\left\vert \Psi\left(  J\right)
\right\vert }_{=\left\vert J\right\vert +\left\vert C\right\vert }  &
=\left\vert A\right\vert -\left(  \left\vert J\right\vert +\left\vert
C\right\vert \right)  =\underbrace{\left\vert A\right\vert -\left\vert
C\right\vert }_{\substack{=\left\vert F\right\vert \\\text{(by
(\ref{pf.lem.set-sandwich-sum.absF=}))}}}-\left\vert J\right\vert \nonumber\\
&  =\left\vert F\right\vert -\left\vert J\right\vert .
\label{pf.lem.set-sandwich-sum.2c}%
\end{align}

Now, (\ref{pf.lem.set-sandwich-sum.1}) becomes%
\begin{align*}
\sum_{\substack{I\subseteq A;\\C\subseteq I}}u^{\left\vert I\right\vert
-\left\vert C\right\vert }v^{\left\vert A\right\vert -\left\vert I\right\vert
}  &  =\sum_{J\subseteq F}\underbrace{u^{\left\vert \Psi\left(  J\right)
\right\vert -\left\vert C\right\vert }}_{\substack{=u^{\left\vert J\right\vert
}\\\text{(by (\ref{pf.lem.set-sandwich-sum.2b}))}}%
}\ \ \underbrace{v^{\left\vert A\right\vert -\left\vert \Psi\left(  J\right)
\right\vert }}_{\substack{=v^{\left\vert F\right\vert -\left\vert J\right\vert
}\\\text{(by (\ref{pf.lem.set-sandwich-sum.2c}))}}}\\
&  =\sum_{J\subseteq F}u^{\left\vert J\right\vert }v^{\left\vert F\right\vert
-\left\vert J\right\vert }=\sum_{k=0}^{\left\vert F\right\vert }%
\ \ \sum_{\substack{J\subseteq F;\\\left\vert J\right\vert =k}%
}\ \ \underbrace{u^{\left\vert J\right\vert }v^{\left\vert F\right\vert
-\left\vert J\right\vert }}_{\substack{=u^{k}v^{\left\vert F\right\vert
-k}\\\text{(since }\left\vert J\right\vert =k\text{)}}}\\
&  \ \ \ \ \ \ \ \ \ \ \ \ \ \ \ \ \ \ \ \ \left(
\begin{array}
[c]{c}%
\text{here, we have split up the sum according to}\\
\text{the value of }\left\vert J\right\vert \text{, since each subset }J\text{
of }F\\
\text{has size }\left\vert J\right\vert \in\left\{  0,1,\ldots,\left\vert
F\right\vert \right\}
\end{array}
\right) \\
&  =\sum_{k=0}^{\left\vert F\right\vert }\ \ \underbrace{\sum
_{\substack{J\subseteq F;\\\left\vert J\right\vert =k}}u^{k}v^{\left\vert
F\right\vert -k}}_{=\left(  \text{the number of all subsets }J\text{ of
}F\text{ satisfying }\left\vert J\right\vert =k\right)  \cdot u^{k}%
v^{\left\vert F\right\vert -k}}\\
&  =\sum_{k=0}^{\left\vert F\right\vert }\underbrace{\left(  \text{the number
of all subsets }J\text{ of }F\text{ satisfying }\left\vert J\right\vert
=k\right)  }_{\substack{=\left(  \text{the number of all }k\text{-element
subsets of }F\right)  \\=\dbinom{\left\vert F\right\vert }{k}\\\text{(by the
combinatorial interpretation of the binomial coefficients)}}}\cdot
\,u^{k}v^{\left\vert F\right\vert -k}\\
&  =\sum_{k=0}^{\left\vert F\right\vert }\dbinom{\left\vert F\right\vert }%
{k}u^{k}v^{\left\vert F\right\vert -k}=\left(  u+v\right)  ^{\left\vert
F\right\vert }%
\end{align*}
(since the binomial formula yields $\left(  u+v\right)  ^{\left\vert
F\right\vert }=\sum_{k=0}^{\left\vert F\right\vert }\dbinom{\left\vert
F\right\vert }{k}u^{k}v^{\left\vert F\right\vert -k}$). In view of $\left\vert
F\right\vert =\left\vert A\right\vert -\left\vert C\right\vert $, we can
rewrite this as%
\[
\sum_{\substack{I\subseteq A;\\C\subseteq I}}u^{\left\vert I\right\vert
-\left\vert C\right\vert }v^{\left\vert A\right\vert -\left\vert I\right\vert
}=\left(  u+v\right)  ^{\left\vert A\right\vert -\left\vert C\right\vert }.
\]
This proves Lemma \ref{lem.set-sandwich-sum}.
\end{proof}

Our next lemma will simplify some sums for us:
\end{verlong}

\begin{lemma}
\label{lem.Comp-sandwich-sum}Let $n\in\mathbb{N}$. Let $\alpha\in
\operatorname*{Comp}\nolimits_{n}$ and $\gamma\in\operatorname*{Comp}%
\nolimits_{n}$ be such that $D\left(  \gamma\right)  \subseteq D\left(
\alpha\right)  $. Then:

\begin{enumerate}
\item[\textbf{(a)}] For any $u,v\in\mathbf{k}$, we have%
\[
\sum_{\substack{\beta\in\operatorname*{Comp}\nolimits_{n};\\D\left(
\beta\right)  \subseteq D\left(  \alpha\right)  \text{ and }D\left(
\gamma\right)  \subseteq D\left(  \beta\right)  }}u^{\ell\left(  \beta\right)
-\ell\left(  \gamma\right)  }v^{\ell\left(  \alpha\right)  -\ell\left(
\beta\right)  }=\left(  u+v\right)  ^{\ell\left(  \alpha\right)  -\ell\left(
\gamma\right)  }.
\]

\item[\textbf{(b)}] For any $u\in\mathbf{k}$, we have%
\[
\sum_{\substack{\beta\in\operatorname*{Comp}\nolimits_{n};\\D\left(
\beta\right)  \subseteq D\left(  \alpha\right)  \text{ and }D\left(
\gamma\right)  \subseteq D\left(  \beta\right)  }}u^{\ell\left(  \beta\right)
}=\left(  u+1\right)  ^{\ell\left(  \alpha\right)  -\ell\left(  \gamma\right)
}u^{\ell\left(  \gamma\right)  }.
\]

\item[\textbf{(c)}] For any $v\in\mathbf{k}$, we have%
\[
\sum_{\substack{\beta\in\operatorname*{Comp}\nolimits_{n};\\D\left(
\beta\right)  \subseteq D\left(  \alpha\right)  \text{ and }D\left(
\gamma\right)  \subseteq D\left(  \beta\right)  }}v^{\ell\left(
\alpha\right)  -\ell\left(  \beta\right)  }=\left(  1+v\right)  ^{\ell\left(
\alpha\right)  -\ell\left(  \gamma\right)  }.
\]

\end{enumerate}
\end{lemma}

\begin{vershort}

\begin{proof}
[Proof of Lemma \ref{lem.Comp-sandwich-sum}.]\textbf{(a)} Let $u,v\in
\mathbf{k}$. Set $A:=D\left(  \alpha\right)  $ and $C:=D\left(  \gamma\right)
$. Then, $A=D\left(  \alpha\right)  \in\mathcal{P}\left(  \left[  n-1\right]
\right)  $ (since $\alpha\in\operatorname*{Comp}\nolimits_{n}$, but
$D:\operatorname*{Comp}\nolimits_{n}\rightarrow\mathcal{P}\left(  \left[
n-1\right]  \right)  $ is a map). In other words, $A\subseteq\left[
n-1\right]  $. Furthermore, $C=D\left(  \gamma\right)  \subseteq D\left(
\alpha\right)  =A\subseteq\left[  n-1\right]  $.

From $C\subseteq A$, we obtain%
\begin{align*}
\left\vert A\setminus C\right\vert  &  =\left\vert A\right\vert -\left\vert
C\right\vert =\left\vert D\left(  \alpha\right)  \right\vert -\left\vert
D\left(  \gamma\right)  \right\vert \ \ \ \ \ \ \ \ \ \ \left(  \text{since
}A=D\left(  \alpha\right)  \text{ and }C=D\left(  \gamma\right)  \right) \\
&  =\ell\left(  \alpha\right)  -\ell\left(  \gamma\right)
\end{align*}
(since Lemma \ref{lem.comps.l-vs-size} \textbf{(b) }yields $\ell\left(
\alpha\right)  -\ell\left(  \gamma\right)  =\left\vert D\left(  \alpha\right)
\right\vert -\left\vert D\left(  \gamma\right)  \right\vert $).

Now,%
\begin{align}
&  \sum_{\substack{\beta\in\operatorname*{Comp}\nolimits_{n};\\D\left(
\beta\right)  \subseteq D\left(  \alpha\right)  \text{ and }D\left(
\gamma\right)  \subseteq D\left(  \beta\right)  }}\underbrace{u^{\ell\left(
\beta\right)  -\ell\left(  \gamma\right)  }}_{\substack{=u^{\left\vert
D\left(  \beta\right)  \right\vert -\left\vert D\left(  \gamma\right)
\right\vert }\\\text{(since Lemma \ref{lem.comps.l-vs-size} \textbf{(b)}%
}\\\text{yields }\ell\left(  \beta\right)  -\ell\left(  \gamma\right)
=\left\vert D\left(  \beta\right)  \right\vert -\left\vert D\left(
\gamma\right)  \right\vert \text{)}}}\ \ \underbrace{v^{\ell\left(
\alpha\right)  -\ell\left(  \beta\right)  }}_{\substack{=v^{\left\vert
D\left(  \alpha\right)  \right\vert -\left\vert D\left(  \beta\right)
\right\vert }\\\text{(since Lemma \ref{lem.comps.l-vs-size} \textbf{(b)}%
}\\\text{yields }\ell\left(  \alpha\right)  -\ell\left(  \beta\right)
=\left\vert D\left(  \alpha\right)  \right\vert -\left\vert D\left(
\beta\right)  \right\vert \text{)}}}\nonumber\\
&  =\sum_{\substack{\beta\in\operatorname*{Comp}\nolimits_{n};\\D\left(
\beta\right)  \subseteq D\left(  \alpha\right)  \text{ and }D\left(
\gamma\right)  \subseteq D\left(  \beta\right)  }}u^{\left\vert D\left(
\beta\right)  \right\vert -\left\vert D\left(  \gamma\right)  \right\vert
}v^{\left\vert D\left(  \alpha\right)  \right\vert -\left\vert D\left(
\beta\right)  \right\vert }\nonumber\\
&  =\sum_{\substack{\beta\in\operatorname*{Comp}\nolimits_{n};\\D\left(
\beta\right)  \subseteq A\text{ and }C\subseteq D\left(  \beta\right)
}}u^{\left\vert D\left(  \beta\right)  \right\vert -\left\vert C\right\vert
}v^{\left\vert A\right\vert -\left\vert D\left(  \beta\right)  \right\vert
}\ \ \ \ \ \ \ \ \ \ \left(  \text{since }D\left(  \gamma\right)  =C\text{ and
}D\left(  \alpha\right)  =A\right) \nonumber\\
&  =\sum_{\substack{I\subseteq\left[  n-1\right]  ;\\I\subseteq A\text{ and
}C\subseteq I}}u^{\left\vert I\right\vert -\left\vert C\right\vert
}v^{\left\vert A\right\vert -\left\vert I\right\vert }%
\ \ \ \ \ \ \ \ \ \ \left(
\begin{array}
[c]{c}%
\text{here, we have substituted }I\text{ for }D\left(  \beta\right) \\
\text{in the sum, since the}\\
\text{map }D:\operatorname*{Comp}\nolimits_{n}\rightarrow\mathcal{P}\left(
\left[  n-1\right]  \right) \\
\text{is a bijection}%
\end{array}
\right) \nonumber\\
&  =\sum_{\substack{I\subseteq A;\\C\subseteq I}}u^{\left\vert I\right\vert
-\left\vert C\right\vert }v^{\left\vert A\right\vert -\left\vert I\right\vert
} \label{pf.lem.Comp-sandwich-sum.1}%
\end{align}
(since the condition \textquotedblleft$I\subseteq\left[  n-1\right]
$\textquotedblright\ under the sum is redundant (because the condition
\textquotedblleft$I\subseteq A$\textquotedblright\ already yields $I\subseteq
A\subseteq\left[  n-1\right]  $)).

Now, each subset $I$ of $A$ that satisfies $C\subseteq I$ can be written as
$C\cup Z$ for some unique subset $Z\subseteq A\setminus C$. Hence, we can
substitute $C\cup Z$ for $I$ in the sum $\sum_{\substack{I\subseteq
A;\\C\subseteq I}}u^{\left\vert I\right\vert -\left\vert C\right\vert
}v^{\left\vert A\right\vert -\left\vert I\right\vert }$. As a consequence, we
obtain%
\begin{align*}
\sum_{\substack{I\subseteq A;\\C\subseteq I}}u^{\left\vert I\right\vert
-\left\vert C\right\vert }v^{\left\vert A\right\vert -\left\vert I\right\vert
}  &  =\sum_{Z\subseteq A\setminus C}\underbrace{u^{\left\vert C\cup
Z\right\vert -\left\vert C\right\vert }v^{\left\vert A\right\vert -\left\vert
C\cup Z\right\vert }}_{\substack{=u^{\left(  \left\vert C\right\vert
+\left\vert Z\right\vert \right)  -\left\vert C\right\vert }v^{\left\vert
A\right\vert -\left(  \left\vert C\right\vert +\left\vert Z\right\vert
\right)  }\\\text{(since }\left\vert C\cup Z\right\vert =\left\vert
C\right\vert +\left\vert Z\right\vert \\\text{(because }Z\subseteq A\setminus
C\text{ shows that}\\\text{the sets }C\text{ and }Z\text{ are disjoint))}}}\\
&  =\sum_{Z\subseteq A\setminus C}\underbrace{u^{\left(  \left\vert
C\right\vert +\left\vert Z\right\vert \right)  -\left\vert C\right\vert }%
}_{=u^{\left\vert Z\right\vert }}\underbrace{v^{\left\vert A\right\vert
-\left(  \left\vert C\right\vert +\left\vert Z\right\vert \right)  }%
}_{\substack{=v^{\left\vert A\right\vert -\left\vert C\right\vert -\left\vert
Z\right\vert }=v^{\left\vert A\setminus C\right\vert -\left\vert Z\right\vert
}\\\text{(since }\left\vert A\right\vert -\left\vert C\right\vert =\left\vert
A\setminus C\right\vert \text{)}}}\\
&  =\sum_{Z\subseteq A\setminus C}u^{\left\vert Z\right\vert }v^{\left\vert
A\setminus C\right\vert -\left\vert Z\right\vert }=\sum_{k=0}^{\left\vert
A\setminus C\right\vert }\dbinom{\left\vert A\setminus C\right\vert }{k}%
u^{k}v^{\left\vert A\setminus C\right\vert -k}%
\end{align*}
(here, we have split the sum according to the size $\left\vert Z\right\vert $
of the subset $Z$, and observed that for each given $k\in\mathbb{N}$, there
are exactly $\dbinom{\left\vert A\setminus C\right\vert }{k}$ many subsets
$Z\subseteq A\setminus C$ that have size $k$). Comparing this with%
\[
\left(  u+v\right)  ^{\left\vert A\setminus C\right\vert }=\sum_{k=0}%
^{\left\vert A\setminus C\right\vert }\dbinom{\left\vert A\setminus
C\right\vert }{k}u^{k}v^{\left\vert A\setminus C\right\vert -k}%
\ \ \ \ \ \ \ \ \ \ \left(  \text{by the binomial formula}\right)  ,
\]
we obtain%
\[
\sum_{\substack{I\subseteq A;\\C\subseteq I}}u^{\left\vert I\right\vert
-\left\vert C\right\vert }v^{\left\vert A\right\vert -\left\vert I\right\vert
}=\left(  u+v\right)  ^{\left\vert A\setminus C\right\vert }=\left(
u+v\right)  ^{\ell\left(  \alpha\right)  -\ell\left(  \gamma\right)  }%
\]
(since $\left\vert A\setminus C\right\vert =\ell\left(  \alpha\right)
-\ell\left(  \gamma\right)  $). Hence, (\ref{pf.lem.Comp-sandwich-sum.1})
becomes%
\[
\sum_{\substack{\beta\in\operatorname*{Comp}\nolimits_{n};\\D\left(
\beta\right)  \subseteq D\left(  \alpha\right)  \text{ and }D\left(
\gamma\right)  \subseteq D\left(  \beta\right)  }}u^{\ell\left(  \beta\right)
-\ell\left(  \gamma\right)  }v^{\ell\left(  \alpha\right)  -\ell\left(
\beta\right)  }=\sum_{\substack{I\subseteq A;\\C\subseteq I}}u^{\left\vert
I\right\vert -\left\vert C\right\vert }v^{\left\vert A\right\vert -\left\vert
I\right\vert }=\left(  u+v\right)  ^{\ell\left(  \alpha\right)  -\ell\left(
\gamma\right)  }.
\]
This proves Lemma \ref{lem.Comp-sandwich-sum} \textbf{(a)}.

\textbf{(b)} This follows by applying Lemma \ref{lem.Comp-sandwich-sum}
\textbf{(a)} to $v=1$ and then multiplying both sides by $u^{\ell\left(
\gamma\right)  }$.

\textbf{(c)} This follows by applying Lemma \ref{lem.Comp-sandwich-sum}
\textbf{(a)} to $u=1$.
\end{proof}
\end{vershort}

\begin{verlong}

\begin{proof}
[Proof of Lemma \ref{lem.Comp-sandwich-sum}.]\textbf{(a)} Let $u,v\in
\mathbf{k}$.

We have $\alpha\in\operatorname*{Comp}\nolimits_{n}$, so that $D\left(
\alpha\right)  \in\mathcal{P}\left(  \left[  n-1\right]  \right)  $ (since
$D:\operatorname*{Comp}\nolimits_{n}\rightarrow\mathcal{P}\left(  \left[
n-1\right]  \right)  $ is a map). In other words, $D\left(  \alpha\right)
\subseteq\left[  n-1\right]  $.

Set $A:=D\left(  \alpha\right)  $ and $C:=D\left(  \gamma\right)  $. Then,
$A=D\left(  \alpha\right)  \subseteq\left[  n-1\right]  $ and $C=D\left(
\gamma\right)  \subseteq D\left(  \alpha\right)  =A\subseteq\left[
n-1\right]  $.

Lemma \ref{lem.comps.l-vs-size} \textbf{(b)} (applied to $\alpha$ and $\gamma$
instead of $\beta$ and $\alpha$) yields
\begin{equation}
\ell\left(  \alpha\right)  -\ell\left(  \gamma\right)  =\left\vert D\left(
\alpha\right)  \right\vert -\left\vert D\left(  \gamma\right)  \right\vert .
\label{pf.lem.Comp-sandwich-sum.a.0}%
\end{equation}
From $A=D\left(  \alpha\right)  $ and $C=D\left(  \gamma\right)  $, we obtain%
\[
\left\vert A\right\vert -\left\vert C\right\vert =\left\vert D\left(
\alpha\right)  \right\vert -\left\vert D\left(  \gamma\right)  \right\vert
=\ell\left(  \alpha\right)  -\ell\left(  \gamma\right)
\ \ \ \ \ \ \ \ \ \ \left(  \text{by (\ref{pf.lem.Comp-sandwich-sum.a.0}%
)}\right)  .
\]

Now, for each $\beta\in\operatorname*{Comp}\nolimits_{n}$, we have%
\begin{align}
\ell\left(  \beta\right)  -\ell\left(  \gamma\right)   &  =\left\vert D\left(
\beta\right)  \right\vert -\left\vert \underbrace{D\left(  \gamma\right)
}_{=C}\right\vert \ \ \ \ \ \ \ \ \ \ \left(
\begin{array}
[c]{c}%
\text{by Lemma \ref{lem.comps.l-vs-size} \textbf{(b)},}\\
\text{applied to }\gamma\text{ instead of }\alpha
\end{array}
\right) \nonumber\\
&  =\left\vert D\left(  \beta\right)  \right\vert -\left\vert C\right\vert
\label{pf.lem.Comp-sandwich-sum.a.1}%
\end{align}
and%
\begin{align}
\ell\left(  \alpha\right)  -\ell\left(  \beta\right)   &  =\left\vert
\underbrace{D\left(  \alpha\right)  }_{=A}\right\vert -\left\vert D\left(
\beta\right)  \right\vert \ \ \ \ \ \ \ \ \ \ \left(
\begin{array}
[c]{c}%
\text{by Lemma \ref{lem.comps.l-vs-size} \textbf{(b)},}\\
\text{applied to }\beta\text{ and }\alpha\text{ instead of }\alpha\text{ and
}\beta
\end{array}
\right) \nonumber\\
&  =\left\vert A\right\vert -\left\vert D\left(  \beta\right)  \right\vert .
\label{pf.lem.Comp-sandwich-sum.a.2}%
\end{align}
Hence,%
\begin{align}
&  \underbrace{\sum_{\substack{\beta\in\operatorname*{Comp}\nolimits_{n}%
;\\D\left(  \beta\right)  \subseteq D\left(  \alpha\right)  \text{ and
}D\left(  \gamma\right)  \subseteq D\left(  \beta\right)  }}}_{\substack{=\sum
_{\substack{\beta\in\operatorname*{Comp}\nolimits_{n};\\D\left(  \beta\right)
\subseteq A\text{ and }C\subseteq D\left(  \beta\right)  }}\\\text{(since
}D\left(  \gamma\right)  =C\text{ and }D\left(  \alpha\right)  =A\text{)}%
}}\underbrace{u^{\ell\left(  \beta\right)  -\ell\left(  \gamma\right)  }%
}_{\substack{=u^{\left\vert D\left(  \beta\right)  \right\vert -\left\vert
C\right\vert }\\\text{(by (\ref{pf.lem.Comp-sandwich-sum.a.1}))}%
}}\ \ \underbrace{v^{\ell\left(  \alpha\right)  -\ell\left(  \beta\right)  }%
}_{\substack{=v^{\left\vert A\right\vert -\left\vert D\left(  \beta\right)
\right\vert }\\\text{(by (\ref{pf.lem.Comp-sandwich-sum.a.2}))}}}\nonumber\\
&  =\sum_{\substack{\beta\in\operatorname*{Comp}\nolimits_{n};\\D\left(
\beta\right)  \subseteq A\text{ and }C\subseteq D\left(  \beta\right)
}}u^{\left\vert D\left(  \beta\right)  \right\vert -\left\vert C\right\vert
}v^{\left\vert A\right\vert -\left\vert D\left(  \beta\right)  \right\vert
}\nonumber\\
&  =\sum_{\substack{I\in\mathcal{P}\left(  \left[  n-1\right]  \right)
;\\I\subseteq A\text{ and }C\subseteq I}}u^{\left\vert I\right\vert
-\left\vert C\right\vert }v^{\left\vert A\right\vert -\left\vert I\right\vert
} \label{pf.lem.Comp-sandwich-sum.a.4}%
\end{align}
(here, we have substituted $I$ for $D\left(  \beta\right)  $ in the sum, since
the map $D:\operatorname*{Comp}\nolimits_{n}\rightarrow\mathcal{P}\left(
\left[  n-1\right]  \right)  $ is a bijection).

However, recall that $A\subseteq\left[  n-1\right]  $. Hence, each subset $I$
of $A$ automatically satisfies $I\in\mathcal{P}\left(  \left[  n-1\right]
\right)  $ (since $I\subseteq A\subseteq\left[  n-1\right]  $) and $I\subseteq
A$. In other words, each subset of $A$ is an $I\in\mathcal{P}\left(  \left[
n-1\right]  \right)  $ that satisfies $I\subseteq A$. Conversely, of course,
every $I\in\mathcal{P}\left(  \left[  n-1\right]  \right)  $ that satisfies
$I\subseteq A$ must be a subset of $A$ (since $I\subseteq A$). Thus, the sets
$I\in\mathcal{P}\left(  \left[  n-1\right]  \right)  $ that satisfy
$I\subseteq A$ are precisely the subsets of $A$. Therefore, we have the
following equality of summation signs:%
\[
\sum_{\substack{I\in\mathcal{P}\left(  \left[  n-1\right]  \right)
;\\I\subseteq A}}=\sum_{I\subseteq A}.
\]
Of course, this equality remains true if we add the extra condition
\textquotedblleft$C\subseteq I$\textquotedblright\ under both summation signs.
Thus, we obtain the following equality of summation signs:%
\[
\sum_{\substack{I\in\mathcal{P}\left(  \left[  n-1\right]  \right)
;\\I\subseteq A\text{ and }C\subseteq I}}=\sum_{\substack{I\subseteq
A;\\C\subseteq I}}.
\]
Hence, we can rewrite (\ref{pf.lem.Comp-sandwich-sum.a.4}) as follows:%
\begin{align*}
\sum_{\substack{\beta\in\operatorname*{Comp}\nolimits_{n};\\D\left(
\beta\right)  \subseteq D\left(  \alpha\right)  \text{ and }D\left(
\gamma\right)  \subseteq D\left(  \beta\right)  }}u^{\ell\left(  \beta\right)
-\ell\left(  \gamma\right)  }v^{\ell\left(  \alpha\right)  -\ell\left(
\beta\right)  }  &  =\sum_{\substack{I\subseteq A;\\C\subseteq I}%
}u^{\left\vert I\right\vert -\left\vert C\right\vert }v^{\left\vert
A\right\vert -\left\vert I\right\vert }\\
&  =\left(  u+v\right)  ^{\left\vert A\right\vert -\left\vert C\right\vert
}\ \ \ \ \ \ \ \ \ \ \left(  \text{by Lemma \ref{lem.set-sandwich-sum}}\right)
\\
&  =\left(  u+v\right)  ^{\ell\left(  \alpha\right)  -\ell\left(
\gamma\right)  }%
\end{align*}
(since $\left\vert A\right\vert -\left\vert C\right\vert =\ell\left(
\alpha\right)  -\ell\left(  \gamma\right)  $). This proves Lemma
\ref{lem.Comp-sandwich-sum} \textbf{(a)}.

\textbf{(b)} Let $u\in\mathbf{k}$. Lemma \ref{lem.Comp-sandwich-sum}
\textbf{(a)} (applied to $v=1$) yields
\[
\sum_{\substack{\beta\in\operatorname*{Comp}\nolimits_{n};\\D\left(
\beta\right)  \subseteq D\left(  \alpha\right)  \text{ and }D\left(
\gamma\right)  \subseteq D\left(  \beta\right)  }}u^{\ell\left(  \beta\right)
-\ell\left(  \gamma\right)  }1^{\ell\left(  \alpha\right)  -\ell\left(
\beta\right)  }=\left(  u+1\right)  ^{\ell\left(  \alpha\right)  -\ell\left(
\gamma\right)  }.
\]
Hence,%
\begin{align*}
\left(  u+1\right)  ^{\ell\left(  \alpha\right)  -\ell\left(  \gamma\right)
}  &  =\sum_{\substack{\beta\in\operatorname*{Comp}\nolimits_{n};\\D\left(
\beta\right)  \subseteq D\left(  \alpha\right)  \text{ and }D\left(
\gamma\right)  \subseteq D\left(  \beta\right)  }}u^{\ell\left(  \beta\right)
-\ell\left(  \gamma\right)  }\underbrace{1^{\ell\left(  \alpha\right)
-\ell\left(  \beta\right)  }}_{=1}\\
&  =\sum_{\substack{\beta\in\operatorname*{Comp}\nolimits_{n};\\D\left(
\beta\right)  \subseteq D\left(  \alpha\right)  \text{ and }D\left(
\gamma\right)  \subseteq D\left(  \beta\right)  }}u^{\ell\left(  \beta\right)
-\ell\left(  \gamma\right)  }.
\end{align*}
Multiplying both sides of this equality by $u^{\ell\left(  \gamma\right)  }$,
we obtain%
\begin{align*}
\left(  u+1\right)  ^{\ell\left(  \alpha\right)  -\ell\left(  \gamma\right)
}u^{\ell\left(  \gamma\right)  }  &  =\left(  \sum_{\substack{\beta
\in\operatorname*{Comp}\nolimits_{n};\\D\left(  \beta\right)  \subseteq
D\left(  \alpha\right)  \text{ and }D\left(  \gamma\right)  \subseteq D\left(
\beta\right)  }}u^{\ell\left(  \beta\right)  -\ell\left(  \gamma\right)
}\right)  u^{\ell\left(  \gamma\right)  }\\
&  =\sum_{\substack{\beta\in\operatorname*{Comp}\nolimits_{n};\\D\left(
\beta\right)  \subseteq D\left(  \alpha\right)  \text{ and }D\left(
\gamma\right)  \subseteq D\left(  \beta\right)  }}\underbrace{u^{\ell\left(
\beta\right)  -\ell\left(  \gamma\right)  }u^{\ell\left(  \gamma\right)  }%
}_{=u^{\left(  \ell\left(  \beta\right)  -\ell\left(  \gamma\right)  \right)
}u^{\ell\left(  \gamma\right)  }=u^{\ell\left(  \beta\right)  }}\\
&  =\sum_{\substack{\beta\in\operatorname*{Comp}\nolimits_{n};\\D\left(
\beta\right)  \subseteq D\left(  \alpha\right)  \text{ and }D\left(
\gamma\right)  \subseteq D\left(  \beta\right)  }}u^{\ell\left(  \beta\right)
}.
\end{align*}
This proves Lemma \ref{lem.Comp-sandwich-sum} \textbf{(b)}.

\textbf{(c)} Let $v\in\mathbf{k}$. Lemma \ref{lem.Comp-sandwich-sum}
\textbf{(a)} (applied to $u=1$) yields%
\[
\sum_{\substack{\beta\in\operatorname*{Comp}\nolimits_{n};\\D\left(
\beta\right)  \subseteq D\left(  \alpha\right)  \text{ and }D\left(
\gamma\right)  \subseteq D\left(  \beta\right)  }}1^{\ell\left(  \beta\right)
-\ell\left(  \gamma\right)  }v^{\ell\left(  \alpha\right)  -\ell\left(
\beta\right)  }=\left(  1+v\right)  ^{\ell\left(  \alpha\right)  -\ell\left(
\gamma\right)  }.
\]
Thus,%
\begin{align*}
\left(  1+v\right)  ^{\ell\left(  \alpha\right)  -\ell\left(  \gamma\right)
}  &  =\sum_{\substack{\beta\in\operatorname*{Comp}\nolimits_{n};\\D\left(
\beta\right)  \subseteq D\left(  \alpha\right)  \text{ and }D\left(
\gamma\right)  \subseteq D\left(  \beta\right)  }}\underbrace{1^{\ell\left(
\beta\right)  -\ell\left(  \gamma\right)  }}_{=1}v^{\ell\left(  \alpha\right)
-\ell\left(  \beta\right)  }\\
&  =\sum_{\substack{\beta\in\operatorname*{Comp}\nolimits_{n};\\D\left(
\beta\right)  \subseteq D\left(  \alpha\right)  \text{ and }D\left(
\gamma\right)  \subseteq D\left(  \beta\right)  }}v^{\ell\left(
\alpha\right)  -\ell\left(  \beta\right)  }.
\end{align*}
This proves Lemma \ref{lem.Comp-sandwich-sum} \textbf{(c)}.
\end{proof}
\end{verlong}

\begin{verlong}
We can now prove Theorem \ref{thm.eta.S2} in a slightly modified form (we will
subsequently derive the actual Theorem \ref{thm.eta.S2} from it):

\begin{lemma}
\label{lem.eta.S2rev}Let $n\in\mathbb{N}$. Let $\alpha\in\operatorname*{Comp}%
\nolimits_{n}$. Then, the antipode $S$ of $\operatorname*{QSym}$ satisfies
\[
S\left(  \eta_{\operatorname*{rev}\alpha}^{\left(  q\right)  }\right)
=\left(  -1\right)  ^{\ell\left(  \alpha\right)  }\sum_{\substack{\beta
\in\operatorname*{Comp}\nolimits_{n};\\D\left(  \beta\right)  \subseteq
D\left(  \alpha\right)  }}\left(  q-1\right)  ^{\ell\left(  \alpha\right)
-\ell\left(  \beta\right)  }\eta_{\beta}^{\left(  q\right)  }.
\]

\end{lemma}

\begin{proof}
[Proof of Lemma \ref{lem.eta.S2rev}.]First, we notice that
\[
\left(  -\underbrace{r}_{=q+1}\right)  +1=\left(  -\left(  q+1\right)
\right)  +1=-q.
\]

We next observe that every $\beta\in\operatorname*{Comp}\nolimits_{n}$
satisfies%
\begin{equation}
S\left(  M_{\operatorname*{rev}\beta}\right)  =\left(  -1\right)
^{\ell\left(  \beta\right)  }\sum_{\substack{\gamma\in\operatorname*{Comp}%
\nolimits_{n};\\D\left(  \gamma\right)  \subseteq D\left(  \beta\right)
}}M_{\gamma} \label{pf.lem.eta.S2rev.SMrev=}%
\end{equation}
\footnote{\textit{Proof of (\ref{pf.lem.eta.S2rev.SMrev=}):} Let $\beta
\in\operatorname*{Comp}\nolimits_{n}$. Write this composition $\beta$ as
$\beta=\left(  \beta_{1},\beta_{2},\ldots,\beta_{k}\right)  $. Then, the
definition of $\operatorname*{rev}\beta$ yields $\operatorname*{rev}%
\beta=\left(  \beta_{k},\beta_{k-1},\ldots,\beta_{1}\right)  $. Hence,
\begin{align}
\left\vert \operatorname*{rev}\beta\right\vert  &  =\left\vert \left(
\beta_{k},\beta_{k-1},\ldots,\beta_{1}\right)  \right\vert =\beta_{k}%
+\beta_{k-1}+\cdots+\beta_{1}\nonumber\\
&  =\beta_{1}+\beta_{2}+\cdots+\beta_{k}. \label{pf.lem.eta.S2rev.SMrev=.pf.1}%
\end{align}
But $\beta\in\operatorname*{Comp}\nolimits_{n}$ entails $\left\vert
\beta\right\vert =n$, so that $n=\left\vert \beta\right\vert =\beta_{1}%
+\beta_{2}+\cdots+\beta_{k}$ (since $\beta=\left(  \beta_{1},\beta_{2}%
,\ldots,\beta_{k}\right)  $). Comparing this with
(\ref{pf.lem.eta.S2rev.SMrev=.pf.1}), we obtain $\left\vert
\operatorname*{rev}\beta\right\vert =n$. Thus, $\operatorname*{rev}\beta$ is a
composition of $n$. In other words, $\operatorname*{rev}\beta\in
\operatorname*{Comp}\nolimits_{n}$.
\par
So we know that $\operatorname*{rev}\beta\in\operatorname*{Comp}\nolimits_{n}$
and $\operatorname*{rev}\beta=\left(  \beta_{k},\beta_{k-1},\ldots,\beta
_{1}\right)  $. Hence, (\ref{eq.SMalpha}) (applied to $k$,
$\operatorname*{rev}\beta$ and $\beta_{k+1-i}$ instead of $\ell$, $\alpha$ and
$\alpha_{i}$) yields%
\begin{equation}
S\left(  M_{\operatorname*{rev}\beta}\right)  =\left(  -1\right)  ^{k}%
\sum_{\substack{\gamma\in\operatorname*{Comp}\nolimits_{n};\\D\left(
\gamma\right)  \subseteq D\left(  \beta_{1},\beta_{2},\ldots,\beta_{k}\right)
}}M_{\gamma}=\left(  -1\right)  ^{k}\sum_{\substack{\gamma\in
\operatorname*{Comp}\nolimits_{n};\\D\left(  \gamma\right)  \subseteq D\left(
\beta\right)  }}M_{\gamma} \label{pf.lem.eta.S2rev.SMrev=.pf.2}%
\end{equation}
(since $\left(  \beta_{1},\beta_{2},\ldots,\beta_{k}\right)  =\beta$).
Moreover, the definition of $\ell\left(  \beta\right)  $ yields $\ell\left(
\beta\right)  =k$ (since $\beta=\left(  \beta_{1},\beta_{2},\ldots,\beta
_{k}\right)  $ is visibly a $k$-tuple). Hence, $k=\ell\left(  \beta\right)  $.
Thus, we can rewrite (\ref{pf.lem.eta.S2rev.SMrev=.pf.2}) as
\[
S\left(  M_{\operatorname*{rev}\beta}\right)  =\left(  -1\right)
^{\ell\left(  \beta\right)  }\sum_{\substack{\gamma\in\operatorname*{Comp}%
\nolimits_{n};\\D\left(  \gamma\right)  \subseteq D\left(  \beta\right)
}}M_{\gamma}.
\]
This proves (\ref{pf.lem.eta.S2rev.SMrev=}).} and%
\begin{equation}
\eta_{\beta}^{\left(  q\right)  }=\sum_{\substack{\gamma\in
\operatorname*{Comp}\nolimits_{n};\\D\left(  \gamma\right)  \subseteq D\left(
\beta\right)  }}r^{\ell\left(  \gamma\right)  }M_{\gamma}
\label{pf.lem.eta.S2rev.eta=}%
\end{equation}
(indeed, this is just the equality (\ref{eq.def.etaalpha.def}), with the
letters $\alpha$ and $\beta$ renamed as $\beta$ and $\gamma$).

We next observe that every composition $\beta\in\operatorname*{Comp}$
satisfies
\begin{equation}
\ell\left(  \operatorname*{rev}\beta\right)  =\ell\left(  \beta\right)
\label{pf.lem.eta.S2rev.lenrev}%
\end{equation}
(this is a trivial consequence of the definition of $\operatorname*{rev}\beta$).

Next, we recall two simple facts from \cite{comps}. First, \cite[Corollary
3.10]{comps} says that the map%
\begin{align*}
\operatorname*{Comp}\nolimits_{n}  &  \rightarrow\operatorname*{Comp}%
\nolimits_{n},\\
\delta &  \mapsto\operatorname*{rev}\delta
\end{align*}
is a bijection. Renaming the letter $\delta$ as $\beta$ in this statement, we
conclude that the map%
\begin{align*}
\operatorname*{Comp}\nolimits_{n}  &  \rightarrow\operatorname*{Comp}%
\nolimits_{n},\\
\beta &  \mapsto\operatorname*{rev}\beta
\end{align*}
is a bijection. Furthermore, \cite[Proposition 3.11]{comps} shows that if
$\beta\in\operatorname*{Comp}\nolimits_{n}$ is arbitrary, then we have the
logical equivalence%
\[
\left(  D\left(  \operatorname*{rev}\beta\right)  \subseteq D\left(
\operatorname*{rev}\alpha\right)  \right)  \ \Longleftrightarrow\ \left(
D\left(  \beta\right)  \subseteq D\left(  \alpha\right)  \right)  .
\]
Hence, we have the following equality of summation signs:%
\begin{equation}
\sum_{\substack{\beta\in\operatorname*{Comp}\nolimits_{n};\\D\left(
\operatorname*{rev}\beta\right)  \subseteq D\left(  \operatorname*{rev}%
\alpha\right)  }}=\sum_{\substack{\beta\in\operatorname*{Comp}\nolimits_{n}%
;\\D\left(  \beta\right)  \subseteq D\left(  \alpha\right)  }}.
\label{pf.lem.eta.S2rev.equivsum}%
\end{equation}

Now, the definition of $\eta_{\operatorname*{rev}\alpha}^{\left(  q\right)  }$
yields%
\begin{align*}
\eta_{\operatorname*{rev}\alpha}^{\left(  q\right)  }  &  =\sum
_{\substack{\beta\in\operatorname*{Comp}\nolimits_{n};\\D\left(  \beta\right)
\subseteq D\left(  \operatorname*{rev}\alpha\right)  }}r^{\ell\left(
\beta\right)  }M_{\beta}=\underbrace{\sum_{\substack{\beta\in
\operatorname*{Comp}\nolimits_{n};\\D\left(  \operatorname*{rev}\beta\right)
\subseteq D\left(  \operatorname*{rev}\alpha\right)  }}}_{\substack{=\sum
_{\substack{\beta\in\operatorname*{Comp}\nolimits_{n};\\D\left(  \beta\right)
\subseteq D\left(  \alpha\right)  }}\\\text{(by
(\ref{pf.lem.eta.S2rev.equivsum}))}}}\underbrace{r^{\ell\left(
\operatorname*{rev}\beta\right)  }}_{\substack{=r^{\ell\left(  \beta\right)
}\\\text{(by (\ref{pf.lem.eta.S2rev.lenrev}))}}}M_{\operatorname*{rev}\beta}\\
&  \ \ \ \ \ \ \ \ \ \ \ \ \ \ \ \ \ \ \ \ \left(
\begin{array}
[c]{c}%
\text{here, we have substituted }\operatorname*{rev}\beta\text{ for }%
\beta\text{ in the sum,}\\
\text{since the map }\operatorname*{Comp}\nolimits_{n}\rightarrow
\operatorname*{Comp}\nolimits_{n},\ \beta\mapsto\operatorname*{rev}\beta\\
\text{is a bijection}%
\end{array}
\right) \\
&  =\sum_{\substack{\beta\in\operatorname*{Comp}\nolimits_{n};\\D\left(
\beta\right)  \subseteq D\left(  \alpha\right)  }}r^{\ell\left(  \beta\right)
}M_{\operatorname*{rev}\beta}.
\end{align*}

Applying the map $S$ to both sides of this equality, we obtain%
\begin{align}
S\left(  \eta_{\operatorname*{rev}\alpha}^{\left(  q\right)  }\right)   &
=S\left(  \sum_{\substack{\beta\in\operatorname*{Comp}\nolimits_{n};\\D\left(
\beta\right)  \subseteq D\left(  \alpha\right)  }}r^{\ell\left(  \beta\right)
}M_{\operatorname*{rev}\beta}\right) \nonumber\\
&  =\sum_{\substack{\beta\in\operatorname*{Comp}\nolimits_{n};\\D\left(
\beta\right)  \subseteq D\left(  \alpha\right)  }}r^{\ell\left(  \beta\right)
}\underbrace{S\left(  M_{\operatorname*{rev}\beta}\right)  }%
_{\substack{=\left(  -1\right)  ^{\ell\left(  \beta\right)  }\sum
_{\substack{\gamma\in\operatorname*{Comp}\nolimits_{n};\\D\left(
\gamma\right)  \subseteq D\left(  \beta\right)  }}M_{\gamma}\\\text{(by
(\ref{pf.lem.eta.S2rev.SMrev=}))}}}\ \ \ \ \ \ \ \ \ \ \left(
\begin{array}
[c]{c}%
\text{since the map }S\\
\text{is }\mathbf{k}\text{-linear}%
\end{array}
\right) \nonumber\\
&  =\sum_{\substack{\beta\in\operatorname*{Comp}\nolimits_{n};\\D\left(
\beta\right)  \subseteq D\left(  \alpha\right)  }}\underbrace{r^{\ell\left(
\beta\right)  }\left(  -1\right)  ^{\ell\left(  \beta\right)  }}_{=\left(
-r\right)  ^{\ell\left(  \beta\right)  }}\sum_{\substack{\gamma\in
\operatorname*{Comp}\nolimits_{n};\\D\left(  \gamma\right)  \subseteq D\left(
\beta\right)  }}M_{\gamma}\nonumber\\
&  =\sum_{\substack{\beta\in\operatorname*{Comp}\nolimits_{n};\\D\left(
\beta\right)  \subseteq D\left(  \alpha\right)  }}\left(  -r\right)
^{\ell\left(  \beta\right)  }\sum_{\substack{\gamma\in\operatorname*{Comp}%
\nolimits_{n};\\D\left(  \gamma\right)  \subseteq D\left(  \beta\right)
}}M_{\gamma}\nonumber\\
&  =\sum_{\substack{\beta\in\operatorname*{Comp}\nolimits_{n};\\D\left(
\beta\right)  \subseteq D\left(  \alpha\right)  }}\ \ \sum_{\substack{\gamma
\in\operatorname*{Comp}\nolimits_{n};\\D\left(  \gamma\right)  \subseteq
D\left(  \beta\right)  }}\left(  -r\right)  ^{\ell\left(  \beta\right)
}M_{\gamma}. \label{pf.lem.eta.S2rev.LHS1}%
\end{align}
However, it is easy to see that we have the following equality of summation
signs:%
\begin{equation}
\sum_{\substack{\beta\in\operatorname*{Comp}\nolimits_{n};\\D\left(
\beta\right)  \subseteq D\left(  \alpha\right)  }}\ \ \sum_{\substack{\gamma
\in\operatorname*{Comp}\nolimits_{n};\\D\left(  \gamma\right)  \subseteq
D\left(  \beta\right)  }}=\sum_{\substack{\gamma\in\operatorname*{Comp}%
\nolimits_{n};\\D\left(  \gamma\right)  \subseteq D\left(  \alpha\right)
}}\ \ \sum_{\substack{\beta\in\operatorname*{Comp}\nolimits_{n};\\D\left(
\beta\right)  \subseteq D\left(  \alpha\right)  \text{ and }D\left(
\gamma\right)  \subseteq D\left(  \beta\right)  }}
\label{pf.lem.eta.S2rev.sums}%
\end{equation}
\footnote{\textit{Proof of (\ref{pf.lem.eta.S2rev.sums}):} A standard
interchange of summation signs yields%
\begin{align*}
\sum_{\substack{\gamma\in\operatorname*{Comp}\nolimits_{n};\\D\left(
\gamma\right)  \subseteq D\left(  \alpha\right)  }}\ \ \sum_{\substack{\beta
\in\operatorname*{Comp}\nolimits_{n};\\D\left(  \beta\right)  \subseteq
D\left(  \alpha\right)  \text{ and }D\left(  \gamma\right)  \subseteq D\left(
\beta\right)  }}  &  =\sum_{\substack{\beta\in\operatorname*{Comp}%
\nolimits_{n};\\D\left(  \beta\right)  \subseteq D\left(  \alpha\right)
}}\ \ \underbrace{\sum_{\substack{\gamma\in\operatorname*{Comp}\nolimits_{n}%
;\\D\left(  \gamma\right)  \subseteq D\left(  \beta\right)  ;\\D\left(
\gamma\right)  \subseteq D\left(  \alpha\right)  }}}_{\substack{=\sum
_{\substack{\gamma\in\operatorname*{Comp}\nolimits_{n};\\D\left(
\gamma\right)  \subseteq D\left(  \beta\right)  }}\\\text{(here, we have
removed the condition \textquotedblleft}D\left(  \gamma\right)  \subseteq
D\left(  \alpha\right)  \text{\textquotedblright}\\\text{under the summation
sign, since this condition}\\\text{follows automatically from the condition
\textquotedblleft}D\left(  \gamma\right)  \subseteq D\left(  \beta\right)
\text{\textquotedblright}\\\text{(because the latter condition entails
}D\left(  \gamma\right)  \subseteq D\left(  \beta\right)  \subseteq D\left(
\alpha\right)  \text{))}}}\\
&  =\sum_{\substack{\beta\in\operatorname*{Comp}\nolimits_{n};\\D\left(
\beta\right)  \subseteq D\left(  \alpha\right)  }}\ \ \sum_{\substack{\gamma
\in\operatorname*{Comp}\nolimits_{n};\\D\left(  \gamma\right)  \subseteq
D\left(  \beta\right)  }}.
\end{align*}
This proves (\ref{pf.lem.eta.S2rev.sums}).}. Thus, we can rewrite
(\ref{pf.lem.eta.S2rev.LHS1}) as
\begin{align}
S\left(  \eta_{\operatorname*{rev}\alpha}^{\left(  q\right)  }\right)   &
=\sum_{\substack{\gamma\in\operatorname*{Comp}\nolimits_{n};\\D\left(
\gamma\right)  \subseteq D\left(  \alpha\right)  }}\ \ \underbrace{\sum
_{\substack{\beta\in\operatorname*{Comp}\nolimits_{n};\\D\left(  \beta\right)
\subseteq D\left(  \alpha\right)  \text{ and }D\left(  \gamma\right)
\subseteq D\left(  \beta\right)  }}\left(  -r\right)  ^{\ell\left(
\beta\right)  }}_{\substack{=\left(  \left(  -r\right)  +1\right)
^{\ell\left(  \alpha\right)  -\ell\left(  \gamma\right)  }\left(  -r\right)
^{\ell\left(  \gamma\right)  }\\\text{(by Lemma \ref{lem.Comp-sandwich-sum}
\textbf{(b)},}\\\text{applied to }u=-r\text{)}}}M_{\gamma}\nonumber\\
&  =\sum_{\substack{\gamma\in\operatorname*{Comp}\nolimits_{n};\\D\left(
\gamma\right)  \subseteq D\left(  \alpha\right)  }}\left(  \underbrace{\left(
-r\right)  +1}_{=-q}\right)  ^{\ell\left(  \alpha\right)  -\ell\left(
\gamma\right)  }\left(  -r\right)  ^{\ell\left(  \gamma\right)  }M_{\gamma
}\nonumber\\
&  =\sum_{\substack{\gamma\in\operatorname*{Comp}\nolimits_{n};\\D\left(
\gamma\right)  \subseteq D\left(  \alpha\right)  }}\underbrace{\left(
-q\right)  ^{\ell\left(  \alpha\right)  -\ell\left(  \gamma\right)  }%
}_{=\left(  -1\right)  ^{\ell\left(  \alpha\right)  -\ell\left(
\gamma\right)  }q^{\ell\left(  \alpha\right)  -\ell\left(  \gamma\right)  }%
}\underbrace{\left(  -r\right)  ^{\ell\left(  \gamma\right)  }}_{=\left(
-1\right)  ^{\ell\left(  \gamma\right)  }r^{\ell\left(  \gamma\right)  }%
}M_{\gamma}\nonumber\\
&  =\sum_{\substack{\gamma\in\operatorname*{Comp}\nolimits_{n};\\D\left(
\gamma\right)  \subseteq D\left(  \alpha\right)  }}\left(  -1\right)
^{\ell\left(  \alpha\right)  -\ell\left(  \gamma\right)  }q^{\ell\left(
\alpha\right)  -\ell\left(  \gamma\right)  }\left(  -1\right)  ^{\ell\left(
\gamma\right)  }r^{\ell\left(  \gamma\right)  }M_{\gamma}\nonumber\\
&  =\sum_{\substack{\gamma\in\operatorname*{Comp}\nolimits_{n};\\D\left(
\gamma\right)  \subseteq D\left(  \alpha\right)  }}\underbrace{\left(
-1\right)  ^{\ell\left(  \alpha\right)  -\ell\left(  \gamma\right)  }\left(
-1\right)  ^{\ell\left(  \gamma\right)  }}_{=\left(  -1\right)  ^{\left(
\ell\left(  \alpha\right)  -\ell\left(  \gamma\right)  \right)  +\ell\left(
\gamma\right)  }=\left(  -1\right)  ^{\ell\left(  \alpha\right)  }}%
q^{\ell\left(  \alpha\right)  -\ell\left(  \gamma\right)  }r^{\ell\left(
\gamma\right)  }M_{\gamma}\nonumber\\
&  =\sum_{\substack{\gamma\in\operatorname*{Comp}\nolimits_{n};\\D\left(
\gamma\right)  \subseteq D\left(  \alpha\right)  }}\left(  -1\right)
^{\ell\left(  \alpha\right)  }q^{\ell\left(  \alpha\right)  -\ell\left(
\gamma\right)  }r^{\ell\left(  \gamma\right)  }M_{\gamma}\nonumber\\
&  =\left(  -1\right)  ^{\ell\left(  \alpha\right)  }\sum_{\substack{\gamma
\in\operatorname*{Comp}\nolimits_{n};\\D\left(  \gamma\right)  \subseteq
D\left(  \alpha\right)  }}q^{\ell\left(  \alpha\right)  -\ell\left(
\gamma\right)  }r^{\ell\left(  \gamma\right)  }M_{\gamma}.
\label{pf.lem.eta.S2rev.LHS2}%
\end{align}

On the other hand,%
\begin{align*}
&  \sum_{\substack{\beta\in\operatorname*{Comp}\nolimits_{n};\\D\left(
\beta\right)  \subseteq D\left(  \alpha\right)  }}\left(  q-1\right)
^{\ell\left(  \alpha\right)  -\ell\left(  \beta\right)  }\underbrace{\eta
_{\beta}^{\left(  q\right)  }}_{\substack{=\sum_{\substack{\gamma
\in\operatorname*{Comp}\nolimits_{n};\\D\left(  \gamma\right)  \subseteq
D\left(  \beta\right)  }}r^{\ell\left(  \gamma\right)  }M_{\gamma}\\\text{(by
(\ref{pf.lem.eta.S2rev.eta=}))}}}\\
&  =\sum_{\substack{\beta\in\operatorname*{Comp}\nolimits_{n};\\D\left(
\beta\right)  \subseteq D\left(  \alpha\right)  }}\left(  q-1\right)
^{\ell\left(  \alpha\right)  -\ell\left(  \beta\right)  }\sum
_{\substack{\gamma\in\operatorname*{Comp}\nolimits_{n};\\D\left(
\gamma\right)  \subseteq D\left(  \beta\right)  }}r^{\ell\left(
\gamma\right)  }M_{\gamma}\\
&  =\underbrace{\sum_{\substack{\beta\in\operatorname*{Comp}\nolimits_{n}%
;\\D\left(  \beta\right)  \subseteq D\left(  \alpha\right)  }}\ \ \sum
_{\substack{\gamma\in\operatorname*{Comp}\nolimits_{n};\\D\left(
\gamma\right)  \subseteq D\left(  \beta\right)  }}}_{\substack{=\sum
_{\substack{\gamma\in\operatorname*{Comp}\nolimits_{n};\\D\left(
\gamma\right)  \subseteq D\left(  \alpha\right)  }}\ \ \sum_{\substack{\beta
\in\operatorname*{Comp}\nolimits_{n};\\D\left(  \beta\right)  \subseteq
D\left(  \alpha\right)  \text{ and }D\left(  \gamma\right)  \subseteq D\left(
\beta\right)  }}\\\text{(by (\ref{pf.lem.eta.S2rev.sums}))}}}\left(
q-1\right)  ^{\ell\left(  \alpha\right)  -\ell\left(  \beta\right)  }%
r^{\ell\left(  \gamma\right)  }M_{\gamma}\\
&  =\sum_{\substack{\gamma\in\operatorname*{Comp}\nolimits_{n};\\D\left(
\gamma\right)  \subseteq D\left(  \alpha\right)  }}\ \ \underbrace{\sum
_{\substack{\beta\in\operatorname*{Comp}\nolimits_{n};\\D\left(  \beta\right)
\subseteq D\left(  \alpha\right)  \text{ and }D\left(  \gamma\right)
\subseteq D\left(  \beta\right)  }}\left(  q-1\right)  ^{\ell\left(
\alpha\right)  -\ell\left(  \beta\right)  }}_{\substack{=\left(  1+\left(
q-1\right)  \right)  ^{\ell\left(  \alpha\right)  -\ell\left(  \gamma\right)
}\\\text{(by Lemma \ref{lem.Comp-sandwich-sum} \textbf{(c)},}\\\text{applied
to }v=q-1\text{)}}}r^{\ell\left(  \gamma\right)  }M_{\gamma}\\
&  =\sum_{\substack{\gamma\in\operatorname*{Comp}\nolimits_{n};\\D\left(
\gamma\right)  \subseteq D\left(  \alpha\right)  }}\left(
\underbrace{1+\left(  q-1\right)  }_{=q}\right)  ^{\ell\left(  \alpha\right)
-\ell\left(  \gamma\right)  }r^{\ell\left(  \gamma\right)  }M_{\gamma}%
=\sum_{\substack{\gamma\in\operatorname*{Comp}\nolimits_{n};\\D\left(
\gamma\right)  \subseteq D\left(  \alpha\right)  }}q^{\ell\left(
\alpha\right)  -\ell\left(  \gamma\right)  }r^{\ell\left(  \gamma\right)
}M_{\gamma}.
\end{align*}
Multiplying this equality by $\left(  -1\right)  ^{\ell\left(  \alpha\right)
}$, we obtain%
\[
\left(  -1\right)  ^{\ell\left(  \alpha\right)  }\sum_{\substack{\beta
\in\operatorname*{Comp}\nolimits_{n};\\D\left(  \beta\right)  \subseteq
D\left(  \alpha\right)  }}\left(  q-1\right)  ^{\ell\left(  \alpha\right)
-\ell\left(  \beta\right)  }\eta_{\beta}^{\left(  q\right)  }=\left(
-1\right)  ^{\ell\left(  \alpha\right)  }\sum_{\substack{\gamma\in
\operatorname*{Comp}\nolimits_{n};\\D\left(  \gamma\right)  \subseteq D\left(
\alpha\right)  }}q^{\ell\left(  \alpha\right)  -\ell\left(  \gamma\right)
}r^{\ell\left(  \gamma\right)  }M_{\gamma}.
\]
Comparing this with (\ref{pf.lem.eta.S2rev.LHS2}), we obtain%
\[
S\left(  \eta_{\operatorname*{rev}\alpha}^{\left(  q\right)  }\right)
=\left(  -1\right)  ^{\ell\left(  \alpha\right)  }\sum_{\substack{\beta
\in\operatorname*{Comp}\nolimits_{n};\\D\left(  \beta\right)  \subseteq
D\left(  \alpha\right)  }}\left(  q-1\right)  ^{\ell\left(  \alpha\right)
-\ell\left(  \beta\right)  }\eta_{\beta}^{\left(  q\right)  }.
\]
Thus, Lemma \ref{lem.eta.S2rev} is proved.
\end{proof}
\end{verlong}

\begin{vershort}

\begin{proof}
[Proof of Theorem \ref{thm.eta.S2}.]We replace $\alpha$ by
$\operatorname*{rev}\alpha$. Thus, $\alpha$ and $\operatorname*{rev}\alpha$
become $\operatorname*{rev}\alpha$ and $\alpha$, respectively, while the
length $\ell\left(  \alpha\right)  $ stays unchanged. Hence, the claim we must
prove becomes%
\begin{equation}
S\left(  \eta_{\operatorname*{rev}\alpha}^{\left(  q\right)  }\right)
=\left(  -1\right)  ^{\ell\left(  \alpha\right)  }\sum_{\substack{\beta
\in\operatorname*{Comp}\nolimits_{n};\\D\left(  \beta\right)  \subseteq
D\left(  \alpha\right)  }}\left(  q-1\right)  ^{\ell\left(  \alpha\right)
-\ell\left(  \beta\right)  }\eta_{\beta}^{\left(  q\right)  }.
\label{pf.thm.eta.S2.short.goal}%
\end{equation}
It is this equality that we will be proving.

First, we observe that every $\beta\in\operatorname*{Comp}\nolimits_{n}$
satisfies%
\begin{equation}
S\left(  M_{\beta}\right)  =\left(  -1\right)  ^{\ell\left(  \beta\right)
}\sum_{\substack{\gamma\in\operatorname*{Comp}\nolimits_{n};\\D\left(
\gamma\right)  \subseteq D\left(  \operatorname*{rev}\beta\right)  }%
}M_{\gamma}. \label{pf.thm.eta.S2.short.S1}%
\end{equation}
(Indeed, this is just the formula (\ref{eq.SMalpha}), applied to $\beta$
instead of $\alpha$ and restated using Definition \ref{def.rev}.) Substituting
$\operatorname*{rev}\beta$ for $\beta$ in (\ref{pf.thm.eta.S2.short.S1}), we
obtain the following: Every $\beta\in\operatorname*{Comp}\nolimits_{n}$
satisfies%
\begin{align}
S\left(  M_{\operatorname*{rev}\beta}\right)   &  =\left(  -1\right)
^{\ell\left(  \operatorname*{rev}\beta\right)  }\sum_{\substack{\gamma
\in\operatorname*{Comp}\nolimits_{n};\\D\left(  \gamma\right)  \subseteq
D\left(  \operatorname*{rev}\left(  \operatorname*{rev}\beta\right)  \right)
}}M_{\gamma}\ \ \ \ \ \ \ \ \ \ \left(  \text{since }\operatorname*{rev}%
\beta\in\operatorname*{Comp}\nolimits_{n}\right) \nonumber\\
&  =\left(  -1\right)  ^{\ell\left(  \beta\right)  }\sum_{\substack{\gamma
\in\operatorname*{Comp}\nolimits_{n};\\D\left(  \gamma\right)  \subseteq
D\left(  \beta\right)  }}M_{\gamma} \label{pf.thm.eta.S2.short.S2}%
\end{align}
(since $\operatorname*{rev}\left(  \operatorname*{rev}\beta\right)  =\beta$
and $\ell\left(  \operatorname*{rev}\beta\right)  =\ell\left(  \beta\right)  $).

Next, we recall two simple facts from \cite{comps}. First, \cite[Corollary
3.10]{comps} says that the map%
\begin{align*}
\operatorname*{Comp}\nolimits_{n}  &  \rightarrow\operatorname*{Comp}%
\nolimits_{n},\\
\delta &  \mapsto\operatorname*{rev}\delta
\end{align*}
is a bijection. Furthermore, \cite[Proposition 3.11]{comps} shows that if
$\beta\in\operatorname*{Comp}\nolimits_{n}$ is arbitrary, then we have the
logical equivalence%
\begin{equation}
\left(  D\left(  \operatorname*{rev}\beta\right)  \subseteq D\left(
\operatorname*{rev}\alpha\right)  \right)  \ \Longleftrightarrow\ \left(
D\left(  \beta\right)  \subseteq D\left(  \alpha\right)  \right)  .
\label{pf.thm.eta.S2.short.equiv}%
\end{equation}

The definition of $\eta_{\operatorname*{rev}\alpha}^{\left(  q\right)  }$
yields%
\begin{align*}
\eta_{\operatorname*{rev}\alpha}^{\left(  q\right)  }  &  =\sum
_{\substack{\beta\in\operatorname*{Comp}\nolimits_{n};\\D\left(  \beta\right)
\subseteq D\left(  \operatorname*{rev}\alpha\right)  }}r^{\ell\left(
\beta\right)  }M_{\beta}=\underbrace{\sum_{\substack{\beta\in
\operatorname*{Comp}\nolimits_{n};\\D\left(  \operatorname*{rev}\beta\right)
\subseteq D\left(  \operatorname*{rev}\alpha\right)  }}}_{\substack{=\sum
_{\substack{\beta\in\operatorname*{Comp}\nolimits_{n};\\D\left(  \beta\right)
\subseteq D\left(  \alpha\right)  }}\\\text{(by the equivalence
(\ref{pf.thm.eta.S2.short.equiv}))}}}\underbrace{r^{\ell\left(
\operatorname*{rev}\beta\right)  }}_{\substack{=r^{\ell\left(  \beta\right)
}\\\text{(since }\ell\left(  \operatorname*{rev}\beta\right)  =\ell\left(
\beta\right)  \text{)}}}M_{\operatorname*{rev}\beta}\\
&  \ \ \ \ \ \ \ \ \ \ \ \ \ \ \ \ \ \ \ \ \left(
\begin{array}
[c]{c}%
\text{here, we have substituted }\operatorname*{rev}\beta\text{ for }%
\beta\text{ in the sum,}\\
\text{since the map }\operatorname*{Comp}\nolimits_{n}\rightarrow
\operatorname*{Comp}\nolimits_{n},\ \delta\mapsto\operatorname*{rev}\delta\\
\text{is a bijection}%
\end{array}
\right) \\
&  =\sum_{\substack{\beta\in\operatorname*{Comp}\nolimits_{n};\\D\left(
\beta\right)  \subseteq D\left(  \alpha\right)  }}r^{\ell\left(  \beta\right)
}M_{\operatorname*{rev}\beta}.
\end{align*}
Applying the $\mathbf{k}$-linear map $S$ to both sides of this equality, we
obtain%
\begin{align*}
S\left(  \eta_{\operatorname*{rev}\alpha}^{\left(  q\right)  }\right)   &
=\sum_{\substack{\beta\in\operatorname*{Comp}\nolimits_{n};\\D\left(
\beta\right)  \subseteq D\left(  \alpha\right)  }}r^{\ell\left(  \beta\right)
}S\left(  M_{\operatorname*{rev}\beta}\right) \\
&  =\sum_{\substack{\beta\in\operatorname*{Comp}\nolimits_{n};\\D\left(
\beta\right)  \subseteq D\left(  \alpha\right)  }}r^{\ell\left(  \beta\right)
}\left(  -1\right)  ^{\ell\left(  \beta\right)  }\sum_{\substack{\gamma
\in\operatorname*{Comp}\nolimits_{n};\\D\left(  \gamma\right)  \subseteq
D\left(  \beta\right)  }}M_{\gamma}\ \ \ \ \ \ \ \ \ \ \left(  \text{by
(\ref{pf.thm.eta.S2.short.S2})}\right) \\
&  =\underbrace{\sum_{\substack{\beta\in\operatorname*{Comp}\nolimits_{n}%
;\\D\left(  \beta\right)  \subseteq D\left(  \alpha\right)  }}\ \ \sum
_{\substack{\gamma\in\operatorname*{Comp}\nolimits_{n};\\D\left(
\gamma\right)  \subseteq D\left(  \beta\right)  }}}_{=\sum_{\substack{\gamma
\in\operatorname*{Comp}\nolimits_{n};\\D\left(  \gamma\right)  \subseteq
D\left(  \alpha\right)  }}\ \ \sum_{\substack{\beta\in\operatorname*{Comp}%
\nolimits_{n};\\D\left(  \beta\right)  \subseteq D\left(  \alpha\right)
\text{ and }D\left(  \gamma\right)  \subseteq D\left(  \beta\right)  }%
}}\underbrace{r^{\ell\left(  \beta\right)  }\left(  -1\right)  ^{\ell\left(
\beta\right)  }}_{=\left(  -r\right)  ^{\ell\left(  \beta\right)  }}M_{\gamma
}\\
&  =\sum_{\substack{\gamma\in\operatorname*{Comp}\nolimits_{n};\\D\left(
\gamma\right)  \subseteq D\left(  \alpha\right)  }}\ \ \underbrace{\sum
_{\substack{\beta\in\operatorname*{Comp}\nolimits_{n};\\D\left(  \beta\right)
\subseteq D\left(  \alpha\right)  \text{ and }D\left(  \gamma\right)
\subseteq D\left(  \beta\right)  }}\left(  -r\right)  ^{\ell\left(
\beta\right)  }}_{\substack{=\left(  -r+1\right)  ^{\ell\left(  \alpha\right)
-\ell\left(  \gamma\right)  }\left(  -r\right)  ^{\ell\left(  \gamma\right)
}\\\text{(by Lemma \ref{lem.Comp-sandwich-sum} \textbf{(b)})}}}M_{\gamma}\\
&  =\sum_{\substack{\gamma\in\operatorname*{Comp}\nolimits_{n};\\D\left(
\gamma\right)  \subseteq D\left(  \alpha\right)  }}\left(  \underbrace{-r+1}%
_{\substack{=-q\\\text{(since }r=q+1\text{)}}}\right)  ^{\ell\left(
\alpha\right)  -\ell\left(  \gamma\right)  }\left(  -r\right)  ^{\ell\left(
\gamma\right)  }M_{\gamma}\\
&  =\sum_{\substack{\gamma\in\operatorname*{Comp}\nolimits_{n};\\D\left(
\gamma\right)  \subseteq D\left(  \alpha\right)  }}\underbrace{\left(
-q\right)  ^{\ell\left(  \alpha\right)  -\ell\left(  \gamma\right)  }\left(
-r\right)  ^{\ell\left(  \gamma\right)  }}_{=\left(  -1\right)  ^{\ell\left(
\alpha\right)  }q^{\ell\left(  \alpha\right)  -\ell\left(  \gamma\right)
}r^{\ell\left(  \gamma\right)  }}M_{\gamma}\\
&  =\left(  -1\right)  ^{\ell\left(  \alpha\right)  }\sum_{\substack{\gamma
\in\operatorname*{Comp}\nolimits_{n};\\D\left(  \gamma\right)  \subseteq
D\left(  \alpha\right)  }}q^{\ell\left(  \alpha\right)  -\ell\left(
\gamma\right)  }r^{\ell\left(  \gamma\right)  }M_{\gamma}.
\end{align*}

In view of%
\begin{align*}
&  \sum_{\substack{\beta\in\operatorname*{Comp}\nolimits_{n};\\D\left(
\beta\right)  \subseteq D\left(  \alpha\right)  }}\left(  q-1\right)
^{\ell\left(  \alpha\right)  -\ell\left(  \beta\right)  }\underbrace{\eta
_{\beta}^{\left(  q\right)  }}_{\substack{=\sum_{\substack{\gamma
\in\operatorname*{Comp}\nolimits_{n};\\D\left(  \gamma\right)  \subseteq
D\left(  \beta\right)  }}r^{\ell\left(  \gamma\right)  }M_{\gamma}\\\text{(by
(\ref{eq.def.etaalpha.def}))}}}\\
&  =\sum_{\substack{\beta\in\operatorname*{Comp}\nolimits_{n};\\D\left(
\beta\right)  \subseteq D\left(  \alpha\right)  }}\left(  q-1\right)
^{\ell\left(  \alpha\right)  -\ell\left(  \beta\right)  }\sum
_{\substack{\gamma\in\operatorname*{Comp}\nolimits_{n};\\D\left(
\gamma\right)  \subseteq D\left(  \beta\right)  }}r^{\ell\left(
\gamma\right)  }M_{\gamma}\\
&  =\underbrace{\sum_{\substack{\beta\in\operatorname*{Comp}\nolimits_{n}%
;\\D\left(  \beta\right)  \subseteq D\left(  \alpha\right)  }}\ \ \sum
_{\substack{\gamma\in\operatorname*{Comp}\nolimits_{n};\\D\left(
\gamma\right)  \subseteq D\left(  \beta\right)  }}}_{=\sum_{\substack{\gamma
\in\operatorname*{Comp}\nolimits_{n};\\D\left(  \gamma\right)  \subseteq
D\left(  \alpha\right)  }}\ \ \sum_{\substack{\beta\in\operatorname*{Comp}%
\nolimits_{n};\\D\left(  \beta\right)  \subseteq D\left(  \alpha\right)
\text{ and }D\left(  \gamma\right)  \subseteq D\left(  \beta\right)  }%
}}\left(  q-1\right)  ^{\ell\left(  \alpha\right)  -\ell\left(  \beta\right)
}r^{\ell\left(  \gamma\right)  }M_{\gamma}\\
&  =\sum_{\substack{\gamma\in\operatorname*{Comp}\nolimits_{n};\\D\left(
\gamma\right)  \subseteq D\left(  \alpha\right)  }}\ \ \underbrace{\sum
_{\substack{\beta\in\operatorname*{Comp}\nolimits_{n};\\D\left(  \beta\right)
\subseteq D\left(  \alpha\right)  \text{ and }D\left(  \gamma\right)
\subseteq D\left(  \beta\right)  }}\left(  q-1\right)  ^{\ell\left(
\alpha\right)  -\ell\left(  \beta\right)  }}_{\substack{=\left(  1+\left(
q-1\right)  \right)  ^{\ell\left(  \alpha\right)  -\ell\left(  \gamma\right)
}\\\text{(by Lemma \ref{lem.Comp-sandwich-sum} \textbf{(c)})}}}r^{\ell\left(
\gamma\right)  }M_{\gamma}\\
&  =\sum_{\substack{\gamma\in\operatorname*{Comp}\nolimits_{n};\\D\left(
\gamma\right)  \subseteq D\left(  \alpha\right)  }}\left(
\underbrace{1+\left(  q-1\right)  }_{=q}\right)  ^{\ell\left(  \alpha\right)
-\ell\left(  \gamma\right)  }r^{\ell\left(  \gamma\right)  }M_{\gamma}%
=\sum_{\substack{\gamma\in\operatorname*{Comp}\nolimits_{n};\\D\left(
\gamma\right)  \subseteq D\left(  \alpha\right)  }}q^{\ell\left(
\alpha\right)  -\ell\left(  \gamma\right)  }r^{\ell\left(  \gamma\right)
}M_{\gamma},
\end{align*}
we can rewrite this as%
\[
S\left(  \eta_{\operatorname*{rev}\alpha}^{\left(  q\right)  }\right)
=\left(  -1\right)  ^{\ell\left(  \alpha\right)  }\sum_{\substack{\beta
\in\operatorname*{Comp}\nolimits_{n};\\D\left(  \beta\right)  \subseteq
D\left(  \alpha\right)  }}\left(  q-1\right)  ^{\ell\left(  \alpha\right)
-\ell\left(  \beta\right)  }\eta_{\beta}^{\left(  q\right)  }.
\]
Thus, (\ref{pf.thm.eta.S2.short.goal}) is proved. As we explained, this proves
Theorem \ref{thm.eta.S2}.
\end{proof}
\end{vershort}

\begin{verlong}
We can now prove Theorem \ref{thm.eta.S2} at last:
\end{verlong}

\begin{verlong}

\begin{proof}
[Proof of Theorem \ref{thm.eta.S2}.]It is easy to see that
$\operatorname*{rev}\left(  \operatorname*{rev}\alpha\right)  =\alpha$ (see,
e.g., \cite[Proposition 3.4]{comps} for a detailed proof) and that $\left\vert
\operatorname*{rev}\alpha\right\vert =\left\vert \alpha\right\vert $ (see,
e.g., \cite[Proposition 3.3]{comps} for a detailed proof). Also, it is clear
(from the definition of $\operatorname*{rev}\alpha$) that $\ell\left(
\operatorname*{rev}\alpha\right)  =\ell\left(  \alpha\right)  $.

Moreover, from $\alpha\in\operatorname*{Comp}\nolimits_{n}$, we can easily
obtain $\operatorname*{rev}\alpha\in\operatorname*{Comp}\nolimits_{n}%
$\ \ \ \ \footnote{\textit{Proof.} Let $\alpha\in\operatorname*{Comp}%
\nolimits_{n}$. Thus, $\alpha$ is a composition of $n$. In other words,
$\alpha$ is a composition such that $\left\vert \alpha\right\vert =n$. Hence,
$\left\vert \operatorname*{rev}\alpha\right\vert =\left\vert \alpha\right\vert
=n$. Thus, $\operatorname*{rev}\alpha$ is a composition of $n$. In other
words, $\operatorname*{rev}\alpha\in\operatorname*{Comp}\nolimits_{n}$.}.
Thus, Lemma \ref{lem.eta.S2rev} (applied to $\operatorname*{rev}\alpha$
instead of $\alpha$) yields%
\[
S\left(  \eta_{\operatorname*{rev}\left(  \operatorname*{rev}\alpha\right)
}^{\left(  q\right)  }\right)  =\left(  -1\right)  ^{\ell\left(
\operatorname*{rev}\alpha\right)  }\sum_{\substack{\beta\in
\operatorname*{Comp}\nolimits_{n};\\D\left(  \beta\right)  \subseteq D\left(
\operatorname*{rev}\alpha\right)  }}\left(  q-1\right)  ^{\ell\left(
\operatorname*{rev}\alpha\right)  -\ell\left(  \beta\right)  }\eta_{\beta
}^{\left(  q\right)  }.
\]
In view of $\operatorname*{rev}\left(  \operatorname*{rev}\alpha\right)
=\alpha$ and $\ell\left(  \operatorname*{rev}\alpha\right)  =\ell\left(
\alpha\right)  $, we can rewrite this as%
\[
S\left(  \eta_{\alpha}^{\left(  q\right)  }\right)  =\left(  -1\right)
^{\ell\left(  \alpha\right)  }\sum_{\substack{\beta\in\operatorname*{Comp}%
\nolimits_{n};\\D\left(  \beta\right)  \subseteq D\left(  \operatorname*{rev}%
\alpha\right)  }}\left(  q-1\right)  ^{\ell\left(  \alpha\right)  -\ell\left(
\beta\right)  }\eta_{\beta}^{\left(  q\right)  }.
\]
Hence, Theorem \ref{thm.eta.S2} is proved.
\end{proof}
\end{verlong}

\subsection{The coproduct of $\eta_{\alpha}^{\left(  q\right)  }$}

The \emph{concatenation} of two compositions $\beta=\left(  \beta_{1}%
,\beta_{2},\ldots,\beta_{i}\right)  $ and $\gamma=\left(  \gamma_{1}%
,\gamma_{2},\ldots,\gamma_{j}\right)  $ is defined to be the composition
$\left(  \beta_{1},\beta_{2},\ldots,\beta_{i},\gamma_{1},\gamma_{2}%
,\ldots,\gamma_{j}\right)  $. It is denoted by $\beta\gamma$.

The coproduct of the Hopf algebra $\operatorname*{QSym}$ is a $\mathbf{k}%
$-linear map $\Delta:\operatorname*{QSym}\rightarrow\operatorname*{QSym}%
\otimes\operatorname*{QSym}$ that can be described by the formula%
\begin{equation}
\Delta\left(  M_{\alpha}\right)  =\sum_{\substack{\beta,\gamma\in
\operatorname*{Comp};\\\alpha=\beta\gamma}}M_{\beta}\otimes M_{\gamma},
\label{eq.Delta-M}%
\end{equation}
which holds for all $\alpha\in\operatorname*{Comp}$. (See \cite[\S 5.1]%
{GriRei} for the definition of $\Delta$, and see \cite[Proposition
5.1.7]{GriRei} for a proof of (\ref{eq.Delta-M}).)

We claim the following simple formula for $\Delta\left(  \eta_{\alpha
}^{\left(  q\right)  }\right)  $ (analogous to (\ref{eq.Delta-M})):

\begin{theorem}
\label{thm.Delta-eta}Let $\alpha\in\operatorname*{Comp}$. Then,%
\[
\Delta\left(  \eta_{\alpha}^{\left(  q\right)  }\right)  =\sum
_{\substack{\beta,\gamma\in\operatorname*{Comp};\\\alpha=\beta\gamma}%
}\eta_{\beta}^{\left(  q\right)  }\otimes\eta_{\gamma}^{\left(  q\right)  }.
\]

\end{theorem}

This generalizes \cite[Corollary 2.7]{Hsiao07}.

\begin{vershort}
We shall first give a direct proof of Theorem \ref{thm.Delta-eta}; later we
will outline another one, which is more circuitous.
\end{vershort}

\begin{verlong}
We shall give two proofs of Theorem \ref{thm.Delta-eta}: a direct one now, and
a more circuitous one later.
\end{verlong}

The direct proof uses the following notion:

\begin{definition}
Let $\gamma$ be a composition. Then, $C\left(  \gamma\right)  $ shall denote
the set of all compositions $\beta\in\operatorname*{Comp}\nolimits_{\left\vert
\gamma\right\vert }$ satisfying $D\left(  \beta\right)  \subseteq D\left(
\gamma\right)  $. (The compositions belonging to $C\left(  \gamma\right)  $
are often called the \emph{coarsenings} of $\gamma$.)
\end{definition}

For instance, $C\left(  2,1,3\right)  =\left\{  \left(  2,1,3\right)
,\ \left(  3,3\right)  ,\ \left(  2,4\right)  ,\ \left(  6\right)  \right\}  $.

Using the notion of $C\left(  \gamma\right)  $, we can restate the definition
of $\eta_{\gamma}^{\left(  q\right)  }$:

\begin{proposition}
\label{prop.eta.through-Cgamma}For any $\gamma\in\operatorname*{Comp}$, we
have%
\[
\eta_{\gamma}^{\left(  q\right)  }=\sum_{\nu\in C\left(  \gamma\right)
}r^{\ell\left(  \nu\right)  }M_{\nu}.
\]

\end{proposition}

\begin{proof}
[Proof of Proposition \ref{prop.eta.through-Cgamma}.]Let $\gamma
\in\operatorname*{Comp}$. Then, $\gamma\in\operatorname*{Comp}%
\nolimits_{\left\vert \gamma\right\vert }$. Thus, (\ref{eq.def.etaalpha.def})
(applied to $\left\vert \gamma\right\vert $ and $\gamma$ instead of $n$ and
$\alpha$) yields%
\[
\eta_{\gamma}^{\left(  q\right)  }=\sum_{\substack{\beta\in
\operatorname*{Comp}\nolimits_{\left\vert \gamma\right\vert };\\D\left(
\beta\right)  \subseteq D\left(  \gamma\right)  }}r^{\ell\left(  \beta\right)
}M_{\beta}=\sum_{\beta\in C\left(  \gamma\right)  }r^{\ell\left(
\beta\right)  }M_{\beta}%
\]
(since the compositions $\beta$ that the previous sum was ranging over are
precisely the elements of $C\left(  \gamma\right)  $). Renaming the summation
index $\beta$ as $\nu$ on the right hand side of this equality, we obtain%
\[
\eta_{\gamma}^{\left(  q\right)  }=\sum_{\nu\in C\left(  \gamma\right)
}r^{\ell\left(  \nu\right)  }M_{\nu}.
\]
This proves Proposition \ref{prop.eta.through-Cgamma}.
\end{proof}

We shall also use a simple summation formula (\cite[Proposition 5.17]{comps}):

\begin{proposition}
\label{prop.concat.sum=sum}Let $\left(  A,+,0\right)  $ be an abelian group.
Let $u_{\mu,\nu}$ be an element of $A$ for each pair $\left(  \mu,\nu\right)
\in\operatorname*{Comp}\times\operatorname*{Comp}$ of two compositions. Let
$\alpha\in\operatorname*{Comp}$. Then,%
\[
\sum_{\substack{\mu,\nu\in\operatorname*{Comp};\\\mu\nu\in C\left(
\alpha\right)  }}u_{\mu,\nu}=\sum_{\substack{\beta,\gamma\in
\operatorname*{Comp};\\\beta\gamma=\alpha}}\ \ \sum_{\mu\in C\left(
\beta\right)  }\ \ \sum_{\nu\in C\left(  \gamma\right)  }u_{\mu,\nu}.
\]

\end{proposition}

We are now ready to prove Theorem \ref{thm.Delta-eta}:

\begin{proof}
[Proof of Theorem \ref{thm.Delta-eta}.]Proposition
\ref{prop.eta.through-Cgamma} (applied to $\gamma=\alpha$) yields%
\[
\eta_{\alpha}^{\left(  q\right)  }=\sum_{\nu\in C\left(  \alpha\right)
}r^{\ell\left(  \nu\right)  }M_{\nu}=\sum_{\lambda\in C\left(  \alpha\right)
}r^{\ell\left(  \lambda\right)  }M_{\lambda}.
\]
Applying the map $\Delta$ to both sides of this equality, we find%
\begin{align}
\Delta\left(  \eta_{\alpha}^{\left(  q\right)  }\right)   &  =\Delta\left(
\sum_{\lambda\in C\left(  \alpha\right)  }r^{\ell\left(  \lambda\right)
}M_{\lambda}\right)  =\sum_{\lambda\in C\left(  \alpha\right)  }r^{\ell\left(
\lambda\right)  }\underbrace{\Delta\left(  M_{\lambda}\right)  }%
_{\substack{=\sum_{\substack{\mu,\nu\in\operatorname*{Comp};\\\lambda=\mu\nu
}}M_{\mu}\otimes M_{\nu}\\\text{(by (\ref{eq.Delta-M}), with the}%
\\\text{letters }\alpha,\beta,\gamma\text{ renamed as }\lambda,\mu,\nu
\text{)}}}\nonumber\\
&  \ \ \ \ \ \ \ \ \ \ \ \ \ \ \ \ \ \ \ \ \left(  \text{since the map }%
\Delta\text{ is }\mathbf{k}\text{-linear}\right) \nonumber\\
&  =\sum_{\lambda\in C\left(  \alpha\right)  }r^{\ell\left(  \lambda\right)
}\sum_{\substack{\mu,\nu\in\operatorname*{Comp};\\\lambda=\mu\nu}}M_{\mu
}\otimes M_{\nu}=\underbrace{\sum_{\lambda\in C\left(  \alpha\right)
}\ \ \sum_{\substack{\mu,\nu\in\operatorname*{Comp};\\\lambda=\mu\nu}}}%
_{=\sum_{\substack{\mu,\nu\in\operatorname*{Comp};\\\mu\nu\in C\left(
\alpha\right)  }}}\underbrace{r^{\ell\left(  \lambda\right)  }}%
_{\substack{=r^{\ell\left(  \mu\nu\right)  }\\\text{(since }\lambda=\mu
\nu\text{)}}}M_{\mu}\otimes M_{\nu}\nonumber\\
&  =\sum_{\substack{\mu,\nu\in\operatorname*{Comp};\\\mu\nu\in C\left(
\alpha\right)  }}r^{\ell\left(  \mu\nu\right)  }M_{\mu}\otimes M_{\nu
}\nonumber\\
&  =\sum_{\substack{\beta,\gamma\in\operatorname*{Comp};\\\beta\gamma=\alpha
}}\ \ \sum_{\mu\in C\left(  \beta\right)  }\ \ \sum_{\nu\in C\left(
\gamma\right)  }r^{\ell\left(  \mu\nu\right)  }M_{\mu}\otimes M_{\nu}
\label{pf.thm.Delta-eta.LHS}%
\end{align}
(by Proposition \ref{prop.concat.sum=sum}, applied to $A=\operatorname*{QSym}%
\otimes\operatorname*{QSym}$ and $u_{\mu,\nu}=r^{\ell\left(  \mu\nu\right)
}M_{\mu}\otimes M_{\nu}$).

On the other hand, if $\mu,\nu\in\operatorname*{Comp}$ are any two
compositions, then%
\[
\ell\left(  \mu\nu\right)  =\ell\left(  \mu\right)  +\ell\left(  \nu\right)
\ \ \ \ \ \ \ \ \ \ \left(  \text{by \cite[Proposition 5.2 \textbf{(a)}%
]{comps}}\right)
\]
and thus%
\begin{equation}
r^{\ell\left(  \mu\nu\right)  }=r^{\ell\left(  \mu\right)  +\ell\left(
\nu\right)  }=r^{\ell\left(  \mu\right)  }r^{\ell\left(  \nu\right)  }.
\label{pf.thm.Delta-eta.rr}%
\end{equation}

Now,%
\begin{align*}
&  \sum_{\substack{\beta,\gamma\in\operatorname*{Comp};\\\alpha=\beta\gamma
}}\underbrace{\eta_{\beta}^{\left(  q\right)  }}_{\substack{=\sum_{\mu\in
C\left(  \beta\right)  }r^{\ell\left(  \mu\right)  }M_{\mu}\\\text{(by
Proposition \ref{prop.eta.through-Cgamma})}}}\otimes\underbrace{\eta_{\gamma
}^{\left(  q\right)  }}_{\substack{=\sum_{\nu\in C\left(  \gamma\right)
}r^{\ell\left(  \nu\right)  }M_{\nu}\\\text{(by Proposition
\ref{prop.eta.through-Cgamma})}}}\\
&  =\sum_{\substack{\beta,\gamma\in\operatorname*{Comp};\\\alpha=\beta\gamma
}}\underbrace{\left(  \sum_{\mu\in C\left(  \beta\right)  }r^{\ell\left(
\mu\right)  }M_{\mu}\right)  \otimes\left(  \sum_{\nu\in C\left(
\gamma\right)  }r^{\ell\left(  \nu\right)  }M_{\nu}\right)  }_{=\sum_{\mu\in
C\left(  \beta\right)  }\ \ \sum_{\nu\in C\left(  \gamma\right)  }%
r^{\ell\left(  \mu\right)  }r^{\ell\left(  \nu\right)  }M_{\mu}\otimes M_{\nu
}}\\
&  =\underbrace{\sum_{\substack{\beta,\gamma\in\operatorname*{Comp}%
;\\\alpha=\beta\gamma}}}_{=\sum_{\substack{\beta,\gamma\in\operatorname*{Comp}%
;\\\beta\gamma=\alpha}}}\ \ \sum_{\mu\in C\left(  \beta\right)  }\ \ \sum
_{\nu\in C\left(  \gamma\right)  }\underbrace{r^{\ell\left(  \mu\right)
}r^{\ell\left(  \nu\right)  }}_{\substack{=r^{\ell\left(  \mu\nu\right)
}\\\text{(by (\ref{pf.thm.Delta-eta.rr}))}}}M_{\mu}\otimes M_{\nu}\\
&  =\sum_{\substack{\beta,\gamma\in\operatorname*{Comp};\\\beta\gamma=\alpha
}}\ \ \sum_{\mu\in C\left(  \beta\right)  }\ \ \sum_{\nu\in C\left(
\gamma\right)  }r^{\ell\left(  \mu\nu\right)  }M_{\mu}\otimes M_{\nu}.
\end{align*}

Comparing this with (\ref{pf.thm.Delta-eta.LHS}), we obtain%
\[
\Delta\left(  \eta_{\alpha}^{\left(  q\right)  }\right)  =\sum
_{\substack{\beta,\gamma\in\operatorname*{Comp};\\\alpha=\beta\gamma}%
}\eta_{\beta}^{\left(  q\right)  }\otimes\eta_{\gamma}^{\left(  q\right)  }.
\]
This proves Theorem \ref{thm.Delta-eta}.
\end{proof}

\subsection{The coalgebra morphism $T_{r}$}

We define a $\mathbf{k}$-linear map $T_{r}:\operatorname*{QSym}\rightarrow
\operatorname*{QSym}$ by setting%
\[
T_{r}\left(  M_{\alpha}\right)  =r^{\ell\left(  \alpha\right)  }M_{\alpha
}\ \ \ \ \ \ \ \ \ \ \text{for each }\alpha\in\operatorname*{Comp}.
\]
This definition is legitimate, since $\left(  M_{\alpha}\right)  _{\alpha
\in\operatorname*{Comp}}$ is a basis of the $\mathbf{k}$-module
$\operatorname*{QSym}$ (and since a $\mathbf{k}$-linear map on a free
$\mathbf{k}$-module can be defined by specifying its values on a basis). The
map $T_{r}$ is usually not a $\mathbf{k}$-algebra homomorphism, but it is
always a $\mathbf{k}$-coalgebra homomorphism. This chiefly relies on the
following lemma:

\begin{lemma}
\label{lem.Tr.Delta}We have $\Delta\circ T_{r}=\left(  T_{r}\otimes
T_{r}\right)  \circ\Delta$ as maps from $\operatorname*{QSym}$ to
$\operatorname*{QSym}\otimes\operatorname*{QSym}$.
\end{lemma}

\begin{vershort}

\begin{proof}
[Proof of Lemma \ref{lem.Tr.Delta}.]Let $\alpha\in\operatorname*{Comp}$ be
arbitrary. The definition of $T_{r}$ yields $T_{r}\left(  M_{\alpha}\right)
=r^{\ell\left(  \alpha\right)  }M_{\alpha}$. Applying the map $\Delta$ to both
sides of this identity, we obtain%
\begin{align*}
\Delta\left(  T_{r}\left(  M_{\alpha}\right)  \right)   &  =\Delta\left(
r^{\ell\left(  \alpha\right)  }M_{\alpha}\right)  =r^{\ell\left(
\alpha\right)  }\Delta\left(  M_{\alpha}\right)  =r^{\ell\left(
\alpha\right)  }\sum_{\substack{\beta,\gamma\in\operatorname*{Comp}%
;\\\alpha=\beta\gamma}}M_{\beta}\otimes M_{\gamma}\ \ \ \ \ \ \ \ \ \ \left(
\text{by (\ref{eq.Delta-M})}\right) \\
&  =\sum_{\substack{\beta,\gamma\in\operatorname*{Comp};\\\alpha=\beta\gamma
}}\underbrace{r^{\ell\left(  \alpha\right)  }}_{\substack{=r^{\ell\left(
\beta\gamma\right)  }\\\text{(since }\alpha=\beta\gamma\text{)}}}M_{\beta
}\otimes M_{\gamma}=\sum_{\substack{\beta,\gamma\in\operatorname*{Comp}%
;\\\alpha=\beta\gamma}}\underbrace{r^{\ell\left(  \beta\gamma\right)  }%
}_{\substack{=r^{\ell\left(  \beta\right)  +\ell\left(  \gamma\right)
}\\\text{(since the definition}\\\text{of concatenation}\\\text{yields }%
\ell\left(  \beta\gamma\right)  =\ell\left(  \beta\right)  +\ell\left(
\gamma\right)  \text{)}}}M_{\beta}\otimes M_{\gamma}\\
&  =\sum_{\substack{\beta,\gamma\in\operatorname*{Comp};\\\alpha=\beta\gamma
}}\underbrace{r^{\ell\left(  \beta\right)  +\ell\left(  \gamma\right)
}M_{\beta}\otimes M_{\gamma}}_{=r^{\ell\left(  \beta\right)  }M_{\beta}\otimes
r^{\ell\left(  \gamma\right)  }M_{\gamma}}=\sum_{\substack{\beta,\gamma
\in\operatorname*{Comp};\\\alpha=\beta\gamma}}r^{\ell\left(  \beta\right)
}M_{\beta}\otimes r^{\ell\left(  \gamma\right)  }M_{\gamma}.
\end{align*}
Comparing this with%
\begin{align*}
\left(  T_{r}\otimes T_{r}\right)  \left(  \Delta\left(  M_{\alpha}\right)
\right)   &  =\left(  T_{r}\otimes T_{r}\right)  \left(  \sum_{\substack{\beta
,\gamma\in\operatorname*{Comp};\\\alpha=\beta\gamma}}M_{\beta}\otimes
M_{\gamma}\right)  \ \ \ \ \ \ \ \ \ \ \left(  \text{by (\ref{eq.Delta-M}%
)}\right) \\
&  =\sum_{\substack{\beta,\gamma\in\operatorname*{Comp};\\\alpha=\beta\gamma
}}\underbrace{T_{r}\left(  M_{\beta}\right)  }_{\substack{=r^{\ell\left(
\beta\right)  }M_{\beta}\\\text{(by the definition of }T_{r}\text{)}}%
}\otimes\underbrace{T_{r}\left(  M_{\gamma}\right)  }_{\substack{=r^{\ell
\left(  \gamma\right)  }M_{\gamma}\\\text{(by the definition of }T_{r}%
\text{)}}}\\
&  =\sum_{\substack{\beta,\gamma\in\operatorname*{Comp};\\\alpha=\beta\gamma
}}r^{\ell\left(  \beta\right)  }M_{\beta}\otimes r^{\ell\left(  \gamma\right)
}M_{\gamma},
\end{align*}
we obtain
\[
\Delta\left(  T_{r}\left(  M_{\alpha}\right)  \right)  =\left(  T_{r}\otimes
T_{r}\right)  \left(  \Delta\left(  M_{\alpha}\right)  \right)  =\left(
\left(  T_{r}\otimes T_{r}\right)  \circ\Delta\right)  \left(  M_{\alpha
}\right)  .
\]
Hence, $\left(  \Delta\circ T_{r}\right)  \left(  M_{\alpha}\right)
=\Delta\left(  T_{r}\left(  M_{\alpha}\right)  \right)  =\left(  \left(
T_{r}\otimes T_{r}\right)  \circ\Delta\right)  \left(  M_{\alpha}\right)  $.

Forget that we fixed $\alpha$. We thus have proved that $\left(  \Delta\circ
T_{r}\right)  \left(  M_{\alpha}\right)  =\left(  \left(  T_{r}\otimes
T_{r}\right)  \circ\Delta\right)  \left(  M_{\alpha}\right)  $ for each
$\alpha\in\operatorname*{Comp}$. In other words, the two maps $\Delta\circ
T_{r}$ and $\left(  T_{r}\otimes T_{r}\right)  \circ\Delta$ agree on every
element of the basis $\left(  M_{\alpha}\right)  _{\alpha\in
\operatorname*{Comp}}$ of $\operatorname*{QSym}$. Since these two maps both
are $\mathbf{k}$-linear, this entails that they are completely identical. In
other words, we have $\Delta\circ T_{r}=\left(  T_{r}\otimes T_{r}\right)
\circ\Delta$. This proves Lemma \ref{lem.Tr.Delta}.
\end{proof}
\end{vershort}

\begin{verlong}

\begin{proof}
[Proof of Lemma \ref{lem.Tr.Delta}.]Let $\alpha\in\operatorname*{Comp}$ be
arbitrary. If $\beta,\gamma\in\operatorname*{Comp}$ are two compositions
satisfying $\alpha=\beta\gamma$, then%
\begin{align*}
\ell\left(  \alpha\right)   &  =\ell\left(  \beta\gamma\right)
\ \ \ \ \ \ \ \ \ \ \left(  \text{since }\alpha=\beta\gamma\right) \\
&  =\ell\left(  \beta\right)  +\ell\left(  \gamma\right)
\ \ \ \ \ \ \ \ \ \ \left(  \text{by \cite[Proposition 5.2 \textbf{(a)}%
]{comps}}\right)
\end{align*}
and therefore%
\begin{equation}
r^{\ell\left(  \alpha\right)  }=r^{\ell\left(  \beta\right)  +\ell\left(
\gamma\right)  }=r^{\ell\left(  \beta\right)  }r^{\ell\left(  \gamma\right)
}. \label{pf.prop.Tr.coalgmor.rl}%
\end{equation}

Now,
\begin{align}
\left(  \Delta\circ T_{r}\right)  \left(  M_{\alpha}\right)   &
=\Delta\left(  \underbrace{T_{r}\left(  M_{\alpha}\right)  }%
_{\substack{=r^{\ell\left(  \alpha\right)  }M_{\alpha}\\\text{(by the
definition of }T_{r}\text{)}}}\right)  =\Delta\left(  r^{\ell\left(
\alpha\right)  }M_{\alpha}\right) \nonumber\\
&  =r^{\ell\left(  \alpha\right)  }\underbrace{\Delta\left(  M_{\alpha
}\right)  }_{\substack{=\sum_{\substack{\beta,\gamma\in\operatorname*{Comp}%
;\\\alpha=\beta\gamma}}M_{\beta}\otimes M_{\gamma}\\\text{(by
(\ref{eq.Delta-M}))}}}\ \ \ \ \ \ \ \ \ \ \left(  \text{since the map }%
\Delta\text{ is }\mathbf{k}\text{-linear}\right) \nonumber\\
&  =r^{\ell\left(  \alpha\right)  }\sum_{\substack{\beta,\gamma\in
\operatorname*{Comp};\\\alpha=\beta\gamma}}M_{\beta}\otimes M_{\gamma}%
=\sum_{\substack{\beta,\gamma\in\operatorname*{Comp};\\\alpha=\beta\gamma
}}\underbrace{r^{\ell\left(  \alpha\right)  }}_{\substack{=r^{\ell\left(
\beta\right)  }r^{\ell\left(  \gamma\right)  }\\\text{(by
(\ref{pf.prop.Tr.coalgmor.rl}))}}}M_{\beta}\otimes M_{\gamma}\nonumber\\
&  =\sum_{\substack{\beta,\gamma\in\operatorname*{Comp};\\\alpha=\beta\gamma
}}\underbrace{r^{\ell\left(  \beta\right)  }r^{\ell\left(  \gamma\right)
}M_{\beta}\otimes M_{\gamma}}_{=r^{\ell\left(  \beta\right)  }M_{\beta}\otimes
r^{\ell\left(  \gamma\right)  }M_{\gamma}}\nonumber\\
&  =\sum_{\substack{\beta,\gamma\in\operatorname*{Comp};\\\alpha=\beta\gamma
}}r^{\ell\left(  \beta\right)  }M_{\beta}\otimes r^{\ell\left(  \gamma\right)
}M_{\gamma}. \label{pf.lem.Tr.Delta.1}%
\end{align}

On the other hand, applying the map $T_{r}\otimes T_{r}$ to both sides of the
equality (\ref{eq.Delta-M}), we find%
\begin{align*}
\left(  T_{r}\otimes T_{r}\right)  \left(  \Delta\left(  M_{\alpha}\right)
\right)   &  =\left(  T_{r}\otimes T_{r}\right)  \left(  \sum_{\substack{\beta
,\gamma\in\operatorname*{Comp};\\\alpha=\beta\gamma}}M_{\beta}\otimes
M_{\gamma}\right) \\
&  =\sum_{\substack{\beta,\gamma\in\operatorname*{Comp};\\\alpha=\beta\gamma
}}\underbrace{T_{r}\left(  M_{\beta}\right)  }_{\substack{=r^{\ell\left(
\beta\right)  }M_{\beta}\\\text{(by the definition of }T_{r}\text{)}}%
}\otimes\underbrace{T_{r}\left(  M_{\gamma}\right)  }_{\substack{=r^{\ell
\left(  \gamma\right)  }M_{\gamma}\\\text{(by the definition of }T_{r}%
\text{)}}}\\
&  \ \ \ \ \ \ \ \ \ \ \ \ \ \ \ \ \ \ \ \ \left(  \text{by the definition of
the map }T_{r}\otimes T_{r}\right) \\
&  =\sum_{\substack{\beta,\gamma\in\operatorname*{Comp};\\\alpha=\beta\gamma
}}r^{\ell\left(  \beta\right)  }M_{\beta}\otimes r^{\ell\left(  \gamma\right)
}M_{\gamma}.
\end{align*}
Comparing this with (\ref{pf.lem.Tr.Delta.1}), we obtain%
\[
\left(  \Delta\circ T_{r}\right)  \left(  M_{\alpha}\right)  =\left(
T_{r}\otimes T_{r}\right)  \left(  \Delta\left(  M_{\alpha}\right)  \right)
=\left(  \left(  T_{r}\otimes T_{r}\right)  \circ\Delta\right)  \left(
M_{\alpha}\right)  .
\]

Forget that we fixed $\alpha$. We thus have proved that $\left(  \Delta\circ
T_{r}\right)  \left(  M_{\alpha}\right)  =\left(  \left(  T_{r}\otimes
T_{r}\right)  \circ\Delta\right)  \left(  M_{\alpha}\right)  $ for each
$\alpha\in\operatorname*{Comp}$. In other words, the two maps $\Delta\circ
T_{r}$ and $\left(  T_{r}\otimes T_{r}\right)  \circ\Delta$ agree on every
element of the basis $\left(  M_{\alpha}\right)  _{\alpha\in
\operatorname*{Comp}}$ of $\operatorname*{QSym}$. Since these two maps both
are $\mathbf{k}$-linear, this entails that they are completely identical. In
other words, we have $\Delta\circ T_{r}=\left(  T_{r}\otimes T_{r}\right)
\circ\Delta$. This proves Lemma \ref{lem.Tr.Delta}.
\end{proof}
\end{verlong}

\begin{proposition}
\label{prop.Tr.coalgmor}The map $T_{r}:\operatorname*{QSym}\rightarrow
\operatorname*{QSym}$ is a $\mathbf{k}$-coalgebra homomorphism.
\end{proposition}

\begin{vershort}

\begin{proof}
[Proof of Proposition \ref{prop.Tr.coalgmor}.]Let $\varepsilon
:\operatorname*{QSym}\rightarrow\mathbf{k}$ be the counit of the $\mathbf{k}%
$-coalgebra $\operatorname*{QSym}$. It is well-known (and easy to see) that
each composition $\alpha\in\operatorname*{Comp}$ satisfies $\varepsilon\left(
M_{\alpha}\right)  =\left[  \alpha=\varnothing\right]  $ (where we are using
Convention \ref{conv.iverson} again). This, in turn, makes it straightforward
to check that $\left(  \varepsilon\circ T_{r}\right)  \left(  M_{\alpha
}\right)  =\varepsilon\left(  M_{\alpha}\right)  $ for each $\alpha
\in\operatorname*{Comp}$. In other words, the two maps $\varepsilon\circ
T_{r}$ and $\varepsilon$ agree on every element of the basis $\left(
M_{\alpha}\right)  _{\alpha\in\operatorname*{Comp}}$ of $\operatorname*{QSym}%
$. Since these two maps both are $\mathbf{k}$-linear, this entails that they
are completely identical. In other words, we have $\varepsilon\circ
T_{r}=\varepsilon$. Combining this with the equality $\Delta\circ
T_{r}=\left(  T_{r}\otimes T_{r}\right)  \circ\Delta$ from Lemma
\ref{lem.Tr.Delta}, we conclude that the linear map $T_{r}$ is a $\mathbf{k}%
$-coalgebra homomorphism. This proves Proposition \ref{prop.Tr.coalgmor}.
\end{proof}
\end{vershort}

\begin{verlong}

\begin{proof}
[Proof of Proposition \ref{prop.Tr.coalgmor}.]Let $\varepsilon
:\operatorname*{QSym}\rightarrow\mathbf{k}$ be the counit of the $\mathbf{k}%
$-coalgebra $\operatorname*{QSym}$. This map $\varepsilon$ is graded (since
$\operatorname*{QSym}$ is a graded Hopf algebra). Thus, it is easy to see that
each composition $\alpha\in\operatorname*{Comp}$ satisfies\footnote{We are
using Convention \ref{conv.iverson} again here.}%
\begin{equation}
\varepsilon\left(  M_{\alpha}\right)  =\left[  \alpha=\varnothing\right]  .
\label{pf.prop.Tr.coalgmor.eps}%
\end{equation}

[\textit{Proof of (\ref{pf.prop.Tr.coalgmor.eps}):} Let $\alpha\in
\operatorname*{Comp}$. We must prove (\ref{pf.prop.Tr.coalgmor.eps}).

Recall that $M_{\varnothing}=1$. Thus, $\varepsilon\left(  M_{\varnothing
}\right)  =\varepsilon\left(  1\right)  =1$. Comparing this with $\left[
\varnothing=\varnothing\right]  =1$ (which holds because $\varnothing
=\varnothing$), we obtain $\varepsilon\left(  M_{\varnothing}\right)  =\left[
\varnothing=\varnothing\right]  $. In other words,
(\ref{pf.prop.Tr.coalgmor.eps}) holds for $\alpha=\varnothing$. Hence, for the
rest of this proof of (\ref{pf.prop.Tr.coalgmor.eps}), we WLOG assume that
$\alpha\neq\varnothing$.

If we had $\left\vert \alpha\right\vert =0$, then we would have $\alpha
=\varnothing$ (by \cite[Proposition 2.4]{comps}), which would contradict
$\alpha\neq\varnothing$. Hence, we cannot have $\left\vert \alpha\right\vert
=0$. Thus, $\left\vert \alpha\right\vert \neq0$.

However, the element $M_{\alpha}\in\operatorname*{QSym}$ is homogeneous of
degree $\left\vert \alpha\right\vert $ (this follows easily from its
definition). Thus, its image $\varepsilon\left(  M_{\alpha}\right)  $ is
homogeneous of degree $\left\vert \alpha\right\vert $ as well (since the map
$\varepsilon$ is graded). Hence, $\varepsilon\left(  M_{\alpha}\right)  $ is
homogeneous of degree $\neq0$ (since $\left\vert \alpha\right\vert \neq0$).

But the graded $\mathbf{k}$-module $\mathbf{k}$ is concentrated in degree $0$;
that is, its only nonzero graded component is the $0$-th graded component. In
other words, every homogeneous element $\lambda\in\mathbf{k}$ of degree
$\neq0$ is $0$. Applying this to $\lambda=\varepsilon\left(  M_{\alpha
}\right)  $, we conclude that $\varepsilon\left(  M_{\alpha}\right)  $ is $0$
(since $\varepsilon\left(  M_{\alpha}\right)  $ is a homogeneous element of
degree $\neq0$). In other words, $\varepsilon\left(  M_{\alpha}\right)  =0$.

But $\left[  \alpha=\varnothing\right]  =0$ as well (since $\alpha
\neq\varnothing$). Comparing this with $\varepsilon\left(  M_{\alpha}\right)
=0$, we obtain $\varepsilon\left(  M_{\alpha}\right)  =\left[  \alpha
=\varnothing\right]  $. Thus, (\ref{pf.prop.Tr.coalgmor.eps}) is proved.]
\medskip

Now, let $\alpha\in\operatorname*{Comp}$ be arbitrary. Then,
\[
\left(  \varepsilon\circ T_{r}\right)  \left(  M_{\alpha}\right)
=\varepsilon\left(  \underbrace{T_{r}\left(  M_{\alpha}\right)  }%
_{\substack{=r^{\ell\left(  \alpha\right)  }M_{\alpha}\\\text{(by the
definition of }T_{r}\text{)}}}\right)  =\varepsilon\left(  r^{\ell\left(
\alpha\right)  }M_{\alpha}\right)  =r^{\ell\left(  \alpha\right)  }%
\varepsilon\left(  M_{\alpha}\right)
\]
(since the map $\varepsilon$ is $\mathbf{k}$-linear). Hence,%
\begin{equation}
\left(  \varepsilon\circ T_{r}\right)  \left(  M_{\alpha}\right)
=r^{\ell\left(  \alpha\right)  }\underbrace{\varepsilon\left(  M_{\alpha
}\right)  }_{\substack{=\left[  \alpha=\varnothing\right]  \\\text{(by
(\ref{pf.prop.Tr.coalgmor.eps}))}}}=r^{\ell\left(  \alpha\right)  }\left[
\alpha=\varnothing\right]  . \label{pf.prop.Tr.coalgmor.1}%
\end{equation}
We shall now prove the equality
\begin{equation}
r^{\ell\left(  \alpha\right)  }\left[  \alpha=\varnothing\right]  =\left[
\alpha=\varnothing\right]  . \label{pf.prop.Tr.coalgmor.2}%
\end{equation}

[\textit{Proof of (\ref{pf.prop.Tr.coalgmor.2}):} The equality
(\ref{pf.prop.Tr.coalgmor.2}) is obvious when $\left[  \alpha=\varnothing
\right]  =0$ (because it boils down to $r^{\ell\left(  \alpha\right)  }%
\cdot0=0$ in this case). Thus, for the rest of its proof, we WLOG assume that
$\left[  \alpha=\varnothing\right]  \neq0$. Hence, $\alpha=\varnothing$ must
be true. Therefore, $\ell\left(  \alpha\right)  =\ell\left(  \varnothing
\right)  =0$, so that $r^{\ell\left(  \alpha\right)  }=r^{0}=1$ and thus
$\underbrace{r^{\ell\left(  \alpha\right)  }}_{=1}\left[  \alpha
=\varnothing\right]  =\left[  \alpha=\varnothing\right]  $. Hence,
(\ref{pf.prop.Tr.coalgmor.2}) is proved.]

Now, (\ref{pf.prop.Tr.coalgmor.1}) becomes
\[
\left(  \varepsilon\circ T_{r}\right)  \left(  M_{\alpha}\right)
=r^{\ell\left(  \alpha\right)  }\left[  \alpha=\varnothing\right]  =\left[
\alpha=\varnothing\right]  =\varepsilon\left(  M_{\alpha}\right)
\ \ \ \ \ \ \ \ \ \ \left(  \text{by (\ref{pf.prop.Tr.coalgmor.eps})}\right)
.
\]

Forget that we fixed $\alpha$. We thus have proved that $\left(
\varepsilon\circ T_{r}\right)  \left(  M_{\alpha}\right)  =\varepsilon\left(
M_{\alpha}\right)  $ for each $\alpha\in\operatorname*{Comp}$. In other words,
the two maps $\varepsilon\circ T_{r}$ and $\varepsilon$ agree on every element
of the basis $\left(  M_{\alpha}\right)  _{\alpha\in\operatorname*{Comp}}$ of
$\operatorname*{QSym}$. Since these two maps both are $\mathbf{k}$-linear,
this entails that they are completely identical. In other words, we have
$\varepsilon\circ T_{r}=\varepsilon$. Combining this with the equality
$\Delta\circ T_{r}=\left(  T_{r}\otimes T_{r}\right)  \circ\Delta$ from Lemma
\ref{lem.Tr.Delta}, we conclude that the linear map $T_{r}$ is a $\mathbf{k}%
$-coalgebra homomorphism. This proves Proposition \ref{prop.Tr.coalgmor}.
\end{proof}
\end{verlong}

To us, the map $T_{r}$ becomes useful thanks to the following slick expression
for $\eta_{\alpha}^{\left(  q\right)  }$ that it allows:

\begin{theorem}
\label{thm.Tr.eta}Let $S:\operatorname*{QSym}\rightarrow\operatorname*{QSym}$
be the antipode of the Hopf algebra $\operatorname*{QSym}$. Let $\alpha
\in\operatorname*{Comp}$. Then,%
\[
\eta_{\alpha}^{\left(  q\right)  }=\left(  -1\right)  ^{\ell\left(
\alpha\right)  }T_{r}\left(  S\left(  M_{\operatorname*{rev}\alpha}\right)
\right)  .
\]

\end{theorem}

\begin{proof}
[Proof of Theorem \ref{thm.Tr.eta}.]Write the composition $\alpha$ in the form
$\alpha=\left(  \alpha_{1},\alpha_{2},\ldots,\alpha_{\ell}\right)  $. Thus,
the definition of $\ell\left(  \alpha\right)  $ yields $\ell\left(
\alpha\right)  =\ell$, whereas the definition of $\operatorname*{rev}\alpha$
yields $\operatorname*{rev}\alpha=\left(  \alpha_{\ell},\alpha_{\ell-1}%
,\ldots,\alpha_{1}\right)  $. Moreover, a trivial fact (\cite[Proposition
3.3]{comps}) yields $\left\vert \operatorname*{rev}\alpha\right\vert
=\left\vert \alpha\right\vert $.

\begin{vershort}
Set $n=\left\vert \alpha\right\vert $. Thus, $\alpha\in\operatorname*{Comp}%
\nolimits_{n}$. Furthermore, $\left\vert \operatorname*{rev}\alpha\right\vert
=\left\vert \alpha\right\vert =n$, so that $\operatorname*{rev}\alpha
\in\operatorname*{Comp}\nolimits_{n}$. Also, as we know, we have
$\operatorname*{rev}\alpha=\left(  \alpha_{\ell},\alpha_{\ell-1},\ldots
,\alpha_{1}\right)  $. Hence, (\ref{eq.SMalpha}) (applied to
$\operatorname*{rev}\alpha$ and $\alpha_{\ell+1-i}$ instead of $\alpha$ and
$\alpha_{i}$) yields%
\[
S\left(  M_{\operatorname*{rev}\alpha}\right)  =\left(  -1\right)  ^{\ell}%
\sum_{\substack{\gamma\in\operatorname*{Comp}\nolimits_{n};\\D\left(
\gamma\right)  \subseteq D\left(  \alpha_{1},\alpha_{2},\ldots,\alpha_{\ell
}\right)  }}M_{\gamma}=\left(  -1\right)  ^{\ell}\sum_{\substack{\gamma
\in\operatorname*{Comp}\nolimits_{n};\\D\left(  \gamma\right)  \subseteq
D\left(  \alpha\right)  }}M_{\gamma}%
\]
(since $\left(  \alpha_{1},\alpha_{2},\ldots,\alpha_{\ell}\right)  =\alpha$).
Applying the map $T_{r}$ to both sides of this equality, we obtain%
\begin{align*}
T_{r}\left(  S\left(  M_{\operatorname*{rev}\alpha}\right)  \right)   &
=T_{r}\left(  \left(  -1\right)  ^{\ell}\sum_{\substack{\gamma\in
\operatorname*{Comp}\nolimits_{n};\\D\left(  \gamma\right)  \subseteq D\left(
\alpha\right)  }}M_{\gamma}\right)  =\left(  -1\right)  ^{\ell}\sum
_{\substack{\gamma\in\operatorname*{Comp}\nolimits_{n};\\D\left(
\gamma\right)  \subseteq D\left(  \alpha\right)  }}\underbrace{T_{r}\left(
M_{\gamma}\right)  }_{\substack{=r^{\ell\left(  \gamma\right)  }M_{\gamma
}\\\text{(by the definition of }T_{r}\text{)}}}\\
&  =\left(  -1\right)  ^{\ell}\underbrace{\sum_{\substack{\gamma
\in\operatorname*{Comp}\nolimits_{n};\\D\left(  \gamma\right)  \subseteq
D\left(  \alpha\right)  }}r^{\ell\left(  \gamma\right)  }M_{\gamma}%
}_{\substack{=\sum_{\substack{\beta\in\operatorname*{Comp}\nolimits_{n}%
;\\D\left(  \beta\right)  \subseteq D\left(  \alpha\right)  }}r^{\ell\left(
\beta\right)  }M_{\beta}\\=\eta_{\alpha}^{\left(  q\right)  }\\\text{(by
(\ref{eq.def.etaalpha.def}))}}}=\left(  -1\right)  ^{\ell}\eta_{\alpha
}^{\left(  q\right)  }.
\end{align*}

\end{vershort}

\begin{verlong}
Set $n=\left\vert \alpha\right\vert $. Thus, $\alpha\in\operatorname*{Comp}%
\nolimits_{n}$. Furthermore, $\left\vert \operatorname*{rev}\alpha\right\vert
=\left\vert \alpha\right\vert =n$, so that $\operatorname*{rev}\alpha
\in\operatorname*{Comp}\nolimits_{n}$. Also, as we know, we have
$\operatorname*{rev}\alpha=\left(  \alpha_{\ell},\alpha_{\ell-1},\ldots
,\alpha_{1}\right)  $. Hence, (\ref{eq.SMalpha}) (applied to
$\operatorname*{rev}\alpha$ and $\alpha_{\ell+1-i}$ instead of $\alpha$ and
$\alpha_{i}$) yields%
\[
S\left(  M_{\operatorname*{rev}\alpha}\right)  =\left(  -1\right)  ^{\ell}%
\sum_{\substack{\gamma\in\operatorname*{Comp}\nolimits_{n};\\D\left(
\gamma\right)  \subseteq D\left(  \alpha_{1},\alpha_{2},\ldots,\alpha_{\ell
}\right)  }}M_{\gamma}=\left(  -1\right)  ^{\ell}\sum_{\substack{\gamma
\in\operatorname*{Comp}\nolimits_{n};\\D\left(  \gamma\right)  \subseteq
D\left(  \alpha\right)  }}M_{\gamma}%
\]
(since $\left(  \alpha_{1},\alpha_{2},\ldots,\alpha_{\ell}\right)  =\alpha$).
Applying the map $T_{r}$ to both sides of this equality, we obtain%
\begin{align*}
T_{r}\left(  S\left(  M_{\operatorname*{rev}\alpha}\right)  \right)   &
=T_{r}\left(  \left(  -1\right)  ^{\ell}\sum_{\substack{\gamma\in
\operatorname*{Comp}\nolimits_{n};\\D\left(  \gamma\right)  \subseteq D\left(
\alpha\right)  }}M_{\gamma}\right) \\
&  =\left(  -1\right)  ^{\ell}\sum_{\substack{\gamma\in\operatorname*{Comp}%
\nolimits_{n};\\D\left(  \gamma\right)  \subseteq D\left(  \alpha\right)
}}\underbrace{T_{r}\left(  M_{\gamma}\right)  }_{\substack{=r^{\ell\left(
\gamma\right)  }M_{\gamma}\\\text{(by the definition of }T_{r}\text{)}%
}}\ \ \ \ \ \ \ \ \ \ \left(
\begin{array}
[c]{c}%
\text{since the map }T_{r}\\
\text{is }\mathbf{k}\text{-linear}%
\end{array}
\right) \\
&  =\left(  -1\right)  ^{\ell}\underbrace{\sum_{\substack{\gamma
\in\operatorname*{Comp}\nolimits_{n};\\D\left(  \gamma\right)  \subseteq
D\left(  \alpha\right)  }}r^{\ell\left(  \gamma\right)  }M_{\gamma}%
}_{\substack{=\sum_{\substack{\beta\in\operatorname*{Comp}\nolimits_{n}%
;\\D\left(  \beta\right)  \subseteq D\left(  \alpha\right)  }}r^{\ell\left(
\beta\right)  }M_{\beta}\\=\eta_{\alpha}^{\left(  q\right)  }\\\text{(by
(\ref{eq.def.etaalpha.def}))}}}=\left(  -1\right)  ^{\ell}\eta_{\alpha
}^{\left(  q\right)  }.
\end{align*}

\end{verlong}

\noindent Multiplying both sides of this equality by $\left(  -1\right)
^{\ell\left(  \alpha\right)  }$, we obtain%
\[
\left(  -1\right)  ^{\ell\left(  \alpha\right)  }T_{r}\left(  S\left(
M_{\operatorname*{rev}\alpha}\right)  \right)  =\underbrace{\left(  -1\right)
^{\ell\left(  \alpha\right)  }}_{\substack{=\left(  -1\right)  ^{\ell
}\\\text{(since }\ell\left(  \alpha\right)  =\ell\text{)}}}\left(  -1\right)
^{\ell}\eta_{\alpha}^{\left(  q\right)  }=\underbrace{\left(  -1\right)
^{\ell}\left(  -1\right)  ^{\ell}}_{\substack{=\left(  -1\right)  ^{\ell+\ell
}=1\\\text{(since }\ell+\ell=2\ell\text{ is even)}}}\eta_{\alpha}^{\left(
q\right)  }=\eta_{\alpha}^{\left(  q\right)  }.
\]
This proves Theorem \ref{thm.Tr.eta}.
\end{proof}

\begin{verlong}

\begin{corollary}
\label{cor.Tr.etarev}Let $S:\operatorname*{QSym}\rightarrow
\operatorname*{QSym}$ be the antipode of the Hopf algebra
$\operatorname*{QSym}$. Let $\delta\in\operatorname*{Comp}$. Then,%
\[
T_{r}\left(  S\left(  M_{\delta}\right)  \right)  =\left(  -1\right)
^{\ell\left(  \delta\right)  }\eta_{\operatorname*{rev}\delta}^{\left(
q\right)  }.
\]

\end{corollary}

\begin{proof}
[Proof of Corollary \ref{cor.Tr.etarev}.]Theorem \ref{thm.Tr.eta} (applied to
$\alpha=\operatorname*{rev}\delta$) yields%
\begin{equation}
\eta_{\operatorname*{rev}\delta}^{\left(  q\right)  }=\left(  -1\right)
^{\ell\left(  \operatorname*{rev}\delta\right)  }T_{r}\left(  S\left(
M_{\operatorname*{rev}\left(  \operatorname*{rev}\delta\right)  }\right)
\right)  . \label{pf.cor.Tr.etarev.1}%
\end{equation}
However, it is easy to see that $\operatorname*{rev}\left(
\operatorname*{rev}\delta\right)  =\delta$ (by \cite[Proposition 3.4]{comps},
applied to $\alpha=\delta$) and $\ell\left(  \operatorname*{rev}\delta\right)
=\ell\left(  \delta\right)  $ (this is clear from the definition of
$\operatorname*{rev}\delta$). Using these equalities, we can rewrite
(\ref{pf.cor.Tr.etarev.1}) as%
\[
\eta_{\operatorname*{rev}\delta}^{\left(  q\right)  }=\left(  -1\right)
^{\ell\left(  \delta\right)  }T_{r}\left(  S\left(  M_{\delta}\right)
\right)  .
\]
Multiplying this equality by $\left(  -1\right)  ^{\ell\left(  \delta\right)
}$, we obtain%
\[
\left(  -1\right)  ^{\ell\left(  \delta\right)  }\eta_{\operatorname*{rev}%
\delta}^{\left(  q\right)  }=\underbrace{\left(  -1\right)  ^{\ell\left(
\delta\right)  }\left(  -1\right)  ^{\ell\left(  \delta\right)  }}_{=\left(
\left(  -1\right)  \left(  -1\right)  \right)  ^{\ell\left(  \delta\right)
}=1^{\ell\left(  \delta\right)  }=1}T_{r}\left(  S\left(  M_{\delta}\right)
\right)  =T_{r}\left(  S\left(  M_{\delta}\right)  \right)  .
\]
This proves Corollary \ref{cor.Tr.etarev}.
\end{proof}
\end{verlong}

\subsection{Another proof of $\Delta\left(  \eta_{\alpha}^{\left(  q\right)
}\right)  $}

\begin{vershort}
Let us now outline another proof of Theorem \ref{thm.Delta-eta}.

\begin{proof}
[Second proof of Theorem \ref{thm.Delta-eta} (sketched).]For every $\alpha
\in\operatorname*{Comp}$, we have%
\begin{align}
\Delta\left(  M_{\operatorname*{rev}\alpha}\right)   &  =\sum_{\substack{\beta
,\gamma\in\operatorname*{Comp};\\\operatorname*{rev}\alpha=\beta\gamma
}}M_{\beta}\otimes M_{\gamma}\ \ \ \ \ \ \ \ \ \ \left(  \text{by
(\ref{eq.Delta-M})}\right) \nonumber\\
&  =\sum_{\substack{\beta,\gamma\in\operatorname*{Comp};\\\operatorname*{rev}%
\alpha=\left(  \operatorname*{rev}\gamma\right)  \left(  \operatorname*{rev}%
\beta\right)  }}M_{\operatorname*{rev}\gamma}\otimes M_{\operatorname*{rev}%
\beta}\ \ \ \ \ \ \ \ \ \ \left(
\begin{array}
[c]{c}%
\text{here, we have}\\
\text{substituted }\operatorname*{rev}\gamma\text{ and }\operatorname*{rev}%
\beta\\
\text{for }\beta\text{ and }\gamma\text{ in the sum}%
\end{array}
\right) \nonumber\\
&  =\sum_{\substack{\beta,\gamma\in\operatorname*{Comp};\\\operatorname*{rev}%
\alpha=\operatorname*{rev}\left(  \beta\gamma\right)  }}M_{\operatorname*{rev}%
\gamma}\otimes M_{\operatorname*{rev}\beta}\ \ \ \ \ \ \ \ \ \ \left(
\begin{array}
[c]{c}%
\text{since it is easy to see}\\
\text{that }\left(  \operatorname*{rev}\gamma\right)  \left(
\operatorname*{rev}\beta\right)  =\operatorname*{rev}\left(  \beta
\gamma\right) \\
\text{for any }\beta,\gamma\in\operatorname*{Comp}%
\end{array}
\right) \nonumber\\
&  =\sum_{\substack{\beta,\gamma\in\operatorname*{Comp};\\\alpha=\beta\gamma
}}M_{\operatorname*{rev}\gamma}\otimes M_{\operatorname*{rev}\beta}
\label{pf.thm.Delta-eta.short.1}%
\end{align}
(since the condition \textquotedblleft$\operatorname*{rev}\alpha
=\operatorname*{rev}\left(  \beta\gamma\right)  $\textquotedblright\ is
equivalent to \textquotedblleft$\alpha=\beta\gamma$\textquotedblright).

Let $S:\operatorname*{QSym}\rightarrow\operatorname*{QSym}$ be the antipode of
the Hopf algebra $\operatorname*{QSym}$. A classical result (see, e.g.,
\cite[Exercise 1.4.28]{GriRei}) says that the antipode of any Hopf algebra is
a $\mathbf{k}$-coalgebra anti-endomorphism (see \cite[Definition
1.4.8]{GriRei} for the definition of this concept). Thus, in particular, the
antipode $S$ of $\operatorname*{QSym}$ is a $\mathbf{k}$-coalgebra
anti-endomorphism. In particular, this says that
\[
\Delta\circ S=T\circ\left(  S\otimes S\right)  \circ\Delta,
\]
where $T:\operatorname*{QSym}\otimes\operatorname*{QSym}\rightarrow
\operatorname*{QSym}\otimes\operatorname*{QSym}$ is the $\mathbf{k}$-linear
map that sends each pure tensor $x\otimes y$ to $y\otimes x$. Applying both
sides of this equality to $M_{\operatorname*{rev}\alpha}$, we obtain%
\begin{align}
\left(  \Delta\circ S\right)  \left(  M_{\operatorname*{rev}\alpha}\right)
&  =\left(  T\circ\left(  S\otimes S\right)  \circ\Delta\right)  \left(
M_{\operatorname*{rev}\alpha}\right)  =T\left(  \left(  S\otimes S\right)
\left(  \Delta\left(  M_{\operatorname*{rev}\alpha}\right)  \right)  \right)
\nonumber\\
&  =T\left(  \left(  S\otimes S\right)  \left(  \sum_{\substack{\beta
,\gamma\in\operatorname*{Comp};\\\alpha=\beta\gamma}}M_{\operatorname*{rev}%
\gamma}\otimes M_{\operatorname*{rev}\beta}\right)  \right)
\ \ \ \ \ \ \ \ \ \ \left(  \text{by (\ref{pf.thm.Delta-eta.short.1})}\right)
\nonumber\\
&  =\sum_{\substack{\beta,\gamma\in\operatorname*{Comp};\\\alpha=\beta\gamma
}}S\left(  M_{\operatorname*{rev}\beta}\right)  \otimes S\left(
M_{\operatorname*{rev}\gamma}\right)  \label{pf.thm.Delta-eta.short.4}%
\end{align}
(by the definitions of the maps $T$ and $S\otimes S$).

However, Lemma \ref{lem.Tr.Delta} yields $\Delta\circ T_{r}=\left(
T_{r}\otimes T_{r}\right)  \circ\Delta$, so that%
\begin{align*}
\left(  \Delta\circ T_{r}\right)  \left(  S\left(  M_{\operatorname*{rev}%
\alpha}\right)  \right)   &  =\left(  \left(  T_{r}\otimes T_{r}\right)
\circ\Delta\right)  \left(  S\left(  M_{\operatorname*{rev}\alpha}\right)
\right)  =\left(  T_{r}\otimes T_{r}\right)  \left(  \left(  \Delta\circ
S\right)  \left(  M_{\operatorname*{rev}\alpha}\right)  \right) \\
&  =\left(  T_{r}\otimes T_{r}\right)  \left(  \sum_{\substack{\beta,\gamma
\in\operatorname*{Comp};\\\alpha=\beta\gamma}}S\left(  M_{\operatorname*{rev}%
\beta}\right)  \otimes S\left(  M_{\operatorname*{rev}\gamma}\right)  \right)
\ \ \ \ \ \ \ \ \ \ \left(  \text{by (\ref{pf.thm.Delta-eta.short.4})}\right)
\\
&  =\sum_{\substack{\beta,\gamma\in\operatorname*{Comp};\\\alpha=\beta\gamma
}}\underbrace{T_{r}\left(  S\left(  M_{\operatorname*{rev}\beta}\right)
\right)  }_{\substack{=\eta_{\beta}^{\left(  q\right)  }/\left(  -1\right)
^{\ell\left(  \beta\right)  }\\\text{(since Theorem \ref{thm.Tr.eta}%
}\\\text{yields }\eta_{\beta}^{\left(  q\right)  }=\left(  -1\right)
^{\ell\left(  \beta\right)  }T_{r}\left(  S\left(  M_{\operatorname*{rev}%
\beta}\right)  \right)  \text{)}}}\otimes\underbrace{T_{r}\left(  S\left(
M_{\operatorname*{rev}\gamma}\right)  \right)  }_{\substack{=\eta_{\gamma
}^{\left(  q\right)  }/\left(  -1\right)  ^{\ell\left(  \gamma\right)
}\\\text{(since Theorem \ref{thm.Tr.eta}}\\\text{yields }\eta_{\gamma
}^{\left(  q\right)  }=\left(  -1\right)  ^{\ell\left(  \gamma\right)  }%
T_{r}\left(  S\left(  M_{\operatorname*{rev}\gamma}\right)  \right)  \text{)}%
}}\\
&  =\sum_{\substack{\beta,\gamma\in\operatorname*{Comp};\\\alpha=\beta\gamma
}}\left(  \eta_{\beta}^{\left(  q\right)  }/\left(  -1\right)  ^{\ell\left(
\beta\right)  }\right)  \otimes\left(  \eta_{\gamma}^{\left(  q\right)
}/\left(  -1\right)  ^{\ell\left(  \gamma\right)  }\right) \\
&  =\sum_{\substack{\beta,\gamma\in\operatorname*{Comp};\\\alpha=\beta\gamma
}}\left(  \eta_{\beta}^{\left(  q\right)  }\otimes\eta_{\gamma}^{\left(
q\right)  }\right)  /\underbrace{\left(  -1\right)  ^{\ell\left(
\beta\right)  +\ell\left(  \gamma\right)  }}_{\substack{=\left(  -1\right)
^{\ell\left(  \alpha\right)  }\\\text{(since it is easy to see that }%
\ell\left(  \beta\right)  +\ell\left(  \gamma\right)  =\ell\left(  \beta
\gamma\right)  \text{,}\\\text{but using }\alpha=\beta\gamma\text{ we can
rewrite}\\\text{this as }\ell\left(  \beta\right)  +\ell\left(  \gamma\right)
=\ell\left(  \alpha\right)  \text{)}}}\\
&  =\sum_{\substack{\beta,\gamma\in\operatorname*{Comp};\\\alpha=\beta\gamma
}}\left(  \eta_{\beta}^{\left(  q\right)  }\otimes\eta_{\gamma}^{\left(
q\right)  }\right)  /\left(  -1\right)  ^{\ell\left(  \alpha\right)  }.
\end{align*}
Multiplying this equality by $\left(  -1\right)  ^{\ell\left(  \alpha\right)
}$, we find%
\[
\left(  -1\right)  ^{\ell\left(  \alpha\right)  }\left(  \Delta\circ
T_{r}\right)  \left(  S\left(  M_{\operatorname*{rev}\alpha}\right)  \right)
=\sum_{\substack{\beta,\gamma\in\operatorname*{Comp};\\\alpha=\beta\gamma
}}\eta_{\beta}^{\left(  q\right)  }\otimes\eta_{\gamma}^{\left(  q\right)  },
\]
so that%
\begin{align*}
\sum_{\substack{\beta,\gamma\in\operatorname*{Comp};\\\alpha=\beta\gamma}%
}\eta_{\beta}^{\left(  q\right)  }\otimes\eta_{\gamma}^{\left(  q\right)  }
&  =\left(  -1\right)  ^{\ell\left(  \alpha\right)  }\left(  \Delta\circ
T_{r}\right)  \left(  S\left(  M_{\operatorname*{rev}\alpha}\right)  \right)
\\
&  =\left(  -1\right)  ^{\ell\left(  \alpha\right)  }\Delta\left(
T_{r}\left(  S\left(  M_{\operatorname*{rev}\alpha}\right)  \right)  \right)
\\
&  =\Delta\left(  \underbrace{\left(  -1\right)  ^{\ell\left(  \alpha\right)
}T_{r}\left(  S\left(  M_{\operatorname*{rev}\alpha}\right)  \right)
}_{\substack{=\eta_{\alpha}^{\left(  q\right)  }\\\text{(by Theorem
\ref{thm.Tr.eta})}}}\right)  =\Delta\left(  \eta_{\alpha}^{\left(  q\right)
}\right)  .
\end{align*}
This proves Theorem \ref{thm.Delta-eta} again.
\end{proof}
\end{vershort}

\begin{verlong}
We shall now give another proof of Theorem \ref{thm.Delta-eta}. We shall first
prove it in a slightly restated form:

\begin{lemma}
\label{lem.Delta-eta.rev}Let $\alpha\in\operatorname*{Comp}$. Then,%
\[
\Delta\left(  \eta_{\operatorname*{rev}\alpha}^{\left(  q\right)  }\right)
=\sum_{\substack{\beta,\gamma\in\operatorname*{Comp};\\\alpha=\beta\gamma
}}\eta_{\operatorname*{rev}\gamma}^{\left(  q\right)  }\otimes\eta
_{\operatorname*{rev}\beta}^{\left(  q\right)  }.
\]

\end{lemma}

\begin{proof}
[Proof of Lemma \ref{lem.Delta-eta.rev}.]Let us first recall a simple fact: If
$\beta$ and $\gamma$ are two compositions, then%
\[
\ell\left(  \beta\gamma\right)  =\ell\left(  \beta\right)  +\ell\left(
\gamma\right)
\]
(by \cite[Proposition 5.2 \textbf{(a)}]{comps}) and thus%
\begin{align}
\left(  -1\right)  ^{\ell\left(  \beta\gamma\right)  }  &  =\left(  -1\right)
^{\ell\left(  \beta\right)  +\ell\left(  \gamma\right)  }\nonumber\\
&  =\left(  -1\right)  ^{\ell\left(  \beta\right)  }\left(  -1\right)
^{\ell\left(  \gamma\right)  }. \label{pf.lem.Delta-eta.rev.-1l}%
\end{align}

For any two $\mathbf{k}$-modules $V$ and $W$, we let $T_{V,W}:V\otimes
W\rightarrow W\otimes V$ be the $\mathbf{k}$-linear map that sends every pure
tensor $v\otimes w$ to $w\otimes v$. (This map $T_{V,W}$ is commonly called
the \emph{twist map}.)

Let $S:\operatorname*{QSym}\rightarrow\operatorname*{QSym}$ be the antipode of
the Hopf algebra $\operatorname*{QSym}$. It is well-known (see, e.g.,
\cite[Exercise 1.4.28]{GriRei}) that the antipode of any Hopf algebra $H$ is a
coalgebra anti-endomorphism (i.e., a $\mathbf{k}$-linear map $f:H\rightarrow
H$ satisfying $\Delta\circ f=T_{H,H}\circ\left(  f\otimes f\right)
\circ\Delta$ and $\varepsilon\circ f=\varepsilon$). Applying this to
$H=\operatorname*{QSym}$, we conclude that the antipode $S$ of
$\operatorname*{QSym}$ is a coalgebra anti-endomorphism. In particular, it
thus satisfies
\[
\Delta\circ S=T_{\operatorname*{QSym},\operatorname*{QSym}}\circ\left(
S\otimes S\right)  \circ\Delta.
\]

Now, Corollary \ref{cor.Tr.etarev} (applied to $\delta=\alpha$) yields%
\[
T_{r}\left(  S\left(  M_{\alpha}\right)  \right)  =\left(  -1\right)
^{\ell\left(  \alpha\right)  }\eta_{\operatorname*{rev}\alpha}^{\left(
q\right)  }.
\]
Applying the map $\Delta$ to this equality, we find%
\[
\Delta\left(  T_{r}\left(  S\left(  M_{\alpha}\right)  \right)  \right)
=\Delta\left(  \left(  -1\right)  ^{\ell\left(  \alpha\right)  }%
\eta_{\operatorname*{rev}\alpha}^{\left(  q\right)  }\right)  =\left(
-1\right)  ^{\ell\left(  \alpha\right)  }\Delta\left(  \eta
_{\operatorname*{rev}\alpha}^{\left(  q\right)  }\right)
\]
(since the map $\Delta$ is $\mathbf{k}$-linear). Hence,
\begin{align*}
&  \left(  -1\right)  ^{\ell\left(  \alpha\right)  }\Delta\left(
\eta_{\operatorname*{rev}\alpha}^{\left(  q\right)  }\right) \\
&  =\Delta\left(  T_{r}\left(  S\left(  M_{\alpha}\right)  \right)  \right) \\
&  =\left(  \underbrace{\Delta\circ T_{r}}_{\substack{=\left(  T_{r}\otimes
T_{r}\right)  \circ\Delta\\\text{(by Lemma \ref{lem.Tr.Delta})}}}\circ
S\right)  \left(  M_{\alpha}\right) \\
&  =\left(  \left(  T_{r}\otimes T_{r}\right)  \circ\underbrace{\Delta\circ
S}_{=T_{\operatorname*{QSym},\operatorname*{QSym}}\circ\left(  S\otimes
S\right)  \circ\Delta}\right)  \left(  M_{\alpha}\right) \\
&  =\left(  \left(  T_{r}\otimes T_{r}\right)  \circ T_{\operatorname*{QSym}%
,\operatorname*{QSym}}\circ\left(  S\otimes S\right)  \circ\Delta\right)
\left(  M_{\alpha}\right) \\
&  =\left(  \left(  T_{r}\otimes T_{r}\right)  \circ T_{\operatorname*{QSym}%
,\operatorname*{QSym}}\circ\left(  S\otimes S\right)  \right)
\underbrace{\left(  \Delta\left(  M_{\alpha}\right)  \right)  }%
_{\substack{=\sum_{\substack{\beta,\gamma\in\operatorname*{Comp}%
;\\\alpha=\beta\gamma}}M_{\beta}\otimes M_{\gamma}\\\text{(by
(\ref{eq.Delta-M}))}}}\\
&  =\left(  \left(  T_{r}\otimes T_{r}\right)  \circ T_{\operatorname*{QSym}%
,\operatorname*{QSym}}\circ\left(  S\otimes S\right)  \right)  \left(
\sum_{\substack{\beta,\gamma\in\operatorname*{Comp};\\\alpha=\beta\gamma
}}M_{\beta}\otimes M_{\gamma}\right) \\
&  =\sum_{\substack{\beta,\gamma\in\operatorname*{Comp};\\\alpha=\beta\gamma
}}\underbrace{\left(  \left(  T_{r}\otimes T_{r}\right)  \circ
T_{\operatorname*{QSym},\operatorname*{QSym}}\circ\left(  S\otimes S\right)
\right)  \left(  M_{\beta}\otimes M_{\gamma}\right)  }_{=\left(  T_{r}\otimes
T_{r}\right)  \left(  T_{\operatorname*{QSym},\operatorname*{QSym}}\left(
\left(  S\otimes S\right)  \left(  M_{\beta}\otimes M_{\gamma}\right)
\right)  \right)  }\\
&  \ \ \ \ \ \ \ \ \ \ \ \ \ \ \ \ \ \ \ \ \left(  \text{since the map
}\left(  T_{r}\otimes T_{r}\right)  \circ T_{\operatorname*{QSym}%
,\operatorname*{QSym}}\circ\left(  S\otimes S\right)  \text{ is }%
\mathbf{k}\text{-linear}\right) \\
&  =\sum_{\substack{\beta,\gamma\in\operatorname*{Comp};\\\alpha=\beta\gamma
}}\left(  T_{r}\otimes T_{r}\right)  \left(  T_{\operatorname*{QSym}%
,\operatorname*{QSym}}\left(  \underbrace{\left(  S\otimes S\right)  \left(
M_{\beta}\otimes M_{\gamma}\right)  }_{=S\left(  M_{\beta}\right)  \otimes
S\left(  M_{\gamma}\right)  }\right)  \right) \\
&  =\sum_{\substack{\beta,\gamma\in\operatorname*{Comp};\\\alpha=\beta\gamma
}}\left(  T_{r}\otimes T_{r}\right)  \left(
\underbrace{T_{\operatorname*{QSym},\operatorname*{QSym}}\left(  S\left(
M_{\beta}\right)  \otimes S\left(  M_{\gamma}\right)  \right)  }%
_{\substack{=S\left(  M_{\gamma}\right)  \otimes S\left(  M_{\beta}\right)
\\\text{(by the definition of the map }T_{\operatorname*{QSym}%
,\operatorname*{QSym}}\text{)}}}\right)
\end{align*}%
\begin{align*}
&  =\sum_{\substack{\beta,\gamma\in\operatorname*{Comp};\\\alpha=\beta\gamma
}}\underbrace{\left(  T_{r}\otimes T_{r}\right)  \left(  S\left(  M_{\gamma
}\right)  \otimes S\left(  M_{\beta}\right)  \right)  }_{=T_{r}\left(
S\left(  M_{\gamma}\right)  \right)  \otimes T_{r}\left(  S\left(  M_{\beta
}\right)  \right)  }\\
&  =\sum_{\substack{\beta,\gamma\in\operatorname*{Comp};\\\alpha=\beta\gamma
}}\underbrace{T_{r}\left(  S\left(  M_{\gamma}\right)  \right)  }%
_{\substack{=\left(  -1\right)  ^{\ell\left(  \gamma\right)  }\eta
_{\operatorname*{rev}\gamma}^{\left(  q\right)  }\\\text{(by Corollary
\ref{cor.Tr.etarev})}}}\otimes\underbrace{T_{r}\left(  S\left(  M_{\beta
}\right)  \right)  }_{\substack{=\left(  -1\right)  ^{\ell\left(
\beta\right)  }\eta_{\operatorname*{rev}\beta}^{\left(  q\right)  }\\\text{(by
Corollary \ref{cor.Tr.etarev})}}}\\
&  =\sum_{\substack{\beta,\gamma\in\operatorname*{Comp};\\\alpha=\beta\gamma
}}\underbrace{\left(  -1\right)  ^{\ell\left(  \gamma\right)  }\eta
_{\operatorname*{rev}\gamma}^{\left(  q\right)  }\otimes\left(  -1\right)
^{\ell\left(  \beta\right)  }\eta_{\operatorname*{rev}\beta}^{\left(
q\right)  }}_{=\left(  -1\right)  ^{\ell\left(  \beta\right)  }\left(
-1\right)  ^{\ell\left(  \gamma\right)  }\eta_{\operatorname*{rev}\gamma
}^{\left(  q\right)  }\otimes\eta_{\operatorname*{rev}\beta}^{\left(
q\right)  }}\\
&  =\sum_{\substack{\beta,\gamma\in\operatorname*{Comp};\\\alpha=\beta\gamma
}}\underbrace{\left(  -1\right)  ^{\ell\left(  \beta\right)  }\left(
-1\right)  ^{\ell\left(  \gamma\right)  }}_{\substack{=\left(  -1\right)
^{\ell\left(  \beta\gamma\right)  }\\\text{(by (\ref{pf.lem.Delta-eta.rev.-1l}%
))}}}\eta_{\operatorname*{rev}\gamma}^{\left(  q\right)  }\otimes
\eta_{\operatorname*{rev}\beta}^{\left(  q\right)  }\\
&  =\sum_{\substack{\beta,\gamma\in\operatorname*{Comp};\\\alpha=\beta\gamma
}}\underbrace{\left(  -1\right)  ^{\ell\left(  \beta\gamma\right)  }%
}_{\substack{=\left(  -1\right)  ^{\ell\left(  \alpha\right)  }\\\text{(since
}\beta\gamma=\alpha\text{)}}}\eta_{\operatorname*{rev}\gamma}^{\left(
q\right)  }\otimes\eta_{\operatorname*{rev}\beta}^{\left(  q\right)  }\\
&  =\sum_{\substack{\beta,\gamma\in\operatorname*{Comp};\\\alpha=\beta\gamma
}}\left(  -1\right)  ^{\ell\left(  \alpha\right)  }\eta_{\operatorname*{rev}%
\gamma}^{\left(  q\right)  }\otimes\eta_{\operatorname*{rev}\beta}^{\left(
q\right)  }=\left(  -1\right)  ^{\ell\left(  \alpha\right)  }\sum
_{\substack{\beta,\gamma\in\operatorname*{Comp};\\\alpha=\beta\gamma}%
}\eta_{\operatorname*{rev}\gamma}^{\left(  q\right)  }\otimes\eta
_{\operatorname*{rev}\beta}^{\left(  q\right)  }.
\end{align*}
We can divide this equality by $\left(  -1\right)  ^{\ell\left(
\alpha\right)  }$ (since $\left(  -1\right)  ^{\ell\left(  \alpha\right)  }%
\in\left\{  1,-1\right\}  $ is clearly invertible), and thus obtain%
\[
\Delta\left(  \eta_{\operatorname*{rev}\alpha}^{\left(  q\right)  }\right)
=\sum_{\substack{\beta,\gamma\in\operatorname*{Comp};\\\alpha=\beta\gamma
}}\eta_{\operatorname*{rev}\gamma}^{\left(  q\right)  }\otimes\eta
_{\operatorname*{rev}\beta}^{\left(  q\right)  }.
\]
This proves Lemma \ref{lem.Delta-eta.rev}.
\end{proof}

\begin{proof}
[Second proof of Theorem \ref{thm.Delta-eta}.]It is easy to show (see, e.g.,
\cite[Proposition 5.3]{comps}) that
\begin{equation}
\operatorname*{rev}\left(  \beta\gamma\right)  =\left(  \operatorname*{rev}%
\gamma\right)  \left(  \operatorname*{rev}\beta\right)
\label{pf.thm.Delta-eta.revbg}%
\end{equation}
for any two compositions $\beta,\gamma\in\operatorname*{Comp}$.

Furthermore, the map
\begin{align*}
\operatorname*{Comp}  &  \rightarrow\operatorname*{Comp},\\
\delta &  \mapsto\operatorname*{rev}\delta
\end{align*}
is a bijection (this is essentially trivial, but see \cite[Corollary
3.5]{comps} for a detailed proof). Moreover, $\operatorname*{rev}\left(
\operatorname*{rev}\alpha\right)  =\alpha$ (by \cite[Proposition 3.4]{comps}).

Applying Lemma \ref{lem.Delta-eta.rev} to $\operatorname*{rev}\alpha$ instead
of $\alpha$, we find%
\begin{equation}
\Delta\left(  \eta_{\operatorname*{rev}\left(  \operatorname*{rev}%
\alpha\right)  }^{\left(  q\right)  }\right)  =\sum_{\substack{\beta,\gamma
\in\operatorname*{Comp};\\\operatorname*{rev}\alpha=\beta\gamma}%
}\eta_{\operatorname*{rev}\gamma}^{\left(  q\right)  }\otimes\eta
_{\operatorname*{rev}\beta}^{\left(  q\right)  }. \label{pf.thm.Delta-eta.4}%
\end{equation}
However, the map%
\begin{align*}
\operatorname*{Comp}  &  \rightarrow\operatorname*{Comp},\\
\delta &  \mapsto\operatorname*{rev}\delta
\end{align*}
is injective (since it is a bijection). In other words, for any two
compositions $\mu$ and $\nu$, we have the logical equivalence%
\begin{equation}
\left(  \mu=\nu\right)  \ \Longleftrightarrow\ \left(  \operatorname*{rev}%
\mu=\operatorname*{rev}\nu\right)  . \label{pf.thm.Delta-eta.revinj}%
\end{equation}

Now, for any two compositions $\beta,\gamma\in\operatorname*{Comp}$, we have
the following chain of logical equivalences:%
\begin{align*}
&  \ \left(  \operatorname*{rev}\alpha=\beta\gamma\right) \\
\Longleftrightarrow &  \ \left(  \underbrace{\operatorname*{rev}\left(
\operatorname*{rev}\alpha\right)  }_{=\alpha}=\underbrace{\operatorname*{rev}%
\left(  \beta\gamma\right)  }_{\substack{=\left(  \operatorname*{rev}%
\gamma\right)  \left(  \operatorname*{rev}\beta\right)  \\\text{(by
(\ref{pf.thm.Delta-eta.revbg}))}}}\right)  \ \ \ \ \ \ \ \ \ \ \left(
\text{by (\ref{pf.thm.Delta-eta.revinj}), applied to }\mu=\operatorname*{rev}%
\alpha\text{ and }\nu=\beta\gamma\right) \\
\Longleftrightarrow &  \ \left(  \alpha=\left(  \operatorname*{rev}%
\gamma\right)  \left(  \operatorname*{rev}\beta\right)  \right)  .
\end{align*}
Thus, the condition \textquotedblleft$\operatorname*{rev}\alpha=\beta\gamma
$\textquotedblright\ under the summation sign in (\ref{pf.thm.Delta-eta.4})
can be replaced by the equivalent condition \textquotedblleft$\alpha=\left(
\operatorname*{rev}\gamma\right)  \left(  \operatorname*{rev}\beta\right)
$\textquotedblright. Hence, (\ref{pf.thm.Delta-eta.4}) rewrites as
\[
\Delta\left(  \eta_{\operatorname*{rev}\left(  \operatorname*{rev}%
\alpha\right)  }^{\left(  q\right)  }\right)  =\sum_{\substack{\beta,\gamma
\in\operatorname*{Comp};\\\alpha=\left(  \operatorname*{rev}\gamma\right)
\left(  \operatorname*{rev}\beta\right)  }}\eta_{\operatorname*{rev}\gamma
}^{\left(  q\right)  }\otimes\eta_{\operatorname*{rev}\beta}^{\left(
q\right)  }.
\]
We can rewrite this further as%
\begin{equation}
\Delta\left(  \eta_{\alpha}^{\left(  q\right)  }\right)  =\sum
_{\substack{\beta,\gamma\in\operatorname*{Comp};\\\alpha=\left(
\operatorname*{rev}\gamma\right)  \left(  \operatorname*{rev}\beta\right)
}}\eta_{\operatorname*{rev}\gamma}^{\left(  q\right)  }\otimes\eta
_{\operatorname*{rev}\beta}^{\left(  q\right)  } \label{pf.thm.Delta-eta.9}%
\end{equation}
(since $\operatorname*{rev}\left(  \operatorname*{rev}\alpha\right)  =\alpha$).

Now,%
\begin{align*}
&  \underbrace{\sum_{\substack{\beta,\gamma\in\operatorname*{Comp}%
;\\\alpha=\left(  \operatorname*{rev}\gamma\right)  \left(
\operatorname*{rev}\beta\right)  }}}_{=\sum_{\beta\in\operatorname*{Comp}%
}\ \ \sum_{\substack{\gamma\in\operatorname*{Comp};\\\alpha=\left(
\operatorname*{rev}\gamma\right)  \left(  \operatorname*{rev}\beta\right)  }%
}}\eta_{\operatorname*{rev}\gamma}^{\left(  q\right)  }\otimes\eta
_{\operatorname*{rev}\beta}^{\left(  q\right)  }\\
&  =\sum_{\beta\in\operatorname*{Comp}}\ \ \sum_{\substack{\gamma
\in\operatorname*{Comp};\\\alpha=\left(  \operatorname*{rev}\gamma\right)
\left(  \operatorname*{rev}\beta\right)  }}\eta_{\operatorname*{rev}\gamma
}^{\left(  q\right)  }\otimes\eta_{\operatorname*{rev}\beta}^{\left(
q\right)  }\\
&  =\sum_{\nu\in\operatorname*{Comp}}\ \ \sum_{\substack{\gamma\in
\operatorname*{Comp};\\\alpha=\left(  \operatorname*{rev}\gamma\right)  \nu
}}\eta_{\operatorname*{rev}\gamma}^{\left(  q\right)  }\otimes\eta_{\nu
}^{\left(  q\right)  }\ \ \ \ \ \ \ \ \ \ \left(
\begin{array}
[c]{c}%
\text{here, we have substituted }\nu\text{ for }\operatorname*{rev}\beta\\
\text{in the outer sum, since the}\\
\text{map }\operatorname*{Comp}\rightarrow\operatorname*{Comp},\ \delta
\mapsto\operatorname*{rev}\delta\\
\text{is a bijection}%
\end{array}
\right) \\
&  =\sum_{\nu\in\operatorname*{Comp}}\ \ \sum_{\substack{\mu\in
\operatorname*{Comp};\\\alpha=\mu\nu}}\eta_{\mu}^{\left(  q\right)  }%
\otimes\eta_{\nu}^{\left(  q\right)  }\ \ \ \ \ \ \ \ \ \ \left(
\begin{array}
[c]{c}%
\text{here, we have substituted }\mu\text{ for }\operatorname*{rev}\gamma\\
\text{in the inner sum, since the}\\
\text{map }\operatorname*{Comp}\rightarrow\operatorname*{Comp},\ \delta
\mapsto\operatorname*{rev}\delta\\
\text{is a bijection}%
\end{array}
\right) \\
&  =\underbrace{\sum_{\gamma\in\operatorname*{Comp}}\ \ \sum_{\substack{\beta
\in\operatorname*{Comp};\\\alpha=\beta\gamma}}}_{=\sum_{\substack{\beta
,\gamma\in\operatorname*{Comp};\\\alpha=\beta\gamma}}}\eta_{\beta}^{\left(
q\right)  }\otimes\eta_{\gamma}^{\left(  q\right)  }%
\ \ \ \ \ \ \ \ \ \ \left(
\begin{array}
[c]{c}%
\text{here, we have renamed the}\\
\text{summation indices }\nu\text{ and }\mu\text{ as }\gamma\text{ and }\beta
\end{array}
\right) \\
&  =\sum_{\substack{\beta,\gamma\in\operatorname*{Comp};\\\alpha=\beta\gamma
}}\eta_{\beta}^{\left(  q\right)  }\otimes\eta_{\gamma}^{\left(  q\right)  }.
\end{align*}
In view of this, we can rewrite (\ref{pf.thm.Delta-eta.9}) as%
\[
\Delta\left(  \eta_{\alpha}^{\left(  q\right)  }\right)  =\sum
_{\substack{\beta,\gamma\in\operatorname*{Comp};\\\alpha=\beta\gamma}%
}\eta_{\beta}^{\left(  q\right)  }\otimes\eta_{\gamma}^{\left(  q\right)  }.
\]
This proves Theorem \ref{thm.Delta-eta} once again.
\end{proof}
\end{verlong}

\section{\label{sec.dual}The dual eta basis of $\operatorname*{NSym}$}

\subsection{$\operatorname*{NSym}$ and the duality pairing}

Let $\operatorname*{NSym}$ denote the free $\mathbf{k}$-algebra with
generators $H_{1},H_{2},H_{3},\ldots$ (that is, the tensor algebra of the free
$\mathbf{k}$-module with basis $\left(  H_{1},H_{2},H_{3},\ldots\right)  $).
This $\mathbf{k}$-algebra $\operatorname*{NSym}$ is known as the \emph{ring of
noncommutative symmetric functions} over $\mathbf{k}$. We refer to
\cite[\S 5.4]{GriRei}, \cite{GKLLRT94} and \cite[\S 6.1]{Meliot17} for more
about this $\mathbf{k}$-algebra\footnote{We note some notational differences:
\par
\begin{itemize}
\item What we call $H_{\alpha}$ is called $S_{\alpha}$ in \cite{GKLLRT94} and
in \cite{Meliot17}.
\par
\item The algebra $\operatorname*{NSym}$ is denoted by $\operatorname*{NCSym}$
in \cite{Meliot17} (unfortunately, since $\operatorname*{NCSym}$ also has a
different meaning).
\end{itemize}
}; we will only need a few basic properties.

We set $H_{0}:=1\in\operatorname*{NSym}$. Thus, an element $H_{n}$ of
$\operatorname*{NSym}$ is defined for each $n\in\mathbb{N}$. For any
composition $\alpha=\left(  \alpha_{1},\alpha_{2},\ldots,\alpha_{k}\right)
\in\operatorname*{Comp}$, we set%
\[
H_{\alpha}:=H_{\alpha_{1}}H_{\alpha_{2}}\cdots H_{\alpha_{k}}\in
\operatorname*{NSym}.
\]
The family $\left(  H_{\alpha}\right)  _{\alpha\in\operatorname*{Comp}}$ is
then a basis of the $\mathbf{k}$-module $\operatorname*{NSym}$. (Note that
$H_{\left(  n\right)  }=H_{n}$ for each $n>0$.)

The $\mathbf{k}$-algebra $\operatorname*{NSym}$ is graded, with each generator
$H_{n}$ being homogeneous of degree $n$ (and thus each basis element
$H_{\alpha}$ being homogeneous of degree $\left\vert \alpha\right\vert $). It
becomes a connected graded $\mathbf{k}$-bialgebra if we define its coproduct
$\Delta:\operatorname*{NSym}\rightarrow\operatorname*{NSym}\otimes
\operatorname*{NSym}$ and its counit $\varepsilon:\operatorname*{NSym}%
\rightarrow\mathbf{k}$ as follows:

\begin{itemize}
\item The coproduct $\Delta:\operatorname*{NSym}\rightarrow
\operatorname*{NSym}\otimes\operatorname*{NSym}$ is the $\mathbf{k}$-algebra
homomorphism that sends each generator $H_{n}$ to $\sum_{i=0}^{n}H_{i}\otimes
H_{n-i}$.

\item The counit $\varepsilon:\operatorname*{NSym}\rightarrow\mathbf{k}$ is
the $\mathbf{k}$-algebra homomorphism that sends each generator $H_{n}$ (with
$n>0$) to $0$.
\end{itemize}

Therefore, $\operatorname*{NSym}$ becomes a Hopf algebra (since any connected
graded $\mathbf{k}$-bialgebra is a Hopf algebra). Its antipode $S$ is
described in \cite[(5.4.12)]{GriRei}.

Most importantly to us, the Hopf algebra $\operatorname*{NSym}$ is isomorphic
to the graded dual of $\operatorname*{QSym}$. Specifically, we can define a
$\mathbf{k}$-bilinear form $\left\langle \cdot,\cdot\right\rangle
:\operatorname*{NSym}\times\operatorname*{QSym}\rightarrow\mathbf{k}$ by
requiring that%
\begin{equation}
\left\langle H_{\alpha},M_{\beta}\right\rangle =\left[  \alpha=\beta\right]
\label{eq.dual.HM}%
\end{equation}
for all $\alpha,\beta\in\operatorname*{Comp}$ (where we are using Convention
\ref{conv.iverson})\footnote{This bilinear form $\left\langle \cdot
,\cdot\right\rangle $ is denoted by $\left(  \cdot,\cdot\right)  $ in
\cite[\S 5.4]{GriRei}.}. It can be seen that this $\mathbf{k}$-bilinear form
produces a canonical isomorphism%
\begin{align*}
\operatorname*{NSym}  &  \rightarrow\operatorname*{QSym}\nolimits^{o},\\
f  &  \mapsto\left\langle f,\cdot\right\rangle
\end{align*}
of graded Hopf algebras, where $\operatorname*{QSym}\nolimits^{o}$ is the
graded dual of the Hopf algebra $\operatorname*{QSym}$. Thus, we can identify
$\operatorname*{NSym}$ with the graded dual of the Hopf algebra
$\operatorname*{QSym}$. (In \cite[\S 5.4]{GriRei}, this is used as a
definition of $\operatorname*{NSym}$, while the properties that we used to
define $\operatorname*{NSym}$ above are stated as \cite[Theorem 5.4.2]{GriRei}.)

\subsection{The dual eta basis}

We shall now construct a basis of $\operatorname*{NSym}$ that is dual to the
basis $\left(  \eta_{\alpha}^{\left(  q\right)  }\right)  _{\alpha
\in\operatorname*{Comp}}$ of $\operatorname*{QSym}$. This requires the
assumption that $r$ is invertible (since this assumption ensures that $\left(
\eta_{\alpha}^{\left(  q\right)  }\right)  _{\alpha\in\operatorname*{Comp}}$
is a basis of $\operatorname*{QSym}$ in the first place\footnote{by Theorem
\ref{thm.eta.basis}}). Thus, we make the following convention:

\begin{convention}
\label{conv.r-ible}For the rest of Section \ref{sec.dual}, we assume that $r$
is invertible in $\mathbf{k}$.
\end{convention}

\begin{definition}
\label{def.etastar}For each $n\in\mathbb{N}$ and each composition $\alpha$ of
$n$, we define an element
\[
\eta_{\alpha}^{\ast\left(  q\right)  }:=\sum_{\substack{\beta\in
\operatorname*{Comp}\nolimits_{n};\\D\left(  \alpha\right)  \subseteq D\left(
\beta\right)  }}\dfrac{1}{r^{\ell\left(  \beta\right)  }}\left(  -1\right)
^{\ell\left(  \beta\right)  -\ell\left(  \alpha\right)  }H_{\beta}%
\in\operatorname*{NSym}.
\]

\end{definition}

\begin{example}
We have
\begin{align*}
\eta_{\left(  {}\right)  }^{\ast\left(  q\right)  }  &  =H_{\left(  {}\right)
}=1_{\operatorname*{NSym}};\\
\eta_{\left(  1\right)  }^{\ast\left(  q\right)  }  &  =\dfrac{1}{r}H_{\left(
1\right)  };\\
\eta_{\left(  2\right)  }^{\ast\left(  q\right)  }  &  =\dfrac{1}{r}H_{\left(
2\right)  }-\dfrac{1}{r^{2}}H_{\left(  1,1\right)  };\\
\eta_{\left(  1,1\right)  }^{\ast\left(  q\right)  }  &  =\dfrac{1}{r^{2}%
}H_{\left(  1,1\right)  }.
\end{align*}

\end{example}

We now claim the following:

\begin{proposition}
\label{prop.etastar-dual-basis}\ \ 

\begin{enumerate}
\item[\textbf{(a)}] The family $\left(  \eta_{\alpha}^{\ast\left(  q\right)
}\right)  _{\alpha\in\operatorname*{Comp}}$ is the basis of
$\operatorname*{NSym}$ dual to the basis $\left(  \eta_{\alpha}^{\left(
q\right)  }\right)  _{\alpha\in\operatorname*{Comp}}$ of $\operatorname*{QSym}%
$ with respect to the bilinear form $\left\langle \cdot,\cdot\right\rangle $.

Here, the notion of a \textquotedblleft dual basis\textquotedblright\ should
be understood in the graded sense, as explained in \cite[\S 1.6]{GriRei}.
Concretely, our claim is saying that $\left(  \eta_{\alpha}^{\ast\left(
q\right)  }\right)  _{\alpha\in\operatorname*{Comp}}$ is a graded basis of
$\operatorname*{NSym}$ and satisfies
\begin{equation}
\left\langle \eta_{\alpha}^{\ast\left(  q\right)  },\eta_{\beta}^{\left(
q\right)  }\right\rangle =\left[  \alpha=\beta\right]
\label{eq.prop.etastar-dual-basis.duality}%
\end{equation}
for all $\alpha,\beta\in\operatorname*{Comp}$.

\item[\textbf{(b)}] Let $n\in\mathbb{N}$. Consider the $n$-th graded
components $\operatorname*{QSym}\nolimits_{n}$ and $\operatorname*{NSym}%
\nolimits_{n}$ of the graded $\mathbf{k}$-modules $\operatorname*{QSym}$ and
$\operatorname*{NSym}$. Then, the family $\left(  \eta_{\alpha}^{\ast\left(
q\right)  }\right)  _{\alpha\in\operatorname*{Comp}\nolimits_{n}}$ is the
basis of $\operatorname*{NSym}\nolimits_{n}$ dual to the basis $\left(
\eta_{\alpha}^{\left(  q\right)  }\right)  _{\alpha\in\operatorname*{Comp}%
\nolimits_{n}}$ of $\operatorname*{QSym}\nolimits_{n}$ with respect to the
bilinear form $\left\langle \cdot,\cdot\right\rangle $.
\end{enumerate}
\end{proposition}

To prove this, we will use a general fact about dual bases of $\mathbf{k}%
$-modules that should be known from linear algebra, but is hard to find
explicitly in the literature. This fact is a close relative of the classical
linear-algebraic result that the transpose of a matrix represents the adjoint
of its linear map:

\begin{lemma}
\label{lem.dual-bases-crit}Let $F$ and $U$ be two free $\mathbf{k}$-modules,
and let $\left\langle \cdot,\cdot\right\rangle :F\times U\rightarrow
\mathbf{k}$ be a $\mathbf{k}$-bilinear form. Let $A$ be a finite set. Let
$\left(  f_{\alpha}\right)  _{\alpha\in A}$ be a basis of the $\mathbf{k}%
$-module $F$, and let $\left(  g_{\alpha}\right)  _{\alpha\in A}$ be a further
family of elements of $F$. Let $\left(  u_{\alpha}\right)  _{\alpha\in A}$ and
$\left(  v_{\alpha}\right)  _{\alpha\in A}$ be two bases of the $\mathbf{k}%
$-module $U$.

Assume that the basis $\left(  u_{\alpha}\right)  _{\alpha\in A}$ of $U$ is
dual to the basis $\left(  f_{\alpha}\right)  _{\alpha\in A}$ of $F$; in other
words, assume that%
\begin{equation}
\left\langle f_{\beta},u_{\alpha}\right\rangle =\left[  \beta=\alpha\right]
\label{eq.lem.dual-bases-crit.dual}%
\end{equation}
for all $\beta,\alpha\in A$.

Furthermore, let $c_{\alpha,\beta}$ be an element of $\mathbf{k}$ for each
pair $\left(  \alpha,\beta\right)  \in A\times A$. Assume that%
\begin{equation}
u_{\beta}=\sum_{\alpha\in A}c_{\alpha,\beta}v_{\alpha}%
\ \ \ \ \ \ \ \ \ \ \text{for each }\beta\in A.
\label{eq.lem.dual-bases-crit.u-v}%
\end{equation}
Assume furthermore that%
\begin{equation}
g_{\alpha}=\sum_{\beta\in A}c_{\alpha,\beta}f_{\beta}%
\ \ \ \ \ \ \ \ \ \ \text{for each }\alpha\in A.
\label{eq.lem.dual-bases-crit.g-f}%
\end{equation}
Then, the families $\left(  v_{\alpha}\right)  _{\alpha\in A}$ and $\left(
g_{\alpha}\right)  _{\alpha\in A}$ are mutually dual bases of the $\mathbf{k}%
$-modules $U$ and $F$, respectively.
\end{lemma}

\begin{proof}
[Proof of Lemma \ref{lem.dual-bases-crit}.]We shall prove several claims:

\begin{statement}
\textit{Claim 1:} Each $w\in F$ satisfies%
\begin{equation}
w=\sum_{\alpha\in A}\left\langle w,u_{\alpha}\right\rangle f_{\alpha}.
\label{pf.lem.dual-bases-crit.1}%
\end{equation}

\end{statement}

\begin{proof}
[Proof of Claim 1.]Let $w\in F$. The equality (\ref{pf.lem.dual-bases-crit.1})
is a well-known property of dual bases (since the basis $\left(  u_{\alpha
}\right)  _{\alpha\in A}$ of $U$ is dual to the basis $\left(  f_{\alpha
}\right)  _{\alpha\in A}$ of $F$), but let us recall its proof for the sake of completeness.

The family $\left(  f_{\alpha}\right)  _{\alpha\in A}$ is a basis of $F$, thus
spans $F$. Hence, $w\in F$ can be written as a $\mathbf{k}$-linear combination
of this family. In other words, there exists a family $\left(  d_{\alpha
}\right)  _{\alpha\in A}\in\mathbf{k}^{A}$ of coefficients such that%
\begin{equation}
w=\sum_{\alpha\in A}d_{\alpha}f_{\alpha}.
\label{pf.pf.lem.dual-bases-crit.1.1}%
\end{equation}

\begin{vershort}
\noindent Consider this family $\left(  d_{\alpha}\right)  _{\alpha\in A}$.
Then,
\[
w=\sum_{\alpha\in A}d_{\alpha}f_{\alpha}=\sum_{\beta\in A}d_{\beta}f_{\beta}%
\]
(here, we have renamed the summation index $\alpha$ as $\beta$). Hence, each
$\alpha\in A$ satisfies%
\begin{equation}
\left\langle w,u_{\alpha}\right\rangle =\left\langle \sum_{\beta\in A}%
d_{\beta}f_{\beta},u_{\alpha}\right\rangle =\sum_{\beta\in A}d_{\beta
}\underbrace{\left\langle f_{\beta},u_{\alpha}\right\rangle }%
_{\substack{=\left[  \beta=\alpha\right]  \\\text{(by
(\ref{eq.lem.dual-bases-crit.dual}))}}}=\sum_{\beta\in A}d_{\beta}\left[
\beta=\alpha\right]  =d_{\alpha}\nonumber
\end{equation}
(since the only nonzero addend in the sum $\sum_{\beta\in A}d_{\beta}\left[
\beta=\alpha\right]  $ is the addend for $\beta=\alpha$, and this latter
addend equals $d_{\alpha}$). Hence, we can rewrite
(\ref{pf.pf.lem.dual-bases-crit.1.1}) as%
\[
w=\sum_{\alpha\in A}\left\langle w,u_{\alpha}\right\rangle f_{\alpha}.
\]

\end{vershort}

\begin{verlong}
\noindent Consider this family $\left(  d_{\alpha}\right)  _{\alpha\in A}$.
Then,
\[
w=\sum_{\alpha\in A}d_{\alpha}f_{\alpha}=\sum_{\beta\in A}d_{\beta}f_{\beta}%
\]
(here, we have renamed the summation index $\alpha$ as $\beta$). Hence, each
$\alpha\in A$ satisfies%
\begin{align}
\left\langle w,u_{\alpha}\right\rangle  &  =\left\langle \sum_{\beta\in
A}d_{\beta}f_{\beta},u_{\alpha}\right\rangle =\sum_{\beta\in A}d_{\beta
}\underbrace{\left\langle f_{\beta},u_{\alpha}\right\rangle }%
_{\substack{=\left[  \beta=\alpha\right]  \\\text{(by
(\ref{eq.lem.dual-bases-crit.dual}))}}}\ \ \ \ \ \ \ \ \ \ \left(
\begin{array}
[c]{c}%
\text{since the form }\left\langle \cdot,\cdot\right\rangle \\
\text{is }\mathbf{k}\text{-bilinear}%
\end{array}
\right) \nonumber\\
&  =\sum_{\beta\in A}d_{\beta}\left[  \beta=\alpha\right]  =d_{\alpha
}\underbrace{\left[  \alpha=\alpha\right]  }_{\substack{=1\\\text{(since
}\alpha=\alpha\text{)}}}+\sum_{\substack{\beta\in A;\\\beta\neq\alpha
}}d_{\beta}\underbrace{\left[  \beta=\alpha\right]  }%
_{\substack{=0\\\text{(since }\beta\neq\alpha\text{)}}}\nonumber\\
&  \ \ \ \ \ \ \ \ \ \ \ \ \ \ \ \ \ \ \ \ \left(
\begin{array}
[c]{c}%
\text{here, we have split the}\\
\text{addend for }\beta=\alpha\text{ from the sum}%
\end{array}
\right) \nonumber\\
&  =d_{\alpha}+\underbrace{\sum_{\substack{\beta\in A;\\\beta\neq\alpha
}}d_{\beta}0}_{=0}=d_{\alpha}. \label{pf.pf.lem.dual-bases-crit.1.2}%
\end{align}

Now, (\ref{pf.pf.lem.dual-bases-crit.1.1}) becomes%
\[
w=\sum_{\alpha\in A}\underbrace{d_{\alpha}}_{\substack{=\left\langle
w,u_{\alpha}\right\rangle \\\text{(by (\ref{pf.pf.lem.dual-bases-crit.1.2}))}%
}}f_{\alpha}=\sum_{\alpha\in A}\left\langle w,u_{\alpha}\right\rangle
f_{\alpha}.
\]

\end{verlong}

\noindent This proves (\ref{pf.lem.dual-bases-crit.1}). Thus, Claim 1 is proved.
\end{proof}

\begin{statement}
\textit{Claim 2:} Each $w\in F$ satisfies%
\begin{equation}
w=\sum_{\alpha\in A}\left\langle w,v_{\alpha}\right\rangle g_{\alpha}.
\label{pf.lem.dual-bases-crit.3}%
\end{equation}

\end{statement}

\begin{vershort}

\begin{proof}
[Proof of Claim 2.]Let $w\in F$. Then,
\begin{align*}
\sum_{\alpha\in A}\left\langle w,v_{\alpha}\right\rangle \underbrace{g_{\alpha
}}_{\substack{=\sum_{\beta\in A}c_{\alpha,\beta}f_{\beta}\\\text{(by
(\ref{eq.lem.dual-bases-crit.g-f}))}}}  &  =\sum_{\alpha\in A}\left\langle
w,v_{\alpha}\right\rangle \sum_{\beta\in A}c_{\alpha,\beta}f_{\beta}%
=\sum_{\beta\in A}\ \ \underbrace{\sum_{\alpha\in A}c_{\alpha,\beta
}\left\langle w,v_{\alpha}\right\rangle }_{=\left\langle w,\sum_{\alpha\in
A}c_{\alpha,\beta}v_{\alpha}\right\rangle }f_{\beta}\\
&  =\sum_{\beta\in A}\left\langle w,\underbrace{\sum_{\alpha\in A}%
c_{\alpha,\beta}v_{\alpha}}_{\substack{=u_{\beta}\\\text{(by
(\ref{eq.lem.dual-bases-crit.u-v}))}}}\right\rangle f_{\beta}=\sum_{\beta\in
A}\left\langle w,u_{\beta}\right\rangle f_{\beta}\\
&  =\sum_{\alpha\in A}\left\langle w,u_{\alpha}\right\rangle f_{\alpha
}=w\ \ \ \ \ \ \ \ \ \ \left(  \text{by (\ref{pf.lem.dual-bases-crit.1}%
)}\right)  .
\end{align*}
This proves Claim 2.
\end{proof}
\end{vershort}

\begin{verlong}

\begin{proof}
[Proof of Claim 2.]Let $w\in F$. Then,
\begin{align*}
\sum_{\alpha\in A}\left\langle w,v_{\alpha}\right\rangle \underbrace{g_{\alpha
}}_{\substack{=\sum_{\beta\in A}c_{\alpha,\beta}f_{\beta}\\\text{(by
(\ref{eq.lem.dual-bases-crit.g-f}))}}}  &  =\sum_{\alpha\in A}\left\langle
w,v_{\alpha}\right\rangle \sum_{\beta\in A}c_{\alpha,\beta}f_{\beta
}=\underbrace{\sum_{\alpha\in A}\ \ \sum_{\beta\in A}}_{=\sum_{\beta\in
A}\ \ \sum_{\alpha\in A}}\underbrace{\left\langle w,v_{\alpha}\right\rangle
c_{\alpha,\beta}}_{=c_{\alpha,\beta}\left\langle w,v_{\alpha}\right\rangle
}f_{\beta}\\
&  =\sum_{\beta\in A}\ \ \sum_{\alpha\in A}c_{\alpha,\beta}\left\langle
w,v_{\alpha}\right\rangle f_{\beta}=\sum_{\beta\in A}\underbrace{\left(
\sum_{\alpha\in A}c_{\alpha,\beta}\left\langle w,v_{\alpha}\right\rangle
\right)  }_{\substack{=\left\langle w,\sum_{\alpha\in A}c_{\alpha,\beta
}v_{\alpha}\right\rangle \\\text{(since the form }\left\langle \cdot
,\cdot\right\rangle \\\text{is }\mathbf{k}\text{-bilinear)}}}f_{\beta}\\
&  =\sum_{\beta\in A}\left\langle w,\underbrace{\sum_{\alpha\in A}%
c_{\alpha,\beta}v_{\alpha}}_{\substack{=u_{\beta}\\\text{(by
(\ref{eq.lem.dual-bases-crit.u-v}))}}}\right\rangle f_{\beta}=\sum_{\beta\in
A}\left\langle w,u_{\beta}\right\rangle f_{\beta}\\
&  =\sum_{\alpha\in A}\left\langle w,u_{\alpha}\right\rangle f_{\alpha
}\ \ \ \ \ \ \ \ \ \ \left(
\begin{array}
[c]{c}%
\text{here, we have renamed}\\
\text{the summation index }\beta\text{ as }\alpha
\end{array}
\right) \\
&  =w\ \ \ \ \ \ \ \ \ \ \left(  \text{by (\ref{pf.lem.dual-bases-crit.1}%
)}\right)  ,
\end{align*}
so that $w=\sum_{\alpha\in A}\left\langle w,v_{\alpha}\right\rangle g_{\alpha
}$. This proves Claim 2.
\end{proof}
\end{verlong}

\begin{statement}
\textit{Claim 3:} The family $\left(  g_{\alpha}\right)  _{\alpha\in A}$ spans
the $\mathbf{k}$-module $F$.
\end{statement}

\begin{proof}
[Proof of Claim 3.]Let $w\in F$. Then, Claim 2 yields $w=\sum_{\alpha\in
A}\left\langle w,v_{\alpha}\right\rangle g_{\alpha}$. Hence, $w$ belongs to
the span of the family $\left(  g_{\alpha}\right)  _{\alpha\in A}$.

Forget that we fixed $w$. We thus have proved that each $w\in F$ belongs to
the span of the family $\left(  g_{\alpha}\right)  _{\alpha\in A}$. Hence,
this family $\left(  g_{\alpha}\right)  _{\alpha\in A}$ spans the $\mathbf{k}%
$-module $F$. This proves Claim 3.
\end{proof}

\begin{statement}
\textit{Claim 4:} The family $\left(  g_{\alpha}\right)  _{\alpha\in A}$ is
$\mathbf{k}$-linearly independent.
\end{statement}

\begin{proof}
[Proof of Claim 4.]Let $\left(  m_{\alpha}\right)  _{\alpha\in A}\in
\mathbf{k}^{A}$ be a family of scalars such that $\sum_{\alpha\in A}m_{\alpha
}g_{\alpha}=0$. We shall show that $m_{\alpha}=0$ for each $\alpha\in A$.

Recall that $\left(  v_{\alpha}\right)  _{\alpha\in A}$ is a basis of the
$\mathbf{k}$-module $U$. Hence, we can define a $\mathbf{k}$-linear map
$M:U\rightarrow\mathbf{k}$ by requiring that%
\[
M\left(  v_{\alpha}\right)  =m_{\alpha}\ \ \ \ \ \ \ \ \ \ \text{for each
}\alpha\in A
\]
(because a $\mathbf{k}$-linear map from a module can be uniquely defined by
specifying its values on a given basis of this module). Consider this map $M$.

\begin{vershort}
Now, for each $\beta\in A$, we can apply the map $M$ to both sides of
(\ref{eq.lem.dual-bases-crit.u-v}), and obtain%
\begin{align}
M\left(  u_{\beta}\right)   &  =M\left(  \sum_{\alpha\in A}c_{\alpha,\beta
}v_{\alpha}\right)  =\sum_{\alpha\in A}c_{\alpha,\beta}\underbrace{M\left(
v_{\alpha}\right)  }_{\substack{=m_{\alpha}\\\text{(by the definition of
}M\text{)}}}\nonumber\\
&  =\sum_{\alpha\in A}c_{\alpha,\beta}m_{\alpha}.
\label{pf.pf.lem.dual-bases-crit.5.short.3}%
\end{align}

From $\sum_{\alpha\in A}m_{\alpha}g_{\alpha}=0$, we obtain%
\begin{align*}
0  &  =\sum_{\alpha\in A}m_{\alpha}g_{\alpha}=\sum_{\alpha\in A}m_{\alpha}%
\sum_{\beta\in A}c_{\alpha,\beta}f_{\beta}\ \ \ \ \ \ \ \ \ \ \left(  \text{by
(\ref{eq.lem.dual-bases-crit.g-f})}\right) \\
&  =\sum_{\beta\in A}\underbrace{\left(  \sum_{\alpha\in A}c_{\alpha,\beta
}m_{\alpha}\right)  }_{\substack{=M\left(  u_{\beta}\right)  \\\text{(by
(\ref{pf.pf.lem.dual-bases-crit.5.short.3}))}}}f_{\beta}=\sum_{\beta\in
A}M\left(  u_{\beta}\right)  f_{\beta}.
\end{align*}

\end{vershort}

\begin{verlong}
Now, for each $\beta\in A$, we have%
\begin{align}
M\left(  u_{\beta}\right)   &  =M\left(  \sum_{\alpha\in A}c_{\alpha,\beta
}v_{\alpha}\right)  \ \ \ \ \ \ \ \ \ \ \left(  \text{by
(\ref{eq.lem.dual-bases-crit.u-v})}\right) \nonumber\\
&  =\sum_{\alpha\in A}c_{\alpha,\beta}\underbrace{M\left(  v_{\alpha}\right)
}_{\substack{=m_{\alpha}\\\text{(by the definition of }M\text{)}%
}}\ \ \ \ \ \ \ \ \ \ \left(  \text{since the map }M\text{ is }\mathbf{k}%
\text{-linear}\right) \nonumber\\
&  =\sum_{\alpha\in A}c_{\alpha,\beta}m_{\alpha}.
\label{pf.pf.lem.dual-bases-crit.5.3}%
\end{align}

From $\sum_{\alpha\in A}m_{\alpha}g_{\alpha}=0$, we obtain%
\begin{align*}
0  &  =\sum_{\alpha\in A}m_{\alpha}\underbrace{g_{\alpha}}_{\substack{=\sum
_{\beta\in A}c_{\alpha,\beta}f_{\beta}\\\text{(by
(\ref{eq.lem.dual-bases-crit.g-f}))}}}=\sum_{\alpha\in A}m_{\alpha}\sum
_{\beta\in A}c_{\alpha,\beta}f_{\beta}=\underbrace{\sum_{\alpha\in A}%
\ \ \sum_{\beta\in A}}_{=\sum_{\beta\in A}\ \ \sum_{\alpha\in A}%
}\underbrace{m_{\alpha}c_{\alpha,\beta}}_{=c_{\alpha,\beta}m_{\alpha}}%
f_{\beta}\\
&  =\sum_{\beta\in A}\ \ \sum_{\alpha\in A}c_{\alpha,\beta}m_{\alpha}f_{\beta
}=\sum_{\beta\in A}\underbrace{\left(  \sum_{\alpha\in A}c_{\alpha,\beta
}m_{\alpha}\right)  }_{\substack{=M\left(  u_{\beta}\right)  \\\text{(by
(\ref{pf.pf.lem.dual-bases-crit.5.3}))}}}f_{\beta}=\sum_{\beta\in A}M\left(
u_{\beta}\right)  f_{\beta}.
\end{align*}

\end{verlong}

In other words,
\begin{equation}
\sum_{\beta\in A}M\left(  u_{\beta}\right)  f_{\beta}=0.
\label{pf.pf.lem.dual-bases-crit.5.1}%
\end{equation}

However, the family $\left(  f_{\alpha}\right)  _{\alpha\in A}$ is a basis of
the $\mathbf{k}$-module $F$. In other words, the family $\left(  f_{\beta
}\right)  _{\beta\in A}$ is a basis of the $\mathbf{k}$-module $F$ (here, we
have renamed the index $\alpha$ as $\beta$). Hence, this family is
$\mathbf{k}$-linearly independent. In other words, if $\left(  n_{\beta
}\right)  _{\beta\in A}\in\mathbf{k}^{A}$ is a family of scalars satisfying
$\sum_{\beta\in A}n_{\beta}f_{\beta}=0$, then $n_{\beta}=0$ for each $\beta\in
A$.

We can apply this to $n_{\beta}=M\left(  u_{\beta}\right)  $ (since the family
$\left(  M\left(  u_{\beta}\right)  \right)  _{\beta\in A}\in\mathbf{k}^{A}$
satisfies $\sum_{\beta\in A}M\left(  u_{\beta}\right)  f_{\beta}=0$ (by
(\ref{pf.pf.lem.dual-bases-crit.5.1}))). Thus, we conclude that%
\begin{equation}
M\left(  u_{\beta}\right)  =0\ \ \ \ \ \ \ \ \ \ \text{for each }\beta\in A.
\label{pf.pf.lem.dual-bases-crit.5.2}%
\end{equation}

Recall that the family $\left(  u_{\alpha}\right)  _{\alpha\in A}$ is a basis
of the $\mathbf{k}$-module $U$. In other words, the family $\left(  u_{\beta
}\right)  _{\beta\in A}$ is a basis of the $\mathbf{k}$-module $U$ (here, we
have renamed the index $\alpha$ as $\beta$). Hence, this family spans $U$.

\begin{vershort}
Now, recall a general fact from linear algebra: If a $\mathbf{k}$-linear map
$f:V\rightarrow W$ sends a basis of its domain $V$ to $0$ (that is, if it
satisfies $f\left(  v\right)  =0$ for all vectors $v$ in a certain basis of
$V$), then $f=0$ as a map. The equality (\ref{pf.pf.lem.dual-bases-crit.5.2})
shows that the $\mathbf{k}$-linear map $M:U\rightarrow\mathbf{k}$ sends a
basis of $U$ to $0$ (namely, the basis $\left(  u_{\beta}\right)  _{\beta\in
A}$); therefore, the previous sentence shows that $M=0$ as a map.

Now, let $\alpha\in A$ be arbitrary. Then, $M\left(  v_{\alpha}\right)  =0$
(since we have just shown that $M=0$). However, the definition of $M$ yields
$M\left(  v_{\alpha}\right)  =m_{\alpha}$. Thus, $m_{\alpha}=M\left(
v_{\alpha}\right)  =0$.
\end{vershort}

\begin{verlong}
Now, let $\alpha\in A$ be arbitrary. Then, $v_{\alpha}\in U$. Thus,
$v_{\alpha}$ can be written as a $\mathbf{k}$-linear combination of the family
$\left(  u_{\beta}\right)  _{\beta\in A}$ (since this family $\left(
u_{\beta}\right)  _{\beta\in A}$ spans $U$). In other words, there exists a
family $\left(  p_{\beta}\right)  _{\beta\in A}\in\mathbf{k}^{A}$ of scalars
satisfying $v_{\alpha}=\sum_{\beta\in A}p_{\beta}u_{\beta}$. Consider this
family $\left(  p_{\beta}\right)  _{\beta\in A}$. Thus,%
\begin{align*}
M\left(  v_{\alpha}\right)   &  =M\left(  \sum_{\beta\in A}p_{\beta}u_{\beta
}\right)  \ \ \ \ \ \ \ \ \ \ \left(  \text{since }v_{\alpha}=\sum_{\beta\in
A}p_{\beta}u_{\beta}\right) \\
&  =\sum_{\beta\in A}p_{\beta}\underbrace{M\left(  u_{\beta}\right)
}_{\substack{=0\\\text{(by (\ref{pf.pf.lem.dual-bases-crit.5.2}))}%
}}\ \ \ \ \ \ \ \ \ \ \left(
\begin{array}
[c]{c}%
\text{since the map }M\\
\text{is }\mathbf{k}\text{-linear}%
\end{array}
\right) \\
&  =\sum_{\beta\in A}p_{\beta}0=0.
\end{align*}
Comparing this with%
\[
M\left(  v_{\alpha}\right)  =m_{\alpha}\ \ \ \ \ \ \ \ \ \ \left(  \text{by
the definition of }M\right)  ,
\]
we obtain $m_{\alpha}=0$.
\end{verlong}

Forget that we fixed $\alpha$. We thus have shown that $m_{\alpha}=0$ for each
$\alpha\in A$.

Forget that we fixed $\left(  m_{\alpha}\right)  _{\alpha\in A}$. We thus have
proved that if $\left(  m_{\alpha}\right)  _{\alpha\in A}\in\mathbf{k}^{A}$ is
a family of scalars such that $\sum_{\alpha\in A}m_{\alpha}g_{\alpha}=0$, then
$m_{\alpha}=0$ for each $\alpha\in A$. In other words, the family $\left(
g_{\alpha}\right)  _{\alpha\in A}$ is $\mathbf{k}$-linearly independent. This
proves Claim 4.
\end{proof}

Now, we know that the family $\left(  g_{\alpha}\right)  _{\alpha\in A}$ spans
the $\mathbf{k}$-module $F$ (by Claim 3) and is $\mathbf{k}$-linearly
independent (by Claim 4). In other words, this family is a basis of $F$.

Next, we claim the following:

\begin{statement}
\textit{Claim 5:} We have $\left\langle g_{\beta},v_{\alpha}\right\rangle
=\left[  \beta=\alpha\right]  $ for all $\alpha,\beta\in A$.
\end{statement}

\begin{proof}
[Proof of Claim 5.]Fix $\beta\in A$.

\begin{vershort}
The sum $\sum_{\alpha\in A}\left[  \beta=\alpha\right]  g_{\alpha}$ equals
$g_{\beta}$, since all its addends except for the $\alpha=\beta$ addend are
$0$. Thus,%
\begin{equation}
\sum_{\alpha\in A}\left[  \beta=\alpha\right]  g_{\alpha}=g_{\beta}%
=\sum_{\alpha\in A}\left\langle g_{\beta},v_{\alpha}\right\rangle g_{\alpha}
\label{pf.lem.dual-bases-crit.b.1}%
\end{equation}
(by Claim 2, applied to $w=g_{\beta}$).
\end{vershort}

\begin{verlong}
We have%
\begin{align*}
\sum_{\alpha\in A}\left[  \beta=\alpha\right]  g_{\alpha}  &
=\underbrace{\left[  \beta=\beta\right]  }_{\substack{=1\\\text{(since }%
\beta=\beta\text{)}}}g_{\beta}+\sum_{\substack{\alpha\in A;\\\alpha\neq\beta
}}\underbrace{\left[  \beta=\alpha\right]  }_{\substack{=0\\\text{(since
}\beta\neq\alpha\\\text{(because }\alpha\neq\beta\text{))}}}g_{\alpha}\\
&  \ \ \ \ \ \ \ \ \ \ \ \ \ \ \ \ \ \ \ \ \left(
\begin{array}
[c]{c}%
\text{here, we have split off the addend}\\
\text{for }\alpha=\beta\text{ from the sum}%
\end{array}
\right) \\
&  =g_{\beta}+\underbrace{\sum_{\substack{\alpha\in A;\\\alpha\neq\beta
}}0g_{\alpha}}_{=0}=g_{\beta}=\sum_{\alpha\in A}\left\langle g_{\beta
},v_{\alpha}\right\rangle g_{\alpha}%
\end{align*}
(by Claim 2, applied to $w=g_{\beta}$). Now,%
\begin{align*}
\sum_{\alpha\in A}\left(  \left\langle g_{\beta},v_{\alpha}\right\rangle
-\left[  \beta=\alpha\right]  \right)  g_{\alpha}  &  =\sum_{\alpha\in
A}\left\langle g_{\beta},v_{\alpha}\right\rangle g_{\alpha}-\underbrace{\sum
_{\alpha\in A}\left[  \beta=\alpha\right]  g_{\alpha}}_{=\sum_{\alpha\in
A}\left\langle g_{\beta},v_{\alpha}\right\rangle g_{\alpha}}\\
&  =\sum_{\alpha\in A}\left\langle g_{\beta},v_{\alpha}\right\rangle
g_{\alpha}-\sum_{\alpha\in A}\left\langle g_{\beta},v_{\alpha}\right\rangle
g_{\alpha}=0.
\end{align*}

\end{verlong}

\begin{vershort}
Now, Claim 4 says that the family $\left(  g_{\alpha}\right)  _{\alpha\in A}$
is $\mathbf{k}$-linearly independent. In other words, if $\left(  m_{\alpha
}\right)  _{\alpha\in A}\in\mathbf{k}^{A}$ and $\left(  n_{\alpha}\right)
_{\alpha\in A}\in\mathbf{k}^{A}$ are two families of scalars such that
$\sum_{\alpha\in A}m_{\alpha}g_{\alpha}=\sum_{\alpha\in A}n_{\alpha}g_{\alpha
}$, then $m_{\alpha}=n_{\alpha}$ for each $\alpha\in A$. Applying this to
$m_{\alpha}=\left\langle g_{\beta},v_{\alpha}\right\rangle $ and $n_{\alpha
}=\left[  \beta=\alpha\right]  $, we conclude that
\[
\left\langle g_{\beta},v_{\alpha}\right\rangle =\left[  \beta=\alpha\right]
\ \ \ \ \ \ \ \ \ \ \text{for each }\alpha\in A
\]
(since (\ref{pf.lem.dual-bases-crit.b.1}) yields $\sum_{\alpha\in
A}\left\langle g_{\beta},v_{\alpha}\right\rangle g_{\alpha}=\sum_{\alpha\in
A}\left[  \beta=\alpha\right]  g_{\alpha}$). This proves Claim 5. \qedhere

\end{vershort}

\begin{verlong}
Now, Claim 4 says that the family $\left(  g_{\alpha}\right)  _{\alpha\in A}$
is $\mathbf{k}$-linearly independent. In other words, if $\left(  m_{\alpha
}\right)  _{\alpha\in A}\in\mathbf{k}^{A}$ is a family of scalars such that
$\sum_{\alpha\in A}m_{\alpha}g_{\alpha}=0$, then $m_{\alpha}=0$ for each
$\alpha\in A$. Applying this to $m_{\alpha}=\left\langle g_{\beta},v_{\alpha
}\right\rangle -\left[  \beta=\alpha\right]  $, we conclude that
\[
\left\langle g_{\beta},v_{\alpha}\right\rangle -\left[  \beta=\alpha\right]
=0\ \ \ \ \ \ \ \ \ \ \text{for each }\alpha\in A
\]
(since $\sum_{\alpha\in A}\left(  \left\langle g_{\beta},v_{\alpha
}\right\rangle -\left[  \beta=\alpha\right]  \right)  g_{\alpha}=0$). In other
words,%
\[
\left\langle g_{\beta},v_{\alpha}\right\rangle =\left[  \beta=\alpha\right]
\ \ \ \ \ \ \ \ \ \ \text{for each }\alpha\in A.
\]
This proves Claim 5. \qedhere

\end{verlong}
\end{proof}

Now, recall that $\left(  g_{\alpha}\right)  _{\alpha\in A}$ is a basis of the
$\mathbf{k}$-module $F$, whereas $\left(  v_{\alpha}\right)  _{\alpha\in A}$
is a basis of the $\mathbf{k}$-module $U$. Furthermore, Claim 5 shows that the
bases $\left(  v_{\alpha}\right)  _{\alpha\in A}$ and $\left(  g_{\alpha
}\right)  _{\alpha\in A}$ of $U$ and $F$ are mutually dual. This completes the
proof of Lemma \ref{lem.dual-bases-crit}.
\end{proof}

We are now ready to prove Proposition \ref{prop.etastar-dual-basis}:

\begin{proof}
[Proof of Proposition \ref{prop.etastar-dual-basis}.]

\begin{verlong}
\textbf{(b)} We will use Convention \ref{conv.iverson}.

Recall that the families $\left(  H_{\alpha}\right)  _{\alpha\in
\operatorname*{Comp}\nolimits_{n}}$ and $\left(  M_{\alpha}\right)
_{\alpha\in\operatorname*{Comp}\nolimits_{n}}$ are mutually dual bases of
$\operatorname*{NSym}\nolimits_{n}$ and $\operatorname*{QSym}\nolimits_{n}$, respectively.
\end{verlong}

\begin{vershort}
\textbf{(b)} The scalar $r$ is invertible (by Convention \ref{conv.r-ible}).
Thus, we define an element%
\[
c_{\alpha,\beta}:=%
\begin{cases}
\dfrac{1}{r^{\ell\left(  \beta\right)  }}\left(  -1\right)  ^{\ell\left(
\beta\right)  -\ell\left(  \alpha\right)  }, & \text{if }D\left(
\alpha\right)  \subseteq D\left(  \beta\right)  ;\\
0, & \text{otherwise}%
\end{cases}
\]
of $\mathbf{k}$ for every $\alpha,\beta\in\operatorname*{Comp}\nolimits_{n}$.
\end{vershort}

\begin{verlong}
The scalar $r$ is invertible (by Convention \ref{conv.r-ible}). Thus, its
power $r^{\ell\left(  \beta\right)  }$ is invertible for each $\beta
\in\operatorname*{Comp}\nolimits_{n}$. We define a scalar%
\[
c_{\alpha,\beta}:=\dfrac{1}{r^{\ell\left(  \beta\right)  }}\left(  -1\right)
^{\ell\left(  \beta\right)  -\ell\left(  \alpha\right)  }\left[  D\left(
\alpha\right)  \subseteq D\left(  \beta\right)  \right]  \in\mathbf{k}%
\]
for every $\alpha,\beta\in\operatorname*{Comp}\nolimits_{n}$ (indeed, this
expression is well-defined\footnote{\textit{Proof.} The previous sentence
shows that $r^{\ell\left(  \beta\right)  }$ is invertible. Hence, $\dfrac
{1}{r^{\ell\left(  \beta\right)  }}$ is well-defined. Also, $\left(
-1\right)  ^{\ell\left(  \beta\right)  -\ell\left(  \alpha\right)  }$ is
well-defined, since $-1$ is invertible.}). Thus, if $\alpha,\beta
\in\operatorname*{Comp}\nolimits_{n}$ are two compositions that satisfy
$D\left(  \alpha\right)  \not \subseteq D\left(  \beta\right)  $, then%
\begin{align}
c_{\alpha,\beta}  &  =\dfrac{1}{r^{\ell\left(  \beta\right)  }}\left(
-1\right)  ^{\ell\left(  \beta\right)  -\ell\left(  \alpha\right)
}\underbrace{\left[  D\left(  \alpha\right)  \subseteq D\left(  \beta\right)
\right]  }_{\substack{=0\\\text{(since }D\left(  \alpha\right)
\not \subseteq D\left(  \beta\right)  \text{)}}}\nonumber\\
&  =0. \label{pf.prop.etastar-dual-basis.c=0}%
\end{align}
On the other hand, if $\alpha,\beta\in\operatorname*{Comp}\nolimits_{n}$ are
two compositions that satisfy $D\left(  \alpha\right)  \subseteq D\left(
\beta\right)  $, then%
\begin{align}
c_{\alpha,\beta}  &  =\dfrac{1}{r^{\ell\left(  \beta\right)  }}\left(
-1\right)  ^{\ell\left(  \beta\right)  -\ell\left(  \alpha\right)
}\underbrace{\left[  D\left(  \alpha\right)  \subseteq D\left(  \beta\right)
\right]  }_{\substack{=1\\\text{(since }D\left(  \alpha\right)  \subseteq
D\left(  \beta\right)  \text{)}}}\nonumber\\
&  =\dfrac{1}{r^{\ell\left(  \beta\right)  }}\left(  -1\right)  ^{\ell\left(
\beta\right)  -\ell\left(  \alpha\right)  }.
\label{pf.prop.etastar-dual-basis.c=1}%
\end{align}

\end{verlong}

Now, let $\beta\in\operatorname*{Comp}\nolimits_{n}$. Then, Proposition
\ref{prop.eta.M-through-eta} shows that%
\begin{equation}
r^{\ell\left(  \beta\right)  }M_{\beta}=\sum_{\substack{\alpha\in
\operatorname*{Comp}\nolimits_{n};\\D\left(  \alpha\right)  \subseteq D\left(
\beta\right)  }}\left(  -1\right)  ^{\ell\left(  \beta\right)  -\ell\left(
\alpha\right)  }\eta_{\alpha}^{\left(  q\right)  }.
\label{pf.prop.etastar-dual-basis.1}%
\end{equation}

\begin{vershort}
\noindent Dividing both sides of this equality by $r^{\ell\left(
\beta\right)  }$, we obtain%
\[
M_{\beta}=\sum_{\substack{\alpha\in\operatorname*{Comp}\nolimits_{n}%
;\\D\left(  \alpha\right)  \subseteq D\left(  \beta\right)  }}\dfrac
{1}{r^{\ell\left(  \beta\right)  }}\left(  -1\right)  ^{\ell\left(
\beta\right)  -\ell\left(  \alpha\right)  }\eta_{\alpha}^{\left(  q\right)
}=\sum_{\alpha\in\operatorname*{Comp}\nolimits_{n}}c_{\alpha,\beta}%
\eta_{\alpha}^{\left(  q\right)  }%
\]
(since our definition of $c_{\alpha,\beta}$ ensures that all addends of the
sum $\sum_{\alpha\in\operatorname*{Comp}\nolimits_{n}}c_{\alpha,\beta}%
\eta_{\alpha}^{\left(  q\right)  }$ vanish except for the addends that satisfy
$D\left(  \alpha\right)  \subseteq D\left(  \beta\right)  $, but these latter
addends are precisely the addends of $\sum_{\substack{\alpha\in
\operatorname*{Comp}\nolimits_{n};\\D\left(  \alpha\right)  \subseteq D\left(
\beta\right)  }}\dfrac{1}{r^{\ell\left(  \beta\right)  }}\left(  -1\right)
^{\ell\left(  \beta\right)  -\ell\left(  \alpha\right)  }\eta_{\alpha
}^{\left(  q\right)  }$).
\end{vershort}

\begin{verlong}
However, every $\alpha\in\operatorname*{Comp}\nolimits_{n}$ satisfies either
$D\left(  \alpha\right)  \subseteq D\left(  \beta\right)  $ or $D\left(
\alpha\right)  \not \subseteq D\left(  \beta\right)  $ (but not both at the
same time). Hence,%
\begin{align*}
\sum_{\alpha\in\operatorname*{Comp}\nolimits_{n}}c_{\alpha,\beta}\eta_{\alpha
}^{\left(  q\right)  }  &  =\sum_{\substack{\alpha\in\operatorname*{Comp}%
\nolimits_{n};\\D\left(  \alpha\right)  \subseteq D\left(  \beta\right)
}}\underbrace{c_{\alpha,\beta}}_{\substack{=\dfrac{1}{r^{\ell\left(
\beta\right)  }}\left(  -1\right)  ^{\ell\left(  \beta\right)  -\ell\left(
\alpha\right)  }\\\text{(by (\ref{pf.prop.etastar-dual-basis.c=1}))}}%
}\eta_{\alpha}^{\left(  q\right)  }+\sum_{\substack{\alpha\in
\operatorname*{Comp}\nolimits_{n};\\D\left(  \alpha\right)  \not \subseteq
D\left(  \beta\right)  }}\underbrace{c_{\alpha,\beta}}%
_{\substack{=0\\\text{(by (\ref{pf.prop.etastar-dual-basis.c=0}))}}%
}\eta_{\alpha}^{\left(  q\right)  }\\
&  =\sum_{\substack{\alpha\in\operatorname*{Comp}\nolimits_{n};\\D\left(
\alpha\right)  \subseteq D\left(  \beta\right)  }}\dfrac{1}{r^{\ell\left(
\beta\right)  }}\left(  -1\right)  ^{\ell\left(  \beta\right)  -\ell\left(
\alpha\right)  }\eta_{\alpha}^{\left(  q\right)  }+\underbrace{\sum
_{\substack{\alpha\in\operatorname*{Comp}\nolimits_{n};\\D\left(
\alpha\right)  \not \subseteq D\left(  \beta\right)  }}0\eta_{\alpha}^{\left(
q\right)  }}_{=0}\\
&  =\sum_{\substack{\alpha\in\operatorname*{Comp}\nolimits_{n};\\D\left(
\alpha\right)  \subseteq D\left(  \beta\right)  }}\dfrac{1}{r^{\ell\left(
\beta\right)  }}\left(  -1\right)  ^{\ell\left(  \beta\right)  -\ell\left(
\alpha\right)  }\eta_{\alpha}^{\left(  q\right)  }\\
&  =\dfrac{1}{r^{\ell\left(  \beta\right)  }}\underbrace{\sum
_{\substack{\alpha\in\operatorname*{Comp}\nolimits_{n};\\D\left(
\alpha\right)  \subseteq D\left(  \beta\right)  }}\left(  -1\right)
^{\ell\left(  \beta\right)  -\ell\left(  \alpha\right)  }\eta_{\alpha
}^{\left(  q\right)  }}_{\substack{=r^{\ell\left(  \beta\right)  }M_{\beta
}\\\text{(by (\ref{pf.prop.etastar-dual-basis.1}))}}}=\dfrac{1}{r^{\ell\left(
\beta\right)  }}r^{\ell\left(  \beta\right)  }M_{\beta}=M_{\beta},
\end{align*}
so that $M_{\beta}=\sum_{\alpha\in\operatorname*{Comp}\nolimits_{n}}%
c_{\alpha,\beta}\eta_{\alpha}^{\left(  q\right)  }$.
\end{verlong}

Forget that we fixed $\beta$. We thus have proved that%
\begin{equation}
M_{\beta}=\sum_{\alpha\in\operatorname*{Comp}\nolimits_{n}}c_{\alpha,\beta
}\eta_{\alpha}^{\left(  q\right)  } \label{pf.prop.etastar-dual-basis.3}%
\end{equation}
holds for each $\beta\in\operatorname*{Comp}\nolimits_{n}$.

\begin{vershort}
Furthermore, for each $\alpha\in\operatorname*{Comp}\nolimits_{n}$, we have%
\begin{align*}
\sum_{\beta\in\operatorname*{Comp}\nolimits_{n}}c_{\alpha,\beta}H_{\beta}  &
=\sum_{\substack{\beta\in\operatorname*{Comp}\nolimits_{n};\\D\left(
\alpha\right)  \subseteq D\left(  \beta\right)  }}\dfrac{1}{r^{\ell\left(
\beta\right)  }}\left(  -1\right)  ^{\ell\left(  \beta\right)  -\ell\left(
\alpha\right)  }H_{\beta}\ \ \ \ \ \ \ \ \ \ \left(  \text{by the definition
of the }c_{\alpha,\beta}\right) \\
&  =\eta_{\alpha}^{\ast\left(  q\right)  }\ \ \ \ \ \ \ \ \ \ \left(  \text{by
Definition \ref{def.etastar}}\right)
\end{align*}

\end{vershort}

\begin{verlong}
Furthermore, for each $\alpha\in\operatorname*{Comp}\nolimits_{n}$, we have%
\begin{align*}
&  \sum_{\beta\in\operatorname*{Comp}\nolimits_{n}}c_{\alpha,\beta}H_{\beta}\\
&  =\sum_{\substack{\beta\in\operatorname*{Comp}\nolimits_{n};\\D\left(
\alpha\right)  \subseteq D\left(  \beta\right)  }}\underbrace{c_{\alpha,\beta
}}_{\substack{=\dfrac{1}{r^{\ell\left(  \beta\right)  }}\left(  -1\right)
^{\ell\left(  \beta\right)  -\ell\left(  \alpha\right)  }\\\text{(by
(\ref{pf.prop.etastar-dual-basis.c=1}))}}}H_{\beta}+\sum_{\substack{\beta
\in\operatorname*{Comp}\nolimits_{n};\\D\left(  \alpha\right)  \not \subseteq
D\left(  \beta\right)  }}\underbrace{c_{\alpha,\beta}}%
_{\substack{=0\\\text{(by (\ref{pf.prop.etastar-dual-basis.c=0}))}}}H_{\beta
}\\
&  \ \ \ \ \ \ \ \ \ \ \ \ \ \ \ \ \ \ \ \ \left(
\begin{array}
[c]{c}%
\text{since each }\beta\in\operatorname*{Comp}\nolimits_{n}\text{ satisfies
either }D\left(  \alpha\right)  \subseteq D\left(  \beta\right) \\
\text{or }D\left(  \alpha\right)  \not \subseteq D\left(  \beta\right)  \text{
(but not both simultaneously)}%
\end{array}
\right) \\
&  =\sum_{\substack{\beta\in\operatorname*{Comp}\nolimits_{n};\\D\left(
\alpha\right)  \subseteq D\left(  \beta\right)  }}\dfrac{1}{r^{\ell\left(
\beta\right)  }}\left(  -1\right)  ^{\ell\left(  \beta\right)  -\ell\left(
\alpha\right)  }H_{\beta}+\underbrace{\sum_{\substack{\beta\in
\operatorname*{Comp}\nolimits_{n};\\D\left(  \alpha\right)  \not \subseteq
D\left(  \beta\right)  }}0H_{\beta}}_{=0}\\
&  =\sum_{\substack{\beta\in\operatorname*{Comp}\nolimits_{n};\\D\left(
\alpha\right)  \subseteq D\left(  \beta\right)  }}\dfrac{1}{r^{\ell\left(
\beta\right)  }}\left(  -1\right)  ^{\ell\left(  \beta\right)  -\ell\left(
\alpha\right)  }H_{\beta}=\eta_{\alpha}^{\ast\left(  q\right)  }%
\ \ \ \ \ \ \ \ \ \ \left(  \text{by Definition \ref{def.etastar}}\right)
\end{align*}

\end{verlong}

\noindent and thus%
\begin{equation}
\eta_{\alpha}^{\ast\left(  q\right)  }=\sum_{\beta\in\operatorname*{Comp}%
\nolimits_{n}}c_{\alpha,\beta}H_{\beta}. \label{pf.prop.etastar-dual-basis.4}%
\end{equation}

Altogether, we now know the following:

\begin{itemize}
\item The $\mathbf{k}$-modules $\operatorname*{NSym}\nolimits_{n}$ and
$\operatorname*{QSym}\nolimits_{n}$ are free, and $\left\langle \cdot
,\cdot\right\rangle :\operatorname*{NSym}\nolimits_{n}\times
\operatorname*{QSym}\nolimits_{n}\rightarrow\mathbf{k}$ is a $\mathbf{k}%
$-bilinear form.

\item The set $\operatorname*{Comp}\nolimits_{n}$ is a finite set.

\item The family $\left(  H_{\alpha}\right)  _{\alpha\in\operatorname*{Comp}%
\nolimits_{n}}$ is a basis of the $\mathbf{k}$-module $\operatorname*{NSym}%
\nolimits_{n}$, and the family $\left(  \eta_{\alpha}^{\ast\left(  q\right)
}\right)  _{\alpha\in\operatorname*{Comp}\nolimits_{n}}$ is a further family
of elements of $\operatorname*{NSym}\nolimits_{n}$ (since
(\ref{pf.prop.etastar-dual-basis.4}) readily yields that $\eta_{\alpha}%
^{\ast\left(  q\right)  }\in\operatorname*{NSym}\nolimits_{n}$ for each
$\alpha\in\operatorname*{Comp}\nolimits_{n}$).

\item The families $\left(  M_{\alpha}\right)  _{\alpha\in\operatorname*{Comp}%
\nolimits_{n}}$ and $\left(  \eta_{\alpha}^{\left(  q\right)  }\right)
_{\alpha\in\operatorname*{Comp}\nolimits_{n}}$ are two bases of the
$\mathbf{k}$-module $\operatorname*{QSym}\nolimits_{n}$ (by Theorem
\ref{thm.eta.basis} \textbf{(b)}).

\item The basis $\left(  M_{\alpha}\right)  _{\alpha\in\operatorname*{Comp}%
\nolimits_{n}}$ of $\operatorname*{QSym}\nolimits_{n}$ is dual to the basis
$\left(  H_{\alpha}\right)  _{\alpha\in\operatorname*{Comp}\nolimits_{n}}$ of
$\operatorname*{NSym}\nolimits_{n}$ (because of (\ref{eq.dual.HM})).

\item The elements $c_{\alpha,\beta}\in\mathbf{k}$ are defined for all
$\left(  \alpha,\beta\right)  \in\operatorname*{Comp}\nolimits_{n}%
\times\operatorname*{Comp}\nolimits_{n}$, and satisfy
(\ref{pf.prop.etastar-dual-basis.3}) for each $\beta\in\operatorname*{Comp}%
\nolimits_{n}$ and (\ref{pf.prop.etastar-dual-basis.4}) for each $\alpha
\in\operatorname*{Comp}\nolimits_{n}$.
\end{itemize}

Thus, Lemma \ref{lem.dual-bases-crit} (applied to $F=\operatorname*{NSym}%
\nolimits_{n}$, $U=\operatorname*{QSym}\nolimits_{n}$, $A=\operatorname*{Comp}%
\nolimits_{n}$, $f_{\alpha}=H_{\alpha}$, $g_{\alpha}=\eta_{\alpha}%
^{\ast\left(  q\right)  }$, $u_{\alpha}=M_{\alpha}$ and $v_{\alpha}%
=\eta_{\alpha}^{\left(  q\right)  }$) shows that the families $\left(
\eta_{\alpha}^{\left(  q\right)  }\right)  _{\alpha\in\operatorname*{Comp}%
\nolimits_{n}}$ and $\left(  \eta_{\alpha}^{\ast\left(  q\right)  }\right)
_{\alpha\in\operatorname*{Comp}\nolimits_{n}}$ are mutually dual bases of
$\operatorname*{QSym}\nolimits_{n}$ and $\operatorname*{NSym}\nolimits_{n}$,
respectively. In other words, the family $\left(  \eta_{\alpha}^{\ast\left(
q\right)  }\right)  _{\alpha\in\operatorname*{Comp}\nolimits_{n}}$ is the
basis of $\operatorname*{NSym}\nolimits_{n}$ dual to the basis $\left(
\eta_{\alpha}^{\left(  q\right)  }\right)  _{\alpha\in\operatorname*{Comp}%
\nolimits_{n}}$ of $\operatorname*{QSym}\nolimits_{n}$. This proves
Proposition \ref{prop.etastar-dual-basis} \textbf{(b)}. \medskip

\begin{vershort}
\textbf{(a)} This follows from part \textbf{(b)} by taking the direct sum over
all $n$. \qedhere

\end{vershort}

\begin{verlong}
\textbf{(a)} Since $\operatorname*{NSym}$ is a graded $\mathbf{k}$-module, we
have $\operatorname*{NSym}=\bigoplus\limits_{n\in\mathbb{N}}%
\operatorname*{NSym}\nolimits_{n}$.

Proposition \ref{prop.etastar-dual-basis} \textbf{(b)} shows that for each
$n\in\mathbb{N}$, the family $\left(  \eta_{\alpha}^{\ast\left(  q\right)
}\right)  _{\alpha\in\operatorname*{Comp}\nolimits_{n}}$ is a basis of the
$\mathbf{k}$-module $\operatorname*{NSym}\nolimits_{n}$. Hence, the union
$\left(  \eta_{\alpha}^{\ast\left(  q\right)  }\right)  _{n\in\mathbb{N}%
,\ \alpha\in\operatorname*{Comp}\nolimits_{n}}$ of all these families is a
basis of the direct sum $\bigoplus\limits_{n\in\mathbb{N}}\operatorname*{NSym}%
\nolimits_{n}=\operatorname*{NSym}$. In other words, the family $\left(
\eta_{\alpha}^{\ast\left(  q\right)  }\right)  _{\alpha\in\operatorname*{Comp}%
}$ is a basis of the direct sum $\operatorname*{NSym}$ (since the family
$\left(  \eta_{\alpha}^{\ast\left(  q\right)  }\right)  _{n\in\mathbb{N}%
,\ \alpha\in\operatorname*{Comp}\nolimits_{n}}$ is just a reindexing of the
family $\left(  \eta_{\alpha}^{\ast\left(  q\right)  }\right)  _{\alpha
\in\operatorname*{Comp}}$ (because $\operatorname*{Comp}=\bigsqcup
\limits_{n\in\mathbb{N}}\operatorname*{Comp}\nolimits_{n}$)). Note that this
basis is actually a graded basis of $\operatorname*{NSym}$ (since its
subfamily $\left(  \eta_{\alpha}^{\ast\left(  q\right)  }\right)  _{\alpha
\in\operatorname*{Comp}\nolimits_{n}}$ is a basis of the $n$-th graded
component $\operatorname*{NSym}\nolimits_{n}$ for each $n\in\mathbb{N}$).

Let us also recall that the family $\left(  \eta_{\alpha}^{\left(  q\right)
}\right)  _{\alpha\in\operatorname*{Comp}}$ is a basis of the $\mathbf{k}%
$-module $\operatorname*{QSym}$ (by Theorem \ref{thm.eta.basis} \textbf{(a)}).
Again, this basis is actually a graded basis of $\operatorname*{QSym}$ (since
Theorem \ref{thm.eta.basis} \textbf{(b)} shows that its subfamily $\left(
\eta_{\alpha}^{\left(  q\right)  }\right)  _{\alpha\in\operatorname*{Comp}%
\nolimits_{n}}$ is a basis of the $n$-th graded component
$\operatorname*{QSym}\nolimits_{n}$ for each $n\in\mathbb{N}$).

Now, we claim that%
\begin{equation}
\left\langle \eta_{\gamma}^{\ast\left(  q\right)  },\eta_{\delta}^{\left(
q\right)  }\right\rangle =\left[  \gamma=\delta\right]
\label{pf.prop.etastar-dual-basis.a.4}%
\end{equation}
for all $\gamma,\delta\in\operatorname*{Comp}$.

[\textit{Proof of (\ref{pf.prop.etastar-dual-basis.a.4}):} Let $\gamma
,\delta\in\operatorname*{Comp}$. We must prove
(\ref{pf.prop.etastar-dual-basis.a.4}). We note that the element $\eta
_{\gamma}^{\ast\left(  q\right)  }$ of $\operatorname*{NSym}$ is homogeneous
of degree $\left\vert \gamma\right\vert $ (this follows easily from Definition
\ref{def.etastar}), whereas the element $\eta_{\delta}^{\left(  q\right)  }$
of $\operatorname*{QSym}$ is homogeneous of degree $\left\vert \delta
\right\vert $ (this follows easily from Definition \ref{def.etaalpha}). Now,
we are in one of the following two cases:

\textit{Case 1:} We have $\left\vert \gamma\right\vert =\left\vert
\delta\right\vert $.

\textit{Case 2:} We have $\left\vert \gamma\right\vert \neq\left\vert
\delta\right\vert $.

Let us first consider Case 1. In this case, we have $\left\vert \gamma
\right\vert =\left\vert \delta\right\vert $. Set $n:=\left\vert \delta
\right\vert $. Thus, $\gamma\in\operatorname*{Comp}\nolimits_{n}$ (since
$\left\vert \gamma\right\vert =\left\vert \delta\right\vert =n$) and
$\delta\in\operatorname*{Comp}\nolimits_{n}$ (since $\left\vert \delta
\right\vert =n$). However, Proposition \ref{prop.etastar-dual-basis}
\textbf{(b)} shows that the family $\left(  \eta_{\alpha}^{\ast\left(
q\right)  }\right)  _{\alpha\in\operatorname*{Comp}\nolimits_{n}}$ is the
basis of $\operatorname*{NSym}\nolimits_{n}$ dual to the basis $\left(
\eta_{\alpha}^{\left(  q\right)  }\right)  _{\alpha\in\operatorname*{Comp}%
\nolimits_{n}}$ of $\operatorname*{QSym}\nolimits_{n}$. Hence,
\[
\left\langle \eta_{\alpha}^{\ast\left(  q\right)  },\eta_{\beta}^{\left(
q\right)  }\right\rangle =\left[  \alpha=\beta\right]
\]
for every $\alpha,\beta\in\operatorname*{Comp}\nolimits_{n}$. Applying this to
$\alpha=\gamma$ and $\beta=\delta$, we find%
\[
\left\langle \eta_{\gamma}^{\ast\left(  q\right)  },\eta_{\delta}^{\left(
q\right)  }\right\rangle =\left[  \gamma=\delta\right]  .
\]
Hence, (\ref{pf.prop.etastar-dual-basis.a.4}) is proved in Case 1.

Let us now consider Case 2. In this case, we have $\left\vert \gamma
\right\vert \neq\left\vert \delta\right\vert $. Hence, $\gamma\neq\delta$, so
that $\left[  \gamma=\delta\right]  =0$. On the other hand, recall that the
element $\eta_{\gamma}^{\ast\left(  q\right)  }$ of $\operatorname*{NSym}$ is
homogeneous of degree $\left\vert \gamma\right\vert $, whereas the element
$\eta_{\delta}^{\left(  q\right)  }$ of $\operatorname*{QSym}$ is homogeneous
of degree $\left\vert \delta\right\vert $. Since the degrees $\left\vert
\gamma\right\vert $ and $\left\vert \delta\right\vert $ are different (because
$\left\vert \gamma\right\vert \neq\left\vert \delta\right\vert $), we thus
conclude that the elements $\eta_{\gamma}^{\ast\left(  q\right)  }%
\in\operatorname*{NSym}$ and $\eta_{\delta}^{\left(  q\right)  }%
\in\operatorname*{QSym}$ are homogeneous of different degrees.

However, the form $\left\langle \cdot,\cdot\right\rangle $ is graded; in other
words, it satisfies $\left\langle f,g\right\rangle =0$ whenever $f\in
\operatorname*{NSym}$ and $g\in\operatorname*{QSym}$ are homogeneous elements
of different degrees. Applying this to $f=\eta_{\gamma}^{\ast\left(  q\right)
}$ and $g=\eta_{\delta}^{\left(  q\right)  }$, we conclude that $\left\langle
\eta_{\gamma}^{\ast\left(  q\right)  },\eta_{\delta}^{\left(  q\right)
}\right\rangle =0$ (since the elements $\eta_{\gamma}^{\ast\left(  q\right)
}\in\operatorname*{NSym}$ and $\eta_{\delta}^{\left(  q\right)  }%
\in\operatorname*{QSym}$ are homogeneous of different degrees). Comparing this
with $\left[  \gamma=\delta\right]  =0$, we obtain
\[
\left\langle \eta_{\gamma}^{\ast\left(  q\right)  },\eta_{\delta}^{\left(
q\right)  }\right\rangle =\left[  \gamma=\delta\right]  .
\]
Hence, (\ref{pf.prop.etastar-dual-basis.a.4}) is proved in Case 2.

We have now proved (\ref{pf.prop.etastar-dual-basis.a.4}) in both Cases 1 and
2. Thus, (\ref{pf.prop.etastar-dual-basis.a.4}) always holds.]

Now, recall that the two families $\left(  \eta_{\alpha}^{\ast\left(
q\right)  }\right)  _{\alpha\in\operatorname*{Comp}}$ and $\left(
\eta_{\alpha}^{\left(  q\right)  }\right)  _{\alpha\in\operatorname*{Comp}}$
are graded bases of the graded $\mathbf{k}$-modules $\operatorname*{NSym}$ and
$\operatorname*{QSym}$, respectively. Hence,
(\ref{pf.prop.etastar-dual-basis.a.4}) shows that these two bases are mutually
dual. In other words, the family $\left(  \eta_{\alpha}^{\ast\left(  q\right)
}\right)  _{\alpha\in\operatorname*{Comp}}$ is the basis of
$\operatorname*{NSym}$ dual to the basis $\left(  \eta_{\alpha}^{\left(
q\right)  }\right)  _{\alpha\in\operatorname*{Comp}}$ of $\operatorname*{QSym}%
$ with respect to the bilinear form $\left\langle \cdot,\cdot\right\rangle $.
This proves Proposition \ref{prop.etastar-dual-basis} \textbf{(a)}. \qedhere

\end{verlong}

\begin{noncompile}
Old stuff, no longer needed:

Let us first show that this family is a basis of the $\mathbf{k}$-module
$\operatorname*{NSym}$ in the first place.

Indeed, we fix $n\in\mathbb{N}$. Consider the $n$-th graded component
$\operatorname*{NSym}\nolimits_{n}$ of the graded $\mathbf{k}$-module
$\operatorname*{NSym}$.

Define a partial order $\prec$ on the finite set $\operatorname*{Comp}%
\nolimits_{n}$ by setting $\beta\prec\alpha$ if and only if $\ell\left(
\beta\right)  >\ell\left(  \alpha\right)  $.

The definition of $\eta_{\alpha}^{\ast\left(  q\right)  }$ shows that%
\begin{align*}
\eta_{\alpha}^{\ast\left(  q\right)  }  &  =\sum_{\substack{\beta
\in\operatorname*{Comp}\nolimits_{n};\\D\left(  \alpha\right)  \subseteq
D\left(  \beta\right)  }}\dfrac{1}{r^{\ell\left(  \beta\right)  }}\left(
-1\right)  ^{\ell\left(  \beta\right)  -\ell\left(  \alpha\right)  }H_{\beta
}\\
&  =\dfrac{1}{r^{\ell\left(  \alpha\right)  }}\underbrace{\left(  -1\right)
^{\ell\left(  \alpha\right)  -\ell\left(  \alpha\right)  }}_{=\left(
-1\right)  ^{0}=1}H_{\alpha}\\
&  \ \ \ \ \ \ \ \ \ \ +\left(  \text{a linear combination of }H_{\beta}\text{
with }\beta\in\operatorname*{Comp}\nolimits_{n}\text{ satisfying }\ell\left(
\beta\right)  >\ell\left(  \alpha\right)  \right) \\
&  =\dfrac{1}{r^{\ell\left(  \alpha\right)  }}H_{\alpha}+\left(  \text{a
linear combination of }H_{\beta}\text{ with }\beta\in\operatorname*{Comp}%
\nolimits_{n}\text{ satisfying }\beta\prec\alpha\right)
\end{align*}
for each $\alpha\in\operatorname*{Comp}\nolimits_{n}$. Since $r$ is
invertible, this shows that the family $\left(  \eta_{\alpha}^{\ast\left(
q\right)  }\right)  _{\alpha\in\operatorname*{Comp}\nolimits_{n}}$ expands
invertibly triangularly in the family $\left(  H_{\alpha}\right)  _{\alpha
\in\operatorname*{Comp}\nolimits_{n}}$ with respect to the partial order
$\prec$ (where we are using the terminology from \cite[\S 11.1]{GriRei}).
Hence, \cite[Corollary 11.1.19(e)]{GriRei} shows that the family $\left(
\eta_{\alpha}^{\ast\left(  q\right)  }\right)  _{\alpha\in\operatorname*{Comp}%
\nolimits_{n}}$ is a basis of the $\mathbf{k}$-module $\operatorname*{NSym}%
\nolimits_{n}$ (since the family $\left(  H_{\alpha}\right)  _{\alpha
\in\operatorname*{Comp}\nolimits_{n}}$ is a basis of $\operatorname*{NSym}%
\nolimits_{n}$).

Forget that we fixed $n$. Thus, we have shown that the family $\left(
\eta_{\alpha}^{\ast\left(  q\right)  }\right)  _{\alpha\in\operatorname*{Comp}%
\nolimits_{n}}$ is a basis of the $\mathbf{k}$-module $\operatorname*{NSym}%
\nolimits_{n}$ for each $n\in\mathbb{N}$. Hence, the family $\left(
\eta_{\alpha}^{\ast\left(  q\right)  }\right)  _{\alpha\in\operatorname*{Comp}%
}$ is a basis of the $\mathbf{k}$-module $\bigoplus_{n\in\mathbb{N}%
}\operatorname*{NSym}\nolimits_{n}=\operatorname*{NSym}$.

It remains to show that this basis is dual to the basis $\left(  \eta_{\alpha
}\right)  _{\alpha\in\operatorname*{Comp}}$ of $\operatorname*{QSym}$. In
other words, it remains to show that
\begin{equation}
\left\langle \eta_{\alpha}^{\ast\left(  q\right)  },\eta_{\beta}^{\left(
q\right)  }\right\rangle =\left[  \alpha=\beta\right]
\label{pf.prop.etastar-dual-basis.goal}%
\end{equation}
for any $\alpha,\beta\in\operatorname*{Comp}$.

So let us show this now. We fix $\alpha,\beta\in\operatorname*{Comp}$. Then......
\end{noncompile}
\end{proof}

\subsection{The dual eta basis: product}

We shall now study the multiplicative structure of the dual eta basis $\left(
\eta_{\alpha}^{\ast\left(  q\right)  }\right)  _{\alpha\in\operatorname*{Comp}%
}$. First, we introduce a notation for the simplest entries of this basis:

\begin{definition}
\label{def.etastarn}For each positive integer $n$, we let%
\begin{align}
\eta_{n}^{\ast\left(  q\right)  }:=  &  \ \eta_{\left(  n\right)  }%
^{\ast\left(  q\right)  }=\sum_{\beta\in\operatorname*{Comp}\nolimits_{n}%
}\dfrac{1}{r^{\ell\left(  \beta\right)  }}\left(  -1\right)  ^{\ell\left(
\beta\right)  -1}H_{\beta}\label{eq.def.etastarn.eq}\\
\in &  \ \operatorname*{NSym}.\nonumber
\end{align}
(The second equality sign here is easy to check.\footnotemark)
\end{definition}

\begin{vershort}
\footnotetext{Indeed, Definition \ref{def.etastar} yields%
\[
\eta_{\left(  n\right)  }^{\ast\left(  q\right)  }=\sum_{\substack{\beta
\in\operatorname*{Comp}\nolimits_{n};\\D\left(  \left(  n\right)  \right)
\subseteq D\left(  \beta\right)  }}\dfrac{1}{r^{\ell\left(  \beta\right)  }%
}\left(  -1\right)  ^{\ell\left(  \beta\right)  -\ell\left(  \left(  n\right)
\right)  }H_{\beta}.
\]
However, every $\beta\in\operatorname*{Comp}\nolimits_{n}$ automatically
satisfies $D\left(  \left(  n\right)  \right)  \subseteq D\left(
\beta\right)  $ (because $D\left(  \left(  n\right)  \right)  =\varnothing$).
Hence, the condition \textquotedblleft$D\left(  \left(  n\right)  \right)
\subseteq D\left(  \beta\right)  $\textquotedblright\ under the summation sign
is superfluous. Also, $\ell\left(  \left(  n\right)  \right)  =1$. Thus, the
above equality simplifies to%
\[
\eta_{\left(  n\right)  }^{\ast\left(  q\right)  }=\sum_{\beta\in
\operatorname*{Comp}\nolimits_{n}}\dfrac{1}{r^{\ell\left(  \beta\right)  }%
}\left(  -1\right)  ^{\ell\left(  \beta\right)  -1}H_{\beta}.
\]
}
\end{vershort}

\begin{verlong}
\footnotetext{\textit{Proof.} Definition \ref{def.etastar} yields%
\begin{equation}
\eta_{\left(  n\right)  }^{\ast\left(  q\right)  }=\sum_{\substack{\beta
\in\operatorname*{Comp}\nolimits_{n};\\D\left(  \left(  n\right)  \right)
\subseteq D\left(  \beta\right)  }}\dfrac{1}{r^{\ell\left(  \beta\right)  }%
}\left(  -1\right)  ^{\ell\left(  \beta\right)  -\ell\left(  \left(  n\right)
\right)  }H_{\beta}. \label{eq.def.etastarn.fn1.1}%
\end{equation}
However, every $\beta\in\operatorname*{Comp}\nolimits_{n}$ automatically
satisfies $D\left(  \left(  n\right)  \right)  \subseteq D\left(
\beta\right)  $ (because $D\left(  \left(  n\right)  \right)  =\varnothing
\subseteq D\left(  \beta\right)  $). Hence, the condition \textquotedblleft%
$D\left(  \left(  n\right)  \right)  \subseteq D\left(  \beta\right)
$\textquotedblright\ under the summation sign $\sum_{\substack{\beta
\in\operatorname*{Comp}\nolimits_{n};\\D\left(  \left(  n\right)  \right)
\subseteq D\left(  \beta\right)  }}$ is superfluous. Therefore,
\[
\sum_{\substack{\beta\in\operatorname*{Comp}\nolimits_{n};\\D\left(  \left(
n\right)  \right)  \subseteq D\left(  \beta\right)  }}=\sum_{\beta
\in\operatorname*{Comp}\nolimits_{n}}%
\]
(an equality between summation signs). Also, $\ell\left(  \left(  n\right)
\right)  =1$. Thus, the equality (\ref{eq.def.etastarn.fn1.1}) becomes%
\[
\eta_{\left(  n\right)  }^{\ast\left(  q\right)  }=\underbrace{\sum
_{\substack{\beta\in\operatorname*{Comp}\nolimits_{n};\\D\left(  \left(
n\right)  \right)  \subseteq D\left(  \beta\right)  }}}_{=\sum_{\beta
\in\operatorname*{Comp}\nolimits_{n}}}\dfrac{1}{r^{\ell\left(  \beta\right)
}}\underbrace{\left(  -1\right)  ^{\ell\left(  \beta\right)  -\ell\left(
\left(  n\right)  \right)  }}_{\substack{=\left(  -1\right)  ^{\ell\left(
\beta\right)  -1}\\\text{(since }\ell\left(  \left(  n\right)  \right)
=1\text{)}}}H_{\beta}=\sum_{\beta\in\operatorname*{Comp}\nolimits_{n}}%
\dfrac{1}{r^{\ell\left(  \beta\right)  }}\left(  -1\right)  ^{\ell\left(
\beta\right)  -1}H_{\beta}.
\]
}
\end{verlong}

It turns out that we can easily express $\eta_{\alpha}^{\ast\left(  q\right)
}$ for any composition $\alpha$ using these $\eta_{n}^{\ast\left(  q\right)
}$:

\begin{proposition}
\label{prop.etastar.mult}We have%
\[
\eta_{\alpha}^{\ast\left(  q\right)  }=\eta_{\alpha_{1}}^{\ast\left(
q\right)  }\eta_{\alpha_{2}}^{\ast\left(  q\right)  }\cdots\eta_{\alpha_{k}%
}^{\ast\left(  q\right)  }\ \ \ \ \ \ \ \ \ \ \text{for each composition
}\alpha=\left(  \alpha_{1},\alpha_{2},\ldots,\alpha_{k}\right)  .
\]

\end{proposition}

The main idea of the proof of Proposition \ref{prop.etastar.mult} is to
recognize that if $n=\left\vert \alpha\right\vert $, then the compositions
$\beta\in\operatorname*{Comp}\nolimits_{n}$ satisfying $D\left(
\alpha\right)  \subseteq D\left(  \beta\right)  $ are precisely the
compositions obtained from $\alpha$ by breaking up each entry of $\alpha$ into
pieces. A slicker way to formalize this proof proceeds using the notion of concatenation:

\begin{definition}
If $\alpha=\left(  \alpha_{1},\alpha_{2},\ldots,\alpha_{\ell}\right)  $ and
$\beta=\left(  \beta_{1},\beta_{2},\ldots,\beta_{k}\right)  $ are two
compositions, then the composition $\alpha\beta$ is defined by%
\[
\alpha\beta=\left(  \alpha_{1},\alpha_{2},\ldots,\alpha_{\ell},\beta_{1}%
,\beta_{2},\ldots,\beta_{k}\right)  .
\]
This composition $\alpha\beta$ is called the \emph{concatenation} of $\alpha$
and $\beta$. The operation of concatenation (sending any two compositions
$\alpha$ and $\beta$ to $\alpha\beta$) is associative, and the empty
composition $\varnothing$ is a neutral element for it; thus, the set of all
compositions is a monoid under this operation.
\end{definition}

The following proposition is saying (in the jargon of combinatorial Hopf
algebras) that the basis $\left(  \eta_{\alpha}^{\ast\left(  q\right)
}\right)  _{\alpha\in\operatorname*{Comp}}$ of $\operatorname*{NSym}$ is multiplicative:

\begin{proposition}
\label{prop.etastar.concat2}Let $\alpha$ and $\beta$ be two compositions.
Then,%
\[
\eta_{\alpha}^{\ast\left(  q\right)  }\eta_{\beta}^{\ast\left(  q\right)
}=\eta_{\alpha\beta}^{\ast\left(  q\right)  }.
\]

\end{proposition}

In order to prove this, we will use the comultiplication $\Delta
:\operatorname*{QSym}\rightarrow\operatorname*{QSym}\otimes
\operatorname*{QSym}$ of the Hopf algebra $\operatorname*{QSym}$ as well as
the duality between $\operatorname*{NSym}$ and $\operatorname*{QSym}$:

\begin{lemma}
\label{lem.NSym-duality.1}Let $f,g\in\operatorname*{NSym}$ and $h\in
\operatorname*{QSym}$ be arbitrary. Let the tensor $\Delta\left(  h\right)
\in\operatorname*{QSym}\otimes\operatorname*{QSym}$ be written in the form
$\Delta\left(  h\right)  =\sum_{i\in I}s_{i}\otimes t_{i}$, where $I$ is a
finite set and where $s_{i},t_{i}\in\operatorname*{QSym}$ for each $i\in I$.
Then,%
\[
\left\langle fg,h\right\rangle =\sum_{i\in I}\left\langle f,s_{i}\right\rangle
\left\langle g,t_{i}\right\rangle .
\]

\end{lemma}

\begin{proof}
Recall that the $\mathbf{k}$-bilinear form $\left\langle \cdot,\cdot
\right\rangle $ identifies $\operatorname*{NSym}$ with the graded dual
$\operatorname*{QSym}\nolimits^{o}$ as Hopf algebras. Thus, in particular, the
multiplication of $\operatorname*{NSym}$ and the comultiplication of
$\operatorname*{QSym}$ are mutually adjoint with respect to this form. In
other words, if $f,g\in\operatorname*{NSym}$ and $h\in\operatorname*{QSym}$,
then%
\[
\left\langle fg,h\right\rangle =\sum_{\left(  h\right)  }\left\langle
f,h_{\left(  1\right)  }\right\rangle \left\langle g,h_{\left(  2\right)
}\right\rangle ,
\]
where we are using the Sweedler notation $\sum_{\left(  h\right)  }h_{\left(
1\right)  }\otimes h_{\left(  2\right)  }$ for $\Delta\left(  h\right)  $
(see, e.g., \cite[(1.2.3)]{GriRei}). Lemma \ref{lem.NSym-duality.1} is just
restating this fact without using the Sweedler notation.
\end{proof}

\begin{proof}
[Proof of Proposition \ref{prop.etastar.concat2}.]This follows by dualization
from Theorem \ref{thm.Delta-eta}. Here are the details:

Forget that we fixed $\alpha$ and $\beta$. Proposition
\ref{prop.etastar-dual-basis} \textbf{(a)} shows that the families $\left(
\eta_{\alpha}^{\ast\left(  q\right)  }\right)  _{\alpha\in\operatorname*{Comp}%
}$ and $\left(  \eta_{\alpha}^{\left(  q\right)  }\right)  _{\alpha
\in\operatorname*{Comp}}$ are mutually dual bases of $\operatorname*{NSym}$
and $\operatorname*{QSym}$ with respect to the bilinear form $\left\langle
\cdot,\cdot\right\rangle $. This shows that%
\begin{equation}
\left\langle \eta_{\lambda}^{\ast\left(  q\right)  },\eta_{\mu}^{\left(
q\right)  }\right\rangle =\left[  \lambda=\mu\right]
\label{pf.prop.etastar.concat2.duality}%
\end{equation}
for all $\lambda,\mu\in\operatorname*{Comp}$. But another consequence of this
duality is that the bilinear form $\left\langle \cdot,\cdot\right\rangle $ is
nondegenerate (since only nondegenerate forms have dual bases), and that the
family $\left(  \eta_{\alpha}^{\left(  q\right)  }\right)  _{\alpha
\in\operatorname*{Comp}}$ is a basis of $\operatorname*{QSym}$. Hence, in
order to prove that two elements $f,g\in\operatorname*{NSym}$ are equal, it
suffices to show that $\left\langle f,\eta_{\gamma}^{\left(  q\right)
}\right\rangle =\left\langle g,\eta_{\gamma}^{\left(  q\right)  }\right\rangle
$ holds for each $\gamma\in\operatorname*{Comp}$.

We shall use this strategy to prove $\eta_{\alpha}^{\ast\left(  q\right)
}\eta_{\beta}^{\ast\left(  q\right)  }=\eta_{\alpha\beta}^{\ast\left(
q\right)  }$ for all $\alpha,\beta\in\operatorname*{Comp}$. Thus, we need to
show that $\left\langle \eta_{\alpha}^{\ast\left(  q\right)  }\eta_{\beta
}^{\ast\left(  q\right)  },\eta_{\gamma}^{\left(  q\right)  }\right\rangle
=\left\langle \eta_{\alpha\beta}^{\ast\left(  q\right)  },\eta_{\gamma
}^{\left(  q\right)  }\right\rangle $ holds for all $\alpha,\beta,\gamma
\in\operatorname*{Comp}$.

To show this, we fix $\alpha,\beta,\gamma\in\operatorname*{Comp}$. Theorem
\ref{thm.Delta-eta} (with the letters $\alpha,\beta,\gamma$ replaced by
$\gamma,\varphi,\psi$) says that%
\[
\Delta\left(  \eta_{\gamma}^{\left(  q\right)  }\right)  =\sum
_{\substack{\varphi,\psi\in\operatorname*{Comp};\\\gamma=\varphi\psi}%
}\eta_{\varphi}^{\left(  q\right)  }\otimes\eta_{\psi}^{\left(  q\right)  }.
\]

Hence, Lemma \ref{lem.NSym-duality.1} (applied to $f=\eta_{\alpha}%
^{\ast\left(  q\right)  }$ and $g=\eta_{\beta}^{\ast\left(  q\right)  }$ and
$h=\eta_{\gamma}^{\left(  q\right)  }$ and
\[
I=\left\{  \left(  \varphi,\psi\right)  \in\operatorname*{Comp}\times
\operatorname*{Comp}\ \mid\ \gamma=\varphi\psi\right\}
\]
and $s_{\left(  \varphi,\psi\right)  }=\eta_{\varphi}^{\left(  q\right)  }$
and $t_{\left(  \varphi,\psi\right)  }=\eta_{\psi}^{\left(  q\right)  }$)
yields%
\begin{align}
\left\langle \eta_{\alpha}^{\ast\left(  q\right)  }\eta_{\beta}^{\ast\left(
q\right)  },\eta_{\gamma}^{\left(  q\right)  }\right\rangle  &  =\sum
_{\substack{\varphi,\psi\in\operatorname*{Comp};\\\gamma=\varphi\psi
}}\underbrace{\left\langle \eta_{\alpha}^{\ast\left(  q\right)  }%
,\eta_{\varphi}^{\left(  q\right)  }\right\rangle }_{\substack{=\left[
\alpha=\varphi\right]  \\\text{(by (\ref{pf.prop.etastar.concat2.duality}))}%
}}\underbrace{\left\langle \eta_{\beta}^{\ast\left(  q\right)  },\eta_{\psi
}^{\left(  q\right)  }\right\rangle }_{\substack{=\left[  \beta=\psi\right]
\\\text{(by (\ref{pf.prop.etastar.concat2.duality}))}}}=\sum
_{\substack{\varphi,\psi\in\operatorname*{Comp};\\\gamma=\varphi\psi
}}\underbrace{\left[  \alpha=\varphi\right]  \cdot\left[  \beta=\psi\right]
}_{=\left[  \alpha=\varphi\text{ and }\beta=\psi\right]  }\nonumber\\
&  =\sum_{\substack{\varphi,\psi\in\operatorname*{Comp};\\\gamma=\varphi\psi
}}\left[  \alpha=\varphi\text{ and }\beta=\psi\right]  .
\label{pf.prop.etastar.concat2.4}%
\end{align}

The sum on the right hand side of this equality has at most one nonzero addend
-- namely the addend for $\varphi=\alpha$ and $\psi=\beta$, if this addend
exists. Of course, this addend exists if and only if $\gamma=\alpha\beta$, and
equals $1$ in this case. Thus, the sum equals $1$ if $\gamma=\alpha\beta$, and
otherwise equals $0$. In other words, this sum equals the truth value $\left[
\gamma=\alpha\beta\right]  $. Hence, we can rewrite
(\ref{pf.prop.etastar.concat2.4}) as%
\[
\left\langle \eta_{\alpha}^{\ast\left(  q\right)  }\eta_{\beta}^{\ast\left(
q\right)  },\eta_{\gamma}^{\left(  q\right)  }\right\rangle =\left[
\gamma=\alpha\beta\right]  .
\]
Comparing this with%
\begin{align*}
\left\langle \eta_{\alpha\beta}^{\ast\left(  q\right)  },\eta_{\gamma
}^{\left(  q\right)  }\right\rangle  &  =\left[  \alpha\beta=\gamma\right]
\ \ \ \ \ \ \ \ \ \ \left(  \text{by (\ref{pf.prop.etastar.concat2.duality}%
)}\right) \\
&  =\left[  \gamma=\alpha\beta\right]  ,
\end{align*}
we obtain $\left\langle \eta_{\alpha}^{\ast\left(  q\right)  }\eta_{\beta
}^{\ast\left(  q\right)  },\eta_{\gamma}^{\left(  q\right)  }\right\rangle
=\left\langle \eta_{\alpha\beta}^{\ast\left(  q\right)  },\eta_{\gamma
}^{\left(  q\right)  }\right\rangle $.

Forget that we fixed $\gamma$. We thus have shown that $\left\langle
\eta_{\alpha}^{\ast\left(  q\right)  }\eta_{\beta}^{\ast\left(  q\right)
},\eta_{\gamma}^{\left(  q\right)  }\right\rangle =\left\langle \eta
_{\alpha\beta}^{\ast\left(  q\right)  },\eta_{\gamma}^{\left(  q\right)
}\right\rangle $ for each $\gamma\in\operatorname*{Comp}$. Since $\left(
\eta_{\gamma}^{\left(  q\right)  }\right)  _{\gamma\in\operatorname*{Comp}}$
is a basis of the $\mathbf{k}$-module $\operatorname*{QSym}$, and since the
bilinear form $\left\langle \cdot,\cdot\right\rangle $ is nondegenerate, we
thus conclude that $\eta_{\alpha}^{\ast\left(  q\right)  }\eta_{\beta}%
^{\ast\left(  q\right)  }=\eta_{\alpha\beta}^{\ast\left(  q\right)  }$. This
proves Proposition \ref{prop.etastar.concat2}. \qedhere

\begin{verlong}
T0D0: verlong proof.
\end{verlong}

\begin{noncompile}
Alternative proof: The definition of a concatenation easily yields that
$\ell\left(  \gamma\right)  +\ell\left(  \delta\right)  =\ell\left(
\gamma\delta\right)  $ for any two compositions $\gamma$ and $\delta$. Thus,
in particular, $\ell\left(  \alpha\beta\right)  =\ell\left(  \alpha\right)
+\ell\left(  \beta\right)  $.

Let $n=\left\vert \alpha\right\vert $ and $m=\left\vert \beta\right\vert $.
Thus, $\alpha\in\operatorname*{Comp}\nolimits_{n}$ and $\beta\in
\operatorname*{Comp}\nolimits_{m}$, so that $\alpha\beta\in
\operatorname*{Comp}\nolimits_{n+m}$.

The definitions of $\eta_{\alpha}^{\ast\left(  q\right)  }$ and $\eta_{\beta
}^{\ast\left(  q\right)  }$ yield%
\begin{align*}
\eta_{\alpha}^{\ast\left(  q\right)  }  &  =\sum_{\substack{\gamma
\in\operatorname*{Comp}\nolimits_{n};\\D\left(  \alpha\right)  \subseteq
D\left(  \gamma\right)  }}\dfrac{1}{r^{\ell\left(  \gamma\right)  }}\left(
-1\right)  ^{\ell\left(  \gamma\right)  -\ell\left(  \alpha\right)  }%
H_{\gamma}\ \ \ \ \ \ \ \ \ \ \text{and}\\
\eta_{\beta}^{\ast\left(  q\right)  }  &  =\sum_{\substack{\delta
\in\operatorname*{Comp}\nolimits_{m};\\D\left(  \beta\right)  \subseteq
D\left(  \delta\right)  }}\dfrac{1}{r^{\ell\left(  \delta\right)  }}\left(
-1\right)  ^{\ell\left(  \delta\right)  -\ell\left(  \beta\right)  }H_{\delta
}.
\end{align*}
Multiplying these two equalities, we obtain%
\begin{align}
\eta_{\alpha}^{\ast\left(  q\right)  }\eta_{\beta}^{\ast\left(  q\right)  }
&  =\left(  \sum_{\substack{\gamma\in\operatorname*{Comp}\nolimits_{n}%
;\\D\left(  \alpha\right)  \subseteq D\left(  \gamma\right)  }}\dfrac
{1}{r^{\ell\left(  \gamma\right)  }}\left(  -1\right)  ^{\ell\left(
\gamma\right)  -\ell\left(  \alpha\right)  }H_{\gamma}\right)  \left(
\sum_{\substack{\delta\in\operatorname*{Comp}\nolimits_{m};\\D\left(
\beta\right)  \subseteq D\left(  \delta\right)  }}\dfrac{1}{r^{\ell\left(
\delta\right)  }}\left(  -1\right)  ^{\ell\left(  \delta\right)  -\ell\left(
\beta\right)  }H_{\delta}\right) \nonumber\\
&  =\sum_{\substack{\left(  \gamma,\delta\right)  \in\operatorname*{Comp}%
\nolimits_{n}\times\operatorname*{Comp}\nolimits_{m};\\D\left(  \alpha\right)
\subseteq D\left(  \gamma\right)  \text{ and }D\left(  \beta\right)  \subseteq
D\left(  \delta\right)  }}\underbrace{\dfrac{1}{r^{\ell\left(  \gamma\right)
}}\cdot\dfrac{1}{r^{\ell\left(  \delta\right)  }}}_{=\dfrac{1}{r^{\ell\left(
\gamma\right)  +\ell\left(  \delta\right)  }}=\dfrac{1}{r^{\ell\left(
\gamma\delta\right)  }}}\underbrace{\left(  -1\right)  ^{\ell\left(
\gamma\right)  -\ell\left(  \alpha\right)  }\left(  -1\right)  ^{\ell\left(
\delta\right)  -\ell\left(  \beta\right)  }}_{\substack{=\left(  -1\right)
^{\left(  \ell\left(  \gamma\right)  -\ell\left(  \alpha\right)  \right)
+\left(  \ell\left(  \delta\right)  -\ell\left(  \beta\right)  \right)
}\\=\left(  -1\right)  ^{\left(  \ell\left(  \gamma\right)  +\ell\left(
\delta\right)  \right)  -\left(  \ell\left(  \alpha\right)  +\ell\left(
\beta\right)  \right)  }\\=\left(  -1\right)  ^{\ell\left(  \gamma
\delta\right)  -\ell\left(  \alpha\beta\right)  }}}\underbrace{H_{\gamma
}H_{\delta}}_{=H_{\gamma\delta}}\nonumber\\
&  =\sum_{\substack{\left(  \gamma,\delta\right)  \in\operatorname*{Comp}%
\nolimits_{n}\times\operatorname*{Comp}\nolimits_{m};\\D\left(  \alpha\right)
\subseteq D\left(  \gamma\right)  \text{ and }D\left(  \beta\right)  \subseteq
D\left(  \delta\right)  }}\dfrac{1}{r^{\ell\left(  \gamma\delta\right)  }%
}\left(  -1\right)  ^{\ell\left(  \gamma\delta\right)  -\ell\left(
\alpha\beta\right)  }H_{\gamma\delta}. \label{pf.prop.etastar.concat2.prod3}%
\end{align}
But it is easy to see that every two compositions $\gamma\in
\operatorname*{Comp}\nolimits_{n}$ and $\delta\in\operatorname*{Comp}%
\nolimits_{m}$ satisfy $D\left(  \gamma\delta\right)  =D\left(  \gamma\right)
\cup\left\{  n\right\}  \cup\left(  D\left(  \delta\right)  +n\right)  $,
where $D\left(  \delta\right)  +n$ denotes the set $\left\{  d+n\ \mid\ d\in
D\left(  \delta\right)  \right\}  $. Using this fact, it is easy to see that
the map%
\begin{align*}
&  \left\{  \left(  \gamma,\delta\right)  \in\operatorname*{Comp}%
\nolimits_{n}\times\operatorname*{Comp}\nolimits_{m}\ \mid\ D\left(
\alpha\right)  \subseteq D\left(  \gamma\right)  \text{ and }D\left(
\beta\right)  \subseteq D\left(  \delta\right)  \right\} \\
&  \rightarrow\left\{  \zeta\in\operatorname*{Comp}\nolimits_{n+m}%
\ \mid\ D\left(  \alpha\beta\right)  \subseteq D\left(  \zeta\right)
\right\}  ,\\
\left(  \gamma,\delta\right)   &  \mapsto\gamma\delta
\end{align*}
is well-defined and is a bijection. Hence, we can substitute $\zeta$ for
$\gamma\delta$ in the sum on the right hand side of
(\ref{pf.prop.etastar.concat2.prod3}). Thus,
(\ref{pf.prop.etastar.concat2.prod3}) rewrites as%
\[
\eta_{\alpha}^{\ast\left(  q\right)  }\eta_{\beta}^{\ast\left(  q\right)
}=\sum_{\substack{\zeta\in\operatorname*{Comp}\nolimits_{n+m};\\D\left(
\alpha\beta\right)  \subseteq D\left(  \zeta\right)  }}\dfrac{1}%
{r^{\ell\left(  \zeta\right)  }}\left(  -1\right)  ^{\ell\left(  \zeta\right)
-\ell\left(  \alpha\beta\right)  }H_{\zeta}.
\]
Comparing this with%
\[
\eta_{\alpha\beta}^{\ast\left(  q\right)  }=\sum_{\substack{\zeta
\in\operatorname*{Comp}\nolimits_{n+m};\\D\left(  \alpha\beta\right)
\subseteq D\left(  \zeta\right)  }}\dfrac{1}{r^{\ell\left(  \zeta\right)  }%
}\left(  -1\right)  ^{\ell\left(  \zeta\right)  -\ell\left(  \alpha
\beta\right)  }H_{\zeta}\ \ \ \ \ \ \ \ \ \ \left(  \text{by
(\ref{eq.def.etastarn.eq})}\right)  ,
\]
we obtain $\eta_{\alpha}^{\ast\left(  q\right)  }\eta_{\beta}^{\ast\left(
q\right)  }=\eta_{\alpha\beta}^{\ast\left(  q\right)  }$. This proves
Proposition \ref{prop.etastar.concat2}.
\end{noncompile}
\end{proof}

\begin{corollary}
\label{cor.etastar.concat}Let $\beta_{1},\beta_{2},\ldots,\beta_{k}$ be
finitely many compositions. Then,%
\[
\eta_{\beta_{1}}^{\ast\left(  q\right)  }\eta_{\beta_{2}}^{\ast\left(
q\right)  }\cdots\eta_{\beta_{k}}^{\ast\left(  q\right)  }=\eta_{\beta
_{1}\beta_{2}\cdots\beta_{k}}^{\ast\left(  q\right)  }.
\]

\end{corollary}

\begin{proof}
This follows by induction on $k$ using Proposition \ref{prop.etastar.concat2}.
(The base case, $k=0$, follows from $\eta_{\left(  {}\right)  }^{\ast\left(
q\right)  }=1$.) \qedhere

\begin{verlong}
T0D0: verlong proof.
\end{verlong}
\end{proof}

\begin{proof}
[Proof of Proposition \ref{prop.etastar.mult}.]Let $\alpha=\left(  \alpha
_{1},\alpha_{2},\ldots,\alpha_{k}\right)  $ be a composition. Then, applying
Corollary \ref{cor.etastar.concat} to the $1$-element compositions $\beta
_{i}=\left(  \alpha_{i}\right)  $, we obtain%
\[
\eta_{\left(  \alpha_{1}\right)  }^{\ast\left(  q\right)  }\eta_{\left(
\alpha_{2}\right)  }^{\ast\left(  q\right)  }\cdots\eta_{\left(  \alpha
_{k}\right)  }^{\ast\left(  q\right)  }=\eta_{\left(  \alpha_{1}\right)
\left(  \alpha_{2}\right)  \cdots\left(  \alpha_{k}\right)  }^{\ast\left(
q\right)  }=\eta_{\alpha}^{\ast\left(  q\right)  }%
\]
(since the concatenation $\left(  \alpha_{1}\right)  \left(  \alpha
_{2}\right)  \cdots\left(  \alpha_{k}\right)  $ equals $\left(  \alpha
_{1},\alpha_{2},\ldots,\alpha_{k}\right)  =\alpha$). Thus,%
\[
\eta_{\alpha}^{\ast\left(  q\right)  }=\eta_{\left(  \alpha_{1}\right)
}^{\ast\left(  q\right)  }\eta_{\left(  \alpha_{2}\right)  }^{\ast\left(
q\right)  }\cdots\eta_{\left(  \alpha_{k}\right)  }^{\ast\left(  q\right)
}=\eta_{\alpha_{1}}^{\ast\left(  q\right)  }\eta_{\alpha_{2}}^{\ast\left(
q\right)  }\cdots\eta_{\alpha_{k}}^{\ast\left(  q\right)  }%
\]
(since $\eta_{\left(  n\right)  }^{\ast\left(  q\right)  }=\eta_{n}%
^{\ast\left(  q\right)  }$ for each $n>0$). This proves Proposition
\ref{prop.etastar.mult}.
\end{proof}

\subsection{The dual eta basis: generating function}

We shall now work in the ring $\operatorname*{NSym}\left[  \left[  t\right]
\right]  $ of formal power series in the indeterminate $t$ over the ring
$\operatorname*{NSym}$. This ring $\operatorname*{NSym}\left[  \left[
t\right]  \right]  $ is noncommutative (since $\operatorname*{NSym}$ is), but
the indeterminate $t$ commutes with all its elements.

We furthermore define two special series in this ring:

\begin{definition}
Define the formal power series%
\[
H\left(  t\right)  :=\sum_{n\geq0}H_{n}t^{n}\in\operatorname*{NSym}\left[
\left[  t\right]  \right]
\]
and
\[
G\left(  t\right)  :=\sum_{n\geq1}\eta_{n}^{\ast\left(  q\right)  }t^{n}%
\in\operatorname*{NSym}\left[  \left[  t\right]  \right]  .
\]

\end{definition}

Now, it is easy to see the following:

\begin{proposition}
\label{prop.Gt-through-Ht}We have%
\[
G\left(  t\right)  =1-\dfrac{1}{1+\dfrac{H\left(  t\right)  -1}{r}}%
=\dfrac{H\left(  t\right)  -1}{H\left(  t\right)  +q}.
\]

\end{proposition}

\begin{proof}
The power series $H\left(  t\right)  $ has constant term $H_{0}=1$. Thus, the
power series $H\left(  t\right)  -1$ has constant term $0$. Hence, the power
series $1+\dfrac{H\left(  t\right)  -1}{r}$ has constant term $1+\dfrac{0}%
{r}=1$, and thus is invertible (since every formal power series with constant
term $1$ is invertible). The fraction $\dfrac{1}{1+\dfrac{H\left(  t\right)
-1}{r}}$ is thus well-defined.

If $u\in\operatorname*{NSym}\left[  \left[  t\right]  \right]  $ is a formal
power series with constant term $0$, then the geometric series formula yields%
\[
\dfrac{1}{1-u}=\sum_{k\geq0}u^{k}=\underbrace{u^{0}}_{=1}+\sum_{k\geq1}%
u^{k}=1+\sum_{k\geq1}u^{k},
\]
so that%
\begin{equation}
\sum_{k\geq1}u^{k}=\dfrac{1}{1-u}-1. \label{pf.prop.Gt-through-Ht.sumuk=}%
\end{equation}

We shall use this result in a somewhat modified form: If $u\in
\operatorname*{NSym}\left[  \left[  t\right]  \right]  $ is a formal power
series with constant term $0$, then $\dfrac{-u}{r}$ is also a formal power
series with constant term $0$, and we have%
\begin{align}
\sum_{k\geq1}\dfrac{1}{r^{k}}\underbrace{\left(  -1\right)  ^{k-1}}_{=-\left(
-1\right)  ^{k}}u^{k}  &  =-\sum_{k\geq1}\underbrace{\dfrac{1}{r^{k}}\left(
-1\right)  ^{k}u^{k}}_{=\left(  \dfrac{-u}{r}\right)  ^{k}}=-\underbrace{\sum
_{k\geq1}\left(  \dfrac{-u}{r}\right)  ^{k}}_{\substack{=\dfrac{1}%
{1-\dfrac{-u}{r}}-1\\\text{(by (\ref{pf.prop.Gt-through-Ht.sumuk=}%
),}\\\text{applied to }\dfrac{-u}{r}\text{ instead of }u\text{)}}}\nonumber\\
&  =-\left(  \dfrac{1}{1-\dfrac{-u}{r}}-1\right)  =1-\dfrac{1}{1-\dfrac{-u}%
{r}}\nonumber\\
&  =1-\dfrac{1}{1+\dfrac{u}{r}}. \label{pf.prop.Gt-through-Ht.sum-uk=}%
\end{align}

We have $H\left(  t\right)  =\sum_{n\geq0}H_{n}t^{n}=\underbrace{H_{0}}%
_{=1}\underbrace{t^{0}}_{=1}+\sum_{n\geq1}H_{n}t^{n}=1+\sum_{n\geq1}H_{n}%
t^{n}$, so that%
\[
\sum_{n\geq1}H_{n}t^{n}=H\left(  t\right)  -1.
\]

The definition of $G\left(  t\right)  $ yields%
\begin{align*}
G\left(  t\right)   &  =\sum_{n\geq1}\eta_{n}^{\ast\left(  q\right)  }t^{n}\\
&  =\sum_{n\geq1}\left(  \sum_{\beta\in\operatorname*{Comp}\nolimits_{n}%
}\dfrac{1}{r^{\ell\left(  \beta\right)  }}\left(  -1\right)  ^{\ell\left(
\beta\right)  -1}H_{\beta}\right)  t^{n}\ \ \ \ \ \ \ \ \ \ \left(  \text{by
(\ref{eq.def.etastarn.eq})}\right) \\
&  =\sum_{n\geq1}\ \ \sum_{\beta\in\operatorname*{Comp}\nolimits_{n}}\dfrac
{1}{r^{\ell\left(  \beta\right)  }}\left(  -1\right)  ^{\ell\left(
\beta\right)  -1}H_{\beta}\underbrace{t^{n}}_{\substack{=t^{\left\vert
\beta\right\vert }\\\text{(since }\beta\in\operatorname*{Comp}\nolimits_{n}%
\text{)}}}\\
&  =\underbrace{\sum_{n\geq1}\ \ \sum_{\beta\in\operatorname*{Comp}%
\nolimits_{n}}}_{\substack{=\sum_{\substack{\beta\in\operatorname*{Comp}%
;\\\left\vert \beta\right\vert \geq1}}\\=\sum_{\substack{\beta\in
\operatorname*{Comp};\\\ell\left(  \beta\right)  \geq1}}\\\text{(because for a
composition }\beta\text{,}\\\text{the condition \textquotedblleft}\left\vert
\beta\right\vert \geq1\text{\textquotedblright\ is}\\\text{equivalent to
\textquotedblleft}\ell\left(  \beta\right)  \geq1\text{\textquotedblright)}%
}}\dfrac{1}{r^{\ell\left(  \beta\right)  }}\left(  -1\right)  ^{\ell\left(
\beta\right)  -1}H_{\beta}t^{\left\vert \beta\right\vert }\\
&  =\underbrace{\sum_{\substack{\beta\in\operatorname*{Comp};\\\ell\left(
\beta\right)  \geq1}}}_{=\sum_{k\geq1}\ \ \sum_{\substack{\beta\in
\operatorname*{Comp};\\\ell\left(  \beta\right)  =k}}}\dfrac{1}{r^{\ell\left(
\beta\right)  }}\left(  -1\right)  ^{\ell\left(  \beta\right)  -1}H_{\beta
}t^{\left\vert \beta\right\vert }\\
&  =\sum_{k\geq1}\ \ \sum_{\substack{\beta\in\operatorname*{Comp}%
;\\\ell\left(  \beta\right)  =k}}\underbrace{\dfrac{1}{r^{\ell\left(
\beta\right)  }}\left(  -1\right)  ^{\ell\left(  \beta\right)  -1}%
}_{\substack{=\dfrac{1}{r^{k}}\left(  -1\right)  ^{k-1}\\\text{(since }%
\ell\left(  \beta\right)  =k\text{)}}}H_{\beta}t^{\left\vert \beta\right\vert
}\\
&  =\sum_{k\geq1}\dfrac{1}{r^{k}}\left(  -1\right)  ^{k-1}\underbrace{\sum
_{\substack{\beta\in\operatorname*{Comp};\\\ell\left(  \beta\right)
=k}}H_{\beta}t^{\left\vert \beta\right\vert }}_{\substack{=\sum_{\left(
n_{1},n_{2},\ldots,n_{k}\right)  \in\operatorname*{Comp}}H_{\left(
n_{1},n_{2},\ldots,n_{k}\right)  }t^{n_{1}+n_{2}+\cdots+n_{k}}\\\text{(here,
we have renamed the summation index }\beta\text{ as }\left(  n_{1}%
,n_{2},\ldots,n_{k}\right)  \text{)}}}\\
&  =\sum_{k\geq1}\dfrac{1}{r^{k}}\left(  -1\right)  ^{k-1}\underbrace{\sum
_{\left(  n_{1},n_{2},\ldots,n_{k}\right)  \in\operatorname*{Comp}}}%
_{=\sum_{n_{1},n_{2},\ldots,n_{k}\geq1}}\underbrace{H_{\left(  n_{1}%
,n_{2},\ldots,n_{k}\right)  }t^{n_{1}+n_{2}+\cdots+n_{k}}}_{\substack{=\left(
H_{n_{1}}H_{n_{2}}\cdots H_{n_{k}}\right)  \left(  t^{n_{1}}t^{n_{2}}\cdots
t^{n_{k}}\right)  \\=\left(  H_{n_{1}}t^{n_{1}}\right)  \left(  H_{n_{2}%
}t^{n_{2}}\right)  \cdots\left(  H_{n_{k}}t^{n_{k}}\right)  }}
\end{align*}%
\begin{align*}
&  =\sum_{k\geq1}\dfrac{1}{r^{k}}\left(  -1\right)  ^{k-1}\underbrace{\sum
_{n_{1},n_{2},\ldots,n_{k}\geq1}\left(  H_{n_{1}}t^{n_{1}}\right)  \left(
H_{n_{2}}t^{n_{2}}\right)  \cdots\left(  H_{n_{k}}t^{n_{k}}\right)
}_{\substack{=\left(  \sum_{n\geq1}H_{n}t^{n}\right)  ^{k}\\\text{(by the
product rule)}}}\\
&  =\sum_{k\geq1}\dfrac{1}{r^{k}}\left(  -1\right)  ^{k-1}\left(
\underbrace{\sum_{n\geq1}H_{n}t^{n}}_{=H\left(  t\right)  -1}\right)
^{k}=\sum_{k\geq1}\dfrac{1}{r^{k}}\left(  -1\right)  ^{k-1}\left(  H\left(
t\right)  -1\right)  ^{k}\\
&  =1-\dfrac{1}{1+\dfrac{H\left(  t\right)  -1}{r}}\ \ \ \ \ \ \ \ \ \ \left(
\text{by (\ref{pf.prop.Gt-through-Ht.sum-uk=}), applied to }u=H\left(
t\right)  -1\right) \\
&  =1-\dfrac{r}{H\left(  t\right)  +r-1}=\dfrac{H\left(  t\right)
-1}{H\left(  t\right)  +r-1}=\dfrac{H\left(  t\right)  -1}{H\left(  t\right)
+q}%
\end{align*}
(since $r-1=q$ (because $r=q+1$)). This proves Proposition
\ref{prop.Gt-through-Ht}. \qedhere

\begin{verlong}
T0D0: verlong proof.
\end{verlong}
\end{proof}

\begin{proposition}
\label{prop.etastar.Gtk}Let $k\in\mathbb{N}$. Then,%
\begin{equation}
G\left(  t\right)  ^{k}=\sum_{\substack{\beta\in\operatorname*{Comp}%
;\\\ell\left(  \beta\right)  =k}}\eta_{\beta}^{\ast\left(  q\right)
}t^{\left\vert \beta\right\vert }. \label{pf.thm.Deltaetastara.Gtk}%
\end{equation}

\end{proposition}

\begin{proof}
From $G\left(  t\right)  =\sum_{n\geq1}\eta_{n}^{\ast\left(  q\right)  }t^{n}%
$, we obtain%
\begin{align*}
G\left(  t\right)  ^{k}  &  =\left(  \sum_{n\geq1}\eta_{n}^{\ast\left(
q\right)  }t^{n}\right)  ^{k}=\sum_{n_{1},n_{2},\ldots,n_{k}\geq
1}\underbrace{\left(  \eta_{n_{1}}^{\ast\left(  q\right)  }t^{n_{1}}\right)
\left(  \eta_{n_{2}}^{\ast\left(  q\right)  }t^{n_{2}}\right)  \cdots\left(
\eta_{n_{k}}^{\ast\left(  q\right)  }t^{n_{k}}\right)  }_{=\eta_{n_{1}}%
^{\ast\left(  q\right)  }\eta_{n_{2}}^{\ast\left(  q\right)  }\cdots
\eta_{n_{k}}^{\ast\left(  q\right)  }t^{n_{1}+n_{2}+\cdots+n_{k}}}\\
&  \ \ \ \ \ \ \ \ \ \ \ \ \ \ \ \ \ \ \ \ \left(  \text{by the product
rule}\right) \\
&  =\sum_{n_{1},n_{2},\ldots,n_{k}\geq1}\eta_{n_{1}}^{\ast\left(  q\right)
}\eta_{n_{2}}^{\ast\left(  q\right)  }\cdots\eta_{n_{k}}^{\ast\left(
q\right)  }t^{n_{1}+n_{2}+\cdots+n_{k}}\\
&  =\sum_{\beta=\left(  \beta_{1},\beta_{2},\ldots,\beta_{k}\right)
\in\operatorname*{Comp}}\ \ \underbrace{\eta_{\beta_{1}}^{\ast\left(
q\right)  }\eta_{\beta_{2}}^{\ast\left(  q\right)  }\cdots\eta_{\beta_{k}%
}^{\ast\left(  q\right)  }}_{\substack{=\eta_{\beta}^{\ast\left(  q\right)
}\\\text{(by Proposition \ref{prop.etastar.mult})}}}\ \ \underbrace{t^{\beta
_{1}+\beta_{2}+\cdots+\beta_{k}}}_{=t^{\left\vert \beta\right\vert }}\\
&  \ \ \ \ \ \ \ \ \ \ \ \ \ \ \ \ \ \ \ \ \left(  \text{here, we have renamed
}n_{1},n_{2},\ldots,n_{k}\text{ as }\beta_{1},\beta_{2},\ldots,\beta
_{k}\right) \\
&  =\underbrace{\sum_{\beta=\left(  \beta_{1},\beta_{2},\ldots,\beta
_{k}\right)  \in\operatorname*{Comp}}}_{=\sum_{\substack{\beta\in
\operatorname*{Comp};\\\ell\left(  \beta\right)  =k}}}\eta_{\beta}%
^{\ast\left(  q\right)  }t^{\left\vert \beta\right\vert }=\sum
_{\substack{\beta\in\operatorname*{Comp};\\\ell\left(  \beta\right)  =k}%
}\eta_{\beta}^{\ast\left(  q\right)  }t^{\left\vert \beta\right\vert }.
\end{align*}
This proves Proposition \ref{prop.etastar.Gtk}.
\end{proof}

\subsection{The dual eta basis: coproduct}

Consider the comultiplication $\Delta:\operatorname*{NSym}\rightarrow
\operatorname*{NSym}\otimes\operatorname*{NSym}$ of the Hopf algebra
$\operatorname*{NSym}$. We again recall the Iverson bracket notation
(Convention \ref{conv.iverson}).

\begin{theorem}
\label{thm.Deltaetastara}For any positive integer $n$, we have%
\[
\Delta\left(  \eta_{n}^{\ast\left(  q\right)  }\right)  =\sum_{\substack{\beta
,\gamma\in\operatorname*{Comp};\\\left\vert \beta\right\vert +\left\vert
\gamma\right\vert =n;\\\left\vert \ell\left(  \beta\right)  -\ell\left(
\gamma\right)  \right\vert \leq1}}\left(  -q\right)  ^{\max\left\{
\ell\left(  \beta\right)  ,\ell\left(  \gamma\right)  \right\}  -1}\left(
q-1\right)  ^{\left[  \ell\left(  \beta\right)  =\ell\left(  \gamma\right)
\right]  }\eta_{\beta}^{\ast\left(  q\right)  }\otimes\eta_{\gamma}%
^{\ast\left(  q\right)  }.
\]

\end{theorem}

\begin{example}
For $n=2$, there are exactly three pairs $\left(  \beta,\gamma\right)  $ of
compositions $\beta,\gamma\in\operatorname*{Comp}$ satisfying $\left\vert
\beta\right\vert +\left\vert \gamma\right\vert =n$ and $\left\vert \ell\left(
\beta\right)  -\ell\left(  \gamma\right)  \right\vert \leq1$: namely, the
pairs $\left(  \varnothing,\left(  2\right)  \right)  $, $\left(  \left(
1\right)  ,\left(  1\right)  \right)  $ and $\left(  \left(  2\right)
,\varnothing\right)  $. Hence, Theorem \ref{thm.Deltaetastara} (applied to
$n=2$) yields%
\begin{align*}
\Delta\left(  \eta_{2}^{\ast\left(  q\right)  }\right)   &  =\left(
-q\right)  ^{1-1}\left(  q-1\right)  ^{0}\eta_{\varnothing}^{\ast\left(
q\right)  }\otimes\eta_{\left(  2\right)  }^{\ast\left(  q\right)  }+\left(
-q\right)  ^{1-1}\left(  q-1\right)  ^{1}\eta_{\left(  1\right)  }%
^{\ast\left(  q\right)  }\otimes\eta_{\left(  1\right)  }^{\ast\left(
q\right)  }\\
&  \ \ \ \ \ \ \ \ \ \ +\left(  -q\right)  ^{1-1}\left(  q-1\right)  ^{0}%
\eta_{\left(  2\right)  }^{\ast\left(  q\right)  }\otimes\eta_{\varnothing
}^{\ast\left(  q\right)  }\\
&  =\eta_{\varnothing}^{\ast\left(  q\right)  }\otimes\eta_{\left(  2\right)
}^{\ast\left(  q\right)  }+\left(  q-1\right)  \eta_{\left(  1\right)  }%
^{\ast\left(  q\right)  }\otimes\eta_{\left(  1\right)  }^{\ast\left(
q\right)  }+\eta_{\left(  2\right)  }^{\ast\left(  q\right)  }\otimes
\eta_{\varnothing}^{\ast\left(  q\right)  }\\
&  =1\otimes\eta_{2}^{\ast\left(  q\right)  }+\left(  q-1\right)  \eta
_{1}^{\ast\left(  q\right)  }\otimes\eta_{1}^{\ast\left(  q\right)  }+\eta
_{2}^{\ast\left(  q\right)  }\otimes1
\end{align*}
(since $\eta_{\left(  2\right)  }^{\ast\left(  q\right)  } =\eta_{2}%
^{\ast\left(  q\right)  }$ and $\eta_{\left(  1\right)  }^{\ast\left(
q\right)  } =\eta_{1}^{\ast\left(  q\right)  }$ and $\eta_{\varnothing}%
^{\ast\left(  q\right)  }=1$).

Similar computations show that%
\[
\Delta\left(  \eta_{1}^{\ast\left(  q\right)  }\right)  =1\otimes\eta
_{1}^{\ast\left(  q\right)  }+\eta_{1}^{\ast\left(  q\right)  }\otimes1
\]
and%
\begin{align*}
\Delta\left(  \eta_{3}^{\ast\left(  q\right)  }\right)   &  =1\otimes\eta
_{3}^{\ast\left(  q\right)  }+\left(  q-1\right)  \eta_{1}^{\ast\left(
q\right)  }\otimes\eta_{2}^{\ast\left(  q\right)  }-q\eta_{1}^{\ast\left(
q\right)  }\otimes\left(  \eta_{1}^{\ast\left(  q\right)  }\right)  ^{2}\\
&  \ \ \ \ \ \ \ \ \ \ -q\left(  \eta_{1}^{\ast\left(  q\right)  }\right)
^{2}\otimes\eta_{1}^{\ast\left(  q\right)  }+\left(  q-1\right)  \eta
_{2}^{\ast\left(  q\right)  }\otimes\eta_{1}^{\ast\left(  q\right)  }+\eta
_{3}^{\ast\left(  q\right)  }\otimes1
\end{align*}
(since Proposition \ref{prop.etastar.mult} yields $\eta_{\left(  1,1\right)
}^{\ast\left(  q\right)  }=\left(  \eta_{1}^{\ast\left(  q\right)  }\right)
^{2}$).
\end{example}

Our proof of Theorem \ref{thm.Deltaetastara} will use the following general
fact from abstract algebra:

\begin{lemma}
\label{lem.AtBt}Let $A$ and $B$ be any two $\mathbf{k}$-algebras. Then, there
is a canonical $\mathbf{k}$-algebra homomorphism%
\begin{align*}
\iota:A\left[  \left[  t\right]  \right]  \otimes_{\mathbf{k}\left[  \left[
t\right]  \right]  }B\left[  \left[  t\right]  \right]   &  \rightarrow\left(
A\otimes B\right)  \left[  \left[  t\right]  \right]  ,\\
\left(  \sum_{i\in\mathbb{N}}a_{i}t^{i}\right)  \otimes\left(  \sum
_{j\in\mathbb{N}}b_{j}t^{j}\right)   &  \mapsto\left(  \sum_{i\in\mathbb{N}%
}\left(  a_{i}\otimes1\right)  t^{i}\right)  \left(  \sum_{j\in\mathbb{N}%
}\left(  1\otimes b_{j}\right)  t^{j}\right)  .
\end{align*}

\end{lemma}

\begin{proof}
[Proof of Lemma \ref{lem.AtBt} (sketched).]To construct the map $\iota$, we
need $A$ and $B$ only to be $\mathbf{k}$-modules, not $\mathbf{k}$-algebras.
The well-definedness follows easily from the fact that $\sum_{k\in\mathbb{N}%
}\left(  \lambda_{k}1\otimes1\right)  t^{k}=\sum_{k\in\mathbb{N}}\left(
1\otimes\lambda_{k}1\right)  t^{k}$ for any formal power series $\sum
_{k\in\mathbb{N}}\lambda_{k}t^{k}\in\mathbf{k}\left[  \left[  t\right]
\right]  $. Finally, to prove that $\iota$ is a $\mathbf{k}$-algebra
homomorphism, observe that any power series of the form $\sum_{i\in\mathbb{N}%
}\left(  a_{i}\otimes1\right)  t^{i}\in\left(  A\otimes B\right)  \left[
\left[  t\right]  \right]  $ commutes with any power series of the form
$\sum_{j\in\mathbb{N}}\left(  1\otimes b_{j}\right)  t^{j}\in\left(  A\otimes
B\right)  \left[  \left[  t\right]  \right]  $. The details are left to the reader.
\end{proof}

\begin{proof}
[Proof of Theorem \ref{thm.Deltaetastara}.]The comultiplication $\Delta
:\operatorname*{NSym}\rightarrow\operatorname*{NSym}\otimes
\operatorname*{NSym}$ is a $\mathbf{k}$-algebra homomorphism (since
$\operatorname*{NSym}$ is a $\mathbf{k}$-bialgebra), and thus induces a
$\mathbf{k}\left[  \left[  t\right]  \right]  $-algebra homomorphism%
\[
\Delta_{t}:\operatorname*{NSym}\left[  \left[  t\right]  \right]
\rightarrow\left(  \operatorname*{NSym}\otimes\operatorname*{NSym}\right)
\left[  \left[  t\right]  \right]
\]
that sends each formal power series $\sum_{i\in\mathbb{N}}a_{i}t^{i}$ to
$\sum_{i\in\mathbb{N}}\Delta\left(  a_{i}\right)  t^{i}$. This $\Delta_{t}$ is
a $\mathbf{k}$-algebra homomorphism as well (since any $\mathbf{k}\left[
\left[  t\right]  \right]  $-algebra homomorphism is a $\mathbf{k}$-algebra homomorphism).

Furthermore, Lemma \ref{lem.AtBt} shows that there is a canonical $\mathbf{k}%
$-algebra homomorphism%
\begin{align*}
\iota:\operatorname*{NSym}\left[  \left[  t\right]  \right]  \otimes
_{\mathbf{k}\left[  \left[  t\right]  \right]  }\operatorname*{NSym}\left[
\left[  t\right]  \right]   &  \rightarrow\left(  \operatorname*{NSym}%
\otimes\operatorname*{NSym}\right)  \left[  \left[  t\right]  \right]  ,\\
\left(  \sum_{i\in\mathbb{N}}a_{i}t^{i}\right)  \otimes\left(  \sum
_{j\in\mathbb{N}}b_{j}t^{j}\right)   &  \mapsto\left(  \sum_{i\in\mathbb{N}%
}\left(  a_{i}\otimes1\right)  t^{i}\right)  \left(  \sum_{j\in\mathbb{N}%
}\left(  1\otimes b_{j}\right)  t^{j}\right)  .
\end{align*}
Unlike some authors, we will not treat $\iota$ as an embedding, but we will
often use the fact that $\iota$ is a $\mathbf{k}$-algebra homomorphism.

We recall that $\mathbf{k}$-algebra homomorphisms are always ring
homomorphisms, and thus respect quotients. That is, if $f:A\rightarrow B$ is a
$\mathbf{k}$-algebra homomorphism, and if $a_{1}$ and $a_{2}$ are two elements
of $A$ such that $a_{2}$ is invertible in $A$, then $f\left(  a_{2}\right)  $
is again invertible and we have $f\left(  \dfrac{a_{1}}{a_{2}}\right)
=\dfrac{f\left(  a_{1}\right)  }{f\left(  a_{2}\right)  }$. We will use this
fact without further comment a few times.

For the sake of brevity, we define the shorthands%
\[
\mathbf{G}:=G\left(  t\right)  \ \ \ \ \ \ \ \ \ \ \text{and}%
\ \ \ \ \ \ \ \ \ \ \mathbf{H}:=H\left(  t\right)  .
\]

Proposition \ref{prop.Gt-through-Ht} says that $G\left(  t\right)
=\dfrac{H\left(  t\right)  -1}{H\left(  t\right)  +q}$. Using our
abbreviations $\mathbf{G}$ and $\mathbf{H}$, we can rewrite this as%
\begin{equation}
\mathbf{G}=\dfrac{\mathbf{H}-1}{\mathbf{H}+q}.
\label{pf.thm.Deltaetastara.G-via-H}%
\end{equation}
Hence,%
\begin{equation}
\Delta_{t}\left(  \mathbf{G}\right)  =\Delta_{t}\left(  \dfrac{\mathbf{H}%
-1}{\mathbf{H}+q}\right)  =\dfrac{\Delta_{t}\left(  \mathbf{H}\right)
-1}{\Delta_{t}\left(  \mathbf{H}\right)  +q}
\label{pf.thm.Deltaetastara.G-via-H2}%
\end{equation}
(since $\Delta_{t}$ is a $\mathbf{k}$-algebra homomorphism).

Next, we observe the following:

\begin{statement}
\textit{Claim 1:} We have%
\begin{equation}
\Delta_{t}\left(  \mathbf{H}\right)  =\iota\left(  \mathbf{H}\otimes
\mathbf{H}\right)  . \label{pf.thm.Deltaetastara.DeltaH}%
\end{equation}

\end{statement}

[\textit{Proof of Claim 1:} From $\mathbf{H}=H\left(  t\right)  =\sum
_{n\in\mathbb{N}}H_{n}t^{n}=\sum_{i\in\mathbb{N}}H_{i}t^{i}$ and
$\mathbf{H}=H\left(  t\right)  =\sum_{n\in\mathbb{N}}H_{n}t^{n}=\sum
_{j\in\mathbb{N}}H_{j}t^{j}$, we obtain%
\[
\mathbf{H}\otimes\mathbf{H}=\left(  \sum_{i\in\mathbb{N}}H_{i}t^{i}\right)
\otimes\left(  \sum_{j\in\mathbb{N}}H_{j}t^{j}\right)  .
\]
By the definition of $\iota$, this entails%
\begin{align}
\iota\left(  \mathbf{H}\otimes\mathbf{H}\right)   &  =\left(  \sum
_{i\in\mathbb{N}}\left(  H_{i}\otimes1\right)  t^{i}\right)  \left(
\sum_{j\in\mathbb{N}}\left(  1\otimes H_{j}\right)  t^{j}\right) \nonumber\\
&  =\sum_{n\in\mathbb{N}}\left(  \sum_{\substack{i,j\in\mathbb{N}%
;\\i+j=n}}\left(  H_{i}\otimes1\right)  \left(  1\otimes H_{j}\right)
\right)  t^{n} \label{pf.thm.Deltaetastara.DeltaH.c1.pf.1}%
\end{align}
(by the definition of the product of two power series). On the other hand,
$\mathbf{H}=\sum_{n\in\mathbb{N}}H_{n}t^{n}$. Thus,%
\begin{equation}
\Delta_{t}\left(  \mathbf{H}\right)  =\Delta_{t}\left(  \sum_{n\in\mathbb{N}%
}H_{n}t^{n}\right)  =\sum_{n\in\mathbb{N}}\Delta\left(  H_{n}\right)  t^{n}
\label{pf.thm.Deltaetastara.DeltaH.c1.pf.2}%
\end{equation}
(by the definition of $\Delta_{t}$). However, for each $n\in\mathbb{N}$, we
have%
\begin{align*}
\Delta\left(  H_{n}\right)   &  =\sum_{\substack{i,j\in\mathbb{N}%
;\\i+j=n}}\ \ \underbrace{H_{i}\otimes H_{j}}_{=\left(  H_{i}\otimes1\right)
\left(  1\otimes H_{j}\right)  }\ \ \ \ \ \ \ \ \ \ \left(  \text{by
\cite[(5.4.2)]{GriRei}}\right) \\
&  =\sum_{\substack{i,j\in\mathbb{N};\\i+j=n}}\left(  H_{i}\otimes1\right)
\left(  1\otimes H_{j}\right)  .
\end{align*}
Hence, the right hand sides of the equalities
(\ref{pf.thm.Deltaetastara.DeltaH.c1.pf.2}) and
(\ref{pf.thm.Deltaetastara.DeltaH.c1.pf.1}) are equal. Therefore, so are their
left hand sides. In other words, we have $\Delta_{t}\left(  \mathbf{H}\right)
=\iota\left(  \mathbf{H}\otimes\mathbf{H}\right)  $. This proves Claim 1.]
\medskip

Define four elements $h_{1}$, $h_{2}$, $g_{1}$ and $g_{2}$ of $\left(
\operatorname*{NSym}\otimes\operatorname*{NSym}\right)  \left[  \left[
t\right]  \right]  $ by%
\begin{align*}
h_{1}  &  =\iota\left(  \mathbf{H}\otimes1\right)
\ \ \ \ \ \ \ \ \ \ \text{and}\ \ \ \ \ \ \ \ \ \ h_{2}=\iota\left(
1\otimes\mathbf{H}\right)  \ \ \ \ \ \ \ \ \ \ \text{and}\\
g_{1}  &  =\iota\left(  \mathbf{G}\otimes1\right)
\ \ \ \ \ \ \ \ \ \ \text{and}\ \ \ \ \ \ \ \ \ \ g_{2}=\iota\left(
1\otimes\mathbf{G}\right)  .
\end{align*}

The equality (\ref{pf.thm.Deltaetastara.DeltaH}) becomes%
\begin{align}
\Delta_{t}\left(  \mathbf{H}\right)   &  =\iota\left(  \underbrace{\mathbf{H}%
\otimes\mathbf{H}}_{=\left(  \mathbf{H}\otimes1\right)  \left(  1\otimes
\mathbf{H}\right)  }\right)  =\iota\left(  \left(  \mathbf{H}\otimes1\right)
\left(  1\otimes\mathbf{H}\right)  \right) \nonumber\\
&  =\underbrace{\iota\left(  \mathbf{H}\otimes1\right)  }_{=h_{1}}%
\cdot\underbrace{\iota\left(  1\otimes\mathbf{H}\right)  }_{=h_{2}%
}\ \ \ \ \ \ \ \ \ \ \left(  \text{since }\iota\text{ is a ring homomorphism}%
\right) \nonumber\\
&  =h_{1}h_{2} \label{pf.thm.Deltaetastara.DeltaH2}%
\end{align}
and
\begin{align*}
\Delta_{t}\left(  \mathbf{H}\right)   &  =\iota\left(  \underbrace{\mathbf{H}%
\otimes\mathbf{H}}_{=\left(  1\otimes\mathbf{H}\right)  \left(  \mathbf{H}%
\otimes1\right)  }\right)  =\iota\left(  \left(  1\otimes\mathbf{H}\right)
\left(  \mathbf{H}\otimes1\right)  \right) \\
&  =\underbrace{\iota\left(  1\otimes\mathbf{H}\right)  }_{=h_{2}}%
\cdot\underbrace{\iota\left(  \mathbf{H}\otimes1\right)  }_{=h_{1}%
}\ \ \ \ \ \ \ \ \ \ \left(  \text{since }\iota\text{ is a ring homomorphism}%
\right) \\
&  =h_{2}h_{1}.
\end{align*}
Comparing these two equalities, we obtain $h_{1}h_{2}=h_{2}h_{1}$. In other
words, the elements $h_{1}$ and $h_{2}$ commute. The elements $\dfrac{1}%
{h_{1}+q}$, $\dfrac{1}{h_{2}+q}$ and $\dfrac{1}{h_{1}h_{2}+q}$ (which are
easily seen to be well-defined\footnote{\textit{Proof.} The power series
$h_{1}=\iota\left(  \mathbf{H}\otimes1\right)  $ has constant term $1$ (since
$\mathbf{H}=\sum_{n\in\mathbb{N}}H_{n}t^{n}$ entails $\iota\left(
\mathbf{H}\otimes1\right)  =\sum_{n\in\mathbb{N}}\left(  H_{n}\otimes1\right)
t^{n}$, and this latter series has constant term $\underbrace{H_{0}}%
_{=1}\otimes1=1\otimes1=1$). Thus, the power series $h_{1}+q$ has constant
term $1+q=q+1=r$, which is invertible (by Convention \ref{conv.r-ible}). Thus,
the power series $h_{1}+q$ itself is invertible (since a formal power series
whose constant term is invertible must itself be invertible). In other words,
$\dfrac{1}{h_{1}+q}$ is well-defined. Similarly, $\dfrac{1}{h_{2}+q}$ and
$\dfrac{1}{h_{1}h_{2}+q}$ are well-defined.}) are rational functions in these
commuting elements $h_{1}$ and $h_{2}$, and therefore also commute with them
(and with each other). Thus, the five elements $h_{1}$, $h_{2}$, $\dfrac
{1}{h_{1}+q}$, $\dfrac{1}{h_{2}+q}$ and $\dfrac{1}{h_{1}h_{2}+q}$ generate a
commutative $\mathbf{k}$-subalgebra of $\left(  \operatorname*{NSym}%
\otimes\operatorname*{NSym}\right)  \left[  \left[  t\right]  \right]  $. Let
us denote this commutative $\mathbf{k}$-subalgebra by $\mathcal{H}$. Clearly,
the elements $h_{1}+q$, $h_{2}+q$ and $h_{1}h_{2}+q$ are invertible in
$\mathcal{H}$. Also, the element $q+1=r$ is invertible in $\mathcal{H}$ (since
it is invertible in $\mathbf{k}$ already).

From (\ref{pf.thm.Deltaetastara.G-via-H}), we obtain
\begin{equation}
g_{1}=\dfrac{h_{1}-1}{h_{1}+q}\ \ \ \ \ \ \ \ \ \ \text{and}%
\ \ \ \ \ \ \ \ \ \ g_{2}=\dfrac{h_{2}-1}{h_{2}+q}
\label{pf.thm.Deltaetastara.g1g2}%
\end{equation}
(since $\iota$ is a $\mathbf{k}$-algebra homomorphism\footnote{Let us give
some more details here: Let $\iota_{1}$ be the map
\begin{align*}
\operatorname*{NSym}\left[  \left[  t\right]  \right]   &  \rightarrow
\operatorname*{NSym}\left[  \left[  t\right]  \right]  \otimes_{\mathbf{k}%
\left[  \left[  t\right]  \right]  }\operatorname*{NSym}\left[  \left[
t\right]  \right]  ,\\
z  &  \mapsto z\otimes1.
\end{align*}
Then, $\iota_{1}$ is a $\mathbf{k}$-algebra homomorphism. Since $\iota$ is a
$\mathbf{k}$-algebra homomorphism as well, we conclude that the composition
$\iota\circ\iota_{1}$ is a $\mathbf{k}$-algebra homomorphism. However, the
definition of $\iota_{1}$ yields $\iota_{1}\left(  \mathbf{G}\right)
=\mathbf{G}\otimes1$, so that
\[
\left(  \iota\circ\iota_{1}\right)  \left(  \mathbf{G}\right)  =\iota\left(
\underbrace{\iota_{1}\left(  \mathbf{G}\right)  }_{=\mathbf{G}\otimes
1}\right)  =\iota\left(  \mathbf{G}\otimes1\right)  =g_{1}%
\]
(by the definition of $g_{1}$). Similarly, $\left(  \iota\circ\iota
_{1}\right)  \left(  \mathbf{H}\right)  =h_{1}$. Now, applying the map
$\iota\circ\iota_{1}$ to both sides of the equality
(\ref{pf.thm.Deltaetastara.G-via-H}), we obtain%
\[
\left(  \iota\circ\iota_{1}\right)  \left(  \mathbf{G}\right)  =\left(
\iota\circ\iota_{1}\right)  \left(  \dfrac{\mathbf{H}-1}{\mathbf{H}+q}\right)
=\dfrac{\left(  \iota\circ\iota_{1}\right)  \left(  \mathbf{H}\right)
-1}{\left(  \iota\circ\iota_{1}\right)  \left(  \mathbf{H}\right)  +q}%
\]
(since $\iota\circ\iota_{1}$ is a $\mathbf{k}$-algebra homomorphism). In view
of $\left(  \iota\circ\iota_{1}\right)  \left(  \mathbf{G}\right)  =g_{1}$ and
$\left(  \iota\circ\iota_{1}\right)  \left(  \mathbf{H}\right)  =h_{1}$, we
can rewrite this as $g_{1}=\dfrac{h_{1}-1}{h_{1}+q}$. Similarly, we can show
that $g_{2}=\dfrac{h_{2}-1}{h_{2}+q}$. Thus, (\ref{pf.thm.Deltaetastara.g1g2})
is proved.}).

Thus, the elements $g_{1}$ and $g_{2}$ also belong to the commutative
$\mathbf{k}$-subalgebra $\mathcal{H}$ generated by $h_{1}$, $h_{2}$,
$\dfrac{1}{h_{1}+q}$, $\dfrac{1}{h_{2}+q}$ and $\dfrac{1}{h_{1}h_{2}+q}$.
Straightforward computations using (\ref{pf.thm.Deltaetastara.g1g2}) (and the
commutativity of $\mathcal{H}$) show that%
\[
1+qg_{1}g_{2}=\dfrac{\left(  q+1\right)  \left(  h_{1}h_{2}+q\right)
}{\left(  h_{1}+q\right)  \left(  h_{2}+q\right)  }.
\]
Thus, $1+qg_{1}g_{2}$ is invertible in $\mathcal{H}$ (since $q+1$, $h_{1}%
h_{2}+q$, $h_{1}+q$ and $h_{2}+q$ are invertible in $\mathcal{H}$).

From (\ref{pf.thm.Deltaetastara.DeltaH2}), we obtain%
\begin{align}
\dfrac{\Delta_{t}\left(  \mathbf{H}\right)  -1}{\Delta_{t}\left(
\mathbf{H}\right)  +q}  &  =\dfrac{h_{1}h_{2}-1}{h_{1}h_{2}+q}\nonumber\\
&  =\dfrac{g_{1}+g_{2}+\left(  q-1\right)  g_{1}g_{2}}{1+qg_{1}g_{2}}.
\label{pf.thm.Deltaetastara.abab2}%
\end{align}
(Indeed, the last equality sign can easily be verified by straightforward
computations in the commutative $\mathbf{k}$-algebra $\mathcal{H}$, using the
equalities (\ref{pf.thm.Deltaetastara.g1g2}). For example, you can plug
(\ref{pf.thm.Deltaetastara.g1g2}) into $\dfrac{g_{1}+g_{2}+\left(  q-1\right)
g_{1}g_{2}}{1+qg_{1}g_{2}}$ and simplify; the result will be $\dfrac
{h_{1}h_{2}-1}{h_{1}h_{2}+q}$.)

From $g_{1}=\iota\left(  \mathbf{G}\otimes1\right)  $ and $g_{2}=\iota\left(
1\otimes\mathbf{G}\right)  $, we obtain%
\begin{align}
g_{1}g_{2}  &  =\iota\left(  \mathbf{G}\otimes1\right)  \cdot\iota\left(
1\otimes\mathbf{G}\right)  =\iota\left(  \underbrace{\left(  \mathbf{G}%
\otimes1\right)  \cdot\left(  1\otimes\mathbf{G}\right)  }_{=\mathbf{G}%
\otimes\mathbf{G}}\right) \nonumber\\
&  \ \ \ \ \ \ \ \ \ \ \ \ \ \ \ \ \ \ \ \ \left(  \text{since }\iota\text{ is
a }\mathbf{k}\text{-algebra homomorphism}\right) \nonumber\\
&  =\iota\left(  \mathbf{G}\otimes\mathbf{G}\right)  .
\label{pf.thm.Deltaetastara.g1g2o}%
\end{align}
Note that the formal power series $qg_{1}g_{2}$ has constant term $0$ (since
it is easy to see that both $g_{1}$ and $g_{2}$ have constant term $0$).

Now, (\ref{pf.thm.Deltaetastara.G-via-H2}) becomes%
\begin{align*}
\Delta_{t}\left(  \mathbf{G}\right)   &  =\dfrac{\Delta_{t}\left(
\mathbf{H}\right)  -1}{\Delta_{t}\left(  \mathbf{H}\right)  +q}=\dfrac
{g_{1}+g_{2}+\left(  q-1\right)  g_{1}g_{2}}{1+qg_{1}g_{2}}%
\ \ \ \ \ \ \ \ \ \ \left(  \text{by (\ref{pf.thm.Deltaetastara.abab2}%
)}\right) \\
&  =\underbrace{\dfrac{1}{1+qg_{1}g_{2}}}_{\substack{=\sum_{i\in\mathbb{N}%
}\left(  -qg_{1}g_{2}\right)  ^{i}\\\text{(by the geometric series formula)}%
}}\cdot\left(  g_{1}+g_{2}+\left(  q-1\right)  g_{1}g_{2}\right) \\
&  =\sum_{i\in\mathbb{N}}\underbrace{\left(  -qg_{1}g_{2}\right)  ^{i}%
}_{=\left(  -q\right)  ^{i}\left(  g_{1}g_{2}\right)  ^{i}}\left(  g_{1}%
+g_{2}+\left(  q-1\right)  g_{1}g_{2}\right) \\
&  =\sum_{i\in\mathbb{N}}\left(  -q\right)  ^{i}\left(  \underbrace{g_{1}%
g_{2}}_{\substack{=\iota\left(  \mathbf{G}\otimes\mathbf{G}\right)
\\\text{(by (\ref{pf.thm.Deltaetastara.g1g2o}))}}}\right)  ^{i}\left(
\underbrace{g_{1}}_{=\iota\left(  \mathbf{G}\otimes1\right)  }%
+\underbrace{g_{2}}_{=\iota\left(  1\otimes\mathbf{G}\right)  }+\left(
q-1\right)  \underbrace{g_{1}g_{2}}_{\substack{=\iota\left(  \mathbf{G}%
\otimes\mathbf{G}\right)  \\\text{(by (\ref{pf.thm.Deltaetastara.g1g2o}))}%
}}\right) \\
&  =\sum_{i\in\mathbb{N}}\left(  -q\right)  ^{i}\underbrace{\left(
\iota\left(  \mathbf{G}\otimes\mathbf{G}\right)  \right)  ^{i}\cdot\left(
\iota\left(  \mathbf{G}\otimes1\right)  +\iota\left(  1\otimes\mathbf{G}%
\right)  +\left(  q-1\right)  \iota\left(  \mathbf{G}\otimes\mathbf{G}\right)
\right)  }_{\substack{=\iota\left(  \left(  \mathbf{G}\otimes\mathbf{G}%
\right)  ^{i}\left(  \mathbf{G}\otimes1+1\otimes\mathbf{G}+\left(  q-1\right)
\mathbf{G}\otimes\mathbf{G}\right)  \right)  \\\text{(since }\iota\text{ is a
}\mathbf{k}\text{-algebra homomorphism)}}}\\
&  =\sum_{i\in\mathbb{N}}\left(  -q\right)  ^{i}\iota\left(  \left(
\mathbf{G}\otimes\mathbf{G}\right)  ^{i}\left(  \mathbf{G}\otimes
1+1\otimes\mathbf{G}+\left(  q-1\right)  \mathbf{G}\otimes\mathbf{G}\right)
\right)  .
\end{align*}

In order to simplify the right hand side, we need two further claims:

\begin{statement}
\textit{Claim 2:} Let $u,v\in\mathbb{N}$. Then,%
\[
\iota\left(  \mathbf{G}^{u}\otimes\mathbf{G}^{v}\right)  =\sum
_{\substack{\beta,\gamma\in\operatorname*{Comp};\\\ell\left(  \beta\right)
=u\text{ and }\ell\left(  \gamma\right)  =v}}\left(  \eta_{\beta}^{\ast\left(
q\right)  }\otimes\eta_{\gamma}^{\ast\left(  q\right)  }\right)  t^{\left\vert
\beta\right\vert +\left\vert \gamma\right\vert }.
\]

\end{statement}

[\textit{Proof of Claim 2:} From $\mathbf{G}=G\left(  t\right)  $, we obtain%
\begin{align}
\mathbf{G}^{u}  &  =G\left(  t\right)  ^{u}=\sum_{\substack{\beta
\in\operatorname*{Comp};\\\ell\left(  \beta\right)  =u}}\eta_{\beta}%
^{\ast\left(  q\right)  }t^{\left\vert \beta\right\vert }%
\ \ \ \ \ \ \ \ \ \ \left(  \text{by (\ref{pf.thm.Deltaetastara.Gtk}), applied
to }k=u\right) \nonumber\\
&  =\sum_{i\in\mathbb{N}}\ \ \sum_{\substack{\beta\in\operatorname*{Comp}%
;\\\ell\left(  \beta\right)  =u;\\\left\vert \beta\right\vert =i}}\eta_{\beta
}^{\ast\left(  q\right)  }\underbrace{t^{\left\vert \beta\right\vert }%
}_{\substack{=t^{i}\\\text{(since }\left\vert \beta\right\vert =i\text{)}%
}}\nonumber\\
&  =\sum_{i\in\mathbb{N}}\left(  \sum_{\substack{\beta\in\operatorname*{Comp}%
;\\\ell\left(  \beta\right)  =u;\\\left\vert \beta\right\vert =i}}\eta_{\beta
}^{\ast\left(  q\right)  }\right)  t^{i}.
\label{pf.thm.Deltaetastara.g1g2o.c2.pf.beta}%
\end{align}
The same argument (applied to $v$ instead of $u$) shows that%
\[
\mathbf{G}^{v}=\sum_{i\in\mathbb{N}}\left(  \sum_{\substack{\beta
\in\operatorname*{Comp};\\\ell\left(  \beta\right)  =v;\\\left\vert
\beta\right\vert =i}}\eta_{\beta}^{\ast\left(  q\right)  }\right)  t^{i}%
=\sum_{j\in\mathbb{N}}\left(  \sum_{\substack{\gamma\in\operatorname*{Comp}%
;\\\ell\left(  \gamma\right)  =v;\\\left\vert \gamma\right\vert =j}%
}\eta_{\gamma}^{\ast\left(  q\right)  }\right)  t^{j}%
\]
(here, we have renamed the summation indices $i$ and $\beta$ as $j$ and
$\gamma$). Substituting these two equalities into $\mathbf{G}^{u}%
\otimes\mathbf{G}^{v}$, we obtain%
\[
\mathbf{G}^{u}\otimes\mathbf{G}^{v}=\left(  \sum_{i\in\mathbb{N}}\left(
\sum_{\substack{\beta\in\operatorname*{Comp};\\\ell\left(  \beta\right)
=u;\\\left\vert \beta\right\vert =i}}\eta_{\beta}^{\ast\left(  q\right)
}\right)  t^{i}\right)  \otimes\left(  \sum_{j\in\mathbb{N}}\left(
\sum_{\substack{\gamma\in\operatorname*{Comp};\\\ell\left(  \gamma\right)
=v;\\\left\vert \gamma\right\vert =j}}\eta_{\gamma}^{\ast\left(  q\right)
}\right)  t^{j}\right)
\]
(note that the two outer sums here are infinite, so that we cannot simply
expand them out using the bilinearity of the tensor product). Applying the map
$\iota$ to both sides of this equality, we obtain%
\begin{align*}
\iota\left(  \mathbf{G}^{u}\otimes\mathbf{G}^{v}\right)   &  =\iota\left(
\left(  \sum_{i\in\mathbb{N}}\left(  \sum_{\substack{\beta\in
\operatorname*{Comp};\\\ell\left(  \beta\right)  =u;\\\left\vert
\beta\right\vert =i}}\eta_{\beta}^{\ast\left(  q\right)  }\right)
t^{i}\right)  \otimes\left(  \sum_{j\in\mathbb{N}}\left(  \sum
_{\substack{\gamma\in\operatorname*{Comp};\\\ell\left(  \gamma\right)
=v;\\\left\vert \gamma\right\vert =j}}\eta_{\gamma}^{\ast\left(  q\right)
}\right)  t^{j}\right)  \right) \\
&  =\left(  \sum_{i\in\mathbb{N}}\left(  \left(  \sum_{\substack{\beta
\in\operatorname*{Comp};\\\ell\left(  \beta\right)  =u;\\\left\vert
\beta\right\vert =i}}\eta_{\beta}^{\ast\left(  q\right)  }\right)
\otimes1\right)  t^{i}\right)  \left(  \sum_{j\in\mathbb{N}}\left(
1\otimes\left(  \sum_{\substack{\gamma\in\operatorname*{Comp};\\\ell\left(
\gamma\right)  =v;\\\left\vert \gamma\right\vert =j}}\eta_{\gamma}%
^{\ast\left(  q\right)  }\right)  \right)  t^{j}\right) \\
&  \ \ \ \ \ \ \ \ \ \ \ \ \ \ \ \ \ \ \ \ \left(  \text{by the definition of
}\iota\right) \\
&  =\sum_{i\in\mathbb{N}}\ \ \sum_{j\in\mathbb{N}}\underbrace{\left(  \left(
\sum_{\substack{\beta\in\operatorname*{Comp};\\\ell\left(  \beta\right)
=u;\\\left\vert \beta\right\vert =i}}\eta_{\beta}^{\ast\left(  q\right)
}\right)  \otimes1\right)  \left(  1\otimes\left(  \sum_{\substack{\gamma
\in\operatorname*{Comp};\\\ell\left(  \gamma\right)  =v;\\\left\vert
\gamma\right\vert =j}}\eta_{\gamma}^{\ast\left(  q\right)  }\right)  \right)
}_{\substack{=\left(  \sum_{\substack{\beta\in\operatorname*{Comp}%
;\\\ell\left(  \beta\right)  =u;\\\left\vert \beta\right\vert =i}}\eta_{\beta
}^{\ast\left(  q\right)  }\right)  \otimes\left(  \sum_{\substack{\gamma
\in\operatorname*{Comp};\\\ell\left(  \gamma\right)  =v;\\\left\vert
\gamma\right\vert =j}}\eta_{\gamma}^{\ast\left(  q\right)  }\right)
\\=\sum_{\substack{\beta,\gamma\in\operatorname*{Comp};\\\ell\left(
\beta\right)  =u\text{ and }\ell\left(  \gamma\right)  =v;\\\left\vert
\beta\right\vert =i\text{ and }\left\vert \gamma\right\vert =j}}\eta_{\beta
}^{\ast\left(  q\right)  }\otimes\eta_{\gamma}^{\ast\left(  q\right)
}\\\text{(here, we expanded the tensor product, }\\\text{since both sums
involved are finite)}}}\underbrace{t^{i}t^{j}}_{=t^{i+j}}\\
&  =\sum_{i\in\mathbb{N}}\ \ \sum_{j\in\mathbb{N}}\ \ \sum_{\substack{\beta
,\gamma\in\operatorname*{Comp};\\\ell\left(  \beta\right)  =u\text{ and }%
\ell\left(  \gamma\right)  =v;\\\left\vert \beta\right\vert =i\text{ and
}\left\vert \gamma\right\vert =j}}\left(  \eta_{\beta}^{\ast\left(  q\right)
}\otimes\eta_{\gamma}^{\ast\left(  q\right)  }\right)  \underbrace{t^{i+j}%
}_{\substack{=t^{\left\vert \beta\right\vert +\left\vert \gamma\right\vert
}\\\text{(since }i=\left\vert \beta\right\vert \text{ and }j=\left\vert
\gamma\right\vert \text{)}}}\\
&  =\underbrace{\sum_{i\in\mathbb{N}}\ \ \sum_{j\in\mathbb{N}}\ \ \sum
_{\substack{\beta,\gamma\in\operatorname*{Comp};\\\ell\left(  \beta\right)
=u\text{ and }\ell\left(  \gamma\right)  =v;\\\left\vert \beta\right\vert
=i\text{ and }\left\vert \gamma\right\vert =j}}}_{=\sum_{\substack{\beta
,\gamma\in\operatorname*{Comp};\\\ell\left(  \beta\right)  =u\text{ and }%
\ell\left(  \gamma\right)  =v}}}\left(  \eta_{\beta}^{\ast\left(  q\right)
}\otimes\eta_{\gamma}^{\ast\left(  q\right)  }\right)  t^{\left\vert
\beta\right\vert +\left\vert \gamma\right\vert }\\
&  =\sum_{\substack{\beta,\gamma\in\operatorname*{Comp};\\\ell\left(
\beta\right)  =u\text{ and }\ell\left(  \gamma\right)  =v}}\left(  \eta
_{\beta}^{\ast\left(  q\right)  }\otimes\eta_{\gamma}^{\ast\left(  q\right)
}\right)  t^{\left\vert \beta\right\vert +\left\vert \gamma\right\vert }.
\end{align*}
This proves Claim 2.]

\begin{statement}
\textit{Claim 3:} Let $i\in\mathbb{N}$. Then,%
\begin{align*}
&  \iota\left(  \left(  \mathbf{G}\otimes\mathbf{G}\right)  ^{i}\left(
\mathbf{G}\otimes1+1\otimes\mathbf{G}+\left(  q-1\right)  \mathbf{G}%
\otimes\mathbf{G}\right)  \right) \\
&  =\sum_{\substack{\beta,\gamma\in\operatorname*{Comp};\\\left\vert
\ell\left(  \beta\right)  -\ell\left(  \gamma\right)  \right\vert \leq
1;\\\max\left\{  \ell\left(  \beta\right)  ,\ell\left(  \gamma\right)
\right\}  =i+1}}\left(  q-1\right)  ^{\left[  \ell\left(  \beta\right)
=\ell\left(  \gamma\right)  \right]  }\left(  \eta_{\beta}^{\ast\left(
q\right)  }\otimes\eta_{\gamma}^{\ast\left(  q\right)  }\right)  t^{\left\vert
\beta\right\vert +\left\vert \gamma\right\vert }.
\end{align*}

\end{statement}

[\textit{Proof of Claim 3:} We have%
\begin{align*}
&  \underbrace{\left(  \mathbf{G}\otimes\mathbf{G}\right)  ^{i}}%
_{=\mathbf{G}^{i}\otimes\mathbf{G}^{i}}\left(  \mathbf{G}\otimes
1+1\otimes\mathbf{G}+\left(  q-1\right)  \mathbf{G}\otimes\mathbf{G}\right) \\
&  =\left(  \mathbf{G}^{i}\otimes\mathbf{G}^{i}\right)  \left(  \mathbf{G}%
\otimes1+1\otimes\mathbf{G}+\left(  q-1\right)  \mathbf{G}\otimes
\mathbf{G}\right) \\
&  =\mathbf{G}^{i+1}\otimes\mathbf{G}^{i}+\mathbf{G}^{i}\otimes\mathbf{G}%
^{i+1}+\left(  q-1\right)  \mathbf{G}^{i+1}\otimes\mathbf{G}^{i+1}.
\end{align*}
Applying the map $\iota$ to both sides of this equality, we obtain%
\begin{align}
&  \iota\left(  \left(  \mathbf{G}\otimes\mathbf{G}\right)  ^{i}\left(
\mathbf{G}\otimes1+1\otimes\mathbf{G}+\left(  q-1\right)  \mathbf{G}%
\otimes\mathbf{G}\right)  \right) \nonumber\\
&  =\iota\left(  \mathbf{G}^{i+1}\otimes\mathbf{G}^{i}+\mathbf{G}^{i}%
\otimes\mathbf{G}^{i+1}+\left(  q-1\right)  \mathbf{G}^{i+1}\otimes
\mathbf{G}^{i+1}\right) \nonumber\\
&  =\underbrace{\iota\left(  \mathbf{G}^{i+1}\otimes\mathbf{G}^{i}\right)
}_{\substack{=\sum_{\substack{\beta,\gamma\in\operatorname*{Comp}%
;\\\ell\left(  \beta\right)  =i+1\text{ and }\ell\left(  \gamma\right)
=i}}\left(  \eta_{\beta}^{\ast\left(  q\right)  }\otimes\eta_{\gamma}%
^{\ast\left(  q\right)  }\right)  t^{\left\vert \beta\right\vert +\left\vert
\gamma\right\vert }\\\text{(by Claim 2)}}}+\underbrace{\iota\left(
\mathbf{G}^{i}\otimes\mathbf{G}^{i+1}\right)  }_{\substack{=\sum
_{\substack{\beta,\gamma\in\operatorname*{Comp};\\\ell\left(  \beta\right)
=i\text{ and }\ell\left(  \gamma\right)  =i+1}}\left(  \eta_{\beta}%
^{\ast\left(  q\right)  }\otimes\eta_{\gamma}^{\ast\left(  q\right)  }\right)
t^{\left\vert \beta\right\vert +\left\vert \gamma\right\vert }\\\text{(by
Claim 2)}}}\nonumber\\
&  \ \ \ \ \ \ \ \ \ \ +\left(  q-1\right)  \underbrace{\iota\left(
\mathbf{G}^{i+1}\otimes\mathbf{G}^{i+1}\right)  }_{\substack{=\sum
_{\substack{\beta,\gamma\in\operatorname*{Comp};\\\ell\left(  \beta\right)
=i+1\text{ and }\ell\left(  \gamma\right)  =i+1}}\left(  \eta_{\beta}%
^{\ast\left(  q\right)  }\otimes\eta_{\gamma}^{\ast\left(  q\right)  }\right)
t^{\left\vert \beta\right\vert +\left\vert \gamma\right\vert }\\\text{(by
Claim 2)}}}\nonumber\\
&  \ \ \ \ \ \ \ \ \ \ \ \ \ \ \ \ \ \ \ \ \left(  \text{since }\iota\text{ is
a }\mathbf{k}\text{-algebra homomorphism}\right) \nonumber\\
&  =\sum_{\substack{\beta,\gamma\in\operatorname*{Comp};\\\ell\left(
\beta\right)  =i+1\text{ and }\ell\left(  \gamma\right)  =i}}\left(
\eta_{\beta}^{\ast\left(  q\right)  }\otimes\eta_{\gamma}^{\ast\left(
q\right)  }\right)  t^{\left\vert \beta\right\vert +\left\vert \gamma
\right\vert }+\sum_{\substack{\beta,\gamma\in\operatorname*{Comp}%
;\\\ell\left(  \beta\right)  =i\text{ and }\ell\left(  \gamma\right)
=i+1}}\left(  \eta_{\beta}^{\ast\left(  q\right)  }\otimes\eta_{\gamma}%
^{\ast\left(  q\right)  }\right)  t^{\left\vert \beta\right\vert +\left\vert
\gamma\right\vert }\nonumber\\
&  \ \ \ \ \ \ \ \ \ \ +\left(  q-1\right)  \sum_{\substack{\beta,\gamma
\in\operatorname*{Comp};\\\ell\left(  \beta\right)  =i+1\text{ and }%
\ell\left(  \gamma\right)  =i+1}}\left(  \eta_{\beta}^{\ast\left(  q\right)
}\otimes\eta_{\gamma}^{\ast\left(  q\right)  }\right)  t^{\left\vert
\beta\right\vert +\left\vert \gamma\right\vert }.
\label{pf.thm.Deltaetastara.g1g2o.c3.pf.2}%
\end{align}

On the other hand, let us observe that two integers $u$ and $v$ satisfy the
two conditions%
\[
\left\vert u-v\right\vert \leq1\ \ \ \ \ \ \ \ \ \ \text{and}%
\ \ \ \ \ \ \ \ \ \ \max\left\{  u,v\right\}  =i+1
\]
if and only if they satisfy one of the three mutually exclusive conditions%
\begin{align*}
&  \left(  u=i+1\text{ and }v=i\right)  ,\\
&  \left(  u=i\text{ and }v=i+1\right)  \ \ \ \ \ \ \ \ \ \ \text{and}\\
&  \left(  u=i+1\text{ and }v=i+1\right)  .
\end{align*}
Hence, two compositions $\beta,\gamma\in\operatorname*{Comp}$ satisfy the two
conditions%
\[
\left\vert \ell\left(  \beta\right)  -\ell\left(  \gamma\right)  \right\vert
\leq1\ \ \ \ \ \ \ \ \ \ \text{and}\ \ \ \ \ \ \ \ \ \ \max\left\{
\ell\left(  \beta\right)  ,\ell\left(  \gamma\right)  \right\}  =i+1
\]
if and only if they satisfy one of the three mutually exclusive conditions%
\begin{align*}
&  \left(  \ell\left(  \beta\right)  =i+1\text{ and }\ell\left(
\gamma\right)  =i\right)  ,\\
&  \left(  \ell\left(  \beta\right)  =i\text{ and }\ell\left(  \gamma\right)
=i+1\right)  \ \ \ \ \ \ \ \ \ \ \text{and}\\
&  \left(  \ell\left(  \beta\right)  =i+1\text{ and }\ell\left(
\gamma\right)  =i+1\right)  .
\end{align*}
Hence,%
\begin{align*}
&  \sum_{\substack{\beta,\gamma\in\operatorname*{Comp};\\\left\vert
\ell\left(  \beta\right)  -\ell\left(  \gamma\right)  \right\vert \leq
1;\\\max\left\{  \ell\left(  \beta\right)  ,\ell\left(  \gamma\right)
\right\}  =i+1}}\left(  q-1\right)  ^{\left[  \ell\left(  \beta\right)
=\ell\left(  \gamma\right)  \right]  }\left(  \eta_{\beta}^{\ast\left(
q\right)  }\otimes\eta_{\gamma}^{\ast\left(  q\right)  }\right)  t^{\left\vert
\beta\right\vert +\left\vert \gamma\right\vert }\\
&  =\sum_{\substack{\beta,\gamma\in\operatorname*{Comp};\\\ell\left(
\beta\right)  =i+1\text{ and }\ell\left(  \gamma\right)  =i}%
}\ \ \underbrace{\left(  q-1\right)  ^{\left[  \ell\left(  \beta\right)
=\ell\left(  \gamma\right)  \right]  }}_{\substack{=1\\\text{(since }%
\ell\left(  \beta\right)  =i+1>i=\ell\left(  \gamma\right)  \\\text{and thus
}\left[  \ell\left(  \beta\right)  =\ell\left(  \gamma\right)  \right]
=0\text{)}}}\left(  \eta_{\beta}^{\ast\left(  q\right)  }\otimes\eta_{\gamma
}^{\ast\left(  q\right)  }\right)  t^{\left\vert \beta\right\vert +\left\vert
\gamma\right\vert }\\
&  \ \ \ \ \ \ \ \ \ \ +\sum_{\substack{\beta,\gamma\in\operatorname*{Comp}%
;\\\ell\left(  \beta\right)  =i\text{ and }\ell\left(  \gamma\right)
=i+1}}\ \ \underbrace{\left(  q-1\right)  ^{\left[  \ell\left(  \beta\right)
=\ell\left(  \gamma\right)  \right]  }}_{\substack{=1\\\text{(since }%
\ell\left(  \beta\right)  =i<i+1=\ell\left(  \gamma\right)  \\\text{and thus
}\left[  \ell\left(  \beta\right)  =\ell\left(  \gamma\right)  \right]
=0\text{)}}}\left(  \eta_{\beta}^{\ast\left(  q\right)  }\otimes\eta_{\gamma
}^{\ast\left(  q\right)  }\right)  t^{\left\vert \beta\right\vert +\left\vert
\gamma\right\vert }\\
&  \ \ \ \ \ \ \ \ \ \ +\sum_{\substack{\beta,\gamma\in\operatorname*{Comp}%
;\\\ell\left(  \beta\right)  =i+1\text{ and }\ell\left(  \gamma\right)
=i+1}}\ \ \underbrace{\left(  q-1\right)  ^{\left[  \ell\left(  \beta\right)
=\ell\left(  \gamma\right)  \right]  }}_{\substack{=q-1\\\text{(since }%
\ell\left(  \beta\right)  =i+1=\ell\left(  \gamma\right)  \\\text{and thus
}\left[  \ell\left(  \beta\right)  =\ell\left(  \gamma\right)  \right]
=1\text{)}}}\left(  \eta_{\beta}^{\ast\left(  q\right)  }\otimes\eta_{\gamma
}^{\ast\left(  q\right)  }\right)  t^{\left\vert \beta\right\vert +\left\vert
\gamma\right\vert }\\
&  =\sum_{\substack{\beta,\gamma\in\operatorname*{Comp};\\\ell\left(
\beta\right)  =i+1\text{ and }\ell\left(  \gamma\right)  =i}}\left(
\eta_{\beta}^{\ast\left(  q\right)  }\otimes\eta_{\gamma}^{\ast\left(
q\right)  }\right)  t^{\left\vert \beta\right\vert +\left\vert \gamma
\right\vert }+\sum_{\substack{\beta,\gamma\in\operatorname*{Comp}%
;\\\ell\left(  \beta\right)  =i\text{ and }\ell\left(  \gamma\right)
=i+1}}\left(  \eta_{\beta}^{\ast\left(  q\right)  }\otimes\eta_{\gamma}%
^{\ast\left(  q\right)  }\right)  t^{\left\vert \beta\right\vert +\left\vert
\gamma\right\vert }\\
&  \ \ \ \ \ \ \ \ \ \ +\left(  q-1\right)  \sum_{\substack{\beta,\gamma
\in\operatorname*{Comp};\\\ell\left(  \beta\right)  =i+1\text{ and }%
\ell\left(  \gamma\right)  =i+1}}\left(  \eta_{\beta}^{\ast\left(  q\right)
}\otimes\eta_{\gamma}^{\ast\left(  q\right)  }\right)  t^{\left\vert
\beta\right\vert +\left\vert \gamma\right\vert }.
\end{align*}
Comparing this with (\ref{pf.thm.Deltaetastara.g1g2o.c3.pf.2}), we obtain%
\begin{align*}
&  \iota\left(  \left(  \mathbf{G}\otimes\mathbf{G}\right)  ^{i}\left(
\mathbf{G}\otimes1+1\otimes\mathbf{G}+\left(  q-1\right)  \mathbf{G}%
\otimes\mathbf{G}\right)  \right) \\
&  =\sum_{\substack{\beta,\gamma\in\operatorname*{Comp};\\\left\vert
\ell\left(  \beta\right)  -\ell\left(  \gamma\right)  \right\vert \leq
1;\\\max\left\{  \ell\left(  \beta\right)  ,\ell\left(  \gamma\right)
\right\}  =i+1}}\left(  q-1\right)  ^{\left[  \ell\left(  \beta\right)
=\ell\left(  \gamma\right)  \right]  }\left(  \eta_{\beta}^{\ast\left(
q\right)  }\otimes\eta_{\gamma}^{\ast\left(  q\right)  }\right)  t^{\left\vert
\beta\right\vert +\left\vert \gamma\right\vert }.
\end{align*}
This proves Claim 3.] \medskip

Now, we can finish our computation of $\Delta_{t}\left(  \mathbf{G}\right)  $:
As we know,
\begin{align*}
&  \Delta_{t}\left(  \mathbf{G}\right) \\
&  =\sum_{i\in\mathbb{N}}\left(  -q\right)  ^{i}\underbrace{\iota\left(
\left(  \mathbf{G}\otimes\mathbf{G}\right)  ^{i}\left(  \mathbf{G}%
\otimes1+1\otimes\mathbf{G}+\left(  q-1\right)  \mathbf{G}\otimes
\mathbf{G}\right)  \right)  }_{\substack{=\sum_{\substack{\beta,\gamma
\in\operatorname*{Comp};\\\left\vert \ell\left(  \beta\right)  -\ell\left(
\gamma\right)  \right\vert \leq1;\\\max\left\{  \ell\left(  \beta\right)
,\ell\left(  \gamma\right)  \right\}  =i+1}}\left(  q-1\right)  ^{\left[
\ell\left(  \beta\right)  =\ell\left(  \gamma\right)  \right]  }\left(
\eta_{\beta}^{\ast\left(  q\right)  }\otimes\eta_{\gamma}^{\ast\left(
q\right)  }\right)  t^{\left\vert \beta\right\vert +\left\vert \gamma
\right\vert }\\\text{(by Claim 3)}}}\\
&  =\sum_{i\in\mathbb{N}}\left(  -q\right)  ^{i}\sum_{\substack{\beta
,\gamma\in\operatorname*{Comp};\\\left\vert \ell\left(  \beta\right)
-\ell\left(  \gamma\right)  \right\vert \leq1;\\\max\left\{  \ell\left(
\beta\right)  ,\ell\left(  \gamma\right)  \right\}  =i+1}}\left(  q-1\right)
^{\left[  \ell\left(  \beta\right)  =\ell\left(  \gamma\right)  \right]
}\left(  \eta_{\beta}^{\ast\left(  q\right)  }\otimes\eta_{\gamma}%
^{\ast\left(  q\right)  }\right)  t^{\left\vert \beta\right\vert +\left\vert
\gamma\right\vert }\\
&  =\sum_{j>0}\left(  -q\right)  ^{j-1}\sum_{\substack{\beta,\gamma
\in\operatorname*{Comp};\\\left\vert \ell\left(  \beta\right)  -\ell\left(
\gamma\right)  \right\vert \leq1;\\\max\left\{  \ell\left(  \beta\right)
,\ell\left(  \gamma\right)  \right\}  =j}}\left(  q-1\right)  ^{\left[
\ell\left(  \beta\right)  =\ell\left(  \gamma\right)  \right]  }\left(
\eta_{\beta}^{\ast\left(  q\right)  }\otimes\eta_{\gamma}^{\ast\left(
q\right)  }\right)  t^{\left\vert \beta\right\vert +\left\vert \gamma
\right\vert }\\
&  \ \ \ \ \ \ \ \ \ \ \ \ \ \ \ \ \ \ \ \ \left(  \text{here, we have
substituted }j-1\text{ for }i\text{ in the outer sum}\right) \\
&  =\sum_{j>0}\ \ \sum_{\substack{\beta,\gamma\in\operatorname*{Comp}%
;\\\left\vert \ell\left(  \beta\right)  -\ell\left(  \gamma\right)
\right\vert \leq1;\\\max\left\{  \ell\left(  \beta\right)  ,\ell\left(
\gamma\right)  \right\}  =j}}\ \ \underbrace{\left(  -q\right)  ^{j-1}%
}_{\substack{=\left(  -q\right)  ^{\max\left\{  \ell\left(  \beta\right)
,\ell\left(  \gamma\right)  \right\}  -1}\\\text{(since }j=\max\left\{
\ell\left(  \beta\right)  ,\ell\left(  \gamma\right)  \right\}  \text{)}%
}}\left(  q-1\right)  ^{\left[  \ell\left(  \beta\right)  =\ell\left(
\gamma\right)  \right]  }\left(  \eta_{\beta}^{\ast\left(  q\right)  }%
\otimes\eta_{\gamma}^{\ast\left(  q\right)  }\right)  t^{\left\vert
\beta\right\vert +\left\vert \gamma\right\vert }\\
&  =\underbrace{\sum_{j>0}\ \ \sum_{\substack{\beta,\gamma\in
\operatorname*{Comp};\\\left\vert \ell\left(  \beta\right)  -\ell\left(
\gamma\right)  \right\vert \leq1;\\\max\left\{  \ell\left(  \beta\right)
,\ell\left(  \gamma\right)  \right\}  =j}}}_{=\sum_{\substack{\beta,\gamma
\in\operatorname*{Comp};\\\left\vert \ell\left(  \beta\right)  -\ell\left(
\gamma\right)  \right\vert \leq1;\\\max\left\{  \ell\left(  \beta\right)
,\ell\left(  \gamma\right)  \right\}  >0}}}\left(  -q\right)  ^{\max\left\{
\ell\left(  \beta\right)  ,\ell\left(  \gamma\right)  \right\}  -1}\left(
q-1\right)  ^{\left[  \ell\left(  \beta\right)  =\ell\left(  \gamma\right)
\right]  }\left(  \eta_{\beta}^{\ast\left(  q\right)  }\otimes\eta_{\gamma
}^{\ast\left(  q\right)  }\right)  t^{\left\vert \beta\right\vert +\left\vert
\gamma\right\vert }\\
&  =\sum_{\substack{\beta,\gamma\in\operatorname*{Comp};\\\left\vert
\ell\left(  \beta\right)  -\ell\left(  \gamma\right)  \right\vert \leq
1;\\\max\left\{  \ell\left(  \beta\right)  ,\ell\left(  \gamma\right)
\right\}  >0}}\left(  -q\right)  ^{\max\left\{  \ell\left(  \beta\right)
,\ell\left(  \gamma\right)  \right\}  -1}\left(  q-1\right)  ^{\left[
\ell\left(  \beta\right)  =\ell\left(  \gamma\right)  \right]  }\left(
\eta_{\beta}^{\ast\left(  q\right)  }\otimes\eta_{\gamma}^{\ast\left(
q\right)  }\right)  t^{\left\vert \beta\right\vert +\left\vert \gamma
\right\vert }.
\end{align*}
Comparing this with%
\begin{align*}
\Delta_{t}\left(  \mathbf{G}\right)   &  =\Delta_{t}\left(  \sum_{n\geq1}%
\eta_{n}^{\ast\left(  q\right)  }t^{n}\right)  \ \ \ \ \ \ \ \ \ \ \left(
\text{since }\mathbf{G}=G\left(  t\right)  =\sum_{n\geq1}\eta_{n}^{\ast\left(
q\right)  }t^{n}\right) \\
&  =\sum_{n\geq1}\Delta\left(  \eta_{n}^{\ast\left(  q\right)  }\right)
t^{n}\ \ \ \ \ \ \ \ \ \ \left(  \text{by the definition of }\Delta
_{t}\right)  ,
\end{align*}
we obtain%
\begin{align*}
&  \sum_{n\geq1}\Delta\left(  \eta_{n}^{\ast\left(  q\right)  }\right)
t^{n}\\
&  =\sum_{\substack{\beta,\gamma\in\operatorname*{Comp};\\\left\vert
\ell\left(  \beta\right)  -\ell\left(  \gamma\right)  \right\vert \leq
1;\\\max\left\{  \ell\left(  \beta\right)  ,\ell\left(  \gamma\right)
\right\}  >0}}\left(  -q\right)  ^{\max\left\{  \ell\left(  \beta\right)
,\ell\left(  \gamma\right)  \right\}  -1}\left(  q-1\right)  ^{\left[
\ell\left(  \beta\right)  =\ell\left(  \gamma\right)  \right]  }\left(
\eta_{\beta}^{\ast\left(  q\right)  }\otimes\eta_{\gamma}^{\ast\left(
q\right)  }\right)  t^{\left\vert \beta\right\vert +\left\vert \gamma
\right\vert }.
\end{align*}
Comparing coefficients of $t^{n}$ on both sides of this equality, we obtain%
\begin{equation}
\Delta\left(  \eta_{n}^{\ast\left(  q\right)  }\right)  =\sum_{\substack{\beta
,\gamma\in\operatorname*{Comp};\\\left\vert \ell\left(  \beta\right)
-\ell\left(  \gamma\right)  \right\vert \leq1;\\\max\left\{  \ell\left(
\beta\right)  ,\ell\left(  \gamma\right)  \right\}  >0;\\\left\vert
\beta\right\vert +\left\vert \gamma\right\vert =n}}\left(  -q\right)
^{\max\left\{  \ell\left(  \beta\right)  ,\ell\left(  \gamma\right)  \right\}
-1}\left(  q-1\right)  ^{\left[  \ell\left(  \beta\right)  =\ell\left(
\gamma\right)  \right]  }\eta_{\beta}^{\ast\left(  q\right)  }\otimes
\eta_{\gamma}^{\ast\left(  q\right)  } \label{pf.thm.Deltaetastara.at}%
\end{equation}
for each $n\geq1$ (in fact, the condition \textquotedblleft$\left\vert
\beta\right\vert +\left\vert \gamma\right\vert =n$\textquotedblright\ under
the summation sign ensures that the monomial $t^{\left\vert \beta\right\vert
+\left\vert \gamma\right\vert }$ is $t^{n}$).

Now, fix a positive integer $n$. Then, any two compositions $\beta$ and
$\gamma$ that satisfy $\left\vert \beta\right\vert +\left\vert \gamma
\right\vert =n$ will automatically satisfy $\max\left\{  \ell\left(
\beta\right)  ,\ell\left(  \gamma\right)  \right\}  >0$ (since otherwise, they
would satisfy $\max\left\{  \ell\left(  \beta\right)  ,\ell\left(
\gamma\right)  \right\}  =0$ and thus $\beta=\varnothing$ and $\gamma
=\varnothing$, which would lead to $\left\vert \beta\right\vert +\left\vert
\gamma\right\vert =\left\vert \varnothing\right\vert +\left\vert
\varnothing\right\vert =0+0=0$, but this would contradict $\left\vert
\beta\right\vert +\left\vert \gamma\right\vert =n>0$). Hence, in the summation
sign on the right hand side of (\ref{pf.thm.Deltaetastara.at}), the condition
\textquotedblleft$\max\left\{  \ell\left(  \beta\right)  ,\ell\left(
\gamma\right)  \right\}  >0$\textquotedblright\ is redundant. We can thus
rewrite this summation sign as follows:%
\[
\sum_{\substack{\beta,\gamma\in\operatorname*{Comp};\\\left\vert \ell\left(
\beta\right)  -\ell\left(  \gamma\right)  \right\vert \leq1;\\\max\left\{
\ell\left(  \beta\right)  ,\ell\left(  \gamma\right)  \right\}
>0;\\\left\vert \beta\right\vert +\left\vert \gamma\right\vert =n}%
}=\sum_{\substack{\beta,\gamma\in\operatorname*{Comp};\\\left\vert \ell\left(
\beta\right)  -\ell\left(  \gamma\right)  \right\vert \leq1;\\\left\vert
\beta\right\vert +\left\vert \gamma\right\vert =n}}=\sum_{\substack{\beta
,\gamma\in\operatorname*{Comp};\\\left\vert \beta\right\vert +\left\vert
\gamma\right\vert =n;\\\left\vert \ell\left(  \beta\right)  -\ell\left(
\gamma\right)  \right\vert \leq1}}.
\]
Hence, (\ref{pf.thm.Deltaetastara.at}) rewrites as%
\[
\Delta\left(  \eta_{n}^{\ast\left(  q\right)  }\right)  =\sum_{\substack{\beta
,\gamma\in\operatorname*{Comp};\\\left\vert \beta\right\vert +\left\vert
\gamma\right\vert =n;\\\left\vert \ell\left(  \beta\right)  -\ell\left(
\gamma\right)  \right\vert \leq1}}\left(  -q\right)  ^{\max\left\{
\ell\left(  \beta\right)  ,\ell\left(  \gamma\right)  \right\}  -1}\left(
q-1\right)  ^{\left[  \ell\left(  \beta\right)  =\ell\left(  \gamma\right)
\right]  }\eta_{\beta}^{\ast\left(  q\right)  }\otimes\eta_{\gamma}%
^{\ast\left(  q\right)  }.
\]
This proves Theorem \ref{thm.Deltaetastara}. \qedhere

\begin{verlong}
T0D0: verlong proof.
\end{verlong}
\end{proof}

Using Theorem \ref{thm.Deltaetastara}, we can easily compute the coproduct of
any $\eta_{\alpha}^{\ast\left(  q\right)  }$:\footnote{The symbol
\textquotedblleft\#\textquotedblright\ means \textquotedblleft
number\textquotedblright. Thus, e.g., we have $\left(  \text{\# of odd numbers
}i\in\left[  2n\right]  \right)  =n$ for each $n\in\mathbb{N}$.}

\begin{corollary}
\label{cor.Deltaetastaral}Let $\alpha=\left(  \alpha_{1},\alpha_{2}%
,\ldots,\alpha_{k}\right)  $ be any composition. Then,%
\begin{align*}
\Delta\left(  \eta_{\alpha}^{\ast\left(  q\right)  }\right)   &
=\sum_{\substack{\beta_{1},\beta_{2},\ldots,\beta_{k}\in\operatorname*{Comp}%
;\\\gamma_{1},\gamma_{2},\ldots,\gamma_{k}\in\operatorname*{Comp};\\\left\vert
\beta_{s}\right\vert +\left\vert \gamma_{s}\right\vert =\alpha_{s}\text{ for
all }s;\\\left\vert \ell\left(  \beta_{s}\right)  -\ell\left(  \gamma
_{s}\right)  \right\vert \leq1\text{ for all }s}}\left(  -q\right)
^{\sum_{s=1}^{k}\max\left\{  \ell\left(  \beta_{s}\right)  ,\ell\left(
\gamma_{s}\right)  \right\}  -k}\\
&  \ \ \ \ \ \ \ \ \ \ \ \ \ \ \ \ \ \ \ \ \cdot\left(  q-1\right)  ^{\left(
\text{\# of all }s\in\left[  k\right]  \text{ such that }\ell\left(  \beta
_{s}\right)  =\ell\left(  \gamma_{s}\right)  \right)  }\left(  \eta_{\beta
_{1}\beta_{2}\cdots\beta_{k}}^{\ast\left(  q\right)  }\otimes\eta_{\gamma
_{1}\gamma_{2}\cdots\gamma_{k}}^{\ast\left(  q\right)  }\right)  .
\end{align*}

\end{corollary}

\begin{proof}
We agree to understand an expression of the form $\prod_{s=1}^{k}u_{s}$ to
mean $u_{1}u_{2}\cdots u_{k}$ whenever $u_{1},u_{2},\ldots,u_{k}$ are any $k$
elements of any (not necessarily commutative) ring.

The comultiplication $\Delta$ of the $\mathbf{k}$-bialgebra
$\operatorname*{NSym}$ is a $\mathbf{k}$-algebra homomorphism (indeed, this is
true for any $\mathbf{k}$-bialgebra), and thus respects products. However,
Proposition \ref{prop.etastar.mult} yields $\eta_{\alpha}^{\ast\left(
q\right)  }=\eta_{\alpha_{1}}^{\ast\left(  q\right)  }\eta_{\alpha_{2}}%
^{\ast\left(  q\right)  }\cdots\eta_{\alpha_{k}}^{\ast\left(  q\right)  }$.
Hence,%
\[
\Delta\left(  \eta_{\alpha}^{\ast\left(  q\right)  }\right)  =\Delta\left(
\eta_{\alpha_{1}}^{\ast\left(  q\right)  }\eta_{\alpha_{2}}^{\ast\left(
q\right)  }\cdots\eta_{\alpha_{k}}^{\ast\left(  q\right)  }\right)
=\Delta\left(  \eta_{\alpha_{1}}^{\ast\left(  q\right)  }\right)
\Delta\left(  \eta_{\alpha_{2}}^{\ast\left(  q\right)  }\right)  \cdots
\Delta\left(  \eta_{\alpha_{k}}^{\ast\left(  q\right)  }\right)
\]
(since $\Delta$ respects products). In other words (using the notation
$\prod_{s=1}^{k}u_{s}$ as explained above),%
\begin{equation}
\Delta\left(  \eta_{\alpha}^{\ast\left(  q\right)  }\right)  =\prod_{s=1}%
^{k}\Delta\left(  \eta_{\alpha_{s}}^{\ast\left(  q\right)  }\right)  .
\label{pf.cor.Deltaetastaral.1}%
\end{equation}

However, Theorem \ref{thm.Deltaetastara} shows that for each $s\in\left[
k\right]  $, we have%
\[
\Delta\left(  \eta_{\alpha_{s}}^{\ast\left(  q\right)  }\right)
=\sum_{\substack{\beta,\gamma\in\operatorname*{Comp};\\\left\vert
\beta\right\vert +\left\vert \gamma\right\vert =\alpha_{s};\\\left\vert
\ell\left(  \beta\right)  -\ell\left(  \gamma\right)  \right\vert \leq
1}}\left(  -q\right)  ^{\max\left\{  \ell\left(  \beta\right)  ,\ell\left(
\gamma\right)  \right\}  -1}\left(  q-1\right)  ^{\left[  \ell\left(
\beta\right)  =\ell\left(  \gamma\right)  \right]  }\eta_{\beta}^{\ast\left(
q\right)  }\otimes\eta_{\gamma}^{\ast\left(  q\right)  }.
\]
Multiplying these equalities for all $s\in\left[  k\right]  $, we obtain%
\begin{align*}
&  \prod_{s=1}^{k}\Delta\left(  \eta_{\alpha_{s}}^{\ast\left(  q\right)
}\right) \\
&  =\prod_{s=1}^{k}\ \ \sum_{\substack{\beta,\gamma\in\operatorname*{Comp}%
;\\\left\vert \beta\right\vert +\left\vert \gamma\right\vert =\alpha
_{s};\\\left\vert \ell\left(  \beta\right)  -\ell\left(  \gamma\right)
\right\vert \leq1}}\left(  -q\right)  ^{\max\left\{  \ell\left(  \beta\right)
,\ell\left(  \gamma\right)  \right\}  -1}\left(  q-1\right)  ^{\left[
\ell\left(  \beta\right)  =\ell\left(  \gamma\right)  \right]  }\eta_{\beta
}^{\ast\left(  q\right)  }\otimes\eta_{\gamma}^{\ast\left(  q\right)  }\\
&  =\sum_{\substack{\beta_{1},\beta_{2},\ldots,\beta_{k}\in
\operatorname*{Comp};\\\gamma_{1},\gamma_{2},\ldots,\gamma_{k}\in
\operatorname*{Comp};\\\left\vert \beta_{s}\right\vert +\left\vert \gamma
_{s}\right\vert =\alpha_{s}\text{ for all }s;\\\left\vert \ell\left(
\beta_{s}\right)  -\ell\left(  \gamma_{s}\right)  \right\vert \leq1\text{ for
all }s}}\ \ \prod_{s=1}^{k}\left(  \left(  -q\right)  ^{\max\left\{
\ell\left(  \beta_{s}\right)  ,\ell\left(  \gamma_{s}\right)  \right\}
-1}\left(  q-1\right)  ^{\left[  \ell\left(  \beta_{s}\right)  =\ell\left(
\gamma_{s}\right)  \right]  }\eta_{\beta_{s}}^{\ast\left(  q\right)  }%
\otimes\eta_{\gamma_{s}}^{\ast\left(  q\right)  }\right) \\
&  \ \ \ \ \ \ \ \ \ \ \ \ \ \ \ \ \ \ \ \ \left(  \text{by the product
rule}\right) \\
&  =\sum_{\substack{\beta_{1},\beta_{2},\ldots,\beta_{k}\in
\operatorname*{Comp};\\\gamma_{1},\gamma_{2},\ldots,\gamma_{k}\in
\operatorname*{Comp};\\\left\vert \beta_{s}\right\vert +\left\vert \gamma
_{s}\right\vert =\alpha_{s}\text{ for all }s;\\\left\vert \ell\left(
\beta_{s}\right)  -\ell\left(  \gamma_{s}\right)  \right\vert \leq1\text{ for
all }s}}\underbrace{\left(  \prod_{s=1}^{k}\left(  -q\right)  ^{\max\left\{
\ell\left(  \beta_{s}\right)  ,\ell\left(  \gamma_{s}\right)  \right\}
-1}\right)  }_{\substack{=\left(  -q\right)  ^{\sum_{s=1}^{k}\left(
\max\left\{  \ell\left(  \beta_{s}\right)  ,\ell\left(  \gamma_{s}\right)
\right\}  -1\right)  }\\=\left(  -q\right)  ^{\sum_{s=1}^{k}\max\left\{
\ell\left(  \beta_{s}\right)  ,\ell\left(  \gamma_{s}\right)  \right\}  -k}%
}}\\
&  \ \ \ \ \ \ \ \ \ \ \ \ \ \ \ \ \ \ \ \ \cdot\underbrace{\left(
\prod_{s=1}^{k}\left(  q-1\right)  ^{\left[  \ell\left(  \beta_{s}\right)
=\ell\left(  \gamma_{s}\right)  \right]  }\right)  }_{\substack{=\left(
q-1\right)  ^{\sum_{s=1}^{k}\left[  \ell\left(  \beta_{s}\right)  =\ell\left(
\gamma_{s}\right)  \right]  }\\=\left(  q-1\right)  ^{\left(  \text{\# of all
}s\in\left[  k\right]  \text{ such that }\ell\left(  \beta_{s}\right)
=\ell\left(  \gamma_{s}\right)  \right)  }}}\underbrace{\left(  \prod
_{s=1}^{k}\left(  \eta_{\beta_{s}}^{\ast\left(  q\right)  }\otimes\eta
_{\gamma_{s}}^{\ast\left(  q\right)  }\right)  \right)  }_{\substack{=\eta
_{\beta_{1}}^{\ast\left(  q\right)  }\eta_{\beta_{2}}^{\ast\left(  q\right)
}\cdots\eta_{\beta_{k}}^{\ast\left(  q\right)  }\otimes\eta_{\gamma_{1}}%
^{\ast\left(  q\right)  }\eta_{\gamma_{2}}^{\ast\left(  q\right)  }\cdots
\eta_{\gamma_{k}}^{\ast\left(  q\right)  }\\=\eta_{\beta_{1}\beta_{2}%
\cdots\beta_{k}}^{\ast\left(  q\right)  }\otimes\eta_{\gamma_{1}\gamma
_{2}\cdots\gamma_{k}}^{\ast\left(  q\right)  }\\\text{(since Corollary
\ref{cor.etastar.concat}}\\\text{yields }\eta_{\beta_{1}}^{\ast\left(
q\right)  }\eta_{\beta_{2}}^{\ast\left(  q\right)  }\cdots\eta_{\beta_{k}%
}^{\ast\left(  q\right)  }=\eta_{\beta_{1}\beta_{2}\cdots\beta_{k}}%
^{\ast\left(  q\right)  }\\\text{and }\eta_{\gamma_{1}}^{\ast\left(  q\right)
}\eta_{\gamma_{2}}^{\ast\left(  q\right)  }\cdots\eta_{\gamma_{k}}%
^{\ast\left(  q\right)  }=\eta_{\gamma_{1}\gamma_{2}\cdots\gamma_{k}}%
^{\ast\left(  q\right)  }\text{)}}}\\
&  =\sum_{\substack{\beta_{1},\beta_{2},\ldots,\beta_{k}\in
\operatorname*{Comp};\\\gamma_{1},\gamma_{2},\ldots,\gamma_{k}\in
\operatorname*{Comp};\\\left\vert \beta_{s}\right\vert +\left\vert \gamma
_{s}\right\vert =\alpha_{s}\text{ for all }s;\\\left\vert \ell\left(
\beta_{s}\right)  -\ell\left(  \gamma_{s}\right)  \right\vert \leq1\text{ for
all }s}}\left(  -q\right)  ^{\sum_{s=1}^{k}\max\left\{  \ell\left(  \beta
_{s}\right)  ,\ell\left(  \gamma_{s}\right)  \right\}  -k}\\
&  \ \ \ \ \ \ \ \ \ \ \ \ \ \ \ \ \ \ \ \ \cdot\left(  q-1\right)  ^{\left(
\text{\# of all }s\in\left[  k\right]  \text{ such that }\ell\left(  \beta
_{s}\right)  =\ell\left(  \gamma_{s}\right)  \right)  }\left(  \eta_{\beta
_{1}\beta_{2}\cdots\beta_{k}}^{\ast\left(  q\right)  }\otimes\eta_{\gamma
_{1}\gamma_{2}\cdots\gamma_{k}}^{\ast\left(  q\right)  }\right)  .
\end{align*}

In view of (\ref{pf.cor.Deltaetastaral.1}), we can rewrite this as
\begin{align*}
\Delta\left(  \eta_{\alpha}^{\ast\left(  q\right)  }\right)   &
=\sum_{\substack{\beta_{1},\beta_{2},\ldots,\beta_{k}\in\operatorname*{Comp}%
;\\\gamma_{1},\gamma_{2},\ldots,\gamma_{k}\in\operatorname*{Comp};\\\left\vert
\beta_{s}\right\vert +\left\vert \gamma_{s}\right\vert =\alpha_{s}\text{ for
all }s;\\\left\vert \ell\left(  \beta_{s}\right)  -\ell\left(  \gamma
_{s}\right)  \right\vert \leq1\text{ for all }s}}\left(  -q\right)
^{\sum_{s=1}^{k}\max\left\{  \ell\left(  \beta_{s}\right)  ,\ell\left(
\gamma_{s}\right)  \right\}  -k}\\
&  \ \ \ \ \ \ \ \ \ \ \ \ \ \ \ \ \ \ \ \ \cdot\left(  q-1\right)  ^{\left(
\text{\# of all }s\in\left[  k\right]  \text{ such that }\ell\left(  \beta
_{s}\right)  =\ell\left(  \gamma_{s}\right)  \right)  }\left(  \eta_{\beta
_{1}\beta_{2}\cdots\beta_{k}}^{\ast\left(  q\right)  }\otimes\eta_{\gamma
_{1}\gamma_{2}\cdots\gamma_{k}}^{\ast\left(  q\right)  }\right)  .
\end{align*}
This proves Corollary \ref{cor.Deltaetastaral}. \qedhere

\begin{verlong}
T0D0: verlong proof.
\end{verlong}
\end{proof}

\section{\label{sec.prod}The product rule for $\eta_{\alpha}^{\left(
q\right)  }$}

We now approach the most intricate of the rules for the $\eta_{\alpha
}^{\left(  q\right)  }$ functions: the product rule, i.e., the expression of a
product $\eta_{\delta}^{\left(  q\right)  }\eta_{\varepsilon}^{\left(
q\right)  }$ as a $\mathbb{Z}\left[  q\right]  $-linear combination of other
$\eta_{\alpha}^{\left(  q\right)  }$'s. We shall give three different versions
of this rule, all equivalent but using somewhat different indexing sets. Only
the first version will be proved in detail, as it suffices for the
applications we have in mind.

\subsection{The product rule in terms of compositions}

Our first version of the product rule is as follows:\footnote{The symbol
\textquotedblleft\#\textquotedblright\ means \textquotedblleft
number\textquotedblright\ (so that, e.g., we have $\left(  \text{\# of odd
numbers }i\in\left[  2n\right]  \right)  =n$ for each $n\in\mathbb{N}$).}

\begin{theorem}
\label{thm.product.1}Let $\delta$ and $\varepsilon$ be two compositions. Then,%
\begin{align*}
\eta_{\delta}^{\left(  q\right)  }\eta_{\varepsilon}^{\left(  q\right)  }  &
=\sum_{\substack{k\in\mathbb{N};\\\beta_{1},\beta_{2},\ldots,\beta_{k}%
\in\operatorname*{Comp};\\\gamma_{1},\gamma_{2},\ldots,\gamma_{k}%
\in\operatorname*{Comp};\\\beta_{1}\beta_{2}\cdots\beta_{k}=\delta
;\\\gamma_{1}\gamma_{2}\cdots\gamma_{k}=\varepsilon;\\\left\vert \ell\left(
\beta_{s}\right)  -\ell\left(  \gamma_{s}\right)  \right\vert \leq1\text{ for
all }s;\\\ell\left(  \beta_{s}\right)  +\ell\left(  \gamma_{s}\right)
>0\text{ for all }s}}\left(  -q\right)  ^{\sum_{s=1}^{k}\max\left\{
\ell\left(  \beta_{s}\right)  ,\ell\left(  \gamma_{s}\right)  \right\}  -k}\\
&  \ \ \ \ \ \ \ \ \ \ \ \ \ \ \ \ \ \ \ \ \cdot\left(  q-1\right)  ^{\left(
\text{\# of all }s\in\left[  k\right]  \text{ such that }\ell\left(  \beta
_{s}\right)  =\ell\left(  \gamma_{s}\right)  \right)  }\\
&  \ \ \ \ \ \ \ \ \ \ \ \ \ \ \ \ \ \ \ \ \cdot\eta_{\left(  \left\vert
\beta_{1}\right\vert +\left\vert \gamma_{1}\right\vert ,\ \left\vert \beta
_{2}\right\vert +\left\vert \gamma_{2}\right\vert ,\ \ldots,\ \left\vert
\beta_{k}\right\vert +\left\vert \gamma_{k}\right\vert \right)  }^{\left(
q\right)  }.
\end{align*}

\end{theorem}

\begin{remark}
The compositions $\beta_{1},\beta_{2},\ldots,\beta_{k}$ and $\gamma_{1}%
,\gamma_{2},\ldots,\gamma_{k}$ in the sum on the right hand side of Theorem
\ref{thm.product.1} are allowed to be empty. Nevertheless, the sum is finite.
Indeed, if $k\in\mathbb{N}$ and $\beta_{1},\beta_{2},\ldots,\beta_{k}%
\in\operatorname*{Comp}$ and $\gamma_{1},\gamma_{2},\ldots,\gamma_{k}%
\in\operatorname*{Comp}$ satisfy%
\begin{align*}
\beta_{1}\beta_{2}\cdots\beta_{k}  &  =\delta\ \ \ \ \ \ \ \ \ \ \text{and}%
\ \ \ \ \ \ \ \ \ \ \gamma_{1}\gamma_{2}\cdots\gamma_{k}=\varepsilon
\ \ \ \ \ \ \ \ \ \ \text{and}\\
\left\vert \ell\left(  \beta_{s}\right)  -\ell\left(  \gamma_{s}\right)
\right\vert  &  \leq1\text{ for all }s\ \ \ \ \ \ \ \ \ \ \text{and}%
\ \ \ \ \ \ \ \ \ \ \ell\left(  \beta_{s}\right)  +\ell\left(  \gamma
_{s}\right)  >0\text{ for all }s,
\end{align*}
then $k\leq\ell\left(  \delta\right)  +\ell\left(  \varepsilon\right)  $,
because
\begin{align*}
&  \underbrace{\ell\left(  \delta\right)  }_{\substack{=\ell\left(  \beta
_{1}\right)  +\ell\left(  \beta_{2}\right)  +\cdots+\ell\left(  \beta
_{k}\right)  \\\text{(since }\delta=\beta_{1}\beta_{2}\cdots\beta_{k}\text{)}%
}}+\underbrace{\ell\left(  \varepsilon\right)  }_{\substack{=\ell\left(
\gamma_{1}\right)  +\ell\left(  \gamma_{2}\right)  +\cdots+\ell\left(
\gamma_{k}\right)  \\\text{(since }\varepsilon=\gamma_{1}\gamma_{2}%
\cdots\gamma_{k}\text{)}}}\\
&  =\left(  \ell\left(  \beta_{1}\right)  +\ell\left(  \beta_{2}\right)
+\cdots+\ell\left(  \beta_{k}\right)  \right)  +\left(  \ell\left(  \gamma
_{1}\right)  +\ell\left(  \gamma_{2}\right)  +\cdots+\ell\left(  \gamma
_{k}\right)  \right) \\
&  =\sum_{s=1}^{k}\ell\left(  \beta_{s}\right)  +\sum_{s=1}^{k}\ell\left(
\gamma_{s}\right)  =\sum_{s=1}^{k}\underbrace{\left(  \ell\left(  \beta
_{s}\right)  +\ell\left(  \gamma_{s}\right)  \right)  }_{\substack{\geq
1\\\text{(since our above assumptions}\\\text{yield }\ell\left(  \beta
_{s}\right)  +\ell\left(  \gamma_{s}\right)  >0\text{,}\\\text{but }%
\ell\left(  \beta_{s}\right)  +\ell\left(  \gamma_{s}\right)  \text{ is an
integer)}}}\geq\sum_{s=1}^{k}1=k.
\end{align*}
This narrows down the options for $k$ to the finite set $\left\{
0,1,\ldots,\ell\left(  \delta\right)  +\ell\left(  \varepsilon\right)
\right\}  $, and thus leaves only finitely many options for $\beta_{1}%
,\beta_{2},\ldots,\beta_{k}$ (since there are only finitely many ways to
decompose the composition $\delta$ as a concatenation $\delta=\beta_{1}%
\beta_{2}\cdots\beta_{k}$ when $k$ is fixed) and for $\gamma_{1},\gamma
_{2},\ldots,\gamma_{k}$ (similarly). Thus, the sum is finite.
\end{remark}

\begin{example}
\label{exa.product.1.ab*c}Let $\delta$ and $\varepsilon$ be two compositions
of the form $\delta=\left(  a,b\right)  $ and $\varepsilon=\left(  c\right)  $
for some positive integers $a,b,c$. Then, Theorem \ref{thm.product.1}
expresses the product $\eta_{\delta}^{\left(  q\right)  }\eta_{\varepsilon
}^{\left(  q\right)  }=\eta_{\left(  a,b\right)  }^{\left(  q\right)  }%
\eta_{\left(  c\right)  }^{\left(  q\right)  }$ as a sum over all choices of
$k\in\mathbb{N}$ and of $k$ compositions $\beta_{1},\beta_{2},\ldots,\beta
_{k}\in\operatorname*{Comp}$ and of $k$ further compositions $\gamma
_{1},\gamma_{2},\ldots,\gamma_{k}\in\operatorname*{Comp}$ satisfying%
\begin{align*}
\beta_{1}\beta_{2}\cdots\beta_{k}  &  =\delta\ \ \ \ \ \ \ \ \ \ \text{and}%
\ \ \ \ \ \ \ \ \ \ \gamma_{1}\gamma_{2}\cdots\gamma_{k}=\varepsilon
\ \ \ \ \ \ \ \ \ \ \text{and}\\
\left\vert \ell\left(  \beta_{s}\right)  -\ell\left(  \gamma_{s}\right)
\right\vert  &  \leq1\text{ for all }s\ \ \ \ \ \ \ \ \ \ \text{and}%
\ \ \ \ \ \ \ \ \ \ \ell\left(  \beta_{s}\right)  +\ell\left(  \gamma
_{s}\right)  >0\text{ for all }s.
\end{align*}
These choices are

\begin{enumerate}
\item having $k=1$ and $\beta_{1}=\delta=\left(  a,b\right)  $ and $\gamma
_{1}=\varepsilon=\left(  c\right)  $;

\item having $k=2$ and $\beta_{1}=\left(  a\right)  $ and $\beta_{2}=\left(
b\right)  $ and $\gamma_{1}=\varnothing$ and $\gamma_{2}=\left(  c\right)  $;

\item having $k=2$ and $\beta_{1}=\left(  a\right)  $ and $\beta_{2}=\left(
b\right)  $ and $\gamma_{1}=\left(  c\right)  $ and $\gamma_{2}=\varnothing$;

\item having $k=3$ and $\beta_{1}=\varnothing$ and $\beta_{2}=\left(
a\right)  $ and $\beta_{3}=\left(  b\right)  $ and $\gamma_{1}=\left(
c\right)  $ and $\gamma_{2}=\varnothing$ and $\gamma_{3}=\varnothing$;

\item having $k=3$ and $\beta_{1}=\left(  a\right)  $ and $\beta
_{2}=\varnothing$ and $\beta_{3}=\left(  b\right)  $ and $\gamma
_{1}=\varnothing$ and $\gamma_{2}=\left(  c\right)  $ and $\gamma
_{3}=\varnothing$;

\item having $k=3$ and $\beta_{1}=\left(  a\right)  $ and $\beta_{2}=\left(
b\right)  $ and $\beta_{3}=\varnothing$ and $\gamma_{1}=\varnothing$ and
$\gamma_{2}=\varnothing$ and $\gamma_{3}=\left(  c\right)  $.
\end{enumerate}

Thus, Theorem \ref{thm.product.1} yields%
\begin{align*}
&  \eta_{\left(  a,b\right)  }^{\left(  q\right)  }\eta_{\left(  c\right)
}^{\left(  q\right)  }\\
&  =\left(  -q\right)  ^{2-1}\left(  q-1\right)  ^{0}\eta_{\left(
a+b+c\right)  }^{\left(  q\right)  }+\left(  -q\right)  ^{1+1-2}\left(
q-1\right)  ^{1}\eta_{\left(  a,\ b+c\right)  }^{\left(  q\right)  }\\
&  \ \ \ \ \ \ \ \ \ \ +\left(  -q\right)  ^{1+1-2}\left(  q-1\right)
^{1}\eta_{\left(  a+c,\ b\right)  }^{\left(  q\right)  }+\left(  -q\right)
^{1+1+1-3}\left(  q-1\right)  ^{0}\eta_{\left(  c,a,b\right)  }^{\left(
q\right)  }\\
&  \ \ \ \ \ \ \ \ \ \ +\left(  -q\right)  ^{1+1+1-3}\left(  q-1\right)
^{0}\eta_{\left(  a,c,b\right)  }^{\left(  q\right)  }+\left(  -q\right)
^{1+1+1-3}\left(  q-1\right)  ^{0}\eta_{\left(  a,b,c\right)  }^{\left(
q\right)  }\\
&  =-q\eta_{\left(  a+b+c\right)  }^{\left(  q\right)  }+\left(  q-1\right)
\eta_{\left(  a,\ b+c\right)  }^{\left(  q\right)  }+\left(  q-1\right)
\eta_{\left(  a+c,\ b\right)  }^{\left(  q\right)  }+\eta_{\left(
c,a,b\right)  }^{\left(  q\right)  }+\eta_{\left(  a,c,b\right)  }^{\left(
q\right)  }+\eta_{\left(  a,b,c\right)  }^{\left(  q\right)  }.
\end{align*}
Note that the last three addends $\eta_{\left(  c,a,b\right)  }^{\left(
q\right)  },\ \eta_{\left(  a,c,b\right)  }^{\left(  q\right)  }%
,\ \eta_{\left(  a,b,c\right)  }^{\left(  q\right)  }$ here come from those
choices in which $\min\left\{  \ell\left(  \beta_{s}\right)  ,\ell\left(
\gamma_{s}\right)  \right\}  =0$ for each $s\in\left[  k\right]  $ (that is,
for each $s\in\left[  k\right]  $, one of the two compositions $\beta_{s}$ and
$\gamma_{s}$ is empty). In these choices, the two powers%
\[
\left(  -q\right)  ^{\sum_{s=1}^{k}\max\left\{  \ell\left(  \beta_{s}\right)
,\ell\left(  \gamma_{s}\right)  \right\}  -k}\ \ \ \ \ \ \ \ \ \ \text{and}%
\ \ \ \ \ \ \ \ \ \ \left(  q-1\right)  ^{\left(  \text{\# of all }s\in\left[
k\right]  \text{ such that }\ell\left(  \beta_{s}\right)  =\ell\left(
\gamma_{s}\right)  \right)  }%
\]
are equal to $1$ (because the exponents are easily seen to be $0$), whereas
the composition $\left(  \left\vert \beta_{1}\right\vert +\left\vert
\gamma_{1}\right\vert ,\ \left\vert \beta_{2}\right\vert +\left\vert
\gamma_{2}\right\vert ,\ \ldots,\ \left\vert \beta_{k}\right\vert +\left\vert
\gamma_{k}\right\vert \right)  $ is a shuffle of $\delta$ with $\varepsilon$.
Thus, these choices contribute terms of the form $\eta_{\varphi}^{\left(
q\right)  }$, where $\varphi$ is a shuffle of $\delta$ with $\varepsilon$, to
the right hand side of Theorem \ref{thm.product.1}, and these terms all have
coefficient $1$. These are the only choices of $k,\ \beta_{1},\beta_{2}%
,\ldots,\beta_{k},\ \gamma_{1},\gamma_{2},\ldots,\gamma_{k}$ that have
$k=\ell\left(  \delta\right)  +\ell\left(  \varepsilon\right)  $. All other
choices have $k<\ell\left(  \delta\right)  +\ell\left(  \varepsilon\right)  $,
and these choices lead to addends that involve either a nontrivial power of
$-q$ or a nontrivial power of $q-1$ (or both). In this sense, we can view
Theorem \ref{thm.product.1} as a deformation of the overlapping shuffle
product formula for $M_{\delta}M_{\varepsilon}$ (see, e.g., \cite[Proposition
5.1.3]{GriRei}), although the concept of a \textquotedblleft
deformation\textquotedblright\ must be understood in a wide sense (we cannot
obtain the latter just by specializing the former).
\end{example}

We will derive Theorem \ref{thm.product.1} from Corollary
\ref{cor.Deltaetastaral}. For this, we will again use the duality between
$\operatorname*{NSym}$ and $\operatorname*{QSym}$:

\begin{lemma}
\label{lem.NSym-duality.2}Let $f,g\in\operatorname*{QSym}$ and $h\in
\operatorname*{NSym}$ be arbitrary. Let the tensor $\Delta\left(  h\right)
\in\operatorname*{NSym}\otimes\operatorname*{NSym}$ be written in the form
$\Delta\left(  h\right)  =\sum_{i\in I}s_{i}\otimes t_{i}$, where $I$ is a
finite set and where $s_{i},t_{i}\in\operatorname*{NSym}$ for each $i\in I$.
Then,%
\[
\left\langle h,fg\right\rangle =\sum_{i\in I}\left\langle s_{i},f\right\rangle
\left\langle t_{i},g\right\rangle .
\]

\end{lemma}

\begin{proof}
This is analogous to Lemma \ref{lem.NSym-duality.1}, except that the roles of
$\operatorname*{QSym}$ and $\operatorname*{NSym}$ have now been switched.
\end{proof}

For the sake of convenience, let us extend Lemma \ref{lem.NSym-duality.2} to
infinite sums with only finitely many nonzero addends:

\begin{lemma}
\label{lem.NSym-duality.2inf}Let $f,g\in\operatorname*{QSym}$ and
$h\in\operatorname*{NSym}$ be arbitrary. Let the tensor $\Delta\left(
h\right)  \in\operatorname*{NSym}\otimes\operatorname*{NSym}$ be written in
the form $\Delta\left(  h\right)  =\sum_{i\in I}s_{i}\otimes t_{i}$, where $I$
is a set and where $s_{i},t_{i}\in\operatorname*{NSym}$ for each $i\in I$ are
chosen such that only finitely many $i\in I$ satisfy $s_{i}\neq0$. Then,%
\[
\left\langle h,fg\right\rangle =\sum_{i\in I}\left\langle s_{i},f\right\rangle
\left\langle t_{i},g\right\rangle .
\]

\end{lemma}

\begin{proof}
This is easily reduced to Lemma \ref{lem.NSym-duality.2} (just replace the set
$I$ by its subset $I^{\prime}:=\left\{  i\in I\ \mid\ s_{i}\neq0\right\}  $).
\end{proof}

\begin{proof}
[Proof of Theorem \ref{thm.product.1}.]Forget that we fixed $\delta$ and
$\varepsilon$. For any three compositions $\alpha=\left(  \alpha_{1}%
,\alpha_{2},\ldots,\alpha_{k}\right)  $, $\delta$ and $\varepsilon$, we define
a polynomial%
\begin{align}
d_{\delta,\varepsilon}^{\alpha}\left(  X\right)   &  :=\sum_{\substack{\beta
_{1},\beta_{2},\ldots,\beta_{k}\in\operatorname*{Comp};\\\gamma_{1},\gamma
_{2},\ldots,\gamma_{k}\in\operatorname*{Comp};\\\beta_{1}\beta_{2}\cdots
\beta_{k}=\delta;\\\gamma_{1}\gamma_{2}\cdots\gamma_{k}=\varepsilon
;\\\left\vert \ell\left(  \beta_{s}\right)  -\ell\left(  \gamma_{s}\right)
\right\vert \leq1\text{ for all }s;\\\left\vert \beta_{s}\right\vert
+\left\vert \gamma_{s}\right\vert =\alpha_{s}\text{ for all }s}}\left(
-X\right)  ^{\sum_{s=1}^{k}\max\left\{  \ell\left(  \beta_{s}\right)
,\ell\left(  \gamma_{s}\right)  \right\}  -k}\nonumber\\
&  \ \ \ \ \ \ \ \ \ \ \ \ \ \ \ \ \ \ \ \ \cdot\left(  X-1\right)  ^{\left(
\text{\# of all }s\in\left[  k\right]  \text{ such that }\ell\left(  \beta
_{s}\right)  =\ell\left(  \gamma_{s}\right)  \right)  }%
\label{pf.thm.product.1.dX=}\\
&  \in\mathbb{Z}\left[  X\right] \nonumber
\end{align}
(this really is a polynomial, since the exponent $\sum_{s=1}^{k}\max\left\{
\ell\left(  \beta_{s}\right)  ,\ell\left(  \gamma_{s}\right)  \right\}  -k$ is
easily seen to be a nonnegative integer). Thus, clearly, for any three
compositions $\alpha=\left(  \alpha_{1},\alpha_{2},\ldots,\alpha_{k}\right)
$, $\delta$ and $\varepsilon$, we have%
\begin{align}
d_{\delta,\varepsilon}^{\alpha}\left(  q\right)   &  =\sum_{\substack{\beta
_{1},\beta_{2},\ldots,\beta_{k}\in\operatorname*{Comp};\\\gamma_{1},\gamma
_{2},\ldots,\gamma_{k}\in\operatorname*{Comp};\\\beta_{1}\beta_{2}\cdots
\beta_{k}=\delta;\\\gamma_{1}\gamma_{2}\cdots\gamma_{k}=\varepsilon
;\\\left\vert \ell\left(  \beta_{s}\right)  -\ell\left(  \gamma_{s}\right)
\right\vert \leq1\text{ for all }s;\\\left\vert \beta_{s}\right\vert
+\left\vert \gamma_{s}\right\vert =\alpha_{s}\text{ for all }s}}\left(
-q\right)  ^{\sum_{s=1}^{k}\max\left\{  \ell\left(  \beta_{s}\right)
,\ell\left(  \gamma_{s}\right)  \right\}  -k}\nonumber\\
&  \ \ \ \ \ \ \ \ \ \ \ \ \ \ \ \ \ \ \ \ \cdot\left(  q-1\right)  ^{\left(
\text{\# of all }s\in\left[  k\right]  \text{ such that }\ell\left(  \beta
_{s}\right)  =\ell\left(  \gamma_{s}\right)  \right)  }
\label{pf.thm.product.1.dq=}%
\end{align}

Note that the sums on the right hand sides of (\ref{pf.thm.product.1.dX=}) and
(\ref{pf.thm.product.1.dq=}) are finite (because for a given $k\in\mathbb{N}$
and given compositions $\delta$ and $\varepsilon$, there are only finitely
many ways to decompose $\delta$ as $\delta=\beta_{1}\beta_{2}\cdots\beta_{k}$,
and only finitely many ways to decompose $\varepsilon$ as $\varepsilon
=\gamma_{1}\gamma_{2}\cdots\gamma_{k}$). All sums that will appear in this
proof will be finite or essentially finite (i.e., have only finitely many
nonzero addends). We note that $\mathbf{k}$-linear maps always respect such sums.

Now, we shall proceed by proving several claims. Our first claim is a
restatement of Corollary \ref{cor.Deltaetastaral}:

\begin{statement}
\textit{Claim 1:} Let $\alpha$ be any composition. Assume that $r$ is
invertible. Then,%
\begin{equation}
\Delta\left(  \eta_{\alpha}^{\ast\left(  q\right)  }\right)  =\sum
_{\delta,\varepsilon\in\operatorname*{Comp}}d_{\delta,\varepsilon}^{\alpha
}\left(  q\right)  \eta_{\delta}^{\ast\left(  q\right)  }\otimes
\eta_{\varepsilon}^{\ast\left(  q\right)  }. \label{pf.thm.product.c1}%
\end{equation}

\end{statement}

\begin{proof}
[Proof of Claim 1.]Write the composition $\alpha$ as $\alpha=\left(
\alpha_{1},\alpha_{2},\ldots,\alpha_{k}\right)  $. Then, Corollary
\ref{cor.Deltaetastaral} yields
\begin{align*}
&  \Delta\left(  \eta_{\alpha}^{\ast\left(  q\right)  }\right) \\
&  =\underbrace{\sum_{\substack{\beta_{1},\beta_{2},\ldots,\beta_{k}%
\in\operatorname*{Comp};\\\gamma_{1},\gamma_{2},\ldots,\gamma_{k}%
\in\operatorname*{Comp};\\\left\vert \beta_{s}\right\vert +\left\vert
\gamma_{s}\right\vert =\alpha_{s}\text{ for all }s;\\\left\vert \ell\left(
\beta_{s}\right)  -\ell\left(  \gamma_{s}\right)  \right\vert \leq1\text{ for
all }s}}}_{\substack{=\sum_{\substack{\beta_{1},\beta_{2},\ldots,\beta_{k}%
\in\operatorname*{Comp};\\\gamma_{1},\gamma_{2},\ldots,\gamma_{k}%
\in\operatorname*{Comp};\\\left\vert \ell\left(  \beta_{s}\right)
-\ell\left(  \gamma_{s}\right)  \right\vert \leq1\text{ for all }%
s;\\\left\vert \beta_{s}\right\vert +\left\vert \gamma_{s}\right\vert
=\alpha_{s}\text{ for all }s}}\\=\sum_{\delta,\varepsilon\in
\operatorname*{Comp}}\ \ \sum_{\substack{_{\substack{\beta_{1},\beta
_{2},\ldots,\beta_{k}\in\operatorname*{Comp};\\\gamma_{1},\gamma_{2}%
,\ldots,\gamma_{k}\in\operatorname*{Comp};\\\beta_{1}\beta_{2}\cdots\beta
_{k}=\delta;\\\gamma_{1}\gamma_{2}\cdots\gamma_{k}=\varepsilon;\\\left\vert
\ell\left(  \beta_{s}\right)  -\ell\left(  \gamma_{s}\right)  \right\vert
\leq1\text{ for all }s;\\\left\vert \beta_{s}\right\vert +\left\vert
\gamma_{s}\right\vert =\alpha_{s}\text{ for all }s}}\\\text{(here, we have
split up the sum according}\\\text{to the values of }\beta_{1}\beta_{2}%
\cdots\beta_{k}\text{ and }\gamma_{1}\gamma_{2}\cdots\gamma_{k}\text{)}}%
}}}\left(  -q\right)  ^{\sum_{s=1}^{k}\max\left\{  \ell\left(  \beta
_{s}\right)  ,\ell\left(  \gamma_{s}\right)  \right\}  -k}\\
&  \ \ \ \ \ \ \ \ \ \ \ \ \ \ \ \ \ \ \ \ \cdot\left(  q-1\right)  ^{\left(
\text{\# of all }s\in\left[  k\right]  \text{ such that }\ell\left(  \beta
_{s}\right)  =\ell\left(  \gamma_{s}\right)  \right)  }\left(  \eta_{\beta
_{1}\beta_{2}\cdots\beta_{k}}^{\ast\left(  q\right)  }\otimes\eta_{\gamma
_{1}\gamma_{2}\cdots\gamma_{k}}^{\ast\left(  q\right)  }\right) \\
&  =\sum_{\delta,\varepsilon\in\operatorname*{Comp}}\ \ \sum_{\substack{\beta
_{1},\beta_{2},\ldots,\beta_{k}\in\operatorname*{Comp};\\\gamma_{1},\gamma
_{2},\ldots,\gamma_{k}\in\operatorname*{Comp};\\\beta_{1}\beta_{2}\cdots
\beta_{k}=\delta;\\\gamma_{1}\gamma_{2}\cdots\gamma_{k}=\varepsilon
;\\\left\vert \ell\left(  \beta_{s}\right)  -\ell\left(  \gamma_{s}\right)
\right\vert \leq1\text{ for all }s;\\\left\vert \beta_{s}\right\vert
+\left\vert \gamma_{s}\right\vert =\alpha_{s}\text{ for all }s}}\left(
-q\right)  ^{\sum_{s=1}^{k}\max\left\{  \ell\left(  \beta_{s}\right)
,\ell\left(  \gamma_{s}\right)  \right\}  -k}\\
&  \ \ \ \ \ \ \ \ \ \ \ \ \ \ \ \ \ \ \ \ \cdot\left(  q-1\right)  ^{\left(
\text{\# of all }s\in\left[  k\right]  \text{ such that }\ell\left(  \beta
_{s}\right)  =\ell\left(  \gamma_{s}\right)  \right)  }\underbrace{\left(
\eta_{\beta_{1}\beta_{2}\cdots\beta_{k}}^{\ast\left(  q\right)  }\otimes
\eta_{\gamma_{1}\gamma_{2}\cdots\gamma_{k}}^{\ast\left(  q\right)  }\right)
}_{\substack{=\eta_{\delta}^{\ast\left(  q\right)  }\otimes\eta_{\varepsilon
}^{\ast\left(  q\right)  }\\\text{(since }\beta_{1}\beta_{2}\cdots\beta
_{k}=\delta\\\text{and }\gamma_{1}\gamma_{2}\cdots\gamma_{k}=\varepsilon
\text{)}}}
\end{align*}%
\begin{align*}
&  =\sum_{\delta,\varepsilon\in\operatorname*{Comp}}\ \ \sum_{\substack{\beta
_{1},\beta_{2},\ldots,\beta_{k}\in\operatorname*{Comp};\\\gamma_{1},\gamma
_{2},\ldots,\gamma_{k}\in\operatorname*{Comp};\\\beta_{1}\beta_{2}\cdots
\beta_{k}=\delta;\\\gamma_{1}\gamma_{2}\cdots\gamma_{k}=\varepsilon
;\\\left\vert \ell\left(  \beta_{s}\right)  -\ell\left(  \gamma_{s}\right)
\right\vert \leq1\text{ for all }s;\\\left\vert \beta_{s}\right\vert
+\left\vert \gamma_{s}\right\vert =\alpha_{s}\text{ for all }s}}\left(
-q\right)  ^{\sum_{s=1}^{k}\max\left\{  \ell\left(  \beta_{s}\right)
,\ell\left(  \gamma_{s}\right)  \right\}  -k}\\
&  \ \ \ \ \ \ \ \ \ \ \ \ \ \ \ \ \ \ \ \ \cdot\left(  q-1\right)  ^{\left(
\text{\# of all }s\in\left[  k\right]  \text{ such that }\ell\left(  \beta
_{s}\right)  =\ell\left(  \gamma_{s}\right)  \right)  }\eta_{\delta}%
^{\ast\left(  q\right)  }\otimes\eta_{\varepsilon}^{\ast\left(  q\right)  }\\
&  =\sum_{\delta,\varepsilon\in\operatorname*{Comp}}\underbrace{\left(
\sum_{\substack{\beta_{1},\beta_{2},\ldots,\beta_{k}\in\operatorname*{Comp}%
;\\\gamma_{1},\gamma_{2},\ldots,\gamma_{k}\in\operatorname*{Comp};\\\beta
_{1}\beta_{2}\cdots\beta_{k}=\delta;\\\gamma_{1}\gamma_{2}\cdots\gamma
_{k}=\varepsilon;\\\left\vert \ell\left(  \beta_{s}\right)  -\ell\left(
\gamma_{s}\right)  \right\vert \leq1\text{ for all }s;\\\left\vert \beta
_{s}\right\vert +\left\vert \gamma_{s}\right\vert =\alpha_{s}\text{ for all
}s}}\left(  -q\right)  ^{\sum_{s=1}^{k}\max\left\{  \ell\left(  \beta
_{s}\right)  ,\ell\left(  \gamma_{s}\right)  \right\}  -k}\left(  q-1\right)
^{\left(  \text{\# of all }s\in\left[  k\right]  \text{ such that }\ell\left(
\beta_{s}\right)  =\ell\left(  \gamma_{s}\right)  \right)  }\right)
}_{\substack{=d_{\delta,\varepsilon}^{\alpha}\left(  q\right)  \\\text{(by
(\ref{pf.thm.product.1.dq=}))}}}\\
&  \ \ \ \ \ \ \ \ \ \ \ \ \ \ \ \ \ \ \ \ \eta_{\delta}^{\ast\left(
q\right)  }\otimes\eta_{\varepsilon}^{\ast\left(  q\right)  }\\
&  =\sum_{\delta,\varepsilon\in\operatorname*{Comp}}d_{\delta,\varepsilon
}^{\alpha}\left(  q\right)  \eta_{\delta}^{\ast\left(  q\right)  }\otimes
\eta_{\varepsilon}^{\ast\left(  q\right)  }.
\end{align*}
Thus, Claim 1 is proved.
\end{proof}

\begin{statement}
\textit{Claim 2:} Let $\delta$ and $\varepsilon$ be two compositions. If $r$
is invertible, then%
\[
\eta_{\delta}^{\left(  q\right)  }\eta_{\varepsilon}^{\left(  q\right)  }%
=\sum_{\alpha\in\operatorname*{Comp}}d_{\delta,\varepsilon}^{\alpha}\left(
q\right)  \eta_{\alpha}^{\left(  q\right)  }.
\]

\end{statement}

\begin{proof}
[Proof of Claim 2.]Essentially, this follows by duality (Lemma
\ref{lem.NSym-duality.2}) from Claim 1. Here are the details:

Assume that $r$ is invertible. For any composition $\alpha$, we have%
\begin{equation}
\Delta\left(  \eta_{\alpha}^{\ast\left(  q\right)  }\right)  =\sum
_{\lambda,\mu\in\operatorname*{Comp}}d_{\lambda,\mu}^{\alpha}\left(  q\right)
\eta_{\lambda}^{\ast\left(  q\right)  }\otimes\eta_{\mu}^{\ast\left(
q\right)  } \label{pf.thm.product.c2.pf.Delta}%
\end{equation}
(by Claim 1, with the letters $\delta$ and $\varepsilon$ renamed as $\lambda$
and $\mu$).

Let $I$ be the set $\operatorname*{Comp}\times\operatorname*{Comp}$. Then, we
can rewrite (\ref{pf.thm.product.c2.pf.Delta}) as follows: For any composition
$\alpha$, we have%
\begin{equation}
\Delta\left(  \eta_{\alpha}^{\ast\left(  q\right)  }\right)  =\sum_{\left(
\lambda,\mu\right)  \in I}d_{\lambda,\mu}^{\alpha}\left(  q\right)
\eta_{\lambda}^{\ast\left(  q\right)  }\otimes\eta_{\mu}^{\ast\left(
q\right)  }. \label{pf.thm.product.c2.pf.DeltaI}%
\end{equation}

Then, Theorem \ref{thm.eta.basis} \textbf{(a)} shows that the family $\left(
\eta_{\alpha}^{\left(  q\right)  }\right)  _{\alpha\in\operatorname*{Comp}}$
is a basis of the $\mathbf{k}$-module $\operatorname*{QSym}$. In other words,
the family $\left(  \eta_{\beta}^{\left(  q\right)  }\right)  _{\beta
\in\operatorname*{Comp}}$ is a basis of the $\mathbf{k}$-module
$\operatorname*{QSym}$. Hence, we can write the quasisymmetric function
$\eta_{\delta}^{\left(  q\right)  }\eta_{\varepsilon}^{\left(  q\right)  }%
\in\operatorname*{QSym}$ as%
\begin{equation}
\eta_{\delta}^{\left(  q\right)  }\eta_{\varepsilon}^{\left(  q\right)  }%
=\sum_{\beta\in\operatorname*{Comp}}c_{\beta}\eta_{\beta}^{\left(  q\right)
}, \label{pf.thm.product.c2.pf.1}%
\end{equation}
where $\left(  c_{\beta}\right)  _{\beta\in\operatorname*{Comp}}\in
\mathbf{k}^{\operatorname*{Comp}}$ is a family of coefficients (with
$c_{\beta}=0$ for all but finitely many $\beta\in\operatorname*{Comp}$).
Consider this family.

For every $\alpha\in\operatorname*{Comp}$, we have%
\begin{align*}
\left\langle \eta_{\alpha}^{\ast\left(  q\right)  },\ \eta_{\delta}^{\left(
q\right)  }\eta_{\varepsilon}^{\left(  q\right)  }\right\rangle  &
=\left\langle \eta_{\alpha}^{\ast\left(  q\right)  },\ \sum_{\beta
\in\operatorname*{Comp}}c_{\beta}\eta_{\beta}^{\left(  q\right)
}\right\rangle \ \ \ \ \ \ \ \ \ \ \left(  \text{by
(\ref{pf.thm.product.c2.pf.1})}\right) \\
&  =\sum_{\beta\in\operatorname*{Comp}}c_{\beta}\underbrace{\left\langle
\eta_{\alpha}^{\ast\left(  q\right)  },\eta_{\beta}^{\left(  q\right)
}\right\rangle }_{\substack{=\left[  \alpha=\beta\right]  \\\text{(by
(\ref{eq.prop.etastar-dual-basis.duality}))}}}=\sum_{\beta\in
\operatorname*{Comp}}c_{\beta}\left[  \alpha=\beta\right]  =c_{\alpha}%
\end{align*}
(since all addends of the sum $\sum_{\beta\in\operatorname*{Comp}}c_{\beta
}\left[  \alpha=\beta\right]  $ except for the $\beta=\alpha$ addend are $0$)
and therefore%
\begin{align}
c_{\alpha}  &  =\left\langle \eta_{\alpha}^{\ast\left(  q\right)  }%
,\ \eta_{\delta}^{\left(  q\right)  }\eta_{\varepsilon}^{\left(  q\right)
}\right\rangle \nonumber\\
&  =\underbrace{\sum_{\left(  \lambda,\mu\right)  \in I}}_{=\sum_{\lambda
,\mu\in\operatorname*{Comp}}}\underbrace{\left\langle d_{\lambda,\mu}^{\alpha
}\left(  q\right)  \eta_{\lambda}^{\ast\left(  q\right)  },\ \eta_{\delta
}^{\left(  q\right)  }\right\rangle }_{=d_{\lambda,\mu}^{\alpha}\left(
q\right)  \left\langle \eta_{\lambda}^{\ast\left(  q\right)  },\ \eta_{\delta
}^{\left(  q\right)  }\right\rangle }\left\langle \eta_{\mu}^{\ast\left(
q\right)  },\ \eta_{\varepsilon}^{\left(  q\right)  }\right\rangle \nonumber\\
&  \ \ \ \ \ \ \ \ \ \ \ \ \ \ \ \ \ \ \ \ \left(
\begin{array}
[c]{c}%
\text{by Lemma \ref{lem.NSym-duality.2inf}, applied to }f=\eta_{\delta
}^{\left(  q\right)  }\text{ and }g=\eta_{\varepsilon}^{\left(  q\right)  }\\
\text{and }h=\eta_{\alpha}^{\ast\left(  q\right)  }\text{ and }s_{\left(
\lambda,\mu\right)  }=d_{\lambda,\mu}^{\alpha}\left(  q\right)  \eta_{\lambda
}^{\ast\left(  q\right)  }\text{ and }t_{\left(  \lambda,\mu\right)  }%
=\eta_{\mu}^{\ast\left(  q\right)  }\\
\text{(since (\ref{pf.thm.product.c2.pf.DeltaI}) yields }\Delta\left(
\eta_{\alpha}^{\ast\left(  q\right)  }\right)  =\sum_{\left(  \lambda
,\mu\right)  \in I}d_{\lambda,\mu}^{\alpha}\left(  q\right)  \eta_{\lambda
}^{\ast\left(  q\right)  }\otimes\eta_{\mu}^{\ast\left(  q\right)  }\text{)}%
\end{array}
\right) \nonumber\\
&  =\sum_{\lambda,\mu\in\operatorname*{Comp}}d_{\lambda,\mu}^{\alpha}\left(
q\right)  \underbrace{\left\langle \eta_{\lambda}^{\ast\left(  q\right)
},\ \eta_{\delta}^{\left(  q\right)  }\right\rangle }_{\substack{=\left[
\lambda=\delta\right]  \\\text{(by (\ref{eq.prop.etastar-dual-basis.duality}%
))}}}\underbrace{\left\langle \eta_{\mu}^{\ast\left(  q\right)  }%
,\ \eta_{\varepsilon}^{\left(  q\right)  }\right\rangle }_{\substack{=\left[
\mu=\varepsilon\right]  \\\text{(by (\ref{eq.prop.etastar-dual-basis.duality}%
))}}}\nonumber\\
&  =\sum_{\lambda,\mu\in\operatorname*{Comp}}d_{\lambda,\mu}^{\alpha}\left(
q\right)  \underbrace{\left[  \lambda=\delta\right]  \cdot\left[
\mu=\varepsilon\right]  }_{\substack{=\left[  \lambda=\delta\text{ and }%
\mu=\varepsilon\right]  \\=\left[  \left(  \lambda,\mu\right)  =\left(
\delta,\varepsilon\right)  \right]  }}=\sum_{\lambda,\mu\in
\operatorname*{Comp}}d_{\lambda,\mu}^{\alpha}\left(  q\right)  \left[  \left(
\lambda,\mu\right)  =\left(  \delta,\varepsilon\right)  \right] \nonumber\\
&  =d_{\delta,\varepsilon}^{\alpha}\left(  q\right)
\label{pf.thm.product.c2.pf.4}%
\end{align}
(since all addends of the sum $\sum_{\lambda,\mu\in\operatorname*{Comp}%
}d_{\lambda,\mu}^{\alpha}\left(  q\right)  \left[  \left(  \lambda,\mu\right)
=\left(  \delta,\varepsilon\right)  \right]  $ except for the $\left(
\lambda,\mu\right)  =\left(  \delta,\varepsilon\right)  $ addend are $0$).

Now, (\ref{pf.thm.product.c2.pf.1}) becomes%
\[
\eta_{\delta}^{\left(  q\right)  }\eta_{\varepsilon}^{\left(  q\right)  }%
=\sum_{\beta\in\operatorname*{Comp}}c_{\beta}\eta_{\beta}^{\left(  q\right)
}=\sum_{\alpha\in\operatorname*{Comp}}\underbrace{c_{\alpha}}%
_{\substack{=d_{\delta,\varepsilon}^{\alpha}\left(  q\right)  \\\text{(by
(\ref{pf.thm.product.c2.pf.4}))}}}\eta_{\alpha}^{\left(  q\right)  }%
=\sum_{\alpha\in\operatorname*{Comp}}d_{\delta,\varepsilon}^{\alpha}\left(
q\right)  \eta_{\alpha}^{\left(  q\right)  }.
\]
This proves Claim 2.
\end{proof}

In the rest of this proof, we will use several different base rings. Thus, we
shall use the notation $\operatorname*{QSym}\nolimits_{\mathbf{k}}$ for what
we have previously been calling $\operatorname*{QSym}$ (that is, the ring of
quasisymmetric functions over the ring $\mathbf{k}$). Clearly, any ring
homomorphism $f:\mathbf{k}\rightarrow\mathbf{l}$ between two commutative rings
$\mathbf{k}$ and $\mathbf{l}$ canonically induces a ring homomorphism
$\operatorname*{QSym}\nolimits_{\mathbf{k}}\rightarrow\operatorname*{QSym}%
\nolimits_{\mathbf{l}}$, which we denote by $\operatorname*{QSym}%
\nolimits_{f}$. Moreover, if $\mathbf{k}$ is a subring of a commutative ring
$\mathbf{l}$, then $\operatorname*{QSym}\nolimits_{\mathbf{k}}$ canonically
becomes a subring of $\operatorname*{QSym}\nolimits_{\mathbf{l}}$.

We note that the definition of the power series $\eta_{\alpha}^{\left(
q\right)  }$ does not depend on the base ring. Thus, if $f:\mathbf{k}%
\rightarrow\mathbf{l}$ is a ring homomorphism between two commutative rings
$\mathbf{k}$ and $\mathbf{l}$, then any $\alpha\in\operatorname*{Comp}$ and
any $q\in\mathbf{k}$ satisfy%
\begin{equation}
\operatorname*{QSym}\nolimits_{f}\left(  \eta_{\alpha}^{\left(  q\right)
}\right)  =\eta_{\alpha}^{\left(  f\left(  q\right)  \right)  }
\label{pf.thm.product.feta}%
\end{equation}
(where the $\eta_{\alpha}^{\left(  q\right)  }$ on the left hand side is
defined in $\operatorname*{QSym}\nolimits_{\mathbf{k}}$, whereas the
$\eta_{\alpha}^{\left(  f\left(  q\right)  \right)  }$ on the right hand side
is defined in $\operatorname*{QSym}\nolimits_{\mathbf{l}}$). Likewise, if
$\mathbf{k}$ is a subring of $\mathbf{l}$, then the $\eta_{\alpha}^{\left(
q\right)  }$ in $\operatorname*{QSym}\nolimits_{\mathbf{k}}$ equals the
$\eta_{\alpha}^{\left(  q\right)  }$ in $\operatorname*{QSym}%
\nolimits_{\mathbf{l}}$. We shall use this tacitly soon.

\begin{statement}
\textit{Claim 3:} Let $\delta$ and $\varepsilon$ be two compositions. If
$\mathbf{k}$ is the polynomial ring $\mathbb{Z}\left[  X\right]  $, and if $q$
is the indeterminate $X$ in this ring, then%
\[
\eta_{\delta}^{\left(  q\right)  }\eta_{\varepsilon}^{\left(  q\right)  }%
=\sum_{\alpha\in\operatorname*{Comp}}d_{\delta,\varepsilon}^{\alpha}\left(
q\right)  \eta_{\alpha}^{\left(  q\right)  }.
\]
In other words, we have
\begin{equation}
\eta_{\delta}^{\left(  X\right)  }\eta_{\varepsilon}^{\left(  X\right)  }%
=\sum_{\alpha\in\operatorname*{Comp}}d_{\delta,\varepsilon}^{\alpha}\left(
X\right)  \eta_{\alpha}^{\left(  X\right)  }\ \ \ \ \ \ \ \ \ \ \text{in
}\operatorname*{QSym}\nolimits_{\mathbb{Z}\left[  X\right]  }.
\label{pf.thm.product.c3.equiv}%
\end{equation}

\end{statement}

\begin{proof}
[Proof of Claim 3.]Consider the field $\mathbb{Q}\left(  X\right)  $ of
rational functions in $X$ over $\mathbb{Q}$. Clearly, $\mathbb{Z}\left[
X\right]  $ is a subring of $\mathbb{Q}\left(  X\right)  $. Thus,
$\operatorname*{QSym}\nolimits_{\mathbb{Z}\left[  X\right]  }$ becomes a
subring of $\operatorname*{QSym}\nolimits_{\mathbb{Q}\left(  X\right)  }$.

In the ring $\mathbb{Q}\left(  X\right)  $, the polynomial $X+1$ is
invertible. Thus, Claim 2 (applied to $\mathbb{Q}\left(  X\right)  $, $X$ and
$X+1$ instead of $\mathbf{k}$, $q$ and $r$) yields that%
\begin{equation}
\eta_{\delta}^{\left(  X\right)  }\eta_{\varepsilon}^{\left(  X\right)  }%
=\sum_{\alpha\in\operatorname*{Comp}}d_{\delta,\varepsilon}^{\alpha}\left(
X\right)  \eta_{\alpha}^{\left(  X\right)  }\ \ \ \ \ \ \ \ \ \ \text{in
}\operatorname*{QSym}\nolimits_{\mathbb{Q}\left(  X\right)  }.
\label{pf.thm.product.c3.pf.1}%
\end{equation}
But $\operatorname*{QSym}\nolimits_{\mathbb{Z}\left[  X\right]  }$ is a
subring of $\operatorname*{QSym}\nolimits_{\mathbb{Q}\left(  X\right)  }$, and
both sides of the equality (\ref{pf.thm.product.c3.pf.1}) belong to
$\operatorname*{QSym}\nolimits_{\mathbb{Z}\left[  X\right]  }$ (since
$d_{\delta,\varepsilon}^{\alpha}\left(  X\right)  \in\mathbb{Z}\left[
X\right]  $ and $\eta_{\alpha}^{\left(  X\right)  }\in\operatorname*{QSym}%
\nolimits_{\mathbb{Z}\left[  X\right]  }$ for all $\alpha\in
\operatorname*{Comp}$), and do not depend on the base ring\footnote{Indeed,
the power series $\eta_{\alpha}^{\left(  X\right)  }$ defined over
$\mathbb{Z}\left[  X\right]  $ equals the power series $\eta_{\alpha}^{\left(
X\right)  }$ defined over $\mathbb{Q}\left(  X\right)  $.}. Hence, the
equality (\ref{pf.thm.product.c3.pf.1}) holds in $\operatorname*{QSym}%
\nolimits_{\mathbb{Z}\left[  X\right]  }$ as well. In other words, we have
\[
\eta_{\delta}^{\left(  X\right)  }\eta_{\varepsilon}^{\left(  X\right)  }%
=\sum_{\alpha\in\operatorname*{Comp}}d_{\delta,\varepsilon}^{\alpha}\left(
X\right)  \eta_{\alpha}^{\left(  X\right)  }\ \ \ \ \ \ \ \ \ \ \text{in
}\operatorname*{QSym}\nolimits_{\mathbb{Z}\left[  X\right]  }.
\]
In other words, (\ref{pf.thm.product.c3.equiv}) holds. This proves Claim 3.
\end{proof}

\begin{statement}
\textit{Claim 4:} Let $\delta$ and $\varepsilon$ be two compositions. Then,%
\[
\eta_{\delta}^{\left(  q\right)  }\eta_{\varepsilon}^{\left(  q\right)  }%
=\sum_{\alpha\in\operatorname*{Comp}}d_{\delta,\varepsilon}^{\alpha}\left(
q\right)  \eta_{\alpha}^{\left(  q\right)  }.
\]

\end{statement}

\begin{proof}
[Proof of Claim 4.]Consider the polynomial ring $\mathbb{Z}\left[  X\right]
$. By the universal property of a polynomial ring, there exists a unique
$\mathbb{Z}$-algebra homomorphism $f:\mathbb{Z}\left[  X\right]
\rightarrow\mathbf{k}$ that sends $X$ to $q$. Consider this $f$. Explicitly,
$f$ is given by%
\begin{equation}
f\left(  u\left(  X\right)  \right)  =u\left(  q\right)
\ \ \ \ \ \ \ \ \ \ \text{for any polynomial }u\left(  X\right)  \in
\mathbb{Z}\left[  X\right]  . \label{pf.thm.product.c4.fuX}%
\end{equation}
The map $f$ is a $\mathbb{Z}$-algebra homomorphism, thus a ring homomorphism,
and therefore induces a ring homomorphism $\operatorname*{QSym}\nolimits_{f}%
:\operatorname*{QSym}\nolimits_{\mathbb{Z}\left[  X\right]  }\rightarrow
\operatorname*{QSym}\nolimits_{\mathbf{k}}$. Applying this ring homomorphism
$\operatorname*{QSym}\nolimits_{f}$ to both sides of
(\ref{pf.thm.product.c3.equiv}), we obtain%
\begin{align*}
\operatorname*{QSym}\nolimits_{f}\left(  \eta_{\delta}^{\left(  X\right)
}\right)  \cdot\operatorname*{QSym}\nolimits_{f}\left(  \eta_{\varepsilon
}^{\left(  X\right)  }\right)   &  =\sum_{\alpha\in\operatorname*{Comp}%
}\operatorname*{QSym}\nolimits_{f}\left(  d_{\delta,\varepsilon}^{\alpha
}\left(  X\right)  \right)  \cdot\operatorname*{QSym}\nolimits_{f}\left(
\eta_{\alpha}^{\left(  X\right)  }\right) \\
&  \ \ \ \ \ \ \ \ \ \ \text{in }\operatorname*{QSym}\nolimits_{\mathbf{k}}.
\end{align*}

Since every composition $\alpha\in\operatorname*{Comp}$ satisfies%
\begin{align*}
\operatorname*{QSym}\nolimits_{f}\left(  \eta_{\alpha}^{\left(  X\right)
}\right)   &  =\eta_{\alpha}^{\left(  f\left(  X\right)  \right)
}\ \ \ \ \ \ \ \ \ \ \left(  \text{by (\ref{pf.thm.product.feta})}\right) \\
&  =\eta_{\alpha}^{\left(  q\right)  }\ \ \ \ \ \ \ \ \ \ \left(  \text{since
}f\left(  X\right)  =q\right)  ,
\end{align*}
we can rewrite this as%
\begin{align*}
\eta_{\delta}^{\left(  q\right)  }\eta_{\varepsilon}^{\left(  q\right)  }  &
=\sum_{\alpha\in\operatorname*{Comp}}\underbrace{\operatorname*{QSym}%
\nolimits_{f}\left(  d_{\delta,\varepsilon}^{\alpha}\left(  X\right)  \right)
}_{\substack{=f\left(  d_{\delta,\varepsilon}^{\alpha}\left(  X\right)
\right)  \\\text{(since the homomorphism }\operatorname*{QSym}\nolimits_{f}%
\\\text{acts as }f\text{ on }\mathbb{Z}\left[  X\right]  \text{)}}%
}\eta_{\alpha}^{\left(  q\right)  }\\
&  =\sum_{\alpha\in\operatorname*{Comp}}\underbrace{f\left(  d_{\delta
,\varepsilon}^{\alpha}\left(  X\right)  \right)  }_{\substack{=d_{\delta
,\varepsilon}^{\alpha}\left(  q\right)  \\\text{(by
(\ref{pf.thm.product.c4.fuX}))}}}\eta_{\alpha}^{\left(  q\right)  }%
=\sum_{\alpha\in\operatorname*{Comp}}d_{\delta,\varepsilon}^{\alpha}\left(
q\right)  \eta_{\alpha}^{\left(  q\right)  }.
\end{align*}
This proves Claim 4.
\end{proof}

\begin{statement}
\textit{Claim 5:} Let $\alpha=\left(  \alpha_{1},\alpha_{2},\ldots,\alpha
_{k}\right)  $, $\delta$ and $\varepsilon$ be three compositions. Then,%
\begin{align}
d_{\delta,\varepsilon}^{\alpha}\left(  q\right)   &  =\sum_{\substack{\beta
_{1},\beta_{2},\ldots,\beta_{k}\in\operatorname*{Comp};\\\gamma_{1},\gamma
_{2},\ldots,\gamma_{k}\in\operatorname*{Comp};\\\beta_{1}\beta_{2}\cdots
\beta_{k}=\delta;\\\gamma_{1}\gamma_{2}\cdots\gamma_{k}=\varepsilon
;\\\left\vert \ell\left(  \beta_{s}\right)  -\ell\left(  \gamma_{s}\right)
\right\vert \leq1\text{ for all }s;\\\alpha=\left(  \left\vert \beta
_{1}\right\vert +\left\vert \gamma_{1}\right\vert ,\ \left\vert \beta
_{2}\right\vert +\left\vert \gamma_{2}\right\vert ,\ \ldots,\ \left\vert
\beta_{k}\right\vert +\left\vert \gamma_{k}\right\vert \right)  }}\left(
-q\right)  ^{\sum_{s=1}^{k}\max\left\{  \ell\left(  \beta_{s}\right)
,\ell\left(  \gamma_{s}\right)  \right\}  -k}\nonumber\\
&  \ \ \ \ \ \ \ \ \ \ \ \ \ \ \ \ \ \ \ \ \cdot\left(  q-1\right)  ^{\left(
\text{\# of all }s\in\left[  k\right]  \text{ such that }\ell\left(  \beta
_{s}\right)  =\ell\left(  \gamma_{s}\right)  \right)  }.
\label{pf.thm.product.1.dq=3}%
\end{align}

\end{statement}

\begin{proof}
[Proof of Claim 5.]The condition \textquotedblleft$\left\vert \beta
_{s}\right\vert +\left\vert \gamma_{s}\right\vert =\alpha_{s}$ for all
$s$\textquotedblright\ under the summation sign in (\ref{pf.thm.product.1.dq=}%
) is equivalent to the condition \textquotedblleft$\alpha=\left(  \left\vert
\beta_{1}\right\vert +\left\vert \gamma_{1}\right\vert ,\ \left\vert \beta
_{2}\right\vert +\left\vert \gamma_{2}\right\vert ,\ \ldots,\ \left\vert
\beta_{k}\right\vert +\left\vert \gamma_{k}\right\vert \right)  $%
\textquotedblright\ (since $\alpha=\left(  \alpha_{1},\alpha_{2},\ldots
,\alpha_{k}\right)  $). Thus, we can replace the former condition in
(\ref{pf.thm.product.1.dq=}) by the latter. The result of this replacement is
precisely the equality (\ref{pf.thm.product.1.dq=3}). Hence, Claim 5 is proved.
\end{proof}

Now, Theorem \ref{thm.product.1} is just a restatement of Claim 4. Indeed, let
$\delta$ and $\varepsilon$ be two compositions. Then,%
\begin{align*}
&  \eta_{\delta}^{\left(  q\right)  }\eta_{\varepsilon}^{\left(  q\right)  }\\
&  =\underbrace{\sum_{\alpha\in\operatorname*{Comp}}}_{\substack{=\sum
_{k\in\mathbb{N}}\ \ \sum_{\alpha=\left(  \alpha_{1},\alpha_{2},\ldots
,\alpha_{k}\right)  \in\operatorname*{Comp}}\\\text{(since any composition}%
\\\text{has a unique length)}}}d_{\delta,\varepsilon}^{\alpha}\left(
q\right)  \eta_{\alpha}^{\left(  q\right)  }\ \ \ \ \ \ \ \ \ \ \left(
\text{by Claim 4}\right) \\
&  =\sum_{k\in\mathbb{N}}\ \ \sum_{\alpha=\left(  \alpha_{1},\alpha_{2}%
,\ldots,\alpha_{k}\right)  \in\operatorname*{Comp}}d_{\delta,\varepsilon
}^{\alpha}\left(  q\right)  \eta_{\alpha}^{\left(  q\right)  }\\
&  =\sum_{k\in\mathbb{N}}\ \ \sum_{\alpha=\left(  \alpha_{1},\alpha_{2}%
,\ldots,\alpha_{k}\right)  \in\operatorname*{Comp}}\left(  \sum
_{\substack{\beta_{1},\beta_{2},\ldots,\beta_{k}\in\operatorname*{Comp}%
;\\\gamma_{1},\gamma_{2},\ldots,\gamma_{k}\in\operatorname*{Comp};\\\beta
_{1}\beta_{2}\cdots\beta_{k}=\delta;\\\gamma_{1}\gamma_{2}\cdots\gamma
_{k}=\varepsilon;\\\left\vert \ell\left(  \beta_{s}\right)  -\ell\left(
\gamma_{s}\right)  \right\vert \leq1\text{ for all }s;\\\alpha=\left(
\left\vert \beta_{1}\right\vert +\left\vert \gamma_{1}\right\vert
,\ \left\vert \beta_{2}\right\vert +\left\vert \gamma_{2}\right\vert
,\ \ldots,\ \left\vert \beta_{k}\right\vert +\left\vert \gamma_{k}\right\vert
\right)  }}\left(  -q\right)  ^{\sum_{s=1}^{k}\max\left\{  \ell\left(
\beta_{s}\right)  ,\ell\left(  \gamma_{s}\right)  \right\}  -k}\right. \\
&  \ \ \ \ \ \ \ \ \ \ \ \ \ \ \ \ \ \ \ \ \left.  \cdot\left(  q-1\right)
^{\left(  \text{\# of all }s\in\left[  k\right]  \text{ such that }\ell\left(
\beta_{s}\right)  =\ell\left(  \gamma_{s}\right)  \right)  }\right)  \cdot
\eta_{\alpha}^{\left(  q\right)  }\\
&  \ \ \ \ \ \ \ \ \ \ \ \ \ \ \ \ \ \ \ \ \ \ \ \ \ \ \ \ \ \ \left(
\text{by (\ref{pf.thm.product.1.dq=3})}\right) \\
&  =\sum_{k\in\mathbb{N}}\ \ \sum_{\alpha=\left(  \alpha_{1},\alpha_{2}%
,\ldots,\alpha_{k}\right)  \in\operatorname*{Comp}}\ \ \sum_{\substack{\beta
_{1},\beta_{2},\ldots,\beta_{k}\in\operatorname*{Comp};\\\gamma_{1},\gamma
_{2},\ldots,\gamma_{k}\in\operatorname*{Comp};\\\beta_{1}\beta_{2}\cdots
\beta_{k}=\delta;\\\gamma_{1}\gamma_{2}\cdots\gamma_{k}=\varepsilon
;\\\left\vert \ell\left(  \beta_{s}\right)  -\ell\left(  \gamma_{s}\right)
\right\vert \leq1\text{ for all }s;\\\alpha=\left(  \left\vert \beta
_{1}\right\vert +\left\vert \gamma_{1}\right\vert ,\ \left\vert \beta
_{2}\right\vert +\left\vert \gamma_{2}\right\vert ,\ \ldots,\ \left\vert
\beta_{k}\right\vert +\left\vert \gamma_{k}\right\vert \right)  }}\left(
-q\right)  ^{\sum_{s=1}^{k}\max\left\{  \ell\left(  \beta_{s}\right)
,\ell\left(  \gamma_{s}\right)  \right\}  -k}\\
&  \ \ \ \ \ \ \ \ \ \ \ \ \ \ \ \ \ \ \ \ \cdot\left(  q-1\right)  ^{\left(
\text{\# of all }s\in\left[  k\right]  \text{ such that }\ell\left(  \beta
_{s}\right)  =\ell\left(  \gamma_{s}\right)  \right)  }\\
&  \ \ \ \ \ \ \ \ \ \ \ \ \ \ \ \ \ \ \ \ \cdot\underbrace{\eta_{\alpha
}^{\left(  q\right)  }}_{\substack{=\eta_{\left(  \left\vert \beta
_{1}\right\vert +\left\vert \gamma_{1}\right\vert ,\ \left\vert \beta
_{2}\right\vert +\left\vert \gamma_{2}\right\vert ,\ \ldots,\ \left\vert
\beta_{k}\right\vert +\left\vert \gamma_{k}\right\vert \right)  }^{\left(
q\right)  }\\\text{(since }\alpha=\left(  \left\vert \beta_{1}\right\vert
+\left\vert \gamma_{1}\right\vert ,\ \left\vert \beta_{2}\right\vert
+\left\vert \gamma_{2}\right\vert ,\ \ldots,\ \left\vert \beta_{k}\right\vert
+\left\vert \gamma_{k}\right\vert \right)  \text{)}}}
\end{align*}%
\begin{align}
&  =\sum_{k\in\mathbb{N}}\ \ \sum_{\alpha=\left(  \alpha_{1},\alpha_{2}%
,\ldots,\alpha_{k}\right)  \in\operatorname*{Comp}}\ \ \sum_{\substack{\beta
_{1},\beta_{2},\ldots,\beta_{k}\in\operatorname*{Comp};\\\gamma_{1},\gamma
_{2},\ldots,\gamma_{k}\in\operatorname*{Comp};\\\beta_{1}\beta_{2}\cdots
\beta_{k}=\delta;\\\gamma_{1}\gamma_{2}\cdots\gamma_{k}=\varepsilon
;\\\left\vert \ell\left(  \beta_{s}\right)  -\ell\left(  \gamma_{s}\right)
\right\vert \leq1\text{ for all }s;\\\alpha=\left(  \left\vert \beta
_{1}\right\vert +\left\vert \gamma_{1}\right\vert ,\ \left\vert \beta
_{2}\right\vert +\left\vert \gamma_{2}\right\vert ,\ \ldots,\ \left\vert
\beta_{k}\right\vert +\left\vert \gamma_{k}\right\vert \right)  }}\left(
-q\right)  ^{\sum_{s=1}^{k}\max\left\{  \ell\left(  \beta_{s}\right)
,\ell\left(  \gamma_{s}\right)  \right\}  -k}\nonumber\\
&  \ \ \ \ \ \ \ \ \ \ \ \ \ \ \ \ \ \ \ \ \cdot\left(  q-1\right)  ^{\left(
\text{\# of all }s\in\left[  k\right]  \text{ such that }\ell\left(  \beta
_{s}\right)  =\ell\left(  \gamma_{s}\right)  \right)  }\nonumber\\
&  \ \ \ \ \ \ \ \ \ \ \ \ \ \ \ \ \ \ \ \ \cdot\eta_{\left(  \left\vert
\beta_{1}\right\vert +\left\vert \gamma_{1}\right\vert ,\ \left\vert \beta
_{2}\right\vert +\left\vert \gamma_{2}\right\vert ,\ \ldots,\ \left\vert
\beta_{k}\right\vert +\left\vert \gamma_{k}\right\vert \right)  }^{\left(
q\right)  }. \label{pf.thm.product.1.at2}%
\end{align}

However, for each $k\in\mathbb{N}$, we have the following equality of
summation signs:%
\begin{align*}
&  \sum_{\alpha=\left(  \alpha_{1},\alpha_{2},\ldots,\alpha_{k}\right)
\in\operatorname*{Comp}}\ \ \sum_{\substack{\beta_{1},\beta_{2},\ldots
,\beta_{k}\in\operatorname*{Comp};\\\gamma_{1},\gamma_{2},\ldots,\gamma_{k}%
\in\operatorname*{Comp};\\\beta_{1}\beta_{2}\cdots\beta_{k}=\delta
;\\\gamma_{1}\gamma_{2}\cdots\gamma_{k}=\varepsilon;\\\left\vert \ell\left(
\beta_{s}\right)  -\ell\left(  \gamma_{s}\right)  \right\vert \leq1\text{ for
all }s;\\\alpha=\left(  \left\vert \beta_{1}\right\vert +\left\vert \gamma
_{1}\right\vert ,\ \left\vert \beta_{2}\right\vert +\left\vert \gamma
_{2}\right\vert ,\ \ldots,\ \left\vert \beta_{k}\right\vert +\left\vert
\gamma_{k}\right\vert \right)  }}\\
&  =\sum_{\substack{\beta_{1},\beta_{2},\ldots,\beta_{k}\in
\operatorname*{Comp};\\\gamma_{1},\gamma_{2},\ldots,\gamma_{k}\in
\operatorname*{Comp};\\\beta_{1}\beta_{2}\cdots\beta_{k}=\delta;\\\gamma
_{1}\gamma_{2}\cdots\gamma_{k}=\varepsilon;\\\left\vert \ell\left(  \beta
_{s}\right)  -\ell\left(  \gamma_{s}\right)  \right\vert \leq1\text{ for all
}s;\\\left(  \left\vert \beta_{1}\right\vert +\left\vert \gamma_{1}\right\vert
,\ \left\vert \beta_{2}\right\vert +\left\vert \gamma_{2}\right\vert
,\ \ldots,\ \left\vert \beta_{k}\right\vert +\left\vert \gamma_{k}\right\vert
\right)  \in\operatorname*{Comp}}}\\
&  \ \ \ \ \ \ \ \ \ \ \ \ \ \ \ \ \ \ \ \ \left(
\begin{array}
[c]{c}%
\text{since the condition \textquotedblleft}\alpha=\left(  \left\vert
\beta_{1}\right\vert +\left\vert \gamma_{1}\right\vert ,\ \left\vert \beta
_{2}\right\vert +\left\vert \gamma_{2}\right\vert ,\ \ldots,\ \left\vert
\beta_{k}\right\vert +\left\vert \gamma_{k}\right\vert \right)
\text{\textquotedblright}\\
\text{under the second summation sign uniquely determines }\alpha
\end{array}
\right) \\
&  =\sum_{\substack{\beta_{1},\beta_{2},\ldots,\beta_{k}\in
\operatorname*{Comp};\\\gamma_{1},\gamma_{2},\ldots,\gamma_{k}\in
\operatorname*{Comp};\\\beta_{1}\beta_{2}\cdots\beta_{k}=\delta;\\\gamma
_{1}\gamma_{2}\cdots\gamma_{k}=\varepsilon;\\\left\vert \ell\left(  \beta
_{s}\right)  -\ell\left(  \gamma_{s}\right)  \right\vert \leq1\text{ for all
}s;\\\left\vert \beta_{s}\right\vert +\left\vert \gamma_{s}\right\vert
>0\text{ for all }s}}\ \ \ \ \ \ \ \ \ \ \left(
\begin{array}
[c]{c}%
\text{since the condition}\\
\text{\textquotedblleft}\left(  \left\vert \beta_{1}\right\vert +\left\vert
\gamma_{1}\right\vert ,\ \left\vert \beta_{2}\right\vert +\left\vert
\gamma_{2}\right\vert ,\ \ldots,\ \left\vert \beta_{k}\right\vert +\left\vert
\gamma_{k}\right\vert \right)  \in\operatorname*{Comp}\text{\textquotedblright%
}\\
\text{is equivalent to \textquotedblleft}\left\vert \beta_{s}\right\vert
+\left\vert \gamma_{s}\right\vert >0\text{ for all }s\text{\textquotedblright}%
\end{array}
\right) \\
&  =\sum_{\substack{\beta_{1},\beta_{2},\ldots,\beta_{k}\in
\operatorname*{Comp};\\\gamma_{1},\gamma_{2},\ldots,\gamma_{k}\in
\operatorname*{Comp};\\\beta_{1}\beta_{2}\cdots\beta_{k}=\delta;\\\gamma
_{1}\gamma_{2}\cdots\gamma_{k}=\varepsilon;\\\left\vert \ell\left(  \beta
_{s}\right)  -\ell\left(  \gamma_{s}\right)  \right\vert \leq1\text{ for all
}s;\\\ell\left(  \beta_{s}\right)  +\ell\left(  \gamma_{s}\right)  >0\text{
for all }s}}\ \ \ \ \ \ \ \ \ \ \left(
\begin{array}
[c]{c}%
\text{since the condition \textquotedblleft}\left\vert \beta_{s}\right\vert
+\left\vert \gamma_{s}\right\vert >0\text{\textquotedblright}\\
\text{on two compositions }\beta_{s}\text{ and }\gamma_{s}\\
\text{is equivalent to \textquotedblleft}\ell\left(  \beta_{s}\right)
+\ell\left(  \gamma_{s}\right)  >0\text{\textquotedblright}\\
\text{(indeed, both conditions are}\\
\text{equivalent to }\left(  \beta_{s},\gamma_{s}\right)  \neq\left(
\varnothing,\varnothing\right)  \text{)}%
\end{array}
\right)  .
\end{align*}
Hence, we can rewrite (\ref{pf.thm.product.1.at2}) as%
\begin{align*}
\eta_{\delta}^{\left(  q\right)  }\eta_{\varepsilon}^{\left(  q\right)  }  &
=\sum_{k\in\mathbb{N}}\ \ \sum_{\substack{\beta_{1},\beta_{2},\ldots,\beta
_{k}\in\operatorname*{Comp};\\\gamma_{1},\gamma_{2},\ldots,\gamma_{k}%
\in\operatorname*{Comp};\\\beta_{1}\beta_{2}\cdots\beta_{k}=\delta
;\\\gamma_{1}\gamma_{2}\cdots\gamma_{k}=\varepsilon;\\\left\vert \ell\left(
\beta_{s}\right)  -\ell\left(  \gamma_{s}\right)  \right\vert \leq1\text{ for
all }s;\\\ell\left(  \beta_{s}\right)  +\ell\left(  \gamma_{s}\right)
>0\text{ for all }s}}\left(  -q\right)  ^{\sum_{s=1}^{k}\max\left\{
\ell\left(  \beta_{s}\right)  ,\ell\left(  \gamma_{s}\right)  \right\}  -k}\\
&  \ \ \ \ \ \ \ \ \ \ \ \ \ \ \ \ \ \ \ \ \cdot\left(  q-1\right)  ^{\left(
\text{\# of all }s\in\left[  k\right]  \text{ such that }\ell\left(  \beta
_{s}\right)  =\ell\left(  \gamma_{s}\right)  \right)  }\\
&  \ \ \ \ \ \ \ \ \ \ \ \ \ \ \ \ \ \ \ \ \cdot\eta_{\left(  \left\vert
\beta_{1}\right\vert +\left\vert \gamma_{1}\right\vert ,\ \left\vert \beta
_{2}\right\vert +\left\vert \gamma_{2}\right\vert ,\ \ldots,\ \left\vert
\beta_{k}\right\vert +\left\vert \gamma_{k}\right\vert \right)  }^{\left(
q\right)  }.
\end{align*}
Combining the two summation signs here into a single sum, we can rewrite this
as%
\begin{align*}
\eta_{\delta}^{\left(  q\right)  }\eta_{\varepsilon}^{\left(  q\right)  }  &
=\sum_{\substack{k\in\mathbb{N};\\\beta_{1},\beta_{2},\ldots,\beta_{k}%
\in\operatorname*{Comp};\\\gamma_{1},\gamma_{2},\ldots,\gamma_{k}%
\in\operatorname*{Comp};\\\beta_{1}\beta_{2}\cdots\beta_{k}=\delta
;\\\gamma_{1}\gamma_{2}\cdots\gamma_{k}=\varepsilon;\\\left\vert \ell\left(
\beta_{s}\right)  -\ell\left(  \gamma_{s}\right)  \right\vert \leq1\text{ for
all }s;\\\ell\left(  \beta_{s}\right)  +\ell\left(  \gamma_{s}\right)
>0\text{ for all }s}}\left(  -q\right)  ^{\sum_{s=1}^{k}\max\left\{
\ell\left(  \beta_{s}\right)  ,\ell\left(  \gamma_{s}\right)  \right\}  -k}\\
&  \ \ \ \ \ \ \ \ \ \ \ \ \ \ \ \ \ \ \ \ \cdot\left(  q-1\right)  ^{\left(
\text{\# of all }s\in\left[  k\right]  \text{ such that }\ell\left(  \beta
_{s}\right)  =\ell\left(  \gamma_{s}\right)  \right)  }\\
&  \ \ \ \ \ \ \ \ \ \ \ \ \ \ \ \ \ \ \ \ \cdot\eta_{\left(  \left\vert
\beta_{1}\right\vert +\left\vert \gamma_{1}\right\vert ,\ \left\vert \beta
_{2}\right\vert +\left\vert \gamma_{2}\right\vert ,\ \ldots,\ \left\vert
\beta_{k}\right\vert +\left\vert \gamma_{k}\right\vert \right)  }^{\left(
q\right)  }.
\end{align*}
Thus, Theorem \ref{thm.product.1} is proved.
\end{proof}

\subsection{The product rule in terms of stufflers}

We will next rewrite Theorem \ref{thm.product.1} in a somewhat different
language, using certain surjective maps instead of factorizations of
compositions. First, we introduce several pieces of notation:

\begin{definition}
Let $i$ and $j$ be two integers. Then, we write $i\approx j$ (and say that $i$
is \emph{nearly equal} to $j$) if and only if $\left\vert i-j\right\vert
\leq1$.

(Of course, $\approx$ is not an equivalence relation.)
\end{definition}

\begin{definition}
\label{def.stufufuffler}Let $\delta=\left(  \delta_{1},\delta_{2}%
,\ldots,\delta_{\ell}\right)  $ and $\varepsilon=\left(  \varepsilon
_{1},\varepsilon_{2},\ldots,\varepsilon_{m}\right)  $ be two compositions.

Fix two chains (i.e., totally ordered sets) $P=\left\{  p_{1}<p_{2}%
<\cdots<p_{\ell}\right\}  $ and $Q=\left\{  q_{1}<q_{2}<\cdots<q_{m}\right\}
$, and let%
\[
U=P\sqcup Q
\]
be their disjoint union. This $U$ is a poset with $\ell+m$ elements
$p_{1},p_{2},\ldots,p_{\ell},q_{1},q_{2},\ldots,q_{m}$, whose relations are
given by $p_{1}<p_{2}<\cdots<p_{\ell}$ and $q_{1}<q_{2}<\cdots<q_{m}$ (while
each $p_{i}$ is incomparable to each $q_{j}$).

If $f:U\rightarrow X$ is a map from $U$ to any set $X$, and if $s\in X$ is any
element, then we define the two sets%
\begin{align*}
f_{P}^{-1}\left(  s\right)   &  :=\left\{  u\in\left[  \ell\right]
\ \mid\ f\left(  p_{u}\right)  =s\right\}  \ \ \ \ \ \ \ \ \ \ \text{and}\\
f_{Q}^{-1}\left(  s\right)   &  :=\left\{  v\in\left[  m\right]
\ \mid\ f\left(  q_{v}\right)  =s\right\}  .
\end{align*}
(Essentially, $f_{P}^{-1}\left(  s\right)  $ and $f_{Q}^{-1}\left(  s\right)
$ are the sets of the preimages of $s$ in $P$ and $Q$, respectively, except
that they consist of numbers instead of actual elements of $P$ and $Q$.)

A \emph{stufufuffler} for $\delta$ and $\varepsilon$ shall mean a surjective
and weakly order-preserving map%
\[
f:U\rightarrow\left\{  1<2<\cdots<k\right\}  \ \ \ \ \ \ \ \ \ \ \text{for
some }k\in\mathbb{N}%
\]
with the property that each $s\in\left\{  1<2<\cdots<k\right\}  $ satisfies%
\begin{equation}
\left\vert f_{P}^{-1}\left(  s\right)  \right\vert \approx\left\vert
f_{Q}^{-1}\left(  s\right)  \right\vert .
\label{eq.def.limstuffler.liminality}%
\end{equation}

(\textquotedblleft Weakly order-preserving\textquotedblright\ means that if
$u$ and $v$ are two elements of the poset $U$ satisfying $u<v$, then $f\left(
u\right)  \leq f\left(  v\right)  $.)

If $f:U\rightarrow\left\{  1<2<\cdots<k\right\}  $ is a stufufuffler for
$\delta$ and $\varepsilon$, then we define three further concepts:

\begin{itemize}
\item We define the \emph{weight} $\operatorname*{wt}\left(  f\right)  $ of
$f$ to be the composition $\left(  \operatorname*{wt}\nolimits_{1}\left(
f\right)  ,\operatorname*{wt}\nolimits_{2}\left(  f\right)  ,\ldots
,\operatorname*{wt}\nolimits_{k}\left(  f\right)  \right)  $, where%
\begin{align*}
\operatorname*{wt}\nolimits_{s}\left(  f\right)   &  =\sum_{u\in f_{P}%
^{-1}\left(  s\right)  }\delta_{u}+\sum_{v\in f_{Q}^{-1}\left(  s\right)
}\varepsilon_{v}\\
&  =\sum_{\substack{u\in\left[  \ell\right]  ;\\f\left(  p_{u}\right)
=s}}\delta_{u}+\sum_{\substack{v\in\left[  m\right]  ;\\f\left(  q_{v}\right)
=s}}\varepsilon_{v}\ \ \ \ \ \ \ \ \ \ \text{for each }s\in\left[  k\right]  .
\end{align*}
(Note that (\ref{eq.def.limstuffler.liminality}) ensures that the two sums on
the right hand side here have nearly equal numbers of addends. Moreover, the
surjectivity of $f$ ensures that at least one of these two sums has at least
one addend, and thus $\operatorname*{wt}\nolimits_{s}\left(  f\right)  $ is a
positive integer; therefore, $\operatorname*{wt}\left(  f\right)  $ is a composition.)

\item We define the \emph{loss} of $f$ to be the nonnegative integer%
\[
\operatorname*{loss}\left(  f\right)  :=\sum_{s=1}^{k}\max\left\{  \left\vert
f_{P}^{-1}\left(  s\right)  \right\vert ,\ \ \left\vert f_{Q}^{-1}\left(
s\right)  \right\vert \right\}  -k.
\]
(This really is a nonnegative integer, since the surjectivity of $f$ yields
that $\max\left\{  \left\vert f_{P}^{-1}\left(  s\right)  \right\vert
,\ \ \left\vert f_{Q}^{-1}\left(  s\right)  \right\vert \right\}  \geq1$ for
each $s\in\left[  k\right]  $, and thus $\operatorname*{loss}\left(  f\right)
=\sum_{s=1}^{k}\underbrace{\max\left\{  \left\vert f_{P}^{-1}\left(  s\right)
\right\vert ,\ \ \left\vert f_{Q}^{-1}\left(  s\right)  \right\vert \right\}
}_{\geq1}-\,k\geq\underbrace{\sum_{s=1}^{k}1}_{=k}-\,k=0$.)

\item We define the \emph{poise} of $f$ to be the nonnegative integer%
\[
\operatorname*{poise}\left(  f\right)  :=\left(  \text{\# of all }s\in\left[
k\right]  \text{ such that }\left\vert f_{P}^{-1}\left(  s\right)  \right\vert
=\left\vert f_{Q}^{-1}\left(  s\right)  \right\vert \right)  .
\]

\end{itemize}
\end{definition}

\begin{example}
Let $\delta=\left(  a,b\right)  $ and $\varepsilon=\left(  c,d,e\right)  $ be
two compositions. Then, the poset $U$ in Definition \ref{def.stufufuffler} is
$U=\left\{  p_{1}<p_{2}\right\}  \sqcup\left\{  q_{1}<q_{2}<q_{3}\right\}  $.
The following maps (written in two-line notation) are stufufufflers for
$\delta$ and $\varepsilon$:%
\begin{align*}
&  \left(
\begin{array}
[c]{ccccc}%
p_{1} & p_{2} & q_{1} & q_{2} & q_{3}\\
1 & 2 & 3 & 4 & 5
\end{array}
\right)  ,\ \ \ \ \ \ \ \ \ \ \left(
\begin{array}
[c]{ccccc}%
p_{1} & p_{2} & q_{1} & q_{2} & q_{3}\\
2 & 5 & 1 & 3 & 4
\end{array}
\right)  ,\\
&  \left(
\begin{array}
[c]{ccccc}%
p_{1} & p_{2} & q_{1} & q_{2} & q_{3}\\
1 & 1 & 1 & 2 & 3
\end{array}
\right)  ,\ \ \ \ \ \ \ \ \ \ \left(
\begin{array}
[c]{ccccc}%
p_{1} & p_{2} & q_{1} & q_{2} & q_{3}\\
1 & 2 & 2 & 2 & 3
\end{array}
\right)  ,\\
&  \left(
\begin{array}
[c]{ccccc}%
p_{1} & p_{2} & q_{1} & q_{2} & q_{3}\\
2 & 2 & 1 & 2 & 3
\end{array}
\right)  ,\ \ \ \ \ \ \ \ \ \ \left(
\begin{array}
[c]{ccccc}%
p_{1} & p_{2} & q_{1} & q_{2} & q_{3}\\
1 & 1 & 1 & 1 & 1
\end{array}
\right)  ,\\
&  \left(
\begin{array}
[c]{ccccc}%
p_{1} & p_{2} & q_{1} & q_{2} & q_{3}\\
1 & 1 & 1 & 1 & 2
\end{array}
\right)  ,\ \ \ \ \ \ \ \ \ \ \left(
\begin{array}
[c]{ccccc}%
p_{1} & p_{2} & q_{1} & q_{2} & q_{3}\\
1 & 2 & 1 & 1 & 2
\end{array}
\right)  .
\end{align*}
(The list is not exhaustive -- there are many more stufufufflers for $\delta$
and $\varepsilon$.)

On the other hand, here are some maps (in two-line notation) that are not
stufufufflers for $\delta$ and $\varepsilon$:

\begin{itemize}
\item The map $\left(
\begin{array}
[c]{ccccc}%
p_{1} & p_{2} & q_{1} & q_{2} & q_{3}\\
1 & 2 & 1 & 1 & 1
\end{array}
\right)  $ is not a stufufuffler, since it violates
(\ref{eq.def.limstuffler.liminality}) for $s=1$.

\item The map $\left(
\begin{array}
[c]{ccccc}%
p_{1} & p_{2} & q_{1} & q_{2} & q_{3}\\
1 & 2 & 2 & 1 & 2
\end{array}
\right)  $ is not a stufufuffler, since it is not weakly increasing ($f\left(
q_{1}\right)  >f\left(  q_{2}\right)  $).

\item The map $\left(
\begin{array}
[c]{ccccc}%
p_{1} & p_{2} & q_{1} & q_{2} & q_{3}\\
2 & 2 & 2 & 2 & 2
\end{array}
\right)  $ is not a stufufuffler, since it fails to be surjective onto
$\left\{  1<2<\cdots<k\right\}  $ whatever $k$ is.
\end{itemize}

Here are the weights of the eight stufufufflers listed above:%
\begin{align*}
\operatorname*{wt}\left(
\begin{array}
[c]{ccccc}%
p_{1} & p_{2} & q_{1} & q_{2} & q_{3}\\
1 & 2 & 3 & 4 & 5
\end{array}
\right)   &  =\left(  a,b,c,d,e\right)  ,\\
\operatorname*{wt}\left(
\begin{array}
[c]{ccccc}%
p_{1} & p_{2} & q_{1} & q_{2} & q_{3}\\
2 & 5 & 1 & 3 & 4
\end{array}
\right)   &  =\left(  c,a,d,e,b\right)  ,\\
\operatorname*{wt}\left(
\begin{array}
[c]{ccccc}%
p_{1} & p_{2} & q_{1} & q_{2} & q_{3}\\
1 & 1 & 1 & 2 & 3
\end{array}
\right)   &  =\left(  a+b+c,\ d,\ e\right)  ,\\
\operatorname*{wt}\left(
\begin{array}
[c]{ccccc}%
p_{1} & p_{2} & q_{1} & q_{2} & q_{3}\\
1 & 2 & 2 & 2 & 3
\end{array}
\right)   &  =\left(  a,\ b+c+d,\ e\right)  ,\\
\operatorname*{wt}\left(
\begin{array}
[c]{ccccc}%
p_{1} & p_{2} & q_{1} & q_{2} & q_{3}\\
2 & 2 & 1 & 2 & 3
\end{array}
\right)   &  =\left(  c,\ a+b+d,\ e\right)  ,\\
\operatorname*{wt}\left(
\begin{array}
[c]{ccccc}%
p_{1} & p_{2} & q_{1} & q_{2} & q_{3}\\
1 & 1 & 1 & 1 & 1
\end{array}
\right)   &  =\left(  a+b+c+d+e\right)  ,\\
\operatorname*{wt}\left(
\begin{array}
[c]{ccccc}%
p_{1} & p_{2} & q_{1} & q_{2} & q_{3}\\
1 & 1 & 1 & 1 & 2
\end{array}
\right)   &  =\left(  a+b+c+d,\ e\right)  ,\\
\operatorname*{wt}\left(
\begin{array}
[c]{ccccc}%
p_{1} & p_{2} & q_{1} & q_{2} & q_{3}\\
1 & 2 & 1 & 1 & 2
\end{array}
\right)   &  =\left(  a+c+d,\ b+e\right)  .
\end{align*}
The losses of these stufufufflers are $0$, $0$, $1$, $1$, $1$, $2$, $1$ and
$1$, respectively. Their poises are $0$, $0$, $0$, $0$, $0$, $0$, $1$ and $1$, respectively.
\end{example}

Intuitively, the composition $\operatorname*{wt}\left(  f\right)  $ in
Definition \ref{def.stufufuffler} can be thought of as a variant of a
stuffle\footnote{\textquotedblleft Stuffles\textquotedblright\ are also known
as \textquotedblleft overlapping shuffles\textquotedblright; see
\cite[Proposition 5.1.3 and Example 5.1.4]{GriRei} for the meaning of this
concept (and \cite{DEMT17} for more).} of $\delta$ with $\varepsilon$, but
instead of adding an entry of $\delta$ with an entry of $\varepsilon$, it
allows adding $i$ consecutive entries of $\delta$ and $j$ consecutive entries
of $\varepsilon$ whenever $i$ and $j$ are integers satisfying $i\approx j$ and
$i+j>0$. (Such a sum can be obtained by starting with $0$ and taking turns at
adding the next available entry from $\delta$ or from $\varepsilon$; thus the
name \textquotedblleft stufufuffle\textquotedblright.) The poise statistic
$\operatorname*{poise}\left(  f\right)  $ tells us how often this $i\approx j$
relation becomes an equality. The loss statistic $\operatorname*{loss}\left(
f\right)  $ tells how much is being added, i.e., how far this
\textquotedblleft stufufuffle\textquotedblright\ deviates from a stuffle.

Now we can restate the multiplication rule for $\eta_{\delta}^{\left(
q\right)  }\eta_{\varepsilon}^{\left(  q\right)  }$ in terms of stufufufflers:

\begin{theorem}
\label{thm.etaaetab}Let $\delta$ and $\varepsilon$ be two compositions. Then,%
\[
\eta_{\delta}^{\left(  q\right)  }\eta_{\varepsilon}^{\left(  q\right)  }%
=\sum_{\substack{f\text{ is a stufufuffler}\\\text{for }\delta\text{ and
}\varepsilon}}\left(  -q\right)  ^{\operatorname*{loss}\left(  f\right)
}\left(  q-1\right)  ^{\operatorname*{poise}\left(  f\right)  }\eta
_{\operatorname*{wt}\left(  f\right)  }^{\left(  q\right)  }.
\]

\end{theorem}

\begin{example}
Let $\delta=\left(  a,b\right)  $ and $\varepsilon=\left(  c,d\right)  $ be
two compositions of length $2$. Let us compute $\eta_{\left(  a,b\right)
}^{\left(  q\right)  }\eta_{\left(  c,d\right)  }^{\left(  q\right)  }$ using
Theorem \ref{thm.etaaetab}. The stufufufflers for $\delta$ and $\varepsilon$
are the maps (written here in two-line notation)%
\begin{align*}
&  \left(
\begin{array}
[c]{cccc}%
p_{1} & p_{2} & q_{1} & q_{2}\\
1 & 2 & 3 & 4
\end{array}
\right)  ,\ \ \ \ \ \ \ \ \ \ \left(
\begin{array}
[c]{cccc}%
p_{1} & p_{2} & q_{1} & q_{2}\\
1 & 3 & 2 & 4
\end{array}
\right)  ,\ \ \ \ \ \ \ \ \ \ \left(
\begin{array}
[c]{cccc}%
p_{1} & p_{2} & q_{1} & q_{2}\\
1 & 4 & 2 & 3
\end{array}
\right)  ,\\
&  \left(
\begin{array}
[c]{cccc}%
p_{1} & p_{2} & q_{1} & q_{2}\\
2 & 3 & 1 & 4
\end{array}
\right)  ,\ \ \ \ \ \ \ \ \ \ \left(
\begin{array}
[c]{cccc}%
p_{1} & p_{2} & q_{1} & q_{2}\\
2 & 4 & 1 & 3
\end{array}
\right)  ,\ \ \ \ \ \ \ \ \ \ \left(
\begin{array}
[c]{cccc}%
p_{1} & p_{2} & q_{1} & q_{2}\\
3 & 4 & 1 & 2
\end{array}
\right)  ,\\
&  \left(
\begin{array}
[c]{cccc}%
p_{1} & p_{2} & q_{1} & q_{2}\\
1 & 2 & 2 & 2
\end{array}
\right)  ,\ \ \ \ \ \ \ \ \ \ \left(
\begin{array}
[c]{cccc}%
p_{1} & p_{2} & q_{1} & q_{2}\\
2 & 2 & 1 & 2
\end{array}
\right)  ,\ \ \ \ \ \ \ \ \ \ \left(
\begin{array}
[c]{cccc}%
p_{1} & p_{2} & q_{1} & q_{2}\\
1 & 1 & 1 & 2
\end{array}
\right)  ,\\
&  \left(
\begin{array}
[c]{cccc}%
p_{1} & p_{2} & q_{1} & q_{2}\\
1 & 2 & 1 & 1
\end{array}
\right)  ,\ \ \ \ \ \ \ \ \ \ \left(
\begin{array}
[c]{cccc}%
p_{1} & p_{2} & q_{1} & q_{2}\\
1 & 2 & 1 & 2
\end{array}
\right)  ,\ \ \ \ \ \ \ \ \ \ \left(
\begin{array}
[c]{cccc}%
p_{1} & p_{2} & q_{1} & q_{2}\\
1 & 1 & 1 & 1
\end{array}
\right)  ,\\
&  \left(
\begin{array}
[c]{cccc}%
p_{1} & p_{2} & q_{1} & q_{2}\\
1 & 2 & 1 & 3
\end{array}
\right)  ,\ \ \ \ \ \ \ \ \ \ \left(
\begin{array}
[c]{cccc}%
p_{1} & p_{2} & q_{1} & q_{2}\\
1 & 3 & 1 & 2
\end{array}
\right)  ,\ \ \ \ \ \ \ \ \ \ \left(
\begin{array}
[c]{cccc}%
p_{1} & p_{2} & q_{1} & q_{2}\\
1 & 3 & 2 & 3
\end{array}
\right)  ,\\
&  \left(
\begin{array}
[c]{cccc}%
p_{1} & p_{2} & q_{1} & q_{2}\\
2 & 3 & 1 & 3
\end{array}
\right)  ,\ \ \ \ \ \ \ \ \ \ \left(
\begin{array}
[c]{cccc}%
p_{1} & p_{2} & q_{1} & q_{2}\\
1 & 2 & 2 & 3
\end{array}
\right)  ,\ \ \ \ \ \ \ \ \ \ \left(
\begin{array}
[c]{cccc}%
p_{1} & p_{2} & q_{1} & q_{2}\\
2 & 3 & 1 & 2
\end{array}
\right)  .
\end{align*}
Their respective weights are
\begin{align*}
&  \left(  a,b,c,d\right)  ,\ \ \ \ \ \ \ \ \ \ \left(  a,c,b,d\right)
,\ \ \ \ \ \ \ \ \ \ \left(  a,c,d,b\right)  ,\\
&  \left(  c,a,b,d\right)  ,\ \ \ \ \ \ \ \ \ \ \left(  c,a,d,b\right)
,\ \ \ \ \ \ \ \ \ \ \left(  c,d,a,b\right)  ,\\
&  \left(  a,\ b+c+d\right)  ,\ \ \ \ \ \ \ \ \ \ \left(  c,\ a+b+d\right)
,\ \ \ \ \ \ \ \ \ \ \left(  a+b+c,\ d\right)  ,\\
&  \left(  a+c+d,\ b\right)  ,\ \ \ \ \ \ \ \ \ \ \left(  a+c,\ b+d\right)
,\ \ \ \ \ \ \ \ \ \ \left(  a+b+c+d\right)  ,\\
&  \left(  a+c,\ b,\ d\right)  ,\ \ \ \ \ \ \ \ \ \ \left(
a+c,\ d,\ b\right)  ,\ \ \ \ \ \ \ \ \ \ \left(  a,\ c,\ b+d\right)  ,\\
&  \left(  c,\ a,\ b+d\right)  ,\ \ \ \ \ \ \ \ \ \ \left(
a,\ b+c,\ d\right)  ,\ \ \ \ \ \ \ \ \ \ \left(  c,\ a+d,\ b\right)  ;
\end{align*}
their respective losses are
\begin{align*}
&  0,0,0,\\
&  0,0,0,\\
&  1,1,1,\\
&  1,0,1,\\
&  0,0,0,\\
&  0,0,0,
\end{align*}
whereas their respective poises are%
\begin{align*}
&  0,0,0,\\
&  0,0,0,\\
&  0,0,0,\\
&  0,2,1,\\
&  1,1,1,\\
&  1,1,1.
\end{align*}
Thus, Theorem \ref{thm.etaaetab} yields%
\begin{align*}
\eta_{\left(  a,b\right)  }^{\left(  q\right)  }\eta_{\left(  c,d\right)
}^{\left(  q\right)  }  &  =\eta_{\left(  a,b,c,d\right)  }^{\left(  q\right)
}+\eta_{\left(  a,c,b,d\right)  }^{\left(  q\right)  }+\eta_{\left(
a,c,d,b\right)  }^{\left(  q\right)  }\\
&  \ \ \ \ \ \ \ \ \ \ +\eta_{\left(  c,a,b,d\right)  }^{\left(  q\right)
}+\eta_{\left(  c,a,d,b\right)  }^{\left(  q\right)  }+\eta_{\left(
c,d,a,b\right)  }^{\left(  q\right)  }\\
&  \ \ \ \ \ \ \ \ \ \ -q\eta_{\left(  a,\ b+c+d\right)  }^{\left(  q\right)
}-q\eta_{\left(  c,\ a+b+d\right)  }^{\left(  q\right)  }-q\eta_{\left(
a+b+c,\ d\right)  }^{\left(  q\right)  }\\
&  \ \ \ \ \ \ \ \ \ \ -q\eta_{\left(  a+c+d,\ b\right)  }^{\left(  q\right)
}+\left(  q-1\right)  ^{2}\eta_{\left(  a+c,\ b+d\right)  }^{\left(  q\right)
}-q\left(  q-1\right)  \eta_{\left(  a+b+c+d\right)  }^{\left(  q\right)  }\\
&  \ \ \ \ \ \ \ \ \ \ +\left(  q-1\right)  \eta_{\left(  a+c,\ b,\ d\right)
}^{\left(  q\right)  }+\left(  q-1\right)  \eta_{\left(  a+c,\ d,\ b\right)
}^{\left(  q\right)  }+\left(  q-1\right)  \eta_{\left(  a,\ c,\ b+d\right)
}^{\left(  q\right)  }\\
&  \ \ \ \ \ \ \ \ \ \ +\left(  q-1\right)  \eta_{\left(  c,\ a,\ b+d\right)
}^{\left(  q\right)  }+\left(  q-1\right)  \eta_{\left(  a,\ b+c,\ d\right)
}^{\left(  q\right)  }+\left(  q-1\right)  \eta_{\left(  c,\ a+d,\ b\right)
}^{\left(  q\right)  }.
\end{align*}

\end{example}

Let us now outline how Theorem \ref{thm.etaaetab} can be derived from Theorem
\ref{thm.product.1}.

\begin{proof}
[Proof of Theorem \ref{thm.etaaetab} (sketched).]Let us define the polynomials
$d_{\delta,\varepsilon}^{\alpha}\left(  X\right)  \in\mathbb{Z}\left[
X\right]  $ as in the proof of Theorem \ref{thm.product.1}. Then, Claim 4 in
said proof shows that%
\begin{equation}
\eta_{\delta}^{\left(  q\right)  }\eta_{\varepsilon}^{\left(  q\right)  }%
=\sum_{\alpha\in\operatorname*{Comp}}d_{\delta,\varepsilon}^{\alpha}\left(
q\right)  \eta_{\alpha}^{\left(  q\right)  }. \label{pf.thm.etaaetab.old}%
\end{equation}

Now, fix a composition $\alpha=\left(  \alpha_{1},\alpha_{2},\ldots,\alpha
_{k}\right)  $. Let $\mathbf{P}$ be the set of all pairs%
\[
\left(  \left(  \beta_{1},\beta_{2},\ldots,\beta_{k}\right)  ,\ \left(
\gamma_{1},\gamma_{2},\ldots,\gamma_{k}\right)  \right)
\]
satisfying the six conditions%
\begin{align*}
\beta_{1},\beta_{2},\ldots,\beta_{k}  &  \in\operatorname*{Comp}%
;\ \ \ \ \ \ \ \ \ \ \gamma_{1},\gamma_{2},\ldots,\gamma_{k}\in
\operatorname*{Comp};\\
\beta_{1}\beta_{2}\cdots\beta_{k}  &  =\delta;\ \ \ \ \ \ \ \ \ \ \gamma
_{1}\gamma_{2}\cdots\gamma_{k}=\varepsilon;\\
\left\vert \ell\left(  \beta_{s}\right)  -\ell\left(  \gamma_{s}\right)
\right\vert  &  \leq1\text{ for each }s;\\
\left\vert \beta_{s}\right\vert +\left\vert \gamma_{s}\right\vert  &
=\alpha_{s}\text{ for each }s.
\end{align*}
Then, the equality (\ref{pf.thm.product.1.dq=}) rewrites as
\begin{align}
d_{\delta,\varepsilon}^{\alpha}\left(  q\right)   &  =\sum_{\left(  \left(
\beta_{1},\beta_{2},\ldots,\beta_{k}\right)  ,\ \left(  \gamma_{1},\gamma
_{2},\ldots,\gamma_{k}\right)  \right)  \in\mathbf{P}}\left(  -q\right)
^{\sum_{s=1}^{k}\max\left\{  \ell\left(  \beta_{s}\right)  ,\ell\left(
\gamma_{s}\right)  \right\}  -k}\nonumber\\
&  \ \ \ \ \ \ \ \ \ \ \ \ \ \ \ \ \ \ \ \ \cdot\left(  q-1\right)  ^{\left(
\text{\# of all }s\in\left[  k\right]  \text{ such that }\ell\left(  \beta
_{s}\right)  =\ell\left(  \gamma_{s}\right)  \right)  }.
\label{pf.thm.etaaetab.dq=}%
\end{align}

On the other hand, let $\mathbf{S}$ be the set of all stufufufflers $f$ for
$\delta$ and $\varepsilon$ satisfying $\operatorname*{wt}\left(  f\right)
=\alpha$.

We shall construct a bijection $\Phi$ from $\mathbf{S}$ to $\mathbf{P}$.
Namely, $\Phi$ shall send any stufufuffler $f\in\mathbf{S}$ to the pair
\[
\left(  \left(  \beta_{1},\beta_{2},\ldots,\beta_{k}\right)  ,\ \left(
\gamma_{1},\gamma_{2},\ldots,\gamma_{k}\right)  \right)  ,
\]
where
\begin{align*}
\beta_{s}  &  =\left(  \text{the composition consisting of the }\delta
_{u}\text{ for all }u\in f_{P}^{-1}\left(  s\right)  \right. \\
&  \ \ \ \ \ \ \ \ \ \ \left.  \text{(in the order of increasing }%
u\text{)}\vphantom{f_P^{-1}}\right)  \ \ \ \ \ \ \ \ \ \ \text{and}\\
\gamma_{s}  &  =\left(  \text{the composition consisting of the }%
\varepsilon_{v}\text{ for all }v\in f_{Q}^{-1}\left(  s\right)  \right. \\
&  \ \ \ \ \ \ \ \ \ \ \left.  \text{(in the order of increasing }%
v\text{)}\vphantom{f_P^{-1}}\right)  \ \ \ \ \ \ \ \ \ \ \text{for all }%
s\in\left[  k\right]  .
\end{align*}
(We are here using the fact that our stufufuffler $f$ must necessarily be a
map from $U$ to $\left\{  1<2<\cdots<k\right\}  $, because its weight
$\operatorname*{wt}\left(  f\right)  =\alpha$ is a composition of length $k$.)
It is easy to see that this pair really belongs to $\mathbf{P}$, and that
$\Phi$ is indeed a bijection\footnote{Its inverse map $\Phi^{-1}$ can easily
be constructed: It sends each pair $\left(  \left(  \beta_{1},\beta_{2}%
,\ldots,\beta_{k}\right)  ,\ \left(  \gamma_{1},\gamma_{2},\ldots,\gamma
_{k}\right)  \right)  \in\mathbf{P}$ to the map $f:U\rightarrow\left[
k\right]  $ that is given by%
\[
f\left(  p_{u}\right)  =\min\left\{  s\in\left[  k\right]  \ \mid\ \ell\left(
\beta_{1}\beta_{2}\cdots\beta_{s}\right)  \geq u\right\}
\ \ \ \ \ \ \ \ \ \ \text{for all }u\in\left[  \ell\right]
\]
and%
\[
f\left(  q_{v}\right)  =\min\left\{  s\in\left[  k\right]  \ \mid\ \ell\left(
\gamma_{1}\gamma_{2}\cdots\gamma_{s}\right)  \geq v\right\}
\ \ \ \ \ \ \ \ \ \ \text{for all }v\in\left[  m\right]  .
\]
The idea behind this is that $f\left(  p_{u}\right)  $ is the number $s$ such
that the $u$-th entry of the concatenated composition $\beta_{1}\beta
_{2}\cdots\beta_{k}$ is taken from its $s$-th factor $\beta_{s}$ (and
similarly $f\left(  q_{v}\right)  $).}.

This bijection $\Phi$ has a further useful property: If $\Phi$ sends a
stufufuffler $f$ to a pair $\left(  \left(  \beta_{1},\beta_{2},\ldots
,\beta_{k}\right)  ,\ \left(  \gamma_{1},\gamma_{2},\ldots,\gamma_{k}\right)
\right)  $, then%
\begin{align*}
\sum_{s=1}^{k}\max\left\{  \ell\left(  \beta_{s}\right)  ,\ell\left(
\gamma_{s}\right)  \right\}  -k  &  =\operatorname*{loss}\left(  f\right)
\ \ \ \ \ \ \ \ \ \ \text{and}\\
\left(  \text{\# of all }s\in\left[  k\right]  \text{ such that }\ell\left(
\beta_{s}\right)  =\ell\left(  \gamma_{s}\right)  \right)   &
=\operatorname*{poise}\left(  f\right)  .
\end{align*}
(This is easily seen from the definitions of $\Phi$, of the loss and of the poise.)

Thus, we can use the bijection $\Phi$ to rewrite (\ref{pf.thm.etaaetab.dq=})
as%
\begin{align}
d_{\delta,\varepsilon}^{\alpha}\left(  q\right)   &  =\sum_{f\in\mathbf{S}%
}\left(  -q\right)  ^{\operatorname*{loss}\left(  f\right)  }\left(
q-1\right)  ^{\operatorname*{poise}\left(  f\right)  }\nonumber\\
&  =\sum_{\substack{f\text{ is a stufufuffler}\\\text{for }\delta\text{ and
}\varepsilon;\\\operatorname*{wt}\left(  f\right)  =\alpha}}\left(  -q\right)
^{\operatorname*{loss}\left(  f\right)  }\left(  q-1\right)
^{\operatorname*{poise}\left(  f\right)  } \label{pf.thm.etaaetab.dq=stuf}%
\end{align}
(by the definition of $\mathbf{S}$).

Forget that we fixed $\alpha$. We thus have proved
(\ref{pf.thm.etaaetab.dq=stuf}) for each composition $\alpha\in
\operatorname*{Comp}$. Hence, we can rewrite (\ref{pf.thm.etaaetab.old}) as%
\begin{align*}
\eta_{\delta}^{\left(  q\right)  }\eta_{\varepsilon}^{\left(  q\right)  }  &
=\sum_{\alpha\in\operatorname*{Comp}}\left(  \sum_{\substack{f\text{ is a
stufufuffler}\\\text{for }\delta\text{ and }\varepsilon;\\\operatorname*{wt}%
\left(  f\right)  =\alpha}}\left(  -q\right)  ^{\operatorname*{loss}\left(
f\right)  }\left(  q-1\right)  ^{\operatorname*{poise}\left(  f\right)
}\right)  \eta_{\alpha}^{\left(  q\right)  }\\
&  =\sum_{\alpha\in\operatorname*{Comp}}\ \ \sum_{\substack{f\text{ is a
stufufuffler}\\\text{for }\delta\text{ and }\varepsilon;\\\operatorname*{wt}%
\left(  f\right)  =\alpha}}\left(  -q\right)  ^{\operatorname*{loss}\left(
f\right)  }\left(  q-1\right)  ^{\operatorname*{poise}\left(  f\right)
}\underbrace{\eta_{\alpha}^{\left(  q\right)  }}_{\substack{=\eta
_{\operatorname*{wt}\left(  f\right)  }^{\left(  q\right)  }\\\text{(since
}\alpha=\operatorname*{wt}\left(  f\right)  \text{)}}}\\
&  =\underbrace{\sum_{\alpha\in\operatorname*{Comp}}\ \ \sum
_{\substack{f\text{ is a stufufuffler}\\\text{for }\delta\text{ and
}\varepsilon;\\\operatorname*{wt}\left(  f\right)  =\alpha}}}_{=\sum
_{\substack{f\text{ is a stufufuffler}\\\text{for }\delta\text{ and
}\varepsilon}}}\left(  -q\right)  ^{\operatorname*{loss}\left(  f\right)
}\left(  q-1\right)  ^{\operatorname*{poise}\left(  f\right)  }\eta
_{\operatorname*{wt}\left(  f\right)  }^{\left(  q\right)  }\\
&  =\sum_{\substack{f\text{ is a stufufuffler}\\\text{for }\delta\text{ and
}\varepsilon}}\left(  -q\right)  ^{\operatorname*{loss}\left(  f\right)
}\left(  q-1\right)  ^{\operatorname*{poise}\left(  f\right)  }\eta
_{\operatorname*{wt}\left(  f\right)  }^{\left(  q\right)  }.
\end{align*}
This proves Theorem \ref{thm.etaaetab}.
\end{proof}

\subsection{The product rule in terms of subsets}

Finally, let us state the product rule for the $\eta_{\alpha}^{\left(
q\right)  }$ (Theorem \ref{thm.product.1}) in yet another form, using
classical shuffles (\cite[Corollary 1]{GriVas22}):

\begin{definition}
If $T$ is any set of integers, then $T-1$ shall denote the set $\left\{
t-1\ \mid\ t\in T\right\}  $.
\end{definition}

\begin{definition}
Let $\alpha=\left(  \alpha_{1},\alpha_{2},\ldots,\alpha_{n}\right)  $ be a
composition with $n$ entries. For any $i\in\left[  n-1\right]  $, we let
$\alpha^{\downarrow i}$ denote the following composition with $n-1$ entries:
\[
\alpha^{\downarrow i}:=\left(  \alpha_{1},\dots,\alpha_{i-1},{\alpha
_{i}+\alpha_{i+1}},\alpha_{i+2},\dots,\alpha_{n}\right)  .
\]
Furthermore, for any subset $I\subseteq\left[  n-1 \right] $, we set
\[
\alpha^{\downarrow I}:=\left(  \left(  \cdots\left(  \alpha^{\downarrow i_{k}%
}\right)  \cdots\right)  ^{\downarrow i_{2}}\right)  ^{\downarrow i_{1}},
\]
where $i_{1},i_{2},\ldots,i_{k}$ are the elements of $I$ in increasing order.

Finally, if $I$ and $J$ are two subsets of $[n-1]$ with $1 \notin J$, then we
set
\[
\alpha^{\downarrow I\downarrow\downarrow J}:=\alpha^{\downarrow K}%
,\ \ \ \ \ \ \ \ \ \ \text{where }K=I\cup J\cup(J-1).
\]

\end{definition}

\begin{example}
Let $\alpha=\left(  a,b,c,d,e,f,g\right)  $ be a composition with $7$ entries.
Then,%
\begin{align*}
\alpha^{\downarrow2}  &  =\left(  a,\ b+c,\ d,\ e,\ f,\ g\right)  ;\\
\alpha^{\downarrow\left\{  2,4,5\right\}  }  &  =\left(
a,\ b+c,\ d+e+f,\ g\right)  ;\\
\alpha^{\downarrow\left\{  2\right\}  \downarrow\downarrow\left\{  6\right\}
}  &  =\alpha^{\downarrow\left\{  2,5,6\right\}  }=\left(
a,\ b+c,\ d,\ e+f+g\right)  .
\end{align*}

\end{example}

\begin{theorem}
\label{thm.etaeta.EEE} Let $\delta=\left(  \delta_{1},\delta_{2},\ldots
,\delta_{n}\right)  $ and $\varepsilon=\left(  \varepsilon_{1},\varepsilon
_{2},\ldots,\varepsilon_{m}\right)  $ be two compositions.

If $T$ is any $m$-element subset of $\left[  n+m\right]  $, then we define the
$T$\emph{-shuffle} of $\delta$ with $\varepsilon$ to be the composition
\[
\delta\left\lfloor T\right\rfloor \varepsilon:=\left(  \gamma_{1},\gamma
_{2},\ldots,\gamma_{n+m}\right)  ,
\]
where
\[
\gamma_{k}:=%
\begin{cases}
\delta_{i}, & \text{if }k\text{ is the }i\text{-th smallest element of
}\left[  n+m\right]  \setminus T;\\
\varepsilon_{j}, & \text{if }k\text{ is the }j\text{-th smallest element of
}T.
\end{cases}
\]

Furthermore, if $T$ is any subset of $\left[  n+m\right]  $, then we define a
further subset%
\[
T^{\prime}:=\left(  T\setminus\left(  T-1\right)  \right)  \setminus\left\{
n+m\right\}  .
\]

Then,%
\[
\eta_{\delta}^{(q)}\eta_{\varepsilon}^{(q)}=\sum_{\substack{\text{triples
}\left(  T,I,J\right)  ;\\T\subseteq\left[  n+m\right]  ;\\\left\vert
T\right\vert =m;\\I\subseteq T^{\prime};\\J\subseteq T^{\prime}\setminus
\left\{  1\right\}  ;\\I\cap J=\varnothing}}\left(  -q\right)  ^{\left\vert
J\right\vert }\left(  q-1\right)  ^{\left\vert I\right\vert }\eta_{\left(
\delta\left\lfloor T\right\rfloor \varepsilon\right)  ^{\downarrow
I\downarrow\downarrow J}}^{\left(  q\right)  }.
\]

\end{theorem}

\begin{example}
Let $\delta=\left(  a\right)  $ and $\varepsilon=\left(  b,c\right)  $ be two
compositions. Then, applying Theorem \ref{thm.etaeta.EEE} (with $n=1$ and
$m=2$), we see that $\eta_{\delta}^{(q)}\eta_{\varepsilon}^{(q)}=\eta_{\left(
a\right)  }^{(q)}\eta_{\left(  b,c\right)  }^{(q)}$ is a sum over all triples
$\left(  T,I,J\right)  $ satisfying%
\[
T\subseteq\left[  3\right]  ,\ \ \ \ \ \ \ \ \ \ \left\vert T\right\vert
=2,\ \ \ \ \ \ \ \ \ \ I\subseteq T^{\prime},\ \ \ \ \ \ \ \ \ \ J\subseteq
T^{\prime}\setminus\left\{  1\right\}  ,\ \ \ \ \ \ \ \ \ \ I\cap
J=\varnothing.
\]
There are exactly six such triples $\left(  T,I,J\right)  $, namely%
\begin{align*}
&  \left(  \left\{  1,2\right\}  ,\ \varnothing,\ \varnothing\right)
,\ \left(  \left\{  1,2\right\}  ,\ \varnothing,\ \left\{  2\right\}  \right)
,\ \left(  \left\{  1,2\right\}  ,\ \left\{  2\right\}  ,\ \varnothing\right)
,\\
&  \left(  \left\{  1,3\right\}  ,\ \varnothing,\ \varnothing\right)
,\ \left(  \left\{  1,3\right\}  ,\ \left\{  1\right\}  ,\ \varnothing\right)
,\ \left(  \left\{  2,3\right\}  ,\ \varnothing,\ \varnothing\right)  .
\end{align*}
Thus, the claim of Theorem \ref{thm.etaeta.EEE} becomes%
\[
\eta_{\left(  a\right)  }^{\left(  q\right)  }\eta_{\left(  b,c\right)
}^{\left(  q\right)  }=\eta_{\left(  b,c,a\right)  }^{\left(  q\right)
}-q\eta_{\left(  a+b+c\right)  }^{\left(  q\right)  }+\left(  q-1\right)
\eta_{\left(  b,\ a+c\right)  }^{\left(  q\right)  }+\eta_{\left(
b,a,c\right)  }^{\left(  q\right)  }+\left(  q-1\right)  \eta_{\left(
a+b,\ c\right)  }^{\left(  q\right)  }+\eta_{\left(  a,b,c\right)  }^{\left(
q\right)  }%
\]
(here, we have listed the addends in the same order in which the corresponding
triples were listed above).
\end{example}

Theorem \ref{thm.etaeta.EEE} can be derived from Theorem \ref{thm.etaaetab} by
constructing a bijection between the stufufufflers of $\delta$ and
$\varepsilon$ and the triples $\left(  T,I,J\right)  $ from Theorem
\ref{thm.etaeta.EEE}. The details of this bijection are somewhat bothersome,
so we shall omit them, not least because Theorem \ref{thm.etaeta.EEE} can also
be proved in a different way (using enriched $P$-partitions). The latter proof
has been outlined in \cite[Corollary 1]{GriVas22} and will be elaborated upon
in forthcoming work.

\section{\label{sec.app}Applications}

We shall now discuss some applications of the basis $\left(  \eta_{\alpha
}^{\left(  q\right)  }\right)  _{\alpha\in\operatorname*{Comp}}$ and its features.

\subsection{\label{subsec.app.subalg}Hopf subalgebras of $\operatorname*{QSym}%
$}

The $q=1$ case in particular is useful for constructing Hopf subalgebras of
$\operatorname*{QSym}$, such as the peak subalgebra $\Pi$ introduced by
Stembridge \cite[\S 3]{Stembr97} and later studied by various authors
(\cite[\S 6, particularly Proposition 6.5]{AgBeSo14}, \cite{BMSW99},
\cite[\S 5]{BMSW00}, \cite{Hsiao07} etc.). We shall now briefly survey some
Hopf subalgebras that can be obtained in this way.

\begin{convention}
For the rest of Subsection \ref{subsec.app.subalg}, we fix a set $T$ of
compositions (i.e., a subset $T$ of $\operatorname*{Comp}$).

We let $\operatorname*{QSym}\nolimits_{T}^{\left(  q\right)  }$ be the
$\mathbf{k}$-submodule of $\operatorname*{QSym}$ spanned by the family
$\left(  \eta_{\alpha}^{\left(  q\right)  }\right)  _{\alpha\in T}$.
\end{convention}

When is this $\mathbf{k}$-submodule $\operatorname*{QSym}\nolimits_{T}%
^{\left(  q\right)  }$ a subcoalgebra of $\operatorname*{QSym}$? The answer is
simple:\footnote{We are being sloppy: For us here, a \textquotedblleft
subcoalgebra\textquotedblright\ of a coalgebra $C$ means a $\mathbf{k}%
$-submodule $D$ of $C$ that satisfies
\[
\Delta\left(  D\right)  \subseteq\left(  \text{image of the canonical map
}D\otimes D\rightarrow C\otimes C\right)  .
\]
This is \textbf{not} the algebraically literate definition of a
\textquotedblleft subcoalgebra\textquotedblright, as it does not imply that
$D$ itself becomes a $\mathbf{k}$-coalgebra (after all, the canonical map
$D\otimes D\rightarrow C\otimes C$ might fail to be injective, and then it is
not clear how to \textquotedblleft restrict\textquotedblright\ $\Delta$ to a
map $D\rightarrow D\otimes D$). Fortunately, the two definitions are
equivalent when $\mathbf{k}$ is a field (or when $D$ is a direct addend of $C$
as a $\mathbf{k}$-module).}

\begin{proposition}
\label{prop.app.subalg.subcoalg}For any subset $Y$ of $\left\{  1,2,3,\ldots
\right\}  $, we let%
\begin{align*}
Y^{\ast}:=\  &  \left\{  \text{all compositions whose entries all belong to
}Y\right\} \\
=\  &  \left\{  \left(  \alpha_{1},\alpha_{2},\ldots,\alpha_{k}\right)
\in\operatorname*{Comp}\ \mid\ \alpha_{i}\in Y\text{ for each }i\right\}  .
\end{align*}

\begin{enumerate}
\item[\textbf{(a)}] If $T=Y^{\ast}$ for some subset $Y$ of $\left\{
1,2,3,\ldots\right\}  $, then $\operatorname*{QSym}\nolimits_{T}^{\left(
q\right)  }$ is a subcoalgebra of $\operatorname*{QSym}$.

\item[\textbf{(b)}] More generally: If $T$ has the property that every two
compositions $\beta$ and $\gamma$ satisfying $\beta\gamma\in T$ satisfy
$\beta\in T$ and $\gamma\in T$, then $\operatorname*{QSym}\nolimits_{T}%
^{\left(  q\right)  }$ is a subcoalgebra of $\operatorname*{QSym}$.

\item[\textbf{(c)}] If $\mathbf{k}$ is a field and $r\neq0$, then the converse
of part \textbf{(b)} holds as well: If $\operatorname*{QSym}\nolimits_{T}%
^{\left(  q\right)  }$ is a subcoalgebra of $\operatorname*{QSym}$, then every
two compositions $\beta$ and $\gamma$ satisfying $\beta\gamma\in T$ satisfy
$\beta\in T$ and $\gamma\in T$.
\end{enumerate}
\end{proposition}

\begin{proof}
[Proof sketch.]\textbf{(b)} This follows from Theorem \ref{thm.Delta-eta}.

\textbf{(a)} Particular case of \textbf{(b)}.

\textbf{(c)} Use the graded dual $\operatorname*{NSym}$ of
$\operatorname*{QSym}$ and Proposition \ref{prop.etastar.concat2}. (The
orthogonal complement of a subcoalgebra is an ideal.)
\end{proof}

In some cases, $\operatorname*{QSym}\nolimits_{T}^{\left(  q\right)  }$ is
even a Hopf subalgebra of $\operatorname*{QSym}$. Here is a sufficient
criterion (with no claims of necessity):

\begin{proposition}
\label{prop.app.subalg.subhopf.Y+Y}Let $Y$ be a subset of $\left\{
1,2,3,\ldots\right\}  $ that is closed under addition (i.e., satisfies $y+z\in
Y$ for every $y,z\in Y$). Let $T:=Y^{\ast}$. Then, $\operatorname*{QSym}%
\nolimits_{T}^{\left(  q\right)  }$ is a Hopf subalgebra of
$\operatorname*{QSym}$.
\end{proposition}

\begin{proof}
[Proof sketch.]Clearly, $1=\eta_{\varnothing}^{\left(  q\right)  }%
\in\operatorname*{QSym}\nolimits_{T}^{\left(  q\right)  }$, and Proposition
\ref{prop.app.subalg.subcoalg} \textbf{(a)} shows that $\operatorname*{QSym}%
\nolimits_{T}^{\left(  q\right)  }$ is a subcoalgebra of $\operatorname*{QSym}%
$. Next, we will show that $\operatorname*{QSym}\nolimits_{T}^{\left(
q\right)  }$ is closed under multiplication. In view of Theorem
\ref{thm.product.1}, this will follow once we can show the following claim:

\begin{statement}
\textit{Claim 1:} Let $k\in\mathbb{N}$. Let $\delta\in Y^{\ast}$ and
$\varepsilon\in Y^{\ast}$ be two compositions all of whose entries are $\in
Y$. Let $\beta_{1},\beta_{2},\ldots,\beta_{k}\in\operatorname*{Comp}$ and
$\gamma_{1},\gamma_{2},\ldots,\gamma_{k}\in\operatorname*{Comp}$ be $2k$
compositions satisfying%
\begin{align*}
\beta_{1}\beta_{2}\cdots\beta_{k}  &  =\delta\ \ \ \ \ \ \ \ \ \ \text{and}%
\ \ \ \ \ \ \ \ \ \ \gamma_{1}\gamma_{2}\cdots\gamma_{k}=\varepsilon\\
\text{and }\ \ \ \ \ \ \ \ \ \ \ell\left(  \beta_{s}\right)  +\ell\left(
\gamma_{s}\right)   &  >0\text{ for all }s.
\end{align*}
Then,
\[
\left(  \left\vert \beta_{1}\right\vert +\left\vert \gamma_{1}\right\vert
,\ \left\vert \beta_{2}\right\vert +\left\vert \gamma_{2}\right\vert
,\ \ldots,\ \left\vert \beta_{k}\right\vert +\left\vert \gamma_{k}\right\vert
\right)  \in Y^{\ast}.
\]

\end{statement}

\begin{proof}
[Proof of Claim 1.]We need to show that $\left\vert \beta_{s}\right\vert
+\left\vert \gamma_{s}\right\vert \in Y$ for each $s\in\left[  k\right]  $. To
do so, we fix $s\in\left[  k\right]  $. Then, $\ell\left(  \beta_{s}\right)
+\ell\left(  \gamma_{s}\right)  >0$ (by assumption). In other words, at least
one of the compositions $\beta_{s}$ and $\gamma_{s}$ is nonempty.

However, all entries of the composition $\beta_{s}$ are entries of the
composition $\beta_{1}\beta_{2}\cdots\beta_{k}=\delta$, and thus belong to $Y$
(since $\delta\in Y^{\ast}$). Thus, the sum of all entries of $\beta_{s}$
either equals $0$ or belongs to $Y$ (since $Y$ is closed under addition). In
other words, the size $\left\vert \beta_{s}\right\vert $ either equals $0$ or
belongs to $Y$. Similarly, $\left\vert \gamma_{s}\right\vert $ either equals
$0$ or belongs to $Y$. Hence, the sum $\left\vert \beta_{s}\right\vert
+\left\vert \gamma_{s}\right\vert $ either equals $0$ or belongs to $Y$ as
well (since $Y$ is closed under addition). Since $\left\vert \beta
_{s}\right\vert +\left\vert \gamma_{s}\right\vert $ cannot equal $0$ (because
at least one of the compositions $\beta_{s}$ and $\gamma_{s}$ is nonempty), we
thus conclude that $\left\vert \beta_{s}\right\vert +\left\vert \gamma
_{s}\right\vert $ belongs to $Y$. In other words, $\left\vert \beta
_{s}\right\vert +\left\vert \gamma_{s}\right\vert \in Y$. As we said, this
completes the proof of Claim 1.
\end{proof}

Now, Claim 1 (together with $1\in\operatorname*{QSym}\nolimits_{T}^{\left(
q\right)  }$) shows that $\operatorname*{QSym}\nolimits_{T}^{\left(  q\right)
}$ is a $\mathbf{k}$-subalgebra of $\operatorname*{QSym}$. As we saw above,
$\operatorname*{QSym}\nolimits_{T}^{\left(  q\right)  }$ is a $\mathbf{k}%
$-subcoalgebra of $\operatorname*{QSym}$ as well, and thus is a $\mathbf{k}%
$-subbialgebra of $\operatorname*{QSym}$. This bialgebra $\operatorname*{QSym}%
\nolimits_{T}^{\left(  q\right)  }$ is connected graded, and therefore a Hopf
algebra (by Takeuchi's famous result \cite[Proposition 1.4.16]{GriRei}). The
inclusion map $\operatorname*{QSym}\nolimits_{T}^{\left(  q\right)
}\rightarrow\operatorname*{QSym}$ is a bialgebra morphism between two Hopf
algebras, and thus automatically a Hopf algebra morphism (by another
well-known result: \cite[Corollary 1.4.27]{GriRei}). Hence,
$\operatorname*{QSym}\nolimits_{T}^{\left(  q\right)  }$ is a Hopf subalgebra
of $\operatorname*{QSym}$. This proves Proposition
\ref{prop.app.subalg.subhopf.Y+Y}.
\end{proof}

\begin{example}
The subset $\left\{  2,4,6,8,\ldots\right\}  $ of $\left\{  1,2,3,\ldots
\right\}  $ is closed under addition. Thus, Proposition
\ref{prop.app.subalg.subhopf.Y+Y} shows that $\operatorname*{QSym}%
\nolimits_{T}^{\left(  q\right)  }$ is a Hopf subalgebra of
$\operatorname*{QSym}$ for $Y:=\left\{  2,4,6,8,\ldots\right\}  $ and
$T:=Y^{\ast}$. This Hopf subalgebra can be viewed as a copy of
$\operatorname*{QSym}$ in the indeterminates $x_{1}^{2},x_{2}^{2},x_{3}%
^{2},\ldots$, and thus is isomorphic to $\operatorname*{QSym}$.
\end{example}

\begin{example}
The subset $\left\{  2,3,4,5,\ldots\right\}  $ of $\left\{  1,2,3,\ldots
\right\}  $ is closed under addition. Thus, Proposition
\ref{prop.app.subalg.subhopf.Y+Y} shows that $\operatorname*{QSym}%
\nolimits_{T}^{\left(  q\right)  }$ is a Hopf subalgebra of
$\operatorname*{QSym}$ for $Y:=\left\{  2,3,4,5,\ldots\right\}  $ and
$T:=Y^{\ast}$.
\end{example}

Proposition \ref{prop.app.subalg.subhopf.Y+Y} is not very surprising. In fact,
(\ref{eq.def.etaalpha.def}) shows that (under the assumptions of Proposition
\ref{prop.app.subalg.subhopf.Y+Y}) the space $\operatorname*{QSym}%
\nolimits_{T}^{\left(  q\right)  }$ is just the $\mathbf{k}$-linear span of
the functions $r^{\ell\left(  \alpha\right)  }M_{\alpha}$ with $\alpha\in
Y^{\ast}$; but the latter span is easily seen to be a Hopf subalgebra (using
\cite[Proposition 5.1.3]{GriRei} and (\ref{eq.Delta-M})).

If $q\neq1$ and if $r$ is invertible, then Proposition
\ref{prop.app.subalg.subhopf.Y+Y} has a converse (i.e., $\operatorname*{QSym}%
\nolimits_{T}^{\left(  q\right)  }$ is a Hopf subalgebra of
$\operatorname*{QSym}$ only when $Y$ is closed under addition), since it is
easy to see that
\[
\eta_{\left(  a\right)  }^{\left(  q\right)  }\eta_{\left(  b\right)
}^{\left(  q\right)  }=\left(  q-1\right)  \eta_{\left(  a+b\right)
}^{\left(  q\right)  }+\eta_{\left(  a,b\right)  }^{\left(  q\right)  }%
+\eta_{\left(  b,a\right)  }^{\left(  q\right)  }\ \ \ \ \ \ \ \ \ \ \text{for
any }a,b\geq1.
\]
However, more interesting behavior emerges when $q=1$:

\begin{proposition}
\label{prop.app.subalg.subhopf.Y+Y+Y}Let $Y$ be a subset of $\left\{
1,2,3,\ldots\right\}  $ that is closed under ternary addition (i.e., satisfies
$y+z+w\in Y$ for every $y,z,w\in Y$). Let $T:=Y^{\ast}$. Then,
$\operatorname*{QSym}\nolimits_{T}^{\left(  1\right)  }$ is a Hopf subalgebra
of $\operatorname*{QSym}$.
\end{proposition}

\begin{proof}
[Proof sketch.]This is similar to Proposition
\ref{prop.app.subalg.subhopf.Y+Y}, but now we set $q=1$ and observe that all
addends on the right hand side of Theorem \ref{thm.product.1} that satisfy%
\[
\ell\left(  \beta_{s}\right)  =\ell\left(  \gamma_{s}\right)
\ \ \ \ \ \ \ \ \ \ \text{for at least one }s\in\left[  k\right]
\]
are $0$ (because they include the factor $\left(  1-1\right)  ^{\text{a
positive integer}}$, which vanishes), and all the remaining addends have the
property that $\left\vert \ell\left(  \beta_{s}\right)  -\ell\left(
\gamma_{s}\right)  \right\vert =1$ for all $s$ (since $\left\vert \ell\left(
\beta_{s}\right)  -\ell\left(  \gamma_{s}\right)  \right\vert \leq1$ and
$\ell\left(  \beta_{s}\right)  \neq\ell\left(  \gamma_{s}\right)  $). Hence,
the following claim now replaces Claim 1:

\begin{statement}
\textit{Claim 1':} Let $k\in\mathbb{N}$. Let $\delta\in Y^{\ast}$ and
$\varepsilon\in Y^{\ast}$ be two compositions all of whose entries are $\in
Y$. Let $\beta_{1},\beta_{2},\ldots,\beta_{k}\in\operatorname*{Comp}$ and
$\gamma_{1},\gamma_{2},\ldots,\gamma_{k}\in\operatorname*{Comp}$ be $2k$
compositions satisfying%
\begin{align*}
\beta_{1}\beta_{2}\cdots\beta_{k}  &  =\delta\ \ \ \ \ \ \ \ \ \ \text{and}%
\ \ \ \ \ \ \ \ \ \ \gamma_{1}\gamma_{2}\cdots\gamma_{k}=\varepsilon\\
\text{and }\ \ \ \ \ \ \ \ \ \ \left\vert \ell\left(  \beta_{s}\right)
-\ell\left(  \gamma_{s}\right)  \right\vert  &  =1\text{ for all }s.
\end{align*}
Then,
\[
\left(  \left\vert \beta_{1}\right\vert +\left\vert \gamma_{1}\right\vert
,\ \left\vert \beta_{2}\right\vert +\left\vert \gamma_{2}\right\vert
,\ \ldots,\ \left\vert \beta_{k}\right\vert +\left\vert \gamma_{k}\right\vert
\right)  \in Y^{\ast}.
\]

\end{statement}

\begin{proof}
[Proof of Claim 1'.]We need to show that $\left\vert \beta_{s}\right\vert
+\left\vert \gamma_{s}\right\vert \in Y$ for each $s\in\left[  k\right]  $. To
do so, we fix $s\in\left[  k\right]  $. Then, $\left\vert \ell\left(
\beta_{s}\right)  -\ell\left(  \gamma_{s}\right)  \right\vert =1$ (by
assumption), and thus $\ell\left(  \beta_{s}\right)  +\ell\left(  \gamma
_{s}\right)  $ is odd. Hence, $\left\vert \beta_{s}\right\vert +\left\vert
\gamma_{s}\right\vert $ is a sum of an odd number of entries of $\delta$ and
$\varepsilon$, and therefore a sum of an odd number of elements of $Y$ (since
$\delta$ and $\varepsilon$ belong to $Y^{\ast}$). But $Y$ is closed under
ternary addition, and therefore any sum of an odd number of elements of $Y$
must belong to $Y$ (easy induction exercise). Hence, $\left\vert \beta
_{s}\right\vert +\left\vert \gamma_{s}\right\vert \in Y$, and thus Claim 1' is proved.
\end{proof}

The rest of the proof proceeds as for Proposition
\ref{prop.app.subalg.subhopf.Y+Y}.
\end{proof}

\begin{example}
The subset $\left\{  1,3,5,7,\ldots\right\}  $ of $\left\{  1,2,3,\ldots
\right\}  $ is closed under ternary addition. Thus, Proposition
\ref{prop.app.subalg.subhopf.Y+Y+Y} shows that $\operatorname*{QSym}%
\nolimits_{T}^{\left(  1\right)  }$ is a Hopf subalgebra of
$\operatorname*{QSym}$ for $Y:=\left\{  1,3,5,7,\ldots\right\}  $ and
$T:=Y^{\ast}$. This Hopf subalgebra is precisely the peak algebra $\Pi$ of
\cite[\S 3]{Stembr97}, \cite[\S 6, particularly Proposition 6.5]{AgBeSo14},
\cite{BMSW99}, \cite[\S 5]{BMSW00} and \cite{Hsiao07} (since \cite[(2.1) and
(2.2)]{Hsiao07} shows that the $\theta_{\alpha}$ for $\alpha$ odd have the
same span as the $\eta_{\alpha}$ for $\alpha$ odd, but \cite[Proposition
2.1]{Hsiao07} shows that the latter $\eta_{\alpha}$ are precisely our
$\eta_{\alpha}^{\left(  1\right)  }$ up to sign).
\end{example}

\begin{example}
The subset $\left\{  \text{positive integers }\neq2\right\}  =\left\{
1,3,4,5,\ldots\right\}  $ of $\left\{  1,2,3,\ldots\right\}  $ is closed under
ternary addition. Thus, Proposition \ref{prop.app.subalg.subhopf.Y+Y+Y} shows
that $\operatorname*{QSym}\nolimits_{T}^{\left(  1\right)  }$ is a Hopf
subalgebra of $\operatorname*{QSym}$ for $Y:=\left\{  \text{positive integers
}\neq2\right\}  $ and $T:=Y^{\ast}$. This Hopf subalgebra is the Hopf
subalgebra $\Xi$ constructed in \cite[Theorem 5.7]{BMSW00}. (Indeed, both Hopf
subalgebras have the same orthogonal complement: the principal ideal of
$\operatorname*{NSym}$ generated by $\eta_{2}^{\ast}=\dfrac{1}{4}X_{2}%
=\dfrac{1}{4}\left(  2H_{2}-H_{1}H_{1}\right)  $.)
\end{example}

\begin{example}
More generally, if we pick a positive integer $k$ and set
\[
Y:=\left\{  \text{odd positive integers}\right\}  \cup\left\{
k,k+1,k+2,\ldots\right\}
\]
and $T:=Y^{\ast}$, then Proposition \ref{prop.app.subalg.subhopf.Y+Y+Y} shows
that $\operatorname*{QSym}\nolimits_{T}^{\left(  1\right)  }$ is a Hopf
subalgebra of $\operatorname*{QSym}$ (since $Y$ is closed under ternary addition).
\end{example}

The reader can find more examples without trouble. When $\mathbf{k}$ is
nontrivial and $2$ is invertible in $\mathbf{k}$, Proposition
\ref{prop.app.subalg.subhopf.Y+Y+Y} is easily seen to have a converse (using
Example \ref{exa.product.1.ab*c}).

\subsection{A new shuffle algebra}

Next, we shall use the enriched $q$-monomial quasisymmetric functions to
realize a certain deformed version of the shuffle product, which has appeared
in recent work of \cite{BoNoTh22} by Bouillot, Novelli and Thibon
(generalizing the \textquotedblleft block shuffle product\textquotedblright%
\ of Hirose and Sato \cite[$\Diamond$]{HirSat22}).

Shuffle products are a broad and deep subject with a long history and many
applications (e.g., to multiple zeta values, algebraic topology and stochastic
differential equations). An overview of known variants (such as the stuffles,
the \textquotedblleft muffles\textquotedblright, the infiltrations and many
more) can be found in \cite[Table 1]{DEMT17}. In the following, we shall
discuss a variant that does not directly fit into the framework of
\cite{DEMT17}, but is sufficiently similar to enjoy some of the same behavior.
To our knowledge, this variant first appeared in \cite{BoNoTh22}. We will use
the letters $a$ and $b$ for what was called $\alpha$ and $\beta$ in
\cite{BoNoTh22}, as we prefer to use Greek letters for compositions.

Let $\mathcal{F}$ be the free $\mathbf{k}$-algebra with generators
$x_{1},x_{2},x_{3},\ldots$. It has a basis consisting of all words over the
alphabet $\left\{  x_{1},x_{2},x_{3},\ldots\right\}  $; these words are in
bijection with the compositions. In fact, let us set
\begin{equation}
x_{\gamma}:=x_{\gamma_{1}}x_{\gamma_{2}}\cdots x_{\gamma_{k}}
\label{eq.abshuf.xgamma=}%
\end{equation}
for every composition $\gamma=\left(  \gamma_{1},\gamma_{2},\ldots,\gamma
_{k}\right)  $. Then, the bijection sends each composition $\gamma$ to the
word $x_{\gamma}$.

For any $k\in\mathbb{N}$, we let $\zeta_{k}:\mathcal{F}\rightarrow\mathcal{F}$
be the $\mathbf{k}$-linear operator defined by%
\begin{align*}
\zeta_{k}\left(  1\right)   &  =0;\\
\zeta_{k}\left(  x_{i}w\right)   &  =x_{i+k}w\ \ \ \ \ \ \ \ \ \ \text{for
each }i\geq1\text{ and any word }w.
\end{align*}
(Thus, explicitly, the map $\zeta_{k}$ sends $1$ to $0$, and transforms any
nonempty word by adding $k$ to the subscript of its first letter. For example,
$\zeta_{k}\left(  x_{u}x_{v}x_{w}\right)  =x_{u+k}x_{v}x_{w}$ for any
$u,v,w\geq1$.)

Fix two elements $a$ and $b$ of the base ring $\mathbf{k}$.

Let $\#:\mathcal{F}\times\mathcal{F}\rightarrow\mathcal{F}$ be the
$\mathbf{k}$-bilinear map on $\mathcal{F}$ defined recursively by the
requirements%
\begin{align*}
1\#w  &  =w\ \ \ \ \ \ \ \ \ \ \text{for any word }w;\\
w\#1  &  =w\ \ \ \ \ \ \ \ \ \ \text{for any word }w;\\
\left(  x_{i}u\right)  \#\left(  x_{j}v\right)   &  =x_{i}\left(  u\#\left(
x_{j}v\right)  \right)  +x_{j}\left(  \left(  x_{i}u\right)  \#v\right)
+ax_{i+j}\left(  u\#v\right)  +b\zeta_{i+j}\left(  u\#v\right) \\
&  \ \ \ \ \ \ \ \ \ \ \ \ \ \ \ \ \ \ \ \ \text{for any }i,j\geq1\text{ and
any words }u\text{ and }v\text{.}%
\end{align*}
We call this bilinear map $\#$ the \emph{stufufuffle}\footnote{This is a riff
on the notion of \textquotedblleft stuffle\textquotedblright\ (which is
recovered when $a=1$ and $b=0$) and the fact that multiple letters of both
words $u$ and $v$ can get combined into one in $u\#v$.}. Explicitly, we can
compute this operation as follows:

\begin{proposition}
\label{prop.stufufuf.explicit}Let $\delta=\left(  \delta_{1},\delta_{2}%
,\ldots,\delta_{\ell}\right)  $ and $\varepsilon=\left(  \varepsilon
_{1},\varepsilon_{2},\ldots,\varepsilon_{m}\right)  $ be two compositions.
Then, using the notation of (\ref{eq.abshuf.xgamma=}), we have%
\[
x_{\delta}\#x_{\varepsilon}=\sum_{\substack{f\text{ is a stufufuffler}%
\\\text{for }\delta\text{ and }\varepsilon}}b^{\operatorname*{loss}\left(
f\right)  }a^{\operatorname*{poise}\left(  f\right)  }x_{\operatorname*{wt}%
\left(  f\right)  }.
\]

\end{proposition}

\begin{proof}
[Proof sketch.]Use strong induction on $\ell+m$.

\textit{Induction step:} If $\delta=\varnothing$ or $\varepsilon=\varnothing$,
then the claim is easy to check. Thus, assume WLOG that neither $\delta$ nor
$\varepsilon$ is $\varnothing$. Let $i=\delta_{1}$ and $j=\varepsilon_{1}$ and
$\overline{\delta}=\left(  \delta_{2},\delta_{3},\ldots,\delta_{\ell}\right)
$ and $\overline{\varepsilon}=\left(  \varepsilon_{2},\varepsilon_{3}%
,\ldots,\varepsilon_{m}\right)  $. Hence, $x_{\delta}=x_{i}x_{\overline
{\delta}}$ and $x_{\varepsilon}=x_{j}x_{\overline{\varepsilon}}$, so that%
\begin{align}
x_{\delta}\#x_{\varepsilon}  &  =\left(  x_{i}x_{\overline{\delta}}\right)
\#\left(  x_{j}x_{\overline{\varepsilon}}\right) \nonumber\\
&  =x_{i}\left(  x_{\overline{\delta}}\#\underbrace{\left(  x_{j}%
x_{\overline{\varepsilon}}\right)  }_{=x_{\varepsilon}}\right)  +x_{j}\left(
\underbrace{\left(  x_{i}x_{\overline{\delta}}\right)  }_{=x_{\delta}%
}\#x_{\overline{\varepsilon}}\right)  +ax_{i+j}\left(  x_{\overline{\delta}%
}\#x_{\overline{\varepsilon}}\right)  +b\zeta_{i+j}\left(  x_{\overline
{\delta}}\#x_{\overline{\varepsilon}}\right) \nonumber\\
&  \ \ \ \ \ \ \ \ \ \ \ \ \ \ \ \ \ \ \ \ \left(  \text{by the recursive
definition of }\#\right) \nonumber\\
&  =x_{i}\left(  x_{\overline{\delta}}\#x_{\varepsilon}\right)  +x_{j}\left(
x_{\delta}\#x_{\overline{\varepsilon}}\right)  +ax_{i+j}\left(  x_{\overline
{\delta}}\#x_{\overline{\varepsilon}}\right)  +b\zeta_{i+j}\left(
x_{\overline{\delta}}\#x_{\overline{\varepsilon}}\right)  .
\label{pf.prop.stufufuf.explicit.1}%
\end{align}
On the other hand, the stufufufflers $f$ for $\delta$ and $\varepsilon$ can be
classified into four types:

\begin{enumerate}
\item \textit{Type 1} consists of those stufufufflers $f$ that satisfy
$\left\vert f_{P}^{-1}\left(  1\right)  \right\vert =1$ and $\left\vert
f_{Q}^{-1}\left(  1\right)  \right\vert =0$ (so that the composition
$\operatorname*{wt}\left(  f\right)  $ begins with the entry $\delta_{1}=i$).

\item \textit{Type 2} consists of those stufufufflers $f$ that satisfy
$\left\vert f_{P}^{-1}\left(  1\right)  \right\vert =0$ and $\left\vert
f_{Q}^{-1}\left(  1\right)  \right\vert =1$ (so that the composition
$\operatorname*{wt}\left(  f\right)  $ begins with the entry $\varepsilon
_{1}=j$).

\item \textit{Type 3} consists of those stufufufflers $f$ that satisfy
$\left\vert f_{P}^{-1}\left(  1\right)  \right\vert =1$ and $\left\vert
f_{Q}^{-1}\left(  1\right)  \right\vert =1$ (so that the composition
$\operatorname*{wt}\left(  f\right)  $ begins with the entry $\delta
_{1}+\varepsilon_{1}=i+j$).

\item \textit{Type 4} consists of those stufufufflers $f$ that satisfy
$\left\vert f_{P}^{-1}\left(  1\right)  \right\vert +\left\vert f_{Q}%
^{-1}\left(  1\right)  \right\vert >2$ (so that both numbers $\left\vert
f_{P}^{-1}\left(  1\right)  \right\vert $ and $\left\vert f_{Q}^{-1}\left(
1\right)  \right\vert $ are positive\footnote{by
(\ref{eq.def.limstuffler.liminality}), applied to $s=1$}, and one of them is
at least $2$, and therefore the composition $\operatorname*{wt}\left(
f\right)  $ begins with the entry $\delta_{1}+\varepsilon_{1}+\left(
\text{some further numbers}\right)  $).
\end{enumerate}

A type-1 stufufuffler $f$ for $\delta$ and $\varepsilon$ becomes a
stufufuffler for $\overline{\delta}$ and $\varepsilon$ if we decrease all its
values by $1$ and remove $p_{1}$ from $P$. This is furthermore a bijection
from $\left\{  \text{type-1 stufufufflers for }\delta\text{ and }%
\varepsilon\right\}  $ to $\left\{  \text{stufufufflers for }\overline{\delta
}\text{ and }\varepsilon\right\}  $, and this bijection preserves both loss
and poise while removing the first entry from the weight. Hence, we obtain%
\begin{align*}
&  \sum_{\substack{f\text{ is a type-1 stufufuffler}\\\text{for }\delta\text{
and }\varepsilon}}b^{\operatorname*{loss}\left(  f\right)  }%
a^{\operatorname*{poise}\left(  f\right)  }x_{\operatorname*{wt}\left(
f\right)  }\\
&  =\sum_{\substack{f\text{ is a stufufuffler}\\\text{for }\overline{\delta
}\text{ and }\varepsilon}}b^{\operatorname*{loss}\left(  f\right)
}a^{\operatorname*{poise}\left(  f\right)  }x_{i}x_{\operatorname*{wt}\left(
f\right)  }\\
&  =x_{i}\cdot\underbrace{\sum_{\substack{f\text{ is a stufufuffler}%
\\\text{for }\overline{\delta}\text{ and }\varepsilon}}b^{\operatorname*{loss}%
\left(  f\right)  }a^{\operatorname*{poise}\left(  f\right)  }%
x_{\operatorname*{wt}\left(  f\right)  }}_{\substack{=x_{\overline{\delta}%
}\#x_{\varepsilon}\\\text{(by the induction hypothesis,}\\\text{since
}\overline{\delta}\text{ has length }\ell-1<\ell\text{)}}}\\
&  =x_{i}\left(  x_{\overline{\delta}}\#x_{\varepsilon}\right)  .
\end{align*}
Similar reasoning leads to%
\begin{align*}
\sum_{\substack{f\text{ is a type-2 stufufuffler}\\\text{for }\delta\text{ and
}\varepsilon}}b^{\operatorname*{loss}\left(  f\right)  }%
a^{\operatorname*{poise}\left(  f\right)  }x_{\operatorname*{wt}\left(
f\right)  }  &  =x_{j}\left(  x_{\delta}\#x_{\overline{\varepsilon}}\right)
;\\
\sum_{\substack{f\text{ is a type-3 stufufuffler}\\\text{for }\delta\text{ and
}\varepsilon}}b^{\operatorname*{loss}\left(  f\right)  }%
a^{\operatorname*{poise}\left(  f\right)  }x_{\operatorname*{wt}\left(
f\right)  }  &  =ax_{i+j}\left(  x_{\overline{\delta}}\#x_{\overline
{\varepsilon}}\right)  ;\\
\sum_{\substack{f\text{ is a type-4 stufufuffler}\\\text{for }\delta\text{ and
}\varepsilon}}b^{\operatorname*{loss}\left(  f\right)  }%
a^{\operatorname*{poise}\left(  f\right)  }x_{\operatorname*{wt}\left(
f\right)  }  &  =b\zeta_{i+j}\left(  x_{\overline{\delta}}\#x_{\overline
{\varepsilon}}\right)  .
\end{align*}
Adding these four equalities together (and recalling that each stufufuffler
for $\delta$ and $\varepsilon$ belongs to exactly one of the four types 1, 2,
3 and 4), we obtain%
\begin{align*}
&  \sum_{\substack{f\text{ is a stufufuffler}\\\text{for }\delta\text{ and
}\varepsilon}}b^{\operatorname*{loss}\left(  f\right)  }%
a^{\operatorname*{poise}\left(  f\right)  }x_{\operatorname*{wt}\left(
f\right)  }\\
&  =x_{i}\left(  x_{\overline{\delta}}\#x_{\varepsilon}\right)  +x_{j}\left(
x_{\delta}\#x_{\overline{\varepsilon}}\right)  +ax_{i+j}\left(  x_{\overline
{\delta}}\#x_{\overline{\varepsilon}}\right)  +b\zeta_{i+j}\left(
x_{\overline{\delta}}\#x_{\overline{\varepsilon}}\right)  .
\end{align*}
Comparing this with (\ref{pf.prop.stufufuf.explicit.1}), we obtain%
\[
x_{\delta}\#x_{\varepsilon}=\sum_{\substack{f\text{ is a stufufuffler}%
\\\text{for }\delta\text{ and }\varepsilon}}b^{\operatorname*{loss}\left(
f\right)  }a^{\operatorname*{poise}\left(  f\right)  }x_{\operatorname*{wt}%
\left(  f\right)  }.
\]
This completes the induction step, and thus Proposition
\ref{prop.stufufuf.explicit} is proved.
\end{proof}

\begin{theorem}
\label{thm.stufufuf.ass}The bilinear map $\#$ is commutative and associative,
and the element $1\in\mathcal{F}$ is a neutral element for it. Thus, the
$\mathbf{k}$-module $\mathcal{F}$, equipped with the operation $\#$ (as
multiplication), becomes a commutative $\mathbf{k}$-algebra with unity $1$.
\end{theorem}

It appears possible to prove Theorem \ref{thm.stufufuf.ass} by induction, but
the most convenient method at this point is to deduce this from the properties
of the enriched $q$-monomial basis of $\operatorname*{QSym}$. To wit, the
following proposition connects the map $\#$ to the latter basis:

\begin{proposition}
\label{prop.stufufuf.eta}Let $q$ and $u$ be two elements of $\mathbf{k}$ such
that $a=\left(  q-1\right)  u$ and $b=-qu^{2}$. (Such $q$ and $u$ do not
always exist, of course.)

Let $\operatorname*{eta}:\mathcal{F}\rightarrow\operatorname*{QSym}$ be the
$\mathbf{k}$-linear map that sends the word $x_{\alpha}=x_{\alpha_{1}%
}x_{\alpha_{2}}\cdots x_{\alpha_{k}}\in\mathcal{F}$ to $u^{\ell\left(
\alpha\right)  }\eta_{\alpha}^{\left(  q\right)  }\in\operatorname*{QSym}$ for
each composition $\alpha=\left(  \alpha_{1},\alpha_{2},\ldots,\alpha
_{k}\right)  $. Then, $\operatorname*{eta}\left(  g\#h\right)  =\left(
\operatorname*{eta}g\right)  \cdot\left(  \operatorname*{eta}h\right)  $ for
any $g,h\in\mathcal{F}$.
\end{proposition}

\begin{proof}
[Proof sketch.]Let $g,h\in\mathcal{F}$. We WLOG assume that $g=x_{\delta}$ and
$h=x_{\varepsilon}$ for two compositions $\delta$ and $\varepsilon$. Consider
these $\delta$ and $\varepsilon$. Thus,%
\[
\operatorname*{eta}g=\operatorname*{eta}x_{\delta}=u^{\ell\left(
\delta\right)  }\eta_{\delta}^{\left(  q\right)  }\ \ \ \ \ \ \ \ \ \ \left(
\text{by the definition of }\operatorname*{eta}\right)
\]
and similarly $\operatorname*{eta}h=u^{\ell\left(  \varepsilon\right)  }%
\eta_{\varepsilon}^{\left(  q\right)  }$. Multiplying these two equalities, we
find%
\begin{align}
&  \left(  \operatorname*{eta}g\right)  \cdot\left(  \operatorname*{eta}%
h\right) \nonumber\\
&  =u^{\ell\left(  \delta\right)  }\eta_{\delta}^{\left(  q\right)  }\cdot
u^{\ell\left(  \varepsilon\right)  }\eta_{\varepsilon}^{\left(  q\right)
}=u^{\ell\left(  \delta\right)  +\ell\left(  \varepsilon\right)  }\eta
_{\delta}^{\left(  q\right)  }\eta_{\varepsilon}^{\left(  q\right)
}\nonumber\\
&  =u^{\ell\left(  \delta\right)  +\ell\left(  \varepsilon\right)  }%
\sum_{\substack{f\text{ is a stufufuffler}\\\text{for }\delta\text{ and
}\varepsilon}}\left(  -q\right)  ^{\operatorname*{loss}\left(  f\right)
}\left(  q-1\right)  ^{\operatorname*{poise}\left(  f\right)  }\eta
_{\operatorname*{wt}\left(  f\right)  }^{\left(  q\right)  }
\label{pf.prop.stufufuf.eta.1}%
\end{align}
(by Theorem \ref{thm.etaaetab}).

On the other hand, from $g=x_{\delta}$ and $h=x_{\varepsilon}$, we obtain%
\[
g\#h=x_{\delta}\#x_{\varepsilon}=\sum_{\substack{f\text{ is a stufufuffler}%
\\\text{for }\delta\text{ and }\varepsilon}}b^{\operatorname*{loss}\left(
f\right)  }a^{\operatorname*{poise}\left(  f\right)  }x_{\operatorname*{wt}%
\left(  f\right)  }%
\]
(by Proposition \ref{prop.stufufuf.explicit}). Hence, by the definition of
$\operatorname*{eta}$, we obtain%
\begin{equation}
\operatorname*{eta}\left(  g\#h\right)  =\sum_{\substack{f\text{ is a
stufufuffler}\\\text{for }\delta\text{ and }\varepsilon}%
}b^{\operatorname*{loss}\left(  f\right)  }a^{\operatorname*{poise}\left(
f\right)  }u^{\ell\left(  \operatorname*{wt}\left(  f\right)  \right)  }%
\eta_{\operatorname*{wt}\left(  f\right)  }^{\left(  q\right)  }.
\label{pf.prop.stufufuf.eta.2}%
\end{equation}

We must prove that the left hand sides of (\ref{pf.prop.stufufuf.eta.2}) and
(\ref{pf.prop.stufufuf.eta.1}) are equal. Of course, it suffices to show that
the right hand sides are equal. For that purpose, it suffices to show that%
\[
b^{\operatorname*{loss}\left(  f\right)  }a^{\operatorname*{poise}\left(
f\right)  }u^{\ell\left(  \operatorname*{wt}\left(  f\right)  \right)
}=u^{\ell\left(  \delta\right)  +\ell\left(  \varepsilon\right)  }\left(
-q\right)  ^{\operatorname*{loss}\left(  f\right)  }\left(  q-1\right)
^{\operatorname*{poise}\left(  f\right)  }%
\]
whenever $f$ is a stufufuffler for $\delta$ and $\varepsilon$. Recalling that
$a=\left(  q-1\right)  u$ and $b=-qu^{2}$, we can easily boil this down to the
fact that every stufufuffler $f$ for $\delta$ and $\varepsilon$ satisfies%
\[
2\operatorname*{loss}\left(  f\right)  +\operatorname*{poise}\left(  f\right)
+\ell\left(  \operatorname*{wt}\left(  f\right)  \right)  =\ell\left(
\delta\right)  +\ell\left(  \varepsilon\right)  ;
\]
but this fact is easily verified combinatorially.
\end{proof}

\begin{proof}
[Proof of Theorem \ref{thm.stufufuf.ass} (sketched).]All claims of this
theorem boil down to polynomial identities in $a$ and $b$. For example,
associativity of $\#$ is saying that the elements $\left(  u\#v\right)  \#w$
and $u\#\left(  v\#w\right)  $ of $\mathcal{F}$ have the same $t$-coefficient
whenever $u,v,w,t$ are four words; but this is easily revealed (upon expanding
all products) to be an equality between two polynomials in $a$ and $b$ (when
$u,v,w,t$ are fixed). Note that all relevant polynomials have integer coefficients.

Thus, in order to prove Theorem \ref{thm.stufufuf.ass}, we can WLOG assume
that $a$ and $b$ are two distinct indeterminates in a polynomial ring over
$\mathbb{Z}$ (for example, $a=X$ and $b=Y$ in the polynomial ring
$\mathbb{Z}\left[  X,Y\right]  $). Even better, we can WLOG assume that $a$
and $b$ are two algebraically independent elements of a $\mathbb{Z}$-algebra.

However, in the ring $\mathbb{Z}\left[  X,Y\right]  $, the two elements $X+Y$
and $XY$ are algebraically independent (since they are the elementary
symmetric polynomials in the indeterminates $X$ and $Y$). Thus, we can WLOG
assume that $\mathbf{k}=\mathbb{Z}\left[  X,Y\right]  $ and that $a=X+Y$ and
$b=XY$. Moreover, we can extend the base ring $\mathbf{k}$ to its quotient
field $\mathbb{Q}\left(  X,Y\right)  $. So we assume that $\mathbf{k}%
=\mathbb{Q}\left(  X,Y\right)  $ and $a=X+Y$ and $b=XY$.

Set $q:=-XY^{-1}$ and $u:=-Y$ in $\mathbf{k}$. Then, simple computations
confirm that $a=\left(  q-1\right)  u$ and $b=-qu^{2}$. Hence, the map
$\operatorname*{eta}:\mathcal{F}\rightarrow\operatorname*{QSym}$ constructed
in Proposition \ref{prop.stufufuf.eta} satisfies
\begin{equation}
\operatorname*{eta}\left(  g\#h\right)  =\left(  \operatorname*{eta}g\right)
\cdot\left(  \operatorname*{eta}h\right)  \ \ \ \ \ \ \ \ \ \ \text{for any
}g,h\in\mathcal{F} \label{pf.thm.stufufuf.ass.eta-nice}%
\end{equation}
(by Proposition \ref{prop.stufufuf.eta}). Moreover, the element $u=-Y\in
\mathbf{k}$ is invertible (since $\mathbf{k}$ is a field), and so is the
element $r:=q+1=-XY^{-1}+1\in\mathbf{k}$ (for the same reason, since $r\neq
0$). Thus, the family $\left(  u^{\ell\left(  \alpha\right)  }\eta_{\alpha
}^{\left(  q\right)  }\right)  _{\alpha\in\operatorname*{Comp}}$ is a basis of
$\operatorname*{QSym}$ (by Theorem \ref{thm.eta.basis} \textbf{(a)}). Hence,
the map $\operatorname*{eta}$ is a $\mathbf{k}$-module isomorphism (since it
sends the basis $\left(  x_{\alpha}\right)  _{\alpha\in\operatorname*{Comp}}$
of $\mathcal{F}$ to the basis $\left(  u^{\ell\left(  \alpha\right)  }%
\eta_{\alpha}^{\left(  q\right)  }\right)  _{\alpha\in\operatorname*{Comp}}$
of $\operatorname*{QSym}$). The equality (\ref{pf.thm.stufufuf.ass.eta-nice})
shows that this isomorphism $\operatorname*{eta}$ transfers the multiplication
of $\operatorname*{QSym}$ to the binary operation $\#$ on $\mathcal{F}$. Since
the former multiplication is associative, we thus conclude that the latter
operation $\#$ is associative as well. Similarly, we can see that $\#$ is
commutative. Finally, it is clear that $1$ is a neutral element for $\#$.
Thus, Theorem \ref{thm.stufufuf.ass} is proved.
\end{proof}

In view of Theorem \ref{thm.stufufuf.ass}, we can restate Proposition
\ref{prop.stufufuf.eta} as follows:

\begin{theorem}
\label{thm.stufufuf.eta-mor}Let $q$ and $u$ be two elements of $\mathbf{k}$
such that $a=\left(  q-1\right)  u$ and $b=-qu^{2}$.

Let $\operatorname*{eta}:\mathcal{F}\rightarrow\operatorname*{QSym}$ be the
$\mathbf{k}$-linear map that sends the word $x_{\alpha}=x_{\alpha_{1}%
}x_{\alpha_{2}}\cdots x_{\alpha_{k}}\in\mathcal{F}$ to $u^{\ell\left(
\alpha\right)  }\eta_{\alpha}^{\left(  q\right)  }\in\operatorname*{QSym}$ for
each composition $\alpha=\left(  \alpha_{1},\alpha_{2},\ldots,\alpha
_{k}\right)  $. Then, $\operatorname*{eta}$ is a $\mathbf{k}$-algebra
homomorphism from the $\mathbf{k}$-algebra $\left(  \mathcal{F},\#\right)  $
to the $\mathbf{k}$-algebra $\operatorname*{QSym}$.
\end{theorem}

We can also turn $\mathcal{F}$ into a coalgebra. In fact, let $\Delta
:\mathcal{F}\rightarrow\mathcal{F}\otimes\mathcal{F}$ be the $\mathbf{k}%
$-linear map that sends each word $w_{1}w_{2}\cdots w_{n}$ to $\sum_{i=0}%
^{n}w_{1}w_{2}\cdots w_{i}\otimes w_{i+1}w_{i+2}\cdots w_{n}$. This map
$\Delta$ is called the \emph{deconcatenation coproduct} (or the \emph{cut
coproduct}). This coproduct turns $\mathcal{F}$ into a coalgebra (with counit
$\varepsilon:\mathcal{F}\rightarrow\mathbf{k}$ sending each word $w_{1}%
w_{2}\cdots w_{n}$ to $%
\begin{cases}
1, & \text{if }n=0;\\
0, & \text{if }n>0
\end{cases}
$). The map $\operatorname*{eta}:\mathcal{F}\rightarrow\operatorname*{QSym}$
from Theorem \ref{thm.stufufuf.eta-mor} is then easily seen to be a
$\mathbf{k}$-coalgebra homomorphism (by Theorem \ref{thm.Delta-eta}).

The stufufuffle product $\#$ on $\mathcal{F}$ respects the deconcatenation
coproduct $\Delta$ of $\mathcal{F}$, in the following sense:

\begin{theorem}
\label{thm.stufufuff.hopf}The $\mathbf{k}$-algebra $\left(  \mathcal{F}%
,\#\right)  $, equipped with the coproduct $\Delta$ and the counit
$\varepsilon$ constructed above, is a commutative connected graded Hopf algebra.
\end{theorem}

\begin{theorem}
\label{thm.stufufuf.eta-hmor}Let $q$ and $u$ be two elements of $\mathbf{k}$
such that $a=\left(  q-1\right)  u$ and $b=-qu^{2}$.

Let $\operatorname*{eta}:\mathcal{F}\rightarrow\operatorname*{QSym}$ be the
$\mathbf{k}$-linear map from Theorem \ref{thm.stufufuf.eta-mor}. Then,
$\operatorname*{eta}$ is a Hopf algebra homomorphism from the Hopf algebra
$\left(  \mathcal{F},\#,\Delta,\varepsilon\right)  $ to the Hopf algebra
$\operatorname*{QSym}$.
\end{theorem}

We leave the proofs of these two theorems to the reader. (They follow the same
mold as our above proof of Theorem \ref{thm.stufufuf.ass}.)

Likewise, using Theorem \ref{thm.eta.S2} and the proof method of Theorem
\ref{thm.stufufuf.ass} above, we can prove the following:

\begin{theorem}
\label{thm.stufufuff.antipode}Let $S$ be the antipode of the Hopf algebra
$\left(  \mathcal{F},\#\right)  $ constructed in Theorem
\ref{thm.stufufuff.hopf}. Let $n\in\mathbb{N}$ and $\alpha\in
\operatorname*{Comp}\nolimits_{n}$. Then, in $\mathcal{F}$, we have%
\[
S\left(  x_{\alpha}\right)  =\left(  -1\right)  ^{\ell\left(  \alpha\right)
}\sum_{\substack{\beta\in\operatorname*{Comp}\nolimits_{n};\\D\left(
\beta\right)  \subseteq D\left(  \operatorname*{rev}\alpha\right)  }%
}a^{\ell\left(  \alpha\right)  -\ell\left(  \beta\right)  }x_{\beta}.
\]

\end{theorem}

The recent work \cite[Theorem 5.2]{BoNoTh22} constructs another basis $\left(
X_{I}\right)  $ of $\operatorname*{QSym}$ (indexed by subsets $I$ of $\left[
n-1\right]  $ instead of compositions $\alpha$, but this difference is
insubstantial) that multiplies according to the stufufuffle product (thus
obtaining another $\mathbf{k}$-algebra homomorphism from $\mathcal{F}$ to
$\operatorname*{QSym}$, and with it another proof of Theorem
\ref{thm.stufufuf.ass}). While similar to ours, it uses the
alphabet-transformed functions $H_{k}\left(  \left(  s-t\right)  A\right)  $
instead of the plain $H_{k}$, which lead to a basis of $\operatorname*{QSym}$
that does not appear to have a simple combinatorial formula like our
$\eta_{\alpha}^{\left(  q\right)  }$.

\begin{remark}
Assume that $\mathbf{k}$ is a field of characteristic $0$. Then, Leray's
theorem (\cite[Theorem 1.7.29(f)]{GriRei}) shows that any commutative
connected graded $\mathbf{k}$-bialgebra $A$ is isomorphic as a graded
$\mathbf{k}$-algebra to the symmetric algebra of a certain graded $\mathbf{k}%
$-module (namely, of $\left( \ker\varepsilon\right)  / \left( \ker
\varepsilon\right) ^{2}$, where $\varepsilon$ is the counit of $A$). In other
words, any such $A$ is isomorphic as a graded $\mathbf{k}$-algebra to a
polynomial ring whose generators are homogeneous of various positive degrees,
with exactly $\dim\left( \left( \ker\varepsilon\right)  / \left(
\ker\varepsilon\right) ^{2}\right) _{i}$ many generators of degree $i$. This
applies, in particular, to our connected graded $\mathbf{k}$-bialgebra
$\left(  \mathcal{F},\#\right)  $. Moreover, using standard Hilbert-series
arguments, it is easy to see that the number of generators of each given
degree does not depend on the parameters $a$ and $b$. Hence, as a graded
$\mathbf{k}$-algebra, our $\left(  \mathcal{F},\#\right)  $ is isomorphic to
the usual shuffle algebra (which is obtained for $a=0$ and $b=0$).

However, this is not a $\mathbf{k}$-coalgebra isomorphism; nor is it canonical
(although we suspect that a canonical $\mathbf{k}$-algebra isomorphism may
exist); nor does it extend to fields of positive characteristic.
\end{remark}

\section{\label{sec.Rq}Appendix: The map $R_{q}$}

We finish by stating yet another formula for $\eta_{\alpha}^{\left(  q\right)
}$, which may eventually prove useful in understanding these functions. This
formula relies on some more notations. We first define a simple combinatorial
operation on compositions:

\begin{definition}
\label{def.composition-complement}Let $\alpha\in\operatorname*{Comp}$, and let
$n=\left\vert \alpha\right\vert $. Then, $\overline{\alpha}$ shall denote the
unique composition $\gamma$ of $n$ such that $D\left(  \gamma\right)  =\left[
n-1\right]  \setminus D\left(  \alpha\right)  $. (This $\gamma$ is indeed
unique, since the map $D$ is a bijection.) This composition $\overline{\alpha
}$ is called the \emph{complement} of $\alpha$.
\end{definition}

For example, $\overline{\left(  2,5,1,1\right)  }=\left(  1,2,1,1,1,3\right)
$. We observe some simple properties of complements of compositions:

\begin{proposition}
\label{prop.composition-complement.basics}\ \ 

\begin{enumerate}
\item[\textbf{(a)}] Every composition $\alpha$ satisfies $\overline
{\overline{\alpha}}=\alpha$.

\item[\textbf{(b)}] For each $n\in\mathbb{N}$, the map $\operatorname*{Comp}%
\nolimits_{n}\rightarrow\operatorname*{Comp}\nolimits_{n},\ \beta
\mapsto\overline{\beta}$ is a bijection.

\item[\textbf{(c)}] If $\alpha$ and $\beta$ are two compositions of $n$ for
some $n\in\mathbb{N}$, then the statements \textquotedblleft$D\left(
\beta\right)  \subseteq D\left(  \alpha\right)  $\textquotedblright\ and
\textquotedblleft$D\left(  \overline{\beta}\right)  \supseteq D\left(
\overline{\alpha}\right)  $\textquotedblright\ are equivalent.

\item[\textbf{(d)}] If $\alpha$ is a composition of a positive integer $n$,
then $\ell\left(  \overline{\alpha}\right)  +\ell\left(  \alpha\right)  =n+1$.
\end{enumerate}
\end{proposition}

\begin{proof}
\textbf{(a)} Let $\alpha$ be a composition. Let $n=\left\vert \alpha
\right\vert $. Thus, $D\left(  \alpha\right)  \subseteq\left[  n-1\right]  $.

The definition of $\overline{\alpha}$ yields that $\overline{\alpha}$ is a
composition of $n$ and satisfies $D\left(  \overline{\alpha}\right)  =\left[
n-1\right]  \setminus D\left(  \alpha\right)  $. Hence, the definition of
$\overline{\overline{\alpha}}$ yields that $\overline{\overline{\alpha}}$ is a
composition of $n$ and satisfies%
\[
D\left(  \overline{\overline{\alpha}}\right)  =\left[  n-1\right]
\setminus\underbrace{D\left(  \overline{\alpha}\right)  }_{=\left[
n-1\right]  \setminus D\left(  \alpha\right)  }=\left[  n-1\right]
\setminus\left(  \left[  n-1\right]  \setminus D\left(  \alpha\right)
\right)  =D\left(  \alpha\right)
\]
(since $D\left(  \alpha\right)  \subseteq\left[  n-1\right]  $). Since the map
$D:\operatorname*{Comp}\nolimits_{n}\rightarrow\mathcal{P}\left(  \left[
n-1\right]  \right)  $ is a bijection (and since both $\overline
{\overline{\alpha}}$ and $\alpha$ belong to $\operatorname*{Comp}%
\nolimits_{n}$), we thus conclude that $\overline{\overline{\alpha}}=\alpha$.
This proves Proposition \ref{prop.composition-complement.basics} \textbf{(a)}.
\medskip

\textbf{(b)} Let $n\in\mathbb{N}$. Then, $\overline{\beta}\in
\operatorname*{Comp}\nolimits_{n}$ for each $\beta\in\operatorname*{Comp}%
\nolimits_{n}$ (by the definition of $\overline{\beta}$). Hence, the map
$\operatorname*{Comp}\nolimits_{n}\rightarrow\operatorname*{Comp}%
\nolimits_{n},\ \beta\mapsto\overline{\beta}$ is well-defined. This map is
furthermore its own inverse (by Proposition
\ref{prop.composition-complement.basics} \textbf{(a)}). Thus, it is
invertible, i.e., a bijection. This proves Proposition
\ref{prop.composition-complement.basics} \textbf{(b)}. \medskip

\textbf{(c)} Let $\alpha$ and $\beta$ be two compositions of $n$ for some
$n\in\mathbb{N}$. Then, $D\left(  \alpha\right)  $ and $D\left(  \beta\right)
$ are two subsets of $\left[  n-1\right]  $. Meanwhile, $D\left(
\overline{\alpha}\right)  $ and $D\left(  \overline{\beta}\right)  $ are the
complements of these two subsets (i.e., we have $D\left(  \overline{\alpha
}\right)  =\left[  n-1\right]  \setminus D\left(  \alpha\right)  $ and
$D\left(  \overline{\beta}\right)  =\left[  n-1\right]  \setminus D\left(
\beta\right)  $), by the definition of $\overline{\alpha}$ and $\overline
{\beta}$. However, it is well-known from set theory that complementation
reverses inclusion of subsets: That is, if $X$ and $Y$ are two subsets of a
given set $Z$, then the statement \textquotedblleft$X\subseteq Y$%
\textquotedblright\ is equivalent to the statement \textquotedblleft%
$Z\setminus X\supseteq Z\setminus Y$\textquotedblright. Applying this to
$X=D\left(  \beta\right)  $ and $Y=D\left(  \alpha\right)  $ and $Z=\left[
n-1\right]  $, we conclude that the statement \textquotedblleft$D\left(
\beta\right)  \subseteq D\left(  \alpha\right)  $\textquotedblright\ is
equivalent to the statement \textquotedblleft$\left[  n-1\right]  \setminus
D\left(  \beta\right)  \supseteq\left[  n-1\right]  \setminus D\left(
\alpha\right)  $\textquotedblright. In other words, the statement
\textquotedblleft$D\left(  \beta\right)  \subseteq D\left(  \alpha\right)
$\textquotedblright\ is equivalent to the statement \textquotedblleft$D\left(
\overline{\beta}\right)  \supseteq D\left(  \overline{\alpha}\right)
$\textquotedblright\ (since $D\left(  \overline{\alpha}\right)  =\left[
n-1\right]  \setminus D\left(  \alpha\right)  $ and $D\left(  \overline{\beta
}\right)  =\left[  n-1\right]  \setminus D\left(  \beta\right)  $). This
proves Proposition \ref{prop.composition-complement.basics} \textbf{(c)}.
\medskip

\textbf{(d)} Let $\alpha$ be a composition of a positive integer $n$. Then,
$\left[  n\neq0\right]  =1$ (since $n$ is positive), so that Lemma
\ref{lem.comps.l-vs-size} \textbf{(a)} yields $\ell\left(  \alpha\right)
=\left\vert D\left(  \alpha\right)  \right\vert +\underbrace{\left[
n\neq0\right]  }_{=1}=\left\vert D\left(  \alpha\right)  \right\vert +1$.
Similarly, $\ell\left(  \overline{\alpha}\right)  =\left\vert D\left(
\overline{\alpha}\right)  \right\vert +1$ (since $\overline{\alpha}$ is a
composition of $n$ as well). But the definition of $\overline{\alpha}$ yields
$D\left(  \overline{\alpha}\right)  =\left[  n-1\right]  \setminus D\left(
\alpha\right)  $, so that%
\[
\left\vert D\left(  \overline{\alpha}\right)  \right\vert =\left\vert \left[
n-1\right]  \setminus D\left(  \alpha\right)  \right\vert =\left\vert \left[
n-1\right]  \right\vert -\left\vert D\left(  \alpha\right)  \right\vert
\]
(since $D\left(  \alpha\right)  \subseteq\left[  n-1\right]  $). Therefore,
$\left\vert D\left(  \overline{\alpha}\right)  \right\vert +\left\vert
D\left(  \alpha\right)  \right\vert =\left\vert \left[  n-1\right]
\right\vert =n-1$ (since $n\geq1$). Now,%
\begin{align*}
\underbrace{\ell\left(  \overline{\alpha}\right)  }_{=\left\vert D\left(
\overline{\alpha}\right)  \right\vert +1}+\underbrace{\ell\left(
\alpha\right)  }_{=\left\vert D\left(  \alpha\right)  \right\vert +1}  &
=\left\vert D\left(  \overline{\alpha}\right)  \right\vert +1+\left\vert
D\left(  \alpha\right)  \right\vert +1=\underbrace{\left\vert D\left(
\overline{\alpha}\right)  \right\vert +\left\vert D\left(  \alpha\right)
\right\vert }_{=n-1}+\,2\\
&  =n-1+2=n+1.
\end{align*}
This proves Proposition \ref{prop.composition-complement.basics} \textbf{(d)}.
\qedhere

\begin{verlong}
T0D0: verlong proof.
\end{verlong}
\end{proof}

We now define a linear endomorphism of $\operatorname*{QSym}$:

\begin{definition}
\label{def.Rq}We let $R_{q}$ be the $\mathbf{k}$-linear map from
$\operatorname*{QSym}$ to $\operatorname*{QSym}$ that sends each $M_{\alpha}$
(with $\alpha\in\operatorname*{Comp}$) to $r^{\ell\left(  \overline{\alpha
}\right)  }M_{\overline{\alpha}}$. (This is well-defined, since $\left(
M_{\alpha}\right)  _{\alpha\in\operatorname*{Comp}}$ is a basis of
$\operatorname*{QSym}$.)
\end{definition}

This map $R_{q}$ is neither an algebra endomorphism nor a coalgebra
endomorphism of $\operatorname*{QSym}$ (not even when $r=1$), but it is
exactly what we need for our formula. First, let us observe that the map
$R_{q}$ is \textquotedblleft close to an involution\textquotedblright\ in the
following sense:

\begin{proposition}
\label{prop.Rq.invol}Let $n$ be a positive integer. Let $f\in
\operatorname*{QSym}$ be homogeneous of degree $n$. Then, $\left(  R_{q}\circ
R_{q}\right)  \left(  f\right)  =r^{n+1}f$.
\end{proposition}

\begin{proof}
Both sides of this equality are $\mathbf{k}$-linear in $f$. Thus, it suffices
to prove this equality for $f=M_{\alpha}$ for all compositions $\alpha
\in\operatorname*{Comp}\nolimits_{n}$ (since the family $\left(  M_{\alpha
}\right)  _{\alpha\in\operatorname*{Comp}\nolimits_{n}}$ is a basis of the
$n$-th graded component of $\operatorname*{QSym}$). In other words, it
suffices to show that $\left(  R_{q}\circ R_{q}\right)  \left(  M_{\alpha
}\right)  =r^{n+1}M_{\alpha}$ for every $\alpha\in\operatorname*{Comp}%
\nolimits_{n}$. However, this is easy:

Let $\alpha\in\operatorname*{Comp}\nolimits_{n}$. Then, the definition of
$R_{q}$ yields $R_{q}\left(  M_{\alpha}\right)  =r^{\ell\left(  \overline
{\alpha}\right)  }M_{\overline{\alpha}}$. Applying the map $R_{q}$ to both
sides of this equality, we find%
\begin{equation}
R_{q}\left(  R_{q}\left(  M_{\alpha}\right)  \right)  =R_{q}\left(
r^{\ell\left(  \overline{\alpha}\right)  }M_{\overline{\alpha}}\right)
=r^{\ell\left(  \overline{\alpha}\right)  }R_{q}\left(  M_{\overline{\alpha}%
}\right)  . \label{pf.prop.Rq.invol.1}%
\end{equation}
But the definition of $R_{q}$ yields $R_{q}\left(  M_{\overline{\alpha}%
}\right)  =r^{\ell\left(  \overline{\overline{\alpha}}\right)  }%
M_{\overline{\overline{\alpha}}}=r^{\ell\left(  \alpha\right)  }M_{\alpha}$
(since Proposition \ref{prop.composition-complement.basics} \textbf{(a)}
yields $\overline{\overline{\alpha}}=\alpha$). Thus, (\ref{pf.prop.Rq.invol.1}%
) becomes%
\[
R_{q}\left(  R_{q}\left(  M_{\alpha}\right)  \right)  =r^{\ell\left(
\overline{\alpha}\right)  }\underbrace{R_{q}\left(  M_{\overline{\alpha}%
}\right)  }_{=r^{\ell\left(  \alpha\right)  }M_{\alpha}}=r^{\ell\left(
\overline{\alpha}\right)  }r^{\ell\left(  \alpha\right)  }M_{\alpha}%
=r^{\ell\left(  \overline{\alpha}\right)  +\ell\left(  \alpha\right)
}M_{\alpha}=r^{n+1}M_{\alpha}%
\]
(since Proposition \ref{prop.composition-complement.basics} \textbf{(d)}
yields $\ell\left(  \overline{\alpha}\right)  +\ell\left(  \alpha\right)
=n+1$). Hence, $\left(  R_{q}\circ R_{q}\right)  \left(  M_{\alpha}\right)
=R_{q}\left(  R_{q}\left(  M_{\alpha}\right)  \right)  =r^{n+1}M_{\alpha}$. As
explained above, this completes the proof of Proposition \ref{prop.Rq.invol}.
\qedhere

\begin{verlong}
T0D0: verlong proof.
\end{verlong}
\end{proof}

Now we can state our final formula for $\eta_{\alpha}^{\left(  q\right)  }$:

\begin{theorem}
\label{thm.Rq.eta}Let $\alpha\in\operatorname*{Comp}$. Then,
\[
\eta_{\alpha}^{\left(  q\right)  }=R_{q}\left(  L_{\overline{\alpha}}\right)
.
\]

\end{theorem}

\begin{proof}
Let $n=\left\vert \alpha\right\vert $, so that $\alpha\in\operatorname*{Comp}%
\nolimits_{n}$. Therefore, $\overline{\alpha}\in\operatorname*{Comp}%
\nolimits_{n}$ as well (by the definition of $\overline{\alpha}$), so that
$n=\left\vert \overline{\alpha}\right\vert $. Hence, applying
(\ref{eq.Lalpha.def}) to $\overline{\alpha}$ instead of $\alpha$, we obtain%
\begin{equation}
L_{\overline{\alpha}}=\sum_{\substack{\beta\in\operatorname*{Comp}%
\nolimits_{n};\\D\left(  \beta\right)  \supseteq D\left(  \overline{\alpha
}\right)  }}M_{\beta}=\sum_{\substack{\beta\in\operatorname*{Comp}%
\nolimits_{n};\\D\left(  \overline{\beta}\right)  \supseteq D\left(
\overline{\alpha}\right)  }}M_{\overline{\beta}} \label{pf.thm.Rq.eta.1}%
\end{equation}
(here, we have substituted $\overline{\beta}$ for $\beta$ in the sum, since
Proposition \ref{prop.composition-complement.basics} \textbf{(b)} shows that
the map $\operatorname*{Comp}\nolimits_{n}\rightarrow\operatorname*{Comp}%
\nolimits_{n},\ \beta\mapsto\overline{\beta}$ is a bijection). However, the
condition \textquotedblleft$D\left(  \overline{\beta}\right)  \supseteq
D\left(  \overline{\alpha}\right)  $\textquotedblright\ under the summation
sign on the right hand side of (\ref{pf.thm.Rq.eta.1}) is equivalent to
\textquotedblleft$D\left(  \beta\right)  \subseteq D\left(  \alpha\right)
$\textquotedblright\ (by Proposition \ref{prop.composition-complement.basics}
\textbf{(c)}). Hence, we can rewrite (\ref{pf.thm.Rq.eta.1}) as%
\[
L_{\overline{\alpha}}=\sum_{\substack{\beta\in\operatorname*{Comp}%
\nolimits_{n};\\D\left(  \beta\right)  \subseteq D\left(  \alpha\right)
}}M_{\overline{\beta}}.
\]
Applying the map $R_{q}$ to both sides of this equality, we find%
\begin{align*}
R_{q}\left(  L_{\overline{\alpha}}\right)   &  =R_{q}\left(  \sum
_{\substack{\beta\in\operatorname*{Comp}\nolimits_{n};\\D\left(  \beta\right)
\subseteq D\left(  \alpha\right)  }}M_{\overline{\beta}}\right)
=\sum_{\substack{\beta\in\operatorname*{Comp}\nolimits_{n};\\D\left(
\beta\right)  \subseteq D\left(  \alpha\right)  }}\underbrace{R_{q}\left(
M_{\overline{\beta}}\right)  }_{\substack{=r^{\ell\left(  \overline
{\overline{\beta}}\right)  }M_{\overline{\overline{\beta}}}\\\text{(by the
definition of }R_{q}\text{)}}}\\
&  =\sum_{\substack{\beta\in\operatorname*{Comp}\nolimits_{n};\\D\left(
\beta\right)  \subseteq D\left(  \alpha\right)  }}\underbrace{r^{\ell\left(
\overline{\overline{\beta}}\right)  }M_{\overline{\overline{\beta}}}%
}_{\substack{=r^{\ell\left(  \beta\right)  }M_{\beta}\\\text{(since
Proposition \ref{prop.composition-complement.basics} \textbf{(a)}%
}\\\text{yields }\overline{\overline{\beta}}=\beta\text{)}}}=\sum
_{\substack{\beta\in\operatorname*{Comp}\nolimits_{n};\\D\left(  \beta\right)
\subseteq D\left(  \alpha\right)  }}r^{\ell\left(  \beta\right)  }M_{\beta
}=\eta_{\alpha}^{\left(  q\right)  }\ \ \ \ \ \ \ \ \ \ \left(  \text{by
(\ref{eq.def.etaalpha.def})}\right)  .
\end{align*}
This proves Theorem \ref{thm.Rq.eta}. \qedhere

\begin{verlong}
T0D0: verlong proof.
\end{verlong}
\end{proof}

\begin{remark}
Let $n$ be a positive integer, and let $\alpha\in\operatorname*{Comp}%
\nolimits_{n}$. Combining Theorem \ref{thm.Rq.eta} with Proposition
\ref{prop.Rq.invol}, we can easily see that $R_{q}\left(  \eta_{\alpha
}^{\left(  q\right)  }\right)  =r^{n+1}L_{\overline{\alpha}}$. Contrasting
this equality with Theorem \ref{thm.Rq.eta} reveals a symmetry of sorts
between the $\eta_{\alpha}^{\left(  q\right)  }$ and $L_{\overline{\alpha}}$.
This symmetry explains the similarity between Proposition
\ref{prop.eta.through-F} and Proposition \ref{prop.eta.F-through} (and allows
one to derive one of these propositions from the other with a bit of work).
\end{remark}

\end{document}